\documentclass{article}

\usepackage{arxiv}

\usepackage[utf8]{inputenc} 
\usepackage[T1]{fontenc}    
\usepackage{hyperref}       
\usepackage{url}            
\usepackage{booktabs}       
\usepackage{amsfonts}       

\usepackage{nicefrac}       
\usepackage{microtype}      
\usepackage{lipsum}		
\usepackage{graphicx}
\usepackage{natbib}
\usepackage{doi}

\usepackage{amssymb}
\usepackage{amsmath}
\usepackage{amsthm}
\usepackage[utf8]{inputenc}
\usepackage{mathtools}
\usepackage{mathrsfs}
\usepackage{multirow}
\usepackage{import}
\usepackage{tikz-cd}
\usepackage[symbols,nogroupskip,sort=none]{glossaries-extra}

\newcommand{\keyword}[1]{\textbf{#1}}

\newcommand{\innerprod}[3]{\langle #1,#2 \rangle_{#3}}
\newcommand{\pinnerprod}[3]{\big( #1,#2 \big)_{#3}}

\newcommand{\tangentbd}[1]{\mathrm{T}#1}

\newcommand{\TM}[1]{\mathrm{T}_{#1} M}

\newcommand{\TMC}[1]{\mathrm{T}_{#1}M^{\mathbb{C}}}

\newcommand{\TstarM}[1]{\mathrm{T}_{#1}^{*} M}

\newcommand{\g}{\mathfrak{g}}
\newcommand{\gC}{\mathfrak{g}^{\mathbb{C}}}
\newcommand{\gk}[1]{\mathfrak{g}^{#1}}
\newcommand{\gpq}[2]{\mathfrak{g}^{#1,#2}}

\newcommand{\gstar}{\mathfrak{g}^{*}}
\newcommand{\gstarC}{\mathfrak{g}^{*\mathbb{C}}}

\newcommand{\gstarpq}[2]{{\mathfrak{g}^*}^{#1,#2}}

\newcommand{\secsp}[2]{\Gamma(#1,#2)}
\newcommand{\smoothfuncsp}[1]{C^{\infty}(#1)}

\newcommand{\Ak}[1]{\mathcal{A}^{#1}_{\mathfrak{g}}}
\newcommand{\Apq}[2]{\mathcal{A}^{#1,#2}_{\mathfrak{g}}}
\newcommand{\cohomo}[2]{\mathrm{H}^{#1}_{#2}}

\newcommand{\diff}{\mathrm{d}}
\newcommand{\LieagbdDiff}{\mathrm{d}_{\mathfrak{g}}}
\newcommand{\LiegpDiff}{\mathrm{d}_{G}}

\newcommand{\LieagbdConn}[2]{\prescript{\mathfrak{g}}{}{\nabla^{#1}_{#2}}}
\newcommand{\LieagbdChernConn}[2]{\prescript{\mathfrak{g}}{}{\nabla^{#1}_{#2}}}
\newcommand{\LieagbdLpls}[2]{\prescript{\mathfrak{g}}{}{\Delta^{#1}_{#2}}}

\newcommand{\LieagbdPartial}{\partial_{\mathfrak{g}}}
\newcommand{\partialbar}{\bar{\partial}}
\newcommand{\LieagbdPartialbar}{\bar{\partial}_{\mathfrak{g}}}

\newcommand{\norm}[2]{\|#1\|_{#2}}

\newcommand{\ext}[1]{\epsilon_{#1}}
\newcommand{\contract}[1]{\iota_{#1}}

\newcommand{\C}{\mathbb{C}}
\newcommand{\R}{\mathbb{R}}

\newcommand{\rank}{\mathrm{rank}}

\newcommand{\kernel}{\mathrm{ker}}

\newcommand{\id}{\mathrm{id}}

\newcommand{\anchor}{\rho}
\newcommand{\End}{\mathrm{End}}

\newcommand{\Ind}{\mathrm{Ind}}
\newcommand{\Th}{\mathrm{Th}}

\newcommand{\chg}[1]{\mathrm{ch}^{#1}}
\newcommand{\cg}[1]{\mathrm{c}^{#1}}
\newcommand{\Tdg}[1]{\mathrm{Td}^{#1}}
\newcommand{\euler}[1]{\mathrm{eu}^{#1}}

\newcommand{\unitarygp}[1]{\mathrm{U}(#1)}
\newcommand{\unitaryLalgb}[1]{\mathfrak{u}(#1)}

\newcommand{\basexi}[1]{\xi_{#1}}
\newcommand{\basev}[1]{v_{#1}}
\newcommand{\basedualxi}[1]{\xi^{'}_{#1}}
\newcommand{\basedualv}[1]{v^{'}_{#1}}

\newcommand{\indexi}[1]{i_{#1}}
\newcommand{\indexip}[1]{i^{'}_{#1}}
\newcommand{\indexj}[1]{j_{#1}}
\newcommand{\indexjp}[1]{j^{'}_{#1}}
\newcommand{\indexI}{\mathbf{I}}
\newcommand{\indexJ}{\mathbf{J}}
\newcommand{\indexIp}{\mathbf{I}^{'}}
\newcommand{\indexJp}{\mathbf{J}^{'}}

\newcommand{\indexic}[1]{i^{c}_{#1}}

\newcommand{\indexjpc}[1]{{j^{'}_{#1}}^{c}}
\newcommand{\indexIc}{\mathbf{I}^{C}}
\newcommand{\indexJc}{\mathbf{J}^{C}}

\newcommand{\indexJpc}{{\mathbf{J}^{'}}^{C}}

\newcommand{\coefA}[2]{A^{#1}_{#2}}
\newcommand{\coefB}[2]{B^{#1}_{#2}}
\newcommand{\coefC}[2]{C^{#1}_{#2}}
\newcommand{\coefD}[2]{D^{#1}_{#2}}

\newcommand{\iotaW}{\iota_{\omega}}
\newcommand{\order}[1]{\mathrm{order}\big( #1 \big)}

\newcommand{\sgn}[2]{\mathrm{sgn} \Big(\begin{smallmatrix} #1 \\#2\end{smallmatrix}\Big)}

\newcommand{\tgdiff}{\mathrm{d}_{\tau}}
\newcommand{\tgAk}[1]{\mathcal{A}^{#1}_{\tau}}

\newcommand{\tgcohomo}[1]{\mathrm{H}^{#1}_{\tau}}

\newcommand{\tgApq}[2]{\mathcal{A}^{#1,#2}_{\tau}}

\newcommand{\tgpartialbar}{\bar{\partial}_{\tau}}
\newcommand{\tgconn}[2]{\prescript{\tau}{}{\nabla_{#2}#1}}

\newcommand{\tgomega}{\prescript{\tau}{}{\omega}}

\newcommand{\tgcherncharact}[1]{\mathrm{ch}^{\tau}(#1)}

\newcommand{\tgchernclass}[1]{\mathrm{c}^{\tau}(#1)}

\newcommand{\tgTdclass}[1]{\mathrm{Td}^{\tau}(#1)}
\newcommand{\tgEulerclass}[1]{\mathrm{eu}^{\tau}(#1)}

\newcommand{\btgbd}[1]{\prescript{b}{}{\mathrm{T}#1}}
\newcommand{\bcotgbd}[1]{\prescript{b}{}{\mathrm{T}^{*} #1}}
\newcommand{\btgbdC}[1]{\prescript{b}{}{\mathrm{T}#1^{\mathbb{C}}}}

\newcommand{\btgbdpq}[3]{\prescript{b}{}{\mathrm{T}#1^{#2,#3}}}
\newcommand{\bApq}[2]{\mathcal{A}^{#1,#2}_{b}}
\newcommand{\bAk}[1]{\mathcal{A}^{#1}_{b}}

\newcommand{\bDiff}{\mathrm{d}_{b}}
\newcommand{\bPartial}{\partial_{b}}
\newcommand{\bPartialbar}{\bar{\partial_{b}}}

\newcommand{\bConn}[1]{\prescript{b}{}{\nabla_{#1}}}

\newcommand{\ddx}[1]{\frac{\partial}{\partial #1}}
\newcommand{\CPk}[1]{\mathbb{C}\mathrm{P}^{#1}}
\newcommand{\bomega}{\prescript{b}{}{\omega}}

\newtheorem{theorem}{Theorem}
\newtheorem*{theorem*}{Theorem}
\newtheorem{definition}{Definition}
\newtheorem{proposition}{Proposition}
\newtheorem*{proposition*}{Proposition}
\newtheorem{lemma}{Lemma}
\newtheorem{corollary}{Corollary}

\glsxtrnewsymbol
    [description={Lie groupoid}]
    {Liegpod}
    {\ensuremath{G\rightrightarrows M}}
    
\glsxtrnewsymbol
    [description={Space of k-tuples of composable arrows}]
    {Gk}
    {\ensuremath{G^{(k)}}}
    
\glsxtrnewsymbol
    [description={Complex Lie algebroid}]
    {Lieagbd}
    {\ensuremath{\g}}

\glsxtrnewsymbol
    [description={Complexified Lie algebroid}]
    {LieagbdC}
    {\ensuremath{\gC}}

\glsxtrnewsymbol
    [description={k-degree exterior algebra of $\g$}]
    {gk}
    {\ensuremath{\g^{k}}}

\glsxtrnewsymbol
    [description={$(p,q)$-degree exterior algebra of $\g$}]
    {gpq}
    {\ensuremath{\gpq{p}{q}}}

\glsxtrnewsymbol
    [description={Almost complex structure}]
    {almostcplxstr}
    {\ensuremath{J}}

\glsxtrnewsymbol
    [description={Nijenhuis tensor}]
    {Njtensor}
    {\ensuremath{\mathcal{N}}}

\glsxtrnewsymbol
    [description={Space of Lie algebroid $k$-forms}]
    {Agk}
    {\ensuremath{\Ak{k}(M)}}

\glsxtrnewsymbol
    [description={Space of Lie algebroid $(p,q)$-forms}]
    {Agpq}
    {\ensuremath{\Apq{p}{q}(M)}}

\glsxtrnewsymbol
    [description={Lie algebroid differential}]
    {gDiff}
    {\ensuremath{\LieagbdDiff}}

\glsxtrnewsymbol
    [description={Lie algebroid holomorphic differential}]
    {gPartial}
    {\ensuremath{\LieagbdPartial}}

\glsxtrnewsymbol
    [description={Lie algebroid anti-holomorphic differential}]
    {gPartialbar}
    {\ensuremath{\LieagbdPartialbar}}

\glsxtrnewsymbol
    [description={Lie algebroid connection}]
    {gConn}
    {\ensuremath{\LieagbdConn{}{}}}

\glsxtrnewsymbol
    [description={Lie algebroid curvature tensor}]
    {gCurv}
    {\ensuremath{R(\LieagbdConn{}{})}}

\glsxtrnewsymbol
    [description={Sign of permutation}]
    {sgn}
    {\ensuremath{\sgn
            {\indexi{1},\cdots,\indexi{n}}
            {\tilde{i}_{1},\cdots,\tilde{i}_{n}}}}

\glsxtrnewsymbol
    [description={Sign of ordered array}]
    {ord}
    {\ensuremath{\order{\indexi{1},\cdots,\indexi{n}}
    }}

\glsxtrnewsymbol
    [description={Unitary local frame of $\gpq{1}{0}$ and $\gpq{0}{1}$}]
    {basevanddualv}
    {\ensuremath{\basev{i} \text{ and } \basedualv{i}}
    }

\glsxtrnewsymbol
    [description={Unitary local frame of $\gstarpq{1}{0}$ and $\gstarpq{0}{1}$}]
    {basexianddualxi}
    {\ensuremath{\basexi{i} \text{ and }\basedualxi{i}}
    }

\glsxtrnewsymbol
    [description={Lie algebroid K\"{a}hler form}]
    {kahlerform}
    {\ensuremath{\omega}
    }

\glsxtrnewsymbol
    [description={Transversal density}]
    {transdensity}
    {\ensuremath{\Omega}
    }
    
\glsxtrnewsymbol
    [description={Hodge star}]
    {Hodgestar}
    {\ensuremath{\star}
    }

\glsxtrnewsymbol
    [description={Wedge product of the K\"{a}hler form}]
    {L}
    {\ensuremath{L}
    }

\glsxtrnewsymbol
    [description={Formal adjoint of the wedge product of K\"{a}hler form}]
    {Lambda}
    {\ensuremath{\Lambda}
    }

\glsxtrnewsymbol
    [description={Contraction operator}]
    {iota}
    {\ensuremath{\iota}
    }

\glsxtrnewsymbol
    [description={$\LieagbdPartialbar$-Laplacian}]
    {partialbarLaplacian}
    {\ensuremath{\LieagbdLpls{}{\LieagbdPartialbar}}
    }

\glsxtrnewsymbol
    [description={Differential complex of a Lie groupoid}]
    {diffcplx}
    {\ensuremath{C^{k}_{\mathrm{diff}}(G,E)}
    }

\glsxtrnewsymbol
    [description={Lie groupoid differential}]
    {Liegpoddiff}
    {\ensuremath{\LiegpDiff}
    }

\glsxtrnewsymbol
    [description={Lie groupoid differential cohomology}]
    {Liegpoddiffcohomo}
    {\ensuremath{\cohomo{\bullet}{\mathrm{diff}}(G,E)}
    }

\glsxtrnewsymbol
    [description={Pullback Lie algebroid along the projection}]
    {pi!g}
    {\ensuremath{\pi^{!}\g}
    }

\glsxtrnewsymbol
    [description={Unitary Lie group}]
    {Un}
    {\ensuremath{\unitarygp{n}}
    }

\glsxtrnewsymbol
    [description={Unitary Lie algebra}]
    {un}
    {\ensuremath{\unitaryLalgb{n}}
    }

\glsxtrnewsymbol
    [description={Ring of $\unitarygp{n}$ adjoint action invariant $\C$-valued polynomials on $\unitaryLalgb{n}$}]
    {invRing}
    {\ensuremath{\C[\unitaryLalgb{n}]^{\unitarygp{n}}}
    }

\glsxtrnewsymbol
    [description={Lie algebroid total Chern class}]
    {cgE}
    {\ensuremath{\cg{\g}(E)}
    }

\glsxtrnewsymbol
    [description={Lie algebroid Chern character}]
    {chgE}
    {\ensuremath{\chg{\g}(E)}
    }

\glsxtrnewsymbol
    [description={Lie algebroid Todd class}]
    {TdgE}
    {\ensuremath{\Tdg{g}(E)}
    }

\glsxtrnewsymbol
    [description={G-translation}]
    {Ug}
    {\ensuremath{U_g}
    }

\glsxtrnewsymbol
    [description={Space of G-invariant differential operators on G}]
    {GinvDiffopOnG}
    {\ensuremath{\mathcal{D}_{\mathrm{inv}}(G;\mathbf{E},\mathbf{F})}
    }

\glsxtrnewsymbol
    [description={Space of restrictions G-invariant differential operator on
$G^{(0)}$}]
    {GinvDiffopOnG0}
    {\ensuremath{\mathcal{U}(\g;E,F)}
    }

\glsxtrnewsymbol
    [description={Principal map}]
    {sigmak}
    {\ensuremath{\sigma_{k}}
    }

\glsxtrnewsymbol
    [description={K-group}]
    {kgp}
    {\ensuremath{K_{0}(\bullet)}
    }

\glsxtrnewsymbol
    [description={Higher index}]
    {higherindex}
    {\ensuremath{\Ind_{v}(D)}
    }

\glsxtrnewsymbol
    [description={van Est map}]
    {vanEstmap}
    {\ensuremath{\Phi_\g}
    }

\glsxtrnewsymbol
    [description={Lie algebroid Thom class}]
    {}
    {\ensuremath{\Th^{\g}(H)}
    }

\glsxtrnewsymbol
    [description={Lie algebroid Euler class}]
    {eug}
    {\ensuremath{\euler{\g}}
    }

\glsxtrnewsymbol
    [description={Local index of an elliptic differential operator}]
    {localindex}
    {\ensuremath{\iota_{D}}
    }

\glsxtrnewsymbol
    [description={Tangential de Rham differential}]
    {tgdiff}
    {\ensuremath{\tgdiff}
    }

\glsxtrnewsymbol
    [description={Tangential anti-holomorphic differential}]
    {tgP}
    {\ensuremath{\tgpartialbar}
    }

\glsxtrnewsymbol
    [description={Space of k-degree E-valued tangential de Rham differential forms}]
    {tgAk}
    {\ensuremath{\tgAk{k}(M,E)}
    }

\glsxtrnewsymbol
    [description={Space of $(p,q)$-type E-valued tangential differential forms}]
    {tgApq}
    {\ensuremath{\tgApq{p}{q}(M,E)}
    }

\glsxtrnewsymbol
    [description={Tangential de Rham cohomology}]
    {tgdohomo}
    {\ensuremath{\tgcohomo{\bullet}(M,E)}
    }

\glsxtrnewsymbol
    [description={Tangential connection}]
    {tgconn}
    {\ensuremath{\tgconn{}{}}
    }

\glsxtrnewsymbol
    [description={Tangential Chern class}]
    {tgchernclass}
    {\ensuremath{\tgchernclass{E}}
    }

\glsxtrnewsymbol
    [description={Tangential Chern character}]
    {tgchernch}
    {\ensuremath{\tgcherncharact{E}}
    }

\glsxtrnewsymbol
    [description={Tangential Todd class}]
    {tgtd}
    {\ensuremath{\tgTdclass{E}}
    }

\glsxtrnewsymbol
    [description={Tangential Euler class}]
    {tgeuler}
    {\ensuremath{\tgEulerclass{E}}
    }

\glsxtrnewsymbol
    [description={Holonomy Lie groupoid}]
    {HF}
    {\ensuremath{H_{\mathscr{F}}}
    }

\glsxtrnewsymbol
    [description={Holonomy Lie algebroid}]
    {gF}
    {\ensuremath{\g_{\mathscr{F}}}
    }

\glsxtrnewsymbol
    [description={Tangential K\"{a}hler form}]
    {tgomega}
    {\ensuremath{\tgomega}
    }

\glsxtrnewsymbol
    [description={Average Euler character}]
    {chiM}
    {\ensuremath{\chi(M)}
    }

\glsxtrnewsymbol
    [description={b-tangent bundle}]
    {btgbd}
    {\ensuremath{\btgbd{M}}
    }

\glsxtrnewsymbol
    [description={b-de Rham differential}]
    {bdeff}
    {\ensuremath{\bDiff}
    }

\glsxtrnewsymbol
    [description={b-holomorphic differential and anti-holomorphic differential}]
    {bPb}
    {\ensuremath{\bPartial\text{ and }\bPartialbar}
    }

\glsxtrnewsymbol
    [description={Space of k-degree b-differential forms}]
    {bAK}
    {\ensuremath{\bAk{k}(M)}
    }

\glsxtrnewsymbol
    [description={Space of $(p,q)$-type b-differential forms}]
    {bApq}
    {\ensuremath{\bApq{p}{q}(M)}
    }

\glsxtrnewsymbol
    [description={b-connection}]
    {bconn}
    {\ensuremath{\bConn{}}
    }

\glsxtrnewsymbol
    [description={$\bPartialbar$-Laplace}]
    {bLaplace}
    {\ensuremath{\Delta_{\bPartialbar}}
    }

\title{Kodaira vanishing theorems for K\"{a}hler Lie algebroids}

\author{ \href{https://orcid.org/0000-0000-0000-0000}{
        \includegraphics[scale=0.06]{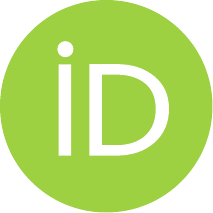}\hspace{1mm}Tengzhou~Hu}
        \thanks{} \\ 
	Department of mathematics\\
	Washington university in St. Louis\\
	Missouri, MO 63130 \\
	\texttt{thu24@wustl.edu} \\
}

\renewcommand{\shorttitle}{Kodaira vanishing theorems for Kaihler Lie algebroids}

\hypersetup{
pdftitle={Kodaira vanishing theorems for Kaihler Lie algebroids},
pdfauthor={Tengzhou~Hu},
pdfkeywords={First keyword, Second keyword, More},
}

\begin{document}
\maketitle

\begin{abstract}
    A Lie algebroid is a generalization of Lie algebra that provides a general framework to describe the symmetries of a manifold. In this paper, we generalize the Kodaira vanishing theorem, which is a basic result in complex geometry, to K\"{a}hler Lie algebroids. The generalization of the Kodaira vanishing theorem states that the kernel of the Lie algebroid Laplace operator on Lie algebroid positive line bundle-valued $(p,q)$-forms vanishes when $p+q$ is sufficiently large. The most difficult part of the proof of the generalized Kodaira vanishing theorem is the generalization of the K\"{a}hler identities to Lie algebroids. In this paper, we provide an approach by using local coordinate calculation.
\end{abstract}



\section{Introduction}
The notion of groupoid, first introduced by Brandt in 1926, refers to a small category in which every morphism is invertible. A groupoid is very much like a group except that the product is only partially defined. As a groupoid can be thought of as a "many-object generalization" of a group, a similar object such as a many-object generalization of a Lie group is called a Lie groupoid. Lie groupoids were introduced by Charles Ehresmann as "differentiable groupoids" \cite{E1959,E1963} in 1959. A Lie groupoid can also be understood as a groupoid equipped with a smooth differential structure.

Recall that Lie algebra is associated with a Lie group as the associated infinitesimal object. Similarly, the infinitesimal object associated with a Lie groupoid is called a Lie algebroid. This concept was first introduced by Jean Pradines \cite{P1963} as the infinitesimal counterpart of Ehresmann's differentiable groupoids in 1967. It was also independently introduced by Lichnerowicz in the study of Poisson manifolds \cite{L1977} in 1977. In the later decades, Lie algebroids appeared in the study of many geometric areas. In Poisson geometry, Weinstein, Coste, and Dazord proved that the Lie algebroid of a symplectic groupoid is the cotangent bundle of the base manifold equipped with a bracket of 1-forms \cite{CDW1987} in 1987. This demonstrated that there is a Lie algebroid structure on the cotangent bundle for any Poisson manifold. The books written by Mackenzie \cite{M1987,M2005} are a good reference for Lie algebroids in the theory of group actions and connection theory. The study of Lie algebroids in foliation theory can be seen in \cite{M2005-2}.

After decades of development, a Lie algebroid has been defined as a triple consisting of a vector bundle $\g\to M$, a Lie bracket $[\,,\,]$ on the space of sections, and a morphism of the vector bundle $\rho:\g\to\TM{}$ preserving the Lie bracket. A natural example of Lie algebroids is $\g=\TM{}$. Thus, we can view a Lie algebroid as a generalization of a tangent bundle. Then, it is natural to generalize notations including differentials, differential forms, connections, and curvatures to Lie algebroids through the same algebraic formal definitions. Analogous to a complex manifold, Chemla defined a complex Lie algebroid on the sheaf of a local section \cite{C1999} in 1999. Instead of using the language of sheaf theory to define complex Lie algebroids, we use the integrable almost complex structure from Weinstein's definition of a complex Lie algebroid \cite{CW1999}.

In this paper, the main object we focus on is the K\"{a}hler Lie algebroid. The concept of a K\"{a}hler manifold was introduced by K\"{a}hler in 1933, where he defined a complex manifold that carries a Hermitian metric which is locally given by $\sqrt{-1}\partial\,\partialbar \phi$ for a pluri-subharmonic function $\phi$. This original definition is close to the current definition of a K\"{a}hler manifold. Through the viewpoint of complex geometry, the K\"{a}hler manifold is a complex manifold with a Hermitian metric whose associated 2-form is closed. This 2-form is called the K\"{a}hler form. Its cohomology class is topological invariant, meaning that it is independent of the metric. Later, people found that this class also represents the first Chern class of the manifold. An important basic result associated with a K\"{a}hler form is the K\"{a}hler identities, which are a collection of identities between operators on a K\"{a}hler manifold that relate $\partialbar$ to their adjoints, contractions, and wedge operators of the K\"{a}hler form. It was first proven by Hodge and appeared in his book on harmonic integrals in 1941. The K\"{a}hler identities, combined with Hodge theory, can produce many results in algebraic geometry. Based on the K\"{a}hler identities, the Bochner–Kodaira–Nakano identity shows the relation between the antiholomorphic Laplacian of a vector bundle over a K\"{a}hler manifold with the K\"{a}hler form and the torsion of the Hermitian metric. Following the Bochner–Kodaira–Nakano identity and Hodge theory, Kodaira proved the famous Kodaira vanishing theorem, which is a basic result in complex geometry and algebraic geometry that describes the cohomology group of positive line bundles over a K\"{a}hler manifold when its degree is sufficiently large.

K\"{a}hler Lie algebroids have not been completely studied. Some mathematicians studied the para-K\"{a}hler structures on Lie algebroids \cite{LTW2007,BB2016,PNM2018}. Benayadi and Boucettaa studied exact para-K\"{a}hler Lie algebroids as a generalization of exact para-K\"{a}hler Lie algebras \cite{BB2016}, which leads to a natural generalization of pseudo-Hessian manifolds. Leichtnam, Tang, and Weinstein introduced the notion of a para-K\"{a}hler structure on a complex Lie algebroid in their study of Poisson geometry and deformation quantization near a strictly pseudoconvex boundary \cite{LTW2007}. Peyghan, Nourmohammadifar, and Makhlouf introduced and studied para-K\"{a}hler hom-Lie algebroids \cite{PNM2018}. Some mathematicians also tried to generalize the Kodaira vanishing theorem on Lie algebroids. In \cite{PR2019,PR2020}, Pirbodaghi and Rezaii defined and studied Hermitian and K\"{a}hler structures on transitive Lie algebroids, which are Lie algebroids whose anchor maps are surjective. They defined the integration of top degree Lie algebroid by introducing an "inner operator" which maps Lie algebroid forms to differential forms. They then generalized the notations from complex manifolds to complex Lie algebroids. Using the generalized Lefschetz decomposition to avoid local coordinate calculations, they obtained generalized Hodge identities on transitive Lie algebroids. At the end of their paper, they presented a generalized Kodaira vanishing theorem on transitive Lie algebroids. The theorem states that if a line bundle is Lie algebroid positive (they did not define positive in their thesis, but the statement is a generalization of a positive line bundle as it requires the first Lie algebroid Chern class to be "positive"), then the space of Lie algebroid harmonic $(p,q)$-forms is trivial when $p+q>n$, where $n$ is the complex dimension of the transitive Lie algebroid.

In this paper, we study and give a proof of the generalization of K\"{a}hler identities on Lie algebroids without any restriction on anchor maps. The proof in our work is based on a local unitary frame. However, since Hodge theory on Lie algebroids is not fully developed, our generalization of the Kodaira vanishing theorem is not as strong as the classical Kodaira vanishing theorem. The vanishing object in our theorem is the kernel of the Laplace operator. The theorem is stated as follows.
\begin{theorem*}[Lie algebroid Kodaira–Nakano vanishing theorem]
    Let $(M,\g,\omega)$ be a K\"{a}hler Lie algebroid and $E\to M$ be a positive line bundle. The kernel of 
    \begin{align*}
        \LieagbdLpls{}{\LieagbdPartialbar}
        :\Apq{p}{q}(M,E)\to\Apq{p}{q}(M,E)
    \end{align*}
    vanishes when $p+q>n$. The same result holds on a negative line bundle under the condition that $p+q<n$.
\end{theorem*}

This paper is organized as follows. 
\begin{itemize}
    \item The main topic of Section \ref{sec:preliminaries} is K\"{a}hler Lie algebroids. After giving the definition of K\"{a}hler Lie algebroids, we generalize notations on manifolds, like differential, connection and formal adjoint, to Lie algebroid. Then, we will define the sign of permutation together with some basic formulas which will be repeatedly used in the calculations of the Lie algebroid K\"{a}hler identities. The rest of the section presents the Lie bracket and the Lie algebroid differential in terms of a local frame, thereby generating four crucial coefficients. The condition of the closedness of the Lie algebroid K\"{a}hler form produces two constraints on these coefficients, which play an essential role in the proof of the Lie algebroid K\"{a}hler identities.
    \item  Section \ref{sec:main section} aims to present a generalized Kodaira vanishing theorem for Lie algebroids. It covers the Lie algebroid K\"{a}hler identities, the Lie algebroid Bochner–Kodaira–Nakano identity and the Lie algebroid Kodaira vanishing theorem. Since the proof of the Lie algebroid K\"{a}hler identities is very long, we put the corresponding calculations into the appendices. At the end of this section, we introduce a K\"{a}hler b-manifold as an example of a K\"{a}hler Lie algebroid. We explain the equivalence of the notations in the b-category and the Lie algebroid category. We then rewrite the vanishing theorem under the terminologies in the b-category. Finally, we show an example of a Kähler b-manifold.

\end{itemize}

\section{Preliminaries}
\label{sec:preliminaries}
\subsection{K\"{a}hler Lie algebroid}
\label{sec: K\"{a}hler Lie algebroid}

To discuss the complex structure of a Lie algebroid, we mainly follow the definition of a complex Lie algebroid by Weinstein in \cite{W2006} and associate this definition with the notion of almost complex Lie algebroids over almost complex manifolds \cite{IP2016}. The almost complex structure of a Lie algebroid is effectively the natural extension of an almost complex manifold when defining an endomorphism $J$ of the Lie algebroid. This is integrable if its Nijenhuis tensor of $J$ is zero.

\begin{definition}
 A \keyword{complex Lie algebroid} $\g\to M$ is a complex vector bundle over a smooth manifold $M$ with the following structures:
\begin{enumerate}
    \item A \keyword{Lie bracket} $
    [\bullet,\bullet]:
    \secsp{M}{\g}\times\secsp{M}{\g}\to\secsp{M}{\g}
    $, where $\secsp{M}{\g}$ is the space of sections of $\g$. We can naturally extend the Lie bracket to sections of $\g^{\C}$ by linearity.
    
    \item A vector bundle morphism $\rho: \g \to \TMC{}$ which is called an \keyword{anchor map}.
    
    \item \keyword{Leibniz rule} $[v_1,f v_2]=f[v_1,v_2]+ (\rho(v_1) f)v_2$ for all $v_1,v_2 \in \secsp{M}{\g}$ and $f \in \smoothfuncsp{M}$, where $\rho(v_1)f$ is the action of a vector field on a smooth function.
    
    \item The almost complex structure $J:\g \to \g$ from the complex vector bundle structure is integrable. That is saying the Nijenhuis tensor 
    \begin{align*}
        \mathcal{N}(\cdot,\cdot):
        \secsp{M}{\gC}\times\secsp{M}{\gC}
        \to
        \C
    \end{align*}
    such that
    $
        \mathcal{N}(\basev{1},\basev{2}):=
        [J\,\basev{1},J\,\basev{2}]
        -J[J\,\basev{1},\basev{2}]
        -J[\basev{1},J\,\basev{2}]
        -[\basev{1},\basev{2}]
    $ vanishes.
\end{enumerate}
\end{definition}

We use $\g$ to denote a complex Lie algebroid. Lie algebroids in this paper are all complex Lie algebroids unless specified otherwise. Analogous to the decomposition of the complexified tangent bundle $\TMC{}$ on a complex manifold, we decompose the complexified Lie algebroid into 
\begin{align*}
    \gC:=\g\otimes_{\R}\C=\gpq{1}{0}\oplus\gpq{0}{1},
\end{align*}
where $\gpq{1}{0}$ is the eigenspace of $J$ for the eigenvalue $\sqrt{-1}$ and $\gpq{0}{1}$ is the eigenspace of $J$ for the eigenvalue $-\sqrt{-1}$.

Recall that the K\"{a}hler manifold is a complex manifold with a Hermitian metric and the corresponding K\"{a}hler form. The smoothly varying positive-definite Hermitian form on each fiber determined by the Hermitian metric is called the \keyword{canonical $(1,1)$-form}, which is called the K\"{a}hler form if it is closed. For the generalization of the K\"{a}hler manifold to Lie algebroids, let us observe a complex Lie algebroid $\g$ with a Hermitian metric
\begin{align*}
    h:\g^{1,0}\otimes\g^{0,1}\to \C
\end{align*}
which determines a smoothly varying positive-definite Hermitian Lie algebroid form on each fiber. We call this form \keyword{the canonical $(1,1)$-Lie algebroid form} (again, corresponding to the Hermitian metric).

\begin{definition}
    A \keyword{K\"{a}hler Lie algebroid} $(\g,M,\rho)$ is a complex Lie algebroid carrying a Hermitian metric on $\g$ whose associated canonical $(1,1)$-Lie algebroid form $\omega\in\Apq{1}{1}(M)$ is closed. We call $\omega$ a \keyword{K\"{a}hler Lie algebroid form}.
\end{definition}

On an open chart $U$ of $M$, the unitary frame $\{\basexi{i}\}$ spans $\gstarpq{1}{0}|_{U}$ with respect to the Hermitian metric, where the length of this frame is $\norm{\basexi{i}}{h}=2$. The following frames for $\gstarpq{0}{1}|_{U}$, $\gpq{1}{0}|_{U}$ and $\gpq{0}{1}|_{U}$ are induced by $\{\basexi{i}\}$.
\begin{itemize}
    \item $\{\basedualxi{i}\}$ spanning $\gstarpq{0}{1}|_{U}$, which is conjugate to $\{\basexi{i}\}$.
    \item $\{\basev{i}\} \in\secsp{M}{\gpq{1}{0}|_{U}}$ is the dual of $\{\basexi{i}\}$, which spans $\gpq{1}{0}|_{U}$.
    \item $\{\basedualv{i}\} \in \secsp{M}{\gpq{0}{1}|_{U}}$ is the dual of $\{\basedualxi{i}\}$, which spans $\gpq{0}{1}|_{U}$.
\end{itemize}

The expression of the K\"{a}hler Lie algebroid form on an open chart $U$ is
\begin{align*}
    \omega|_{U}=\frac{\sqrt{-1}}{2}\sum \basexi{i}\wedge\basedualxi{i},
\end{align*}
where $\{\basexi{i}\}$ and $\{\basedualxi{i}\}$ are the frames of $\gstarpq{1}{0}|_{U}$ and $\gstarpq{0}{1}|_{U}$, respectively. This expression is formally the same as the standard K\"{a}hler form on $\C^{n}$ when replacing $\basexi{i}$ with $\diff z_{i}$ and $\basedualxi{i}$ with $\diff \Bar{z}_{i}$.

Some notations on differential manifolds can be generalized to K\"{a}hler Lie algebroid. We define the notations of the exterior algebra $\g^k:=\bigwedge^{k}\gC$ and $\gpq{p}{q}:=\bigwedge^{p}\gpq{1}{0} \otimes \bigwedge^{q}\gpq{0}{1}$. The natural properties of $\otimes$ and $\oplus$ in the exterior algebra give the equality $\gk{k}=\bigoplus_{p+q=k}\gpq{p}{q}$. A similar construction can be made on the dual of a complex Lie algebroid $\gstar$ as long as this gives a complex structure on $\gstar$ through the pullback of $J$. By using the decompositions $\gC=\gpq{1}{0}\oplus\gpq{0}{1}$ and $\gstarC=\gstarpq{1}{0}\oplus\gstarpq{0}{1}$, we define the conjugate map
\begin{align*}
    \bar{}:\gC  &\to\gC
    \\
    v+\sqrt{-1}J(v)&\mapsto v-\sqrt{-1}J(v)
    \\
    v-\sqrt{-1}J(v)&\mapsto v+\sqrt{-1}J(v)
\end{align*}
for $v\in\g$, and
\begin{align*}
    \bar{}:\gstarC  &\to\gstarC
    \\
    \xi+\sqrt{-1}J(\xi)&\mapsto \xi-\sqrt{-1}J(\xi)
    \\
    \xi-\sqrt{-1}J(\xi)&\mapsto \xi+\sqrt{-1}J(\xi)
\end{align*}
for $\xi\in\gstar$. The definition implies $\bar{v}=v$ for any $v\in\g\subset\gC$ and $\bar{\xi}=\xi$ for any $\xi\in\gstar\subset\gstarC$.

Let $\Ak{k}(M):=\secsp{M}{{\gstar}^{k}}$ be the space of \keyword{Lie algebroid $k$-forms}. Then, the Lie algebroid differential is defined as the generalization of the de Rham differential.
\begin{definition}
The Lie algebroid differential $\LieagbdDiff:\Ak{k}(M) \to \Ak{k+1}(M) $ is defined as
\begin{align*}
    \LieagbdDiff \phi(v_1,\cdots,v_{k+1})
    =
    \begin{pmatrix*}[l]
        \displaystyle\sum_{i=1}^{k+1}
        (-1)^{i}
        \rho(v_i) (\phi(v_1,\cdots,\widehat{v_i},\cdots,v_{k+1}))
        \\
        +
        \\
        \displaystyle\sum_{i<j}
        (-1)^{i+j-1}
        \phi([v_i,v_j],v_{1},\cdots,\widehat{v_i},\cdots,\widehat{v_j},\cdots,v_{k+1})
    \end{pmatrix*}.
\end{align*}
\end{definition}

The definition of $\LieagbdDiff$ is formally same as the definition of the classical de Rham differential, so its application to the exterior product of local frames is formally the same as the case for the classical de Rham. That is

\begin{align*}
    \LieagbdDiff 
    (f\,\zeta_{\indexi{1}}\wedge\cdots\wedge\zeta_{\indexi{k}})
    =
    (\LieagbdDiff f)
    \wedge\zeta_{\indexi{1}}\wedge\cdots\wedge\zeta_{\indexi{k}}
    +
    \displaystyle\sum_{j=1}^{k}
    (-1)^{j}
    f\, \zeta_{\indexi{1}}
    \wedge\cdots\wedge
    (\LieagbdDiff \zeta_{\indexi{j}})
    \wedge\cdots\wedge
    \zeta_{\indexi{k}}
    ,
\end{align*}
where $\{\zeta_{\indexi{}}\}$ is a local frame, and $f$ is a smooth function on $U$. Its definition can also naturally be generalized and applied to the Lie algebroid $(p,q)$-forms $\Apq{p}{q}(M):=\secsp{M}{\gstarpq{p}{q}}$. To achieve the definition of $\LieagbdPartial$ and $\LieagbdPartialbar$ on any degree of $(p,q)$ Lie algebroid differential form, we can check the image of 
\begin{align*}
    \LieagbdDiff:
    \Apq{1}{0}(M)\to\Ak{2}(M)
    =
    \Apq{2}{0}(M)\oplus\Apq{1}{1}(M)\oplus\Apq{0}{2}(M)
\end{align*}
does not actually contain any components in $\Apq{0}{2}(M)$. This is promised since
\begin{align*}
    [\basev{1},\basev{2}]\subset\secsp{M}{\g^{1,0}}
    \qquad\text{and}\qquad
    [\basedualv{1},\basedualv{2}]\subset\secsp{M}{\g^{0,1}}
\end{align*}
for any  $\basev{1},\basev{2}\in\secsp{M}{\g^{1,0}}$ and $\basedualv{1},\basedualv{2}\in\secsp{M}{\g^{0,1}}$ since the almost complex structure J is integrable. Then, using the natural projections
\begin{align*}
    \pi_1:
    \Apq{2}{0}(M)\oplus\Apq{0}{1}(M)
    \to
    \Apq{2}{0}(M)
\qquad\text{and}\qquad
    \pi_2:
    \Apq{2}{0}(M)\oplus\Apq{0}{1}(M)
    \to
    \Apq{0}{1}(M)
    ,
\end{align*}
we define $\LieagbdPartial$ and $\LieagbdPartialbar$ on $\Apq{1}{0}(M)$ by setting
\begin{align*}
    \LieagbdPartial:=
    \pi_{1}\, \LieagbdDiff
    :\Apq{1}{0}(M) \to \Apq{2}{0}(M)
\qquad\text{and}\qquad
    \LieagbdPartialbar:=
    \pi_{2}\, \LieagbdDiff
    :\Apq{1}{0}(M) \to \Apq{1}{1}(M)
    .
\end{align*}
Therefore, we can conclude the decomposition of $\LieagbdDiff$ on $\Apq{1}{0}(M)$:
\begin{align*}
    \LieagbdDiff=
    \LieagbdPartial+\LieagbdPartialbar:
    \Apq{1}{0}(M)
    \to
    \Apq{2}{0}(M)\oplus\Apq{1}{1}(M)
    .
\end{align*}
Analogous to the case of $\Apq{1}{1}(M)$, the similar decomposition of the Lie algebroid differential
\begin{align*}
    \LieagbdDiff=
    \LieagbdPartial+\LieagbdPartialbar:
    \Apq{0}{1}(M)
    \to
    \Apq{1}{1}(M)\oplus\Apq{0}{2}(M)
\end{align*}
also holds on $\Apq{0}{1}(M)$. By checking $
\LieagbdDiff 
    (
        f\,
        \basexi{1}\wedge\cdots\wedge\basexi{p}
        \wedge
        \basedualxi{1}\wedge\cdots\wedge\basedualxi{q}
    )
$, the decomposition $\LieagbdDiff=\LieagbdPartial+\LieagbdPartialbar$ can be generalized to
\begin{align*}
    \LieagbdDiff=
    \LieagbdPartial+\LieagbdPartialbar:
    \Apq{p}{q}(M) \to \Apq{p+1}{q}(M)\oplus\Apq{p}{q+1}(M).
\end{align*}
for any $(p,q)$-degree Lie algebroid differential form.

\begin{definition}
    Suppose $(\g,M,\anchor)$ is a Lie algebroid and $E\to M$ is a vector bundle. A \keyword{Lie algebroid connection} on $E$ is a linear operator
    \begin{align*}
        \LieagbdConn{}{}:\Ak{0}(M,E)\to\Ak{1}(M,E)
    \end{align*}
satisfying the Leibniz rule
\begin{align*}
    \LieagbdConn{}{v}(f\,s)=f\,\LieagbdConn{}{v}(s)+f_{*}\big(\anchor(v)\big)(s),
\end{align*}
where  $v\in\secsp{M}{\g},f\in C^{\infty}(M),s\in\secsp{M}{E}$.
\end{definition}

Analogous to the classical connection construction, we can generalize the Lie algebroid connection to the higher order:
\begin{align*}
    \LieagbdConn{}{}:\Ak{k}(M,E)\to\Ak{k+1}(M,E),
\end{align*}
and define its corresponding \keyword{Lie algebroid curvature tensor} as 
\begin{align*}
    R(\LieagbdConn{}{}):=\LieagbdConn{}{}^2\in\Ak{2}(M,\End(E)).
\end{align*}

In the classical case of integration on a manifold, a non-vanishing section of $\wedge^{top}\TstarM{}$ gives a volume form. However, in the case of a Lie algebroid, any section of $\wedge^{top}\gstar$ does not automatically generate a volume form directly. The solution is reached by introducing \keyword{transversal densities}, which are sections of $L_{\g}:=\bigwedge^{top}\TstarM{} \otimes \bigwedge^{top}\g$.

\begin{definition}
Suppose $\g\to M$ is the Lie algebroid associated with the Lie groupoid $G\rightrightarrows M$. We say the Lie groupoid $G$ is unimodular if there exists a global nonvanishing section $\Omega$ of $L_{\g}$.
\end{definition}
For any $s\in\secsp{M}{\wedge^{top}\gstar}$, its integral over M is 
\begin{align*}
    \int_{M}
    \langle s,\Omega \rangle,
\end{align*}
where $\langle \cdot,\cdot\rangle$ pairs the $\wedge^{top}\g$ and $\wedge^{top}\gstar$ components and produces $\innerprod{s}{\Omega}{}$, a volume form on M.


Recall that the top degree wedge product of a K\"{a}hler form is a canonical volume form. Following this idea, the top degree wedge product of the Lie algebroid K\"{a}hler form $\omega^n\in\secsp{M}{\bigwedge^{top}\gstar}$ gives rise to a Lie algebroid volume form since its combination with the transversal density $\Omega\in \secsp{M}{\bigwedge^{top}\TstarM{} \otimes \bigwedge^{top}\g}$ produces an actual volume form:
\begin{align*}
    \innerprod{\omega^n}{\Omega}{}\in\secsp{M}{\bigwedge^{top}\TstarM{}}.
\end{align*}
We use this canonical Lie algebroid volume form to define the inner product
\begin{align*}
    \pinnerprod{\basexi}{\phi}{\Ak{k}}
    :=
    \int_{M}
    \innerprod
        {\pinnerprod{\basexi}{\phi}{h}\frac{\omega^n}{n!}}
        {\Omega}
        {}
\end{align*}
for $\basexi,\phi \in \Ak{k}(M)$, where $\pinnerprod{}{}{h}$ is the Hermitian metric on $\wedge^{k}\gstar$. The inner product leads to the introduction of an adjoint operator. The formal adjoint of a Lie algebroid differential $\LieagbdDiff$ is an operator
\begin{align*}
    \LieagbdDiff^{*}:\Ak{k+1}(M)\to\Ak{k}(M)
\end{align*}
satisfying $\pinnerprod{\LieagbdDiff^{*}\basexi}{\phi}{\Ak{k}}=\pinnerprod{\basexi}{\LieagbdDiff\phi}{\Ak{k}}$. Similarly, $\LieagbdPartial^{*}$ and $\LieagbdPartialbar^{*}$ are the formal adjoints of $\LieagbdPartial$ and $\LieagbdPartialbar$. Let $\ext{\phi}(\cdot):=\phi \wedge(\cdot)$ be the exterior product operator. The contraction operator is $\contract{\phi}=\ext{\phi}^*$.

The Hodge star operator on Lie algebroid differential forms plays a critical role in the proof of the K\"{a}hler identities lemma. Our definition is slightly different from the standard reference \cite{GH1994}. This leads to some sign differences in the properties of the Hodge star but does not change the main results of the theory.
\begin{definition}
The Hodge star $\star$ is an operator mapping Lie algebroid $(p,q)$-forms to Lie algebroid $(n-q,n-p)$-forms, 
\begin{align*}
    \star:
        \mathscr{A}^{(p,q)}(M)
        &\to
        \mathscr{A}^{(n-q,n-p)}(M),
\end{align*}
such that
\begin{align*}
    \pinnerprod{\alpha}{\beta}{\Apq{p}{q}}
    =
    \int_M
    \innerprod{\alpha \wedge \overline{\star\beta}}{\Omega}{}.
\end{align*}
\end{definition}
In other words, we require $\alpha \wedge \overline{\star\beta}=\pinnerprod{\alpha}{\beta}{h}\frac{\omega^{n}}{n!}$, which is the exact same as the ordinary definition of the Hodge star on a complex manifold if we omit the extra $\Omega$ in the integration. This leads to the local expression of the Hodge star $\star$ under the frames $\{\basexi{}\}$ and $\{\basedualxi{}\}$, which is formally the same as the classical expression:
\begin{align*}
    \star:
    f\,
    \basexi{\indexI}\wedge\basedualxi{\indexJp}
    \mapsto
    2^{p+q-n}
    f\,
    \sgn{1,\cdots,n,1^{'},\cdots,n^{'}}{\indexI,\indexJp,\indexIc,\indexJpc}
    \basedualxi{\indexIc}\wedge
    \basexi{\indexJpc},
\end{align*}
where $\indexI=\{\indexi{1},\dots,\indexi{p}\}$ and $\indexJp=\{\indexjp{1},\dots,\indexjp{q}\}$ are the index sets, $\indexIc$ and $\indexJpc$ are their complements in $\{1,\cdots,n\}$ and $\{1^{'},\cdots,n^{'}\}$, and the sign function gives the sign of permutation from $(1,\cdots,n,1^{'},\cdots,n^{'})$ to $(\indexI,\indexJp,\indexIc,\indexJpc)$.

In the remaining of this section, we state some properties of the Hodge star and the Lie algebroid differential, which are similar as the cases on differential manifolds.


\begin{proposition}
\label{prop: a list of prop for Hodge star and Lie algebroid diff}
The following equalities for the Hodge star and the Lie algebroid differential hold:
\begin{enumerate}
    \item $
            \overline
                {\star(f\,\basexi{\indexI}\wedge\basedualxi{\indexJp})}
            =
            \star
                (\overline{f\,\basexi{\indexI}\wedge\basedualxi{\indexJp}}).
        $

    \item $
            \overline
                {
                \LieagbdDiff
                (f\,\basexi{\indexI}\wedge\basedualxi{\indexJp})
                }
            =
            \LieagbdDiff
            \overline
                {
                (f\,\basexi{\indexI}\wedge\basedualxi{\indexJp})
                }
            .
        $

    \item $\star^{2}=(-1)^{n^2}(-1)^{p+q}\,\id:\Apq{p}{q}(M)\to \Apq{p}{q}(M)$.

    \item (Stoke's theorem on Lie algebroids)$\int_{M}
    \innerprod{\LieagbdDiff\phi}{\Omega}{}
    =0$.

    \item $-(-1)^{n^2}\star\,\LieagbdDiff\,\star$. This leads $\LieagbdPartial^{*}=-(-1)^{n^2}\star\,\LieagbdPartialbar\,\star$ and $\LieagbdPartialbar^{*}=-(-1)^{n^2}\star\,\LieagbdPartial\,\star$.
    
\end{enumerate}
\end{proposition}
\begin{proof}
    See Section \ref{sec:proofs of some short properties}.
\end{proof}

\subsection{The Lie bracket and the Lie algebroid differential in terms of a local frame}
We will express the Lie bracket and the Lie algebroid differential by the frames $\{\basexi{i}\}$ and $\{\basedualxi{i}\}$ in this subsection. Their expression are useful in the proof of the Lie algebroid Hodge identities through local charts.

Firstly, We introduce the function giving the sign of a permutation and the function which sorts the array by size. 
\begin{definition}
    Suppose $\indexI=(\indexi{1},\cdots,\indexi{n})$ is an array consisting of a sequence of numbers and $\tilde{\indexI}=(\tilde{i}_{1},\cdots,\tilde{i}_n)$ is an array obtained by permuting $\indexI$. We then define
    \begin{align*}
        \sgn
            {\indexi{1},\cdots,\indexi{n}}
            {\tilde{i}_{1},\cdots,\tilde{i}_{n}}
        :=
        \text{the sign of the permutation}.
    \end{align*}
\end{definition}

We introduce some propositions for the sign of a permutation which are frequently used in calculations.

\begin{proposition}
    \text{}
    \begin{enumerate}
    \item Swapping the two rows in a sign of permutation does not change the sign. For example,
    \begin{align*}
        \sgn
            {\indexi{1},\cdots,\indexi{n}}
            {\tilde{i}_{1},\cdots,\tilde{i}_{n}}
        =
        \sgn
            {\tilde{i}_{1},\cdots,\tilde{i}_{n}}
            {\indexi{1},\cdots,\indexi{n}}
        .
    \end{align*}    
    \item Suppose $\indexJ$ is an array such that $\indexJ\cap(\tilde{i}_{1},\cdots,\tilde{i}_n)=\emptyset$ and $\sgn{\indexi{1},\cdots,\indexi{n}}{\tilde{i}_{1},\cdots,\tilde{i}_{n}}$ is a sign of the permutation. We can insert $\indexJ$ in the same position inside the two rows without changing the sign of the permutation:
    \begin{align*}
        \sgn
            {\indexi{1},\cdots,\indexi{n}}
            {\tilde{i}_{1},\cdots,\tilde{i}_{n}}
        =
        \sgn
            {\indexJ,\indexi{1},\cdots,\indexi{n}}
            {\indexJ,\tilde{i}_{1},\cdots,\tilde{i}_{n}}
        =
        \sgn
            {\indexi{1},\cdots,\indexi{k},\indexJ,\indexi{k+1},\cdots,\indexi{n}}
            {\tilde{i}_{1},\cdots,\tilde{i}_{k}\indexJ,\indexi{k+1},\cdots,\tilde{i}_{n}}
        =
        \sgn
            {\indexi{1},\cdots,\indexi{n},\indexJ}
            {\tilde{i}_{1},\cdots,\tilde{i}_{n},\indexJ}.
    \end{align*}
    \item Let $\indexI=(\indexi{1},\cdots,\indexi{n})$ be an array. Suppose that $\tilde{\indexI}=(\tilde{\indexi{1}},\cdots,\tilde{\indexi{n}})$ and $\hat{\indexI}=(\hat{\indexi{1}},\cdots,\hat{\indexi{n}})$ are two arrays obtained by permuting $\indexI$. The following signs of permutations can be merged into one:
    \begin{align*}
        \sgn
            {\indexi{1},\cdots,\indexi{n}}
            {\tilde{\indexi{1}},\cdots,\tilde{\indexi{n}}}
        \sgn
            {\indexi{1},\cdots,\indexi{n}}
            {\hat{\indexi{1}},\cdots,\hat{\indexi{n}}}
        =
        \sgn
            {\hat{\indexi{1}},\cdots,\hat{\indexi{n}}}
            {\tilde{\indexi{1}},\cdots,\tilde{\indexi{n}}}.
    \end{align*}
    In other words, we can merge two signs of permutations by cancelling the rows whose arrays are the same. Furthermore, if a row of a sign of permutation agrees with part of a row in another sign of permutation, we can merge the two signs of permutations into one. For example,
    \begin{align*}
        \sgn
            {\indexJ,\indexi{1},\cdots,\indexi{n},K}
            {\text{an array by permuting }\indexJ,\indexi{1},\cdots,\indexi{n},K}
        \sgn
            {\indexi{1},\cdots,\indexi{n}}
            {\tilde{\indexi{1}},\cdots,\tilde{\indexi{n}}}
        =
        \sgn
            {\indexJ,\tilde{\indexi{1}},\cdots,\tilde{\indexi{n}},K}
            {\text{an array by permuting }\indexJ,\indexi{1},\cdots,\indexi{n},K}.
    \end{align*}
    We can also split one sign of permutation into two signs of permutations by reversing the process in the above example:
    \begin{align*}
        \sgn
            {\indexJ,\tilde{\indexi{1}},\cdots,\tilde{\indexi{n}},K}
            {\text{an array by permuting }\indexJ,\indexi{1},\cdots,\indexi{n},K}
        =
        \sgn
            {\indexJ,\indexi{1},\cdots,\indexi{n},K}
            {\text{an array by permuting }\indexJ,\indexi{1},\cdots,\indexi{n},K}
        \sgn
            {\indexi{1},\cdots,\indexi{n}}
            {\tilde{\indexi{1}},\cdots,\tilde{\indexi{n}}}.
    \end{align*}
    \item If $\indexI$ and $\indexJ$ are two arrays inside the same row of a sign of permutation, then we multiply $(-1)^{|\indexI||\indexJ|}$ after switching $\indexI$ and $\indexJ$. For example,
    \begin{align*}
        \sgn
            {K,\indexI,\indexJ,L}
            {\text{a permutation of } K,\indexI,\indexJ,L}
        =
        (-1)^{|\indexI||\indexJ|}
        \sgn
            {K,\indexJ,\indexI,L}
            {\text{a permutation of } K,\indexI,\indexJ,L}.
    \end{align*}
    A particular application is that we can simultaneously move two adjacent indices to any position without changing the sign. For example,
    \begin{align*}
        \sgn
            {\indexI,\alpha,\beta,\indexJ}
            {\text{a permutation of }\indexI,\alpha,\beta,\indexJ }
        =
        \sgn
            {\alpha,\beta,\indexI,\indexJ}
            {\text{a permutation of }\indexI,\alpha,\beta,\indexJ }
        =
        \sgn
            {\indexI,\indexJ,\alpha,\beta,}
            {\text{a permutation of }\indexI,\alpha,\beta,\indexJ }.
    \end{align*}
\end{enumerate}
\end{proposition}

Then, we define the notation of an ordered array is as follows.
\begin{definition}
    Suppose $\indexI=(\indexi{1},\cdots,\indexi{n})$ is an array. We define
    \begin{align*}
        \order{\indexi{1},\cdots,\indexi{n}}
        :=
        (\tilde{i}_{1},\cdots,\tilde{i}_{n})
    \end{align*}
    by requiring $\tilde{i}_{1}<\tilde{i}_{2}<\cdots<\tilde{i}_{n-1}<\tilde{i}_{n}$ while $
    \{\indexi{1},\cdots,\indexi{n}\}
    =
    \{\tilde{i}_{1},\cdots,\tilde{i}_{n}\}
    $.
\end{definition}

 Now, let $\basev{i} \in\secsp{M}{\gpq{1}{0}|_{U}}$ be the dual of $\basexi{i}$ and $\basedualv{i} \in \secsp{M}{\gpq{0}{1}|_{U}}$ be the dual of $\basedualxi{i}$. Then, the Lie bracket of $\g$ can be expressed as
\begin{align*}
    [\basev{\indexi{}},\basev{\indexj{}}]
        &=  \sum_{k=1}^{n}
            \coefA{k}{\indexi{},\indexj{}} \basev{k}
    ,
    \\
    [\basev{\indexi{}},\basedualv{\indexjp{}}]
        &=  \sum_{k=1}^{n}
            \coefB{k}{\indexi{} \indexjp{}} \basev{k}
            +
            \coefC{k}{\indexi{} \indexjp{}}\basedualv{k}
    ,
    \\
    [\basedualv{\indexi{}},\basedualv{\indexjp{}}]
        &=  \sum_{k=1}^n
            \coefD{k}{\indexi{} \indexjp{}} \basedualv{k},
\end{align*}
where the coefficients $\coefA{k}{i j}, \coefB{k}{i j}$, and $\coefC{k}{i j}$ are smooth functions on $U$. These coefficients naturally carry constraints. The first one is 
\begin{align}
    \overline{\coefA{k}{i j}}=\coefD{k}{i j}
\label{eq-kahler lie agbd-bar Akij= Dkij}
\end{align}
since
$
    \overline{[\basev{\indexi{}},\basev{\indexj{}}]}
    =
    [\basedualv{\indexi{}},\basedualv{\indexj{}}]
$
. The second one is
\begin{align}
    \overline{\coefB{k}{i j}}=-\coefC{k}{j i}
\label{eq-kahler lie agbd-bar Bkij=-Ckji}
\end{align}
since $
    \overline{[\basev{\indexi{}},\basedualv{\indexj{}}]}
    =
    [\basedualv{\indexi{}},\basev{\indexj{}}]
    =
    -[\basev{\indexj{}},\basedualv{\indexi{}}]
$. The coefficients in the local formula of the Lie bracket are also the coefficients of the Lie algebroid cohomology differential. We compute the differential of the frame $\{\basexi{i}\}$:
\begin{align*}
    \LieagbdDiff \basexi{i}(\basev{j},\basev{k})
    =
    \begin{pmatrix*}[l]
        -\rho(\basev{j}) \big( \basexi{i}(\basev{k}) \big)
        +\rho(\basev{k}) \big( \basexi{i}(\basev{j}) \big)
        \\
        +
        \\
        \basexi{i}([\basev{j},\basev{k}])
    \end{pmatrix*}
    =
    \coefA{i}{j k},
\end{align*}
and similarly
\begin{align*}
    \begin{matrix*}[l]
        &
        \LieagbdDiff \basexi{i}(\basev{j},\basedualv{k})=\coefB{i}{j k},
        &
        \LieagbdDiff \basexi{i}(\basedualv{j},\basedualv{k})=0,
        \\
        \LieagbdDiff \basedualxi{i}(\basev{j},\basev{k})=0,
        &
        \LieagbdDiff \basedualxi{i}(\basev{j},\basedualv{k})=\coefC{i}{j k},
        &
        \LieagbdDiff \basedualxi{i}(\basedualv{j},\basedualv{k})=\coefD{i}{j k} .
    \end{matrix*}
\end{align*}
Thus, we summarize the formulas of $\LieagbdDiff$ on the frames $\{\basexi{i}\}$ and $\{\basedualxi{j}\}$ as follows:
\begin{align}
    \LieagbdDiff \basexi{i}
    =&
        \sum_{j<k}^n
        \coefA{i}{jk}
        \basexi{j} \wedge \basexi{k} 
        +
        \sum_{j,k}^n
        \coefB{i}{jk}
        \basexi{j} \wedge \basedualxi{k}
    ,
\label{eq-kahler Lie agbd-d xi i}
\\
    \LieagbdDiff \basedualxi{}
    =&
        \sum_{j,k}^n
        \coefC{\indexi{}}{j k}
        \xi_j \wedge \basedualxi{k}
        +
        \sum_{j<k}^n
        \coefD{i}{jk}
        \basedualxi{j} \wedge \basedualxi{k}
    ,
\label{eq-kahler Lie agbd-d xi bar i}
\end{align}
and the formulas of $\LieagbdPartial$ and $\LieagbdPartialbar$ on the frames $\{\basexi{i}\}$ and $\{\basedualxi{j}\}$ as follows:
\begin{align*}
    \LieagbdPartial \xi_i
    &
        =
        \sum_{j<k}^n
        \coefA{i}{jk}
        \basexi{j} \wedge \basexi{k} 
        =
        \sum_{j,k=1}^n
        \frac{1}{2}
        \coefA{i}{j k}
        \basexi{j}\wedge\basexi{k}
        =
        \sum_{j,k=1}^n
        \frac{1}{2}
        \sgn
            {\order{j,k}}
            {j,k}
        \coefA{i}{\order{j,k}}
        \basexi{j}\wedge\basexi{k}
    ,
    \\
    \LieagbdPartialbar \xi_i
    &
        =
        \sum_{j,k}^n
        \coefB{i}{jk}
        \basexi{j} \wedge \basedualxi{k}
    ,
    \\
    \LieagbdPartialbar \basedualxi{i}
    &   
        =
        \sum_{j<k}^n
        \coefD{i}{jk}
        \basedualxi{j} \wedge \basedualxi{k}
        =
        \sum_{j,k=1}^n
        \frac{1}{2}
        \coefD{i}{jk}
        \basedualxi{j} \wedge \basedualxi{k}
        =
        \sum_{j,k=1}^n
        \frac{1}{2}
        \sgn
            {\order{j,k}}
            {j,k}
        \coefD{i}{\order{j,k}}
        \basedualxi{j}\wedge\basedualxi{k}
    ,
    \\
    \LieagbdPartial \basedualxi{i}
        &= \sum_{j,k}^n
            \coefC{\indexi{}}{j k}
            \xi_j \wedge \basedualxi{k},
\end{align*}
where the order function is as follows:
\begin{align*}
    \order{j,k}=
    \begin{cases}
        (j,k),          &   j<k,\\
        (k,j),          &   j>k,\\
        (j,k)=(k,j),    &   j=k,
    \end{cases}
\end{align*}
and the sign function is as follows:
\begin{align*}
    \sgn{j,k}{j,k}=1
    \qquad\text{and}\qquad
    \sgn{k,j}{j,k}=-1.
\end{align*}
The condition that $\omega$ is closed applies another constraint on the coefficients $A,B,C$, and $D$ through the local formulas of $\LieagbdPartial$ and $\LieagbdPartialbar$.

\begin{lemma}
\label{lemma-Kahler Lie algebroid-trick for ABCD}
In a local chart, the canonical $(1,1)$-Lie algebroid form is closed if and only if the constraints
\begin{align*}
    \coefB{k}{i j}-\coefB{j}{i k}-\coefD{i}{j k}=0,
    \qquad j<k,
    \qquad\text{and}\qquad
    \coefA{i}{j k}-\coefC{j}{k i}+\coefC{k}{j i}=0,
    \qquad j<k,
\end{align*}
on the coefficients hold.

Applying the order function and the sign of permutation, we obtain the constraints 
\begin{align}
    \coefB{k}{i j}-\coefB{j}{i k}
    =
    \sgn{\order{j,k}}{j,k}\coefD{i}{\order{j,k}}
    \label{EQ-Part 1-(1)-equation for coef B,D}
\end{align}
and
\begin{align}
    \sgn{\order{j,k}}{j,k}\coefA{i}{\order{j,k}}
    =
    \coefC{j}{k i}-\coefC{k}{j i}
    \label{EQ-Part 1-(1)-equation for coef A,C}
\end{align}
without considering the order of j and k.
\end{lemma}

\begin{proof}
The desired identities follow from the following computation of $\LieagbdDiff \omega$:
\begin{align*}
    \LieagbdDiff \omega
    &=  
        \frac{\sqrt{-1}}{2}
        \sum_{i}^{n}
        \LieagbdDiff \basexi{i}\wedge\basedualxi{i}
        -
        \basexi{i}\wedge \LieagbdDiff\basedualxi{i}
    \\
    &=
        \frac{\sqrt{-1}}{2}
        \begin{pmatrix*}[l]
            \displaystyle\sum_{i=1}^{n}
            \displaystyle\sum_{j<k}^{n}
            \coefA{i}{j k}
            \basexi{j}\wedge\basexi{k}\wedge\basedualxi{i}
            \\
            +
            \\
            \displaystyle\sum_{j=1}^{n}
            \displaystyle\sum_{k<i}^{n}
            (\coefB{i}{j k}-\coefB{k}{j i})
            \basexi{j}\wedge\basedualxi{k}\wedge\basedualxi{i}
            \\
            -
            \\
            \displaystyle\sum_{k=1}^{n}
            \displaystyle\sum_{i<j}^{n}
            (\coefC{i}{j k}-\coefC{j}{i k})
            \basexi{i}\wedge\basexi{j}\wedge\basedualxi{k}
            \\
            -
            \\
            \displaystyle\sum_{i=1}^{n}
            \displaystyle\sum_{j<k}^{n}
            \coefD{i}{j k}
            \basexi{i}\wedge\basedualxi{j}\wedge\basedualxi{k}
        \end{pmatrix*}
    \\
    &\downarrow(reindex)
    \\
    &=
        \frac{\sqrt{-1}}{2}
        \displaystyle\sum_{i=1}^{n}
        \displaystyle\sum_{j<k}^{n}
        \begin{pmatrix*}[l]
            (\coefB{k}{i j}-\coefB{j}{i k}-\coefD{i}{j k})
            \basexi{i}\wedge\basedualxi{j}\wedge\basedualxi{k}
            \\
            +
            \\
            (\coefA{i}{j k}-\coefC{j}{k i}+\coefC{k}{j i})
            \basexi{j}\wedge\basexi{k}\wedge\basedualxi{i}
        \end{pmatrix*}.
    \end{align*}
\end{proof}

These two equations are key to the proof of the \keyword{Lie algebroid K\"{a}hler identities} through local chart. This is not unexpected as we see that the closeness of the classical canonical $(1,1)$-form deduces the classical K\"{a}hler identities. Notably, if we use the unitary frame of $\TM{}$ to represent the classical $(1,1)$-form and the Lie bracket, then the corresponding coefficients $A, B, C$, and $D$ are in the exact same form discussed above for a K\"{a}hler manifold. Thus, we claim that equations
\begin{align*}
    \coefB{k}{i j}-\coefB{j}{i k}
    =
    \sgn{\order{j,k}}{j,k}\coefD{i}{\order{j,k}}
\end{align*}
and
\begin{align*}
    \sgn{\order{j,k}}{j,k}\coefA{i}{\order{j,k}}
    =
    \coefC{j}{k i}-\coefC{k}{j i}
\end{align*}
are the sufficient conditions for the Lie algebroid K\"{a}hler identities, which are a generalization of the classical K\"{a}hler identities.


\section{Lie algebroid Kodaira vanishing theorem}
\label{sec:main section}
The statement of the classical Kodaira vanishing theorem is that for a compact K\"{a}hler manifold of complex dimension $n$ with a holomorphic vector bundle, the sheaf cohomology $\cohomo{q}{}(M,\Omega^{r}(E))$ vanishes if $q+r>n$, where $\Omega^{r}(E)$ denotes the sheaf of holomorphic $(n,0)$-forms on $M$ with values on $E$. Our generalized theorem on the Lie algebroid is slightly weaker than the classical version, which states that the kernel of Lie algebroid $\LieagbdPartialbar$-Laplace operator
\begin{align*}
    \LieagbdLpls{\LieagbdPartialbar}{}:
    \Apq{p}{q}(M,E)
    \to
    \Apq{p}{q}(M,E)
\end{align*}
vanishes if $p+q>n$, where $\g\to M$ is a K\"{a}hler Lie algebroid and $E\to M$ is a K\"{a}hler holomorphic vector bundle. We present this theorem in Section \ref{section-Lie algebroid Kodaira vanishing theorem} as the main result of this paper. Many ingredients are required to prove the vanishing theorem. In Section \ref{sec:Lie algebroid Kahler identities}, we give the statement of the Lie algebroid K\"{a}hler identities, which is the key to the proof of the Lie algebroid Bochner–Kodaira–Nakano identity. The proofs of these identities use several lemmas. Some short lemmas are directly given in Section \ref{sec:Lie algebroid Kahler identities}, while those that involve long calculations are given in Section \ref{sec:very long proof for Lie algebroid Kahler identity} and \ref{sec:extra long proof for lemmas of Lie algebroid Kahler identity}. In Section \ref{section-Lie algebroid Bochner–Kodaira-Nakano identity}, we define the Lie algebroid Chern connection and its associated Lie algebroid Laplacians. We then use the Lie algebroid K\"{a}hler identities to prove the Lie algebroid Bochner–Kodaira–Nakano identity, which is the key to the proof of the vanishing theorem. In Section \ref{section-Lie algebroid Kodaira vanishing theorem}, we state the generalized vanishing theorem. The Section \ref{sec: app of b manifold} discuss the vanishing theorem on b-manifold, which has a canonical Lie algebroid structure. We define K\"{a}hler b-manifold and give an example showing the existence of K\"{a}hler Lie algebroid.
\subsection{Lie algebroid K\"{a}hler identities}
\label{sec:Lie algebroid Kahler identities}
The classical K\"{a}hler identities on a K\"{a}hler manifold are
\begin{align*}
    [\Lambda,\partial]&=\sqrt{-1}{\partialbar}^{*},
    \\
    [\Lambda,\partialbar]&=\sqrt{-1}{\partial}^{*},
    \\
    [{\partial}^*,L]&=\sqrt{-1}\partialbar,
    \\
    [{\partialbar}^*,L]&=\sqrt{-1}\partial,
\end{align*}
where $L$ and $\Lambda$ are the exterior and interior multiplications on the classical K\"{a}hler form. For the proof of the classical K\"{a}hler identities, it is enough to prove $[\partialbar,\iotaW]=\sqrt{-1}\partial^{*}$ as the remaining three identities automatically hold by taking conjugation and dualization. More details can be found in \cite{M2007,GH1994}. Our aim is to generalize these identities to the K\"{a}hler Lie algebroid. The strategies for the proof of the Lie algebroid K\"{a}hler identities are the same, and we only need to check that one equation among these four holds for the Lie algebroid versions of operators, like $\LieagbdPartialbar$, $\iotaW$, and $\LieagbdPartial^{*}$. In the proof of the classical K\"{a}hler identities on the K\"{a}hler manifold, the local expressions of operators such as $\partial$ and $\iotaW$ are all written in terms of the frame $\{dz_{i},d\Bar{z}_{i}\}$, whose advantages are $\diff^{2}z_{i}=0$ and $\diff^{2}\Bar{z}_{i}=0$. Unfortunately, finding the frame $\{\basexi{i}\}$ on a Lie algebroid such that $\LieagbdDiff^2\basexi{}=0$ may not be possible unless we add extra conditions (e.g., the anchor map is injective) that restrict the generality of our theorem. Therefore, we apply the left-hand side and the right-hand side of the equation on the Lie algebroid form $f\,\basexi{\indexI}\wedge\basedualxi{\indexJp}$ directly and see that the constraints (\ref{EQ-Part 1-(1)-equation for coef B,D}) and (\ref{EQ-Part 1-(1)-equation for coef A,C}) are sufficient, which means that the K\"{a}hler Lie algebroid naturally leads to the Lie algebroid K\"{a}hler identities.

\begin{theorem}
\label{thm-Lie algebroid K\"{a}hler identities}
Suppose $(M,\g,\omega)$ is a K\"{a}hler Lie algebroid, then the following equations 
\begin{align*}
    [\Lambda,\LieagbdPartial]&=\sqrt{-1}{\LieagbdPartialbar}^{*},
    \\
    [\Lambda,\LieagbdPartialbar]&=\sqrt{-1}{\LieagbdPartial}^{*},
    \\
    [{\LieagbdPartial}^*,L]&=\sqrt{-1}\LieagbdPartialbar,
    \\
    [{\LieagbdPartialbar}^*,L]&=\sqrt{-1}\LieagbdPartial
\end{align*}
hold, where $L:=\omega\wedge:\Apq{p}{q}(M)\to\Apq{p+1}{q+1}(M)$ and $\Lambda:=L^*:\Apq{p}{q}(M)\to\Apq{p-1}{q-1}(M)$. We call these equations the \keyword{Lie algebroid K\"{a}hler identities}.
\end{theorem}

\begin{proof}
Recall that the unitary frame $\{\basexi{i}\}_{i=1}^{n}$ spans $\gstarpq{1}{0}|_{U}$ and its dual $\{\basedualxi{i}\}_{i=1}^{n}$ spans $\gstarpq{0}{1}|_{U}$ over a small neighborhood $U$. We defined the coefficients $\coefA{}{}$, $\coefB{}{}$, $\coefC{}{}$, and $\coefD{}{}$ in the local expression of the Lie bracket and obtained two constraints on these coefficients from the condition that $\omega$ is closed. Our goal is to use these operators to check that the equation
\begin{align*}
    [\LieagbdPartialbar,\iotaW]
    (f\,\basexi{\indexI}\wedge\basedualxi{\indexJp})
    &=
    \sqrt{-1}\LieagbdPartial^{*}
    (f\,\basexi{\indexI}\wedge\basedualxi{\indexJp})
\end{align*}
holds for any $(f\,\basexi{\indexI}\wedge\basedualxi{\indexJp})\in\Apq{p}{q}(U)$, where $\indexI$ and $\indexJp$ are index number sets such that $|\indexI|=p$ and $|\indexJp|=q$ and $f$ is a smooth function. We take three steps to complete this check. 
\begin{enumerate}
    \item Prove
\begin{align}
    [\LieagbdPartialbar,\iotaW]
    (\basexi{i}\wedge\basedualxi{j})
    =
    \sqrt{-1}\LieagbdPartial^{*}
    (\basexi{i}\wedge\basedualxi{j})
    .
\label{EQ-Kahler identity on Lie agbd-induction-formula 1}
\end{align}
    We prove this equation in Lemma \ref{Lemma-Kahler identity on Lie algebroid-the first lemma}.

    \item Assume the equation
$
    [\LieagbdPartialbar,\iotaW]
    (\basexi{\indexI}\wedge\basedualxi{\indexJp})
    =
    \sqrt{-1}\LieagbdPartial^{*}
    (\basexi{\indexI}\wedge\basedualxi{\indexJp})
$
    holds for the given index number sets $\indexI=\{\indexi{1},\cdots,\indexi{p}\}$ and $\indexJp=\{\indexjp{1},\cdots,\indexjp{q}\}$. Prove that the equations
\begin{align}
    [\LieagbdPartialbar,\iotaW]
    (\basexi{\alpha}\wedge\basexi{\indexI}\wedge\basedualxi{\indexJp})
    &=
    \sqrt{-1}\LieagbdPartial^{*}
    (\basexi{\alpha}\wedge\basexi{\indexI}\wedge\basedualxi{\indexJp})
    \label{EQ-Kahler identity on Lie agbd-induction-formula 2}
\end{align}
and
\begin{align}
    [\LieagbdPartialbar,\iotaW]
    (\basedualxi{\alpha}\wedge\basexi{\indexI}\wedge\basedualxi{\indexJp})
    &=
    \sqrt{-1}\LieagbdPartial^{*}
    (\basedualxi{\alpha}\wedge\basexi{\indexI}\wedge\basedualxi{\indexJp})
    \label{EQ-Kahler identity on Lie agbd-induction-formula 2.5}
\end{align}
    hold. Without loss of generality, we can assume $\alpha<\indexi{1}$. The proof is complicated and long, so we use Section \ref{sec:very long proof for Lie algebroid Kahler identity} to present it.

    \item Assume that the equation
$
    [\LieagbdPartialbar,\iotaW]
    (\basexi{\indexI}\wedge\basedualxi{\indexJp})
    =
    \sqrt{-1}\LieagbdPartial^{*}
    (\basexi{\indexI}\wedge\basedualxi{\indexJp})
$ holds. Prove
\begin{align}
    [\LieagbdPartialbar,\iotaW]
    (f \basexi{\indexI}\wedge\basedualxi{\indexJp})
    =
    \sqrt{-1}\LieagbdPartial^{*}
    (f \basexi{\indexI}\wedge\basedualxi{\indexJp})
    ,
    \label{EQ-Kahler identity on Lie agbd-induction-formula 3}
\end{align}
     where $f$ is a smooth function. We prove this equation in Lemma \ref{Lemma-Kahler identity on Lie algebroid-the third lemma}.
\end{enumerate}

\end{proof}

\begin{lemma}
\label{Lemma-Kahler identity on Lie algebroid-the first lemma}
    $
    [\LieagbdPartialbar,\iotaW]
    (\basexi{\indexI}\wedge\basedualxi{\indexJp})
    =
    \sqrt{-1}\LieagbdPartial^{*}
    (\basexi{\indexI}\wedge\basedualxi{\indexJp})
    $.
\end{lemma}
\begin{proof}
    Let us look at the left-hand side for $\basexi{\indexI},\basedualxi{\indexJp}$. The Lie bracket is $\LieagbdPartialbar\,\iotaW-\iotaW\,\LieagbdPartialbar$, where the first term $\LieagbdPartialbar\,\iotaW(\basexi{i}\wedge\basedualxi{j})=0$ since 
\begin{align*}
    \iotaW(\basexi{i}\wedge\basedualxi{j})
    =
    \begin{cases}
        0           & i\neq j
    \\
        2\sqrt{-1}  & i=j
    \end{cases}
\end{align*}
is a constant function vanishing under $\LieagbdPartialbar$. This leads to the remaining part, $-\iotaW\,\LieagbdPartialbar(\basexi{i}\wedge\basedualxi{j})$. The expression of $\LieagbdPartialbar(\basexi{i}\wedge\basedualxi{j})$ is
\begin{align}
    \LieagbdPartialbar
    (\basexi{i}\wedge\basedualxi{j^{'}})
    =&
        \displaystyle\sum_{k=1}^{n}
        \displaystyle\sum_{l^{'}=1}^{n}
        (\LieagbdPartialbar\basexi{i})
        \wedge
        \basedualxi{j^{'}}
        -
        \displaystyle\sum_{l_{1}^{'}=1}^{n}
        \displaystyle\sum_{l_{2}^{'}=1}^{n}
        \basexi{i}
        \wedge
        (\LieagbdPartialbar\basedualxi{j^{'}})
    \notag
    \\
    =&
        \displaystyle\sum_{k=1}^{n}
        \displaystyle\sum_{l^{'}=1}^{n}
        \coefB{i}{k l^{'}}
        \basexi{k}\wedge\basedualxi{l^{'}}
        \wedge
        \basedualxi{j^{'}}
        -
        \displaystyle\sum_{l_{1}^{'}=1}^{n}
        \displaystyle\sum_{l_{2}^{'}=1}^{n}
        \basexi{i}
        \wedge
        \frac{1}{2}
        \sgn
            {\order{l_{1}^{'},l_{2}^{'}}}
            {l_{1}^{'},l_{2}^{'}}
        \coefD{j^{'}}{\order{l_{1}^{'},l_{2}^{'}}}
        \basedualxi{l_{1}^{'}}
        \wedge
        \basedualxi{l_{2}^{'}}.
    \label{EQ-Kahler identity on Lie agbd-calculation for LHS of (i,j)-no1}
\end{align}
We plug $\iotaW$ into the front of (\ref{EQ-Kahler identity on Lie agbd-calculation for LHS of (i,j)-no1}) and obtain
\begin{align}
    \iotaW\,\LieagbdPartialbar
    (\basexi{i}\wedge\basedualxi{j^{'}})
    &
        =
        -\frac{\sqrt{-1}}{2}
        \displaystyle\sum_{h=1}^{n}
        \iota_{\basedualxi{h^{'}}}\,\iota_{\basexi{h}}
        \begin{pmatrix*}[l]
            \displaystyle\sum_{k=1}^{n}
            \displaystyle\sum_{l^{'}=1}^{n}
            \coefB{i}{k l^{'}}
            \basexi{k}\wedge\basedualxi{l^{'}}
            \wedge
            \basedualxi{j^{'}}
            \\
            -
            \\
            \displaystyle\sum_{l_{1}^{'}=1}^{n}
            \displaystyle\sum_{l_{2}^{'}=1}^{n}
            \basexi{i}
            \wedge
            \frac{1}{2}
            \sgn
                {\order{l_{1}^{'},l_{2}^{'}}}
                {l_{1}^{'},l_{2}^{'}}
            \coefD{j^{'}}{\order{l_{1}^{'},l_{2}^{'}}}
            \basedualxi{l_{1}^{'}}
            \wedge
            \basedualxi{l_{2}^{'}}
        \end{pmatrix*}.
    \label{EQ-Kahler identity on Lie agbd-calculation for LHS of (i,j)-no2}
\end{align}
Using the contractions $\contract{\basexi{h}}(\basexi{k})=2\delta_{h,k}$ and $\contract{\basedualxi{h}}(\basedualxi{k})=2\delta_{h,k}$, we simplify (\ref{EQ-Kahler identity on Lie agbd-calculation for LHS of (i,j)-no2}) as follows:
\begin{align*}
    \iotaW\,\LieagbdPartialbar
    (\basexi{i}\wedge\basedualxi{j^{'}})
    &
        =
        -\frac{\sqrt{-1}}{2}
        \begin{pmatrix*}[l]
            \displaystyle\sum_{h=1}^{n}
            \displaystyle\sum_{l^{'}=1}^{n}
            \iota_{\basedualxi{h^{'}}}
            (
            2
            \coefB{i}{h l^{'}}
            \basedualxi{l^{'}}
            \wedge
            \basedualxi{j^{'}}
            )
            \\
            -
            \\
            \displaystyle\sum_{l_{1}^{'}=1}^{n}
            \displaystyle\sum_{l_{2}^{'}=1}^{n}
            \frac{1}{2}
            \sgn
                {\order{l_{1}^{'},l_{2}^{'}}}
                {l_{1}^{'},l_{2}^{'}}
            \iota_{\basedualxi{i^{'}}}
            (
            2
            \coefD{j^{'}}{\order{l_{1}^{'},l_{2}^{'}}}
            \basedualxi{l_{1}^{'}}
            \wedge
            \basedualxi{l_{2}^{'}}
            )
        \end{pmatrix*}
    \\
    &
        =
        -\sqrt{-1}
        \begin{pmatrix*}[l]
            \displaystyle\sum_{l^{'}=1}^{n}
            2\coefB{i}{l l^{'}}
            \basedualxi{j^{'}}
            \\
            -
            \\
            \displaystyle\sum_{l^{'}=1}^{n}
            2\coefB{i}{j l^{'}}
            \basedualxi{l^{'}}
            \\
            -
            \\
            \displaystyle\sum_{l_{2}^{'}=1}^{n}
            \frac{1}{2}
            \sgn
                {\order{i^{'},l_{2}^{'}}}
                {i^{'},l_{2}^{'}}
            2\coefD{j^{'}}{\order{i^{'},l_{2}^{'}}}
            \basedualxi{l_{2}^{'}}
            \\
            +
            \\
            \displaystyle\sum_{l_{1}^{'}=1}^{n}
            \frac{1}{2}
            \sgn
                {\order{l_{1}^{'},i^{'}}}
                {l_{1}^{'},i^{'}}
            2\coefD{j^{'}}{\order{l_{1}^{'},i^{'}}}
            \basedualxi{l_{1}^{'}}
        \end{pmatrix*}.
\end{align*}
We use the constraint $\coefB{l}{j i} -\coefB{i}{j l}=\sgn{\order{i,l}}{i,l}\coefD{j}{\order{i,l}}$ to simplify the term
\begin{align*}
    \iotaW\,\LieagbdPartialbar
    (\basexi{i}\wedge\basedualxi{j^{'}})
    &=
        -2\sqrt{-1}
        \begin{pmatrix*}[l]
            \displaystyle\sum_{l^{'}=1}^{n}
            \coefB{i}{l l^{'}}
            \basedualxi{j^{'}}
            \\
            -
            \\
            \displaystyle\sum_{l^{'}=1}^{n}
            \coefB{i}{j l^{'}}
            \basedualxi{l^{'}}
            \\
            -
            \\
            \displaystyle\sum_{l^{'}=1}^{n}
            \sgn
                {\order{i^{'},l^{'}}}
                {i^{'},l^{'}}
            \coefD{j^{'}}{\order{i^{'},l^{'}}}
            \basedualxi{l^{'}}
        \end{pmatrix*}
    \\
    &=
        -2\sqrt{-1}
        \begin{pmatrix*}[l]
            \displaystyle\sum_{l^{'}=1}^{n}
            \coefB{i}{l l^{'}}
            \basedualxi{j^{'}}
            \\
            -
            \\
            \displaystyle\sum_{l^{'}=1}^{n}
            (
            \coefB{i}{j l^{'}}
            +
            \sgn
                {\order{i^{'},l^{'}}}
                {i^{'},l^{'}}
            \coefD{j^{'}}{\order{i^{'},l^{'}}}
            )
            \basedualxi{l^{'}}
        \end{pmatrix*}
    \\
    &
        =
        -2\sqrt{-1}
        \begin{pmatrix*}[l]
            \displaystyle\sum_{l^{'}=1}^{n}
            \coefB{i}{l l^{'}}
            \basedualxi{j^{'}}
            -
            \displaystyle\sum_{l^{'}=1}^{n}
            \coefB{l^{'}}{j i}
            \basedualxi{l^{'}}
        \end{pmatrix*}.
\end{align*}
We conclude the left-hand side:
\begin{align}
    [\LieagbdPartialbar,\iotaW]
    (\basexi{i}\wedge\basedualxi{j})
    =
    -2\sqrt{-1}
    \begin{pmatrix*}[l]
        \displaystyle\sum_{l^{'}=1}^{n}
        \coefB{i}{l l^{'}}
        \basedualxi{j^{'}}
        -
        \displaystyle\sum_{l^{'}=1}^{n}
        \coefB{l^{'}}{j i}
        \basedualxi{l^{'}}
    \end{pmatrix*}.
    \label{EQ-Kahler identity on Lie agbd-calculation for LHS of (i,j)-result}
\end{align}
    Next, let us look at the right-hand side. Recall that $\LieagbdPartial^{*}=-(-1)^{n^2}\star\,\LieagbdPartialbar\,\star$, so the right-hand side can be written as
\begin{align}
    \sqrt{-1}\LieagbdPartial^{*}
    (\basexi{i}\wedge\basedualxi{j^{'}})
    =
    -\sqrt{-1}(-1)^{n^2}\star\,\LieagbdPartialbar\,\star
    (\basexi{i}\wedge\basedualxi{j^{'}}).
    \label{EQ-Kahler identity on Lie agbd-calculation for RHS of (i,j)-no1}
\end{align}
Calculating (\ref{EQ-Kahler identity on Lie agbd-calculation for RHS of (i,j)-no1}) step-by-step, we first apply the Hodge star to $(\basexi{i}\wedge\basedualxi{j^{'}})$:
\begin{align}
    \star(\basexi{i}\wedge\basedualxi{j^{'}})
    =
    2^{2-n}
        \sgn
            {1,\cdots,n,1^{'},\cdots,n^{'}}
            {i,j^{'},1,\cdots,\widehat{i},\cdots,n,1^{'},\cdots,\widehat{j^{'}},\cdots,n^{'}}
        \basedualxi{1}
        \wedge\cdots\wedge
        \widehat{\basedualxi{i}}
        \wedge\cdots\wedge
        \basedualxi{n}
        \wedge
        \basexi{1^{'}}
        \wedge\cdots\wedge
        \widehat{\basexi{j^{'}}}
        \wedge\cdots\wedge
        \basexi{n^{'}}.
    \label{EQ-Kahler identity on Lie agbd-calculation for RHS of (i,j)-no2}
\end{align}
We then apply $\LieagbdPartialbar$ to (\ref{EQ-Kahler identity on Lie agbd-calculation for RHS of (i,j)-no2}) and obtain
\begin{align}
    &
        \LieagbdPartialbar\,\star(\basexi{i}\wedge\basedualxi{j^{'}})
    \notag
    \\
    =&
        2^{2-n}
        \begin{pmatrix*}[l]
            \displaystyle\sum_{k\neq i}
            \begin{matrix*}[l]
                \sgn
                    {1,\cdots,n,1^{'},c\dots,n^{'}}
                    {i,j^{'},1,\dots,\widehat{i},\cdots,n,1^{'},\cdots,\widehat{j^{'}},\cdots,n^{'}}
                \sgn
                    {1,\cdots,\widehat{i},\cdots,n}
                    {k,1,\cdots,\widehat{i},\cdots,\widehat{k},\cdots,n}
                \\
                \basedualxi{1}
                \wedge\cdots\wedge
                \widehat{\basedualxi{i}}
                \wedge\cdots\wedge
                (\LieagbdPartialbar\basedualxi{k})
                \wedge\cdots\wedge
                \basedualxi{n}
                \wedge
                \basexi{1^{'}}
                \wedge\cdots\wedge
                \widehat{\basexi{j^{'}}}
                \wedge\cdots\wedge
                \basexi{n^{'}}
            \end{matrix*}
            \\
            +
            \\
            \displaystyle\sum_{l^{'}\neq j^{'}}
            \begin{matrix*}[l]
                \sgn
                    {1,\cdots,n,1^{'},\cdots,n^{'}}
                    {i,j^{'},1,\cdots,\widehat{i},\cdots,n,1^{'},\cdots,\widehat{j^{'}},\cdots,n^{'}}
                \sgn
                    {1,\cdots,\widehat{i},\cdots,n,1^{'},\cdots,\widehat{j^{'}},\cdots,n^{'}}
                    {l^{'},1,\cdots,\widehat{i},\cdots,n,1^{'},\cdots,\widehat{j^{'}},\cdots,\widehat{l^{'}},\cdots,n^{'}}
                \\
                \basedualxi{1}
                \wedge\cdots\wedge
                \widehat{\basedualxi{i}}
                \wedge\cdots\wedge
                \basedualxi{n}
                \wedge
                \basexi{1^{'}}
                \wedge\cdots\wedge
                \widehat{\basexi{j^{'}}}
                \wedge\cdots\wedge
                (\LieagbdPartialbar\basexi{l^{'}})
                \wedge\cdots\wedge
                \basexi{n^{'}}
            \end{matrix*}
        \end{pmatrix*}
    \notag
    \\
    =&
        2^{2-n}
        \begin{pmatrix*}[l]
            \displaystyle\sum_{k\neq i}
            \begin{pmatrix*}[l]
                \sgn
                    {1,\cdots,n,1^{'},\cdots,n^{'}}
                    {i,j^{'},1,\cdots,\widehat{i},\cdots,n,1^{'},\cdots,\widehat{j^{'}},\cdots,n^{'}}
                \sgn
                    {1,\cdots,\widehat{i},\cdots,n}
                    {k,1,\cdots,\widehat{i},\cdots,\widehat{k},\cdots,n}
                \\
                \basedualxi{1}
                \wedge\cdots\wedge
                \widehat{\basedualxi{i}}
                \wedge\cdots\wedge
                \basedualxi{k-1}
                \wedge
                \begin{pmatrix*}[l]
                    \frac{1}{2}
                    \coefD{k}{i k}
                    \basedualxi{i}\wedge\basedualxi{k}
                    \\
                    +
                    \\
                    \frac{1}{2}
                    \coefD{k}{k i}
                    \basedualxi{k}\wedge\basedualxi{i}
                \end{pmatrix*}
                \wedge
                \basedualxi{k+1}
                \wedge\cdots\wedge
                \basedualxi{n}
                \wedge
                \\
                \basexi{1^{'}}
                \wedge\cdots\wedge
                \widehat{\basexi{j^{'}}}
                \wedge\cdots\wedge
                \basexi{n^{'}}
            \end{pmatrix*}
            \\
            +
            \\
            \displaystyle\sum_{l^{'}\neq j^{'}}
            \begin{pmatrix*}[l]
                \sgn
                    {1,\cdots,n,1^{'},\cdots,n^{'}}
                    {i,j^{'},1,\cdots,\widehat{i},\cdots,n,1^{'},\cdots,\widehat{j^{'}},\cdots,n^{'}}
                \sgn
                    {1,\cdots,\widehat{i},\cdots,n,1^{'},\cdots,\widehat{j^{'}},\cdots,n^{'}}
                    {l^{'},1,\cdots,\widehat{i},\cdots,n,1^{'},\cdots,\widehat{j^{'}},\cdots,\widehat{l^{'}},\cdots,n^{'}}
                \\
                \basedualxi{1}
                \wedge\cdots\wedge
                \widehat{\basedualxi{i}}
                \wedge\cdots\wedge
                \basedualxi{n}
                \wedge
                \\
                \basexi{1^{'}}
                \wedge\cdots\wedge
                \widehat{\basexi{j^{'}}}
                \wedge\cdots\wedge
                \basexi{(l-1)^{'}}
                \wedge
                \begin{pmatrix*}[l]
                    \coefB{l^{'}}{l^{'} i}
                    \basexi{l^{'}}\wedge\basedualxi{i}
                    \\
                    +
                    \\
                    \coefB{l^{'}}{j^{'} i}
                    \basexi{j^{'}}\wedge\basedualxi{i}
                \end{pmatrix*}
                \wedge
                \basexi{(l+1)^{'}}
                \wedge\cdots\wedge
                \basexi{n^{'}}
            \end{pmatrix*}
        \end{pmatrix*}.
    \label{EQ-Kahler identity on Lie agbd-calculation for RHS of (i,j)-no3}
\end{align}
Looking at the component with $
    \begin{pmatrix*}[l]
    \frac{1}{2}
    \coefD{k}{i k}
    \basedualxi{i}\wedge\basedualxi{k}
    \\
    +
    \\
    \frac{1}{2}
    \coefD{k}{k i}
    \basedualxi{k}\wedge\basedualxi{i}
    \end{pmatrix*}
$, it can be simplified to 
\begin{align*}
    \sgn{k,i}{\order{k,i}}\coefD{k}{k i}\basedualxi{k}\wedge\basedualxi{i}.
\end{align*}
For the component with coefficient $B$, we break the sum and rearrange the terms in (\ref{EQ-Kahler identity on Lie agbd-calculation for RHS of (i,j)-no3}):
\begin{align}
    &
        \LieagbdPartialbar\,\star(\basexi{i}\wedge\basedualxi{j^{'}})
    \notag
    \\
    =&
        2^{2-n}
        \begin{pmatrix*}[l]
            \displaystyle\sum_{k\neq i}
            \begin{pmatrix*}[l]
                \sgn
                    {1,\cdots,n,1^{'},\cdots,n^{'}}
                    {i,j^{'},1,\cdots,\widehat{i},\cdots,n,1^{'},\cdots,\widehat{j^{'}},\cdots,n^{'}}
                \sgn
                    {1,\cdots,\widehat{i},\cdots,n}
                    {k,1,\cdots,\widehat{i},\cdots,\widehat{k},\cdots,n}
                \\
                \basedualxi{1}
                \wedge\cdots\wedge
                \widehat{\basedualxi{i}}
                \wedge\cdots\wedge
                \basedualxi{k-1}
                \wedge
                \sgn
                    {k,i}
                    {\order{k,i}}
                \coefD{k}{k i}
                \basedualxi{k}\wedge\basedualxi{i}
                \wedge
                \basedualxi{k+1}
                \wedge\cdots\wedge
                \basedualxi{n}
                \\
                \wedge\basexi{1^{'}}
                \wedge\cdots\wedge
                \widehat{\basexi{j^{'}}}
                \wedge\cdots\wedge
                \basexi{n^{'}}
            \end{pmatrix*}
            \\
            +
            \\
            \displaystyle\sum_{l^{'}\neq j^{'}}
            \begin{pmatrix*}[l]
                \sgn
                    {1,\cdots,n,1^{'},\cdots,n^{'}}
                    {i,j^{'},1,\cdots,\widehat{i},\cdots,n,1^{'},\cdots,\widehat{j^{'}},\cdots,n^{'}}
                \sgn
                    {1,\cdots,\widehat{i},\cdots,n,1^{'},\cdots,\widehat{j^{'}},\cdots,n^{'}}
                    {l^{'},1,\cdots,\widehat{i},\cdots,n,1^{'},\cdots,\widehat{j^{'}},\cdots,\widehat{l^{'}},\cdots,n^{'}}
                \\
                \basedualxi{1}
                \wedge\cdots\wedge
                \widehat{\basedualxi{i}}
                \wedge\cdots\wedge
                \basedualxi{n}
                \\
                \wedge\basexi{1^{'}}
                \wedge\cdots\wedge
                \widehat{\basexi{j^{'}}}
                \wedge\cdots\wedge
                \basexi{(l-1)^{'}}
                \wedge
                \coefB{l^{'}}{l^{'} i}
                \basexi{l^{'}}\wedge\basedualxi{i}
                \wedge
                \basexi{(l+1)^{'}}
                \wedge\cdots\wedge
                \basexi{n^{'}}
            \end{pmatrix*}
            \\
            +
            \\
            \displaystyle\sum_{l^{'}\neq j^{'}}
            \begin{pmatrix*}[l]
                \sgn
                    {1,\cdots,n,1^{'},\cdots,n^{'}}
                    {i,j^{'},1,\cdots,\widehat{i},\cdots,n,1^{'},\cdots,\widehat{j^{'}},\cdots,n^{'}}
                \sgn
                    {1,\cdots,\widehat{i},\cdots,n,1^{'},\cdots,\widehat{j^{'}},\cdots,n^{'}}
                    {l^{'},1,\cdots,\widehat{i},\cdots,n,1^{'},\cdots,\widehat{j^{'}},\cdots,\widehat{l^{'}},\cdots,n^{'}}
                \\
                \basedualxi{1}
                \wedge\cdots\wedge
                \widehat{\basedualxi{i}}
                \wedge\cdots\wedge
                \basedualxi{n}
                \\
                \wedge\basexi{1^{'}}
                \wedge\cdots\wedge
                \widehat{\basexi{j^{'}}}
                \wedge\cdots\wedge
                \basexi{(l-1)^{'}}
                \wedge
                \coefB{l^{'}}{j^{'} i}
                \basexi{j^{'}}\wedge\basedualxi{i}
                \wedge
                \basexi{(l+1)^{'}}
                \wedge\cdots\wedge
                \basexi{n^{'}}
            \end{pmatrix*}
        \end{pmatrix*}
    \notag
    \\
    =&
        2^{2-n}
        \begin{pmatrix*}[l]
            \displaystyle\sum_{k\neq i}
            \begin{pmatrix*}[l]
                \sgn
                    {1,\cdots,n,1^{'},\cdots,n^{'}}
                    {i,j^{'},1,\cdots,\widehat{i},\cdots,n,1^{'},\cdots,\widehat{j^{'}},\cdots,n^{'}}
                \sgn
                    {1,\cdots,\widehat{i},\cdots,n}
                    {k,1,\cdots,\widehat{i},\cdots,\widehat{k},\cdots,n}
                \sgn
                    {1,\cdots,\widehat{i},\cdots,k-1,k,i,k+1,\cdots,n}
                    {1,\cdots,n}
                \\
                \sgn
                    {k,i}
                    {\order{k,i}}
                \\
                \coefD{k}{\order{k,i}}
                \basedualxi{1}
                \wedge\cdots\wedge
                \basedualxi{n}
                \wedge
                \basexi{1^{'}}
                \wedge\cdots\wedge
                \widehat{\basexi{j^{'}}}
                \wedge\cdots\wedge
                \basexi{n^{'}}
            \end{pmatrix*}
            \\
            +
            \\
            \displaystyle\sum_{l^{'}\neq j^{'}}
            \begin{pmatrix*}[l]
                \sgn
                    {1,\cdots,n,1^{'},\cdots,n^{'}}
                    {i,j^{'},1,\cdots,\widehat{i},\cdots,n,1^{'},\cdots,\widehat{j^{'}},\cdots,n^{'}}
                \sgn
                    {1,\cdots,\widehat{i},\cdots,n,1^{'},\cdots,\widehat{j^{'}},\cdots,n^{'}}
                    {l^{'},1,\cdots,\widehat{i},\cdots,n,1^{'},\cdots,\widehat{j^{'}},\cdots,\widehat{l^{'}},\cdots,n^{'}}
                \\
                \sgn
                    {1,\cdots,\widehat{i}\cdots,n,1^{'},\cdots,\widehat{j^{'}},\cdots,(l-1)^{'},l^{'},i,(l+1)^{'},\cdots,n^{'}}
                    {1,\cdots,n,1^{'},\cdots,\widehat{j^{'}},\cdots,n^{'}}
                \\
                \coefB{l^{'}}{l^{'} i}
                \basedualxi{1}
                \wedge\cdots\wedge
                \basedualxi{n}
                \wedge
                \basexi{1^{'}}
                \wedge\cdots\wedge
                \widehat{\basexi{j^{'}}}
                \wedge\cdots\wedge
                \basexi{n^{'}}
            \end{pmatrix*}
            \\
            +
            \\
            \displaystyle\sum_{l^{'}\neq j^{'}}
            \begin{pmatrix*}[l]
                \sgn
                    {1,\cdots,n,1^{'},\cdots,n^{'}}
                    {i,j^{'},1,\cdots,\widehat{i},\cdots,n,1^{'},\cdots,\widehat{j^{'}},\cdots,n^{'}}
                \sgn
                    {1,\cdots,\widehat{i},\cdots,n,1^{'},\cdots,\widehat{j^{'}},\cdots,n^{'}}
                    {l^{'},1,\cdots,\widehat{i},\cdots,n,1^{'},\cdots,\widehat{j^{'}},\cdots,\widehat{l^{'}},\cdots,n^{'}}
                \\
                \sgn
                    {1,\cdots,\widehat{i}\cdots,n,1^{'},\cdots,\widehat{j^{'}},\cdots,(l-1)^{'},j^{'},i,(l+1)^{'},\cdots,n^{'}}
                    {1,\cdots,n,1^{'},\cdots,\widehat{l^{'}},\cdots,n^{'}}
                \\
                \coefB{l^{'}}{j^{'} i}
                \basedualxi{1}
                \wedge\cdots\wedge
                \basedualxi{n}
                \wedge
                \basexi{1^{'}}
                \wedge\cdots\wedge
                \widehat{\basexi{l^{'}}}
                \wedge\cdots\wedge
                \basexi{n^{'}}
            \end{pmatrix*}
        \end{pmatrix*}.
    \label{EQ-Kahler identity on Lie agbd-calculation for RHS of (i,j)-no4}
\end{align}
The next step is applying $\star$ to (\ref{EQ-Kahler identity on Lie agbd-calculation for RHS of (i,j)-no4}). Observe that the exterior product of the frames in (\ref{EQ-Kahler identity on Lie agbd-calculation for RHS of (i,j)-no4}), like $\basedualxi{1}\wedge\cdots\wedge\basedualxi{n}\wedge\basexi{1^{'}}\wedge\cdots\wedge\widehat{\basexi{j^{'}}}\wedge\cdots\wedge\basexi{n^{'}}$, is not in the standard form where the holomorphic frames are in front of the anti-holomorphic frames. The expression of the Hodge star applying to this type of local frames is
\begin{align*}
    \star
    (\basedualxi{\indexIc}\wedge\basexi{\indexJpc})
    =
    2^{(|\indexIc|+|\indexJpc|-n)}
    \sgn
        {1^{'},\cdots,n^{'},1,\cdots,n}{\indexIc,\indexJpc,\indexI,\indexJp}\basexi{\indexI}\wedge\basedualxi{\indexJp},
\end{align*}
where $\indexI$ and $\indexJp$ are the arrays of the index numbers. We use this formula and obtain the following expression: 
\begin{align}
    &
        \star\,\LieagbdPartialbar\,\star(\basexi{i}\wedge\basedualxi{j^{'}})
    \notag
    \\
    =&
        2^{2-n}
        \begin{pmatrix*}[l]
            \displaystyle\sum_{k\neq i}
            \begin{pmatrix*}[l]
                \sgn
                    {1,\cdots,n,1^{'},\cdots,n^{'}}
                    {i,j^{'},1,\cdots,\widehat{i},\cdots,n,1^{'},\cdots,\widehat{j^{'}},\cdots,n^{'}}
                \sgn
                    {1,\cdots,\widehat{i},\cdots,n}
                    {k,1,\cdots,\widehat{i},\cdots,\widehat{k},\cdots,n}
                \sgn
                    {1,\cdots,\widehat{i},\cdots,k-1,k,i,k+1,\cdots,n}
                    {1,\cdots,n}
                \\
                \sgn
                    {k,i}
                    {\order{k,i}}
                \\
                \coefD{k}{\order{k,i}}
                \star
                (
                \basedualxi{1}
                \wedge\cdots\wedge
                \basedualxi{n}
                \wedge
                \basexi{1^{'}}
                \wedge\cdots\wedge
                \widehat{\basexi{j^{'}}}
                \wedge\cdots\wedge
                \basexi{n^{'}}
                )
            \end{pmatrix*}
            \\
            +
            \\
            \displaystyle\sum_{l^{'}\neq j^{'}}
            \begin{pmatrix*}[l]
                \sgn
                    {1,\cdots,n,1^{'},\cdots,n^{'}}
                    {i,j^{'},1,\cdots,\widehat{i},\cdots,n,1^{'},\cdots,\widehat{j^{'}},\cdots,n^{'}}
                \sgn
                    {1,\cdots,\widehat{i},\cdots,n,1^{'},\cdots,\widehat{j^{'}},\cdots,n^{'}}
                    {l^{'},1,\cdots,\widehat{i},\cdots,n,1^{'},\cdots,\widehat{j^{'}},\cdots,\widehat{l^{'}},\cdots,n^{'}}
                \\
                \sgn
                    {1,\cdots,\widehat{i}\cdots,n,1^{'},\cdots,\widehat{j^{'}},\cdots,(l-1)^{'},l^{'},i,(l+1)^{'},\cdots,n^{'}}
                    {1,\cdots,n,1^{'},\cdots,\widehat{j^{'}},\cdots,n^{'}}
                \\
                \coefB{l^{'}}{l^{'} i}
                \star
                (
                \basedualxi{1}
                \wedge\cdots\wedge
                \basedualxi{n}
                \wedge
                \basexi{1^{'}}
                \wedge\cdots\wedge
                \widehat{\basexi{j^{'}}}
                \wedge\cdots\wedge
                \basexi{n^{'}}
                )
            \end{pmatrix*}
            \\
            +
            \\
            \displaystyle\sum_{l^{'}\neq j^{'}}
            \begin{pmatrix*}[l]
                \sgn
                    {1,\cdots,n,1^{'},\cdots,n^{'}}
                    {i,j^{'},1,\cdots,\widehat{i},\cdots,n,1^{'},\cdots,\widehat{j^{'}},\cdots,n^{'}}
                \sgn
                    {1,\cdots,\widehat{i},\cdots,n,1^{'},\cdots,\widehat{j^{'}},\cdots,n^{'}}
                    {l^{'},1,\cdots,\widehat{i},\cdots,n,1^{'},\cdots,\widehat{j^{'}},\cdots,\widehat{l^{'}},\cdots,n^{'}}
                \\
                \sgn
                    {1,\cdots,\widehat{i}\cdots,n,1^{'},\cdots,\widehat{j^{'}},\cdots,(l-1)^{'},j^{'},i,(l+1)^{'},\cdots,n^{'}}
                    {1,\cdots,n,1^{'},\cdots,\widehat{l^{'}},\cdots,n^{'}}
                \\
                \coefB{l^{'}}{j^{'} i}
                \star
                (
                \basedualxi{1}
                \wedge\cdots\wedge
                \basedualxi{n}
                \wedge
                \basexi{1^{'}}
                \wedge\cdots\wedge
                \widehat{\basexi{l^{'}}}
                \wedge\cdots\wedge
                \basexi{n^{'}}
                )
            \end{pmatrix*}
        \end{pmatrix*}
    \notag
    \\
    =&
        2^{2-n}
        \begin{pmatrix*}[l]
            \displaystyle\sum_{k\neq i}
            \begin{pmatrix*}[l]
                \sgn
                    {1,\cdots,n,1^{'},\cdots,n^{'}}
                    {i,j^{'},1,\cdots,\widehat{i},\cdots,n,1^{'},\cdots,\widehat{j^{'}},\cdots,n^{'}}
                \sgn
                    {1,\cdots,\widehat{i},\cdots,n}
                    {k,1,\cdots,\widehat{i},\cdots,\widehat{k},\cdots,n}
                \sgn
                    {1,\cdots,\widehat{i},\cdots,k-1,k,i,k+1,\cdots,n}
                    {1,\cdots,n}
                \\
                \sgn
                    {k,i}
                    {\order{k,i}}
                \\
                \coefD{k}{\order{k,i}}
                2^{n+n-1-n}
                \sgn
                    {1^{'},\cdots,n^{'},1,\cdots,n}
                    {j^{'},1,\cdots,n,1^{'},\cdots,\widehat{j^{'}},\cdots,n^{'}}
                \basexi{j^{'}}
            \end{pmatrix*}
            \\
            +
            \\
            \displaystyle\sum_{l^{'}\neq j^{'}}
            \begin{pmatrix*}[l]
                \sgn
                    {1,\cdots,n,1^{'},\cdots,n^{'}}
                    {i,j^{'},1,\cdots,\widehat{i},\cdots,n,1^{'},\cdots,\widehat{j^{'}},\cdots,n^{'}}
                \sgn
                    {1,\cdots,\widehat{i},\cdots,n,1^{'},\cdots,\widehat{j^{'}},\cdots,n^{'}}                       {l^{'},1,\cdots,\widehat{i},\cdots,n,1^{'},\cdots,\widehat{j^{'}},\cdots,\widehat{l^{'}},\cdots,n^{'}}
                \\
                \sgn
                    {1,\cdots,\widehat{i}\cdots,n,1^{'},\cdots,\widehat{j^{'}},\cdots,(l-1)^{'},l^{'},i,(l+1)^{'},\cdots,n^{'}}
                    {1,\cdots,n,1^{'},\cdots,\widehat{j^{'}},\cdots,n^{'}}
                \\
                \coefB{l^{'}}{l^{'} i}
                2^{n+n-1-n}
                \sgn
                    {1^{'},\cdots,n^{'},1,\cdots,n}
                    {j^{'},1,\cdots,n,1^{'},\cdots,\widehat{j^{'}},\cdots,n^{'}}
                \basexi{j^{'}}
            \end{pmatrix*}
            \\
            +
            \\
            \displaystyle\sum_{l^{'}\neq j^{'}}
            \begin{pmatrix*}[l]
                \sgn
                    {1,\cdots,n,1^{'},\cdots,n^{'}}
                    {i,j^{'},1,\cdots,\widehat{i},\cdots,n,1^{'},\cdots,\widehat{j^{'}},\cdots,n^{'}}
                \sgn
                    {1,\cdots,\widehat{i},\cdots,n,1^{'},\cdots,\widehat{j^{'}},\cdots,n^{'}}
                    {l^{'},1,\cdots,\widehat{i},\cdots,n,1^{'},\cdots,\widehat{j^{'}},\cdots,\widehat{l^{'}},\cdots,n^{'}}
                \\
                \sgn
                    {1,\cdots,\widehat{i}\cdots,n,1^{'},\cdots,\widehat{j^{'}},\cdots,(l-1)^{'},j^{'},i,(l+1)^{'},\cdots,n^{'}}
                    {1,\cdots,n,1^{'},\cdots,\widehat{l^{'}},\cdots,n^{'}}
                \\
                \coefB{l^{'}}{j^{'} i}
                2^{n+n-1-n}
                \sgn
                    {1^{'},\cdots,n^{'},1,\cdots,n}
                    {l^{'},1,\cdots,n,1^{'},\cdots,\widehat{l^{'}},\cdots,n^{'}}
                \basexi{l^{'}}
            \end{pmatrix*}
        \end{pmatrix*}
    \nonumber\\
    =&
        2
        \begin{pmatrix*}[l]
            \displaystyle\sum_{k\neq i}
            \begin{pmatrix*}[l]
                \sgn
                    {1,\cdots,n,1^{'},\cdots,n^{'}}
                    {i,j^{'},1,\cdots,\widehat{i},\cdots,n,1^{'},\cdots,\widehat{j^{'}},\cdots,n^{'}}
                \sgn
                    {1,\cdots,\widehat{i},\cdots,n}
                    {k,1,\cdots,\widehat{i},\cdots,\widehat{k},\cdots,n}
                \sgn
                    {1,\cdots,\widehat{i},\cdots,k-1,k,i,k+1,\cdots,n}
                    {1,\cdots,n}
                \\
                \sgn
                    {1^{'},\cdots,n^{'},1,\cdots,n}
                    {j^{'},1,\cdots,n,1^{'},\cdots,\widehat{j^{'}},\cdots,n^{'}}
                \\
                \sgn
                    {k,i}
                    {\order{k,i}}
                \\
                \coefD{k}{\order{k,i}}
                \basexi{j^{'}}
            \end{pmatrix*}
            \\
            +
            \\
            \displaystyle\sum_{l^{'}\neq j^{'}}
            \begin{pmatrix*}[l]
                \sgn
                    {1,\cdots,n,1^{'},\cdots,n^{'}}
                    {i,j^{'},1,\cdots,\widehat{i},\cdots,n,1^{'},\cdots,\widehat{j^{'}},\cdots,n^{'}}
                \sgn
                    {1,\cdots,\widehat{i},\cdots,n,1^{'},\cdots,\widehat{j^{'}},\cdots,n^{'}}
                    {l^{'},1,\cdots,\widehat{i},\cdots,n,1^{'},\cdots,\widehat{j^{'}},\cdots,\widehat{l^{'}},\cdots,n^{'}}
                \\
                \sgn
                    {1,\cdots,\widehat{i}\cdots,n,1^{'},\cdots,\widehat{j^{'}},\cdots,(l-1)^{'},l^{'},i,(l+1)^{'},\cdots,n^{'}}
                    {1,\cdots,n,1^{'},\cdots,\widehat{j^{'}},\cdots,n^{'}}
                \sgn
                    {1^{'},\cdots,n^{'},1,\cdots,n}
                    {j^{'},1,\cdots,n,1^{'},\cdots,\widehat{j^{'}},\cdots,n^{'}}
                \\
                \coefB{l^{'}}{l^{'} i}
                \basexi{j^{'}}
            \end{pmatrix*}
            \\
            +
            \\
            \displaystyle\sum_{l^{'}\neq j^{'}}
            \begin{pmatrix*}[l]
                \sgn
                    {1,\cdots,n,1^{'},\cdots,n^{'}}
                    {i,j^{'},1,\cdots,\widehat{i},\cdots,n,1^{'},\cdots,\widehat{j^{'}},\cdots,n^{'}}
                \sgn
                    {1,\cdots,\widehat{i},\cdots,n,1^{'},\cdots,\widehat{j^{'}},\cdots,n^{'}}
                    {l^{'},1,\cdots,\widehat{i},\cdots,n,1^{'},\cdots,\widehat{j^{'}},\cdots,\widehat{l^{'}},\cdots,n^{'}}
                \\
                \sgn
                    {1,\cdots,\widehat{i}\cdots,n,1^{'},\cdots,\widehat{j^{'}},\cdots,(l-1)^{'},j^{'},i,(l+1)^{'},\cdots,n^{'}}
                    {1,\cdots,n,1^{'},\cdots,\widehat{l^{'}},\cdots,n^{'}}
                \sgn
                    {1^{'},\cdots,n^{'},1,\cdots,n}
                    {l^{'},1,\cdots,n,1^{'},\cdots,\widehat{l^{'}},\cdots,n^{'}}
                \\
                \coefB{l^{'}}{j^{'} i}
                \basexi{l^{'}}
            \end{pmatrix*}
        \end{pmatrix*}.
    \label{EQ-Kahler identity on Lie agbd-calculation for RHS of (i,j)-no5}
\end{align}
The signs of permutations in (\ref{EQ-Kahler identity on Lie agbd-calculation for RHS of (i,j)-no5}) can be simplified. The simplification can be found in Section \ref{apex-simplification for RHS of (i,j)}. We obtain the simplification of (\ref{EQ-Kahler identity on Lie agbd-calculation for RHS of (i,j)-no5}):
\begin{align}
        \star\,\LieagbdPartialbar\,\star(\basexi{i}\wedge\basedualxi{j^{'}})
    =
        2
        \begin{pmatrix*}[l]
            \displaystyle\sum_{k\neq i}
            (-1)^{n^2}
            \coefD{k}{\order{k,i}}
            \basexi{j^{'}}
            \\
            +
            \\
            \displaystyle\sum_{l^{'}\neq j^{'}}
            (-1)^{n^2}
            \coefB{l^{'}}{l^{'} i}
            \basexi{j^{'}}
            \\
            +
            \\
            \displaystyle\sum_{l^{'}\neq j^{'}}
            -(-1)^{n^2}
            \coefB{l^{'}}{j^{'} i}
            \basexi{l^{'}}
        \end{pmatrix*}
    =
        2(-1)^{n^2}
        \begin{pmatrix*}[l]
            \displaystyle\sum_{k\neq i}
            \coefD{k}{\order{k,i}}
            \basexi{j^{'}}
            \\
            +
            \\
            \displaystyle\sum_{l^{'}\neq j^{'}}
            \coefB{l^{'}}{l^{'} i}
            \basexi{j^{'}}
            \\
            -
            \\
            \displaystyle\sum_{l^{'}\neq j^{'}}
            \coefB{l^{'}}{j^{'} i}
            \basexi{l^{'}}
        \end{pmatrix*}
    .
\label{EQ-Kahler identity on Lie agbd-calculation for RHS of (i,j)-no6}
\end{align}
We rewrite 
$
\displaystyle\sum_{l^{'}\neq j^{'}}
\coefB{l^{'}}{l^{'} i}
\basexi{j^{'}}
=
\begin{pmatrix*}[l]
    \displaystyle\sum_{l^{'}=1}^{n}
    \coefB{l^{'}}{l^{'} i}
    \basexi{j^{'}}
\end{pmatrix*}
-
\coefB{j^{'}}{j^{'} i}
\basexi{j^{'}}
$
and
$
\displaystyle\sum_{l^{'}\neq j^{'}}
\coefB{l^{'}}{j^{'} i}
\basexi{l^{'}}
=
\begin{pmatrix*}[l]
                \displaystyle\sum_{l^{'}=1}^{n}
                \coefB{l^{'}}{j^{'} i}
                \basexi{l^{'}}
            \end{pmatrix*}
            -
                \coefB{j^{'}}{j^{'} i}
                \basexi{j^{'}}
$
. Then, (\ref{EQ-Kahler identity on Lie agbd-calculation for RHS of (i,j)-no6}) can be written into
\begin{align}
        2(-1)^{n^2}
        \begin{pmatrix*}[l]
            \displaystyle\sum_{l=1}^{n}
            \sgn
                {l,i}
                {\order{l,i}}
            \coefD{l}{\order{l,i}}
            \basexi{j^{'}}
            \\
            +
            \\
            \begin{pmatrix*}[l]
                \displaystyle\sum_{l^{'}=1}^{n}
                \coefB{l^{'}}{l^{'} i}
                \basexi{j^{'}}
            \end{pmatrix*}
            -
            \coefB{j^{'}}{j^{'} i}
            \basexi{j^{'}}
            \\
            -
            \\
            (
            \begin{pmatrix*}[l]
                \displaystyle\sum_{l^{'}=1}^{n}
                \coefB{l^{'}}{j^{'} i}
                \basexi{l^{'}}
            \end{pmatrix*}
            -
            \coefB{j^{'}}{j^{'} i}
            \basexi{j^{'}}
            )
        \end{pmatrix*}
    =
        2(-1)^{n^2}
        \begin{pmatrix*}[l]
            \displaystyle\sum_{l=1}^{n}
            \sgn
                {l,i}
                {\order{l,i}}
            \coefD{l}{\order{l,i}}
            \basexi{j^{'}}
            \\
            +
            \\
            \displaystyle\sum_{l^{'}=1}^{n}
            \coefB{l^{'}}{l^{'} i}
            \basexi{j^{'}}
            \\
            -
            \\
            \displaystyle\sum_{l^{'}=1}^{n}
            \coefB{l^{'}}{j^{'} i}
            \basexi{l^{'}}
        \end{pmatrix*}
    .
\label{EQ-Kahler identity on Lie agbd-calculation for RHS of (i,j)-no7}
\end{align}
Using $l$ to replace $l^{'}$ in $\displaystyle\sum_{l^{'}=1}^{n}\coefB{l^{'}}{l^{'} i}\basexi{j^{'}}$, we further simplify (\ref{EQ-Kahler identity on Lie agbd-calculation for RHS of (i,j)-no7}) and obtain
\begin{align}
        2(-1)^{n^2}
        \begin{pmatrix*}[l]
            \displaystyle\sum_{l=1}^{n}
            \sgn
                {l,i}
                {\order{l,i}}
            \coefD{l}{\order{l,i}}
            \basexi{j^{'}}
            \\
            +
            \\
            \displaystyle\sum_{l=1}^{n}
            \coefB{l}{l i}
            \basexi{j^{'}}
            \\
            -
            \\
            \displaystyle\sum_{l^{'}=1}^{n}
            \coefB{l^{'}}{j^{'} i}
            \basexi{l^{'}}
        \end{pmatrix*}
    =
        2(-1)^{n^2}
        \begin{pmatrix*}[l]
            \displaystyle\sum_{l=1}^{n}
            \sgn
                {l,i}
                {\order{l,i}}
            \coefD{l}{\order{l,i}}
            \basexi{j^{'}}
            +
            \coefB{l}{l i}
            \basexi{j^{'}}
            \\
            -
            \\
            \displaystyle\sum_{l^{'}=1}^{n}
            \coefB{l^{'}}{j^{'} i}
            \basexi{l^{'}}
        \end{pmatrix*},
\end{align}
then we use the constraint for coefficients $B$ and $D$ to achieve the final simplification of 
$\star\,\LieagbdPartialbar\,\star(\basexi{i}\wedge\basedualxi{j^{'}})$:
\begin{align*}
        \star\,\LieagbdPartialbar\,\star(\basexi{i}\wedge\basedualxi{j^{'}})
    =
        2(-1)^{n^2}
        \begin{pmatrix*}[l]
            \displaystyle\sum_{l=1}^{n}
            \coefB{i}{l l}
            \basexi{j^{'}}
            \\
            -
            \\
            \displaystyle\sum_{l^{'}=1}^{n}
            \coefB{l^{'}}{j^{'} i}
            \basexi{l^{'}}
        \end{pmatrix*}.
\end{align*}
Recall that $\partial^{*}=-(-1)^{n^2}\star\,\partialbar\,\star$. Thus, we can conclude that the right-hand side
\begin{align*}
    \sqrt{-1}\partial^{*}(\basexi{i}\wedge\basedualxi{j})
    =
    -2\sqrt{-1}
    \begin{pmatrix*}[l]
        \displaystyle\sum_{l=1}^{n}
        \coefB{i}{l l}
        \basexi{j^{'}}
        -
        \displaystyle\sum_{l^{'}=1}^{n}
        \coefB{l^{'}}{j^{'} i}
        \basexi{l^{'}}
     \end{pmatrix*}
\end{align*}
is equal to the left-hand side (\ref{EQ-Kahler identity on Lie agbd-calculation for LHS of (i,j)-result}).

\end{proof}


\begin{lemma}
\label{Lemma-Kahler identity on Lie algebroid-the third lemma}
$
    [\LieagbdPartialbar,\iotaW]
    (f \basexi{\indexI}\wedge\basedualxi{\indexJp})
    =
    \sqrt{-1}\LieagbdPartial^{*}
    (f \basexi{\indexI}\wedge\basedualxi{\indexJp}).
$
\end{lemma}
\begin{proof}
Our proof comprises the following steps:
\begin{enumerate}
    \item Calculate $
    [\LieagbdPartialbar,\iotaW]
    (f\basexi{\indexI}\wedge\basedualxi{\indexJp})$.

    \item Calculate $
    \sqrt{-1}\LieagbdPartial^{*}
    (f\basexi{\indexI}\wedge\basedualxi{\indexJp})
    $.

    \item Calculate $
    (\LieagbdPartialbar f)
    \wedge
    \iotaW(\basexi{\indexI}\wedge\basedualxi{\indexJp})
    -
    (\iotaW\,\LieagbdPartialbar\,f)
    \basexi{\indexI}\wedge\basedualxi{\indexJp}
    $, which is obtained from Step 1.

    \item Calculate $
    -\sqrt{-1}(-1)^{n^{2}}
    \star
    \bigg(
        (\LieagbdPartialbar f)
        \wedge
        \star(\basexi{\indexI}\wedge\basedualxi{\indexJp})
    \bigg)
    $, which is obtained from Step 2.

    \item Assemble the results from Steps 3 and 4 and obtain the expression of $
    \sqrt{-1}\LieagbdPartial^{*}
    (f\basexi{\indexI}\wedge\basedualxi{\indexJp})
    $, then check this is equal to the expression of $
    [\LieagbdPartialbar,\iotaW]
    (f\basexi{\indexI}\wedge\basedualxi{\indexJp})$.
\end{enumerate}
\keyword{Step 1: } Applying $
[\LieagbdPartialbar,\iotaW]
=
\LieagbdPartialbar\,\iotaW
-
\iotaW\,\LieagbdPartialbar
$ to $(f\basexi{\indexI}\wedge\basedualxi{\indexJp})$,
\begin{align*}
    [\LieagbdPartialbar,\iotaW](f\basexi{\indexI}\wedge\basedualxi{\indexJp})
    =&
        \LieagbdPartialbar\,\iotaW
        (f\basexi{\indexI}\wedge\basedualxi{\indexJp})
        -
        \iotaW\,\LieagbdPartialbar
        (f\basexi{\indexI}\wedge\basedualxi{\indexJp})
    \\
    =&
        \LieagbdPartialbar
        \begin{pmatrix*}[l]
            (\iotaW f)
            \wedge
            \basexi{\indexI}\wedge\basedualxi{\indexJp}
            +
            f\,\iotaW(\basexi{\indexI}\wedge\basedualxi{\indexJp})
        \end{pmatrix*}
        -
        \iotaW
        \begin{pmatrix*}[l]
            (\LieagbdPartialbar f)
            \wedge\basexi{\indexI}\wedge\basedualxi{\indexJp}
            +
            f\,
            \LieagbdPartialbar(\basexi{\indexI}\wedge\basedualxi{\indexJp})
        \end{pmatrix*},
\end{align*}
where $\iotaW f=0$. Continuing the calculation,
\begin{align}
    &
        \LieagbdPartialbar
        \begin{pmatrix*}[l]
            f\,\iotaW(\basexi{\indexI}\wedge\basedualxi{\indexJp})
        \end{pmatrix*}
        -
        \iotaW\begin{pmatrix*}[l]
            (\LieagbdPartialbar f)
            \wedge\basexi{\indexI}\wedge\basedualxi{\indexJp}
            +
            f\,
            \LieagbdPartialbar(\basexi{\indexI}\wedge\basedualxi{\indexJp})
        \end{pmatrix*}
\notag
    \\
    =&
        (\LieagbdPartialbar f)
        \wedge
        \iotaW(\basexi{\indexI}\wedge\basedualxi{\indexJp})
        +
        f\,\LieagbdPartialbar\,\iotaW
        (\basexi{\indexI}\wedge\basedualxi{\indexJp})
        -
        \begin{pmatrix*}
            (\iotaW\,\LieagbdPartialbar\,f)
            \basexi{\indexI}\wedge\basedualxi{\indexJp}
            +
            f\cdot\iotaW\,\LieagbdPartialbar
            (\basexi{\indexI}\wedge\basedualxi{\indexJp})
        \end{pmatrix*} 
\notag
    \\
    =&
        f\,(\LieagbdPartialbar\,\iotaW-\iotaW\,\LieagbdPartialbar)
        (\basexi{\indexI}\wedge\basedualxi{\indexJp})
        +
        (\LieagbdPartialbar f)
        \wedge
        \iotaW(\basexi{\indexI}\wedge\basedualxi{\indexJp})
        -
        (\iotaW\,\LieagbdPartialbar\,f)
        \basexi{\indexI}\wedge\basedualxi{\indexJp}
\notag
    \\
    =&
        f\,[\LieagbdPartialbar,\iotaW]
        (\basexi{\indexI}\wedge\basedualxi{\indexJp})
        +
        (\LieagbdPartialbar f)
        \wedge
        \iotaW(\basexi{\indexI}\wedge\basedualxi{\indexJp})
        -
        (\iotaW\,\LieagbdPartialbar\,f)
        \basexi{\indexI}\wedge\basedualxi{\indexJp}
    .
\label{EQ-For f (I,J), calculation both side-expression of LHS}
\end{align}
We can separate (\ref{EQ-For f (I,J), calculation both side-expression of LHS}) into two parts:
\begin{itemize}
    \item The first part is $
    f\,[\LieagbdPartialbar,\iotaW]
    (\basexi{\indexI}\wedge\basedualxi{\indexJp})
    $.

    \item The second part is
    \begin{align}
        (\LieagbdPartialbar f)
        \wedge
        \iotaW(\basexi{\indexI}\wedge\basedualxi{\indexJp})
        -
        (\iotaW\,\LieagbdPartialbar\,f)
        \basexi{\indexI}\wedge\basedualxi{\indexJp}
        ,
    \label{EQ-For f (I,J), calculation both side-expression of LHS, the second part}
    \end{align}
    which is equal to a corresponding term in the expression of $
    \sqrt{-1}\LieagbdPartial^{*}
    (f\basexi{\indexI}\wedge\basedualxi{\indexJp})
    $. We return to this later.
\end{itemize}
\keyword{Step 2:} Recall that $
\partial^{*}=-(-1)^{n^2}\star\,\LieagbdPartialbar\,\star
$. Thus, we complete the following calculation:
\begin{align}
    &
        -\sqrt{-1}(-1)^{n^{2}}
        \star\,\LieagbdPartialbar\,\star
        (f\basexi{\indexI}\wedge\basedualxi{\indexJp})
\notag
    \\
    =&
        -\sqrt{-1}(-1)^{n^{2}}
        \star\,\LieagbdPartialbar
        \big(
            f\,\star
            (\basexi{\indexI}\wedge\basedualxi{\indexJp})
        \big)
\notag
    \\
    =&
        -\sqrt{-1}(-1)^{n^{2}}
        \star
        \bigg(
            (\LieagbdPartialbar f)
            \wedge
            \star(\basexi{\indexI}\wedge\basedualxi{\indexJp})
            +
            f\, \LieagbdPartialbar\,\star
            (\basexi{\indexI}\wedge\basedualxi{\indexJp})
        \bigg)
\notag
    \\
    =&
        -\sqrt{-1}(-1)^{n^{2}}
        \star
        \bigg(
            (\LieagbdPartialbar f)
            \wedge
            \star(\basexi{\indexI}\wedge\basedualxi{\indexJp})
        \bigg)
        +
        \sqrt{-1}f\,
        \bigg(
            -(-1)^{n^{2}}
            \star\,\LieagbdPartialbar\,\star
            (\basexi{\indexI}\wedge\basedualxi{\indexJp})
        \bigg)
\notag
    \\
    =&
        -\sqrt{-1}(-1)^{n^{2}}
        \star
        \bigg(
            (\LieagbdPartialbar f)
            \wedge
            \star(\basexi{\indexI}\wedge\basedualxi{\indexJp})
        \bigg)
        +
        f\,\sqrt{-1}\LieagbdPartial^{*}
        (\basexi{\indexI}\wedge\basedualxi{\indexJp})
    .
\label{EQ-For f (I,J), calculation both side-expression of RHS}
\end{align}
Summarizing the above computation, we can separate (\ref{EQ-For f (I,J), calculation both side-expression of LHS}) into two parts:
\begin{itemize}
    \item The first part is $
    f\,\sqrt{-1}\LieagbdPartial^{*}
    (\basexi{\indexI}\wedge\basedualxi{\indexJp})
    $.

    \item The second part is 
    \begin{align}
        -\sqrt{-1}(-1)^{n^{2}}
        \star
        \bigg(
            (\LieagbdPartialbar f)
            \wedge
            \star(\basexi{\indexI}\wedge\basedualxi{\indexJp})
        \bigg)
        ,
    \label{EQ-For f (I,J), calculation both side-expression of RHS, the second part}
    \end{align}
     which we check in the next step and is equal to (\ref{EQ-For f (I,J), calculation both side-expression of LHS, the second part}).
\end{itemize}
\keyword{Step 3:} Using $
    \iotaW
    =
    -\frac{\sqrt{-1}}{2}\displaystyle\sum_{r=1}^{n}\contract{\basedualxi{r^{'}}}\,\contract{\basexi{r}}
$ and the definition of $\LieagbdPartialbar$, we can write
\begin{align}
    &
        (\LieagbdPartialbar f)
        \wedge
        \iotaW(\basexi{\indexI}\wedge\basedualxi{\indexJp})
        -
        (\iotaW\,\LieagbdPartialbar\,f)
        \basexi{\indexI}\wedge\basedualxi{\indexJp}
\notag
    \\
    =&
        (
        \displaystyle\sum_{j^{'}=1}^{n}
        \rho(\basedualv{j^{'}})(f) 
        \basedualxi{j^{'}}
        )
        \wedge
        (-\frac{\sqrt{-1}}{2})
        \displaystyle\sum_{r=1}^{n}
        \contract{\basedualxi{r^{'}}}\,\contract{\basexi{r}}
        (
        \basexi{\indexI}\wedge\basedualxi{\indexJp}
        )
\label{EQ-For f (I,J), calculation both side-step3, first part in LHS}
    \\
    &   
        -
\notag
    \\
    &
        (-\frac{\sqrt{-1}}{2})
        \displaystyle\sum_{r=1}^{n}
        \contract{\basedualxi{r^{'}}}\,\contract{\basexi{r}}
        \big(
        \displaystyle\sum_{j=1}^{n}
        \rho(\basedualv{j^{'}})(f)\,
        \basedualxi{j^{'}}
        \wedge
        \basexi{\indexI}
        \wedge
        \basedualxi{\indexJp}
        \big)
    .
\label{EQ-For f (I,J), calculation both side-step3, second part in LHS}
\end{align}


To simplify (\ref{EQ-For f (I,J), calculation both side-step3, first part in LHS}), we write 
\begin{align*}
    \contract{\basedualxi{r^{'}}}\,\contract{\basexi{r}}
    (\basexi{\indexI}\wedge\basedualxi{\indexJp})
    =& 
        \contract{\basedualxi{r^{'}}}
        \big(
            2\,
            \sgn
                {\indexI,\indexJp}
                {r,\indexI\backslash\{r\},\indexJp}
            \basexi{\indexI\backslash\{r\}}\wedge\basedualxi{\indexJp}
        \big)
    \\
    =&
        4\,
        \sgn
            {\indexI,\indexJp}
            {r,\indexI\backslash\{r\},\indexJp}
        \sgn
            {\indexI\backslash\{r\},\indexJp}
            {r^{'},\indexI\backslash\{r\},\indexJp\backslash\{r'\}}
        \basexi{\indexI\backslash\{r\}}\wedge\basedualxi{\indexJp\backslash\{r'\}}
    \\
    =&
        4\,
        \sgn
            {\indexI,\indexJp}
            {r,r^{'},\indexI\backslash\{r\},\indexJp\backslash\{r'\}}
        \basexi{\indexI\backslash\{r\}}\wedge\basedualxi{\indexJp\backslash\{r'\}}
\end{align*}
for $r\in\indexI\cap\indexJ$ and $
\contract{\basedualxi{r^{'}}}\,\contract{\basexi{r}}
(\basexi{\indexI}\wedge\basedualxi{\indexJp})
=0
$ for $r\in(\indexI\cup\indexJ)^{C}$. We then plug this into (\ref{EQ-For f (I,J), calculation both side-step3, first part in LHS}):
\begin{align}
    &
        \displaystyle\sum_{j^{'}=1}^{n}
        \rho(\basedualv{j^{'}})(f)\,
        \basedualxi{j^{'}}
        \wedge
        (-\frac{\sqrt{-1}}{2})
        \displaystyle\sum_{r\in\indexI\cap\indexJ}
        4\, \sgn
            {\indexI,\indexJp}
            {r,r^{'},\indexI\backslash\{r\},\indexJp\backslash\{r^{'}\}}
        \basexi{\indexI\backslash\{r\}}
        \wedge
        \basedualxi{\indexJp\backslash\{r^{'}\}}
    .
\label{EQ-For f (I,J), calculation both side-step3, first part in LHS, simplifying no1}
\end{align}
The terms such that $j^{'}\in\indexJp\backslash\{r^{'}\}$ are zero due to the repeated frames in the wedge product. We separate the sum into two cases, $j^{'}\in\indexJpc$ and $j^{'}=r^{'}$, and rewrite (\ref{EQ-For f (I,J), calculation both side-step3, first part in LHS, simplifying no1}):
\begin{align*}
    &
        \begin{pmatrix*}[l]
            \displaystyle\sum_{j^{'}\in\indexJpc}
            \rho(\basedualv{j^{'}})(f)\,
            \basedualxi{j^{'}}
            \wedge
            (-\frac{\sqrt{-1}}{2})
            \displaystyle\sum_{r\in\indexI\cap\indexJ}
            4\, \sgn
                {\indexI,\indexJp}
                {r,r^{'},\indexI\backslash\{r\},\indexJp\backslash\{r^{'}\}}
            \basexi{\indexI\backslash\{r\}}
            \wedge
            \basedualxi{\indexJp\backslash\{r^{'}\}}
            \\
            +
            \\
            \displaystyle\sum_{j^{'}=r^{'}}
            \rho(\basedualv{j^{'}})(f)\,
            \basedualxi{j^{'}}
            \wedge
            (-\frac{\sqrt{-1}}{2})
            \displaystyle\sum_{r\in\indexI\cap\indexJ}
            4\, \sgn
                {\indexI,\indexJp}
                {r,r^{'},\indexI\backslash\{r\},\indexJp\backslash\{r^{'}\}}
            \basexi{\indexI\backslash\{r\}}
            \wedge
            \basedualxi{\indexJp\backslash\{r^{'}\}}
        \end{pmatrix*}
    \\
    =&
        \begin{pmatrix*}[l]
            -2\sqrt{-1}
            \displaystyle\sum_{j^{'}\in\indexJpc}
            \displaystyle\sum_{r\in\indexI\cap\indexJ}
            \sgn
                {\indexI,\indexJp}
                {r,r^{'},\indexI\backslash\{r\},\indexJp\backslash\{r^{'}\}}
            \rho(\basedualv{j^{'}})(f)\,
            \basedualxi{j^{'}}
            \wedge
            \basexi{\indexI\backslash\{r\}}
            \wedge
            \basedualxi{\indexJp\backslash\{r^{'}\}}
            \\
            +
            \\
            -2\sqrt{-1}
            \displaystyle\sum_{r\in\indexI\cap\indexJ}
            \sgn
                {\indexI,\indexJp}
                {r,r^{'},\indexI\backslash\{r\},\indexJp\backslash\{r^{'}\}}
            \rho(\basedualv{r^{'}})(f)\,
            \basedualxi{r^{'}}
            \wedge
            \basexi{\indexI\backslash\{r\}}
            \wedge
            \basedualxi{\indexJp\backslash\{r^{'}\}}
        \end{pmatrix*}
    .
\end{align*}
Checking that
\begin{itemize}
    \item 
$
        \basedualxi{r^{'}}
        \wedge
        \basexi{\indexI\backslash\{r\}}
        \wedge
        \basedualxi{\indexJp\backslash\{r^{'}\}}
    =
        \sgn
            {r^{'},\indexI\backslash\{r\},\indexJp\backslash\{r^{'}\}}
            {\indexI\backslash\{r\},\indexJp}
        \basexi{\indexI\backslash\{r\}}
        \wedge
        \basedualxi{\indexJp}
$,
    \item 
$
        \sgn
            {\indexI,\indexJp}
            {r,r^{'},\indexI\backslash\{r\},\indexJp\backslash\{r^{'}\}}
        \sgn
            {r^{'},\indexI\backslash\{r\},\indexJp\backslash\{r^{'}\}}
            {\indexI\backslash\{r\},\indexJp}
        =
        \sgn
            {\indexI,\indexJp}
            {r,\indexI\backslash\{r\},\indexJp}
$,
\end{itemize}
we conclude that the expression of $
(\LieagbdPartialbar f)
\wedge
\iotaW(\basexi{\indexI}\wedge\basedualxi{\indexJp})
$ is 
\begin{align}
    \begin{pmatrix*}[l]
        (-2\sqrt{-1})
        \displaystyle\sum_{j^{'}\in{\indexJpc}}
        \displaystyle\sum_{r\in\indexI\cap\indexJ}
        \sgn
            {\indexI,\indexJp}
            {r,r^{'},\indexI\backslash\{r\},\indexJp\backslash\{r^{'}\}}
        \rho(\basedualv{j^{'}})(f)\,
        \basedualxi{j^{'}}
        \wedge
        \basexi{\indexI\backslash\{r\}}
        \wedge
        \basedualxi{\indexJp\backslash\{r^{'}\}}
        \\
        +
        \\(-2\sqrt{-1})
        \displaystyle\sum_{r\in\indexI\cap\indexJ}
        \sgn
            {\indexI}
            {r,\indexI\backslash\{r\}}
        \rho(\basedualv{r^{'}})(f)\,
        \basexi{\indexI\backslash\{r\}}
        \wedge
        \basedualxi{\indexJp}
    \end{pmatrix*}
    .
\label{EQ-For f (I,J), calculation both side-step 3, before final result, the simplified first term}
\end{align}

Next, we simplify (\ref{EQ-For f (I,J), calculation both side-step3, second part in LHS}). Observing that $\basedualxi{j^{'}}
\wedge
\basexi{\indexI}
\wedge
\basedualxi{\indexJp}=0
$ when $j^{'}\in\indexJp$, we first simplify (\ref{EQ-For f (I,J), calculation both side-step3, second part in LHS}) to
\begin{align}
    (-\frac{\sqrt{-1}}{2})
    \displaystyle\sum_{r=1}^{n}
    \contract{\basedualxi{r^{'}}}\,\contract{\basexi{r}}
    \big(
        \displaystyle\sum_{j\in\indexJpc}
        \rho(\basedualv{j^{'}})(f)\,
        \basedualxi{j^{'}}
        \wedge
        \basexi{\indexI}
        \wedge
        \basedualxi{\indexJp}
    \big)
    .
\label{EQ-For f (I,J), calculation both side-step3, second part in LHS,simplying no1}
\end{align}
We then observe that $
\contract{\basedualxi{r^{'}}}\,\contract{\basexi{r}}
(
    \basedualxi{j^{'}}
    \wedge
    \basexi{\indexI}
    \wedge
    \basedualxi{\indexJp}
)
$ is equal to zero except for $r\in\indexI\cap\indexJ$ or $r=j\in\indexI\cap\indexJc$. Thus, (\ref{EQ-For f (I,J), calculation both side-step3, second part in LHS,simplying no1}) can be split into
\begin{align}
    &
        \begin{pmatrix*}[l]
            (-\frac{\sqrt{-1}}{2})
            \displaystyle\sum_{r\in\indexI\cap\indexJ}
            \contract{\basedualxi{r^{'}}}\,\contract{\basexi{r}}
            \big(
                \displaystyle\sum_{j\in\indexJpc}
                \rho(\basedualv{j^{'}})(f)\,
                \basedualxi{j^{'}}
                \wedge
                \basexi{\indexI}
                \wedge
                \basedualxi{\indexJp}
            \big)
        \\
        +
        \\
            (-\frac{\sqrt{-1}}{2})
            \displaystyle\sum_{r\in\indexI\cap\indexJc}
            \contract{\basedualxi{r^{'}}}\,\contract{\basexi{r}}
            \big(
                \rho(\basedualv{r^{'}})(f)\,
                \basedualxi{r^{'}}
                \wedge
                \basexi{\indexI}
                \wedge
                \basedualxi{\indexJp}
            \big)
        \end{pmatrix*}
\notag
    \\
    =&
        \begin{pmatrix*}[l]
            (-\frac{\sqrt{-1}}{2})
            \displaystyle\sum_{r\in\indexI\cap\indexJ}
            \displaystyle\sum_{j\in\indexJpc}
            \rho(\basedualv{j^{'}})(f)\,
            \contract{\basedualxi{r^{'}}}\,\contract{\basexi{r}}
            \big(
                \basedualxi{j^{'}}
                \wedge
                \basexi{\indexI}
                \wedge
                \basedualxi{\indexJp}
            \big)
        \\
        +
        \\
            (-\frac{\sqrt{-1}}{2})
            \displaystyle\sum_{r\in\indexI\cap\indexJc}
            \rho(\basedualv{r^{'}})(f)\,
            \contract{\basedualxi{r^{'}}}\,\contract{\basexi{r}}
            \big(
                \basedualxi{r^{'}}
                \wedge
                \basexi{\indexI}
                \wedge
                \basedualxi{\indexJp}
            \big)
        \end{pmatrix*},
\label{EQ-For f (I,J), calculation both side-step3, second part in LHS,simplying no2}
\end{align}
in which
\begin{align}
    \contract{\basedualxi{r^{'}}}\,\contract{\basexi{r}}
    (
        \basedualxi{j^{'}}
        \wedge
        \basexi{\indexI}
        \wedge
        \basedualxi{\indexJp}
    )
    =&
        \contract{\basedualxi{r^{'}}}
        \big(
            2\,
            \sgn
                {j^{'},\indexI,\indexJp}
                {r,j^{'},\indexI\backslash\{r\},\indexJp}
            \basedualxi{j^{'}}
            \wedge
            \basexi{\indexI\backslash\{r\}}
            \wedge
            \basedualxi{\indexJp}
        \big)
\notag
    \\
    =&
        4\,
        \sgn
            {j^{'},\indexI,\indexJp}
            {r,j^{'},\indexI\backslash\{r\},\indexJp}
        \sgn
            {j^{'},\indexI\backslash\{r\},\indexJp}
            {r^{'},j^{'},\indexI\backslash\{r\},\indexJp\backslash\{r^{'}\}}
        \basedualxi{j^{'}}
        \wedge
        \basexi{\indexI\backslash\{r\}}
        \wedge
        \basedualxi{\indexJp\backslash\{r^{'}\}}
\notag
    \\
    =&
        4\,
        \sgn
            {j^{'},\indexI,\indexJp}
            {r,r^{'},j^{'},\indexI\backslash\{r\},\indexJp\backslash\{r^{'}\}}
        \basedualxi{j^{'}}
        \wedge
        \basexi{\indexI\backslash\{r\}}
        \wedge
        \basedualxi{\indexJp\backslash\{r^{'}\}}
    \\
    =&
        4\,
        \sgn
            {\indexI,\indexJp}
            {r,r^{'},\indexI\backslash\{r\},\indexJp\backslash\{r^{'}\}}
        \basedualxi{j^{'}}
        \wedge
        \basexi{\indexI\backslash\{r\}}
        \wedge
        \basedualxi{\indexJp\backslash\{r^{'}\}},
\label{EQ-For f (I,J), calculation both side-step3, second part in LHS,calculate contraction no1}
\end{align}
where the number $2$ in the first row is obtained from $\contract{\basexi{r}}(\basexi{r})=2$ since $\norm{\basexi{r}}{h}=2$, and
\begin{align}
    \contract{\basedualxi{r^{'}}}\,\contract{\basexi{r}}
    (
        \basedualxi{r^{'}}
        \wedge
        \basexi{\indexI}
        \wedge
        \basedualxi{\indexJp}
    )
    =&
        \contract{\basedualxi{r^{'}}}
        \big(
            2\,
            \sgn
                {r^{'},\indexI,\indexJp}
                {r,r^{'},\indexI\backslash\{r\},\indexJp}
            \basedualxi{r^{'}}
            \wedge
            \basexi{\indexI\backslash\{r\}}
            \wedge
            \basedualxi{\indexJp}
        \big)
\notag
    \\
    =&
        4\,
        \sgn
            {r^{'},\indexI,\indexJp}
            {r,r^{'},\indexI\backslash\{r\},\indexJp}
        \basexi{\indexI\backslash\{r\}}
        \wedge
        \basedualxi{\indexJp}
\notag
    \\
    =&
        -4\,
        \sgn
            {\indexI,\indexJp}
            {r,\indexI\backslash\{r\},\indexJp}
        \basexi{\indexI\backslash\{r\}}
        \wedge
        \basedualxi{\indexJp}
    .
\label{EQ-For f (I,J), calculation both side-step3, second part in LHS,calculate contraction no2}
\end{align}
Plugging (\ref{EQ-For f (I,J), calculation both side-step3, second part in LHS,calculate contraction no1}) and (\ref{EQ-For f (I,J), calculation both side-step3, second part in LHS,calculate contraction no2}) into (\ref{EQ-For f (I,J), calculation both side-step3, second part in LHS,simplying no2}), we conclude that the expression of $
(\iotaW\,\LieagbdPartialbar\,f)
\basexi{\indexI}\wedge\basedualxi{\indexJp}
$ is
\begin{align}
    &
        \begin{pmatrix*}[l]
        (-\frac{\sqrt{-1}}{2})
        \displaystyle\sum_{r\in\indexI\cap\indexJ}
        \displaystyle\sum_{j\in\indexJpc}
        \rho(\basedualv{j^{'}})(f)\,
        4\,
        \sgn
            {\indexI,\indexJp}
            {r,r^{'},\indexI\backslash\{r\},\indexJp\backslash\{r^{'}\}}
        \basedualxi{j^{'}}
        \wedge
        \basexi{\indexI\backslash\{r\}}
        \wedge
        \basedualxi{\indexJp\backslash\{r^{'}\}}
        \\
        +
        \\
        (-\frac{\sqrt{-1}}{2})
        \displaystyle\sum_{r\in\indexI\cap\indexJc}
        \rho(\basedualv{r^{'}})(f)\,
        (-4)\,
        \sgn
            {\indexI}
            {r,\indexI\backslash\{r\}}
        \basexi{\indexI\backslash\{r\}}
        \wedge
        \basedualxi{\indexJp}
    \end{pmatrix*}
\notag
    \\
    =&
        \begin{pmatrix*}[l]
        (-2\sqrt{-1})
        \displaystyle\sum_{r\in\indexI\cap\indexJ}
        \displaystyle\sum_{j\in\indexJpc}
        \sgn
            {\indexI,\indexJp}
            {r,r^{'},\indexI\backslash\{r\},\indexJp\backslash\{r^{'}\}}
        \rho(\basedualv{j^{'}})(f)\,
        \basedualxi{j^{'}}
        \wedge
        \basexi{\indexI\backslash\{r\}}
        \wedge
        \basedualxi{\indexJp\backslash\{r^{'}\}}
        \\
        +
        \\
        (-2\sqrt{-1})
        \displaystyle\sum_{r\in\indexI\cap\indexJc}
        (-1)
        \sgn
            {\indexI}
            {r,\indexI\backslash\{r\}}
        \rho(\basedualv{r^{'}})(f)\,
        \basexi{\indexI\backslash\{r\}}
        \wedge
        \basedualxi{\indexJp}
    \end{pmatrix*}
    .
\label{EQ-For f (I,J), calculation both side-step 3, before final result, the simplified second term}
\end{align}

Finally, we conclude that $
(\LieagbdPartialbar f)
\wedge
\iotaW(\basexi{\indexI}\wedge\basedualxi{\indexJp})
-
(\iotaW\,\LieagbdPartialbar\,f)
\basexi{\indexI}\wedge\basedualxi{\indexJp}
$ is equal to (\ref{EQ-For f (I,J), calculation both side-step 3, before final result, the simplified first term}) \(-\) (\ref{EQ-For f (I,J), calculation both side-step 3, before final result, the simplified second term}), which is
\begin{align*}
    &
    \begin{pmatrix*}[l]
        (-2\sqrt{-1})
        \displaystyle\sum_{j^{'}\in{\indexJpc}}
        \displaystyle\sum_{r\in\indexI\cap\indexJ}
        \sgn
            {\indexI,\indexJp}
            {r,r^{'},\indexI\backslash\{r\},\indexJp\backslash\{r^{'}\}}
        \rho(\basedualv{j^{'}})(f)\,
        \basedualxi{j^{'}}
        \wedge
        \basexi{\indexI\backslash\{r\}}
        \wedge
        \basedualxi{\indexJp\backslash\{r^{'}\}}
        \\
        +
        \\
        (-2\sqrt{-1})
        \displaystyle\sum_{r\in\indexI\cap\indexJ}
        \sgn
            {\indexI}
            {r,\indexI\backslash\{r\}}
        \rho(\basedualv{r^{'}})(f)\,
        \basexi{\indexI\backslash\{r\}}
        \wedge
        \basedualxi{\indexJp}
    \end{pmatrix*}
    \\
    &-
    \\
    &
    \begin{pmatrix*}[l]
        (-2\sqrt{-1})
        \displaystyle\sum_{r\in\indexI\cap\indexJ}
        \displaystyle\sum_{j\in\indexJpc}
        \sgn
            {\indexI,\indexJp}
            {r,r^{'},\indexI\backslash\{r\},\indexJp\backslash\{r^{'}\}}
        \rho(\basedualv{j^{'}})(f)\,
        \basedualxi{j^{'}}
        \wedge
        \basexi{\indexI\backslash\{r\}}
        \wedge
        \basedualxi{\indexJp\backslash\{r^{'}\}}
        \\
        +
        \\
        (-2\sqrt{-1})
        \displaystyle\sum_{r\in\indexI\cap\indexJc}
        (-1)
        \sgn
            {\indexI}
            {r,\indexI\backslash\{r\}}
        \rho(\basedualv{r^{'}})(f)\,
        \basexi{\indexI\backslash\{r\}}
        \wedge
        \basedualxi{\indexJp}
    \end{pmatrix*}
    .
\end{align*}
The cancellation is clear and results in the final simplified expression:
\begin{align}
    &
        \begin{pmatrix*}[l]
            (-2\sqrt{-1})
            \displaystyle\sum_{r\in\indexI\cap\indexJ}
            \sgn
                {\indexI}
                {r,\indexI\backslash\{r\}}
            \rho(\basedualv{r^{'}})(f)\,
            \basexi{\indexI\backslash\{r\}}
            \wedge
            \basedualxi{\indexJp}
            \\
            +
            \\
            (-2\sqrt{-1})
            \displaystyle\sum_{r\in\indexI\cap\indexJc}
            (-1)
            \sgn
                {\indexI}
                {r,\indexI\backslash\{r\}}
            \rho(\basedualv{r^{'}})(f)\,
            \basexi{\indexI\backslash\{r\}}
            \wedge
            \basedualxi{\indexJp}
        \end{pmatrix*}
\notag
    \\
    =&
        (-2\sqrt{-1})
        \displaystyle\sum_{r\in\indexI}
        \sgn
            {\indexI}
            {r,\indexI\backslash\{r\}}
        \rho(\basedualv{r^{'}})(f)\,
        \basexi{\indexI\backslash\{r\}}
        \wedge
        \basedualxi{\indexJp}.
\label{EQ-For f (I,J), calculation both side-step 3, result}
\end{align}

\keyword{Step 4:} Recalling the definitions of $(\LieagbdPartialbar f)$ and $\star$, we obtain
\begin{align}
    (\LieagbdPartialbar f)
    \wedge
    \star(\basexi{\indexI}\wedge\basedualxi{\indexJp})
    =
    \bigg(
        \displaystyle\sum_{i=1}^{n}
        \rho(\basedualv{i^{'}})(f)\,\basedualxi{i}
    \bigg)
    \wedge
    \bigg(
        2^{p+q-n}
        \sgn
            {1,\cdots,n,1^{'},1^{'},\cdots,n^{'}}
            {\indexI,\indexJp,\indexIc,\indexJpc}
        \basedualxi{\indexIc}\wedge\basexi{\indexJpc}
    \bigg)
    .
\label{EQ-For f (I,J), calculation both side-step4,equation no1}
\end{align}
It can be observed that the terms such that $i\in\indexIc$ are zero due to the repeated $\basedualxi{i}$ in the wedge product of frames. Thus, we simplify (\ref{EQ-For f (I,J), calculation both side-step4,equation no1}) to
\begin{align}
    &
    \bigg(
        \displaystyle\sum_{i\in\indexI}
        \rho(\basedualv{i^{'}})(f)\,\basedualxi{i}
    \bigg)
    \wedge
    \bigg(
        2^{p+q-n}
        \sgn
            {1,\cdots,n,1^{'},1^{'},\cdots,n^{'}}
            {\indexI,\indexJp,\indexIc,\indexJpc}
        \basedualxi{\indexIc}\wedge\basexi{\indexJpc}
    \bigg)
\notag
    \\
    =&
        \displaystyle\sum_{i\in\indexI}
        2^{p+q-n}
        \rho(\basedualv{i^{'}})(f)\,
        \sgn
            {1,\cdots,n,1^{'},1^{'},\cdots,n^{'}}
            {\indexI,\indexJp,\indexIc,\indexJpc}
        \basedualxi{i}
        \wedge
        \basedualxi{\indexIc}\wedge\basexi{\indexJpc}
    .
\label{EQ-For f (I,J), calculation both side-step4,equation no2}
\end{align}
We rearrange (\ref{EQ-For f (I,J), calculation both side-step4,equation no2}) by taking 
\begin{align*}
    \basedualxi{i}
    \wedge
    \basedualxi{\indexIc}\wedge\basexi{\indexJpc}
    =
    \sgn
        {i,\indexIc,\indexJpc}
        {\order{\indexIc\cup\{i\}},\indexJpc}
    \basedualxi{\order{\indexIc\cup\{i\}}}
    \wedge
    \basexi{\indexJpc}
,
\end{align*}
then the simplified expression of $
(\LieagbdPartialbar f)
\wedge
\star(\basexi{\indexI}\wedge\basedualxi{\indexJp})
$ is
\begin{align}
    \displaystyle\sum_{i\in\indexI}
    2^{p+q-n}
    \rho(\basedualv{i^{'}})(f)\,
    \sgn
        {1,\cdots,n,1^{'},1^{'},\cdots,n^{'}}
        {\indexI,\indexJp,\indexIc,\indexJpc}
    \sgn
        {i,\indexIc,\indexJpc}
        {\order{\indexIc\cup\{i\}},\indexJpc}
    \basedualxi{\order{\indexIc\cup\{i\}}}
    \wedge
    \basexi{\indexJpc}
    .
\label{EQ-For f (I,J), calculation both side-step4,equation no3}
\end{align}

Finally, we apply $\star$ to (\ref{EQ-For f (I,J), calculation both side-step4,equation no3}):
\begin{align*}
    \displaystyle\sum_{i\in\indexI}
    2^{p+q-n}
    \rho(\basedualv{i^{'}})(f)\,
    \sgn
        {1,\cdots,n,1^{'},1^{'},\cdots,n^{'}}
        {\indexI,\indexJp,\indexIc,\indexJpc}
    \sgn
        {i,\indexIc,\indexJpc}
        {\order{\indexIc\cup\{i\}},\indexJpc}
    \star
    (
        \basedualxi{\order{\indexIc\cup\{i\}}}
        \wedge
        \basexi{\indexJpc}
    )
    ,
\end{align*}
in which
\begin{align*}
    \star
    (
        \basedualxi{\order{\indexIc\cup\{i\}}}
        \wedge
        \basexi{\indexJpc}
    )
    =&
        2^{(n-p+1)+(n-q)-n}
        \sgn
            {1^{'},\cdots,n^{'},1,\cdots,n}
            {\order{\indexIc\cup\{i\}},\indexJpc,\indexI\backslash\{i\},\indexJp}
        \basexi{\indexI\backslash\{i\}}
        \wedge
        \basedualxi{\indexJp}
    \\
    =&
        2^{n-p-q+1}
        \sgn
            {1^{'},\cdots,n^{'},1,\cdots,n}
            {\order{\indexIc\cup\{i\}},\indexJpc,\indexI\backslash\{i\},\indexJp}
        \basexi{\indexI\backslash\{i\}}
        \wedge
        \basedualxi{\indexJp}
    ,
\end{align*}
so the expression of $
\star
\bigg(
    (\LieagbdPartialbar f)
    \wedge
    \star(\basexi{\indexI}\wedge\basedualxi{\indexJp})
\bigg)
$ is
\begin{align}
    \displaystyle\sum_{j\in\indexJ}
    2\,
    \rho(\basedualv{i^{'}})(f)\,
    \begin{pmatrix*}[l]
        \sgn
            {1,\cdots,n,1^{'},1^{'},\cdots,n^{'}}
            {\indexI,\indexJp,\indexIc,\indexJpc}
        \\
        \sgn
            {i,\indexIc,\indexJpc}
            {\order{\indexIc\cup\{i\}},\indexJpc}
        \\
        \sgn
            {1^{'},\cdots,n^{'},1,\cdots,n}
            {\order{\indexIc\cup\{i\}},\indexJpc,\indexI\backslash\{i\},\indexJp}
    \end{pmatrix*}
    \basexi{\indexI\backslash\{i\}}
    \wedge
    \basedualxi{\indexJp}
    .
\label{EQ-For f (I,J), calculation both side-step4,close to finish, unsimplified}
\end{align}
We simplify the signs of permutations through
\begin{itemize}
    \item 
$
    \sgn
        {i,\indexIc,\indexJpc}
        {\order{\indexIc\cup\{i\}},\indexJpc}
    \sgn
        {1^{'},\cdots,n^{'},1,\cdots,n}
        {\order{\indexIc\cup\{i\}},\indexJpc,\indexI\backslash\{i\},\indexJp}
    =
    \sgn
        {1^{'},\cdots,n^{'},1,\cdots,n}
        {i,\indexIc,\indexJpc,\indexI\backslash\{i\},\indexJp}
$,
    \item 
$
    \sgn
        {1^{'},\cdots,n^{'},1,\cdots,n}
        {i,\indexIc,\indexJpc,\indexI\backslash\{i\},\indexJp}
    =
    (-1)^{(n-p+n-q)(p-1+q)}
    \sgn
        {1^{'},\cdots,n^{'},1,\cdots,n}
        {i,\indexI\backslash\{i\},\indexJp,\indexIc,\indexJpc}
    \\
    =
    (-1)^{(p+q)(p-1+q)}
    \sgn
        {1^{'},\cdots,n^{'},1,\cdots,n}
        {i,\indexI\backslash\{i\},\indexJp,\indexIc,\indexJpc}
    =
    \sgn
        {1^{'},\cdots,n^{'},1,\cdots,n}
        {i,\indexI\backslash\{i\},\indexJp,\indexIc,\indexJpc}
    \\
    =
    \sgn
        {1^{'},\cdots,n^{'},1,\cdots,n}
        {\indexI,\indexJp,\indexIc,\indexJpc}
    \sgn
        {\indexI}
        {i,\indexI\backslash\{i\}}
$,
    \item 
$
    \sgn
        {1^{'},\cdots,n^{'},1,\cdots,n}
        {\indexI,\indexJp,\indexIc,\indexJpc}
    \sgn
        {1,\cdots,n,1^{'},1^{'},\cdots,n^{'}}
        {\indexI,\indexJp,\indexIc,\indexJpc}
    =
    \sgn
        {1^{'},\cdots,n^{'},1,\cdots,n}
        {1,\cdots,n,1^{'},1^{'},\cdots,n^{'}}
    =
    (-1)^{n^2}
$
\end{itemize}
and obtain
\begin{align*}
    \begin{pmatrix*}[l]
        \sgn
            {1,\cdots,n,1^{'},1^{'},\cdots,n^{'}}
            {\indexI,\indexJp,\indexIc,\indexJpc}
        \\
        \sgn
            {i,\indexIc,\indexJpc}
            {\order{\indexIc\cup\{i\}},\indexJpc}
        \\
        \sgn
            {1^{'},\cdots,n^{'},1,\cdots,n}
            {\order{\indexIc\cup\{i\}},\indexJpc,\indexI\backslash\{i\},\indexJp}
    \end{pmatrix*}
    =
    (-1)^{n^2}
    \sgn
        {\indexI}
        {i,\indexI\backslash\{i\}}
    .
\end{align*}

We can now conclude that the expression of $
-\sqrt{-1}(-1)^{n^2}
\star
\bigg(
    (\LieagbdPartialbar f)
    \wedge
    \star(\basexi{\indexI}\wedge\basedualxi{\indexJp})
\bigg)
$ is
\begin{align}
    &
        -\sqrt{-1}(-1)^{n^2}
        \displaystyle\sum_{i\in\indexI}
        2\,
        \rho(\basedualv{i^{'}})(f)\,
        (-1)^{n^2}
        \sgn
            {\indexI}
            {i,\indexI\backslash\{i\}}
        \basexi{\indexI\backslash\{i\}}
        \wedge
        \basedualxi{\indexJp}
\notag
    \\
    =&
        -2\sqrt{-1}
        \displaystyle\sum_{i\in\indexI}
        \sgn
            {\indexI}
            {i,\indexI\backslash\{i\}}
        \rho(\basedualv{i^{'}})(f)\,
        \basexi{\indexI\backslash\{i\}}
        \wedge
        \basedualxi{\indexJp}.
\label{EQ-For f (I,J), calculation both side-step4,result}
\end{align}
\keyword{Step 5:} Recalling Step 1, the left-hand side $[\LieagbdPartialbar,\iotaW]
(f\,\basexi{\indexI}\wedge\basedualxi{\indexJp})$ is equal to 
\begin{align*}
    f\,[\LieagbdPartialbar,\iotaW]
    (\basexi{\indexI}\wedge\basedualxi{\indexJp})
    +
    (\LieagbdPartialbar f)
    \wedge
    \iotaW(\basexi{\indexI}\wedge\basedualxi{\indexJp})
    -
    (\iotaW\,\LieagbdPartialbar\,f)
    \basexi{\indexI}\wedge\basedualxi{\indexJp}
    .
\end{align*}
The second term in the above expression, $
(\LieagbdPartialbar f)
\wedge
\iotaW(\basexi{\indexI}\wedge\basedualxi{\indexJp})
-
(\iotaW\,\LieagbdPartialbar\,f)
\basexi{\indexI}\wedge\basedualxi{\indexJp}
$, is equal to (\ref{EQ-For f (I,J), calculation both side-step 3, result}). Therefore,
\begin{align*}
    [\LieagbdPartialbar,\iotaW]
    (f\,\basexi{\indexI}\wedge\basedualxi{\indexJp})
    =
        \begin{pmatrix*}[l]
            f\,[\LieagbdPartialbar,\iotaW]
            (\basexi{\indexI}\wedge\basedualxi{\indexJp})
            \\
            +
            \\
            (-2\sqrt{-1})
            \displaystyle\sum_{r\in\indexI}
            \sgn
                {\indexI}
                {r,\indexI\backslash\{r\}}
            \rho(\basedualv{r^{'}})(f)\,
            \basexi{\indexI\backslash\{r\}}
            \wedge
            \basedualxi{\indexJp}
        \end{pmatrix*}
    .
\end{align*}

Recalling Step 2, the right-hand side $
\sqrt{-1}\LieagbdPartial^{*}
(f\basexi{\indexI}\wedge\basedualxi{\indexJp})
$ is equal to
\begin{align*}
    f\,\sqrt{-1}\LieagbdPartial^{*}
    (\basexi{\indexI}\wedge\basedualxi{\indexJp})
    +
    (-\sqrt{-1})(-1)^{n^{2}}
    \star
    \bigg(
        (\LieagbdPartialbar f)
        \wedge
        \star(\basexi{\indexI}\wedge\basedualxi{\indexJp})
    \bigg)
    .
\end{align*}
The second term in the above expression, $(-\sqrt{-1})(-1)^{n^{2}}
\star
\bigg(
    (\LieagbdPartialbar f)
    \wedge
    \star(\basexi{\indexI}\wedge\basedualxi{\indexJp})
\bigg)$, is equal to (\ref{EQ-For f (I,J), calculation both side-step4,result}). Therefore,
\begin{align*}
    \sqrt{-1}\LieagbdPartial^{*}
    (f\basexi{\indexI}\wedge\basedualxi{\indexJp})
    =
    \begin{pmatrix*}[l]
        f\,\sqrt{-1}\LieagbdPartial^{*}
        (\basexi{\indexI}\wedge\basedualxi{\indexJp})
        \\
        +
        \\
        (-2\sqrt{-1})
        \displaystyle\sum_{i\in\indexI}
        \sgn
            {\indexI}
            {i,\indexI\backslash\{i\}}
        \rho(\basedualv{i^{'}})(f)\,
        \basexi{\indexI\backslash\{i\}}
        \wedge
        \basedualxi{\indexJp}
    \end{pmatrix*}
    .
\end{align*}
We complete the proof by observing that the left-hand side and the right-hand side are the same.


\end{proof}

\subsection{Lie algebroid Bochner–Kodaira–Nakano identity}
\label{section-Lie algebroid Bochner–Kodaira-Nakano identity}

In classical differential geometry, a holomorphic vector bundle over a complex manifold is a vector bundle compatible with the complex structure of the base manifold. That is to say, all entries of the transition matrices are holomorphic functions. Furthermore, recall that a holomorphic vector bundle equipped with a Hermitian metric is called a Hermitian vector bundle. We generalize these concepts to the K\"{a}hler Lie algebroid.

\begin{definition}
For a K\"{a}hler Lie algebroid $\g \to M$, we say a complex vector bundle $E \to M$ is a \keyword{Lie algebroid holomorphic vector bundle} if its transition matrices 
\begin{align*}
    \{\phi_{\alpha,\beta}:\C^{\rank_{\C}(E)}\to\C^{\rank_{\C}(E)}\}
\end{align*}
associated with an open covering of $M$ are Lie algebroid holomorphic. That is to say, $\LieagbdPartialbar (\phi_{\alpha,\beta})_{i,j}=0$ for all entries of the transition matrix. We say a complex vector bundle is a \keyword{Lie algebroid Hermitian vector bundle} if it is equipped with a Hermitian metric.
\end{definition}
As the entries of the transition matrix are all Lie algebroid holomorphic, the operator $\LieagbdPartialbar$ can act on $E$-valued differential forms. That is to say,
\begin{align*}
    \LieagbdPartialbar:\Apq{p}{q}(M,E)\to\Apq{p}{q+1}(M,E)
\end{align*}
is well defined when $E$ is a Lie algebroid holomorphic vector bundle.

Next, we generalize the Chern connection. Before giving its formal definition, we define the splitting of a connection.
\begin{definition}
Suppose that 
$
    \LieagbdConn{}{}:
    \secsp{M}{E}
    \to
    \Apq{1}{0}(M,E)\oplus\Apq{0}{1}(M,E)
$, then we define the holomorphic part as
\begin{align*}
    \LieagbdConn{1,0}{}:\secsp{M}{E}\to\Apq{1}{0}(M,E)
\end{align*}
and the anti-holomorphic part as
\begin{align*}
    \LieagbdConn{0,1}{}:\secsp{M}{E}\to\Apq{0}{1}(M,E).
\end{align*}
\end{definition}
Using the splitting $\LieagbdConn{}{}=\LieagbdConn{1,0}{}+\LieagbdConn{0,1}{}$, we define the Lie algebroid Chern connection.
\begin{definition}[Lie algebroid Chern connection]
For a Lie algebroid Hermitian vector bundle $E$, a Lie algebroid Chern connection on $E$ is a Lie algebroid connection 
\begin{align*}
    \LieagbdConn{}{}:
    \secsp{M}{E}
    \to
    \Apq{1}{0}(M,E)\oplus\Apq{0}{1}(M,E),
\end{align*}
which is compatible with the Hermitian metric, i.e.,
\begin{align*}
    \LieagbdDiff 
    \langle \alpha,\beta \rangle
    (\gamma)
    =
    \langle \LieagbdChernConn{}{\gamma}\alpha,\beta \rangle
    +
    \langle \alpha,\LieagbdChernConn{}{\gamma}\beta \rangle
\end{align*}
for any $\alpha,\beta,\gamma\in\secsp{M}{E}$, and its anti-holomorphic part $\LieagbdChernConn{0,1}{}=\LieagbdPartialbar:\secsp{M}{E} \to  \Apq{0}{1}(M,E)$.
\end{definition}

\begin{proposition}
There exists a unique Lie algebroid Chern connection for a Lie algebroid Hermitian vector bundle.
\end{proposition}
\begin{proof}
The proof is similar to the proof of the existence and uniqueness of the classical Chern connection. Suppose $\{e_{i}\}$ is a Lie algebroid holomorphic frame for $E$ and set $h_{ij}=\langle e_i,e_j\rangle$. Using the holomorphic frame, we consider the $\Ak{1}(M)$-valued matrix $(\theta_{ij})$ such that $\LieagbdChernConn{}{}(e_{i})=\sum_{j}\theta_{ij}e_j$. Observing that $\LieagbdChernConn{}{}(e_i)\in\Apq{1}{0}(M,E)$ since $\LieagbdPartialbar(e_i)=0$, then
\begin{align*}
    (\LieagbdPartial+\LieagbdPartialbar)h_{i,j}
    =
    \LieagbdDiff h_{i,j}
    &=
    \LieagbdDiff
    \langle e_i,e_j \rangle
    \\
    &=
    \langle \LieagbdChernConn{}{} e_i,e_j \rangle
    +
    \langle e_i,\LieagbdChernConn{}{} e_j \rangle
    \\
    &=
    \langle \sum_{k}\theta_{ik}e_k,e_j \rangle
    +
    \langle e_i,\sum_{k}\theta_{jk}e_k \rangle
    \\
    &=
    \sum_{k}
    \theta_{ik} h_{kj}
    +
    \sum_{k}
    \overline{\theta_{jk}} h_{ik},
\end{align*}
where $\theta_{ik} h_{kj}\in \Apq{1}{0}(M)$ and $\overline{\theta_{jk}} h_{ik}\in \Apq{0}{1}(M)$. This illustrates the relation between the Hermitian metric matrix $(h_{ij})$ and the structure matrix $(\theta_{ij})$ for $\LieagbdChernConn{}{}$ such that the Chern connection is uniquely determined by the Hermitian metric. Moreover, the type $(1,0)$-part $\LieagbdPartial h_{i,j}=\sum_{k}\theta_{ik} h_{kj}$ in the equation gives $\theta=\LieagbdPartial h \cdot h^{-1}$.
\end{proof}

The curvature tensor $R(\LieagbdChernConn{}{})=\LieagbdChernConn{}{}^2$ for the Lie algebroid Chern connection is $(1,1)$-type since $(\LieagbdChernConn{0,1}{})^2=\LieagbdPartialbar^{2}=0$ and the curvature matrix $\Theta^{2,0}=-(\overline{(\Theta)^{0,2}})^{tr}=0$ by the property that the connection matrix $\theta$ is skew-Hermitian. Thus, the expansion of the curvature tensor in terms of $\LieagbdChernConn{1,0}{}$ and $\LieagbdChernConn{0,1}{}$ is
\begin{align*}
    \LieagbdChernConn{1,0}{}
    \,
    \LieagbdChernConn{0,1}{}
    +
    \LieagbdChernConn{0,1}{}
    \,
    \LieagbdChernConn{1,0}{}
    .
\end{align*}

We can further simplify the notations on the right-hand side using the superbracket $[\LieagbdChernConn{1,0}{},\LieagbdChernConn{0,1}{}]$. This involves the concept called superbundle \cite{BGV}.
\begin{definition}
A \keyword{superspace} $V$ is a $\mathbb{Z}_2$-graded vector space $V=V^{+}\oplus V^{-}$. A \keyword{superbundle} $\mathcal{E}\to M$ is a vector bundle $\mathcal{E}\to M$ admitting a splitting $\mathcal{E}=\mathcal{E}^{+}\oplus\mathcal{E}^{-}$ such that $\mathcal{E}_{x}=\mathcal{E}^{+}_{x}\oplus \mathcal{E}^{-}_{x}$ is a superspace for all $x\in M$. 
\end{definition}
An operator $D$ is called even (degree $(D)=0$) if it preserves the degree of the element in $\mathcal{E}$:
\begin{align*}
    D:\mathcal{E}^{+}\to\mathcal{E}^{+}
    \qquad\text{and}\qquad
    D:\mathcal{E}^{-}\to\mathcal{E}^{-},
\end{align*}
 and is called odd (degree $(D)=1$) if it swamps the degree of the elements in $\mathcal{E}$:
 \begin{align*}
     D:\mathcal{E}^{+}\to\mathcal{E}^{-}
     \qquad\text{and}\qquad
     D:\mathcal{E}^{+}\to\mathcal{E}^{-}.
 \end{align*}
With the degree of superoperators, the \keyword{superbracket} is defined by 
\begin{align*}
    [D_1,D_2]
    =
    D_{1}D_{2}
    -(-1)^{|(D_1)\cdot(D_2)}
    D_{2}D_{1}.
\end{align*}
Looking at the space of $E$-valued Lie algebroid forms $\mathcal{E}=\bigoplus_{p,q}\Apq{p}{q}(M,E)$, we define the $\mathbb{Z}_2$ grading
\begin{align*}
    \mathcal{E}^{+}
    =
    \bigoplus_{\text{p+q are even}}
    \Apq{p}{q}(M,E)
    \qquad\text{and}\qquad
    \mathcal{E}^{-}
    =
    \bigoplus_{\text{p+q are odd}}
    \Apq{p}{q}(M,E).
\end{align*}
Following the definitions, we can see that $\LieagbdChernConn{}{}$, $\LieagbdChernConn{1,0}{}$, and $\LieagbdChernConn{0,1}{}$ are all odd-degree operators. Consequently, the superbracket between $\LieagbdChernConn{1,0}{}$ and $\LieagbdChernConn{0,1}{}$ yields
\begin{align*}
    R(\LieagbdChernConn{}{})
    =
    \LieagbdChernConn{1,0}{}\,\LieagbdChernConn{0,1}{}
    -
    (-1)^{(-1)(-1)}
    \LieagbdChernConn{0,1}{}\,\LieagbdChernConn{1,0}{}
    =
    [\LieagbdChernConn{1,0}{},\LieagbdChernConn{0,1}{}]
    .
\end{align*}

Let us define and generalize some notations and equations that help us achieve the aim: exploring a generalization of the Bochner–Kodaira–Nakano identity on the Lie algebroid. The first one is \keyword{Lie algebroid Laplacians} induced by the Lie algebroid Chern connection:
\begin{align*}
    \LieagbdLpls{1,0}{}&
    :=
        {\LieagbdChernConn{1,0}{}}^{*}
        \,
        \LieagbdChernConn{1,0}{}
    +
        \LieagbdChernConn{1,0}{}
        \,
        {\LieagbdChernConn{1,0}{}}^{*}
    =
    [\LieagbdChernConn{1,0}{},{\LieagbdChernConn{1,0}{}}^*]
    \\
    \LieagbdLpls{0,1}{}&
    :=
    {\LieagbdChernConn{0,1}{}}^{*}\,\LieagbdChernConn{0,1}{} 
    +
    \LieagbdChernConn{0,1}{}\,{\LieagbdChernConn{0,1}{}}^{*}
    =\LieagbdPartialbar^{*}\,\LieagbdPartialbar+\LieagbdPartialbar\,\LieagbdPartialbar^{*}.
\end{align*}
We can use the notation \keyword{$\LieagbdPartialbar$-Laplacian} or $\LieagbdLpls{}{\LieagbdPartialbar}$ for $\LieagbdLpls{0,1}{}$. Then, let us generalize the exterior and interior multiplications of the Lie algebroid K\"{a}hler form. In Section \ref{sec:Lie algebroid Kahler identities}, we have already defined $L$ and $\Lambda$ on the space of Lie algebroid forms in Theorem \ref{thm-Lie algebroid K\"{a}hler identities}. We generalize these to the bundle-valued Lie algebroid forms,
\begin{align*}
    L:=\Apq{p}{q}(M,E)
    &\to\Apq{p+1}{q+1}(M,E)
    \\
    \phi\otimes e
    &\mapsto \omega\wedge\phi\otimes e,
\end{align*}
and the formal adjoint of $L$,
\begin{align*}
    \Lambda:=L^*:\Apq{p}{q}(M,E)\to\Apq{p-1}{q-1}(M,E),
\end{align*}
on the space of $E$-valued Lie algebroid differential forms. Another generalization concerns the Lie algebroid K\"{a}hler identities on bundle-valued $(p,q)$-forms.
\begin{proposition}
For a K\"{a}hler Lie algebroid $\g\to M$ and a Lie algebroid Hermitian vector bundle $E\to M$, the Lie algebroid K\"{a}hler identities on $E$-valued $(p,q)$-forms are:
\begin{align*}
    [\Lambda,\LieagbdChernConn{1,0}{}]&=\sqrt{-1}{\LieagbdChernConn{0,1}{}}^{*},
    \\
    [\Lambda,\LieagbdChernConn{0,1}{}]&=\sqrt{-1}{\LieagbdChernConn{1,0}{}}^{*},
    \\
    [{\LieagbdChernConn{1,0}{}}^*,L]&=\sqrt{-1}\LieagbdChernConn{0,1}{},
    \\
    [{\LieagbdChernConn{0,1}{}}^*,L]&=\sqrt{-1}\LieagbdChernConn{1,0}{}.
\end{align*}
\end{proposition}
\begin{proof}
We can pick a local normal frame on $E$ such that the connection $\LieagbdChernConn{1,0}{}=\LieagbdDiff+\phi$, where $\phi$ vanishes at the origin of the local coordinate. Then, we use $\LieagbdChernConn{1,0}{}$ and $\LieagbdChernConn{0,1}{}$ to replace $\LieagbdPartial$ and $\LieagbdPartialbar$ in the Lie algebroid K\"{a}hler identities in Theorem \ref{thm-Lie algebroid K\"{a}hler identities}.
\end{proof}

Using all the ingredients we have defined and constructed, we can obtain the Lie algebroid version of the Bochner–Kodaira–Nakano identity, which formally agrees with the identity on a complex manifold (e.g., \cite{GH1994}), and the proof is similar as well.
\begin{lemma}
    For a K\"{a}hler Lie algebroid, the equation
    \begin{align}
        \LieagbdLpls{}{\LieagbdPartialbar}
        =
        \LieagbdLpls{1,0}{}
        +
        \sqrt{-1}[\LieagbdChernConn{2}{},\Lambda]
        \label{EQ-Part 1-(3)-lie agbd BK identity}
    \end{align}
    holds and is called the \keyword{Lie algebroid Bochner–Kodaira–Nakano identity}.
\end{lemma}
\begin{proof}
Following the definition of $\LieagbdLpls{}{\LieagbdPartialbar}$,
\begin{align*}
    \LieagbdLpls{}{\LieagbdPartialbar}
    &=
    [\LieagbdPartialbar,\LieagbdPartialbar^{*}]
    \\
    &=
    [\LieagbdChernConn{0,1}{},{\LieagbdChernConn{0,1}{}}^*]
    .
\end{align*}
We use the identity $[\Lambda,\LieagbdChernConn{1,0}{}]=\sqrt{-1}{\LieagbdChernConn{0,1}{}}^{*}$ and plug it into the expansion of the above Laplacian, obtaining the following formula:
\begin{align}
    \LieagbdLpls{}{\LieagbdPartialbar}
    =
    -\sqrt{-1}
    \big[\LieagbdChernConn{0,1}{},[\Lambda,\LieagbdChernConn{1,0}{}]\big].
    \label{EQ-Part 1-(3)-eq 1 in the pf of lie agbd BK identity}
\end{align}
We now look at the Jacobi identity:
\begin{align*}
    \big[ \LieagbdChernConn{0,1}{},[\Lambda,\LieagbdChernConn{1,0}{}]\big]
    +
    \big[ \Lambda,[\LieagbdChernConn{1,0}{},\LieagbdChernConn{0,1}{}]\big]
    +
    \big[ \LieagbdChernConn{1,0}{},[\LieagbdChernConn{0,1}{},\Lambda]\big]
    =0.
\end{align*}
We replace the right-hand side of (\ref{EQ-Part 1-(3)-eq 1 in the pf of lie agbd BK identity}) with the last two terms in the Jacobi identity and deduce
\begin{align}
    \LieagbdLpls{}{\LieagbdPartialbar}
    &=
    \sqrt{-1}
    \big[ \Lambda,[\LieagbdChernConn{1,0}{},\LieagbdChernConn{0,1}{}]\big]
    +
    \sqrt{-1}
    \big[ \LieagbdChernConn{1,0}{},[\LieagbdChernConn{0,1}{},\Lambda]\big]
    \nonumber\\
    &=
    \sqrt{-1}
    \big[ \Lambda,[\LieagbdChernConn{1,0}{},\LieagbdChernConn{0,1}{}]\big]
    +
    \sqrt{-1}
    \big[ \LieagbdChernConn{1,0}{},-\sqrt{-1}{\LieagbdChernConn{1,0}{}}^*\big]
    \nonumber\\
    &=
    \sqrt{-1}
    \big[ \Lambda,[\LieagbdChernConn{1,0}{},\LieagbdChernConn{0,1}{}]\big]
    +
    [\LieagbdChernConn{1,0}{},{\LieagbdChernConn{1,0}{}}^{*}],
\label{EQ-Part 1-(3)-eq 2 in the pf of lie agbd BK identity}
\end{align}
where we use the Lie algebroid K\"{a}hler identity $
[\Lambda,\LieagbdChernConn{0,1}{}]
=
\sqrt{-1}{\LieagbdChernConn{1,0}{}}^*
$ in the second equal sign. 

Recalling that $\LieagbdLpls{1,0}{}=[\LieagbdChernConn{1,0}{},{\LieagbdChernConn{1,0}{}}^{*}]$ and $\LieagbdChernConn{2}{}=[\LieagbdChernConn{1,0}{},\LieagbdChernConn{0,1}{}]$, we conclude the Lie algebroid Bochner–Kodaira–Nakano identity from (\ref{EQ-Part 1-(3)-eq 2 in the pf of lie agbd BK identity}).
\end{proof}

\subsection{Main theory}
\label{section-Lie algebroid Kodaira vanishing theorem}
Applying the Lie algebroid Bochner–Kodaira–Nakano identity, we can establish the Lie algebroid Kodaira vanishing theorem. The original statement of the theorem on the complex manifold is that the cohomology group of a positive line bundle vanishes when its degree is large enough. For a Lie algebroid, however, the Hodge theory between the space of harmonic E-valued Lie algebroid $(p,q)$-forms, i.e., $\kernel(\Delta_{\LieagbdPartialbar}|_{\Apq{p}{q}(M,E)})$, and the sheaf cohomology of $E$ for a Lie algebroid holomorphic vector bundle $E$ over $M$, i.e., $H^p(M,\Omega^q(E))$, is yet to be developed. Our vanishing theorem is slightly weaker than the classical case and gives some sufficient conditions: the type $(p,q)$ such that $p+q$ is large enough, which implies the kernel of $\nabla_{\LieagbdPartialbar}$ is trivial. First, let us define the Lie algebroid positive line bundle.


\begin{definition}
    A Lie algebroid holomorphic line bundle $E\to M$ is \keyword{positive}, if there is a hermitian metric on $E$ such that the curvature $R\in\Apq{1}{1}(M)$ of the Chern connection satisfies that $\sqrt{-1}R$ is positive definite. Meanwhile, we say $E\to M$ is \keyword{negative} if $-\sqrt{-1}R$ is positive definite.
\end{definition}

Recall that the operator $L$ is the exterior multiplication of $\omega$ and $\Lambda$ is its dual. Both operators are even-degree operators. In the classical case, their commutator $[L,\Lambda]=L\Lambda-\Lambda L$ is a scalar multiplication, and this result can be generalized to the K\"{a}hler Lie algebroid.
\begin{lemma}
For a K\"{a}hler Lie algebroid $(M,\g,\omega)$ and the vector bundle $E\to M$, the commutator $[L,\Lambda]$ on $\Apq{p}{q}(M,E)$ is equal to $(p+q-n)\id$, where $\id:\Apq{p}{q}(M,E)\to\Apq{p}{q}(M,E)$ is an identity map.
\end{lemma}
\begin{proof}
    It is enough to verify the identity on a local neighborhood. Recall that locally 
    $\omega=
        \frac{\sqrt{-1}}{2}
        \displaystyle\sum_{k=1}^{n}
        \basexi{k}\wedge\basedualxi{k^{'}}
    $. Using the notations that $e_{\basexi{k}}:=\basexi{k}\wedge$ and $\contract{\basexi{k}}:=(e_{\basexi{k}})^{*}$, we can compute the expressions of $L$ and $\Lambda$:
    \begin{align*}
        L=e_{\omega}=
        \frac{\sqrt{-1}}{2}
        \displaystyle\sum_{k=1}^{n}
        e_{\basexi{k}}
        \,
        e_{\basedualxi{k^{'}}}
        \qquad\text{and}\qquad
        \Lambda=\iotaW=
        -\frac{\sqrt{-1}}{2}
        \displaystyle\sum_{k=1}^{n}
        \contract{\basedualxi{k^{'}}}
        \,
        \contract{\basexi{k}}.
    \end{align*}
    Then, for any $(p,q)$-Lie algebroid form $\basexi{\indexI}\wedge\basedualxi{\indexJp}$, we can expect that the applications of $L$ and $\Lambda$ on $\basexi{\indexI}\wedge\basedualxi{\indexJp}$ will be similar to the classical case. 
    
    We first compute $
L\,\Lambda
(\basexi{\indexI}\wedge\basedualxi{\indexJp})
$. The following details the calculation:
\begin{align*}
   \Lambda(\basexi{\indexI}\wedge\basedualxi{\indexJp})
   &=
    -\frac{\sqrt{-1}}{2}
    \displaystyle\sum_{k=1}^{n}
    \contract{\basedualxi{k^{'}}}
    \,
    \contract{\basexi{k}}
    (\basexi{\indexI}\wedge\basedualxi{\indexJp})
   \\
   &=
    -\frac{\sqrt{-1}}{2}
    \displaystyle\sum_{k\in\indexI}
    \contract{\basedualxi{k}}
    (
    2\,
    \sgn
         {\indexI}
         {k,\indexI\backslash{\{k\}}}
    \basexi{\indexI\backslash\{k\}}\wedge\basedualxi{\indexJp}
    )
   \\
   &=
    -\frac{\sqrt{-1}}{2}
    \displaystyle\sum_{k\in\indexI\cap\indexJ}
    4\,
    \sgn
         {\indexI}
         {k,\indexI\backslash{\{k\}}}
    \sgn
        {\indexI\backslash{\{k\}},\indexJp}
        {k^{'},\indexI\backslash{\{k\}},\indexJp\backslash{\{k^{'}\}}}
    \basexi{\indexI\backslash\{k\}}
    \wedge
    \basedualxi{\indexJp\backslash{\{k^{'}\}}}
   \\
   &=
    -2\sqrt{-1}
    \displaystyle\sum_{k\in\indexI\cap\indexJ}
    \sgn
        {\indexI,\indexJp}
        {k,k^{'},\indexI\backslash{\{k\}},\indexJp\backslash{\{k^{'}\}}}
    \basexi{\indexI\backslash\{k\}}
    \wedge
    \basedualxi{\indexJp\backslash{\{k^{'}\}}},
\end{align*}
    
\begin{align*}
    L\,\Lambda
    (\basexi{\indexI}\wedge\basedualxi{\indexJp})
    &=
        \frac{\sqrt{-1}}{2}
        \displaystyle\sum_{h=1}^{n}
        e_{\basexi{h}}
        \,
        e_{\basedualxi{h^{'}}}
        (-2\sqrt{-1})
        \displaystyle\sum_{k\in\indexI\cap\indexJ}
        \sgn
            {\indexI,\indexJp}
            {k,k^{'},\indexI\backslash{\{k\}},\indexJp\backslash{\{k^{'}\}}}
        \basexi{\indexI\backslash\{k\}}
        \wedge
        \basedualxi{\indexJp\backslash{\{k^{'}\}}}
    \\
    &=
        \begin{pmatrix*}[l]
            \displaystyle\sum_{k\in\indexI\cap\indexJ}
            \basexi{k}\wedge\basedualxi{k^{'}}\wedge
            \sgn
                {\indexI,\indexJp}
                {k,k^{'},\indexI\backslash{\{k\}},\indexJp\backslash{\{k^{'}\}}}
            \basexi{\indexI\backslash\{k\}}
            \wedge
            \basedualxi{\indexJp\backslash{\{k^{'}\}}}
            \\
            +
            \\
            \displaystyle\sum_{h\in\indexIc\cap\indexJc}
            \basexi{h}\wedge\basedualxi{h^{'}}\wedge
            \displaystyle\sum_{k\in\indexI\cap\indexJ}
            \sgn
                {\indexI,\indexJp}
                {k,k^{'},\indexI\backslash{\{k\}},\indexJp\backslash{\{k^{'}\}}}
            \basexi{\indexI\backslash\{k\}}
            \wedge
            \basedualxi{\indexJp\backslash{\{k^{'}\}}}
        \end{pmatrix*}
    \\
    &=
        \begin{pmatrix*}[l]
            \displaystyle\sum_{k\in\indexI\cap\indexJ}
            \basexi{\indexI}
            \wedge
            \basedualxi{\indexJp}
            \\
            +
            \\
            \displaystyle\sum_{h\in\indexIc\cap\indexJc}
            \basexi{h}\wedge\basedualxi{h^{'}}\wedge
            \displaystyle\sum_{k\in\indexI\cap\indexJ}
            \sgn
                {\indexI,\indexJp}
                {k,k^{'},\indexI\backslash{\{k\}},\indexJp\backslash{\{k^{'}\}}}
            \basexi{\indexI\backslash\{k\}}
            \wedge
            \basedualxi{\indexJp\backslash{\{k^{'}\}}}
        \end{pmatrix*}
        .
\end{align*}
Next, we compute $
\Lambda\,L
(\basexi{\indexI}\wedge\basedualxi{\indexJp})
$. The following details the calculation:
\begin{align*}
    L(\basexi{\indexI}\wedge\basedualxi{\indexJp})
    =&
        \frac{\sqrt{-1}}{2}
        \displaystyle\sum_{k\in\indexIc\cap\indexJc}
        \basexi{k}\wedge\basedualxi{k^{'}}\wedge
        \basexi{\indexI}\wedge\basedualxi{\indexJp}
        ,
\end{align*}

\begin{align*}
    \Lambda\, L
    (\basexi{\indexI}\wedge\basedualxi{\indexJp})
    =&
        -\frac{\sqrt{-1}}{2}
        \displaystyle\sum_{h=1}^{n}
        \contract{\basedualxi{h^{'}}}
        \,
        \contract{\basexi{h}}
        \big(
        \frac{\sqrt{-1}}{2}
        \displaystyle\sum_{k\in\indexIc\cap\indexJc}
        \basexi{k}\wedge\basedualxi{k^{'}}\wedge
        \basexi{\indexI}\wedge\basedualxi{\indexJp}
        \big)
    \\
    =&
        \frac{1}{4}
        \begin{pmatrix*}[l]
            \displaystyle\sum_{h\in\indexIc\cap\indexJc}
            \contract{\basedualxi{h^{'}}}
            \,
            \contract{\basexi{h}}
            (
            \displaystyle\sum_{k\in\indexIc\cap\indexJc}
            \basexi{k}\wedge\basedualxi{k^{'}}\wedge
            \basexi{\indexI}\wedge\basedualxi{\indexJp}
            )
            \\
            +
            \\
            \displaystyle\sum_{h\in\indexI\cap\indexJ}
            \contract{\basedualxi{h^{'}}}
            \,
            \contract{\basexi{h}}
            (
            \displaystyle\sum_{k\in\indexIc\cap\indexJc}
            \basexi{k}\wedge\basedualxi{k^{'}}\wedge
            \basexi{\indexI}\wedge\basedualxi{\indexJp}
            )
        \end{pmatrix*}
    \\
    =&
        \frac{1}{4}
        \begin{pmatrix*}[l]
            \displaystyle\sum_{k\in\indexIc\cap\indexJc}
            4\,
            \basexi{\indexI}\wedge\basedualxi{\indexJp}
            \\
            +
            \\
            \displaystyle\sum_{h\in\indexI\cap\indexJ}
            4\,
            \sgn
                {k,k^{'},\indexI\backslash\{h\},\indexJp}
                {h^{'},k,k^{'},\indexI\backslash\{h\},\indexJp\backslash\{h^{'}\}}
            \sgn
                {k,k^{'},\indexI}
                {h,k,k^{'},\indexI\backslash\{h\}}
            \displaystyle\sum_{k\in\indexIc\cap\indexJc}
            \basexi{k}\wedge\basedualxi{k^{'}}\wedge
            \basexi{\indexI\backslash\{h\}}
            \wedge
            \basedualxi{\indexJp\backslash\{h^{'}\}}
        \end{pmatrix*}
    \\
    =&
        \begin{pmatrix*}[l]
            \displaystyle\sum_{k\in\indexIc\cap\indexJc}
            \basexi{\indexI}\wedge\basedualxi{\indexJp}
            \\
            +
            \\
            \displaystyle\sum_{h\in\indexI\cap\indexJ}
            \sgn
                {\indexI,\indexJp}
                {h,h^{'},\indexI\backslash\{h\},\indexJp\backslash\{h^{'}\}}
            \displaystyle\sum_{k\in\indexIc\cap\indexJc}
            \basexi{k}\wedge\basedualxi{k^{'}}\wedge
            \basexi{\indexI\backslash\{h\}}
            \wedge
            \basedualxi{\indexJp\backslash\{h^{'}\}}
        \end{pmatrix*}.
\end{align*}
    From the calculations above, we cancel the repeated terms and obtain $[L,\Lambda]=e_{\omega}\,\contract{\omega}-\contract{\omega}\, e_{\omega}=(|\indexI\cap\indexJ|-|\indexIc\cap\indexJc|)\mathrm{\id}=(p+q-n)\mathrm{\id}$.
\end{proof}

We are now ready to state and prove the vanishing theorem.
\begin{theorem}[Lie algebroid Kodaira–Nakano vanishing theorem]
    Give a K\"{a}hler Lie algebroid $\g\to M$ and a positive line bundle $E\to M$. The kernel of 
    \begin{align*}
        \LieagbdLpls{}{\LieagbdPartialbar}
        :\Apq{p}{q}(M,E)\to\Apq{p}{q}(M,E)
    \end{align*}
    vanishes when $p+q>n$. The same result holds on a negative line bundle under the change of the condition that $p+q<n$.
\end{theorem}
\begin{proof}
By the definition of positivity, the curvature of the Lie algebroid Chern connection for $E$ is a closed positive $(1,1)$-form. Thus, it determines a Lie algebroid K\"{a}hler form $\omega$. Use this Lie algebroid K\"{a}hler form in the Lie algebroid Bochner–Kodaira–Nakano identity
\begin{align*}
    \LieagbdLpls{}{\LieagbdPartialbar}
    =
    \LieagbdLpls{1,0}{}+\sqrt{-1}[\LieagbdChernConn{2}{},\Lambda]
    .
\end{align*}
The $\LieagbdLpls{1,0}{}$ in the equation is a non-negative operator (i.e., $\innerprod{\LieagbdLpls{1,0}{}\phi}{\phi}{}\geq 0, \forall \phi\in \Apq{p}{q}(M,E)$) and $\LieagbdChernConn{2}{}$ is equal to $\sqrt{-1}L$ from the definition of a positive line bundle. Thus, we obtain
\begin{align*}
    \LieagbdLpls{}{\LieagbdPartialbar}
    =
    \LieagbdLpls{1,0}{}-[L,\Lambda]
    =
    \LieagbdLpls{1,0}{}+(p+q-n)\mathrm{\id},
\end{align*}
which this implies that $\LieagbdLpls{}{\LieagbdPartialbar}$ has a trivial kernel if $p+q>n$.
\end{proof}


The following corollary will be applied to the kernel of the (Lie algebroid) Dolbeault operator.
\begin{corollary}
\label{cor-vanishign theorem for Dolbeault}
Suppose $M\to\g$ is a K\"{a}hler Lie algebroid and $E\to M$ a positive Lie algebroid holomorphic line bundle. The kernel of 
\begin{align*}
    \LieagbdPartialbar+\LieagbdPartialbar^{*}
    :\Apq{p}{q}(M,E)
    \to
    \Apq{p}{q+1}(M,E)
    \oplus
    \Apq{p}{q-1}(M,E)
\end{align*}
vanishes when $p+q>n$. The same result holds on a negative line bundle under the change of the condition that $p+q<n$.
\end{corollary}
\begin{proof}
The Laplace operator $
\LieagbdLpls{}{\LieagbdPartialbar}
=
\LieagbdPartialbar^{*}\LieagbdPartialbar
+
\LieagbdPartialbar\LieagbdPartialbar^{*}
=
(\LieagbdPartialbar+\LieagbdPartialbar^{*})^{2}
$, so 
\begin{align*}
    \ker(\LieagbdPartialbar+\LieagbdPartialbar^{*})
    \subset
    \ker(\LieagbdLpls{}{\LieagbdPartialbar})
\end{align*}
is clear. On the other side, if $s\in\ker(\LieagbdLpls{}{\LieagbdPartialbar})$, then
\begin{align*}
    0=
    \pinnerprod
        {\LieagbdLpls{}{\LieagbdPartialbar}s}
        {s}
        {\Apq{p}{q}}
    &=
    \pinnerprod
        {(\LieagbdPartialbar+\LieagbdPartialbar^{*})^{2} s}
        {t}
        {\Apq{p}{q}}
    \\
    &=
    \pinnerprod
        {(\LieagbdPartialbar+\LieagbdPartialbar^{*}) s}
        {(\LieagbdPartialbar+\LieagbdPartialbar^{*}) s}
        {\Apq{p}{q+1}(M,E)
        \oplus
        \Apq{p}{q-1}(M,E)}
    \\
    &=
    ||(\LieagbdPartialbar+\LieagbdPartialbar^{*}) s||^{2}
    ,
\end{align*}
which implies that $s\in\ker(\LieagbdPartialbar+\LieagbdPartialbar^{*})$ as well, thereby completing the proof of Corollary \ref{cor-vanishign theorem for Dolbeault}.
\end{proof}


\subsection{Application of a b-manifold}
\label{sec: app of b manifold}

The notion of b-manifold was introduced by Melrose \cite{M1993} as the framework for studying differential calculus and differential operators on the manifold with boundary. However, our vanishing theorem is based on manifolds without boundary. Thus, our notation of b-manifold follows \cite{GMP2014, GZ2020}, making it a manifold with a distinguished hypersurface.

\begin{definition}
A \keyword{b-manifold} is a pair $(M,Z)$ consisting of a manifold $M$ and a codimension-one submanifold $Z\subset M$. The Serre–Swan theorem assumes that there is a vector bundle $\btgbd{M}\to M$ such that the space of its sections is equal to the space of vector fields on $M$ that are tangent to $Z$. We denote $\btgbd{M}$ the \keyword{b-tangent bundle}.
\end{definition}

The space of vector fields which are tangent to $Z$ is a locally free $C^{\infty}(M)$-module with generators $\{ x_1\ddx{x_1},\ddx{x_2},\cdots,\ddx{x_n} \}$ in a coordinate chart $(x_1,\cdots,x_n)$ that is adapted to $Z=\{ x_1=0 \}$. The natural inclusion $\secsp{M}{\btgbd{M}}\xhookrightarrow{}\secsp{M}{TM}$ induces a vector bundle morphism $\rho:\btgbd{M}\to\tangentbd{M}$, which is an isomorphism on $\btgbd{M}|_{M\backslash Z}$ and an epimorphism on $\btgbd{M}|_{Z}$. Using the vector bundle morphism $\rho$, we can construct a local frame $X_1,\cdots,X_n$ of $\btgbd{M}$ near the coordinate chart adapted to $Z$ which is determined by
\begin{align}
    \rho(X_1)=x_1 \ddx{x_1}
    \qquad
    \rho(X_2)=\ddx{x_2}
    \qquad\cdots\qquad
    \rho(X_n)=\ddx{x_n}.
\label{eq-b manifold-list of rho for the frame on the neighbourhood of Z}
\end{align}

The space of sections of $\btgbd{M}$ is also equipped with a Lie bracket inheriting from the Lie bracket of vector fields on $M$. To describe it more precisely, the Lie bracket on $\btgbd{M}|_{M\backslash Z}$ is entirely determined by the Lie bracket of vector fields on $\tangentbd{M}|_{M\backslash Z}$ through the isomorphism $\rho:\btgbd{M}|_{M\backslash Z}\to\tangentbd{M}|_{M\backslash Z}$, and the neighborhood containing $Z$ is determined through (\ref{eq-b manifold-list of rho for the frame on the neighbourhood of Z}). Then, $(\btgbd{M},\rho,[,])$ forms a Lie algebroid.

The complex b-manifold is introduced by Mendoza \cite{M2013}. Its definition is as follows.
\begin{definition}
Let $(M,Z)$ be a b-manifold of even real dimension. An \keyword{almost CR b-structure} on b-manifold is a subbundle $\mathcal{W}$ of the complexification $\btgbdC{M}$ such that $\mathcal{W}\cap\overline{\mathcal{W}}=0$, where $\mathcal{W}$ is the conjugation of $\mathcal{W}$. The "almost" can be dropped if $\overline{\mathcal{W}}$ is involutive, i.e., $[s_1,s_2] \subset \secsp{M}{\overline{\mathcal{W}}}$ for any $s_1,s_2\in\secsp{M}{\overline{\mathcal{W}}}$.
\end{definition}

An involutive almost CR b-structure on $\btgbd{M}$ is equivalent to the existence of an almost complex structure $J$ on $\btgbd{M}$ whose Nijenhuis tensor vanishes. Therefore, the b-tangent bundle $\btgbd{M}$ is a complex Lie algebroid. The vector bundle morphism $\rho:\btgbd{M}\to\tangentbd{M}$ can be extended to an anchor map $\rho:\btgbd{M}\to\tangentbd{M}^{\mathbb{C}}$ through the natural inclusion $\tangentbd{M}\subset\tangentbd{M}^{\mathbb{C}}$.

Many notations in the b-category agree with the corresponding notations in the Lie algebroid category. Let $\g=\btgbd{M}$. The subbundles $\mathcal{W}$ and $\overline{\mathcal{W}}$ are denoted by $\btgbdpq{M}{1}{0}$ and $\btgbdpq{M}{0}{1}$, which are equal to $\g^{1,0}$ and $\g^{0,1}$ in the Lie algebroid category. Their dual bundles $\bcotgbd{M}^{1,0}$ and $\bcotgbd{M}^{0,1}$ are equal to ${\gstar}^{1,0}$ and ${\gstar}^{0,1}$. B-forms are sections of exterior power of $\btgbd{M}$. The set of k-degree b-forms is denoted by $\bAk{k}(M)$, where $\bAk{k}(M)=\Ak{k}(M)$. A similar definition for the space of sections of $(p,q)$-b-forms is $\bApq{p}{q}(M)$, where $\bApq{p}{q}(M)=\Apq{p}{q}(M)$.

The definition of the b-differential, $\bDiff:\bAk{k}(M)\to\bAk{k+1}(M)$, agrees with the Lie algebroid differential. Thus, the splitting $\bDiff=\bPartial+\bPartialbar$ on $(p,q)$-b-forms also holds where
\begin{align*}
    \bPartial:\bApq{p}{q}(M)\to\bApq{p+1}{q}(M)
    \qquad\text{and}\qquad
    \bPartialbar:\bApq{p}{q}(M)\to\bApq{p}{q+1}(M).
\end{align*}

A Hermitian metric on $\btgbd{M}$ has a canonical $(1,1)$-b-form $\omega$ which is obtained by identifying $\Apq{1}{1}(M)$ and $\bApq{1}{1}(M)$. The definition of a K\"{a}hler b-manifold is similar to the classical case and the Lie algebroid case, which is stated as follows.
\begin{definition}
A complex b-tangent bundle $\btgbd{M}$ is \keyword{K\"{a}hler} if it carries a Hermitian metric whose associated canonical $(1,1)$-b-form $\omega\in\bApq{1}{1}(M)$ is closed. We call $\omega$ a \keyword{K\"{a}hler b-form}.
\end{definition}

The connection in b-category 
\begin{align*}
    \bConn{}:\secsp{M}{E}\to\bAk{1}(M,E)
\end{align*}
also agrees with the Lie algebroid connection. Mendoza used a b-connection to define a holomorphic vector bundle in b-category. This definition is equivalent to the definition that transition matrices of vector bundles are $\bPartialbar$-holomorphic. Since the connections in Lie algebroid category and b-category are equivalent, we denote the Lie algebroid Chern connection by the \keyword{b-Chern connection} in the b-category. Using the b-Chern connection, we define a positive line bundle on a b-manifold as follows.
\begin{definition}
    A b-holomorphic line bundle $E\to M$ is \keyword{positive} if there is a Hermitian metric on $E$ such that the curvature $\Theta$ of the corresponding  b-Chern connection satisfying that $\sqrt{-1}\Omega$ is positive definite.
\end{definition} 

We now apply the Lie algebroid Kodaira–Nakano vanishing theorem to the positive line bundle over a K\"{a}hler b-manifold. Using Corollary \ref{cor-vanishign theorem for Dolbeault}, we obtain the following result.
\begin{proposition}
Let $(M,Z)$ be a K\"{a}hler b-manifold and $E\to M$ be a positive line bundle. Then, the kernel of $\bPartialbar$-Laplace
\begin{align*}
    \Delta_{\bPartialbar}:=
    \bPartialbar\,\bPartialbar^{*}+\bPartialbar^{*}\,\bPartialbar
    :\bApq{p}{q}(M,E)\to\bApq{p}{q}
    (M,E)
\end{align*}
vanishes when $p+q>n$.
\end{proposition}

Before finishing this section, we give an example of a K\"{a}hler b-manifold. Let $M$ be the Riemann sphere
\begin{align*}
    \CPk{1}=
    \big\{
        [z_0:z_1]
        \big|
        \, (z_0,z_1)\neq(0,0)
    \big\}.
\end{align*}
Let 
$
    U_0=\big\{ [z_0:1]\big\}
$ and $
    U_1=\big\{ [1:z_1]\big\}
$
be the Riemann sphere without the south and north poles, which naturally carry the coordinate charts maps
$
[z_0:1]\mapsto z_0
$ and 
$
[1:z_1]\mapsto z_1
$. Then, $U_0$ and $U_1$ form a collection of coordinate charts covering $\CPk{1}$ and their transition function on the overlap is $\phi_{0,1}:z_0\mapsto\frac{1}{z_0}$. Let $Z$ be the equator of the Riemann sphere, whose precise definition is the unit circle in the coordinate charts of $U_0$ and $U_1$.

We look at the b-tangent bundle $\btgbd{\CPk{1}}$. To describe it precisely, we need to rewrite the coordinate charts $U_0$ and $U_1$ into polar coordinates. For $z=r e^{\sqrt{-1}\theta}\neq 0$, its polar coordinate is $(r,\theta)$. Thus, we redefine the coordinate charts on the Riemann sphere without the north and south poles as follows:
\begin{align*}
    U_0\backslash\{\text{south pole}\}
    &\to
    [0,\infty)\times [0,2\pi)
    \\
    [z_0=r_0 e^{\sqrt{-1}\theta_0}:1]
    &\mapsto
    (r_0,\theta_0)
\end{align*}
and
\begin{align*}
    U_1\backslash\{\text{north pole}\}
    &\to
    [0,\infty)\times [0,2\pi)
    \\
    [1:z_1=r_1 e^{\sqrt{-1}\theta_1}]
    &\mapsto
    (r_1,\theta_1)
    .
\end{align*}
The coordinate transformation is: 
\begin{align*}
    \phi_{0.1}:
    (r_0,\theta_0)
    \mapsto
    (\frac{1}{r_0},-\theta_0).
\end{align*}
These coordinates naturally induce $\{\ddx{r_0},\ddx{\theta_0}\}$ and $\{\ddx{r_1},\ddx{\theta_1}\}$ as the frame of tangent bundle over $U_0\backslash\{\text{south pole}\}$ and $U_1\backslash\{\text{north pole}\}$.

We define a smooth function $f(r)$ on $(0,\infty)$ with constant value 1 on $(0,\frac{1}{2}]$ and $[2,\infty)$, and 0 at $r=1$. Moreover, we require $f(r)=f(\frac{1}{r})$. Using $f(r)$, we define $\btgbd{\CPk{1}}$ on $U_0\backslash\{\text{north pole}\}$ as the span of the frame $f(r)\ddx{r_0}$ and $\ddx{\theta_0}$, and use $R_0$ and $\Theta_0$ to denote these frames with the anchor maps
\begin{align*}
    \rho(R_0)=f(r_0)\ddx{r_0}
    \qquad\text{and}\qquad
    \rho(\Theta_0)=\ddx{\theta_0}.
\end{align*}
The definition of $\btgbd{\CPk{1}}$ on $U_1\backslash\{\text{north pole}\}$ is similar as it is spanned by the frame $R_1$ and $\Theta_1$, where $\rho(R_1)=f(r_1)\ddx{r_1}$
and
$\rho(\Theta_1)=\ddx{\theta_1}$. The assumption that $f(r)=f(\frac{1}{r})$ means that $R_0,\Theta_0$ and $R_1\Theta_1$ are well defined.

The complex structure of $\btgbd{\CPk{1}}$ is formally the same as the tangent bundle of $\CPk{1}$. The almost complex structure on $\tangentbd{\CPk{1}}$ is
\begin{align*}
    J
    \big(
        \ddx{r_0}
        \bigg|_{(r_0,\theta_0)}
    \big)
    =
    \frac{1}{r_0}
    \ddx{\theta_0}
    \bigg|_{(r_0,\theta_0)}
    \qquad\text{and}\qquad
    J
    \big(
        \ddx{\theta_0}
        \bigg|_{(r_0,\theta_0)}
    \big)
    =
    -r_0
    \ddx{r_0}
    \bigg|_{(r_0,\theta_0)}
\end{align*}
on $U_0\backslash\{\text{south pole}\}$ and the same on $U_1\backslash\{\text{north pole}\}$. Following this idea, we define the almost complex structure of $\btgbd{\CPk{1}}$, which is 
\begin{align*}
    J
    \big(
        R_0
        \big|_{(r_0,\theta_0)}
    \big)
    =
    \frac{1}{r_0}
    \Theta_0
    \big|_{(r_0,\theta_0)}
    \qquad\text{and}\qquad
    J
    \big(
        \Theta_0
        \big|_{(r_0,\theta_0)}
    \big)
    =
    -r_0
    R_0
    \big|_{(r_0,\theta_0)}
\end{align*}
on $U_0\backslash\{\text{south pole}\}$ and the same on $U_1\backslash\{\text{north pole}\}$. We can easily check that $J$ is well defined under the assumption that $f(r)=f(\frac{1}{r})$.

Recall the Fubini–Study form on $\CPk{1}$, which is equal to
\begin{align*}
    \frac{\sqrt{-1}}{2}
    \frac{1}{1+z_0 \bar{z_0}}
    dz_0\wedge d \bar{z_0}
    =
    \frac{1}{(1+r_{0}^{2})^2}
    d\theta_0 \wedge dr_0
\end{align*}
on the coordinate chart $U_0$ and the same on the coordinate $U_1$. To generalize this to $\btgbd{\CPk{1}}$, we first define the frame $R_{0}^*,\Theta_{0}^*$ for the dual bundle $\btgbd{\CPk{1}}^{*}$ on $U_0\backslash\{\text{south pole}\}$ by setting 
\begin{align*}
    R_{0}^*(R_{0})=1,
    \quad\quad
    R_{0}^*(\Theta_{0})=0,
    \quad\quad
    \Theta_{0}^*(R_{0})=0,
    \quad\text{and}\quad
    \Theta_{0}^*(\Theta_{0})=1.
\end{align*}
The same definition is applied to $R_{1}^*,\Theta_{1}^*$ on $U_1\backslash\{\text{north pole}\}$. Then, we construct a b-tangent K\"{a}hler form $\bomega$  which is equal to
\begin{align*}
    \frac{1}{(1+r_{0}^{2})^2}
    \Theta_{0}^* \wedge R_{0}^*
\end{align*}
on $U_0\backslash\{\text{south pole}\}$ and the same on $U_1\backslash\{\text{north pole}\}$. It is easy to check that this is well defined: it is a closed form since it is a top degree form.

Everything we defined above is naturally extended to the two poles since $R_0,\Theta_0$, and $R_1,\Theta_1$ are in total agreement with $\ddx{r_0},\ddx{\theta_0}$ and $\ddx{r_1},\ddx{\theta_1}$ on the region that the radial coordinate is in: $(0,\frac{1}{2})$ and $(2,\infty)$.


\section{Appendix I}
\label{sec:very long proof for Lie algebroid Kahler identity}
\begin{lemma}
\label{Lemma-Kahler identity on Lie algebroid-the second lemma}
The equations
\begin{align}
    [\LieagbdPartialbar,\iotaW]
    (\basexi{\alpha}\wedge\basexi{\indexI}\wedge\basedualxi{\indexJp})
    =
    \sqrt{-1}\LieagbdPartial^{*}
    (\basexi{\alpha}\wedge\basexi{\indexI}\wedge\basedualxi{\indexJp})
\label{eq-Kahler identity on Lie algebroid-the second lemma-target 1, main}
\end{align}  
and 
\begin{align}
    [\LieagbdPartialbar,\iotaW]
    (\basedualxi{\alpha}\wedge\basexi{\indexI}\wedge\basedualxi{\indexJp})
    =
    \sqrt{-1}\LieagbdPartial^{*}
    (\basedualxi{\alpha}\wedge\basexi{\indexI}\wedge\basedualxi{\indexJp})
\label{eq-Kahler identity on Lie algebroid-the second lemma-target 2, not main}
\end{align}
hold.
\end{lemma}

It is enough to verify (\ref{eq-Kahler identity on Lie algebroid-the second lemma-target 1, main}) since (\ref{eq-Kahler identity on Lie algebroid-the second lemma-target 2, not main}) is equal to the conjugation of (\ref{eq-Kahler identity on Lie algebroid-the second lemma-target 1, main}). We prove the lemma by induction. Our proof is based on the assumption that the equation
\begin{align*}
    [\LieagbdPartialbar,\iotaW]
    (\basexi{\indexI}\wedge\basedualxi{\indexJp})
    =
    \sqrt{-1}\LieagbdPartial^{*}
    (\basexi{\indexI}\wedge\basedualxi{\indexJp})
\end{align*}
holds from the induction condition. We separate the proof of Lemma \ref{Lemma-Kahler identity on Lie algebroid-the second lemma} into four parts:
\begin{enumerate}
    \item Calculate $
    [\LieagbdPartialbar,\iotaW]
    (\basexi{\alpha}\wedge\basexi{\indexI}\wedge\basedualxi{\indexJp})
    $ in Section \ref{subsection-For (alpha, I, J), calculation for LHS}.

    \item Calculate $
    (\sqrt{-1}\LieagbdPartial^{*}\basedualxi{\alpha})
    \wedge\basexi{\indexI}\wedge\basedualxi{\indexJp}
    $ in Section \ref{subsection-For (alpha), I, J, calculation for RHS}.

    \item Calculate $
    \basedualxi{\alpha}\wedge\sqrt{-1}\LieagbdPartial^{*}(\basexi{\indexI}\wedge\basedualxi{\indexJp})
    $ in Section \ref{subsection-For alpha, (I, J), calculation for RHS}.

    \item Calculate $
    \sqrt{-1}\LieagbdPartial^{*}(\basexi{\alpha}\wedge\basexi{\indexI}\wedge\basedualxi{\indexJp})
    $ in Section \ref{subsection-For (alpha, I, J), calculation for RHS}, where we use the expression of $
    (\sqrt{-1}\LieagbdPartial^{*}\basedualxi{\alpha})
    \wedge\basexi{\indexI}\wedge\basedualxi{\indexJp}
    $ and $
    \basedualxi{\alpha}\wedge\sqrt{-1}\LieagbdPartial^{*}(\basexi{\indexI}\wedge\basedualxi{\indexJp})
    $.
    
    \item Prove $
    [\LieagbdPartialbar,\iotaW]
    (\basexi{\alpha}\wedge\basexi{\indexI}\wedge\basedualxi{\indexJp})
    =
    \sqrt{-1}\LieagbdPartial^{*}
    (\basexi{\alpha}\wedge\basexi{\indexI}\wedge\basedualxi{\indexJp})
    $ in Section \ref{subsection-prove LHS is equal to RHS} by checking that the left-hand side and the right-hand side agree.
\end{enumerate}

\subsection{
        Calculation of \texorpdfstring{$
        [\LieagbdPartialbar,\iotaW]
        (\basexi{\alpha}\wedge\basexi{\indexI}\wedge\basedualxi{\indexJp})
        $}{}
    }
\label{subsection-For (alpha, I, J), calculation for LHS}
In this subsection, our goal is to express $
    [\LieagbdPartialbar,\iotaW]
    (\basexi{\alpha}\wedge\basexi{\indexI}\wedge\basedualxi{\indexJp})
$. The proof comprises the following steps.
\begin{enumerate}
    \item Calculate $[\LieagbdPartialbar,\iotaW]
    (\basexi{\alpha}\wedge\basexi{\indexI}\wedge\basedualxi{\indexJp})
    $ in Lemma \ref{lemma-For (alpha, I, J), calculation for LHS-initial resutl]}. This is not the final version of the expression as the next two steps continue computing two terms in the result of Lemma \ref{lemma-For (alpha, I, J), calculation for LHS-initial resutl]}.

    \item Calculate $
    (-\sqrt{-1})
    (\LieagbdPartialbar \, \iota_{\basedualxi{\alpha^{'}}}
    +
    \iota_{\basedualxi{\alpha^{'}}}\,\LieagbdPartialbar)
    (\basexi{\indexI}\wedge\basedualxi{\indexJp})
    $ in Lemma \ref{lemma-For (alpha, I, J), calculation for LHS-extra 1]}.

    \item Calculate $
    (-\frac{\sqrt{-1}}{2})
    \displaystyle\sum_{h=1}^{n}
    \big(
        (\iota_{\basexi{h}} \, \LieagbdPartialbar \basexi{\alpha})
        \wedge
        \iota_{\basedualxi{h^{'}}}(\basexi{\indexI}\wedge\basedualxi{\indexJp})
        -
        (\iota_{\basedualxi{h^{'}}} \, \LieagbdPartialbar \basexi{\alpha})
        \wedge
        \iota_{\basexi{h}}(\basexi{\indexI}\wedge\basedualxi{\indexJp})
    \big)
    $ in Lemma \ref{lemma-For (alpha, I, J), calculation for LHS-extra 2]}.

    \item Summarize Steps 1–3 and give the final version of the expression of $
    [\LieagbdPartialbar,\iotaW]
    (\basexi{\alpha}\wedge\basexi{\indexI}\wedge\basedualxi{\indexJp})
    $ in Lemma \ref{lemma-the final version of the LHS}.
\end{enumerate}

\begin{lemma}
\label{lemma-For (alpha, I, J), calculation for LHS-initial resutl]}
    The expression of $
    [\LieagbdPartialbar,\iotaW]
    (\basexi{\alpha}\wedge\basexi{\indexI}\wedge\basedualxi{\indexJp})
    $ is    
\begin{align}
    \begin{pmatrix*}[l]
        [\LieagbdPartialbar,\iotaW]
        (\basexi{\alpha})
        \wedge
        \basexi{\indexI}\wedge\basedualxi{\indexJp}
        \\
        -
        \\
        \basexi{\alpha}
        \wedge
        [\LieagbdPartialbar,\iotaW]
        (\basexi{\indexI}\wedge\basedualxi{\indexJp})
        \\
        +
        \\
        (-\sqrt{-1})
        (\LieagbdPartialbar \, \iota_{\basedualxi{\alpha^{'}}}
        +
        \iota_{\basedualxi{\alpha^{'}}}\,\LieagbdPartialbar)
        (\basexi{\indexI}\wedge\basedualxi{\indexJp})
        \\
        +
        \\
        (-\frac{\sqrt{-1}}{2})
        \displaystyle\sum_{h=1}^{n}
        \big(
            (\iota_{\basexi{h}} \, \LieagbdPartialbar \basexi{\alpha})
            \wedge
            \iota_{\basedualxi{h^{'}}}(\basexi{\indexI}\wedge\basedualxi{\indexJp})
            -
            (\iota_{\basedualxi{h^{'}}} \, \LieagbdPartialbar \basexi{\alpha})
            \wedge
            \iota_{\basexi{h}}(\basexi{\indexI}\wedge\basedualxi{\indexJp})
        \big)
    \end{pmatrix*}
.
\label{EQ-For (alpha, I, J), calculation for LHS-half summary}
\end{align}
\end{lemma}

\begin{proof}
Observe that
\begin{align*}
    [\LieagbdPartialbar,\iotaW]
    (
        \basexi{\alpha}\wedge\basexi{\indexI}\wedge\basedualxi{\indexJp}
    )
    =
    \LieagbdPartialbar\,\iotaW
    (
        \basexi{\alpha}\wedge\basexi{\indexI}\wedge\basedualxi{\indexJp}
    )
    -
    \iotaW\,\LieagbdPartialbar
    (
        \basexi{\alpha}\wedge\basexi{\indexI}\wedge\basedualxi{\indexJp}
    ).
\end{align*}
Recalling that $
\iotaW
=
-\frac{\sqrt{-1}}{2}
\displaystyle\sum_{h=1}^{n}
\iota_{\basedualxi{h^{'}}} \, \iota_{\basexi{h}}
$, we apply this to $\basexi{\alpha}\wedge\basexi{\indexI}\wedge\basedualxi{\indexJp}$ and obtain
\begin{align}
        \iotaW
        (\basexi{\alpha}\wedge\basexi{\indexI}\wedge\basedualxi{\indexJp})
    =&
        -\frac{\sqrt{-1}}{2}
        \displaystyle\sum_{h=1}^{n}
        \iota_{\basedualxi{h^{'}}} \, \iota_{\basexi{h}}
        (\basexi{\alpha}\wedge\basexi{\indexI}\wedge\basedualxi{\indexJp})
\notag
    \\
    =&
        -\frac{\sqrt{-1}}{2}
        \displaystyle\sum_{h=1}^{n}
        \iota_{\basedualxi{h^{'}}}
        \begin{pmatrix*}[l]
            (\iota_{\basexi{h}}\basexi{\alpha})
            \wedge\basexi{\indexI}\wedge\basedualxi{\indexJp}
            \\
            -
            \\
            \basexi{\alpha}
            \wedge
            \iota_{\basexi{h}}(\basexi{\indexI}\wedge\basedualxi{\indexJp})
        \end{pmatrix*}
\notag
    \\
    =&
        -\frac{\sqrt{-1}}{2}
        \LieagbdPartialbar
        \displaystyle\sum_{h=1}^{n}
        \begin{pmatrix*}[l]
            \begin{pmatrix*}[l]
                (\iota_{\basedualxi{h^{'}}} \, \iota_{\basexi{h}}\basexi{\alpha})
                \wedge
                \basexi{\indexI}\wedge\basedualxi{\indexJp}
                \\
                +
                \\
                (\iota_{\basexi{h}}\basexi{\alpha})
                \wedge
                \iota_{\basedualxi{h^{'}}} (\basexi{\indexI}\wedge\basedualxi{\indexJp})
            \end{pmatrix*}
            \\
            -
            \\
            \begin{pmatrix*}[l]
                (\iota_{\basedualxi{h^{'}}} \basexi{\alpha})
                \wedge
                \iota_{\basexi{h}}(\basexi{\indexI}\wedge\basedualxi{\indexJp})
                \\
                -
                \\
                \basexi{\alpha}
                \wedge
                \iota_{\basedualxi{h^{'}}} \, \iota_{\basexi{h}}
                (\basexi{\indexI}\wedge\basedualxi{\indexJp})
            \end{pmatrix*}
        \end{pmatrix*}.
\label{EQ-For (alpha, I, J), calculation for LHS-No 1}
\end{align}
We note that $\contract{\basexi{h}}\basexi{\alpha}=2 \delta_{\alpha,h}$, $\contract{h}\basexi{\alpha}=2 \delta_{\alpha,h}$, and $\contract{\basedualxi{h^{'}}}\basexi{\alpha}=0$ are constant functions. Thus, the formula 
$\contract{\basedualxi{h^{'}}}\,\contract{\basexi{h}}\basexi{\alpha}=0$ as the contraction on a constant function is zero. Plugging these into (\ref{EQ-For (alpha, I, J), calculation for LHS-No 1}), we obtain
\begin{align}
        -\frac{\sqrt{-1}}{2}
        \displaystyle\sum_{h=1}^{n}
        \begin{pmatrix*}[l]
            \begin{pmatrix*}[l]
                (0)
                \wedge
                \basexi{\indexI}\wedge\basedualxi{\indexJp}
                \\
                +
                \\
                (2\delta_{\alpha, h})
                \wedge
                \iota_{\basedualxi{h^{'}}} (\basexi{\indexI}\wedge\basedualxi{\indexJp})
            \end{pmatrix*}
            \\
            -
            \\
            \begin{pmatrix*}[l]
                (0)
                \wedge
                \iota_{\basexi{h}}(\basexi{\indexI}\wedge\basedualxi{\indexJp})
                \\
                -
                \\
                \basexi{\alpha}
                \wedge
                \iota_{\basedualxi{h^{'}}} \, \iota_{\basexi{h}}
                (\basexi{\indexI}\wedge\basedualxi{\indexJp})
            \end{pmatrix*}
        \end{pmatrix*}
    =
        -\frac{\sqrt{-1}}{2}
        \begin{pmatrix*}[l]
            2\iota_{\basedualxi{\alpha^{'}}} (\basexi{\indexI}\wedge\basedualxi{\indexJp})
            \\
            +
            \\
            \displaystyle\sum_{h=1}^{n}
            \basexi{\alpha}
            \wedge
            \iota_{\basedualxi{h}} \, \iota_{\basexi{h}}
            (\basexi{\indexI}\wedge\basedualxi{\indexJp})
        \end{pmatrix*}.
\label{EQ-For (alpha, I, J), calculation for LHS-No 2}
\end{align}
The next step is to apply $\LieagbdPartialbar$ on $\basexi{\alpha}\wedge\basexi{\indexI}\wedge\basedualxi{\indexJp}$: 
\begin{align}
        \LieagbdPartialbar\,\iotaW
        (\basexi{\alpha}\wedge\basexi{\indexI}\wedge\basedualxi{\indexJp})
    =
        -\frac{\sqrt{-1}}{2}
        \begin{pmatrix*}[l]
            2\LieagbdPartialbar \, \iota_{\basedualxi{\alpha^{'}}}
            (\basexi{\indexI}\wedge\basedualxi{\indexJp})
            \\
            +
            \\
            \displaystyle\sum_{h=1}^{n}
            (\LieagbdPartialbar \basexi{\alpha})
            \wedge
            \iota_{\basedualxi{h}} \, \iota_{\basexi{h}}
            (\basexi{\indexI}\wedge\basedualxi{\indexJp})
            \\
            -
            \\
            \displaystyle\sum_{h=1}^{n}
            \basexi{\alpha}
            \wedge
            \LieagbdPartialbar \, \iota_{\basedualxi{h}} \, \iota_{\basexi{h}}
            (\basexi{\indexI}\wedge\basedualxi{\indexJp})
        \end{pmatrix*}.
    \label{EQ-For (alpha, I, J), calculation for LHS-No 3}
\end{align}
It is not necessary to calculate the expressions of $
\LieagbdPartialbar \, \iota_{\basedualxi{h}} \, \iota_{\basexi{h}}(\basexi{\indexI}\wedge\basedualxi{\indexJp})
$
and 
$
(\LieagbdPartialbar \basexi{\alpha})\wedge \iota_{\basedualxi{h}} \, \iota_{\basexi{h}}(\basexi{\indexI}\wedge\basedualxi{\indexJp})
$ 
as they may represent 
$
[\LieagbdPartialbar,\iotaW]
(\basexi{\alpha}\wedge\basexi{\indexI}\wedge\basedualxi{\indexJp})
$.

Next, we compute $\iotaW\,\LieagbdPartialbar(\basexi{\alpha}\wedge\basexi{\indexI}\wedge\basedualxi{\indexJp})$. Applying $\LieagbdPartialbar$ to $\basexi{\alpha}\wedge\basexi{\indexI}\wedge\basedualxi{\indexJp}$, we obtain
\begin{align*}
    (\LieagbdPartialbar \basexi{\alpha})
    \wedge
    \basexi{\indexI}\wedge\basedualxi{\indexJp}
    -
    \basexi{\alpha}
    \wedge
    \LieagbdPartialbar(\basexi{\indexI}\wedge\basedualxi{\indexJp}).
\end{align*}
We then apply $
\iotaW
=
-\frac{\sqrt{-1}}{2}
\displaystyle\sum_{h=1}^{n}
\iota_{\basedualxi{h^{'}}} \, \iota_{\basexi{h}}
$ to $\LieagbdPartialbar(\basexi{\alpha}\wedge\basexi{\indexI}\wedge\basedualxi{\indexJp})$ and obtain 
\begin{align}
    &
        \iotaW\,\LieagbdPartialbar
        (\basexi{\alpha}\wedge\basexi{\indexI}\wedge\basedualxi{\indexJp})
    \notag
    \\
    =&
        -\frac{\sqrt{-1}}{2}
        \displaystyle\sum_{h=1}^{n}
        \iota_{\basedualxi{h^{'}}} \, \iota_{\basexi{h}}
        \begin{pmatrix*}[l]
            (\LieagbdPartialbar \basexi{\alpha})
            \wedge
            \basexi{\indexI}\wedge\basedualxi{\indexJp}
            -
            \basexi{\alpha}
            \wedge
            \LieagbdPartialbar(\basexi{\indexI}\wedge\basedualxi{\indexJp})
        \end{pmatrix*}
\notag
    \\
    =&
        -\frac{\sqrt{-1}}{2}
        \displaystyle\sum_{h=1}^{n}
        \iota_{\basedualxi{h^{'}}}
        \begin{pmatrix*}[l]
            \begin{pmatrix*}
                (\iota_{\basexi{h}} \, \LieagbdPartialbar \basexi{\alpha})
                \wedge
                \basexi{\indexI}\wedge\basedualxi{\indexJp}
                \\
                +
                \\
                (\LieagbdPartialbar \basexi{\alpha})
                \wedge
                \iota_{\basexi{h}}(\basexi{\indexI}\wedge\basedualxi{\indexJp})
            \end{pmatrix*}   
            -
            \begin{pmatrix*}
                (\iota_{\basexi{h}}\basexi{\alpha})
                \wedge
                \LieagbdPartialbar(\basexi{\indexI}\wedge\basedualxi{\indexJp})
                \\
                -
                \\
                \basexi{\alpha}
                \wedge
                \iota_{\basexi{h}}\,\LieagbdPartialbar
                (\basexi{\indexI}\wedge\basedualxi{\indexJp})
            \end{pmatrix*} 
        \end{pmatrix*}
\notag
    \\
    =&
        -\frac{\sqrt{-1}}{2}
        \displaystyle\sum_{h=1}^{n}
        \begin{pmatrix*}[l]
            \begin{pmatrix*}
                \begin{pmatrix*}                                    
                    (\iota_{\basedualxi{h^{'}}}\,\iota_{\basexi{h}} \, \LieagbdPartialbar \basexi{\alpha})
                    \wedge
                    \basexi{\indexI}\wedge\basedualxi{\indexJp}
                    \\
                    -
                    \\
                    (\iota_{\basexi{h}} \, \LieagbdPartialbar \basexi{\alpha})
                    \wedge
                    \iota_{\basedualxi{h^{'}}}(\basexi{\indexI}\wedge\basedualxi{\indexJp})
                \end{pmatrix*}
                \\
                +
                \\
                \begin{pmatrix*}
                    (\iota_{\basedualxi{h^{'}}} \, \LieagbdPartialbar \basexi{\alpha})
                    \wedge
                    \iota_{\basexi{h}}(\basexi{\indexI}\wedge\basedualxi{\indexJp})
                    \\
                    +
                    \\
                    (\LieagbdPartialbar \basexi{\alpha})
                    \wedge
                    \iota_{\basedualxi{h^{'}}} \, \iota_{\basexi{h}}
                    (\basexi{\indexI}\wedge\basedualxi{\indexJp})
                \end{pmatrix*}
            \end{pmatrix*}
            -
            \begin{pmatrix*}
                \begin{pmatrix*}
                    (\iota_{\basedualxi{h^{'}}}\,\iota_{\basexi{h}}\basexi{\alpha})
                    \wedge
                    \LieagbdPartialbar
                    (\basexi{\indexI}\wedge\basedualxi{\indexJp})
                    \\
                    +
                    \\
                    (\iota_{\basexi{h}}\basexi{\alpha})
                    \wedge
                    \iota_{\basedualxi{h^{'}}}\,\LieagbdPartialbar
                    (\basexi{\indexI}\wedge\basedualxi{\indexJp})
                \end{pmatrix*}
                \\
                -
                \\
                \begin{pmatrix*}  
                    (\iota_{\basedualxi{h^{'}}}\basexi{\alpha})
                    \wedge
                    \iota_{\basexi{h}}\,\LieagbdPartialbar
                    (\basexi{\indexI}\wedge\basedualxi{\indexJp})
                    \\
                    -
                    \\
                    \basexi{\alpha}
                    \wedge
                    \iota_{\basedualxi{h^{'}}}\,\iota_{\basexi{h}}\,\LieagbdPartialbar
                    (\basexi{\indexI}\wedge\basedualxi{\indexJp})
                \end{pmatrix*}
            \end{pmatrix*}
        \end{pmatrix*}.
\label{EQ-For (alpha, I, J), calculation for LHS-No 4}
\end{align}
Using the property that $\contract{\basexi{h}}\basexi{\alpha}=2\delta_{h,\alpha}$, $\contract{\basedualxi{h^{'}}}\basexi{\alpha}=0$ and $(\iota_{\basedualxi{h^{'}}}\,\iota_{\basexi{h}}\basexi{\alpha})=0$, we can simplify the terms in (\ref{EQ-For (alpha, I, J), calculation for LHS-No 4}) to
\begin{align*}
    &
        -\frac{\sqrt{-1}}{2}
        \displaystyle\sum_{h=1}^{n}
        \begin{pmatrix*}[l]
            \begin{pmatrix*}
                \begin{pmatrix*}                                    
                    (\iota_{\basedualxi{h^{'}}}\,\iota_{\basexi{h}} \, \LieagbdPartialbar \basexi{\alpha})
                    \wedge
                    \basexi{\indexI}\wedge\basedualxi{\indexJp}
                    \\
                    -
                    \\
                    (\iota_{\basexi{h}} \, \LieagbdPartialbar \basexi{\alpha})
                    \wedge
                    \iota_{\basedualxi{h^{'}}}(\basexi{\indexI}\wedge\basedualxi{\indexJp})
                \end{pmatrix*}
                \\
                +
                \\
                \begin{pmatrix*}
                    (\iota_{\basedualxi{h^{'}}} \, \LieagbdPartialbar \basexi{\alpha})
                    \wedge
                    \iota_{\basexi{h}}(\basexi{\indexI}\wedge\basedualxi{\indexJp})
                    \\
                    +
                    \\
                    (\LieagbdPartialbar \basexi{\alpha})
                    \wedge
                    \iota_{\basedualxi{h^{'}}} \, \iota_{\basexi{h}}
                    (\basexi{\indexI}\wedge\basedualxi{\indexJp})
                \end{pmatrix*}
            \end{pmatrix*}
            -
            \begin{pmatrix*}
                \begin{pmatrix*}
                    0
                    \wedge
                    \LieagbdPartialbar
                    (\basexi{\indexI}\wedge\basedualxi{\indexJp})
                    \\
                    +
                    \\
                    2\delta_{h,\alpha}
                    \wedge
                    \iota_{\basedualxi{h^{'}}}\,\LieagbdPartialbar
                    (\basexi{\indexI}\wedge\basedualxi{\indexJp})
                \end{pmatrix*}
                \\
                -
                \\
                \begin{pmatrix*}  
                    0
                    \wedge
                    \iota_{\basexi{h}}\,\LieagbdPartialbar
                    (\basexi{\indexI}\wedge\basedualxi{\indexJp})
                    \\
                    -
                    \\
                    \basexi{\alpha}
                    \wedge
                    \iota_{\basedualxi{h^{'}}}\,\iota_{\basexi{h}}\,\LieagbdPartialbar
                    (\basexi{\indexI}\wedge\basedualxi{\indexJp})
                \end{pmatrix*}
            \end{pmatrix*}
        \end{pmatrix*}
    \\
    =&
        -\frac{\sqrt{-1}}{2}
        \displaystyle\sum_{h=1}^{n}
        \begin{pmatrix*}[l]
            (\iota_{\basedualxi{h^{'}}}\,\iota_{\basexi{h}} \, \LieagbdPartialbar \basexi{\alpha})
            \wedge
            \basexi{\indexI}\wedge\basedualxi{\indexJp}
            \\
            -
            \\
            (\iota_{\basexi{h}} \, \LieagbdPartialbar \basexi{\alpha})
            \wedge
            \iota_{\basedualxi{h^{'}}}(\basexi{\indexI}\wedge\basedualxi{\indexJp})
            \\
            +
            \\
            (\iota_{\basedualxi{h^{'}}} \, \LieagbdPartialbar \basexi{\alpha})
            \wedge
            \iota_{\basexi{h}}(\basexi{\indexI}\wedge\basedualxi{\indexJp})
            \\
            +
            \\
            (\LieagbdPartialbar \basexi{\alpha})
            \wedge
            \iota_{\basedualxi{h^{'}}} \, \iota_{\basexi{h}}
            (\basexi{\indexI}\wedge\basedualxi{\indexJp})
            \\
            -
            \\
            2\delta_{h,\alpha}
            \wedge
            \iota_{\basedualxi{h^{'}}}\,\LieagbdPartialbar
            (\basexi{\indexI}\wedge\basedualxi{\indexJp})
            \\
            -
            \\
            \basexi{\alpha}
            \wedge
            \iota_{\basedualxi{h^{'}}}\,\iota_{\basexi{h}}\,\LieagbdPartialbar
            (\basexi{\indexI}\wedge\basedualxi{\indexJp})
        \end{pmatrix*}.
\end{align*}
Taking $
-2\delta_{h,\alpha}\wedge\iota_{\basedualxi{h^{'}}}
\,\LieagbdPartialbar
(\basexi{\indexI}\wedge\basedualxi{\indexJp})
=-2\iota_{\basedualxi{\alpha^{'}}}
\,\LieagbdPartialbar
(\basexi{\indexI}\wedge\basedualxi{\indexJp})
$, we conclude that
\begin{align}
    \iotaW \, \LieagbdPartialbar
    (\basexi{\alpha}\wedge\basexi{\indexI}\wedge\basedualxi{\indexJp})
    =&
        -\frac{\sqrt{-1}}{2}
        (
        \displaystyle\sum_{h=1}^{n}
        \begin{pmatrix*}[l]
            (\iota_{\basedualxi{h^{'}}}\iota_{\basexi{h}} \, \LieagbdPartialbar \basexi{\alpha})
            \wedge
            \basexi{\indexI}\wedge\basedualxi{\indexJp}
            \\
            -
            \\
            (\iota_{\basexi{h}} \, \LieagbdPartialbar \basexi{\alpha})
            \wedge
            \iota_{\basedualxi{h^{'}}}(\basexi{\indexI}\wedge\basedualxi{\indexJp})
            \\
            +
            \\
            (\iota_{\basedualxi{h^{'}}} \, \LieagbdPartialbar \basexi{\alpha})
            \wedge
            \iota_{\basexi{h}}(\basexi{\indexI}\wedge\basedualxi{\indexJp})
            \\
            +
            \\
            (\LieagbdPartialbar \basexi{\alpha})
            \wedge
            \iota_{\basedualxi{h^{'}}} \, \iota_{\basexi{h}}
            (\basexi{\indexI}\wedge\basedualxi{\indexJp})
            \\
            -
            \\
            2\iota_{\basedualxi{\alpha^{'}}}\,\LieagbdPartialbar
            (\basexi{\indexI}\wedge\basedualxi{\indexJp})
            \\
            -
            \\
            \basexi{\alpha}
            \wedge
            \iota_{\basedualxi{h^{'}}}\,\iota_{\basexi{h}}\,\LieagbdPartialbar
            (\basexi{\indexI}\wedge\basedualxi{\indexJp})
        \end{pmatrix*}
        .
\label{EQ-For (alpha, I, J), calculation for LHS-No 5}
\end{align}

Combining (\ref{EQ-For (alpha, I, J), calculation for LHS-No 3}) and (\ref{EQ-For (alpha, I, J), calculation for LHS-No 5}), we find that 
\begin{align*}
    [\LieagbdPartialbar,\iotaW]
    (\basexi{\alpha}\wedge\basexi{\indexI}\wedge\basedualxi{\indexJp})
    =
        \LieagbdPartialbar\,\iotaW
        (\basexi{\alpha}\wedge\basexi{\indexI}\wedge\basedualxi{\indexJp})
        -
        \iotaW\,\LieagbdPartialbar
        (\basexi{\alpha}\wedge\basexi{\indexI}\wedge\basedualxi{\indexJp})
\end{align*}
is equal to
\begin{align*}
        (-\frac{\sqrt{-1}}{2})
        \Big(
        \begin{pmatrix*}[l]
            2\LieagbdPartialbar \, \iota_{\basedualxi{\alpha^{'}}}
            (\basexi{\indexI}\wedge\basedualxi{\indexJp})
            \\
            +
            \\
            \displaystyle\sum_{h=1}^{n}
            (\LieagbdPartialbar \basexi{\alpha})
            \wedge
            \iota_{\basedualxi{h}} \, \iota_{\basexi{h}}
            (\basexi{\indexI}\wedge\basedualxi{\indexJp})
            \\
            -
            \\
            \displaystyle\sum_{h=1}^{n}
            \basexi{\alpha}
            \wedge
            \LieagbdPartialbar \, \iota_{\basedualxi{h}} \, \iota_{\basexi{h}}
            (\basexi{\indexI}\wedge\basedualxi{\indexJp})
        \end{pmatrix*}
        -
        \begin{pmatrix*}[l]
            (\iota_{\basedualxi{h^{'}}}\iota_{\basexi{h}} \, \LieagbdPartialbar \basexi{\alpha})
            \wedge
            \basexi{\indexI}\wedge\basedualxi{\indexJp}
            \\
            -
            \\
            (\iota_{\basexi{h}} \, \LieagbdPartialbar \basexi{\alpha})
            \wedge
            \iota_{\basedualxi{h^{'}}}(\basexi{\indexI}\wedge\basedualxi{\indexJp})
            \\
            +
            \\
            (\iota_{\basedualxi{h^{'}}} \, \LieagbdPartialbar \basexi{\alpha})
            \wedge
            \iota_{\basexi{h}}(\basexi{\indexI}\wedge\basedualxi{\indexJp})
            \\
            +
            \\
            (\LieagbdPartialbar \basexi{\alpha})
            \wedge
            \iota_{\basedualxi{h^{'}}} \, \iota_{\basexi{h}}
            (\basexi{\indexI}\wedge\basedualxi{\indexJp})
            \\
            -
            \\
            2\iota_{\basedualxi{\alpha^{'}}}\,\LieagbdPartialbar
            (\basexi{\indexI}\wedge\basedualxi{\indexJp})
            \\
            -
            \\
            \basexi{\alpha}
            \wedge
            \iota_{\basedualxi{h^{'}}}\,\iota_{\basexi{h}}\,\LieagbdPartialbar
            (\basexi{\indexI}\wedge\basedualxi{\indexJp})
        \end{pmatrix*}
        \Big)
    ,
    \notag\\
\end{align*}
which can be simplified to
\begin{align*}
        (-\frac{\sqrt{-1}}{2})
        \begin{pmatrix*}[l]
            (0-\iota_{\basedualxi{h^{'}}}\iota_{\basexi{h}} \, \LieagbdPartialbar \basexi{\alpha})
            \wedge
            \basexi{\indexI}\wedge\basedualxi{\indexJp}
            \\
            -
            \\
            \basexi{\alpha}
            \wedge
            (\LieagbdPartialbar \, \iota_{\basedualxi{h}} \, \iota_{\basexi{h}}
            -
            \iota_{\basedualxi{h^{'}}}\,\iota_{\basexi{h}}\,\LieagbdPartialbar)
            (\basexi{\indexI}\wedge\basedualxi{\indexJp})
            \\
            +
            \\
            2(
            \LieagbdPartialbar \, \iota_{\basedualxi{\alpha^{'}}}
            +
            \iota_{\basedualxi{\alpha^{'}}}\,\LieagbdPartialbar
            )
            (\basexi{\indexI}\wedge\basedualxi{\indexJp})
            \\
            +
            \\
            \displaystyle\sum_{h=1}^{n}
            \big(
                (\iota_{\basexi{h}} \, \LieagbdPartialbar \basexi{\alpha})
                \wedge
                \iota_{\basedualxi{h^{'}}}(\basexi{\indexI}\wedge\basedualxi{\indexJp})
                -
                (\iota_{\basedualxi{h^{'}}} \, \LieagbdPartialbar \basexi{\alpha})
                \wedge
                \iota_{\basexi{h}}(\basexi{\indexI}\wedge\basedualxi{\indexJp})
            \big)
        \end{pmatrix*}
    .
\end{align*}
Observing that $
(0-\iota_{\basedualxi{h^{'}}}\iota_{\basexi{h}} \, \LieagbdPartialbar \basexi{\alpha})
=
[\LieagbdPartialbar,\iotaW](\basexi{\alpha})
$ and $
(\LieagbdPartialbar \, \iota_{\basedualxi{h}} \, \iota_{\basexi{h}}
-
\iota_{\basedualxi{h^{'}}}\,\iota_{\basexi{h}}\,\LieagbdPartialbar)
=
[\LieagbdPartialbar,\iotaW]
$, we have completed the proof.

\end{proof}

\newpage
\begin{lemma}
\label{lemma-For (alpha, I, J), calculation for LHS-extra 1]}
\begin{align*}
    (-\sqrt{-1})(\LieagbdPartialbar\,\contract{\basedualxi{\alpha^{'}}}
    +
    \contract{\basedualxi{\alpha^{'}}}\,\LieagbdPartialbar)
    (\basexi{\indexI}\wedge\basedualxi{\indexJp})
\end{align*}
is equal to
\begin{align}
    -2\sqrt{-1}
    \begin{pmatrix*}[l]
        \displaystyle\sum_{\indexi{k}\in\indexI}
        \begin{pmatrix*}[l]
            (-1)
            \coefB{\indexi{k}}{\indexi{k} \alpha^{'}}
            \basexi{\indexI}
            \wedge
            \basedualxi{\indexJp}
        \end{pmatrix*}
    \\
    +
    \\
        \displaystyle\sum_{\indexi{k}\in\indexI}
        \displaystyle\sum_{s\in\indexIc}
        \begin{pmatrix*}[l]
            \sgn
                {s,\indexI,\indexJp}
                {\indexi{k},\order{(\indexI\backslash\{\indexi{k}\})\cup\{s\}},\indexJp}
            \coefB{\indexi{k}}{s \alpha^{'}}
            \basexi{\order{(\indexI\backslash\{\indexi{k}\})\cup\{s\}}}
            \wedge
            \basedualxi{\indexJp}
        \end{pmatrix*}
    \\
    +
    \\
        \displaystyle\sum_{\indexj{l}\in\indexJp\backslash\{\alpha^{'}\}}
        \begin{pmatrix*}[l]
            \sgn
                {\order{\indexjp{l},\alpha^{'}}}
                {\alpha^{'},\indexjp{l}}
            \coefD{\indexjp{l}}{\order{\indexjp{l} \alpha^{'}}}
            \basexi{\indexI}
            \wedge
            \basedualxi{\indexJp}
        \end{pmatrix*}
    \\
    +
    \\
        \displaystyle\sum_{\indexj{l}\in\indexJp}
        \displaystyle\sum_{t^{'}\in\indexJpc}
        \begin{pmatrix*}[l]
            \sgn
                {\order{\alpha^{'},t^{'}}}
                {t^{'},\alpha^{'}}
            \sgn
                {t^{'},\indexI,\indexJp}
                {\indexjp{l},\order{(\indexJp\backslash\{\indexjp{l}\})\cup\{t^{'}\}}}
            \\
            \coefD{\indexjp{l}}{\order{\alpha^{'} t^{'}}}
            \basexi{\indexI}
            \wedge
            \basedualxi{\order{(\indexJp\backslash\{\indexjp{l}\})\cup\{t^{'}\}}}
        \end{pmatrix*}
    \end{pmatrix*}
    .
\label{EQ-For (alpha, I, J), calculation for LHS, extra 1-result}
\end{align}
\end{lemma}

\begin{proof}

The proof comprises the following steps.
\begin{enumerate}
    \item Calculate $
    \contract{\basedualxi{\alpha^{'}}}\,\LieagbdPartialbar
    (\basexi{\indexI}\wedge\basedualxi{\indexJp})
    $.

    \item Calculate $
    \LieagbdPartialbar\,\contract{\basedualxi{\alpha^{'}}}
    (\basexi{\indexI}\wedge\basedualxi{\indexJp})
    $.

    \item Combine the results from Steps 1–2.
\end{enumerate}
\keyword{Step 1:} The expression of $\LieagbdPartialbar(\basexi{\indexI} \wedge \basedualxi{\indexJp})$ is 
\begin{align}
    \begin{pmatrix*}[l]
        \displaystyle\sum_{\indexi{k}\in\indexI}
        \sgn
            {\indexI}
            {\indexi{k},\indexI\backslash\{\indexi{k}\}}
        \basexi{\indexi{1}}
        \wedge\cdots\wedge
        \LieagbdPartialbar \basexi{\indexi{k}}
        \wedge\cdots\wedge
        \basexi{\indexi{p}}
        \wedge
        \basedualxi{\indexJp}
        \\
        +
        \\
        \displaystyle\sum_{\indexjp{k}\in\indexJp}
        \sgn
            {\indexI,\indexJp}
            {\indexjp{k},\indexI,\indexJp\backslash\{\indexjp{k}\}}
        \basexi{\indexIp}
        \wedge
        \basedualxi{\indexjp{1}}
        \wedge\cdots\wedge
        \LieagbdPartialbar \basedualxi{\indexjp{k}}
        \wedge\cdots\wedge
        \basedualxi{\indexjp{q}}
    \end{pmatrix*}
    .
\label{EQ-For (alpha, I, J), calculation for LHS, extra 1-no1}
\end{align}
Recalling that $
    \LieagbdPartialbar \xi_i
    =
        \sum_{j,k}^n
        \coefB{i}{jk}
        \basexi{j} \wedge \basedualxi{k}
$
and
\begin{align*}
    \LieagbdPartialbar \basedualxi{i}
    =
        \sum_{j,k=1}^n
        \frac{1}{2}
        \coefD{i}{jk}
        \basedualxi{j} \wedge \basedualxi{k}
    =
        \sum_{j,k=1}^n
        \frac{1}{2}
        \sgn
            {\order{j,k}}
            {j,k}
        \coefD{i}{\order{j,k}}
        \basedualxi{j}\wedge\basedualxi{k}
    ,
\end{align*}
we plug these into (\ref{EQ-For (alpha, I, J), calculation for LHS, extra 1-no1}) and cancel out the terms containing repeated frame bases, then obtain
\begin{align}
    \begin{pmatrix*}[l]
        \displaystyle\sum_{\indexi{k}\in\indexI}
        \sgn
            {\indexI}
            {\indexi{k},\indexI\backslash\{\indexi{k}\}}
        \basexi{\indexi{1}}
        \wedge\cdots\wedge
        \basexi{\indexi{k-1}}
        \wedge
        \begin{pmatrix*}[l]
            \displaystyle\sum_{t^{'}\in\indexJpc}
            \coefB{\indexi{k}}{\indexi{k} t^{'}}
            \basexi{\indexi{k}}\wedge\basedualxi{t^{'}}
            \\
            +
            \\
            \displaystyle\sum_{s\in\indexIc}
            \displaystyle\sum_{t^{'}\in\indexJpc}
            \coefB{\indexi{k}}{s t^{'}}
            \basexi{s}\wedge\basedualxi{t^{'}}
        \end{pmatrix*}
        \wedge
        \basexi{\indexi{k+1}}
        \wedge\cdots\wedge
        \basexi{\indexi{p}}
        \wedge
        \basedualxi{\indexJp}
        \\
        +
        \\
        \displaystyle\sum_{\indexjp{l}\in\indexJp}
        \begin{pmatrix*}[l]
            \sgn
                {\indexI,\indexJp}
                {\indexjp{l},\indexI,\indexJp\backslash\{\indexjp{l}\}}
            \basexi{\indexI}
            \wedge
            \basexi{\indexjp{1}}
            \wedge\cdots\wedge
            \basedualxi{\indexjp{l-1}}
            \wedge
            \begin{pmatrix*}[l]
                \displaystyle\sum_{t^{'}\in\indexJpc}
                \frac{1}{2}
                (
                \coefD{\indexjp{l}}{\indexjp{l} t^{'}}
                \basedualxi{\indexjp{l}}\wedge\basedualxi{t^{'}}
                +
                \coefD{\indexjp{l}}{t^{'} \indexjp{l}}
                \basedualxi{t^{'}}\wedge\basedualxi{\indexjp{l}}
                )
                \\
                +
                \\
                \displaystyle\sum_{t^{'}\in\indexJpc}
                \displaystyle\sum_{u^{'}\in\indexJpc}
                \frac{1}{2}
                \coefD{\indexjp{l}}{u^{'} t^{'}}
                \basedualxi{u^{'}}\wedge\basedualxi{t^{'}}
            \end{pmatrix*}
            \\
            \wedge
            \basedualxi{\indexjp{l+1}}
            \wedge\cdots\wedge
            \basedualxi{\indexjp{q}}
        \end{pmatrix*}
    \end{pmatrix*}
    .
\label{EQ-For (alpha, I, J), calculation for LHS, extra 1-no2}
\end{align}
We reorder the above wedge product of frames as follows:
\begin{itemize}
    \item 
$   
        \qquad
        \basexi{\indexi{1}}
        \wedge\cdots\wedge
        \basexi{\indexi{k-1}}\wedge
        \bigg(
            \displaystyle\sum_{t^{'}\in\indexJpc}
            \coefB{\indexi{k}}{\indexi{k} t^{'}}
            \basexi{\indexi{k}}\wedge\basedualxi{t^{'}}
        \bigg)
        \wedge\basexi{\indexi{k+1}}
        \wedge\cdots\wedge
        \basexi{\indexi{p}}
        \wedge
        \basedualxi{\indexJp}
$
\\
$
        =
            \displaystyle\sum_{t^{'}\in\indexJpc}
            \sgn
                {\indexi{1},\cdots,\indexi{k},t^{'},\indexi{k+1},\cdots,\indexi{p},\indexJp}
                {\indexI,\order{\indexJp\cup\{t^{'}\}}}
            \coefB{\indexi{k}}{\indexi{k} t^{'}}
            \basexi{\indexI}
            \wedge
            \basedualxi{\order{\indexJp\cup\{t^{'}\}}}
$,
    \item 
$
        \qquad
        \basexi{\indexi{1}}
        \wedge\cdots\wedge
        \basexi{\indexi{k-1}}\wedge
        \bigg(
            \displaystyle\sum_{s\in\indexIc}
            \displaystyle\sum_{t^{'}\in\indexJpc}
            \coefB{\indexi{k}}{s t^{'}}
            \basexi{s}\wedge\basedualxi{t^{'}}
        \bigg)
        \wedge\basexi{\indexi{k+1}}
        \wedge\cdots\wedge
        \basexi{\indexi{p}}
        \wedge
        \basedualxi{\indexJp}
$
\\
$
        =
        \displaystyle\sum_{s\in\indexIc}
        \displaystyle\sum_{t^{'}\in\indexJpc}
        \sgn
            {\indexi{1},\cdots,\indexi{k},t^{'},\indexi{k+1},\cdots,\indexi{p},\indexJp}
            {\indexI,\order{\indexJp\cup\{t^{'}\}}}
        \coefB{\indexi{k}}{\indexi{k} t^{'}}
        \basexi{\indexI}
        \wedge
        \basedualxi{\order{\indexJp\cup\{t^{'}\}}}
$,
    \item
$
        \qquad
        \basexi{\indexI}
        \wedge
        \basexi{\indexjp{1}}
        \wedge\cdots\wedge
        \basedualxi{\indexjp{l-1}}
        \wedge
        \bigg(
            \displaystyle\sum_{t^{'}\in\indexJpc}
            \frac{1}{2}
            (
            \coefD{\indexjp{l}}{\indexjp{l} t^{'}}
            \basedualxi{\indexjp{l}}\wedge\basedualxi{t^{'}}
            +
            \coefD{\indexjp{l}}{t^{'} \indexjp{l}}
            \basedualxi{t^{'}}\wedge\basedualxi{\indexjp{l}}
            )
        \bigg)
        \wedge
        \basedualxi{\indexjp{l+1}}
        \wedge\cdots\wedge
        \basedualxi{\indexjp{q}}
$
\\
$
        =
        \basexi{\indexI}
        \wedge
        \basexi{\indexjp{1}}
        \wedge\cdots\wedge
        \basedualxi{\indexjp{l-1}}
        \wedge
        \bigg(
            \displaystyle\sum_{t^{'}\in\indexJpc}
            \coefD{\indexjp{l}}{\order{\indexjp{l} t^{'}}}
            \basedualxi{\order{ \indexjp{l} t^{'}}}
        \bigg)
        \wedge
        \basedualxi{\indexjp{l+1}}
        \wedge\cdots\wedge
        \basedualxi{\indexjp{q}}  
$
\\
$
        =
        \sgn
                {\indexjp{1},\cdots,\indexjp{l-1},\order{\indexjp{l},t^{'}},\indexjp{l+1},\cdots,\indexjp{q}}
                {\order{\indexJp\cup\{t^{'}\}}}
            \coefD{\indexjp{l}}{\order{\indexjp{l},t^{'}}}
            \basexi{\indexI}
            \wedge
            \basedualxi{\order{\indexJp\cup\{t^{'}\}}}
$,
    \item
$
        \qquad
        \basexi{\indexI}
        \wedge
        \basexi{\indexjp{1}}
        \wedge\cdots\wedge
        \basedualxi{\indexjp{l-1}}
        \wedge
        \bigg(
            \displaystyle\sum_{t^{'}\in\indexJpc}
            \displaystyle\sum_{u^{'}\in\indexJpc}
            \frac{1}{2}
            \coefD{\indexjp{l}}{u^{'} t^{'}}
            \basedualxi{u^{'}}\wedge\basedualxi{t^{'}}
        \bigg)
        \wedge
        \basedualxi{\indexjp{l+1}}
        \wedge\cdots\wedge
        \basedualxi{\indexjp{q}}
$
\\
$
        =
        \basexi{\indexI}
        \wedge
        \basexi{\indexjp{1}}
        \wedge\cdots\wedge
        \basedualxi{\indexjp{l-1}}
        \wedge
        \bigg(
            \displaystyle\sum_{t^{'}\in\indexJpc}
            \displaystyle\sum_{u^{'}\in\indexJpc}
            \frac{1}{2}
            \coefD{\indexjp{l}}{\order{u^{'} t^{'}}}
            \basedualxi{\order{u^{'} t^{'}}}
        \bigg)
        \wedge
        \basedualxi{\indexjp{l+1}}
        \wedge\cdots\wedge
        \basedualxi{\indexjp{q}}
$
\\
$
        =
        \sgn
            {\indexjp{1},\cdots,\indexjp{l-1},\order{t^{'},u^{'}},\indexjp{l+1},\cdots,\indexjp{q}}
            {\order{(\indexJp\backslash\{\indexjp{l}\})\cup\{t^{'},u^{'}\}}}
        \frac{1}{2}
        \coefD{\indexjp{l}}{\order{t^{'},u^{'}}}
        \basexi{\indexI}
        \wedge
        \basedualxi{\order{(\indexJp\backslash\{\indexjp{l}\})\cup\{t^{'},u^{'}\}}}
$.
\end{itemize}
We then simplify (\ref{EQ-For (alpha, I, J), calculation for LHS, extra 1-no2}) to
\begin{align}
    \begin{pmatrix*}[l]
        \displaystyle\sum_{\indexi{k}\in\indexI}
        \displaystyle\sum_{t^{'}\in\indexJpc}
        \sgn
            {\indexI}
            {\indexi{k},\indexI\backslash\{\indexi{k}\}}
        \sgn
            {\indexi{1},\cdots,\indexi{k},t^{'},\indexi{k+1},\cdots,\indexi{p},\indexJp}
            {\indexI,\order{\indexJp\cup\{t^{'}\}}}
        \coefB{\indexi{k}}{\indexi{k} t^{'}}
        \basexi{\indexI}
        \wedge
        \basedualxi{\order{\indexJp\cup\{t^{'}\}}}
        \\
        +
        \\
        \displaystyle\sum_{\indexi{k}\in\indexI}
        \displaystyle\sum_{s\in\indexIc}
        \displaystyle\sum_{t^{'}\in\indexJpc}
        \begin{pmatrix*}[l]
            \sgn
                {\indexI}
                {\indexi{k},\indexI\backslash\{\indexi{k}\}}
            \sgn
                {\indexi{1},\cdots,\indexi{k-1},s,t^{'},\indexi{k+1},\cdots,\indexi{p},\indexJp}
                {\order{(\indexI\backslash\{\indexi{k}\})\cup\{s\}},\order{\indexJp\cup\{t^{'}\}}}
            \\
            \coefB{\indexi{k}}{s t^{'}}
            \basexi{\order{(\indexI\backslash\{\indexi{k}\})\cup\{s\}}}
            \wedge
            \basedualxi{\order{\indexJp\cup\{t^{'}\}}}
        \end{pmatrix*}
        \\
        +
        \\
        \displaystyle\sum_{\indexjp{l}\in\indexJp}
        \displaystyle\sum_{t^{'}\in\indexJpc}
        \sgn
            {\indexI,\indexJp}
            {\indexjp{l},\indexI,\indexJp\backslash\{\indexjp{l}\}}
        \sgn
            {\indexjp{1},\cdots,\indexjp{l-1},\order{\indexjp{l},t^{'}},\indexjp{l+1},\cdots,\indexjp{q}}
            {\order{\indexJp\cup\{t^{'}\}}}
        \coefD{\indexjp{l}}{\order{\indexjp{l},t^{'}}}
        \basexi{\indexI}
        \wedge
        \basedualxi{\order{\indexJp\cup\{t^{'}\}}}
        \\
        +
        \\
        \displaystyle\sum_{\indexjp{l}\in\indexJp}
        \displaystyle\sum_{t^{'}\in\indexJpc}
        \displaystyle\sum_{u^{'}\in\indexJpc}
        \begin{pmatrix*}[l]
            \sgn
                {\indexI,\indexJp}
                {\indexjp{l},\indexI,\indexJp\backslash\{\indexjp{l}\}}
            \sgn
                {\indexjp{1},\cdots,\indexjp{l-1},\order{t^{'},u^{'}},\indexjp{l+1},\cdots,\indexjp{q}}
                {\order{(\indexJp\backslash\{\indexjp{l}\})\cup\{t^{'},u^{'}\}}}
            \\
            \frac{1}{2}
            \coefD{\indexjp{l}}{\order{t^{'},u^{'}}}
            \basexi{\indexI}
            \wedge
            \basedualxi{\order{(\indexJp\backslash\{\indexjp{l}\})\cup\{t^{'},u^{'}\}}}
        \end{pmatrix*}
    \end{pmatrix*}
    ,
\label{EQ-For (alpha, I, J), calculation for LHS, extra 1-no3}
\end{align}
which is the expression of $\LieagbdPartialbar(\basexi{\indexI}\wedge\basedualxi{\indexJp})$.

Next, we apply $\contract{\basedualxi{\alpha^{'}}}$ to (\ref{EQ-For (alpha, I, J), calculation for LHS, extra 1-no3}). Two cases, $\alpha\in\indexJ$ and $\alpha\in\indexJc$, must be considered. We observe that if $\alpha\in\indexJ$, then

\begin{itemize}
    \item 
$
    \contract{\basedualxi{\alpha^{'}}}
    \big(
    \basexi{\indexI}
    \wedge
    \basedualxi{\order{\indexJp\cup\{t^{'}\}}}
    \big)
    =
    \sgn
        {\indexI,\order{\indexJp\cup\{t^{'}\}}}
        {\alpha,\indexI,\order{(\indexJp\backslash\{\alpha^{'}\})\cup\{t^{'}\}}}
    \basexi{\indexI}
    \wedge
    \basedualxi{\order{(\indexJp\backslash\{\alpha^{'}\})\cup\{t^{'}\}}}
$,\\
    \item 
$
    \qquad
    \contract{\basedualxi{\alpha^{'}}}
    \big(
    \basexi{\order{(\indexI\backslash\{\indexi{k}\})\cup\{s\}}}
    \wedge
    \basedualxi{\order{\indexJp\cup\{t^{'}\}}}
    \big)
$
\\
$
    =
    \sgn
        {\order{(\indexI\backslash\{\indexi{k}\})\cup\{s\}},\order{\indexJp\cup\{t^{'}\}}}
        {\alpha^{'},\order{(\indexI\backslash\{\indexi{k}\})\cup\{s\}},\order{(\indexJp\backslash\{\alpha^{'}\})\cup\{t^{'}\}}}
    \basexi{\order{(\indexI\backslash\{\indexi{k}\})\cup\{s\}}}
    \wedge
    \basedualxi{\order{(\indexJp\backslash\{\alpha^{'}\})\cup\{t^{'}\}}}
$,\\
    \item 
$
    \contract{\basedualxi{\alpha^{'}}}
    \big(
    \basexi{\indexI}
    \wedge
    \basedualxi{\order{\indexJp\cup\{t^{'}\}}}
    \big)
    =
    \sgn
        {\indexI,\order{\indexJp\cup\{t^{'}\}}}
        {\alpha^{'},\indexI,\order{(\indexJp\backslash\{\alpha^{'}\})\cup\{t^{'}\}}}
    \basexi{\indexI}
    \wedge
    \basedualxi{\order{(\indexJp\backslash\{\alpha^{'}\})\cup\{t^{'}\}}}
$,\\
    \item 
$
    \qquad
    \contract{\basedualxi{\alpha^{'}}}
    \big(
    \basexi{\indexI}
    \wedge
    \basedualxi{\order{(\indexJp\backslash\{\indexjp{l}\})\cup\{t^{'},u^{'}\}}}
    \big)
$
\\
$
    =
    \begin{cases}
        0       
        &
        \text{if }\alpha=\indexj{l}
        \\
        \sgn
            {\indexI,\order{(\indexJp\backslash\{\indexjp{l}\})\cup\{t^{'},u^{'}\}}}
            {\alpha^{'},\indexI,\order{(\indexJp\backslash\{\indexjp{l},\alpha^{'}\})\cup\{t^{'},u^{'}\}}}
        \basexi{\indexI}
        \wedge
        \basedualxi{\order{(\indexJp\backslash\{\indexjp{l},\alpha^{'}\})\cup\{t^{'},u^{'}\}}}
        &
        \text{if }\alpha\neq\indexj{l}
    \end{cases}
$.
\end{itemize}
The above computation leads to
\begin{align}
    &
        \iota_{\basedualxi{\alpha}}\,\LieagbdPartialbar
        (\basexi{\indexI} \wedge \basedualxi{\indexJp})
    \notag\\
    =&
        2
        \begin{pmatrix*}[l]
            \displaystyle\sum_{\indexi{k}\in\indexI}
            \displaystyle\sum_{t^{'}\in\indexJpc}
            \begin{pmatrix*}[l]
                \sgn
                    {\indexI}
                    {\indexi{k},\indexI\backslash\{\indexi{k}\}}
                \sgn
                    {\indexi{1},\cdots,\indexi{k},t^{'},\indexi{k+1},\cdots,\indexi{p},\indexJp}
                    {\indexI,\order{\indexJp\cup\{t^{'}\}}}
                \sgn
                    {\indexI,\order{\indexJp\cup\{t^{'}\}}}
                    {\alpha,\indexI,\order{(\indexJp\backslash\{\alpha^{'}\})\cup\{t^{'}\}}}
                \\
                \coefB{\indexi{k}}{\indexi{k} t^{'}}
                \basexi{\indexI}
                \wedge
                \basedualxi{\order{(\indexJp\backslash\{\alpha^{'}\})\cup\{t^{'}\}}}
            \end{pmatrix*}
            \\
            +
            \\
            \displaystyle\sum_{\indexi{k}\in\indexI}
            \displaystyle\sum_{s\in\indexIc}
            \displaystyle\sum_{t^{'}\in\indexJpc}
            \begin{pmatrix*}[l]
                \sgn
                    {\indexI}
                    {\indexi{k},\indexI\backslash\{\indexi{k}\}}
                \sgn
                    {\indexi{1},\cdots,\indexi{k-1},s,t^{'},\indexi{k+1},\cdots,\indexi{p},\indexJp}
                    {\order{(\indexI\backslash\{\indexi{k}\})\cup\{s\}},\order{\indexJp\cup\{t^{'}\}}}
                \\
                \sgn
                    {\order{(\indexI\backslash\{\indexi{k}\})\cup\{s\}},\order{\indexJp\cup\{t^{'}\}}}
                    {\alpha^{'},\order{(\indexI\backslash\{\indexi{k}\})\cup\{s\}},\order{(\indexJp\backslash\{\alpha^{'}\})\cup\{t^{'}\}}}
                \\
                \coefB{\indexi{k}}{s t^{'}}
                \basexi{\order{(\indexI\backslash\{\indexi{k}\})\cup\{s\}}}
                \wedge
                \basedualxi{\order{(\indexJp\backslash\{\alpha^{'}\})\cup\{t^{'}\}}}
            \end{pmatrix*}
            \\
            +
            \\
            \displaystyle\sum_{\indexjp{l}\in\indexJp}
            \displaystyle\sum_{t^{'}\in\indexJpc}
            \begin{pmatrix*}[l]
                \sgn
                    {\indexI,\indexJp}
                    {\indexjp{l},\indexI,\indexJp\backslash\{\indexjp{l}\}}
                \sgn
                    {\indexjp{1},\cdots,\indexjp{l-1},\order{\indexjp{l},t^{'}},\indexjp{l+1},\cdots,\indexjp{q}}
                    {\order{\indexJp\cup\{t^{'}\}}}
                \\
                \sgn
                    {\indexI,\order{\indexJp\cup\{t^{'}\}}}
                    {\alpha^{'},\indexI,\order{(\indexJp\backslash\{\alpha^{'}\})\cup\{t^{'}\}}}
                \\
                \coefD{\indexjp{l}}{\order{\indexjp{l},t^{'}}}
                \basexi{\indexI}
                \wedge
                \basedualxi{\order{(\indexJp\backslash\{\alpha^{'}\})\cup\{t^{'}\}}}
            \end{pmatrix*}
            \\
            +
            \\
            \displaystyle\sum_{\indexjp{l}\in\indexJp\backslash\{\alpha^{'}\}}
            \displaystyle\sum_{t^{'}\in\indexJpc}
            \displaystyle\sum_{u^{'}\in\indexJpc}
            \begin{pmatrix*}[l]
                \sgn
                    {\indexI,\indexJp}
                    {\indexjp{l},\indexI,\indexJp\backslash\{\indexjp{l}\}}
                \sgn
                    {\indexjp{1},\cdots,\indexjp{l-1},\order{t^{'},u^{'}},\indexjp{l+1},\cdots,\indexjp{q}}
                    {\order{(\indexJp\backslash\{\indexjp{l}\})\cup\{t^{'},u^{'}\}}}
                \\
                \sgn
                    {\indexI,\order{(\indexJp\backslash\{\indexjp{l}\})\cup\{t^{'},u^{'}\}}}
                    {\alpha^{'},\indexI,\order{(\indexJp\backslash\{\indexjp{l},\alpha^{'}\})\cup\{t^{'},u^{'}\}}}
                \\
                \frac{1}{2}
                \coefD{\indexjp{l}}{\order{t^{'},u^{'}}}
                \basexi{\indexI}
                \wedge
                \basedualxi{\order{(\indexJp\backslash\{\indexjp{l},\alpha^{'}\})\cup\{t^{'},u^{'}\}}}
            \end{pmatrix*}
        \end{pmatrix*}
\label{EQ-For (alpha, I, J), calculation for LHS, extra 1-unsimplified the first half, alpha in J}
\end{align}
for the case that $\alpha\in\indexJ$.

The reader can find Lemmas \ref{lemma-simplification for LHS, extra 1, part 1-lemma 1}–\ref{lemma-simplification for LHS, extra 1, part 1-lemma 4} in Section \ref{apex-simplification for LHS, extra 1, part 1} for the calculation of the simplification of the signs of permutations. Thus, we conclude that when $\alpha\in\indexJ$ the expression of $
\iota_{\basedualxi{\alpha}}\,\LieagbdPartialbar
(\basexi{\indexI} \wedge \basedualxi{\indexJp})
$ is
\begin{align}
    2
    \begin{pmatrix*}[l]
        \displaystyle\sum_{\indexi{k}\in\indexI}
        \displaystyle\sum_{t^{'}\in\indexJpc}
        \begin{pmatrix*}[l]
            (-1)^{k}
            \sgn
                {\indexI}
                {\indexi{k},\indexI\backslash\{\indexi{k}\}}
            \sgn
                {t^{'},\indexI,\indexJp}
                {\alpha,\indexI,\order{(\indexJp\backslash\{\alpha^{'}\})\cup\{t^{'}\}}}
            \\
            \coefB{\indexi{k}}{\indexi{k} t^{'}}
            \basexi{\indexI}
            \wedge
            \basedualxi{\order{(\indexJp\backslash\{\alpha^{'}\})\cup\{t^{'}\}}}
        \end{pmatrix*}
        \\
        +
        \\
        \displaystyle\sum_{\indexi{k}\in\indexI}
        \displaystyle\sum_{s\in\indexIc}
        \displaystyle\sum_{t^{'}\in\indexJpc}
        \begin{pmatrix*}[l]
            \sgn
                {s,t^{'},\indexI,\indexJp}
                {\indexi{k},\alpha^{'},\order{(\indexI\backslash\{\indexi{k}\})\cup\{s\}},\order{(\indexJp\backslash\{\alpha^{'}\})\cup\{t^{'}\}}}
            \\
            \coefB{\indexi{k}}{s t^{'}}
            \basexi{\order{(\indexI\backslash\{\indexi{k}\})\cup\{s\}}}
            \wedge
            \basedualxi{\order{(\indexJp\backslash\{\alpha^{'}\})\cup\{t^{'}\}}}
        \end{pmatrix*}
        \\
        +
        \\
        \displaystyle\sum_{\indexjp{l}\in\indexJp}
        \displaystyle\sum_{t^{'}\in\indexJpc}
        \begin{pmatrix*}[l]
            \sgn
                {\order{\indexjp{l},t^{'}}}
                {t^{'},\indexjp{l}}
            \sgn
                {t^{'},\indexI,\indexJp}
                {\alpha^{'},\indexI,\order{(\indexJp\backslash\{\alpha^{'}\})\cup\    {t^{'}\}}}}
            \\
            \coefD{\indexjp{l}}{\order{\indexjp{l},t^{'}}}
            \basexi{\indexI}
            \wedge
            \basedualxi{\order{(\indexJp\backslash\{\alpha^{'}\})\cup\{t^{'}\}}}
        \end{pmatrix*}
        \\
        +
        \\
        \displaystyle\sum_{\indexjp{l}\in\indexJp\backslash\{\alpha^{'}\}}
        \displaystyle\sum_{t^{'}\in\indexJpc}
        \displaystyle\sum_{u^{'}\in\indexJpc}
        \begin{pmatrix*}[l]
            \sgn
                {\order{t^{'},u^{'}},\indexI,\indexJp}
                {\indexjp{l},\alpha^{'},\indexI,\order{(\indexJp\backslash\{\indexjp{l},\alpha^{'}\})\cup\{t^{'},u^{'}\}}}
            \\
            \frac{1}{2}
            \coefD{\indexjp{l}}{\order{t^{'},u^{'}}}
            \basexi{\indexI}
            \wedge
            \basedualxi{\order{(\indexJp\backslash\{\indexjp{l},\alpha^{'}\})\cup\{t^{'},u^{'}\}}}
        \end{pmatrix*}
    \end{pmatrix*}
    .
\label{EQ-For (alpha, I, J), calculation for LHS, extra 1-simplified result of part 1, alpha in J}
\end{align}



Next, we observe that if $\alpha\notin\indexJ$, then the terms in (\ref{EQ-For (alpha, I, J), calculation for LHS, extra 1-no3}) are calculated as follows.
\begin{itemize}
    \item 
$
    \contract{\basedualxi{\alpha^{'}}}
    \big(
    \basexi{\indexI}
    \wedge
    \basedualxi{\order{\indexJp\cup\{t^{'}\}}}
    \big)
    =
    \begin{cases}
            0
            &
            \text{if }\alpha\neq t
        \\
            \sgn
                {\indexI,\order{\indexJp\cup\{\alpha^{'}\}}}
                {\alpha^{'},\indexI,\indexJp}
            \basexi{\indexI}
            \wedge
            \basedualxi{\indexJp}
            &
            \text{if }\alpha=t
    \end{cases}
$.
\\
    \item 
$
    \qquad
    \contract{\basedualxi{\alpha^{'}}}
    \big(
    \basexi{\order{(\indexI\backslash\{\indexi{k}\})\cup\{s\}}}
    \wedge
    \basedualxi{\order{\indexJp\cup\{t^{'}\}}}
    \big)
$
\\
$
    =
    \begin{cases}
            0
            &
            \text{if }\alpha\neq t
        \\
            \sgn
                {\order{(\indexI\backslash\{\indexi{k}\})\cup\{s\}},\order{\indexJp\cup\{\alpha^{'}\}}}
                {\alpha^{'},\order{(\indexI\backslash\{\indexi{k}\})\cup\{s\}},\indexJp}
            \basexi{\order{(\indexI\backslash\{\indexi{k}\})\cup\{s\}}}
            \wedge
            \basedualxi{\indexJp}
            &
            \text{if }\alpha=t
    \end{cases}
$.
\\
    \item 
$
    \contract{\basedualxi{\alpha^{'}}}
    \big(
    \basexi{\indexI}
    \wedge
    \basedualxi{\order{\indexJp\cup\{t^{'}\}}}
    \big)
    =
    \begin{cases}
            0
        &
            \text{if }\alpha\neq t
        \\
            \sgn
                {\indexI,\order{\indexJp\cup\{\alpha^{'}\}}}
                {\alpha^{'},\indexI,\indexJp}
            \basexi{\indexI}
            \wedge
            \basedualxi{\indexJp}
        &
            \text{if }\alpha=t
    \end{cases}
$.
    \item 
$
    \qquad
    \contract{\basedualxi{\alpha^{'}}}
    \big(
    \basexi{\indexI}
    \wedge
    \basedualxi{\order{(\indexJp\backslash\{\indexjp{l}\})\cup\{t^{'},u^{'}\}}}
    \big)
$
\\
$
    =
    \begin{cases}
            0       
        &
            \text{if }\alpha\neq t,u
        \\
            \sgn
                {\indexI,\order{(\indexJp\backslash\{\indexjp{l}\})\cup\{t^{'},\alpha^{'}\}}}
                {\alpha^{'},\indexI,\order{(\indexJp\backslash\{\indexjp{l}\})\cup\{t^{'}\}}}
            \basexi{\indexI}
            \wedge
            \basedualxi{\order{(\indexJp\backslash\{\indexjp{l}\})\cup\{t^{'}\}}}
        &
            \text{if }\alpha=u
        \\
            \sgn
                {\indexI,\order{(\indexJp\backslash\{\indexjp{l}\})\cup\{\alpha^{'},u^{'}\}}}
                {\alpha^{'},\indexI,\order{(\indexJp\backslash\{\indexjp{l}\})\cup\{u^{'}\}}}
            \basexi{\indexI}
            \wedge
            \basedualxi{\order{(\indexJp\backslash\{\indexjp{l}\})\cup\{u^{'}\}}}
        &
            \text{if }\alpha=t
    \end{cases}
$ .
\end{itemize}
We use the above observation to calculate the expression of $
    \iota_{\basedualxi{\alpha}}\,\LieagbdPartialbar
    (\basexi{\indexI}\wedge\basedualxi{\indexJp})
$:
\begin{align}
        2
        \begin{pmatrix*}[l]
            \displaystyle\sum_{\indexi{k}\in\indexI}
            \sgn
                {\indexI}
                {\indexi{k},\indexI\backslash\{\indexi{k}\}}
            \sgn
                {\indexi{1},\cdots,\indexi{k},\alpha^{'},\indexi{k+1},\cdots,\indexi{p},\indexJp}
                {\indexI,\order{\indexJp\cup\{\alpha^{'}\}}}
            \sgn
                {\indexI,\order{\indexJp\cup\{\alpha^{'}\}}}
                {\alpha^{'},\indexI,\indexJp}
            \coefB{\indexi{k}}{\indexi{k} \alpha^{'}}
            \basexi{\indexI}
            \wedge
            \basedualxi{\indexJp}
            \\
            +
            \\
            \displaystyle\sum_{\indexi{k}\in\indexI}
            \displaystyle\sum_{s\in\indexIc}
            \begin{pmatrix*}[l]
                \sgn
                    {\indexI}
                    {\indexi{k},\indexI\backslash\{\indexi{k}\}}
                \sgn
                    {\indexi{1},\cdots,\indexi{k-1},s,\alpha^{'},\indexi{k+1},\cdots,\indexi{p},\indexJp}
                    {\order{(\indexI\backslash\{\indexi{k}\})\cup\{s\}},\order{\indexJp\cup\{\alpha^{'}\}}}
                \\
                \sgn
                    {\order{(\indexI\backslash\{\indexi{k}\})\cup\{s\}},\order{\indexJp\cup\{\alpha^{'}\}}}
                    {\alpha^{'},\order{(\indexI\backslash\{\indexi{k}\})\cup\{s\}},\indexJp}
                \\
                \coefB{\indexi{k}}{s \alpha^{'}}
                \basexi{\order{(\indexI\backslash\{\indexi{k}\})\cup\{s\}}}
                \wedge
                \basedualxi{\indexJp}
            \end{pmatrix*}
            \\
            +
            \\
            \displaystyle\sum_{\indexjp{l}\in\indexJp}
                \begin{pmatrix*}[l]
                \sgn
                    {\indexI,\indexJp}
                    {\indexjp{l},\indexI,\indexJp\backslash\{\indexjp{l}\}}
                \sgn
                    {\indexjp{1},\cdots,\indexjp{l-1},\order{\indexjp{l},\alpha^{'}},\indexjp{l+1},\cdots,\indexjp{q}}
                    {\order{\indexJp\cup\{\alpha^{'}\}}}
                \sgn
                    {\indexI,\order{\indexJp\cup\{\alpha^{'}\}}}
                    {\alpha^{'},\indexI,\indexJp}
                \\
                \coefD{\indexjp{l}}{\order{\indexjp{l},\alpha^{'}}}
                \basexi{\indexI}
                \wedge
                \basedualxi{\indexJp}
            \end{pmatrix*}
            \\
            +
            \\
            \displaystyle\sum_{\indexjp{l}\in\indexJp}
            \displaystyle\sum_{t^{'}\in\indexJpc\backslash\{\alpha^{'}\}}
            \begin{pmatrix*}[l]
                \sgn
                    {\indexI,\indexJp}
                    {\indexjp{l},\indexI,\indexJp\backslash\{\indexjp{l}\}}
                \sgn
                    {\indexjp{1},\cdots,\indexjp{l-1},\order{t^{'},\alpha^{'}},\indexjp{l+1},\cdots,\indexjp{q}}
                    {\order{(\indexJp\backslash\{\indexjp{l}\})\cup\{t^{'},\alpha^{'}\}}}
                \\
                \sgn
                    {\indexI,\order{(\indexJp\backslash\{\indexjp{l}\})\cup\{t^{'},\alpha^{'}\}}}
                    {\alpha^{'},\indexI,\order{(\indexJp\backslash\{\indexjp{l}\})\cup\{t^{'}\}}}
                \\
                \frac{1}{2}
                \coefD{\indexjp{l}}{\order{t^{'},\alpha^{'}}}
                \basexi{\indexI}
                \wedge
                \basedualxi{\order{(\indexJp\backslash\{\indexjp{l}\})\cup\{t^{'}\}}}
            \end{pmatrix*}
            \\
            +
            \\
            \displaystyle\sum_{\indexjp{l}\in\indexJp}
            \displaystyle\sum_{u^{'}\in\indexJpc\backslash\{\alpha^{'}\}}
            \begin{pmatrix*}[l]
                \sgn
                    {\indexI,\indexJp}
                    {\indexjp{l},\indexI,\indexJp\backslash\{\indexjp{l}\}}
                \sgn
                    {\indexjp{1},\cdots,\indexjp{l-1},\order{\alpha^{'},u^{'}},\indexjp{l+1},\cdots,\indexjp{q}}
                    {\order{(\indexJp\backslash\{\indexjp{l}\})\cup\{\alpha^{'},u^{'}\}}}
                \\
                \sgn
                    {\indexI,\order{(\indexJp\backslash\{\indexjp{l}\})\cup\{\alpha^{'},u^{'}\}}}
                {\alpha^{'},\indexI,\order{(\indexJp\backslash\{\indexjp{l}\})\cup\{u^{'}\}}}
                \\
                \frac{1}{2}
                \coefD{\indexjp{l}}{\order{\alpha^{'},u^{'}}}
                \basexi{\indexI}
                \wedge
                \basedualxi{\order{(\indexJp\backslash\{\indexjp{l}\})\cup\{u^{'}\}}}
            \end{pmatrix*}
        \end{pmatrix*}
    .
\label{EQ-For (alpha, I, J), calculation for LHS, extra 1-no4}
\end{align}
The last two terms in (\ref{EQ-For (alpha, I, J), calculation for LHS, extra 1-no4}) are same if we identify $u$ with $t$. Therefore, we simplify (\ref{EQ-For (alpha, I, J), calculation for LHS, extra 1-no4}) to
\begin{align}
    &
        \iota_{\basedualxi{\alpha}}\,\LieagbdPartialbar
        (\basexi{\indexI}\wedge\basedualxi{\indexJp})
    \notag\\
    =&
        2
        \begin{pmatrix*}[l]
            \displaystyle\sum_{\indexi{k}\in\indexI}
            \begin{pmatrix*}[l]
                \sgn
                    {\indexI}
                    {\indexi{k},\indexI\backslash\{\indexi{k}\}}
                \sgn
                    {\indexi{1},\cdots,\indexi{k},\alpha^{'},\indexi{k+1},\cdots,\indexi{p},\indexJp}
                    {\indexI,\order{\indexJp\cup\{\alpha^{'}\}}}
                \sgn
                    {\indexI,\order{\indexJp\cup\{\alpha^{'}\}}}
                    {\alpha^{'},\indexI,\indexJp}
                \coefB{\indexi{k}}{\indexi{k} \alpha^{'}}
                \basexi{\indexI}
                \wedge
                \basedualxi{\indexJp}
            \end{pmatrix*}
            \\
            +
            \\
            \displaystyle\sum_{\indexi{k}\in\indexI}
            \displaystyle\sum_{s\in\indexIc}
            \begin{pmatrix*}[l]
                \sgn
                    {\indexI}
                    {\indexi{k},\indexI\backslash\{\indexi{k}\}}
                \sgn
                    {\indexi{1},\cdots,\indexi{k-1},s,\alpha^{'},\indexi{k+1},\cdots,\indexi{p},\indexJp}
                    {\order{(\indexI\backslash\{\indexi{k}\})\cup\{s\}},\order{\indexJp\cup\{\alpha^{'}\}}}
                \\
                \sgn
                    {\order{(\indexI\backslash\{\indexi{k}\})\cup\{s\}},\order{\indexJp\cup\{\alpha^{'}\}}}
                    {\alpha^{'},\order{(\indexI\backslash\{\indexi{k}\})\cup\{s\}},\indexJp}
                \\
                \coefB{\indexi{k}}{s \alpha^{'}}
                \basexi{\order{(\indexI\backslash\{\indexi{k}\})\cup\{s\}}}
                \wedge
                \basedualxi{\indexJp}
            \end{pmatrix*}
            \\
            +
            \\
            \displaystyle\sum_{\indexjp{l}\in\indexJp}
                \begin{pmatrix*}[l]
                \sgn
                    {\indexI,\indexJp}
                    {\indexjp{l},\indexI,\indexJp\backslash\{\indexjp{l}\}}
                \sgn
                    {\indexjp{1},\cdots,\indexjp{l-1},\order{\indexjp{l},\alpha^{'}},\indexjp{l+1},\cdots,\indexjp{q}}
                    {\order{\indexJp\cup\{\alpha^{'}\}}}
                \\
                \sgn
                    {\indexI,\order{\indexJp\cup\{\alpha^{'}\}}}
                    {\alpha^{'},\indexI,\indexJp}
                \\
                \coefD{\indexjp{l}}{\order{\indexjp{l},\alpha^{'}}}
                \basexi{\indexI}
                \wedge
                \basedualxi{\indexJp}
            \end{pmatrix*}
            \\
            +
            \\
            \displaystyle\sum_{\indexjp{l}\in\indexJp}
            \displaystyle\sum_{t^{'}\in\indexJpc}
            \begin{pmatrix*}[l]
                \sgn
                    {\indexI,\indexJp}
                    {\indexjp{l},\indexI,\indexJp\backslash\{\indexjp{l}\}}
                \sgn
                    {\indexjp{1},\cdots,\indexjp{l-1},\order{t^{'},\alpha^{'}},\indexjp{l+1},\cdots,\indexjp{q}}
                    {\order{(\indexJp\backslash\{\indexjp{l}\})\cup\{t^{'},\alpha^{'}\}}}
                \\
                \sgn
                    {\indexI,\order{(\indexJp\backslash\{\indexjp{l}\})\cup\{t^{'},\alpha^{'}\}}}
                    {\alpha^{'},\indexI,\order{(\indexJp\backslash\{\indexjp{l}\})\cup\{t^{'}\}}}
                \\
                \coefD{\indexjp{l}}{\order{t^{'},\alpha^{'}}}
                \basexi{\indexI}
                \wedge
                \basedualxi{\order{(\indexJp\backslash\{\indexjp{l}\})\cup\{t^{'}\}}}
            \end{pmatrix*}
        \end{pmatrix*}.
\label{EQ-For (alpha, I, J), calculation for LHS, extra 1-unsimplified the first half, alpha not in J}
\end{align}
The reader can find the simplification of the signs of permutations in Lemmas \ref{lemma-simplification for LHS, extra 1, part 1-lemma 5}–\ref{lemma-simplification for LHS, extra 1, part 1-lemma 8} in Section \ref{apex-simplification for LHS, extra 1, part 1}. We conclude that when $\alpha\notin\indexJ$ the expression of $
\iota_{\basedualxi{\alpha}}\,\LieagbdPartialbar
(\basexi{\indexI} \wedge \basedualxi{\indexJp})
$ is
\begin{align}
    &
        \iota_{\basedualxi{\alpha}}\,\LieagbdPartialbar
        (\basexi{\indexI}\wedge\basedualxi{\indexJp})
    \notag\\
    =&
        2
        \begin{pmatrix*}[l]
            \displaystyle\sum_{\indexi{k}\in\indexI}
                (-1)
                \coefB{\indexi{k}}{\indexi{k} \alpha^{'}}
                \basexi{\indexI}
                \wedge
                \basedualxi{\indexJp}
            \\
            +
            \\
            \displaystyle\sum_{\indexi{k}\in\indexI}
            \displaystyle\sum_{s\in\indexIc}
                \sgn
                    {s,\indexI,\indexJp}
                    {\indexi{k},\order{(\indexI\backslash\{\indexi{k}\})\cup\{s\}},\indexJp}
                \coefB{\indexi{k}}{s \alpha^{'}}
                \basexi{\order{(\indexI\backslash\{\indexi{k}\})\cup\{s\}}}
                \wedge
                \basedualxi{\indexJp}
            \\
            +
            \\
            \displaystyle\sum_{\indexjp{l}\in\indexJp}
                \sgn
                    {\order{\indexjp{l},\alpha^{'}}}
                    {\alpha^{'},\indexjp{l}}
                \coefD{\indexjp{l}}{\order{\indexjp{l},\alpha^{'}}}
                \basexi{\indexI}
                \wedge
                \basedualxi{\indexJp}
            \\
            +
            \\
            \displaystyle\sum_{\indexjp{l}\in\indexJp}
            \displaystyle\sum_{t^{'}\in\indexJpc}
            \begin{pmatrix*}[l]
                \sgn
                    {\order{\alpha^{'},t^{'}}}
                    {t^{'},\alpha^{'}}
                \sgn
                    {t^{'},\indexI,\indexJp}
                    {\indexjp{l},\indexI,\order{(\indexJp\backslash\{\indexjp{l}\})\cup\{t^{'}\}}}
            \\
                \coefD{\indexjp{l}}{\order{t^{'},\alpha^{'}}}
                \basexi{\indexI}
                \wedge
                \basedualxi{\order{(\indexJp\backslash\{\indexjp{l}\})\cup\{t^{'}\}}}
            \end{pmatrix*}
        \end{pmatrix*}
        .
\label{EQ-For (alpha, I, J), calculation for LHS, extra 1-simplified the first half, alpha not in J}
\end{align}
\keyword{Step 2:} Looking at $\contract{\basedualxi{\alpha}}(\basexi{\indexI}\wedge\basedualxi{\indexJp})$,
\begin{align}
    \contract{\basedualxi{\alpha}}
    (\basexi{\indexI}\wedge\basedualxi{\indexJp})
    =
    \begin{cases}
            0
        &
            \alpha\notin\indexJ
        \\
            2\sgn
                {\indexI,\indexJp}
                {\alpha^{'},\indexI,\indexJp\backslash\{\alpha^{'}\}}
            \basexi{\indexI}
            \wedge
            \basedualxi{\indexJp\backslash\{\alpha^{'}\}}
        &
            \alpha\in\indexJ
    \end{cases}.
\label{EQ-For (alpha, I, J), calculation for LHS, extra 1-no5}
\end{align}
We just need to consider the case that $\alpha\in\indexJ$. Applying $\LieagbdPartialbar$ to the formula in (\ref{EQ-For (alpha, I, J), calculation for LHS, extra 1-no5}) such that $\alpha\in\indexJ$, we find that $
\LieagbdPartialbar\,\iota_{\basedualxi{\alpha}}(\basexi{\indexI}\wedge\basedualxi{\indexJp})
$ is equal to
\begin{align}
        2\sgn
            {\indexI,\indexJp}
            {\alpha^{'},\indexI,\indexJp\backslash\{\alpha^{'}\}}
        \begin{pmatrix*}[l]
            \displaystyle\sum_{\indexi{k}\in\indexI}
            \sgn
                {\indexI}
                {\indexi{k},\indexI\backslash\{\indexi{k}\}}
            \basexi{\indexi{1}}
            \wedge\cdots\wedge
            \LieagbdPartialbar\basexi{\indexi{k}}
            \wedge\cdots\wedge
            \basexi{\indexi{p}}
            \wedge
            \basedualxi{\indexJp\backslash\{\alpha^{'}\}}
            \\
            +
            \\
            \displaystyle\sum_{\indexj{l}\in\indexJp\backslash\{\alpha^{'}\}}
            \sgn
                {\indexI,\indexJp\backslash\{\alpha^{'}\}}
                {\indexjp{l},\indexI,\indexJp\backslash\{\alpha^{'},\indexjp{l}\}}
            \basexi{\indexI}
            \wedge
            \basedualxi{\indexjp{1}}
            \wedge\cdots\wedge
            \LieagbdPartialbar\basedualxi{\indexjp{l}}
            \wedge\cdots\wedge
            \widehat{\basedualxi{\alpha^{'}}}
            \wedge\cdots\wedge
            \basedualxi{\indexjp{q}}
        \end{pmatrix*}
\label{EQ-For (alpha, I, J), calculation for LHS, extra 1-no6}
\end{align}
for $\alpha\in\indexJ$. Recalling that $
    \LieagbdPartialbar \xi_i
    =
        \sum_{j,k}^n
        \coefB{i}{jk}
        \basexi{j} \wedge \basedualxi{k}
$
and
\begin{align*}
    \LieagbdPartialbar \basedualxi{i}
    =
        \sum_{j,k=1}^n
        \frac{1}{2}
        \coefD{i}{jk}
        \basedualxi{j} \wedge \basedualxi{k}
    =
        \sum_{j,k=1}^n
        \frac{1}{2}
        \sgn
            {\order{j,k}}
            {j,k}
        \coefD{i}{\order{j,k}}
        \basedualxi{j}\wedge\basedualxi{k}
    ,
\end{align*}
we plug these into (\ref{EQ-For (alpha, I, J), calculation for LHS, extra 1-no6}) and cancel out the terms containing repeated frames. Thus, $
\LieagbdPartialbar\,\iota_{\basedualxi{\alpha}}(\basexi{\indexI}\wedge\basedualxi{\indexJp})
$ for $\alpha\in\indexJ$ has the following expression:
\begin{align*}
    2
    \begin{pmatrix*}[l]
        \displaystyle\sum_{\indexi{k}\in\indexI}
        \begin{pmatrix*}[l]
            \sgn
                {\indexI,\indexJp}
                {\alpha^{'},\indexI,\indexJp\backslash\{\alpha^{'}\}}
            \sgn
                {\indexI}
                {\indexi{k},\indexI\backslash\{\indexi{k}\}}
            \\
            \basexi{\indexi{1}}
            \wedge\cdots\wedge
            \basexi{\indexi{k-1}}\wedge
            \begin{pmatrix*}[l]
                \coefB{\indexi{k}}{\indexi{k} \alpha^{'}}
                \basexi{\indexi{k}}
                \wedge
                \basedualxi{\alpha^{'}}
                \\
                +
                \\
                \displaystyle\sum_{t^{'}\in \indexJpc}
                \coefB{\indexi{k}}{\indexi{k} t^{'}}
                \basexi{\indexi{k}}
                \wedge
                \basedualxi{t^{'}}
                \\
                +
                \\
                \displaystyle\sum_{s\in \indexIc}
                \coefB{\indexi{k}}{s \alpha^{'}}
                \basexi{s}
                \wedge
                \basedualxi{\alpha^{'}}
                \\
                +
                \\
                \displaystyle\sum_{s\in \indexIc}
                \displaystyle\sum_{t^{'}\in \indexJpc}
                \coefB{\indexi{k}}{s t^{'}}
                \basexi{s}
                \wedge
                \basedualxi{t^{'}}
            \end{pmatrix*}
            \wedge\basexi{\indexi{k+1}}
            \wedge\cdots\wedge
            \basexi{\indexi{p}}
            \wedge
            \basedualxi{\indexJp\backslash\{\alpha^{'}\}}
        \end{pmatrix*}
        \\
        +
        \\
        \displaystyle\sum_{\indexjp{l}\in\indexJp\backslash\{\alpha^{'}\}}
        \begin{pmatrix*}[l]
            \sgn
                {\indexI,\indexJp}
                {\alpha^{'},\indexI,\indexJp\backslash\{\alpha^{'}\}}
            \sgn
                {\indexI,\indexJp\backslash\{\alpha^{'}\}}
                {\indexjp{l},\indexI,\indexJp\backslash\{\alpha^{'},\indexjp{l}\}}
            \\
            \basexi{\indexI}
            \wedge
            \basedualxi{\indexjp{1}}
            \wedge\cdots\wedge
            \basedualxi{\indexjp{l-1}}\wedge
            \begin{pmatrix*}[l]
                \frac{1}{2}
                (
                \coefD{\indexjp{l}}{\indexjp{l} \alpha^{'}}
                \basedualxi{\indexjp{l}}
                \wedge
                \basedualxi{\alpha^{'}}
                +
                \coefD{\indexjp{l}}{\alpha^{'} \indexjp{l}}
                \basedualxi{\alpha^{'}}
                \wedge
                \basedualxi{\indexjp{l}}
                )
                \\
                +
                \\
                \frac{1}{2}
                \displaystyle\sum_{t^{'}\in\indexJpc}
                \coefD{\indexjp{l}}{\indexjp{l} t^{'}}
                \basedualxi{\indexjp{l}}
                \wedge
                \basedualxi{t^{'}}
                +
                \coefD{\indexjp{l}}{t^{'} \indexjp{l}}
                \basedualxi{t^{'}}
                \wedge
                \basedualxi{\indexjp{l}}
                \\
                +
                \\
                \frac{1}{2}
                \displaystyle\sum_{t^{'}\in\indexJpc}
                \coefD{\indexjp{l}}{\alpha^{'} t^{'}}
                \basedualxi{\alpha^{'}}
                \wedge
                \basedualxi{t^{'}}
                +
                \coefD{\indexjp{l}}{t^{'} \alpha^{'}}
                \basedualxi{t^{'}}
                \wedge
                \basedualxi{\alpha^{'}}
                \\
                +
                \\
                \frac{1}{2}
                \displaystyle\sum_{t^{'}\in\indexJpc}
                \displaystyle\sum_{u^{'}\in\indexJpc}
                \coefD{\indexjp{l}}{u^{'} t^{'}}
                \basedualxi{u^{'}}
                \wedge
                \basedualxi{t^{'}}
            \end{pmatrix*}
            \wedge\basedualxi{\indexjp{l+1}}
            \\
            \wedge\cdots\wedge
            \widehat{\basedualxi{\alpha^{'}}}
            \wedge\cdots\wedge
            \basedualxi{\indexjp{q}}
        \end{pmatrix*}
    \end{pmatrix*}.
\end{align*}

We reorder the wedge product of frames as follows:
\begin{itemize}
    \item 
$
        \qquad
        \basexi{\indexi{1}}
        \wedge\cdots\wedge
        \basexi{\indexi{k-1}}\wedge
        \bigg(
            \coefB{\indexi{k}}{\indexi{k} \alpha^{'}}
            \basexi{\indexi{k}}
            \wedge
            \basedualxi{\alpha^{'}}
        \bigg)
        \wedge\basexi{\indexi{k+1}}
        \wedge\cdots\wedge
        \basexi{\indexi{p}}
        \wedge
        \basedualxi{\indexJp\backslash\{\alpha^{'}\}}
$
\\
$
    =
        \sgn        
            {\indexi{1},\cdots,\indexi{k},\alpha^{'},\indexi{k+1},\cdots,\indexi{p},\indexJp\backslash\{\alpha^{'}\}}
            {\indexI,\indexJp}
        \coefB{\indexi{k}}{\indexi{k} \alpha^{'}}
        \basexi{\indexI}
        \wedge
        \basedualxi{\indexJp}
$,
    \\
    \item 
$
        \qquad
        \basexi{\indexi{1}}
        \wedge\cdots\wedge
        \basexi{\indexi{k-1}}\wedge
        \bigg(
            \displaystyle\sum_{t^{'}\in \indexJpc}
            \coefB{\indexi{k}}{\indexi{k} t^{'}}
            \basexi{\indexi{k}}
            \wedge
            \basedualxi{t^{'}}
        \bigg)
        \wedge\basexi{\indexi{k+1}}
        \wedge\cdots\wedge
        \basexi{\indexi{p}}
        \wedge
        \basedualxi{\indexJp\backslash\{\alpha^{'}\}}
$
\\
$
        =
        \displaystyle\sum_{t^{'}\in \indexJpc}
        \sgn
            {\indexi{1},\cdots,\indexi{k},t^{'},\indexi{k+1},\cdots,\indexi{p},\indexJp\backslash\{\alpha^{'}\}}
            {\indexI,\order{(\indexJp\backslash\{\alpha^{'}\})\cup\{t^{'}\}}}
        \coefB{\indexi{k}}{\indexi{k} t^{'}}
        \basexi{\indexI}
        \wedge
        \basedualxi{\order{(\indexJp\backslash\{\alpha^{'}\})\cup\{t^{'}\}}} 
$,
    \\
    \item 
$
        \qquad
        \basexi{\indexi{1}}
        \wedge\cdots\wedge
        \basexi{\indexi{k-1}}\wedge
        \bigg(
            \displaystyle\sum_{s\in \indexIc}
            \coefB{\indexi{k}}{s \alpha^{'}}
            \basexi{s}
            \wedge
            \basedualxi{\alpha^{'}}
        \bigg)
        \wedge\basexi{\indexi{k+1}}
        \wedge\cdots\wedge
        \basexi{\indexi{p}}
        \wedge
        \basedualxi{\indexJp\backslash\{\alpha^{'}\}}
$
\\
$
        =
        \displaystyle\sum_{s\in \indexIc}
        \sgn
            {\indexi{1},\cdots,\indexi{k-1},s,\alpha^{'},\indexi{k+1},\cdots,\indexi{p},\indexJp\backslash\{\alpha^{'}\}}
            {\order{(\indexI\backslash\{\indexi{k}\})\cup\{s\}},\indexJp}
        \coefB{\indexi{k}}{s \alpha^{'}}
        \basexi{\order{(\indexI\backslash\{\indexi{k}\})\cup\{s\}}}
        \wedge
        \basedualxi{\indexJp}
$,
    \\
    \item  
$
        \qquad
        \basexi{\indexi{1}}
            \wedge\cdots\wedge
            \basexi{\indexi{k-1}}\wedge
            \bigg(
                \displaystyle\sum_{s\in \indexIc}
                \displaystyle\sum_{t^{'}\in \indexJpc}
                \coefB{\indexi{k}}{s t^{'}}
                \basexi{s}
                \wedge
                \basedualxi{t^{'}}
            \bigg)
            \wedge\basexi{\indexi{k+1}}
            \wedge\cdots\wedge
            \basexi{\indexi{p}}
            \wedge
            \basedualxi{\indexJp\backslash\{\alpha^{'}\}}
$
\\
$
        =
        \displaystyle\sum_{s\in \indexIc}
        \displaystyle\sum_{t^{'}\in \indexJpc}
        \begin{pmatrix*}[l]
            \sgn
                {\indexi{1},\cdots,\indexi{k-1},s,t^{'},\indexi{k+1},\cdots,\indexi{p},\indexJp\backslash\{\alpha^{'}\}}
                {\order{(\indexI\backslash\{\indexi{k}\})\cup\{s\}},\order{(\indexJp\backslash\alpha^{'})\cup{t^{'}}}}
            \\
            \coefB{\indexi{k}}{s t^{'}}
            \basexi{\order{(\indexI\backslash\{\indexi{k}\})\cup\{s\}}}
            \wedge
            \basedualxi{\order{(\indexJp\backslash\alpha^{'})\cup{t^{'}}}}
        \end{pmatrix*}   
$,
    \item 
$
        \qquad
        \basexi{\indexI}
        \wedge
        \basedualxi{\indexjp{1}}
        \wedge\cdots\wedge
        \basedualxi{\indexjp{l-1}}\wedge
        \frac{1}{2}
        \begin{pmatrix*}[l]
            \coefD{\indexjp{l}}{\indexjp{l} \alpha^{'}}
            \basedualxi{\indexjp{l}}
            \wedge
            \basedualxi{\alpha^{'}}
            \\
            +
            \\
            \coefD{\indexjp{l}}{\alpha^{'} \indexjp{l}}
            \basedualxi{\alpha^{'}}
            \wedge
            \basedualxi{\indexjp{l}}
        \end{pmatrix*}
        \wedge\basedualxi{\indexjp{l+1}}
        \wedge\cdots\wedge
        \widehat{\basedualxi{\alpha^{'}}}
        \wedge\cdots\wedge
        \basedualxi{\indexjp{q}}
$
\\
$
        =
        \basexi{\indexI}
        \wedge
        \basedualxi{\indexjp{1}}
        \wedge\cdots\wedge
        \basedualxi{\indexjp{l-1}}\wedge
            \coefD{\indexjp{l}}{\order{\indexjp{l} \alpha^{'}}}
            \basedualxi{\order{\indexjp{l} \alpha^{'}}}
        \wedge\basedualxi{\indexjp{l+1}}
        \wedge\cdots\wedge
        \widehat{\basedualxi{\alpha^{'}}}
        \wedge\cdots\wedge
        \basedualxi{\indexjp{q}}
$
\\
$
        =
        \sgn
            {\indexjp{1},\cdots,\indexjp{l-1},\order{\indexjp{l},\alpha^{'}},\indexjp{l+1},\cdots,\widehat{\alpha^{'}},\cdots,\indexjp{q}}
            {\indexJp}
        \coefD{\indexjp{l}}{\order{\indexjp{l} \alpha^{'}}}
        \basexi{\indexI}
        \wedge
        \basedualxi{\indexJp}
$,
    \\
    \item 
$
        \qquad
        \basexi{\indexI}
        \wedge
        \basedualxi{\indexjp{1}}
        \wedge\cdots\wedge
        \basedualxi{\indexjp{l-1}}\wedge
        \frac{1}{2}
        \displaystyle\sum_{t^{'}\in\indexJpc}
        \begin{pmatrix*}[l]
            \coefD{\indexjp{l}}{\indexjp{l} t^{'}}
            \basedualxi{\indexjp{l}}
            \wedge
            \basedualxi{t^{'}}
            \\
            +
            \\
            \coefD{\indexjp{l}}{t^{'} \indexjp{l}}
            \basedualxi{t^{'}}
            \wedge
            \basedualxi{\indexjp{l}}
        \end{pmatrix*}
        \wedge\basedualxi{\indexjp{l+1}}
        \wedge\cdots\wedge
        \widehat{\basedualxi{\alpha^{'}}}
        \wedge\cdots\wedge
        \basedualxi{\indexjp{q}}
$
\\
$
        =
        \basexi{\indexI}
            \wedge
            \basedualxi{\indexjp{1}}
            \wedge\cdots\wedge
            \basedualxi{\indexjp{l-1}}\wedge
                \displaystyle\sum_{t^{'}\in\indexJpc}
                \coefD{\indexjp{l}}{\order{ \indexjp{l} t^{'}} }
                \basedualxi{\order {\indexjp{l} t^{'}}}
            \wedge\basedualxi{\indexjp{l+1}}
            \wedge\cdots\wedge
            \widehat{\basedualxi{\alpha^{'}}}
            \wedge\cdots\wedge
            \basedualxi{\indexjp{q}}
$
\\
$
        =
        \displaystyle\sum_{t^{'}\in\indexJpc}
        \sgn
            {\indexjp{1},\cdots,\indexjp{l-1},\order{\indexjp{l},t^{'}},\indexjp{l+1},\cdots,\widehat{\alpha^{'}},\cdots,\indexjp{q}}
            {\order{(\indexJp\backslash\{\alpha^{'}\})\cup\{t^{'}\}}}
        \coefD{\indexjp{l}}{\order{\indexjp{l} t^{'}}}
        \basexi{\indexI}
        \wedge
        \basedualxi{\order{(\indexJp\backslash\{\alpha^{'}\})\cup\{t^{'}\}}}
$,
    \\
    \item 
$
        \qquad
        \basexi{\indexI}
        \wedge
        \basedualxi{\indexjp{1}}
        \wedge\cdots\wedge
        \basedualxi{\indexjp{l-1}}\wedge
        \frac{1}{2}
        \displaystyle\sum_{t^{'}\in\indexJpc}
        \begin{pmatrix*}[l]
            \coefD{\indexjp{l}}{\alpha^{'} t^{'}}
            \basedualxi{\alpha^{'}}
            \wedge
            \basedualxi{t^{'}}
            \\
            +
            \\
            \coefD{\indexjp{l}}{t^{'} \alpha^{'}}
            \basedualxi{t^{'}}
            \wedge
            \basedualxi{\alpha^{'}}
        \end{pmatrix*}
        \wedge\basedualxi{\indexjp{l+1}}
        \wedge\cdots\wedge
        \widehat{\basedualxi{\alpha^{'}}}
        \wedge\cdots\wedge
        \basedualxi{\indexjp{q}}
$
\\
$
        =
        \basexi{\indexI}
        \wedge
        \basedualxi{\indexjp{1}}
        \wedge\cdots\wedge
        \basedualxi{\indexjp{l-1}}\wedge
            \displaystyle\sum_{t^{'}\in\indexJpc}
            \coefD{\indexjp{l}}{\order{\alpha^{'} t^{'}}}
            \basedualxi{\order{ \alpha^{'}  t^{'}}}
        \wedge\basedualxi{\indexjp{l+1}}
        \wedge\cdots\wedge
        \widehat{\basedualxi{\alpha^{'}}}
        \wedge\cdots\wedge
        \basedualxi{\indexjp{q}}
$
\\
$
        =
        \displaystyle\sum_{t^{'}\in\indexJpc}
        \sgn
            {\indexjp{1},\cdots,\indexjp{l-1},\order{\alpha^{'},t^{'}},\indexjp{l+1},\cdots,\widehat{\alpha^{'}},\cdots,\indexjp{q}}
            {\order{(\indexJp\backslash\{\indexjp{l}\})\cup\{t^{'}\}}}
        \coefD{\indexjp{l}}{\order{\alpha^{'} t^{'}}}
        \basexi{\indexI}
        \wedge
        \basedualxi{\order{(\indexJp\backslash\{\indexjp{l}\})\cup\{t^{'}\}}}
$,
    \\
    \item 
$
        \qquad
        \begin{matrix*}[l]
            \basexi{\indexI}
            \wedge
            \basedualxi{\indexjp{1}}
            \wedge\cdots\wedge
            \basedualxi{\indexjp{l-1}}\wedge
            \frac{1}{2}
            \bigg(
                \displaystyle\sum_{t^{'}\in\indexJpc}
                \displaystyle\sum_{u^{'}\in\indexJpc}
                \coefD{\indexjp{l}}{u^{'} t^{'}}
                \basedualxi{u^{'}}
                \wedge
                \basedualxi{t^{'}}
            \bigg)
            \wedge\basedualxi{\indexjp{l+1}}
            \wedge\cdots\wedge
            \widehat{\basedualxi{\alpha^{'}}}
            \wedge\cdots
            \\
            \wedge\basedualxi{\indexjp{q}}
        \end{matrix*}
$
\\
$
        =
        \displaystyle\sum_{t^{'}\in\indexJpc}
        \displaystyle\sum_{u^{'}\in\indexJpc}
        \begin{pmatrix*}[l]
            \sgn
                {\indexjp{1},\cdots,\indexjp{l-1},\order{t^{'},u^{'}},\indexjp{l+1},\cdots,\widehat{\alpha^{'}},\cdots,\indexjp{q}}
                {\order{(\indexJp\backslash\{\indexjp{l},\alpha^{'}\})\cup\{t^{'},u^{'}\}}}
            \\
            \frac{1}{2}
            \coefD{\indexjp{l}}{\order{t^{'} u^{'}}}
            \basexi{\indexI}
            \wedge
            \basedualxi{\order{(\indexJp\backslash\{\indexjp{l}, \alpha^{'}\})\cup\{t^{'},u^{'}\}}}
        \end{pmatrix*}
$.
    \\
\end{itemize}
Using the above properties, we calculate $\LieagbdPartialbar\,\contract{\basedualxi{\alpha}}(\basexi{\indexI}\wedge\basedualxi{\indexJp})$ for $\alpha\in\indexJ$ as a sum of eight groups of terms:
\begin{flalign}
    &
        \displaystyle\sum_{\indexi{k}\in\indexI}
        2
        \begin{pmatrix*}[l]
            \sgn
                {\indexI,\indexJp}
                {\alpha^{'},\indexI,\indexJp\backslash\{\alpha^{'}\}}
            \sgn
                {\indexI}
                {\indexi{k},\indexI\backslash\{\indexi{k}\}}
            \sgn
                {\indexi{1},\cdots,\indexi{k},\alpha^{'},\indexi{k+1},\cdots,\indexi{p},\indexJp\backslash\{\alpha^{'}\}}
                {\indexI,\indexJp}
            \\
            \coefB{\indexi{k}}{\indexi{k} \alpha^{'}}
            \basexi{\indexI}
            \wedge
            \basedualxi{\indexJp}
        \end{pmatrix*}
        &&
        \notag\\
    &   +&&
        \notag\\
    &       
        \displaystyle\sum_{\indexi{k}\in\indexI}
        \displaystyle\sum_{t^{'}\in\indexJpc}
        2
        \begin{pmatrix*}[l]
            \sgn
                {\indexI,\indexJp}
                {\alpha^{'},\indexI,\indexJp\backslash\{\alpha^{'}\}}
            \sgn
                {\indexI}
                {\indexi{k},\indexI\backslash\{\indexi{k}\}}
            \sgn
                {\indexi{1},\cdots,\indexi{k},t^{'},\indexi{k+1},\cdots,\indexi{p},\indexJp\backslash\{\alpha^{'}\}}
                {\indexI,\order{(\indexJp\backslash\{\alpha^{'}\})\cup\{t^{'}\}}}
            \\
            \coefB{\indexi{k}}{\indexi{k} t^{'}}
            \basexi{\indexI}
            \wedge
            \basedualxi{\order{(\indexJp\backslash\{\alpha^{'}\})\cup\{t^{'}\}}}
        \end{pmatrix*}
        &&
        \notag\\
    &   +&&
        \notag\\
    &       
        \displaystyle\sum_{\indexi{k}\in\indexI}
        \displaystyle\sum_{s\in\indexIc}
        2
        \begin{pmatrix*}[l]
            \sgn
                {\indexI,\indexJp}
                {\alpha^{'},\indexI,\indexJp\backslash\{\alpha^{'}\}}
            \sgn
                {\indexI}
                {\indexi{k},\indexI\backslash\{\indexi{k}\}}
            \sgn
                {\indexi{1},\cdots,\indexi{k-1},s,\alpha^{'},\indexi{k+1},\cdots,\indexi{p},\indexJp\backslash\{\alpha^{'}\}}
                {\order{(\indexI\backslash\{\indexi{k}\})\cup\{s\}},\indexJp}
            \\
            \coefB{\indexi{k}}{s \alpha^{'}}
            \basexi{\order{(\indexI\backslash\{\indexi{k}\})\cup\{s\}}}
            \wedge
            \basedualxi{\indexJp}
        \end{pmatrix*}
        &&
        \notag\\
    &   +&&
        \notag\\
    &   
        \displaystyle\sum_{\indexi{k}\in\indexI}
        \displaystyle\sum_{s\in\indexIc}
        \displaystyle\sum_{t^{'}\in\indexJpc}
        2
        \begin{pmatrix*}[l]
            \sgn
                {\indexI,\indexJp}
                {\alpha^{'},\indexI,\indexJp\backslash\{\alpha^{'}\}}
            \sgn
                {\indexI}
                {\indexi{k},\indexI\backslash\{\indexi{k}\}}
            \sgn
                {\indexi{1},\cdots,\indexi{k-1},s,t^{'},\indexi{k+1},\cdots,\indexi{p},\indexJp\backslash\{\alpha^{'}\}}
                {\order{(\indexI\backslash\{\indexi{k}\})\cup\{s\}},\order{(\indexJp\backslash\alpha^{'})\cup{t^{'}}}}
            \\
            \coefB{\indexi{k}}{s t^{'}}
            \basexi{\order{(\indexI\backslash\{\indexi{k}\})\cup\{s\}}}
            \wedge
            \basedualxi{\order{(\indexJp\backslash\alpha^{'})\cup{t^{'}}}}
        \end{pmatrix*}
        &&
        \notag\\
    &   +&&
        \notag\\
    &        
        \displaystyle\sum_{\indexjp{l}\in\indexJp\backslash\{\alpha^{'}\}}
        2
        \begin{pmatrix*}[l]
            \sgn
                {\indexI,\indexJp}
                {\alpha^{'},\indexI,\indexJp\backslash\{\alpha^{'}\}}
            \sgn
                {\indexI,\indexJp\backslash\{\alpha^{'}\}}
                {\indexjp{l},\indexI,\indexJp\backslash\{\alpha^{'},\indexjp{l}\}}
            \sgn
                {\indexjp{1},\cdots,\indexjp{l-1},\order{\indexjp{l},\alpha^{'}},\indexjp{l+1},\cdots,\widehat{\alpha^{'}},\cdots,\indexjp{q}}
                {\indexJp}
            \\
            \coefD{\indexjp{l}}{\order{\indexjp{l} \alpha^{'}}}
            \basexi{\indexI}
            \wedge
            \basedualxi{\indexJp}
        \end{pmatrix*}
        &&
        \notag\\
    &   +&&
        \notag\\
    &        
        \displaystyle\sum_{\indexjp{l}\in\indexJp\backslash\{\alpha^{'}\}}
        \displaystyle\sum_{t^{'}\in\indexJpc}
        2
        \begin{pmatrix*}[l]
            \sgn
                {\indexI,\indexJp}
                {\alpha^{'},\indexI,\indexJp\backslash\{\alpha^{'}\}}
            \sgn
                {\indexI,\indexJp\backslash\{\alpha^{'}\}}
                {\indexjp{l},\indexI,\indexJp\backslash\{\alpha^{'},\indexjp{l}\}}
            \sgn
                {\indexjp{1},\cdots,\indexjp{l-1},\order{\indexjp{l},t^{'}},\indexjp{l+1},\cdots,\widehat{\alpha^{'}},\cdots,\indexjp{q}}
                {\order{(\indexJp\backslash\{\alpha^{'}\})\cup\{t^{'}\}}}
            \\
            \coefD{\indexjp{l}}{\order{\indexjp{l} t^{'}}}
            \basexi{\indexI}
            \wedge
            \basedualxi{\order{(\indexJp\backslash\{\alpha^{'}\})\cup\{t^{'}\}}}
        \end{pmatrix*}
        &&
        \notag\\
    &   +&&
        \notag\\
    &       
        \displaystyle\sum_{\indexjp{l}\in\indexJp\backslash\{\alpha^{'}\}}
        \displaystyle\sum_{t^{'}\in\indexJpc}
        2
        \begin{pmatrix*}[l]
            \sgn
                {\indexI,\indexJp}
                {\alpha^{'},\indexI,\indexJp\backslash\{\alpha^{'}\}}
            \sgn
                {\indexI,\indexJp\backslash\{\alpha^{'}\}}
                {\indexjp{l},\indexI,\indexJp\backslash\{\alpha^{'},\indexjp{l}\}}
            \sgn
                {\indexjp{1},\cdots,\indexjp{l-1},\order{\alpha^{'},t^{'}},\indexjp{l+1},\cdots,\widehat{\alpha^{'}},\cdots,\indexjp{q}}
                {\order{(\indexJp\backslash\{\indexjp{l}\})\cup\{t^{'}\}}}
            \\
            \coefD{\indexjp{l}}{\order{\alpha^{'} t^{'}}}
            \basexi{\indexI}
            \wedge
            \basedualxi{\order{(\indexJp\backslash\{\indexjp{l}\})\cup\{t^{'}\}}}
        \end{pmatrix*}
        &&
        \notag\\
    &   +&&
        \notag\\
    &        
        \displaystyle\sum_{\indexjp{l}\in\indexJp\backslash\{\alpha^{'}\}}
        \displaystyle\sum_{t^{'}\in\indexJpc}
        \displaystyle\sum_{u^{'}\in\indexJpc}
        2
        \begin{pmatrix*}[l]
            \sgn
                {\indexI,\indexJp}
                {\alpha^{'},\indexI,\indexJp\backslash\{\alpha^{'}\}}
            \sgn
                {\indexI,\indexJp\backslash\{\alpha^{'}\}}
                {\indexjp{l},\indexI,\indexJp\backslash\{\alpha^{'},\indexjp{l}\}}
            \\
            \sgn
                {\indexjp{1},\cdots,\indexjp{l-1},\order{t^{'},u^{'}},\indexjp{l+1},\cdots,\widehat{\alpha^{'}},\cdots,\indexjp{q}}
                {\order{(\indexJp\backslash\{\indexjp{l},\alpha^{'}\})\cup\{t^{'},u^{'}\}}}
            \\
            \frac{1}{2}
            \coefD{\indexjp{l}}{\order{t^{'} u^{'}}}
            \basexi{\indexI}
            \wedge
            \basedualxi{\order{(\indexJp\backslash\{\indexjp{l}, \alpha^{'}\})\cup\{t^{'},u^{'}\}}}
        \end{pmatrix*}.
        &&
    \label{EQ-For (alpha, I, J), calculation for LHS, extra 1-unsimplified result of part 2, alpha in J}
\end{flalign}

We simplify the coefficients in Lemmas \ref{lemma-simplification for LHS, extra 1, part 2-lemma 1}–\ref{lemma-simplification for LHS, extra 1, part 2-lemma 8} in Section $\ref{apex-simplification for LHS, extra 1, part 2}$. The simplified expression of $\LieagbdPartialbar\,\contract{\basedualxi{\alpha}}(\basexi{\indexI}\wedge\basedualxi{\indexJp})$ for $\alpha\in\indexJ$ is:

\begin{flalign}
    &
        \displaystyle\sum_{\indexi{k}\in\indexI}
        2
        \begin{pmatrix*}[l]
            (-1)
            \coefB{\indexi{k}}{\indexi{k} \alpha^{'}}
            \basexi{\indexI}
            \wedge
            \basedualxi{\indexJp}
        \end{pmatrix*}
        &&
        \notag\\
    &   +&&
        \notag\\
    &       
        \displaystyle\sum_{\indexi{k}\in\indexI}
        \displaystyle\sum_{t^{'}\in\indexJpc}
        2
        \begin{pmatrix*}[l]
            -(-1)^{k}
            \sgn
                {\indexI}
                {\indexi{k},\indexI\backslash\{\indexi{k}\}}
            \sgn
                {t^{'},\indexI,\indexJp}
                {\alpha,\indexI,\order{(\indexJp\backslash\{\alpha^{'}\})\cup\{t^{'}\}}}
            \coefB{\indexi{k}}{\indexi{k} t^{'}}
            \basexi{\indexI}
            \wedge
            \basedualxi{\order{(\indexJp\backslash\{\alpha^{'}\})\cup\{t^{'}\}}}
        \end{pmatrix*}
        &&
        \notag\\
    &   +&&
        \notag\\
    &       
        \displaystyle\sum_{\indexi{k}\in\indexI}
        \displaystyle\sum_{s\in\indexIc}
        2
        \begin{pmatrix*}[l]
            \sgn
                {s,\indexI,\indexJp}
                {\indexi{k},\order{(\indexI\backslash\{\indexi{k}\})\cup\{s\}},\indexJp}
            \coefB{\indexi{k}}{s \alpha^{'}}
            \basexi{\order{(\indexI\backslash\{\indexi{k}\})\cup\{s\}}}
            \wedge
            \basedualxi{\indexJp}
        \end{pmatrix*}
        &&
        \notag\\
    &   +&&
        \notag\\
    &   
        \displaystyle\sum_{\indexi{k}\in\indexI}
        \displaystyle\sum_{s\in\indexIc}
        \displaystyle\sum_{t^{'}\in\indexJpc}
        2
        \begin{pmatrix*}[l]
            (-1)
            \sgn
                {s,t^{'},\indexI,\indexJp}
                {\indexi{k},\alpha^{'},\order{(\indexI\backslash\{\indexi{k}\})\cup\{s\}},\order{(\indexJp\backslash\alpha^{'})\cup{t^{'}}}}
            \\
            \coefB{\indexi{k}}{s t^{'}}
            \basexi{\order{(\indexI\backslash\{\indexi{k}\})\cup\{s\}}}
            \wedge
            \basedualxi{\order{(\indexJp\backslash\alpha^{'})\cup{t^{'}}}}
        \end{pmatrix*}
        &&
        \notag\\
    &   +&&
        \notag\\
    &        
        \displaystyle\sum_{\indexjp{l}\in\indexJp\backslash\{\alpha^{'}\}}
        2
        \begin{pmatrix*}[l]
            \sgn
                {\order{\indexjp{l},\alpha^{'}}}
                {\alpha^{'},\indexjp{l}}
            \coefD{\indexjp{l}}{\order{\indexjp{l} \alpha^{'}}}
            \basexi{\indexI}
            \wedge
            \basedualxi{\indexJp}
        \end{pmatrix*}
        &&
        \notag\\
    &   +&&
        \notag\\
    &        
        \displaystyle\sum_{\indexjp{l}\in\indexJp\backslash\{\alpha^{'}\}}
        \displaystyle\sum_{t^{'}\in\indexJpc}
        2
        \begin{pmatrix*}[l]
            (-1)
            \sgn
                {\order{\indexjp{l},t^{'}}}
                {t^{'},\indexjp{l}}
            \sgn
                {t^{'},\indexI,\indexJp}
                {\alpha^{'},\indexI,\order{(\indexJp\backslash\{\alpha^{'}\})\cup\{t^{'}\}}}
            \\
            \coefD{\indexjp{l}}{\order{\indexjp{l} t^{'}}}
            \basexi{\indexI}
            \wedge
            \basedualxi{\order{(\indexJp\backslash\{\alpha^{'}\})\cup\{t^{'}\}}}
        \end{pmatrix*}
        &&
        \notag\\
    &   +&&
        \notag\\
    &       
        \displaystyle\sum_{\indexjp{l}\in\indexJp\backslash\{\alpha^{'}\}}
        \displaystyle\sum_{t^{'}\in\indexJpc}
        2
        \begin{pmatrix*}[l]
            \sgn
                {\order{\alpha^{'},t^{'}}}
                {t^{'},\alpha^{'}}
            \sgn
                {t^{'},\indexI,\indexJp}
                {\indexjp{l},\indexI,\order{(\indexJp\backslash\{\indexjp{l}\})\cup\{t^{'}\}}}
            \\
            \coefD{\indexjp{l}}{\order{\alpha^{'} t^{'}}}
            \basexi{\indexI}
            \wedge
            \basedualxi{\order{(\indexJp\backslash\{\indexjp{l}\})\cup\{t^{'}\}}}
        \end{pmatrix*}
        &&
        \notag\\
    &   +&&
        \notag\\
    &        
        \displaystyle\sum_{\indexjp{l}\in\indexJp\backslash\{\alpha^{'}\}}
        \displaystyle\sum_{t^{'}\in\indexJpc}
        \displaystyle\sum_{u^{'}\in\indexJpc}
        2
        \begin{pmatrix*}[l]
            (-1)
            \sgn
                {\order{t^{'},u^{'}},\indexI,\indexJp}
                {\indexjp{l},\alpha^{'},\indexI,\order{(\indexJp\backslash\{\indexjp{l},\alpha^{'}\})\cup\{t^{'},u^{'}\}}}
            \\
            \frac{1}{2}
            \coefD{\indexjp{l}}{\order{t^{'} u^{'}}}
            \basexi{\indexI}
            \wedge
            \basedualxi{\order{(\indexJp\backslash\{\indexjp{l}, \alpha^{'}\})\cup\{t^{'},u^{'}\}}}
        \end{pmatrix*}.
        &&
    \label{EQ-For (alpha, I, J), calculation for LHS, extra 1-simplified result of part 2, alpha in J}
\end{flalign}
\keyword{Step 3:} We combine the computations of $
\LieagbdPartialbar\,\contract{\basedualxi{\alpha^{'}}}
(\basexi{\indexI}\wedge\basedualxi{\indexJp})
$ and $
\contract{\basedualxi{\alpha^{'}}}\,\LieagbdPartialbar
(\basexi{\indexI}\wedge\basedualxi{\indexJp})
$. The first step is to consider the case that $\alpha\in\indexJ$. The expression of $
(
\LieagbdPartialbar\,\contract{\basedualxi{\alpha^{'}}}
+
\contract{\basedualxi{\alpha^{'}}}\,\LieagbdPartialbar
)
(\basexi{\indexI}\wedge\basedualxi{\indexJp})
$ is equal to (\ref{EQ-For (alpha, I, J), calculation for LHS, extra 1-simplified result of part 1, alpha in J}) + (\ref{EQ-For (alpha, I, J), calculation for LHS, extra 1-simplified result of part 2, alpha in J}):

\begin{align*}
    &2
    \begin{pmatrix*}[l]
        \displaystyle\sum_{\indexi{k}\in\indexI}
        \displaystyle\sum_{t^{'}\in\indexJpc}
        \begin{pmatrix*}[l]
            (-1)^{k}
            \sgn
                {\indexI}
                {\indexi{k},\indexI\backslash\{\indexi{k}\}}
            \sgn
                {t^{'},\indexI,\indexJp}
                {\alpha,\indexI,\order{(\indexJp\backslash\{\alpha^{'}\})\cup\{t^{'}\}}}
            \\
            \coefB{\indexi{k}}{\indexi{k} t^{'}}
            \basexi{\indexI}
            \wedge
            \basedualxi{\order{(\indexJp\backslash\{\alpha^{'}\})\cup\{t^{'}\}}}
        \end{pmatrix*}
        \\
        +
        \\
        \displaystyle\sum_{\indexi{k}\in\indexI}
        \displaystyle\sum_{s\in\indexIc}
        \displaystyle\sum_{t^{'}\in\indexJpc}
        \begin{pmatrix*}[l]
            \sgn
                {s,t^{'},\indexI,\indexJp}
                {\indexi{k},\alpha^{'},\order{(\indexI\backslash\{\indexi{k}\})\cup\{s\}},\order{(\indexJp\backslash\{\alpha^{'}\})\cup\{t^{'}\}}}
            \\
            \coefB{\indexi{k}}{s t^{'}}
            \basexi{\order{(\indexI\backslash\{\indexi{k}\})\cup\{s\}}}
            \wedge
            \basedualxi{\order{(\indexJp\backslash\{\alpha^{'}\})\cup\{t^{'}\}}}
        \end{pmatrix*}
        \\
        +
        \\
        \displaystyle\sum_{\indexjp{l}\in\indexJp}
        \displaystyle\sum_{t^{'}\in\indexJpc}
        \begin{pmatrix*}[l]
            \sgn
                {\order{\indexjp{l},t^{'}}}
                {t^{'},\indexjp{l}}
            \sgn
                {t^{'},\indexI,\indexJp}
                {\alpha^{'},\indexI,\order{(\indexJp\backslash\{\alpha^{'}\})\cup\    {t^{'}\}}}}
            \\
            \coefD{\indexjp{l}}{\order{\indexjp{l},t^{'}}}
            \basexi{\indexI}
            \wedge
            \basedualxi{\order{(\indexJp\backslash\{\alpha^{'}\})\cup\{t^{'}\}}}
        \end{pmatrix*}
        \\
        +
        \\
        \displaystyle\sum_{\indexjp{l}\in\indexJp\backslash\{\alpha^{'}\}}
        \displaystyle\sum_{t^{'}\in\indexJpc}
        \displaystyle\sum_{u^{'}\in\indexJpc}
        \begin{pmatrix*}[l]
            \sgn
                {\order{t^{'},u^{'}},\indexI,\indexJp}
                {\indexjp{l},\alpha^{'},\indexI,\order{(\indexJp\backslash\{\indexjp{l},\alpha^{'}\})\cup\{t^{'},u^{'}\}}}
            \\
            \frac{1}{2}
            \coefD{\indexjp{l}}{\order{t^{'},u^{'}}}
            \basexi{\indexI}
            \wedge
            \basedualxi{\order{(\indexJp\backslash\{\indexjp{l},\alpha^{'}\})\cup\{t^{'},u^{'}\}}}
        \end{pmatrix*}
    \end{pmatrix*}
    \\
    +&
    \\
        &
        \displaystyle\sum_{\indexi{k}\in\indexI}
        2
        \begin{pmatrix*}[l]
            (-1)
            \coefB{\indexi{k}}{\indexi{k} \alpha^{'}}
            \basexi{\indexI}
            \wedge
            \basedualxi{\indexJp}
        \end{pmatrix*}
    \\
    +&
    \\
    &   
        \displaystyle\sum_{\indexi{k}\in\indexI}
        \displaystyle\sum_{t^{'}\in\indexJpc}
        2
        \begin{pmatrix*}[l]
            -(-1)^{k}
            \sgn
                {\indexI}
                {\indexi{k},\indexI\backslash\{\indexi{k}\}}
            \sgn
                {t^{'},\indexI,\indexJp}
                {\alpha,\indexI,\order{(\indexJp\backslash\{\alpha^{'}\})\cup\{t^{'}\}}}
            \coefB{\indexi{k}}{\indexi{k} t^{'}}
            \basexi{\indexI}
            \wedge
            \basedualxi{\order{(\indexJp\backslash\{\alpha^{'}\})\cup\{t^{'}\}}}
        \end{pmatrix*}
    \\
    +&
    \\
    &   
        \displaystyle\sum_{\indexi{k}\in\indexI}
        \displaystyle\sum_{s\in\indexIc}
        2
        \begin{pmatrix*}[l]
            \sgn
                {s,\indexI,\indexJp}
                {\indexi{k},\order{(\indexI\backslash\{\indexi{k}\})\cup\{s\}},\indexJp}
            \coefB{\indexi{k}}{s \alpha^{'}}
            \basexi{\order{(\indexI\backslash\{\indexi{k}\})\cup\{s\}}}
            \wedge
            \basedualxi{\indexJp}
        \end{pmatrix*}
    \\
    +&
    \\
    &   
        \displaystyle\sum_{\indexi{k}\in\indexI}
        \displaystyle\sum_{s\in\indexIc}
        \displaystyle\sum_{t^{'}\in\indexJpc}
        2
        \begin{pmatrix*}[l]
            (-1)
            \sgn
                {s,t^{'},\indexI,\indexJp}
                {\indexi{k},\alpha^{'},\order{(\indexI\backslash\{\indexi{k}\})\cup\{s\}},\order{(\indexJp\backslash\alpha^{'})\cup{t^{'}}}}
            \\
            \coefB{\indexi{k}}{s t^{'}}
            \basexi{\order{(\indexI\backslash\{\indexi{k}\})\cup\{s\}}}
            \wedge
            \basedualxi{\order{(\indexJp\backslash\alpha^{'})\cup{t^{'}}}}
        \end{pmatrix*}
    \\
    +&
    \\
    &   
        \displaystyle\sum_{\indexjp{l}\in\indexJp\backslash\{\alpha^{'}\}}
        2
        \begin{pmatrix*}[l]
            \sgn
                {\order{\indexjp{l},\alpha^{'}}}
                {\alpha^{'},\indexjp{l}}
            \coefD{\indexjp{l}}{\order{\indexjp{l} \alpha^{'}}}
            \basexi{\indexI}
            \wedge
            \basedualxi{\indexJp}
        \end{pmatrix*}
    \\
    +&
    \\
    &   
        \displaystyle\sum_{\indexjp{l}\in\indexJp\backslash\{\alpha^{'}\}}
        \displaystyle\sum_{t^{'}\in\indexJpc}
        2
        \begin{pmatrix*}[l]
            (-1)
            \sgn
                {\order{\indexjp{l},t^{'}}}
                {t^{'},\indexjp{l}}
            \sgn
                {t^{'},\indexI,\indexJp}
                {\alpha^{'},\indexI,\order{(\indexJp\backslash\{\alpha^{'}\})\cup\{t^{'}\}}}
            \\
            \coefD{\indexjp{l}}{\order{\indexjp{l} t^{'}}}
            \basexi{\indexI}
            \wedge
            \basedualxi{\order{(\indexJp\backslash\{\alpha^{'}\})\cup\{t^{'}\}}}
        \end{pmatrix*}
    \\
    +&
    \\
    &   
        \displaystyle\sum_{\indexjp{l}\in\indexJp\backslash\{\alpha^{'}\}}
        \displaystyle\sum_{t^{'}\in\indexJpc}
        2
        \begin{pmatrix*}[l]
            \sgn
                {\order{\alpha^{'},t^{'}}}
                {t^{'},\alpha^{'}}
            \sgn
                {t^{'},\indexI,\indexJp}
                {\indexjp{l},\indexI,\order{(\indexJp\backslash\{\indexjp{l}\})\cup\{t^{'}\}}}
            \\
            \coefD{\indexjp{l}}{\order{\alpha^{'} t^{'}}}
            \basexi{\indexI}
            \wedge
            \basedualxi{\order{(\indexJp\backslash\{\indexjp{l}\})\cup\{t^{'}\}}}
        \end{pmatrix*}
    \\
    +&
    \\
    &   
        \displaystyle\sum_{\indexjp{l}\in\indexJp\backslash\{\alpha^{'}\}}
        \displaystyle\sum_{t^{'}\in\indexJpc}
        \displaystyle\sum_{u^{'}\in\indexJpc}
        2
        \begin{pmatrix*}[l]
            (-1)
            \sgn
                {\order{t^{'},u^{'}},\indexI,\indexJp}
                {\indexjp{l},\alpha^{'},\indexI,\order{(\indexJp\backslash\{\indexjp{l},\alpha^{'}\})\cup\{t^{'},u^{'}\}}}
            \\
            \frac{1}{2}
            \coefD{\indexjp{l}}{\order{t^{'} u^{'}}}
            \basexi{\indexI}
            \wedge
            \basedualxi{\order{(\indexJp\backslash\{\indexjp{l}, \alpha^{'}\})\cup\{t^{'},u^{'}\}}}
        \end{pmatrix*}
    .
\end{align*}
We merge the terms that have the same wedge product of frames and simplify the expression of $
(
\LieagbdPartialbar\,\contract{\basedualxi{\alpha^{'}}}
+
\contract{\basedualxi{\alpha^{'}}}\,\LieagbdPartialbar
)
(\basexi{\indexI}\wedge\basedualxi{\indexJp})
$ for $\alpha\in\indexJ$ to the following one:
\begin{align}
    &
        2
        \begin{pmatrix*}[l]
            \displaystyle\sum_{\indexi{k}\in\indexI}
            \displaystyle\sum_{t^{'}\in\indexJpc}
            \begin{pmatrix*}[l]
                (-1)^{k}
                \sgn
                    {\indexI}
                    {\indexi{k},\indexI\backslash\{\indexi{k}\}}
                \sgn
                    {t^{'},\indexI,\indexJp}
                    {\alpha,\indexI,\order{(\indexJp\backslash\{\alpha^{'}\})\cup\{t^{'}\}}}
                \\
                \coefB{\indexi{k}}{\indexi{k} t^{'}}
                \basexi{\indexI}
                \wedge
                \basedualxi{\order{(\indexJp\backslash\{\alpha^{'}\})\cup\{t^{'}\}}}
            \end{pmatrix*}
        \\
        +
        \\
            \displaystyle\sum_{\indexi{k}\in\indexI}
            \displaystyle\sum_{t^{'}\in\indexJpc}
            \begin{pmatrix*}[l]
                -1(-1)^{k}
                \sgn
                    {\indexI}
                    {\indexi{k},\indexI\backslash\{\indexi{k}\}}
                \sgn
                    {t^{'},\indexI,\indexJp}
                    {\alpha,\indexI,\order{(\indexJp\backslash\{\alpha^{'}\})\cup\{t^{'}\}}}
                \\
                \coefB{\indexi{k}}{\indexi{k} t^{'}}
                \basexi{\indexI}
                \wedge
                \basedualxi{\order{(\indexJp\backslash\{\alpha^{'}\})\cup\{t^{'}\}}}
            \end{pmatrix*}
        \end{pmatrix*}
\label{EQ-For (alpha, I, J), calculation for LHS-together-block 1}
    \\
    +&\notag
    \\
    &
        2
        \begin{pmatrix*}[l]
            \displaystyle\sum_{\indexi{k}\in\indexI}
            \displaystyle\sum_{s\in\indexIc}
            \displaystyle\sum_{t^{'}\in\indexJpc}
            \begin{pmatrix*}[l]
                \sgn
                    {s,t^{'},\indexI,\indexJp}
                    {\indexi{k},\alpha^{'},\order{(\indexI\backslash\{\indexi{k}\})\cup\{s\}},\order{(\indexJp\backslash\{\alpha^{'}\})\cup\{t^{'}\}}}
                \\
                \coefB{\indexi{k}}{s t^{'}}
                \basexi{\order{(\indexI\backslash\{\indexi{k}\})\cup\{s\}}}
                \wedge
                \basedualxi{\order{(\indexJp\backslash\{\alpha^{'}\})\cup\{t^{'}\}}}
            \end{pmatrix*}
        \\
        +
        \\
            \displaystyle\sum_{\indexi{k}\in\indexI}
            \displaystyle\sum_{s\in\indexIc}
            \displaystyle\sum_{t^{'}\in\indexJpc}
            \begin{pmatrix*}[l]
                (-1)
                \sgn
                    {s,t^{'},\indexI,\indexJp}
                    {\indexi{k},\alpha^{'},\order{(\indexI\backslash\{\indexi{k}\})\cup\{s\}},\order{(\indexJp\backslash\alpha^{'})\cup{t^{'}}}}
                \\
                \coefB{\indexi{k}}{s t^{'}}
                \basexi{\order{(\indexI\backslash\{\indexi{k}\})\cup\{s\}}}
                \wedge
                \basedualxi{\order{(\indexJp\backslash\alpha^{'})\cup{t^{'}}}}
            \end{pmatrix*}
        \end{pmatrix*}
\label{EQ-For (alpha, I, J), calculation for LHS-together-block 2}
    \\
    +&\notag
    \\
    &
        2
        \begin{pmatrix*}[l] 
            \displaystyle\sum_{\indexjp{l}\in\indexJp}
            \displaystyle\sum_{t^{'}\in\indexJpc}
            \begin{pmatrix*}[l]
                \sgn
                    {\order{\indexjp{l},t^{'}}}
                    {t^{'},\indexjp{l}}
                \sgn
                    {t^{'},\indexI,\indexJp}
                    {\alpha^{'},\indexI,\order{(\indexJp\backslash\{\alpha^{'}\})\cup\    {t^{'}\}}}}
                \\
                \coefD{\indexjp{l}}{\order{\indexjp{l},t^{'}}}
                \basexi{\indexI}
                \wedge
                \basedualxi{\order{(\indexJp\backslash\{\alpha^{'}\})\cup\{t^{'}\}}}
            \end{pmatrix*}
        \\
        +
        \\
            \displaystyle\sum_{\indexjp{l}\in\indexJp\backslash\{\alpha^{'}\}}
            \displaystyle\sum_{t^{'}\in\indexJpc}
            \begin{pmatrix*}[l]
                (-1)
                \sgn
                    {\order{\indexjp{l},t^{'}}}
                    {t^{'},\indexjp{l}}
                \sgn
                    {t^{'},\indexI,\indexJp}
                    {\alpha^{'},\indexI,\order{(\indexJp\backslash\{\alpha^{'}\})\cup\{t^{'}\}}}
                \\
                \coefD{\indexjp{l}}{\order{\indexjp{l} t^{'}}}
                \basexi{\indexI}
                \wedge
                \basedualxi{\order{(\indexJp\backslash\{\alpha^{'}\})\cup\{t^{'}\}}}
            \end{pmatrix*}
        \end{pmatrix*}
\label{EQ-For (alpha, I, J), calculation for LHS-together-block 3}
    \\
    +&\notag
    \\
    &
        2
        \begin{pmatrix*}[l] 
            \displaystyle\sum_{\indexjp{l}\in\indexJp\backslash\{\alpha^{'}\}}
            \displaystyle\sum_{t^{'}\in\indexJpc}
            \displaystyle\sum_{u^{'}\in\indexJpc}
            \begin{pmatrix*}[l]
                \sgn
                    {\order{t^{'},u^{'}},\indexI,\indexJp}
                    {\indexjp{l},\alpha^{'},\indexI,\order{(\indexJp\backslash\{\indexjp{l},\alpha^{'}\})\cup\{t^{'},u^{'}\}}}
                \\
                \frac{1}{2}
                \coefD{\indexjp{l}}{\order{t^{'},u^{'}}}
                \basexi{\indexI}
                \wedge
                \basedualxi{\order{(\indexJp\backslash\{\indexjp{l},\alpha^{'}\})\cup\{t^{'},u^{'}\}}}
            \end{pmatrix*}
        \\
        +
        \\
            \displaystyle\sum_{\indexjp{l}\in\indexJp\backslash\{\alpha^{'}\}}
            \displaystyle\sum_{t^{'}\in\indexJpc}
            \displaystyle\sum_{u^{'}\in\indexJpc}
            \begin{pmatrix*}[l]
                (-1)
                \sgn
                    {\order{t^{'},u^{'}},\indexI,\indexJp}
                    {\indexjp{l},\alpha^{'},\indexI,\order{(\indexJp\backslash\{\indexjp{l},\alpha^{'}\})\cup\{t^{'},u^{'}\}}}
                \\
                \frac{1}{2}
                \coefD{\indexjp{l}}{\order{t^{'} u^{'}}}
                \basexi{\indexI}
                \wedge
                \basedualxi{\order{(\indexJp\backslash\{\indexjp{l}, \alpha^{'}\})\cup\{t^{'},u^{'}\}}}
            \end{pmatrix*}
        \end{pmatrix*}
\label{EQ-For (alpha, I, J), calculation for LHS-together-block 4}
    \\
    +&\notag
    \\
    &
        2
        \begin{pmatrix*}[l]
            \displaystyle\sum_{\indexi{k}\in\indexI}
            \begin{pmatrix*}[l]
                (-1)
                \coefB{\indexi{k}}{\indexi{k} \alpha^{'}}
                \basexi{\indexI}
                \wedge
                \basedualxi{\indexJp}
            \end{pmatrix*}
        \\
        +
        \\
            \displaystyle\sum_{\indexi{k}\in\indexI}
            \displaystyle\sum_{s\in\indexIc}
            \begin{pmatrix*}[l]
                \sgn
                    {s,\indexI}
                    {\indexi{k},\order{(\indexI\backslash\{\indexi{k}\})\cup\{s\}}}
                \coefB{\indexi{k}}{s \alpha^{'}}
                \basexi{\order{(\indexI\backslash\{\indexi{k}\})\cup\{s\}}}
                \wedge
                \basedualxi{\indexJp}
            \end{pmatrix*}
        \\
        +
        \\
            \displaystyle\sum_{\indexjp{l}\in\indexJp\backslash\{\alpha^{'}\}}
            \begin{pmatrix*}[l]
                \sgn
                    {\order{\indexjp{l},\alpha^{'}}}
                    {\alpha^{'},\indexjp{l}}
                \coefD{\indexjp{l}}{\order{\indexjp{l} \alpha^{'}}}
                \basexi{\indexI}
                \wedge
                \basedualxi{\indexJp}
            \end{pmatrix*}
        \\
        +
        \\
            \displaystyle\sum_{\indexjp{l}\in\indexJp\backslash\{\alpha^{'}\}}
            \displaystyle\sum_{t^{'}\in\indexJpc}
            \begin{pmatrix*}[l]
                \sgn
                    {\order{\alpha^{'},t^{'}}}
                    {t^{'},\alpha^{'}}
                \sgn
                    {t^{'},\indexI,\indexJp}
                    {\indexjp{l},\indexI,\order{(\indexJp\backslash\{\indexjp{l}\})\cup\{t^{'}\}}}
                \\
                \coefD{\indexjp{l}}{\order{\alpha^{'} t^{'}}}
                \basexi{\indexI}
                \wedge
                \basedualxi{\order{(\indexJp\backslash\{\indexjp{l}\})\cup\{t^{'}\}}}
            \end{pmatrix*}
        \end{pmatrix*}
\label{EQ-For (alpha, I, J), calculation for LHS-together-block 5}
    .
\end{align}
We can easily check that:
\begin{itemize}
    \item (\ref{EQ-For (alpha, I, J), calculation for LHS-together-block 1}) is equal to 0.
    \item (\ref{EQ-For (alpha, I, J), calculation for LHS-together-block 2}) is equal to 0.
    \item (\ref{EQ-For (alpha, I, J), calculation for LHS-together-block 3}) is equal to
\begin{align}
        2
        \displaystyle\sum_{t^{'}\in\indexJpc}
        \begin{pmatrix*}[l]
            \sgn
                {\order{\alpha^{'},t^{'}}}
                {t^{'},\alpha^{'}}
            \sgn
                {t^{'},\indexI,\indexJp}
                {\alpha^{'},\indexI,\order{(\indexJp\backslash\{\alpha^{'}\})\cup\    {t^{'}\}}}}
            \\
            \coefD{\alpha^{'}}{\order{\alpha^{'},t^{'}}}
            \basexi{\indexI}
            \wedge
            \basedualxi{\order{(\indexJp\backslash\{\alpha^{'}\})\cup\{t^{'}\}}}
        \end{pmatrix*}
    ,
\label{EQ-For (alpha, I, J), calculation for LHS-together-simplified block 3}
\end{align}
    since it can be viewed as 
\begin{align*}
        2
        \displaystyle\sum_{\indexjp{l}=\alpha^{'}}
        \displaystyle\sum_{t^{'}\in\indexJpc}
        \begin{pmatrix*}[l]
            \sgn
                {\order{\alpha^{'},t^{'}}}
                {t^{'},\alpha^{'}}
            \sgn
                {t^{'},\indexI,\indexJp}
                {\indexjp{l},\indexI,\order{(\indexJp\backslash\{\indexjp{l}\})\cup\    {t^{'}\}}}}
            \\
            \coefD{\alpha^{'}}{\order{\alpha^{'},t^{'}}}
            \basexi{\indexI}
            \wedge
            \basedualxi{\order{(\indexJp\backslash\{\indexjp{l}\})\cup\{t^{'}\}}}
        \end{pmatrix*}
    .
\end{align*}
    \item (\ref{EQ-For (alpha, I, J), calculation for LHS-together-block 4}) is equal to 0.
\end{itemize}
Thus, (\ref{EQ-For (alpha, I, J), calculation for LHS-together-simplified block 3}) + (\ref{EQ-For (alpha, I, J), calculation for LHS-together-block 5}) is the simplified expression of $
(
\LieagbdPartialbar\,\contract{\basedualxi{\alpha^{'}}}
+
\contract{\basedualxi{\alpha^{'}}}\,\LieagbdPartialbar
)
(\basexi{\indexI}\wedge\basedualxi{\indexJp})
$ for $\alpha\in\indexJ$. Multiplying this by $-\sqrt{-1}$, we conclude that the expression of $
-\sqrt{-1}
(
\LieagbdPartialbar\,\contract{\basedualxi{\alpha^{'}}}
+
\contract{\basedualxi{\alpha^{'}}}\,\LieagbdPartialbar
)
(\basexi{\indexI}\wedge\basedualxi{\indexJp})
$ for $\alpha\in\indexJ$ is equal to
\begin{align}
    -2\sqrt{-1}
    \begin{pmatrix*}[l]
        \displaystyle\sum_{\indexi{k}\in\indexI}
        \begin{pmatrix*}[l]
            (-1)
            \coefB{\indexi{k}}{\indexi{k} \alpha^{'}}
            \basexi{\indexI}
            \wedge
            \basedualxi{\indexJp}
        \end{pmatrix*}
    \\
    +
    \\
        \displaystyle\sum_{\indexi{k}\in\indexI}
        \displaystyle\sum_{s\in\indexIc}
        \begin{pmatrix*}[l]
            \sgn
                {s,\indexI,}
                {\indexi{k},\order{(\indexI\backslash\{\indexi{k}\})\cup\{s\}},}
            \coefB{\indexi{k}}{s \alpha^{'}}
            \basexi{\order{(\indexI\backslash\{\indexi{k}\})\cup\{s\}}}
            \wedge
            \basedualxi{\indexJp}
        \end{pmatrix*}
    \\
    +
    \\
        \displaystyle\sum_{\indexj{l}\in\indexJp\backslash\{\alpha^{'}\}}
        \begin{pmatrix*}[l]
            \sgn
                {\order{\indexjp{l},\alpha^{'}}}
                {\alpha^{'},\indexjp{l}}
            \coefD{\indexjp{l}}{\order{\indexjp{l} \alpha^{'}}}
            \basexi{\indexI}
            \wedge
            \basedualxi{\indexJp}
        \end{pmatrix*}
    \\
    +
    \\
        \displaystyle\sum_{\indexj{l}\in\indexJp}
        \displaystyle\sum_{t^{'}\in\indexJpc}
        \begin{pmatrix*}[l]
            \sgn
                {\order{\alpha^{'},t^{'}}}
                {t^{'},\alpha^{'}}
            \sgn
                {t^{'},\indexI,\indexJp}
                {\indexjp{l},\indexI,\order{(\indexJp\backslash\{\indexjp{l}\})\cup\{t^{'}\}}}
            \\
            \coefD{\indexjp{l}}{\order{\alpha^{'} t^{'}}}
            \basexi{\indexI}
            \wedge
            \basedualxi{\order{(\indexJp\backslash\{\indexjp{l}\})\cup\{t^{'}\}}}
        \end{pmatrix*}
    \end{pmatrix*}
    .
\label{EQ-For (alpha, I, J), calculation for LHS, extra 1-result, alpha in J}
\end{align}

Next, we look at the case that $\alpha\notin\indexJ$. We observe that $
-\sqrt{-1}(
\LieagbdPartialbar\,\contract{\basedualxi{\alpha^{'}}}
+
\contract{\basedualxi{\alpha^{'}}}\,\LieagbdPartialbar
)
(\basexi{\indexI}\wedge\basedualxi{\indexJp})
$ is equal to 0 + (\ref{EQ-For (alpha, I, J), calculation for LHS, extra 1-simplified the first half, alpha not in J}) multiplied by $-\sqrt{-1}$, which is computed as follows:
\begin{align}
    -2\sqrt{-1}
    \begin{pmatrix*}[l]
        \displaystyle\sum_{\indexi{k}\in\indexI}
            (-1)
            \coefB{\indexi{k}}{\indexi{k} \alpha^{'}}
            \basexi{\indexI}
            \wedge
            \basedualxi{\indexJp}
        \\
        +
        \\
        \displaystyle\sum_{\indexi{k}\in\indexI}
        \displaystyle\sum_{s\in\indexIc}
            \sgn
                {s,\indexI}
                {\indexi{k},\order{(\indexI\backslash\{\indexi{k}\})\cup\{s\}}}
            \coefB{\indexi{k}}{s \alpha^{'}}
            \basexi{\order{(\indexI\backslash\{\indexi{k}\})\cup\{s\}}}
            \wedge
            \basedualxi{\indexJp}
        \\
        +
        \\
        \displaystyle\sum_{\indexjp{l}\in\indexJp}
            \sgn
                {\order{\indexjp{l},\alpha^{'}}}
                {\alpha^{'},\indexjp{l}}
            \coefD{\indexjp{l}}{\order{\indexjp{l},\alpha^{'}}}
            \basexi{\indexI}
            \wedge
            \basedualxi{\indexJp}
        \\
        +
        \\
        \displaystyle\sum_{\indexjp{l}\in\indexJp}
        \displaystyle\sum_{t^{'}\in\indexJpc}
        \begin{pmatrix*}
            \sgn
                {\order{\alpha^{'},t^{'}}}
                {t^{'},\alpha^{'}}
            \sgn
                {t^{'},\indexI,\indexJp}
                {\indexjp{l},\indexI,\order{(\indexJp\backslash\{\indexjp{l}\})\cup\{t^{'}\}}}
            \\
            \coefD{\indexjp{l}}{\order{t^{'},\alpha^{'}}}
            \basexi{\indexI}
            \wedge
            \basedualxi{\order{(\indexJp\backslash\{\indexjp{l}\})\cup\{t^{'}\}}}
        \end{pmatrix*}
    \end{pmatrix*}
    ,
\label{EQ-For (alpha, I, J), calculation for LHS, extra 1-result, alpha not in J}
\end{align}
since $\LieagbdPartialbar\,\contract{\basedualxi{\alpha^{'}}}=0$ when $\alpha\notin\indexJ$.


We see that (\ref{EQ-For (alpha, I, J), calculation for LHS, extra 1-result, alpha in J}) and (\ref{EQ-For (alpha, I, J), calculation for LHS, extra 1-result, alpha not in J}) are same, which implies that the expression of $
(-\sqrt{-1})(\LieagbdPartialbar\,\contract{\basedualxi{\alpha^{'}}}
+
\contract{\basedualxi{\alpha^{'}}}\,\LieagbdPartialbar)
(\basexi{\indexI}\wedge\basedualxi{\indexJp})
$ does not depend on $\alpha\in\indexJ$ or $\alpha\notin\indexJ$.

\end{proof}

\begin{lemma}
\label{lemma-For (alpha, I, J), calculation for LHS-extra 2]}
\begin{align}
&
    (-\frac{\sqrt{-1}}{2})
    \displaystyle\sum_{h=1}^{n}
    \big(
        (\iota_{\basexi{h}} \, \LieagbdPartialbar \basexi{\alpha})
        \wedge
        \iota_{\basedualxi{h^{'}}}(\basexi{\indexI}\wedge\basedualxi{\indexJp})
        -
        (\iota_{\basedualxi{h^{'}}} \, \LieagbdPartialbar \basexi{\alpha})
        \wedge
        \iota_{\basexi{h}}
        (\basexi{\indexI}\wedge\basedualxi{\indexJp})
    \big)
\notag    
\\
=&
    -2\sqrt{-1}
    \begin{pmatrix*}[l]
        \displaystyle\sum_{\indexjp{l}\in\indexJp}
        \begin{matrix*}[l]
            \coefB{\alpha}{\indexj{l} \indexjp{l}}
            \basexi{\indexI}
            \wedge
            \basedualxi{\indexJp}
        \end{matrix*}
        \\
        +
        \\
        \displaystyle\sum_{\indexjp{l}\in\indexJp}
        \displaystyle\sum_{t^{'}\in\indexJpc}
        \begin{matrix*}[l]
            (-1)
            \sgn
                {t^{'},\indexI,\indexJp}
                {\indexjp{l},\indexI,\order{(\indexJp\backslash\{\indexjp{l}\})\cup\{t^{'}\}}}
            \coefB{\alpha}{\indexj{l} t^{'}}
            \basexi{\indexI}
            \wedge
            \basedualxi{\order{(\indexJp\backslash\{\indexjp{l}\})\cup\{t^{'}\}}}
        \end{matrix*}
        \\
        +
        \\
        \displaystyle\sum_{\indexi{k}\in\indexI}
        \begin{matrix*}[l]
            \coefB{\alpha}{\indexi{k} \indexip{k}}
            \basexi{\indexI}
            \wedge
            \basedualxi{\indexJp}
        \end{matrix*}
        \\
        +
        \\
        \displaystyle\sum_{\indexi{k}\in\indexI}
        \displaystyle\sum_{s\in\indexIc}
        \begin{matrix*}[l]
            (-1)
            \sgn
                {s,\indexI}
                {\indexi{k},\order{(\indexI\backslash\{\indexi{k}\})\cup\{s\}}}
            \coefB{\alpha}{s \indexip{k}}
            \basexi{\order{(\indexI\backslash\{\indexi{k}\})\cup\{s\}}}
            \wedge
            \basedualxi{\indexJp}
        \end{matrix*}
    \end{pmatrix*}
    .
\label{EQ-For (alpha, I, J), calculation for LHS, extra 2-result}
\end{align}
\end{lemma}

\begin{proof}
Let us calculate the first half,
\begin{align}
    (-\frac{\sqrt{-1}}{2})
    \displaystyle\sum_{h=1}^{n}
    (\iota_{\basexi{h}}\,\LieagbdPartialbar\basexi{\alpha})
    \wedge
    \iota_{\basedualxi{h^{'}}}(\basexi{\indexI}\wedge\basedualxi{\indexJp})
    .
\label{EQ-For (alpha, I, J), calculation for LHS, extra 2-first part, target}
\end{align}
With the contraction computation
\begin{align*}
    \iota_{\basedualxi{h}}(\basexi{\indexI}\wedge\basedualxi{\indexJp})
    =
    \begin{cases}
        0       & \,h\notin\indexJ \\
        2
        \sgn
            {\indexI,\indexJp}
            {h^{'},\indexI,\indexJp\backslash\{h^{'}\}}
        \basexi{\indexI}
        \wedge
        \basedualxi{\indexJp\backslash\{h^{'}\}}
                & \,h\in\indexJ
    \end{cases},
\end{align*}
we use $\indexj{l}$ to replace $h$ and simplify (\ref{EQ-For (alpha, I, J), calculation for LHS, extra 2-first part, target}) to
\begin{align}
    (-\frac{\sqrt{-1}}{2})
    \displaystyle\sum_{\indexj{l}\in\indexJ}
    (\iota_{\basexi{\indexj{l}}}\,\LieagbdPartialbar\basexi{\alpha})
    \wedge
    \bigg(
    2\,
        \sgn
            {\indexI,\indexJp}
            {\indexjp{l},\indexI,\indexJp\backslash\{\indexjp{l}\}}
        \basexi{\indexI}
        \wedge
        \basedualxi{\indexJp\backslash\{\indexjp{l}\}}
    \bigg).
\label{EQ-For (alpha, I, J), calculation for LHS, extra 2-first part, no 1}
\end{align}

Recall that $
\LieagbdPartialbar \basexi{\alpha}
=
\displaystyle\sum_{s,t}
\coefB{\alpha}{s t^{'}}
\basexi{s}\wedge\basedualxi{t^{'}}
$. We plug this into (\ref{EQ-For (alpha, I, J), calculation for LHS, extra 2-first part, no 1}) and cancel out the terms containing repeated wedge product of frames, obtaining
\begin{align}
    (-\frac{\sqrt{-1}}{2})
    \displaystyle\sum_{\indexjp{l}\in\indexJp}
    (
    \iota_{\basexi{\indexj{l}}}
    \begin{pmatrix*}[l]
        \coefB{\alpha}{\indexj{l} \indexjp{l}}
        \basexi{\indexj{l}}
        \wedge
        \basedualxi{\indexjp{l}}
        \\
        +
        \\
        \displaystyle\sum_{t^{'}\in\indexJpc}
        \coefB{\alpha}{\indexj{l} t^{'}}
        \basexi{\indexj{l}}
        \wedge
        \basedualxi{t^{'}}
    \end{pmatrix*}
    )
    \wedge
    \bigg(
    2
        \sgn
            {\indexI,\indexJp}
            {\indexjp{l},\indexI,\indexJp\backslash\{\indexjp{l}\}}
        \basexi{\indexI}
        \wedge
        \basedualxi{\indexJp\backslash\{\indexjp{l}\}}
    \bigg)
    .
\label{EQ-For (alpha, I, J), calculation for LHS, extra 2-first part, no 2}
\end{align}
The contractions in the first parenthesis are computed as follows:
\begin{align*}
    \iota_{\basexi{\indexj{l}}}
    (\basexi{\indexj{l}}\wedge\basedualxi{\indexjp{l}})
    =
    2\basedualxi{\indexjp{l}}
    \qquad \text{and} \qquad 
    \iota_{\basexi{\indexj{l}}}
    (\basexi{\indexj{l}}\wedge\basedualxi{t^{'}})
    =
    2\basedualxi{t^{'}}
    .
\end{align*}
Thus, we simplify (\ref{EQ-For (alpha, I, J), calculation for LHS, extra 2-first part, no 2}) to
\begin{align}
    (-\frac{\sqrt{-1}}{2})
    \displaystyle\sum_{\indexjp{l}\in\indexJp}
    2
    \begin{pmatrix*}[l]
        \coefB{\alpha}{\indexj{l} \indexjp{l}}
        \basedualxi{\indexjp{l}}
        \\
        +
        \\
        \displaystyle\sum_{t^{'}\in\indexJpc}
        \coefB{\alpha}{\indexj{l} t^{'}}
        \basedualxi{t^{'}}
    \end{pmatrix*}
    \wedge
    \big(
        2
        \sgn
            {\indexI,\indexJp}
            {\indexjp{l},\indexI,\indexJp\backslash\{\indexjp{l}\}}
        \basexi{\indexI}
        \wedge
        \basedualxi{\indexJp\backslash\{\indexjp{l}\}}
    \big)
    .
\label{EQ-For (alpha, I, J), calculation for LHS, extra 2-first part, no 3}
\end{align}
We observe that
\begin{itemize}
    \item 
$
    \basedualxi{\indexjp{l}}
    \wedge
    \basexi{\indexI}\wedge\basedualxi{\indexJp\backslash\{\indexjp{l}\}}
    =
    \sgn
        {\indexjp{l},\indexI,\indexJp\backslash\{\indexjp{l}\}}
        {\indexI,\indexJp}
    \basexi{\indexI}\wedge\basedualxi{\indexJp}    
$,
    \item 
$
    \basedualxi{t^{'}}
    \wedge
    \basexi{\indexI}\wedge\basedualxi{\indexJp\backslash\{\indexjp{l}\}}
    =
    \sgn
        {t^{'},\indexI,\indexJp\backslash\{\indexjp{l}\}}
        {\indexI,\order{(\indexJp\backslash\{\indexjp{l}\})\cup\{t^{'}\}}}
    \basexi{\indexI}
    \wedge
    \basedualxi{\order{(\indexJp\backslash\{\indexjp{l}\})\cup\{t^{'}\}}}
$.
\end{itemize}
After ordering the wedge product of frames, we can simplify (\ref{EQ-For (alpha, I, J), calculation for LHS, extra 2-first part, no 3}) to
\begin{align}
    -2\sqrt{-1}
    \begin{pmatrix*}[l]
        \displaystyle\sum_{\indexjp{l}\in\indexJp}
        \sgn
            {\indexjp{l},\indexI,\indexJp\backslash\{\indexjp{l}\}}
            {\indexI,\indexJp}
        \sgn
            {\indexI,\indexJp}
            {\indexjp{l},\indexI,\indexJp\backslash\{\indexjp{l}\}}
        \coefB{\alpha}{\indexj{l} \indexjp{l}}
        \basexi{\indexI}
        \wedge
        \basedualxi{\indexJp}
        \\
        +
        \\
        \displaystyle\sum_{\indexjp{l}\in\indexJp}
        \displaystyle\sum_{t^{'}\in\indexJpc}
        \begin{matrix*}[l]
            \sgn
                {\indexI,\indexJp}
                {\indexjp{l},\indexI,\indexJp\backslash\{\indexjp{l}\}}
            \sgn
                {t^{'},\indexI,\indexJp\backslash\{\indexjp{l}\}}
                {\indexI,\order{(\indexJp\backslash\{\indexjp{l}\})\cup\{t^{'}\}}}
            \\
            \coefB{\alpha}{\indexj{l} t^{'}}
            \basexi{\indexI}
            \wedge
            \basedualxi{\order{(\indexJp\backslash\{\indexjp{l}\})\cup\{t^{'}\}}}
        \end{matrix*}
    \end{pmatrix*}.
\label{EQ-For (alpha, I, J), calculation for LHS, extra 2-first part, no 4}
\end{align}
We then calculate the sign of permutations:
\begin{itemize}
    \item 
$
    \sgn
        {\indexjp{l},\indexI,\indexJp\backslash\{\indexjp{l}\}}
        {\indexI,\indexJp}
    \sgn
        {\indexI,\indexJp}
        {\indexjp{l},\indexI,\indexJp\backslash\{\indexjp{l}\}}
    =
    1,
$
    \item 
$
    \sgn
        {\indexI,\indexJp}
        {\indexjp{l},\indexI,\indexJp\backslash\{\indexjp{l}\}}
    \sgn
        {t^{'},\indexI,\indexJp\backslash\{\indexjp{l}\}}
        {\indexI,\order{(\indexJp\backslash\{\indexjp{l}\})\cup\{t^{'}\}}}
    =
    \sgn
        {\indexI,\indexJp}
        {\indexjp{l},\indexI,\indexJp\backslash\{\indexjp{l}\}}
    \sgn
        {\indexjp{l},t^{'},\indexI,\indexJp\backslash\{\indexjp{l}\}}
        {\indexjp{l},\indexI,\order{(\indexJp\backslash\{\indexjp{l}\})\cup\{t^{'}\}}}
    \\
    =
    \sgn
        {\indexI,\indexJp}
        {\indexjp{l},\indexI,\indexJp\backslash\{\indexjp{l}\}}
    (-1)
    \sgn
        {t^{'},\indexjp{l},\indexI,\indexJp\backslash\{\indexjp{l}\}}
        {\indexjp{l},\indexI,\order{(\indexJp\backslash\{\indexjp{l}\})\cup\{t^{'}\}}}
    =
    (-1)
    \sgn
        {t^{'},\indexI,\indexJp}
        {\indexjp{l},\indexI,\order{(\indexJp\backslash\{\indexjp{l}\})\cup\{t^{'}\}}}.
$
\end{itemize}

Rearranging (\ref{EQ-For (alpha, I, J), calculation for LHS, extra 2-first part, no 4}), we conclude that the simplified expression of the first half is:
\begin{align}
    -2\sqrt{-1}
    \begin{pmatrix*}[l]
        \displaystyle\sum_{\indexjp{l}\in\indexJp}
        \coefB{\alpha}{\indexj{l} \indexjp{l}}
        \basexi{\indexI}
        \wedge
        \basedualxi{\indexJp}
        \\
        +
        \\
        \displaystyle\sum_{\indexjp{l}\in\indexJp}
        \displaystyle\sum_{t^{'}\in\indexJpc}
        \begin{matrix*}[l]
            (-1)
            \sgn
                {t^{'},\indexI,\indexJp}
                {\indexjp{l},\indexI,\order{(\indexJp\backslash\{\indexjp{l}\})\cup\{t^{'}\}}}
            \coefB{\alpha}{\indexj{l} t^{'}}
            \basexi{\indexI}
            \wedge
            \basedualxi{\order{(\indexJp\backslash\{\indexjp{l}\})\cup\{t^{'}\}}}
        \end{matrix*}
    \end{pmatrix*}.
\label{EQ-For (alpha, I, J), calculation for LHS, extra 2-first part, simplified result}
\end{align}

We then look at the second part,
\begin{align}
    -\frac{\sqrt{-1}}{2}
    \displaystyle\sum_{h=1}^{n}
    (\iota_{\basedualxi{h}}\,\LieagbdPartialbar\basexi{\alpha})
    \wedge
    \iota_{\basexi{h}}(\basexi{\indexI}\wedge\basedualxi{\indexJp})
.
\label{EQ-For (alpha, I, J), calculation for LHS, extra 2-second part, target}
\end{align}
Using the contraction computation
\begin{align}
    \iota_{\basexi{h}}(\basexi{\indexI}\wedge\basedualxi{\indexJp})
    =
    \begin{cases}
        0       & \,h\notin\indexI \\
        2
        \sgn
            {\indexI}
            {h,\indexI\backslash\{h\}}
        \basexi{\indexI\backslash\{h\}}
        \wedge
        \basedualxi{\indexJp}
                & \,h\in\indexI
    \end{cases},
\end{align}
we use $\indexi{k}$ to replace h and simplify (\ref{EQ-For (alpha, I, J), calculation for LHS, extra 2-second part, target}) to
\begin{align}
    -\frac{\sqrt{-1}}{2}
    \displaystyle\sum_{\indexi{k}\in\indexI}
    (\iota_{\basedualxi{k^{'}}}\,\LieagbdPartialbar\basexi{\alpha})
    \wedge
    \bigg(
        2
        \sgn
            {\indexI}
            {\indexi{k},\indexI\backslash\{\indexi{k}\}}
        \basexi{\indexI\backslash\{\indexi{k}\}}
        \wedge
        \basedualxi{\indexJp}
    \bigg)
    .
\label{EQ-For (alpha, I, J), calculation for LHS, extra 2-second part, no 1}
\end{align}
We plug $
\LieagbdPartialbar \basexi{\alpha}
=
\displaystyle\sum_{s,t}
\coefB{\alpha}{s t^{'}}
\basexi{s}\wedge\basedualxi{t^{'}}
$ into (\ref{EQ-For (alpha, I, J), calculation for LHS, extra 2-second part, no 1}) and cancel out the terms containing repeated wedge product of frames, obtaining
\begin{align}
    -\frac{\sqrt{-1}}{2}
    \displaystyle\sum_{\indexi{k}\in\indexI}
    (
    \iota_{\basedualxi{k^{'}}}
    \begin{pmatrix*}[l]
        \coefB{\alpha}{\indexi{k} \indexip{k}}
        \basexi{\indexi{k}}
        \wedge
        \basedualxi{\indexip{k}}
        \\
        +
        \\
        \displaystyle\sum_{s\in\indexIc}
        \coefB{\alpha}{s \indexip{k}}
        \basexi{s}
        \wedge
        \basedualxi{\indexip{k}}
    \end{pmatrix*}
    )
    \wedge
    \bigg(
        2
        \sgn
            {\indexI}
            {\indexi{k},\indexI\backslash\{\indexi{k}\}}
        \basexi{\indexI\backslash\{\indexi{k}\}}
        \wedge
        \basedualxi{\indexJp}
    \bigg)
    .
\label{EQ-For (alpha, I, J), calculation for LHS, extra 2-second part, no 2}
\end{align}
The contractions in the first parenthesis are computed as follows:
\begin{align*}
    \iota_{\basedualxi{\indexip{k}}}
    (\basexi{\indexi{k}}\wedge\basedualxi{\indexip{k}})
    =
    -2\basexi{\indexi{k}}
    \qquad \text{and} \qquad 
    \iota_{\basedualxi{\indexip{k}}}
    (\basexi{s}\wedge\basedualxi{\indexip{k}})
    =
    -2\basexi{s}
    .
\end{align*}
Thus, we simplify (\ref{EQ-For (alpha, I, J), calculation for LHS, extra 2-second part, no 2}) to
\begin{align}
    -\frac{\sqrt{-1}}{2}
    \displaystyle\sum_{\indexi{k}\in\indexI}
    \begin{pmatrix*}[l]
        -2
        \coefB{\alpha}{\indexi{k} \indexip{k}}
        \basexi{\indexi{k}}
        \\
        +
        \\
        \displaystyle\sum_{s\in\indexIc}
        -2
        \coefB{\alpha}{s \indexip{k}}
        \basexi{s}
    \end{pmatrix*}
    \wedge
    \bigg(
        2
        \sgn
            {\indexI}
            {\indexi{k},\indexI\backslash\{\indexi{k}\}}
        \basexi{\indexI\backslash\{\indexi{k}\}}
        \wedge
        \basedualxi{\indexJp}
    \bigg)
    ,
\label{EQ-For (alpha, I, J), calculation for LHS, extra 2-second part, no 3}
\end{align}
in which
\begin{itemize}
    \item 
$
    \basexi{\indexi{k}}
    \wedge\basexi{\indexI\backslash\{\indexi{k}\}}    
    \wedge\basedualxi{\indexJp}
    =
    \sgn
        {\indexI}
        {\indexi{k},\indexI\backslash\{\indexi{k}\}}
    \basexi{\indexI}\wedge\basedualxi{\indexJp}
$,
    \item 
$
    \basexi{s}
    \wedge\basexi{\indexI\backslash\{\indexi{k}\}}    
    \wedge\basedualxi{\indexJp}
    =
    \sgn
        {s,\indexI\backslash\{\indexi{k}\}}
        {\order{(\indexI\backslash\{\indexi{k}\})\cup\{s\}}}
    \basexi{\order{(\indexI\backslash\{\indexi{k}\})\cup\{s\}}}
    \wedge
    \basedualxi{\indexJp}
$.
\end{itemize}
Thus, after ordering the wedge product of frames, we simplify (\ref{EQ-For (alpha, I, J), calculation for LHS, extra 2-second part, no 3}) to
\begin{align}
    -2\sqrt{-1}
    \begin{pmatrix*}[l]
        \displaystyle\sum_{\indexi{k}\in\indexI}
        (-1)
        \sgn
            {\indexI}
            {\indexi{k},\indexI\backslash\{\indexi{k}\}}
        \sgn
            {\indexI}
            {\indexi{k},\indexI\backslash\{\indexi{k}\}}
        \coefB{\alpha}{\indexi{k} \indexip{k}}
        \basexi{\indexI}
        \wedge
        \basedualxi{\indexJp}
        \\
        +
        \\
        \displaystyle\sum_{\indexi{k}\in\indexI}
        \displaystyle\sum_{s\in\indexIc}
        \begin{matrix*}[l]
            \sgn
                {\indexI}
                {\indexi{k},\indexI\backslash\{\indexi{k}\}}
            \sgn
                {s,\indexI\backslash\{\indexi{k}\}}
                {\order{(\indexI\backslash\{\indexi{k}\})\cup\{s\}}}
            (-1)
            \coefB{\alpha}{s \indexip{k}}
            \basexi{\order{(\indexI\backslash\{\indexi{k}\})\cup\{s\}}}
            \wedge
            \basedualxi{\indexJp}
        \end{matrix*}
    \end{pmatrix*}
    .
\label{EQ-For (alpha, I, J), calculation for LHS, extra 2-second part, no 4}
\end{align}
We calculate the sign of permutations:
\begin{itemize}
    \item 
$
    \sgn
        {\indexI}
        {\indexi{k},\indexI\backslash\{\indexi{k}\}}
    \sgn
        {\indexI}
        {\indexi{k},\indexI\backslash\{\indexi{k}\}}
    =
    1
$,
    \item 
$
    \sgn
        {\indexI}
        {\indexi{k},\indexI\backslash\{\indexi{k}\}}
    \sgn
        {s,\indexI\backslash\{\indexi{k}\}}
        {\order{(\indexI\backslash\{\indexi{k}\})\cup\{s\}}}
    =
    \sgn
        {\indexI}
        {\indexi{k},\indexI\backslash\{\indexi{k}\}}
    \sgn
        {\indexi{k},s,\indexI\backslash\{\indexi{k}\}}
        {\indexi{k},\order{(\indexI\backslash\{\indexi{k}\})\cup\{s\}}}
    \\
    =
    \sgn
        {\indexI}
        {\indexi{k},\indexI\backslash\{\indexi{k}\}}
    (-1)
    \sgn
        {s,\indexi{k},\indexI\backslash\{\indexi{k}\}}
        {\indexi{k},\order{(\indexI\backslash\{\indexi{k}\})\cup\{s\}}}
    =
    (-1)
    \sgn
        {s,\indexI}
        {\indexi{k},\order{(\indexI\backslash\{\indexi{k}\})\cup\{s\}}}
$.
\end{itemize}
We then calculate (\ref{EQ-For (alpha, I, J), calculation for LHS, extra 2-second part, no 4}):
\begin{align}
    -2\sqrt{-1}
    \begin{pmatrix*}[l]
        \displaystyle\sum_{\indexi{k}\in\indexI}
        (-1)
        \coefB{\alpha}{\indexi{k} \indexip{k}}
        \basexi{\indexI}
        \wedge
        \basedualxi{\indexJp}
        \\
        +
        \\
        \displaystyle\sum_{\indexi{k}\in\indexI}
        \displaystyle\sum_{s\in\indexIc}
        \begin{matrix*}[l]
            \sgn
                {s,\indexI}
                {\indexi{k},\order{(\indexI\backslash\{\indexi{k}\})\cup\{s\}}}
            \coefB{\alpha}{s \indexip{k}}
            \basexi{\order{(\indexI\backslash\{\indexi{k}\})\cup\{s\}}}
            \wedge
            \basedualxi{\indexJp}
        \end{matrix*}
    \end{pmatrix*}.
\label{EQ-For (alpha, I, J), calculation for LHS, extra 2-second part, simplified result}
\end{align}

Finally, we assemble (\ref{EQ-For (alpha, I, J), calculation for LHS, extra 2-first part, simplified result}) and (\ref{EQ-For (alpha, I, J), calculation for LHS, extra 2-second part, simplified result}) and complete the proof of Lemma \ref{lemma-For (alpha, I, J), calculation for LHS-extra 2]}:
\begin{align*}
    &
    (-\frac{\sqrt{-1}}{2})
    \displaystyle\sum_{k=1}^{n}
    \bigg(
        (\iota_{\basexi{k}}\,\LieagbdPartialbar\basexi{\alpha})
        \wedge
        \iota_{\basedualxi{k}}(\basexi{\indexI}\wedge\basedualxi{\indexJp})
        -
        (\iota_{\basedualxi{k}}\,\LieagbdPartialbar\basexi{\alpha})
        \wedge
        \iota_{\basexi{k}}(\basexi{\indexI}\wedge\basedualxi{\indexJp})
    \bigg)
    \\
    =&
    -2\sqrt{-1}
    \begin{pmatrix*}[l]
        \displaystyle\sum_{\indexjp{l}\in\indexJp}
        \begin{matrix*}[l]
            \coefB{\alpha}{\indexj{l} \indexjp{l}}
            \basexi{\indexI}
            \wedge
            \basedualxi{\indexJp}
        \end{matrix*}
        \\
        +
        \\
        \displaystyle\sum_{\indexjp{l}\in\indexJp}
        \displaystyle\sum_{t^{'}\in\indexJpc}
        \begin{matrix*}[l]
            (-1)
            \sgn
                {t^{'},\indexI,\indexJp}
                {\indexjp{l},\indexI,\order{(\indexJp\backslash\{\indexjp{l}\})\cup\{t^{'}\}}}
            \coefB{\alpha}{\indexj{l} t^{'}}
            \basexi{\indexI}
            \wedge
            \basedualxi{\order{(\indexJp\backslash\{\indexjp{l}\})\cup\{t^{'}\}}}
        \end{matrix*}
        \\
        +
        \\
        \displaystyle\sum_{\indexi{k}\in\indexI}
        \begin{matrix*}[l]
            \coefB{\alpha}{\indexi{k} \indexip{k}}
            \basexi{\indexI}
            \wedge
            \basedualxi{\indexJp}
        \end{matrix*}
        \\
        +
        \\
        \displaystyle\sum_{\indexi{k}\in\indexI}
        \displaystyle\sum_{s\in\indexIc}
        \begin{matrix*}[l]
            (-1)
            \sgn
                {s,\indexI}
                {\indexi{k},\order{(\indexI\backslash\{\indexi{k}\})\cup\{s\}}}
            \coefB{\alpha}{s \indexip{k}}
            \basexi{\order{(\indexI\backslash\{\indexi{k}\})\cup\{s\}}}
            \wedge
            \basedualxi{\indexJp}
        \end{matrix*}
    \end{pmatrix*}.
\end{align*}

\end{proof}

\begin{lemma}
\label{lemma-the final version of the LHS}
\begin{align}
    &
    [\LieagbdPartialbar,\iotaW]
    (\basexi{\alpha}\wedge\basexi{\indexI}\wedge\basedualxi{\indexJp})
\notag
\\
    =&
    \begin{pmatrix*}[l]
        [\partialbar,\iotaW]
        \wedge
        \basexi{\indexI}\wedge\basedualxi{\indexJp}
        -
        \basexi{\alpha}
        \wedge
        [\partialbar,\iotaW]
        (\basexi{\indexI}\wedge\basedualxi{\indexJp})
        \\
        +
        \\
        (-2\sqrt{-1})
        \begin{pmatrix*}[l]
            \displaystyle\sum_{\indexi{k}\in\indexI}
            \begin{matrix*}[l]
                \sgn
                    {\order{\indexi{k},\alpha}}
                    {\indexi{k},\alpha}
                \coefD{\indexi{k}}{\order{\indexi{k},\alpha}}
                \basexi{\indexI}
                \wedge
                \basedualxi{\indexJp}
            \end{matrix*}
            \\
            +
            \\
            \displaystyle\sum_{\indexj{l}\in\indexJp}
            \begin{matrix*}[l]
                \coefB{\indexjp{l}}{\indexjp{l} \alpha}
                \basexi{\indexI}
                \wedge
                \basedualxi{\indexJp}
            \end{matrix*}
            \\
            +
            \\
            \displaystyle\sum_{\indexi{k}\in\indexI}
            \begin{matrix*}[l]
                \sgn
                    {\indexI}
                    {\indexi{k},\indexI\backslash\{\indexi{k}\}}
                \sgn
                    {\order{\indexi{k},\alpha}}
                    {\indexi{k},\alpha}
                \\
                \coefD{\alpha}{\order{\indexi{k},\alpha}}
                \basexi{\alpha}
                \wedge
                \basexi{(\indexI\backslash\{\indexi{k}\})}
                \wedge
                \basedualxi{\indexJp}
            \end{matrix*}
            \\
            +
            \\
            \displaystyle\sum_{\indexi{k}\in\indexI}
            \displaystyle\sum_{s\in\indexIc\backslash\{\alpha\}}
            \begin{matrix*}[l]
                \sgn
                    {s,\indexI}
                    {\indexi{k},\order{(\indexI\backslash\{\indexi{k}\})\cup\{s\}}}
                \sgn
                    {\order{\alpha,\indexi{k}}}
                    {\alpha,\indexi{k}}  
                \\
                \coefD{s}{\order{\indexi{k},\alpha}}
                \basexi{\order{(\indexI\backslash\{\indexi{k}\})\cup\{s\}}}
                \wedge
                \basedualxi{\indexJp}
            \end{matrix*}
            \\
            +
            \\
            \displaystyle\sum_{\indexj{l}\in\indexJp}
            \displaystyle\sum_{t^{'}\in\indexJpc}
            \begin{matrix*}[l]
                \sgn
                    {t^{'},\indexI,\indexJp}
                    {\indexjp{l},\indexI,\order{(\indexJp\backslash\{\indexjp{l}\})\cup\{t^{'}\}}}
                \\
                (-1)
                \coefB{t^{'}}{\indexjp{l} \alpha}
                \basexi{\indexI}
                \wedge
                \basedualxi{\order{(\indexJp\backslash\{\indexjp{l}\})\cup\{t^{'}\}}}
            \end{matrix*}
        \end{pmatrix*}
    \end{pmatrix*}.
\label{EQ-For (alpha, I, J), calculation for LHS-total summary}
\end{align}
\end{lemma}


\begin{proof}
Recall that (\ref{EQ-For (alpha, I, J), calculation for LHS-half summary}) is 
\begin{align*}
    \begin{pmatrix*}[l]
        [\LieagbdPartialbar,\iotaW]
        (\basexi{\alpha})
        \wedge
        \basexi{\indexI}\wedge\basedualxi{\indexJp}
        \\
        -
        \\
        \basexi{\alpha}
        \wedge
        [\LieagbdPartialbar,\iotaW]
        (\basexi{\indexI}\wedge\basedualxi{\indexJp})
        \\
        +
        \\
        (-\sqrt{-1})
        (\LieagbdPartialbar \, \iota_{\basedualxi{\alpha^{'}}}
        +
        \iota_{\basedualxi{\alpha^{'}}}\,\LieagbdPartialbar)
        (\basexi{\indexI}\wedge\basedualxi{\indexJp})
        \\
        +
        \\
        (-\frac{\sqrt{-1}}{2})
        \displaystyle\sum_{h=1}^{n}
        \big(
            (\iota_{\basexi{h}} \, \LieagbdPartialbar \basexi{\alpha})
            \wedge
            \iota_{\basedualxi{h^{'}}}(\basexi{\indexI}\wedge\basedualxi{\indexJp})
            -
            (\iota_{\basedualxi{h^{'}}} \, \LieagbdPartialbar \basexi{\alpha})
            \wedge
            \iota_{\basexi{h}}(\basexi{\indexI}\wedge\basedualxi{\indexJp})
        \big)
    \end{pmatrix*}
    .
\end{align*}
We prove the lemma by substituting the expressions of $
(-\sqrt{-1})
(\LieagbdPartialbar \, \iota_{\basedualxi{\alpha^{'}}}
+
\iota_{\basedualxi{\alpha^{'}}}\,\LieagbdPartialbar)
(\basexi{\indexI}\wedge\basedualxi{\indexJp})
$ and $
(-\frac{\sqrt{-1}}{2})
\displaystyle\sum_{h=1}^{n}
\big(
    (\iota_{\basexi{h}} \, \LieagbdPartialbar \basexi{\alpha})
    \wedge
    \iota_{\basedualxi{h^{'}}}(\basexi{\indexI}\wedge\basedualxi{\indexJp})
    -
    (\iota_{\basedualxi{h^{'}}} \, \LieagbdPartialbar \basexi{\alpha})
    \wedge
    \iota_{\basexi{h}}(\basexi{\indexI}\wedge\basedualxi{\indexJp})
\big)
$ into the above formula.

Recall that Lemma \ref{lemma-For (alpha, I, J), calculation for LHS-extra 1]} states that $(-\sqrt{-1})
(\LieagbdPartialbar \, \iota_{\basedualxi{\alpha^{'}}}
+
\iota_{\basedualxi{\alpha^{'}}}\,\LieagbdPartialbar)
(\basexi{\indexI}\wedge\basedualxi{\indexJp})
$ is equal to
\begin{align}
    &
    (-2\sqrt{-1})
        \displaystyle\sum_{\indexi{k}\in\indexI}
        \begin{pmatrix*}[l]
            (-1)
            \coefB{\indexi{k}}{\indexi{k} \alpha^{'}}
            \basexi{\indexI}
            \wedge
            \basedualxi{\indexJp}
        \end{pmatrix*}
\label{EQ-For (alpha, I, J), calculation for LHS, conclude-ref extra 1,1}
    \\
    +&\notag
    \\
    &
        (-2\sqrt{-1})
        \displaystyle\sum_{\indexi{k}\in\indexI}
        \displaystyle\sum_{s\in\indexIc}
        \begin{pmatrix*}[l]
            \sgn
                {s,\indexI}
                {\indexi{k},\order{(\indexI\backslash\{\indexi{k}\})\cup\{s\}}}
            \coefB{\indexi{k}}{s \alpha^{'}}
            \basexi{\order{(\indexI\backslash\{\indexi{k}\})\cup\{s\}}}
            \wedge
            \basedualxi{\indexJp}
        \end{pmatrix*}
\label{EQ-For (alpha, I, J), calculation for LHS, conclude-ref extra 1,2}
    \\
    +&\notag
    \\
    &
        (-2\sqrt{-1})
        \displaystyle\sum_{\indexj{l}\in\indexJp\backslash\{\alpha^{'}\}}
        \begin{pmatrix*}[l]
            \sgn
                {\order{\indexjp{l},\alpha^{'}}}
                {\alpha^{'},\indexjp{l}}
            \coefD{\indexjp{l}}{\order{\indexjp{l},\alpha^{'}}}
            \basexi{\indexI}
            \wedge
            \basedualxi{\indexJp}
        \end{pmatrix*}
\label{EQ-For (alpha, I, J), calculation for LHS, conclude-ref extra 1,3}
    \\
    +&\notag
    \\
    &
        (-2\sqrt{-1})
        \displaystyle\sum_{\indexj{l}\in\indexJp}
        \displaystyle\sum_{t^{'}\in\indexJpc}
        \begin{pmatrix*}[l]
            \sgn
                {\order{\alpha^{'},t^{'}}}
                {t^{'},\alpha^{'}}
            \sgn
                {t^{'},\indexI,\indexJp}
                {\indexjp{l},\order{(\indexJp\backslash\{\indexjp{l}\})\cup\{t^{'}\}}}
            \\
            \coefD{\indexjp{l}}{\order{\alpha^{'},t^{'}}}
            \basexi{\indexI}
            \wedge
            \basedualxi{\order{(\indexJp\backslash\{\indexjp{l}\})\cup\{t^{'}\}}}
        \end{pmatrix*}
    .
\label{EQ-For (alpha, I, J), calculation for LHS, conclude-ref extra 1,4}
\end{align}
Likewise, recall that Lemma \ref{lemma-For (alpha, I, J), calculation for LHS-extra 2]} states that $
(-\frac{\sqrt{-1}}{2})
\displaystyle\sum_{h=1}^{n}
\big(
    (\iota_{\basexi{h}} \, \LieagbdPartialbar \basexi{\alpha})
    \wedge
    \iota_{\basedualxi{h^{'}}}(\basexi{\indexI}\wedge\basedualxi{\indexJp})
    -
    (\iota_{\basedualxi{h^{'}}} \, \LieagbdPartialbar \basexi{\alpha})
    \wedge
    \iota_{\basexi{h}}(\basexi{\indexI}\wedge\basedualxi{\indexJp})
\big)
$  is equal to
\begin{align}
    &
        (-2\sqrt{-1})
        \displaystyle\sum_{\indexjp{l}\in\indexJp}
        \begin{matrix*}[l]
            \coefB{\alpha}{\indexj{l} \indexjp{l}}
            \basexi{\indexI}
            \wedge
            \basedualxi{\indexJp}
        \end{matrix*}
\label{EQ-For (alpha, I, J), calculation for LHS, conclude-ref extra 2,1}
    \\
    +&\notag
    \\
    &
        (-2\sqrt{-1})
        \displaystyle\sum_{\indexjp{l}\in\indexJp}
        \displaystyle\sum_{t^{'}\in\indexJpc}
        \begin{matrix*}[l]
            (-1)
            \sgn
                {t^{'},\indexI,\indexJp}
                {\indexjp{l},\indexI,\order{(\indexJp\backslash\{\indexjp{l}\})\cup\{t^{'}\}}}
            \coefB{\alpha}{\indexj{l} t^{'}}
            \basexi{\indexI}
            \wedge
            \basedualxi{\order{(\indexJp\backslash\{\indexjp{l}\})\cup\{t^{'}\}}}
        \end{matrix*}
\label{EQ-For (alpha, I, J), calculation for LHS, conclude-ref extra 2,2}
    \\
    +&\notag
    \\
    &
        (-2\sqrt{-1})
        \displaystyle\sum_{\indexi{k}\in\indexI}
        \begin{matrix*}[l]
            \coefB{\alpha}{\indexi{k} \indexip{k}}
            \basexi{\indexI}
            \wedge
            \basedualxi{\indexJp}
        \end{matrix*}
\label{EQ-For (alpha, I, J), calculation for LHS, conclude-ref extra 2,3}
    \\
    +&\notag
    \\
    &
        (-2\sqrt{-1})
        \displaystyle\sum_{\indexi{k}\in\indexI}
        \displaystyle\sum_{s\in\indexIc}
        \begin{matrix*}[l]
            (-1)
            \sgn
                {s,\indexI}
                {\indexi{k},\order{(\indexI\backslash\{\indexi{k}\})\cup\{s\}}}
            \coefB{\alpha}{s \indexip{k}}
            \basexi{\order{(\indexI\backslash\{\indexi{k}\})\cup\{s\}}}
            \wedge
            \basedualxi{\indexJp}
        \end{matrix*}
    .
\label{EQ-For (alpha, I, J), calculation for LHS, conclude-ref extra 2,4}
\end{align}
We combine the similar terms in the above sums:
\begin{itemize}
    \item (\ref{EQ-For (alpha, I, J), calculation for LHS, conclude-ref extra 2,3}) + (\ref{EQ-For (alpha, I, J), calculation for LHS, conclude-ref extra 1,1}) is equal to
    \begin{align}
        (-2\sqrt{-1})
        \displaystyle\sum_{\indexi{k}\in\indexI}
            \bigg(
                \coefB{\alpha}{\indexi{k} \indexip{k}}
                -
                \coefB{\indexi{k}}{\indexi{k} \alpha^{'}}
            \bigg)
            \basexi{\indexI}
            \wedge
            \basedualxi{\indexJp}
        .
    \label{EQ-For (alpha, I, J), calculation for LHS, conclude-classified term 1}
    \end{align}
    Consider the constraint
    \begin{align*}
        \coefB{\alpha}{\indexi{k} \indexip{k}}
        -
        \coefB{\indexi{k}}{\indexi{k} \alpha}
        =
        \sgn
            {\order{\indexi{k}, \alpha}}
            {\indexi{k}, \alpha}
        \coefD{\indexi{k}}{\order{\indexi{k},\alpha}}
        .
    \end{align*}
    Applying it to the terms in (\ref{EQ-For (alpha, I, J), calculation for LHS, conclude-classified term 1}), we obtain the following simplified expression:
    \begin{align}
        (-2\sqrt{-1})
        \displaystyle\sum_{\indexi{k}\in\indexI}
            \sgn
                {\order{\indexi{k}, \alpha}}
                {\indexi{k}, \alpha}
            \coefD{\indexi{k}}{\order{\indexi{k},\alpha}}
            \basexi{\indexI}
            \wedge
            \basedualxi{\indexJp}
        .
    \label{EQ-For (alpha, I, J), calculation for LHS, conclude-simplified term 1}   
    \end{align}

    \item We can add the case that $\indexjp{l}=\alpha^{'}$ to the summation inside (\ref{EQ-For (alpha, I, J), calculation for LHS, conclude-ref extra 1,3}) since $\coefD{\indexjp{l}}{\alpha^{'},\alpha^{'}}=0$. Then, 
    (\ref{EQ-For (alpha, I, J), calculation for LHS, conclude-ref extra 1,3}) + (\ref{EQ-For (alpha, I, J), calculation for LHS, conclude-ref extra 2,1}) is equal to
    \begin{align}
        (-2\sqrt{-1})
        \displaystyle\sum_{\indexjp{l}\in\indexJp}
        \begin{pmatrix*}[l]
            \sgn
                {\order{\indexjp{l},\alpha^{'}}}
                {\alpha^{'},\indexjp{l}}
            \coefD{\indexjp{l}}{\order{\indexjp{l},\alpha^{'}}}
            +
            \coefB{\alpha}{\indexj{l} \indexjp{l}}
        \end{pmatrix*}
        \basexi{\indexI}
        \wedge
        \basedualxi{\indexJp}
        .
    \label{EQ-For (alpha, I, J), calculation for LHS, conclude-classified term 2}
    \end{align}
    Recall the constraint
    \begin{align*}
        \coefB{\indexj{l}}{\indexj{l} \alpha}
        -
        \coefB{\alpha}{\indexj{l} \indexjp{l}}
        =
        \sgn
            {\order{\alpha,\indexjp{l}}}
            {\alpha,\indexj{l}}
        \coefD{\indexj{l}}{\order{\alpha,\indexjp{l}}}
        .
    \end{align*}
    Applying it to the terms in (\ref{EQ-For (alpha, I, J), calculation for LHS, conclude-classified term 2}), we obtain the following simplified expression:
    \begin{align}
        (-2\sqrt{-1})
        \displaystyle\sum_{\indexjp{l}\in\indexJp}
        \coefB{\indexj{l}}{\indexj{l} \alpha}
        \basexi{\indexI}
        \wedge
        \basedualxi{\indexJp}
        .
    \label{EQ-For (alpha, I, J), calculation for LHS, conclude-simplified term 2}
    \end{align}

    \item (\ref{EQ-For (alpha, I, J), calculation for LHS, conclude-ref extra 1,2}) + (\ref{EQ-For (alpha, I, J), calculation for LHS, conclude-ref extra 2,4}) is equal to
    \begin{align}
        (-2\sqrt{-1})
        \displaystyle\sum_{\indexi{k}\in\indexI}
        \displaystyle\sum_{s\in\indexIc}
        \begin{matrix*}[l]
            \sgn
                {s,\indexI}
                {\indexi{k},\order{(\indexI\backslash\{\indexi{k}\})\cup\{s\}}}
            \big(
                \coefB{\indexi{k}}{s \alpha^{'}}
                -
                \coefB{\alpha}{s \indexip{k}}
            \big)
            \basexi{\order{(\indexI\backslash\{\indexi{k}\})\cup\{s\}}}
            \wedge
            \basedualxi{\indexJp}
        \end{matrix*}
    .
    \label{EQ-For (alpha, I, J), calculation for LHS, conclude-classified term 3}
    \end{align}
    Recall the constraint
    \begin{align*}
        \coefB{\indexi{k}}{s \alpha^{'}}
        -
        \coefB{\alpha}{s \indexi{k}}
        =
        \sgn
            {\order{\alpha,\indexi{k}}}
            {\alpha,\indexi{k}}
        \coefD{s}{\order{\alpha,\indexi{k}}}
        .
    \end{align*}
    Applying it to the terms in (\ref{EQ-For (alpha, I, J), calculation for LHS, conclude-classified term 3}), we obtain the following simplified expression:
    \begin{align}
        (-2\sqrt{-1})
        \displaystyle\sum_{\indexi{k}\in\indexI}
        \displaystyle\sum_{s\in\indexIc}
        \begin{pmatrix*}[l]
            \sgn
                {s,\indexI}
                {\indexi{k},\order{(\indexI\backslash\{\indexi{k}\})\cup\{s\}}}
            \sgn
                {\order{\alpha,\indexi{k}}}
                {\alpha,\indexi{k}}
            \\
            \coefD{s}{\order{\alpha,\indexi{k}}}
            \basexi{\order{(\indexI\backslash\{\indexi{k}\})\cup\{s\}}}
            \wedge
            \basedualxi{\indexJp}
        \end{pmatrix*}
    .
    \label{EQ-For (alpha, I, J), calculation for LHS, conclude-simplified term 3}
    \end{align}
    To match the right-hand side, which we discuss later, we split the sum in (\ref{EQ-For (alpha, I, J), calculation for LHS, conclude-simplified term 3}) into two parts:
    \begin{align}
        &
        (-2\sqrt{-1})
        \displaystyle\sum_{\indexi{k}\in\indexI}
        \begin{pmatrix*}[l]
            (-1)
            \sgn
                {\indexI}
                {\indexi{k},\indexI\backslash\{\indexi{k}\}}
            \sgn
                {\order{\alpha,\indexi{k}}}
                {\alpha,\indexi{k}}
            \\
            \coefD{\alpha}{\order{\alpha,\indexi{k}}}
            \basexi{\alpha}
            \wedge
            \basexi{\indexI\backslash\{\indexi{k}\}}
            \wedge
            \basedualxi{\indexJp}
        \end{pmatrix*}
    \label{EQ-For (alpha, I, J), calculation for LHS, conclude-simplified term 3.1}
        \\
        +&\notag
        \\
        &
        (-2\sqrt{-1})
        \displaystyle\sum_{\indexi{k}\in\indexI}
        \displaystyle\sum_{s\in\indexIc\backslash\{\alpha\}}
        \begin{pmatrix*}[l]
            \sgn
                {s,\indexI}
                {\indexi{k},\order{(\indexI\backslash\{\indexi{k}\})\cup\{s\}}}
            \sgn
                {\order{\alpha,\indexi{k}}}
                {\alpha,\indexi{k}}
            \\
            \coefD{s}{\order{\alpha,\indexi{k}}}
            \basexi{\order{(\indexI\backslash\{\indexi{k}\})\cup\{s\}}}
            \wedge
            \basedualxi{\indexJp}
        \end{pmatrix*}
    ,
    \label{EQ-For (alpha, I, J), calculation for LHS, conclude-simplified term 3.2}
    \end{align}
    where (\ref{EQ-For (alpha, I, J), calculation for LHS, conclude-simplified term 3.1}) is obtained from
    \begin{align*}
        (-2\sqrt{-1})
        \displaystyle\sum_{\indexi{k}\in\indexI}
        \displaystyle\sum_{s=\alpha}
        \begin{pmatrix*}[l]
            \sgn
                {s,\indexI}
                {\indexi{k},\order{(\indexI\backslash\{\indexi{k}\})\cup\{s\}}}
            \sgn
                {\order{\alpha,\indexi{k}}}
                {\alpha,\indexi{k}}
            \\
            \coefD{s}{\order{\alpha,\indexi{k}}}
            \basexi{\order{(\indexI\backslash\{\indexi{k}\})\cup\{s\}}}
            \wedge
            \basedualxi{\indexJp}
        \end{pmatrix*}
    ,
    \end{align*}
    where $
    \order{(\indexI\backslash\{\indexi{k}\})\cup\{\alpha\}}
    =
    \{ \alpha,\indexI\backslash\{\indexi{k}\} \}
    $ since we assumed $\alpha<\indexi{1}$.

    \item (\ref{EQ-For (alpha, I, J), calculation for LHS, conclude-ref extra 1,4}) + (\ref{EQ-For (alpha, I, J), calculation for LHS, conclude-ref extra 2,2}) is equal to
    \begin{align}
        (-2\sqrt{-1})
        \displaystyle\sum_{\indexj{l}\in\indexJp}
        \displaystyle\sum_{t^{'}\in\indexJpc}
        \begin{pmatrix*}[l]
            \sgn
                {t^{'},\indexI,\indexJp}
                {\indexjp{l},\indexI,\order{(\indexJp\backslash\{\indexjp{l}\})\cup\{t^{'}\}}}
            \\
            \bigg(
                \sgn
                    {\order{\alpha^{'},t^{'}}}
                    {t^{'},\alpha^{'}}
                \coefD{\indexjp{l}}{\order{\alpha^{'},t^{'}}}
                -
                \coefB{\alpha}{\indexj{l} t^{'}}
            \bigg)
            \\
            \basexi{\indexI}
            \wedge
            \basedualxi{\order{(\indexJp\backslash\{\indexjp{l}\})\cup\{t^{'}\}}}
        \end{pmatrix*}
        .
    \label{EQ-For (alpha, I, J), calculation for LHS, conclude-classified term 4}
    \end{align}
    Recall the constraint
    \begin{align*}
        \coefB{\alpha}{\indexjp{l} t^{'}}
        -
        \coefB{t^{'}}{ \indexjp{l} \alpha }
        =
        \sgn
            {\order{t^{'},\alpha^{'}}}
            {t^{'},\alpha^{'}}
        \coefD{\indexjp{l}}{\order{t^{'},\alpha^{'}}}
        .
    \end{align*}
    Applying it to the terms in (\ref{EQ-For (alpha, I, J), calculation for LHS, conclude-classified term 4}), we obtain the following simplified expression:
    \begin{align}
        (-2\sqrt{-1})
        \displaystyle\sum_{\indexj{l}\in\indexJp}
        \displaystyle\sum_{t^{'}\in\indexJpc}
        \begin{pmatrix*}[l]
            \sgn
                {t^{'},\indexI,\indexJp}
                {\indexjp{l},\indexI,\order{(\indexJp\backslash\{\indexjp{l}\})\cup\{t^{'}\}}}
            \\
            (-1)
            \coefB{t^{'}}{\indexjp{l} \alpha}
            \basexi{\indexI}
            \wedge
            \basedualxi{\order{(\indexJp\backslash\{\indexjp{l}\})\cup\{t^{'}\}}}
        \end{pmatrix*}
        .
    \label{EQ-For (alpha, I, J), calculation for LHS, conclude-simplified term 4}
    \end{align}
\end{itemize}

We can now summarize the expression of $
[\LieagbdPartialbar,\iotaW]
(\basexi{\alpha}\wedge\basexi{\indexI}\wedge\basedualxi{\indexJp})
$ by summing up 
\begin{align*}
    \begin{pmatrix*}[l]
        [\LieagbdPartialbar,\iotaW]
        (\basexi{\alpha})
        \wedge
        \basexi{\indexI}\wedge\basedualxi{\indexJp}
        -
        \basexi{\alpha}
        \wedge
        [\LieagbdPartialbar,\iotaW]
        (\basexi{\indexI}\wedge\basedualxi{\indexJp})
    \end{pmatrix*}
\end{align*}
 with (\ref{EQ-For (alpha, I, J), calculation for LHS, conclude-simplified term 1}), (\ref{EQ-For (alpha, I, J), calculation for LHS, conclude-simplified term 2}), (\ref{EQ-For (alpha, I, J), calculation for LHS, conclude-simplified term 3.1}), (\ref{EQ-For (alpha, I, J), calculation for LHS, conclude-simplified term 3.2}), and (\ref{EQ-For (alpha, I, J), calculation for LHS, conclude-simplified term 4}):

\begin{align*}
    \begin{pmatrix*}[l]
            [\partialbar,\iotaW]
            \wedge
            \basexi{\indexI}\wedge\basedualxi{\indexJp}
            -
            \basexi{\alpha}
            \wedge
            [\partialbar,\iotaW]
            (\basexi{\indexI}\wedge\basedualxi{\indexJp})
            \\
            +
            \\
            (-2\sqrt{-1})
            \begin{pmatrix*}[l]
                \displaystyle\sum_{\indexi{k}\in\indexI}
                \begin{matrix*}[l]
                    \sgn
                        {\order{\indexi{k},\alpha}}
                        {\indexi{k},\alpha}
                    \coefD{\indexi{k}}{\order{\indexi{k},\alpha}}
                    \basexi{\indexI}
                    \wedge
                    \basedualxi{\indexJp}
                \end{matrix*}
                \\
                +
                \\
                \displaystyle\sum_{\indexj{l}\in\indexJp}
                \begin{matrix*}[l]
                    \coefB{\indexjp{l}}{\indexjp{l} \alpha}
                    \basexi{\indexI}
                    \wedge
                    \basedualxi{\indexJp}
                \end{matrix*}
                \\
                +
                \\
                \displaystyle\sum_{\indexi{k}\in\indexI}
                \begin{matrix*}[l]
                    \sgn
                        {\indexI}
                        {\indexi{k},\indexI\backslash\{\indexi{k}\}}
                    \sgn
                        {\order{\indexi{k},\alpha}}
                        {\indexi{k},\alpha}
                    \\
                    \coefD{\alpha}{\order{\indexi{k},\alpha}}
                    \basexi{\alpha}
                    \wedge
                    \basexi{(\indexI\backslash\{\indexi{k}\})}
                    \wedge
                    \basedualxi{\indexJp}
                \end{matrix*}
                \\
                +
                \\
                \displaystyle\sum_{\indexi{k}\in\indexI}
                \displaystyle\sum_{s\in\indexIc\backslash\{\alpha\}}
                \begin{matrix*}[l]
                    \sgn
                        {s,\indexI}
                        {\indexi{k},\order{(\indexI\backslash\{\indexi{k}\})\cup\{s\}}}
                    \sgn
                        {\order{\alpha,\indexi{k}}}
                        {\alpha,\indexi{k}}  
                    \\
                    \coefD{s}{\order{\indexi{k},\alpha}}
                    \basexi{\order{(\indexI\backslash\{\indexi{k}\})\cup\{s\}}}
                    \wedge
                    \basedualxi{\indexJp}
                \end{matrix*}
                \\
                +
                \\
                \displaystyle\sum_{\indexj{l}\in\indexJp}
                \displaystyle\sum_{t^{'}\in\indexJpc}
                \begin{matrix*}[l]
                    \sgn
                        {t^{'},\indexI,\indexJp}
                        {\indexjp{l},\indexI,\order{(\indexJp\backslash\{\indexjp{l}\})\cup\{t^{'}\}}}
                    \\
                    (-1)
                    \coefB{t^{'}}{\indexjp{l} \alpha}
                    \basexi{\indexI}
                    \wedge
                    \basedualxi{\order{(\indexJp\backslash\{\indexjp{l}\})\cup\{t^{'}\}}}
                \end{matrix*}
            \end{pmatrix*}
        \end{pmatrix*},
\end{align*}
which is the right-hand side in the statement of Lemma \ref{lemma-the final version of the LHS}.
\end{proof}

\subsection{
        Calculation of 
        \texorpdfstring{       $(\sqrt{-1}\LieagbdPartial^{*}\basexi{\alpha})\wedge\basexi{\indexI}\wedge\basedualxi{\indexJp}$}{}
    }
\label{subsection-For (alpha), I, J, calculation for RHS}
\begin{lemma}
\label{lemma-Kahler identity on Lie agbd-For (alpha),I,J, calculation for RHS}
    \begin{align}                                                       
        (\sqrt{-1}\LieagbdPartial^{*}\basexi{\alpha})
        \wedge\basexi{\indexI}\wedge\basedualxi{\indexJp}
        =
        -2\sqrt{-1}
        \begin{pmatrix*}[l]
            \displaystyle\sum_{k\neq\alpha}
            \begin{matrix*}[l]
                \sgn
                    {\order{k,\alpha}}
                    {\alpha,k}
                \coefD{k}{\order{k,\alpha}}
                \basexi{\indexI}\wedge\basedualxi{\indexJp}
            \end{matrix*}
            \\
            +
            \\
            \displaystyle\sum_{l^{'}=1}^{n}
            \begin{matrix*}[l]
                (-1)\coefB{l^{'}}{l^{'} \alpha}
                \basexi{\indexI}\wedge\basedualxi{\indexJp}
            \end{matrix*}
        \end{pmatrix*}.
    \label{EQ-Kahler identity on Lie agbd-For (alpha),I,J, calculation for RHS-target}
    \end{align}
\end{lemma}

\begin{proof}
Recall that $\LieagbdPartial^{*}=-(-1)^{n^2}\star\,\LieagbdPartialbar\,\star$. Our proof comprises the following steps:
\begin{enumerate}
    \item Calculate $\star(\basexi{\alpha})$.
    \item Calculate $
    \LieagbdPartialbar\,\star
    (\basexi{\alpha})
    $.
    \item Calculate $
    \star\,\LieagbdPartialbar\,\star
    (\basexi{\alpha})
    $ and conclude the expression of $
    (\sqrt{-1}\LieagbdPartial^{*}\basexi{\alpha})
    \wedge\basexi{\indexI}\wedge\basedualxi{\indexJp}
    $.
\end{enumerate}
\keyword{Step 1:} We take $\star$ on $\basexi{\alpha}$ and obtain
\begin{align}
    \star(\basexi{\alpha})
    =
    2^{1+0-n}
    \sgn
        {1,\cdots,n}
        {\alpha,1,\cdots,\widehat{\alpha},\cdots,n}
    \basedualxi{1,\cdots,\widehat{\alpha},\cdots,n}
    \wedge
    \basexi{1^{'},\cdots,n^{'}}
    .
    \label{EQ-Kahler identity on Lie agbd-For (alpha),I,J, calculation for RHS-No1}
\end{align}
\keyword{Step 2:} We apply $\LieagbdPartialbar$ to (\ref{EQ-Kahler identity on Lie agbd-For (alpha),I,J, calculation for RHS-No1}) and obtain
\begin{align}
    &
    2^{1-n}
    \begin{pmatrix*}[l]
        \displaystyle\sum_{k\neq\alpha}
        \begin{matrix*}[l]
            \sgn
                {1,\cdots,n}
                {\alpha,1,\cdots,\widehat{\alpha},\cdots,n}
            \sgn
                {1,\cdots,\widehat{\alpha},\cdots,n}
                {k,1,\cdots,\widehat{k},\cdots,\widehat{\alpha},\cdots,n}
            \\
            \basedualxi{1}
            \wedge\cdots\wedge
            \LieagbdPartialbar\basedualxi{k}
            \wedge\cdots\wedge
            \widehat{\basedualxi{\alpha}}
            \wedge\cdots\wedge
            \basedualxi{n}
            \wedge
            \basexi{1^{'},\cdots,n^{'}}
        \end{matrix*}
        \\
        +
        \\
        \displaystyle\sum_{l^{'}=1}^{n}
        \begin{matrix*}[l]
            \sgn
                {1,\cdots,n}
                {\alpha,1,\cdots,\widehat{\alpha},\cdots,n}
            \sgn
                {1,\cdots,\widehat{\alpha},\cdots,n,1^{'},\cdots,n^{'}}
                {l^{'},1,\cdots,\widehat{\alpha},\cdots,n,1^{'},\cdots,\widehat{l^{'}},\cdots,n^{'}}
            \\
            \basedualxi{1}
            \wedge\cdots\wedge
            \widehat{\basedualxi{\alpha}}
            \wedge\cdots\wedge
            \basedualxi{n}
            \wedge
            \basexi{1^{'}}
            \wedge\cdots\wedge
            \LieagbdPartialbar\basexi{l^{'}}
            \wedge\cdots\wedge
            \basexi{n^{'}}
        \end{matrix*}
    \end{pmatrix*}
    .
\label{EQ-Kahler identity on Lie agbd-For (alpha),I,J, calculation for RHS-No1.5}
\end{align}
Recalling that $
\LieagbdPartialbar\basedualxi{k}
=
\frac{1}{2}\displaystyle\sum_{s,t}
\coefD{k}{s t}\basedualxi{s}\wedge\basedualxi{t}
$ and $
\LieagbdPartialbar\basexi{l^{'}}
=
\displaystyle\sum_{s,t}
\coefB{l^{'}}{s t}\basexi{s}\wedge\basedualxi{t}
$, we cancel out the terms containing repeated frame bases. The formula of (\ref{EQ-Kahler identity on Lie agbd-For (alpha),I,J, calculation for RHS-No1.5}) is simplified to
\begin{align}
    2^{1-n}
    \begin{pmatrix*}[l]
        \displaystyle\sum_{k\neq\alpha}
        \begin{matrix*}[l]
            \sgn
                {1,\cdots,n}
                {\alpha,1,\cdots,\widehat{\alpha},\cdots,n}
            \sgn
                {1,\cdots,\widehat{\alpha},\cdots,n}
                {k,1,\cdots,\widehat{k},\cdots,\widehat{\alpha},\cdots,n}
            \\
            \basedualxi{1}
            \wedge\cdots\wedge
            \begin{pmatrix*}[l]
                \frac{1}{2}
                (
                \coefD{k}{k \alpha}
                \basedualxi{k}\wedge\basedualxi{\alpha}
                +
                \coefD{k}{\alpha k}
                \basedualxi{\alpha}\wedge\basedualxi{k}
                )
            \end{pmatrix*}
            \wedge\cdots\wedge
            \widehat{\basedualxi{\alpha}}
            \wedge\cdots\wedge
            \basedualxi{n}
            \wedge
            \basexi{1^{'},\cdots,n^{'}}
        \end{matrix*}
        \\
        +
        \\
        \displaystyle\sum_{l^{'}=1}^{n}
        \begin{matrix*}[l]
            \sgn
                {1,\cdots,n}
                {\alpha,1,\cdots,\widehat{\alpha},\cdots,n}
            \sgn
                {1,\cdots,\widehat{\alpha},\cdots,n,1^{'},\cdots,n^{'}}
                {l^{'},1,\cdots,\widehat{\alpha},\cdots,n,1^{'},\cdots,\widehat{l^{'}},\cdots,n^{'}}
            \\
            \basedualxi{1}
            \wedge\cdots\wedge
            \widehat{\basedualxi{\alpha}}
            \wedge\cdots\wedge
            \basedualxi{n}
            \wedge
            \basexi{1^{'}}
            \wedge\cdots\wedge
            \begin{pmatrix*}[l]
                \coefB{l^{'}}{l^{'} \alpha}
                \basexi{l^{'}}\wedge\basedualxi{\alpha}
            \end{pmatrix*}
            \wedge\cdots\wedge
            \basexi{n^{'}}
        \end{matrix*}
    \end{pmatrix*}.
\label{EQ-Kahler identity on Lie agbd-For (alpha),I,J, calculation for RHS-No2}
\end{align}

We use the notation
\begin{align*}
    \basedualxi{\order{j k}}
    =
    \begin{cases}
        \basedualxi{j}\wedge\basedualxi{k}
        \qquad\text{if }j<k
        \\
        \basedualxi{k}\wedge\basedualxi{j}
        \qquad\text{if }k<j
    \end{cases}
\end{align*}
to simplify the formula
\begin{align*}
        \coefD{i}{j k}
        \basedualxi{j}\wedge\basedualxi{k}
    =
        \coefD{i}{k j}
        \basedualxi{k}\wedge\basedualxi{j}
    =  
        \coefD{i}{\order{j k}}
        \basedualxi{\order{j k}},
\end{align*}
where $\coefD{i}{j k}=-\coefD{i}{k j}$. Thus, the wedge product of frames in (\ref{EQ-Kahler identity on Lie agbd-For (alpha),I,J, calculation for RHS-No2}) can be simplified to
\begin{align}
    &
        \basedualxi{1}
        \wedge\cdots\wedge
        \begin{pmatrix*}[l]
            \frac{1}{2}
            (
            \coefD{k}{k \alpha}
            \basedualxi{k}\wedge\basedualxi{\alpha}
            +
            \coefD{k}{\alpha k}
            \basedualxi{\alpha}\wedge\basedualxi{k}
            )
        \end{pmatrix*}
        \wedge\cdots\wedge
        \widehat{\basedualxi{\alpha}}
        \wedge\cdots\wedge
        \basedualxi{n}
        \wedge
        \basexi{1^{'},\cdots,n^{'}}
    \notag
    \\
    =&
        \basedualxi{1}
        \wedge\cdots\wedge  
            \coefD{k}{\order{k \alpha}}
            \basedualxi{\order{k \alpha}}
        \wedge\cdots\wedge
        \widehat{\basedualxi{\alpha}}
        \wedge\cdots\wedge
        \basedualxi{n}
        \wedge
        \basexi{1^{'},\cdots,n^{'}}
    \notag
    \\
    =&
        \sgn
            {1,\cdots,k-1,\order{k,\alpha},k+1,\cdots,\widehat{\alpha},\cdots,n}
            {1,\cdots,n}
        \coefD{k}{\order{k,\alpha}}
        \basedualxi{1,\cdots,n}
        \wedge
        \basexi{1^{'},\cdots,n^{'}}
\label{EQ-Kahler identity on Lie agbd-For (alpha),I,J, calculation for RHS-No3}
\end{align}
and
\begin{align}
    &
        \basedualxi{1}
        \wedge\cdots\wedge
        \widehat{\basedualxi{\alpha}}
        \wedge\cdots\wedge
        \basedualxi{n}
        \wedge
        \basexi{1^{'}}
        \wedge\cdots\wedge
        \begin{pmatrix*}[l]
            \coefB{l^{'}}{l^{'} \alpha}
            \basexi{l^{'}}\wedge\basedualxi{\alpha}
        \end{pmatrix*}
        \wedge\cdots\wedge
        \basexi{n^{'}}
    \notag
    \\
    =&
        \sgn
            {1,\cdots,\widehat{\alpha},\cdots,n,1^{'},\cdots,l^{'},\alpha,l+1^{'},\cdots,n^{'}}
            {1,\cdots,n,1^{'},\cdots,n^{'}}
        \coefB{l^{'}}{l^{'} \alpha}
        \basedualxi{1,\cdots,n}
        \wedge
        \basexi{1^{'},\cdots,n^{'}}
    .
    \label{EQ-Kahler identity on Lie agbd-For (alpha),I,J, calculation for RHS-No4}
\end{align}
Using (\ref{EQ-Kahler identity on Lie agbd-For (alpha),I,J, calculation for RHS-No3}) and
(\ref{EQ-Kahler identity on Lie agbd-For (alpha),I,J, calculation for RHS-No4}) to replace the corresponding terms in (\ref{EQ-Kahler identity on Lie agbd-For (alpha),I,J, calculation for RHS-No2}), we obtain the expression of $\LieagbdPartialbar\,\star(\basexi{\alpha})$:
\begin{align}
    2^{1-n}
    \begin{pmatrix*}[l]
        \displaystyle\sum_{k\neq\alpha}
        \begin{matrix*}[l]
            \sgn
                {1,\cdots,n}
                {\alpha,1,\cdots,\widehat{\alpha},\cdots,n}
            \sgn
                {1,\cdots,\widehat{\alpha},\cdots,n}
                {k,1,\cdots,\widehat{k},\cdots,\widehat{\alpha},\cdots,n}
            \sgn
                {1,\cdots,k-1,\order{k,\alpha},k+1,\cdots,\widehat{\alpha},\cdots,n}
                {1,\cdots,n}
            \\
            \coefD{k}{\order{k,\alpha}}
            \basedualxi{1,\cdots,n}
            \wedge
            \basexi{1^{'},\cdots,n^{'}}
        \end{matrix*}
        \\
        +
        \\
        \displaystyle\sum_{l^{'}=1}^{n}
        \begin{matrix*}[l]
            \sgn
                {1,\cdots,n}
                {\alpha,1,\cdots,\widehat{\alpha},\cdots,n}
            \sgn
                {1,\cdots,\widehat{\alpha},\cdots,n,1^{'},\cdots,n^{'}}
                {l^{'},1,\cdots,\widehat{\alpha},\cdots,n,1^{'},\cdots,\widehat{l^{'}},\cdots,n^{'}}
            \sgn
                {1,\cdots,\widehat{\alpha},\cdots,n,1^{'},\cdots,l^{'},\alpha,l+1^{'},\cdots,n^{'}}
                {1,\cdots,n,1^{'},\cdots,n^{'}}
            \\
            \coefB{l^{'}}{l^{'} \alpha}
            \basedualxi{1,\cdots,n}
            \wedge
            \basexi{1^{'},\cdots,n^{'}}
        \end{matrix*}
    \end{pmatrix*}.
\label{EQ-Kahler identity on Lie agbd-For (alpha),I,J, calculation for RHS-No5}
\end{align}
\keyword{Step 3:} Take $\star$ on (\ref{EQ-Kahler identity on Lie agbd-For (alpha),I,J, calculation for RHS-No5}). Using the formula
\begin{align*}
    \star
    (
    \basedualxi{1,\cdots,n}
    \wedge
    \basexi{1^{'},\cdots,n^{'}}
    )
    =
    2^{n+n-n}
    \sgn
        {1^{'},\cdots,n^{'},1,\cdots,n}
        {1,\cdots,n,1^{'},\cdots,n^{'}}
    =
    2^{n}(-1)^{n^2}
    ,
\end{align*}
we compute $\star\,\LieagbdPartialbar\,\star(\basexi{\alpha})$ as
\begin{align}
    &
        2^{1-n}
        \begin{pmatrix*}[l]
            2^{n}(-1)^{n^2}
            \displaystyle\sum_{k\neq\alpha}
            \begin{pmatrix*}[l]
                \sgn
                    {1,\cdots,n}
                    {\alpha,1,\cdots,\widehat{\alpha},\cdots,n}
                \sgn
                    {1,\cdots,\widehat{\alpha},\cdots,n}
                    {k,1,\cdots,\widehat{k},\cdots,\widehat{\alpha},\cdots,n}
                \\
                \sgn
                    {1,\cdots,k-1,\order{k,\alpha},k+1,\cdots,\widehat{\alpha},\cdots,n}
                    {1,\cdots,n}
                \coefD{k}{\order{k,\alpha}}
            \end{pmatrix*}
            \\
            +
            \\
            2^{n}(-1)^{n^2}
            \displaystyle\sum_{l^{'}=1}^{n}
            \begin{pmatrix*}[l]
                \sgn
                    {1,\cdots,n}
                    {\alpha,1,\cdots,\widehat{\alpha},\cdots,n}
                \sgn
                    {1,\cdots,\widehat{\alpha},\cdots,n,1^{'},\cdots,n^{'}}
                    {l^{'},1,\cdots,\widehat{\alpha},\cdots,n,1^{'},\cdots,\widehat{l^{'}},\cdots,n^{'}}
                \\
                \sgn
                    {1,\cdots,\widehat{\alpha},\cdots,n,1^{'},\cdots,l^{'},\alpha,l+1^{'},\cdots,n^{'}}
                    {1,\cdots,n,1^{'},\cdots,n^{'}}
                \coefB{l^{'}}{l^{'} \alpha}
            \end{pmatrix*}
        \end{pmatrix*}
    \notag
    \\
    =&
        (-1)^{n^2}2
        \begin{pmatrix*}[l]
            \displaystyle\sum_{k\neq\alpha}
            \begin{pmatrix*}[l]
                \sgn
                    {1,\cdots,n}
                    {\alpha,1,\cdots,\widehat{\alpha},\cdots,n}
                \sgn
                    {1,\cdots,\widehat{\alpha},\cdots,n}
                    {k,1,\cdots,\widehat{k},\cdots,\widehat{\alpha},\cdots,n}
                \\
                \sgn
                    {1,\cdots,k-1,\order{k,\alpha},k+1,\cdots,\widehat{\alpha},\cdots,n}
                    {1,\cdots,n}
                \coefD{k}{\order{k,\alpha}}
            \end{pmatrix*}
            \\
            +
            \\
            \displaystyle\sum_{l^{'}=1}^{n}
            \begin{pmatrix*}[l]
                \sgn
                    {1,\cdots,n}
                    {\alpha,1,\cdots,\widehat{\alpha},\cdots,n}
                \sgn
                    {1,\cdots,\widehat{\alpha},\cdots,n,1^{'},\cdots,n^{'}}
                    {l^{'},1,\cdots,\widehat{\alpha},\cdots,n,1^{'},\cdots,\widehat{l^{'}},\cdots,n^{'}}]
                \\
                \sgn
                    {1,\cdots,\widehat{\alpha},\cdots,n,1^{'},\cdots,l^{'},\alpha,l+1^{'},\cdots,n^{'}}
                    {1,\cdots,n,1^{'},\cdots,n^{'}}
                \coefB{l^{'}}{l^{'} \alpha}
            \end{pmatrix*}
        \end{pmatrix*}
    .
\label{EQ-Kahler identity on Lie agbd-For (alpha),I,J, calculation for RHS-No6}
\end{align}
The signs of permutation in (\ref{EQ-Kahler identity on Lie agbd-For (alpha),I,J, calculation for RHS-No6}) can be simplified as follows:
\begin{align}
    &
        \sgn
            {1,\cdots,n}
            {\alpha,1,\cdots,\widehat{\alpha},\cdots,n}
        \sgn
            {1,\cdots,\widehat{\alpha},\cdots,n}
            {k,1,\cdots,\widehat{k},\cdots,\widehat{\alpha},\cdots,n}
        \sgn
            {1,\cdots,k-1,\order{k,\alpha},k+1,\cdots,\widehat{\alpha},\cdots,n}
            {1,\cdots,n}
    \notag
    \\
    =&
        \sgn
            {1,\cdots,n}
            {\alpha,k,1,\cdots,\widehat{k},\cdots,\widehat{\alpha},\cdots,z}
        \sgn
            {1,\cdots,k-1,\order{k,\alpha},k+1,\cdots,\widehat{\alpha},\cdots,n}
            {1,\cdots,n}
    \notag
    \\
    =&
        \sgn
            {1,\cdots,n}
            {\alpha,k,1,\cdots,\widehat{k},\cdots,\widehat{\alpha},\cdots,z}
        \sgn
            {\order{k,\alpha},1,\cdots,\widehat{k},\cdots,\widehat{\alpha},\cdots,n}
            {1,\cdots,n}
    \notag
    \\
    =&
        \sgn
            {1,\cdots,n}
            {\alpha,k,1,\cdots,\widehat{k},\cdots,\widehat{\alpha},\cdots,z}
        \sgn
            {\order{k,\alpha}}
            {\alpha,k}
        \sgn
            {\alpha,k,1,\cdots,\widehat{k},\cdots,\widehat{\alpha},\cdots,n}
            {1,\cdots,n}
    \notag
    \\
    =&
        \sgn
            {\order{k,\alpha}}
            {\alpha,k}
\label{EQ-Kahler identity on Lie agbd-For (alpha),I,J, calculation for RHS-No7}
\end{align}
and
\begin{align}
    &
        \sgn
            {1,\cdots,n}
            {\alpha,1,\cdots,\widehat{\alpha},\cdots,n}
        \sgn
            {1,\cdots,\widehat{\alpha},\cdots,n,1^{'},\cdots,n^{'}}
            {l^{'},1,\cdots,\widehat{\alpha},\cdots,n,1^{'},\cdots,\widehat{l^{'}},\cdots,n^{'}}
        \sgn
            {1,\cdots,\widehat{\alpha},\cdots,n,1^{'},\cdots,l^{'},\alpha,l+1^{'},\cdots,n^{'}}
            {1,\cdots,n,1^{'},\cdots,n^{'}}
    \notag
    \\
    =&
        \sgn
            {1,\cdots,n}
            {\alpha,1,\cdots,\widehat{\alpha},\cdots,n}
        \sgn
            {1,\cdots,\widehat{\alpha},\cdots,n,1^{'},\cdots,n^{'}}
            {l^{'},1,\cdots,\widehat{\alpha},\cdots,n,1^{'},\cdots,\widehat{l^{'}},\cdots,n^{'}}
        \sgn
            {l^{'},\alpha,1,\cdots,\widehat{\alpha},\cdots,n,1^{'},\cdots,\widehat{l},\cdots,n^{'}}
            {1,\cdots,n,1^{'},\cdots,n^{'}}
    \notag
    \\
    =&
        \sgn
            {1,\cdots,n}
            {\alpha,1,\cdots,\widehat{\alpha},\cdots,n}
        (-1)
        \sgn
            {\alpha,1,\cdots,\widehat{\alpha},\cdots,n,1^{'},\cdots,n^{'}}
            {l^{'},\alpha,1,\cdots,\widehat{\alpha},\cdots,n,1^{'},\cdots,\widehat{l^{'}},\cdots,n^{'}}
        \sgn
            {l^{'},\alpha,1,\cdots,\widehat{\alpha},\cdots,n,1^{'},\cdots,\widehat{l},\cdots,n^{'}}
            {1,\cdots,n,1^{'},\cdots,n^{'}}
    \notag
    \\
    =&
        (-1)
        \sgn
            {1,\cdots,n,1^{'},\cdots,n^{'}}
            {l^{'},\alpha,1,\cdots,\widehat{\alpha},\cdots,n,1^{'},\cdots,\widehat{l^{'}},\cdots,n^{'}}
        \sgn
            {l^{'},\alpha,1,\cdots,\widehat{\alpha},\cdots,n,1^{'},\cdots,\widehat{l},\cdots,n^{'}}
            {1,\cdots,n,1^{'},\cdots,n^{'}}
    \notag
    \\
    =&
        (-1)
        \sgn
            {1,\cdots,n,1^{'},\cdots,n^{'}}
            {1,\cdots,n,1^{'},\cdots,n^{'}}
    \notag
    \\
    =&
        (-1)
    .
\label{EQ-Kahler identity on Lie agbd-For (alpha),I,J, calculation for RHS-No8}
\end{align}
We plug (\ref{EQ-Kahler identity on Lie agbd-For (alpha),I,J, calculation for RHS-No7}) and (\ref{EQ-Kahler identity on Lie agbd-For (alpha),I,J, calculation for RHS-No8}) into (\ref{EQ-Kahler identity on Lie agbd-For (alpha),I,J, calculation for RHS-No6}) and simplify the expression of $\star\,\LieagbdPartial\,\star(\basexi{\alpha})$ to
\begin{align*}
    2(-1)^{n^2}
    \begin{pmatrix*}[l]
        \displaystyle\sum_{k\neq\alpha}
        \begin{matrix*}[l]
            \sgn
                {\order{k,\alpha}}
                {\alpha,k}
            \coefD{k}{\order{k,\alpha}}
        \end{matrix*}
        \\
        +
        \\
        \displaystyle\sum_{l^{'}=1}^{n}
        \begin{matrix*}[l]
            (-1)\coefB{l^{'}}{l^{'} \alpha}
        \end{matrix*}
    \end{pmatrix*}.
\end{align*}
Thus, we arrive at the final expression:
\begin{align*}
    (\sqrt{-1}\LieagbdPartial^{*}\basexi{\alpha})
    \wedge
    \basexi{\indexI}\wedge\basedualxi{\indexJp}
    =&
        \sqrt{-1}
        \big(
        -(-1)^{n^2}\star\,\LieagbdPartialbar\star
        (\basexi{\alpha})
        \big)
        \wedge
        \basexi{\indexI}\wedge\basedualxi{\indexJp}
    \\
    =&
    -2\sqrt{-1}
    \begin{pmatrix*}[l]
        \displaystyle\sum_{k\neq\alpha}
        \begin{matrix*}[l]
            \sgn
                {\order{k,\alpha}}
                {\alpha,k}
            \coefD{k}{\order{k,\alpha}}
            \basexi{\indexI}\wedge\basedualxi{\indexJp}
        \end{matrix*}
        \\
        +
        \\
        \displaystyle\sum_{l^{'}=1}^{n}
        \begin{matrix*}[l]
            (-1)\coefB{l^{'}}{l^{'} \alpha}
            \basexi{\indexI}\wedge\basedualxi{\indexJp}
        \end{matrix*}
    \end{pmatrix*}.
\end{align*}

\end{proof}

\newpage
\subsection{
        Calculation of \texorpdfstring{
        $\basexi{\alpha}\wedge\sqrt{-1}\LieagbdPartial^{*}(\basexi{\indexI}\wedge\basedualxi{\indexJp})$}{}
    }
\label{subsection-For alpha, (I, J), calculation for RHS}
\begin{lemma}
\label{lemma-Kahler identity on Lie agbd-For alpha,(I,J), calculation for RHS}
    \begin{align}
        &
        \basexi{\alpha}\wedge
        \sqrt{-1}\LieagbdPartial^{*}
        (\basexi{\indexI}\wedge\basedualxi{\indexJp})
    \\
    =&
        -2\sqrt{-1}
        \begin{pmatrix*}[l]
            \displaystyle\sum_{\indexic{k}\in\indexIc}
            \displaystyle\sum_{r\in\indexI}
            \begin{matrix*}[l]
                \sgn
                    {r,\indexic{k}}
                    {\order{\indexic{k},r}}
                \sgn
                    {r,\indexI\backslash\{r\}}                    
                    {\indexI}
                \coefD{\indexic{k}}{\order{\indexic{k},r}}
                \basexi{\alpha}
                \wedge
                \basexi{\indexI\backslash\{r\}}
                \wedge
                \basedualxi{\indexJp} 
            \end{matrix*}
            \\
            +
            \\
            \displaystyle\sum_{\indexic{k}\in\indexIc\backslash\{\alpha\}}
            \displaystyle\sum_{r_1\in\indexI}
            \displaystyle\sum_{r_2\in\indexI}
            \begin{pmatrix*}[l]
                (-1)
                \sgn
                    {\order{r_1,r_2},\order{(\indexI\backslash\{r_1,r_2\})}}
                    {\indexI}
                \\
                \sgn
                    {\order{(\indexI\backslash\{r_1,r_2\})\cup\{\indexic{k}\}}}
                    {\indexic{k},\indexI\backslash\{r_1,r_2\})}
                \\
                \frac{1}{2}
                \coefD{\indexic{k}}{\order{r_1,r_2}}
                \basexi{\alpha}
                \wedge
                \basexi{\order{(\indexI\backslash\{r_1,r_2\})\cup\{\indexic{k}\}}}
                \wedge
                \basedualxi{\indexJp} 
            \end{pmatrix*}
            \\
            +
            \\
            \displaystyle\sum_{\indexjpc{l}\in\indexJpc}
            \displaystyle\sum_{r\in\indexI}
            \begin{matrix*}[l]
                (-1)
                \sgn
                    {r,\indexI\backslash\{r\}}
                    {\indexI}
                \coefB{\indexjpc{l}}{\indexjpc{l} r}
                \basexi{\alpha}
                \wedge
                \basexi{\indexI\backslash\{r\}}
                \wedge
                \basedualxi{\indexJp} 
            \end{matrix*}
            \\
            +
            \\
            \displaystyle\sum_{\indexjpc{l}\in\indexJpc}
            \displaystyle\sum_{r\in\indexI}
            \displaystyle\sum_{u^{'}\in\indexJp}
            \begin{pmatrix*}[l]
                (-1)^{p}
                \sgn
                    {u^{'},r,\indexI\backslash\{r\},\order{(\indexJp\backslash\{u^{'}\})}}
                    {\indexI,\indexJp}
                \\
                \sgn
                    {\order{(\indexJp\backslash\{u^{'}\})\cup\{\indexjpc{l}\}}}
                    {\indexjpc{l},(\indexJp\backslash\{u^{'}\})}
                \\
                \coefB{\indexjpc{l}}{u^{'} r}
                \basexi{\alpha}
                \wedge
                \basexi{\indexI\backslash\{r\}}
                \wedge
                \basedualxi{\order{(\indexJp\backslash\{u^{'}\})\cup\{\indexjpc{l}\}}} 
            \end{pmatrix*}
        \end{pmatrix*}.
    \label{EQ-Kahler identity on Lie agbd-For alpha,(I,J), calculation for RHS-target}
    \end{align}
\end{lemma}

\begin{proof}
Recall that $\LieagbdPartial^{\star}=-(-1)^{n^2}\star\,\LieagbdPartialbar\,\star$. Our proof comprises the following steps:
\begin{enumerate}
    \item Calculate $\star(\basexi{\indexI}\wedge\basedualxi{\indexJp})$.
    \item Calculate $
    \LieagbdPartialbar\,\star
    (\basexi{\indexI}\wedge\basedualxi{\indexJp})
    $.
    \item Calculate $
    \star\,\LieagbdPartialbar\,\star
    (\basexi{\indexI}\wedge\basedualxi{\indexJp})
    $ and conclude the expression of $
    \basexi{\alpha}\wedge
    \sqrt{-1}\LieagbdPartial^{*}
    (\basexi{\indexI}\wedge\basedualxi{\indexJp})
    $.
\end{enumerate}
\keyword{Step 1:} We take $\star$ on $\basexi{\indexI}\wedge\basedualxi{\indexJp}$ and obtain
\begin{align}
    \star(\basexi{\indexI}\wedge\basedualxi{\indexJp})
    =
    2^{p+q-n}
    \sgn
        {1,\cdots,n,1^{'},\cdots,n^{'}}
        {\indexI,\indexJp,\indexIc,\indexJpc}
    \basedualxi{\indexIc}
    \wedge
    \basexi{\indexJpc}
    .
    \label{EQ-Kahler identity on Lie agbd-For alpha,(I,J), calculation for RHS-no1}
\end{align}
\keyword{Step 2:} We apply $\LieagbdPartialbar$ to (\ref{EQ-Kahler identity on Lie agbd-For alpha,(I,J), calculation for RHS-no1}):
\begin{align}
    2^{p+q-n}
    \begin{pmatrix*}[l]
        \displaystyle\sum_{\indexic{k}\in\indexIc}
        \sgn
            {1,\cdots,n,1^{'},\cdots,n^{'}}
            {\indexI,\indexJp,\indexIc,\indexJpc}
        \sgn
            {\indexIc}
            {\indexic{k},\indexIc\backslash\{\indexic{k}\}}
        \basedualxi{\indexic{1}}
        \wedge\cdots\wedge
        \LieagbdPartialbar\basedualxi{\indexic{k}}
        \wedge\cdots\wedge
        \basedualxi{\indexic{n-p}}
        \wedge
        \basexi{\indexJpc} 
        \\
        +
        \\
        \displaystyle\sum_{\indexjpc{l}\in\indexJpc}
        \sgn
            {1,\cdots,n,1^{'},\cdots,n^{'}}
            {\indexI,\indexJp,\indexIc,\indexJpc}
        \sgn
            {\indexIc,\indexJpc}
            {\indexjpc{l},\indexIc,\indexJpc\backslash\{\indexjpc{l}\}}
        \basedualxi{\indexIc}
        \wedge
        \basexi{\indexjpc{1}} 
        \wedge\cdots\wedge
        \LieagbdPartialbar\basexi{\indexjpc{l}}
        \wedge\cdots\wedge
        \basexi{\indexjpc{n-q}} 
        \\
    \end{pmatrix*}
    .
    \label{EQ-Kahler identity on Lie agbd-For alpha,(I,J), calculation for RHS-no2}
\end{align}
Recall that
\begin{align*}
    \LieagbdPartialbar\basedualxi{k}
    =
    \frac{1}{2}\displaystyle\sum_{s,t}
    \coefD{k}{s t}\basedualxi{s}\wedge\basedualxi{t}
    \qquad\text{and}\qquad
    \LieagbdPartialbar\basexi{l^{'}}
    =
    \displaystyle\sum_{s,t}
    \coefB{l^{'}}{s t}\basexi{s}\wedge\basedualxi{t}
    .
\end{align*} 
We plug these into (\ref{EQ-Kahler identity on Lie agbd-For alpha,(I,J), calculation for RHS-no2}) and cancel out the terms containing repeated frames in the wedge product, obtaining the following expression of $\LieagbdPartialbar\star(\basexi{\indexI}\wedge\basedualxi{\indexJp})$:
\begin{align}
     2^{p+q-n}
    \begin{pmatrix*}[l]
        \displaystyle\sum_{\indexic{k}\in\indexIc}
        \begin{pmatrix*}[l]
            \sgn
                {1,\cdots,n,1^{'},\cdots,n^{'}}
                {\indexI,\indexJp,\indexIc,\indexJpc}
            \sgn
                {\indexIc}
                {\indexic{k},\indexIc\backslash\{\indexic{k}\}}
            \\
            \basedualxi{\indexic{1}}
            \wedge\cdots\wedge
            \basedualxi{\indexic{k-1}}
            \wedge
            \begin{pmatrix*}[l]
                \frac{1}{2}
                \displaystyle\sum_{r\in\indexI}
                \coefD{\indexic{k}}{\indexic{k} r}
                \basedualxi{\indexic{k}}
                \wedge
                \basedualxi{r}
                +
                \coefD{\indexic{k}}{r \indexic{k}}
                \basedualxi{r}
                \wedge
                \basedualxi{\indexic{k}}
                \\
                +
                \\
                \frac{1}{2}
                \displaystyle\sum_{r_1\in\indexI}
                \displaystyle\sum_{r_2\in\indexI}
                \coefD{\indexic{k}}{r_1 r_2}
                \basedualxi{r_1}
                \wedge
                \basedualxi{r_2}
            \end{pmatrix*}
            \wedge
            \basedualxi{\indexic{k+1}}
            \wedge\cdots\wedge
            \\
            \basedualxi{\indexic{n-p}}
            \wedge
            \basexi{\indexJpc} 
        \end{pmatrix*}
        \\
        +
        \\
        \displaystyle\sum_{\indexjpc{l}\in\indexJpc}
        \begin{pmatrix*}[l]
            \sgn
                {1,\cdots,n,1^{'},\cdots,n^{'}}
            {\indexI,\indexJp,\indexIc,\indexJpc}
            \sgn
                {\indexIc,\indexJpc}
                {\indexjpc{l},\indexIc,\indexJpc\backslash\{\indexjpc{l}\}}
            \\
            \basedualxi{\indexIc}
            \wedge
            \basexi{\indexjpc{1}} 
            \wedge\cdots\wedge
            \basexi{\indexjpc{l-1}}
            \wedge
            \begin{pmatrix*}
                \displaystyle\sum_{r\in\indexI}
                \coefB{\indexjpc{l}}{\indexjpc{l} r}
                \basexi{\indexjpc{l}}
                \wedge
                \basedualxi{r}
                \\
                +
                \\
                \displaystyle\sum_{r\in\indexI}
                \displaystyle\sum_{u^{'}\in\indexJp}
                \coefB{\indexjpc{l}}{u^{'} r}
                \basexi{u^{'}}
                \wedge
                \basedualxi{r}
            \end{pmatrix*}
            \wedge
            \basexi{\indexjpc{l+1}}
            \wedge\cdots\wedge
            \basexi{\indexjpc{n-q}} 
        \end{pmatrix*}
    \end{pmatrix*}.
\label{EQ-Kahler identity on Lie agbd-For alpha,(I,J), calculation for RHS-no3}
\end{align}
We use the notation
$
    \basedualxi{\order{j k}}=
    \begin{cases}
        \basedualxi{j}\wedge\basedualxi{k}
        \qquad\text{if }j<k
        \\
        \basedualxi{k}\wedge\basedualxi{j}
        \qquad\text{if }k<j
    \end{cases}
$ 
to rewrite the formula
\begin{align*}
        \coefD{i}{j k}
        \basedualxi{j}\wedge\basedualxi{k}
    =
        \coefD{i}{k j}
        \basedualxi{k}\wedge\basedualxi{j}
    =  
        \coefD{i}{\order{j k}}
        \basedualxi{\order{j k}}
    ,
\end{align*}
where $\coefD{i}{j k}=-\coefD{i}{k j}$. Thus, the wedge product of frames in (\ref{EQ-Kahler identity on Lie agbd-For alpha,(I,J), calculation for RHS-no3}) can be simplified as follows.
\begin{itemize}
    \item 
\begin{align}
    &
        \basedualxi{\indexic{1}}
        \wedge\cdots\wedge
        \basedualxi{\indexic{k-1}}
        \wedge
        \frac{1}{2}
        \bigg(
            \displaystyle\sum_{r\in\indexI}
            \coefD{\indexic{k}}{\indexic{k} r}
            \basedualxi{\indexic{k}}
            \wedge
            \basedualxi{r}
            +
            \coefD{\indexic{k}}{r \indexic{k}}
            \basedualxi{r}
            \wedge
            \basedualxi{\indexic{k}}
        \bigg)
        \wedge
        \basedualxi{\indexic{k+1}}
        \wedge\cdots\wedge
        \basedualxi{\indexic{n-p}}
        \wedge
        \basexi{\indexJpc} 
\notag
    \\
    =&
        \basedualxi{\indexic{1}}
        \wedge\cdots\wedge
        \basedualxi{\indexic{k-1}}
        \wedge
        \bigg(
            \displaystyle\sum_{r\in\indexI}
            \coefD{\indexic{k}}{\order{\indexic{k} r}}
            \basedualxi{\order{\indexic{k} r}}
            \bigg)
        \wedge
        \basedualxi{\indexic{k+1}}
        \wedge\cdots\wedge
        \basedualxi{\indexic{n-p}}
        \wedge
        \basexi{\indexJpc} 
\notag
    \\
    =&
        \displaystyle\sum_{r\in\indexI}
        \sgn
            {\indexic{1},\cdots,\indexic{k-1},\order{\indexic{k},r},\indexic{k+1},\cdots,\indexic{n-p}}
            {\order{\indexIc\cup\{r\}}}
        \coefD{\indexic{k}}{\order{\indexic{k},r}}
        \basedualxi{\order{\indexIc\cup\{r\}}}
        \wedge
        \basexi{\indexJpc} 
    .
\label{EQ-Kahler identity on Lie agbd-For alpha,(I,J), calculation for RHS-no4}
\end{align}

\item
\begin{align}
    &
        \basedualxi{\indexic{1}}
        \wedge\cdots\wedge
        \basedualxi{\indexic{k-1}}
        \wedge
        \bigg(
            \frac{1}{2}
            \displaystyle\sum_{r_1\in\indexI}
            \displaystyle\sum_{r_2\in\indexI}
            \coefD{\indexic{k}}{r_1 r_2}
            \basedualxi{r_1}
            \wedge
            \basedualxi{r_2}
        \bigg)
        \wedge
        \basedualxi{\indexic{k+1}}
        \wedge\cdots\wedge
        \basedualxi{\indexic{n-p}}
        \wedge
        \basexi{\indexJpc} 
\notag
    \\
    =&
        \basedualxi{\indexic{1}}
        \wedge\cdots\wedge
        \basedualxi{\indexic{k-1}}
        \wedge
        \bigg(
            \frac{1}{2}
            \displaystyle\sum_{r_1\in\indexI}
            \displaystyle\sum_{r_2\in\indexI}
            \coefD{\indexic{k}}{\order{r_1 r_2}}
            \basedualxi{\order{r_1 r_2}}
        \bigg)
        \wedge
        \basedualxi{\indexic{k+1}}
        \wedge\cdots\wedge
        \basedualxi{\indexic{n-p}}
        \wedge
        \basexi{\indexJpc} 
\notag
    \\
    =&
        \displaystyle\sum_{r_1\in\indexI}
        \displaystyle\sum_{r_2\in\indexI}
        \frac{1}{2}
        \sgn
            {\indexic{1},\cdots,\indexic{k-1},\order{r_1,r_2},\indexic{k+1},\cdots,\indexic{n-p}}
            {\order{(\indexIc\backslash\{\indexic{k}\})\cup\{r_1,r_2\}}}
        \coefD{\indexic{k}}{\order{r_1,r_2}}
        \basedualxi{\order{(\indexIc\backslash\{\indexic{k}\})\cup\{r_1,r_2\}}}
        \wedge
        \basexi{\indexJpc}
    .
\label{EQ-Kahler identity on Lie agbd-For alpha,(I,J), calculation for RHS-no5}
\end{align}

\item
\begin{align}
    &
        \basedualxi{\indexIc}
        \wedge
        \basexi{\indexjpc{1}} 
        \wedge\cdots\wedge
        \basexi{\indexjpc{l-1}}
        \wedge
        \bigg(
            \displaystyle\sum_{r\in\indexI}
            \coefB{\indexjpc{l}}{\indexjpc{l} r}
            \basexi{\indexjpc{l}}
            \wedge
            \basedualxi{r}
        \bigg)
        \wedge
        \basexi{\indexjpc{l+1}}
        \wedge\cdots\wedge
        \basexi{\indexjpc{n-q}} 
\notag
    \\
    =&
        \displaystyle\sum_{r\in\indexI}
        \sgn
            {\indexIc,\indexjpc{1},\cdots,\indexjpc{l-1},\indexjpc{l},r,\indexjpc{l+1},\cdots,\indexjpc{n-1}}
            {\order{\indexIc\cup\{r\}},\indexJpc}
        \coefB{\indexjpc{l}}{\indexjpc{l} r}
        \basedualxi{\order{\indexIc\cup\{r\}}}
        \wedge
        \basexi{\indexJpc} 
    .
\label{EQ-Kahler identity on Lie agbd-For alpha,(I,J), calculation for RHS-no6}
\end{align}

\item
\begin{align}
    &
        \basedualxi{\indexIc}
        \wedge
        \basexi{\indexjpc{1}} 
        \wedge\cdots\wedge
        \basexi{\indexjpc{l-1}}
        \wedge
        \bigg(
            \displaystyle\sum_{r\in\indexI}
            \displaystyle\sum_{u^{'}\in\indexJp}
            \coefB{\indexjpc{l}}{u^{'} r}
            \basexi{u^{'}}
            \wedge
            \basedualxi{r}
        \bigg)
        \wedge
        \basexi{\indexjpc{l+1}}
        \wedge\cdots\wedge
        \basexi{\indexjpc{n-q}} 
\notag
    \\
    =&
        \displaystyle\sum_{r\in\indexI}
        \displaystyle\sum_{u^{'}\in\indexJp}
        \begin{matrix*}[l]
            \sgn
                {\indexIc,\indexjpc{1},\cdots,\indexjpc{l-1},u^{'},r,\indexjpc{l+1},\cdots,\indexjpc{n-1}}
                {\order{\indexIc\cup\{r\}},\order{(\indexJpc\backslash\{\indexjpc{l}\})\cup\{u^{'}\}}}
            \\
            \coefB{\indexjpc{l}}{u^{'} r}
            \basedualxi{\order{\indexIc\cup\{r\}}}
            \wedge
            \basexi{\order{(\indexJpc\backslash\{\indexjpc{l}\})\cup\{u^{'}\}}} 
        \end{matrix*}
    .
\label{EQ-Kahler identity on Lie agbd-For alpha,(I,J), calculation for RHS-no7}
\end{align}

\end{itemize}
Using (\ref{EQ-Kahler identity on Lie agbd-For alpha,(I,J), calculation for RHS-no4}), (\ref{EQ-Kahler identity on Lie agbd-For alpha,(I,J), calculation for RHS-no5}), (\ref{EQ-Kahler identity on Lie agbd-For alpha,(I,J), calculation for RHS-no6}), and (\ref{EQ-Kahler identity on Lie agbd-For alpha,(I,J), calculation for RHS-no7}) to simplify the corresponding terms in (\ref{EQ-Kahler identity on Lie agbd-For alpha,(I,J), calculation for RHS-no3}), we obtain the expression of $\LieagbdPartialbar\,\star(\basexi{\indexI}\wedge\basedualxi{\indexJp})$:
\begin{align}
    2^{p+q-n}
    \begin{pmatrix*}[l]
        \displaystyle\sum_{\indexic{k}\in\indexIc}
        \displaystyle\sum_{r\in\indexI}
        \begin{pmatrix*}[l]
            \sgn
                {1,\cdots,n,1^{'},\cdots,n^{'}}
                {\indexI,\indexJp,\indexIc,\indexJpc}                
            \sgn
                {\indexIc}
                {\indexic{k},\indexIc\backslash\{\indexic{k}\}}
            \\
            \sgn
                {\indexic{1},\cdots,\indexic{k-1},\order{\indexic{k},r},\indexic{k+1},\cdots,\indexic{n-p}}
                {\order{\indexIc\cup\{r\}}}
            \\
            \coefD{\indexic{k}}{\order{\indexic{k},r}}
            \basedualxi{\order{\indexIc\cup\{r\}}}
            \wedge
            \basexi{\indexJpc} 
        \end{pmatrix*}
        \\
        +
        \\
        \displaystyle\sum_{\indexic{k}\in\indexIc}
        \displaystyle\sum_{r_1\in\indexI}
        \displaystyle\sum_{r_2\in\indexI}
        \begin{pmatrix*}[l]
            \sgn
                {1,\cdots,n,1^{'},\cdots,n^{'}}
                {\indexI,\indexJp,\indexIc,\indexJpc}
            \sgn
                {\indexIc}
                {\indexic{k},\indexIc\backslash\{\indexic{k}\}}
            \\
            \sgn
                {\indexic{1},\cdots,\indexic{k-1},\order{r_1,r_2},\indexic{k+1},\cdots,\indexic{n-p}}
                {\order{(\indexIc\backslash\{\indexic{k}\})\cup\{r_1,r_2\}}}
            \\
            \frac{1}{2}
            \coefD{\indexic{k}}{\order{r_1,r_2}}
            \basedualxi{\order{(\indexIc\backslash\{\indexic{k}\})\cup\{r_1,r_2\}}}
            \wedge
            \basexi{\indexJpc} 
        \end{pmatrix*}
        \\
        +
        \\
        \displaystyle\sum_{\indexjpc{l}\in\indexJpc}
        \displaystyle\sum_{r\in\indexI}
        \begin{pmatrix*}[l]
            \sgn
                {1,\cdots,n,1^{'},\cdots,n^{'}}
                {\indexI,\indexJp,\indexIc,\indexJpc}
            \sgn
                {\indexIc,\indexJpc}
                {\indexjpc{l},\indexIc,\indexJpc\backslash\{\indexjpc{l}\}}
            \\
            \sgn
                {\indexIc,\indexjpc{1},\cdots,\indexjpc{l-1},\indexjpc{l},r,\indexjpc{l+1},\cdots,\indexjpc{n-1}}
                {\order{\indexIc\cup\{r\}},\indexJpc}
            \\
            \coefB{\indexjpc{l}}{\indexjpc{l} r}
            \basedualxi{\order{\indexIc\cup\{r\}}}
            \wedge
            \basexi{\indexJpc} 
        \end{pmatrix*}
        \\
        +
        \\
        \displaystyle\sum_{\indexjpc{l}\in\indexJpc}
        \displaystyle\sum_{r\in\indexI}
        \displaystyle\sum_{u^{'}\in\indexJp}
        \begin{pmatrix*}[l]
            \sgn
                {1,\cdots,n,1^{'},\cdots,n^{'}}
                {\indexI,\indexJp,\indexIc,\indexJpc}
            \sgn
                {\indexIc,\indexJpc}
                {\indexjpc{l},\indexIc,\indexJpc\backslash\{\indexjpc{l}\}}
            \\
            \sgn
                {\indexIc,\indexjpc{1},\cdots,\indexjpc{l-1},u^{'},r,\indexjpc{l+1},\cdots,\indexjpc{n-1}}
                {\order{\indexIc\cup\{r\}},\order{(\indexJpc\backslash\{\indexjpc{l}\})\cup\{u^{'}\}}}
            \\
            \coefB{\indexjpc{l}}{u^{'} r}
            \basedualxi{\order{\indexIc\cup\{r\}}}
            \wedge
            \basexi{\order{(\indexJpc\backslash\{\indexjpc{l}\})\cup\{u^{'}\}}} 
        \end{pmatrix*}
    \end{pmatrix*}.
\label{EQ-Kahler identity on Lie agbd-For alpha,(I,J), calculation for RHS-no8}
\end{align}
\keyword{Step 3:} We apply $\star$ to (\ref{EQ-Kahler identity on Lie agbd-For alpha,(I,J), calculation for RHS-no8}). Recall that
\begin{align*}
    \star
    (\basedualxi{\indexIc}\wedge\basexi{\indexJpc})
    =
    2^{(|\indexIc|+|\indexJpc|-n)}
    \sgn
        {1^{'},\cdots,n^{'},1,\cdots,n}{\indexIc,\indexJpc,\indexI,\indexJp}\basexi{\indexI}\wedge\basedualxi{\indexJp}.
\end{align*}
The expression of $
\star\,\LieagbdPartialbar\,\star
(\basexi{\indexI}\wedge\basedualxi{\indexJp})
$ is:
\begin{align}
    2
    \begin{pmatrix*}[l]
            \displaystyle\sum_{\indexic{k}\in\indexIc}
            \displaystyle\sum_{r\in\indexI}
            \begin{pmatrix*}[l]
                \sgn
                    {1,\cdots,n,1^{'},\cdots,n^{'}}
                    {\indexI,\indexJp,\indexIc,\indexJpc}
                \sgn
                    {\indexIc}
                    {\indexic{k},\indexIc\backslash\{\indexic{k}\}}
                \sgn
                    {\indexic{1},\cdots,\indexic{k-1},\order{\indexic{k},r},\indexic{k+1},\cdots,\indexic{n-p}}
                    {\order{\indexIc\cup\{r\}}}
                \\
                \sgn
                    {1^{'},\cdots,n^{'},1,\cdots,n}
                    {\order{\indexIc\cup\{r\}},\indexJpc,\indexI\backslash\{r\},\indexJp}
                \\
                \coefD{\indexic{k}}{\order{\indexic{k},r}}
                \basexi{\indexI\backslash\{r\}}
                \wedge
                \basedualxi{\indexJp} 
            \end{pmatrix*}
            \\
            +
            \\
            \displaystyle\sum_{\indexic{k}\in\indexIc}
            \displaystyle\sum_{r_1\in\indexI}
            \displaystyle\sum_{r_2\in\indexI}
            \begin{pmatrix*}[l]
                \sgn
                    {1,\cdots,n,1^{'},\cdots,n^{'}}
                    {\indexI,\indexJp,\indexIc,\indexJpc}
                \sgn
                    {\indexIc}
                    {\indexic{k},\indexIc\backslash\{\indexic{k}\}}
                \sgn
                    {\indexic{1},\cdots,\indexic{k-1},\order{r_1,r_2},\indexic{k+1},\cdots,\indexic{n-p}}
                    {\order{(\indexIc\backslash\{\indexic{k}\})\cup\{r_1,r_2\}}}
                \\
                \sgn
                    {1^{'},\cdots,n^{'},1,\cdots,n}
                    {\order{(\indexIc\backslash\{\indexic{k}\})\cup\{r_1,r_2\}},\indexJpc,\order{(\indexI\backslash\{r_1,r_2\})\cup\{\indexic{k}\}},\indexJp}
                \\
                \frac{1}{2}
                \coefD{\indexic{k}}{\order{r_1,r_2}}
                \basexi{\order{(\indexI\backslash\{r_1,r_2\})\cup\{\indexic{k}\}}}
                \wedge
                \basedualxi{\indexJp} 
            \end{pmatrix*}
            \\
            +
            \\
            \displaystyle\sum_{\indexjpc{l}\in\indexJpc}
            \displaystyle\sum_{r\in\indexI}
            \begin{pmatrix*}[l]
                \sgn
                    {1,\cdots,n,1^{'},\cdots,n^{'}}
                    {\indexI,\indexJp,\indexIc,\indexJpc}
                \sgn
                    {\indexIc,\indexJpc}
                    {\indexjpc{l},\indexIc,\indexJpc\backslash\{\indexjpc{l}\}}
                \\
                \sgn
                    {\indexIc,\indexjpc{1},\cdots,\indexjpc{l-1},\indexjpc{l},r,\indexjpc{l+1},\cdots,\indexjpc{n-1}}
                    {\order{\indexIc\cup\{r\}},\indexJpc}
                \sgn
                    {1^{'},\cdots,n^{'},1,\cdots,n}
                    {\order{\indexIc\cup\{r\}},\indexJpc,\indexI\backslash\{r\},\indexJp}
                \\
                \coefB{\indexjpc{l}}{\indexjpc{l} r}
                \basexi{\indexI\backslash\{r\}}
                \wedge
                \basedualxi{\indexJp} 
            \end{pmatrix*}
            \\
            +
            \\
            \displaystyle\sum_{\indexjpc{l}\in\indexJpc}
            \displaystyle\sum_{r\in\indexI}
            \displaystyle\sum_{u^{'}\in\indexJp}
            \begin{pmatrix*}[l]
                \sgn
                    {1,\cdots,n,1^{'},\cdots,n^{'}}
                    {\indexI,\indexJp,\indexIc,\indexJpc}
                \sgn
                    {\indexIc,\indexJpc}
                    {\indexjpc{l},\indexIc,\indexJpc\backslash\{\indexjpc{l}\}}
                \\
                \sgn
                    {\indexIc,\indexjpc{1},\cdots,\indexjpc{l-1},u^{'},r,\indexjpc{l+1},\cdots,\indexjpc{n-1}}
                    {\order{\indexIc\cup\{r\}},\order{(\indexJpc\backslash\{\indexjpc{l}\})\cup\{u^{'}\}}}
                \\
                \sgn
                    {1^{'},\cdots,n^{'},1,\cdots,n}
                    {\order{\indexIc\cup\{r\}},\order{(\indexJpc\backslash\{\indexjpc{l}\})\cup\{u^{'}\}},\indexI\backslash\{r\},\order{(\indexJp\backslash\{u^{'}\})\cup\{\indexjpc{l}\}}}
                \\
                \coefB{\indexjpc{l}}{u^{'} r}
                \basexi{\indexI\backslash\{r\}}
                \wedge
                \basedualxi{\order{(\indexJp\backslash\{u^{'}\})\cup\{\indexjpc{l}\}}} 
            \end{pmatrix*}
        \end{pmatrix*}.
\label{EQ-For alpha, (I ,J), calculation for RHS-step 3,unsimplified close to the target}
\end{align}

We simplify the signs of permutation in Section \ref{apex-simplification for RHS of alpha,(I,J)}. The details can be found in Lemmas \ref{Lemma-simplification for RHS of alpha,(I,J)-lemma 1}–\ref{Lemma-simplification for RHS of alpha,(I,J)-lemma 4}. We simplify the expression of $
\star\,\LieagbdPartialbar\,\star
(\basexi{\indexI}\wedge\basedualxi{\indexJp})
$ to
\begin{align*}
    2(-1)^{n^2}
    \begin{pmatrix*}[l]
        \displaystyle\sum_{\indexic{k}\in\indexIc}
        \displaystyle\sum_{r\in\indexI}
        \begin{matrix*}[l]
            \sgn
                {r,\indexic{k}}
                {\order{\indexic{k},r}}
            \sgn
                {r,\indexI\backslash\{r\}}                    
                {\indexI}
            \coefD{\indexic{k}}{\order{\indexic{k},r}}
            \basexi{\indexI\backslash\{r\}}
            \wedge
            \basedualxi{\indexJp} 
        \end{matrix*}
        \\
        +
        \\
        \displaystyle\sum_{\indexic{k}\in\indexIc\backslash\{\alpha\}}
        \displaystyle\sum_{r_1\in\indexI}
        \displaystyle\sum_{r_2\in\indexI}
        \begin{matrix*}[l]
            (-1)
            \sgn
                {\order{r_1,r_2},\order{(\indexI\backslash\{r_1,r_2\})}}
                {\indexI}
            \sgn
                {\order{(\indexI\backslash\{r_1,r_2\})\cup\{\indexic{k}\}}}
                {\indexic{k},\indexI\backslash\{r_1,r_2\})}
            \\
            \frac{1}{2}
            \coefD{\indexic{k}}{\order{r_1,r_2}}
            \basexi{\order{(\indexI\backslash\{r_1,r_2\})\cup\{\indexic{k}\}}}
            \wedge
            \basedualxi{\indexJp} 
        \end{matrix*}
        \\
        +
        \\
        \displaystyle\sum_{\indexjpc{l}\in\indexJpc}
        \displaystyle\sum_{r\in\indexI}
        \begin{matrix*}[l]
            (-1)
            \sgn
                {r,\indexI\backslash\{r\}}
                {\indexI}
            \coefB{\indexjpc{l}}{\indexjpc{l} r}
            \basexi{\indexI\backslash\{r\}}
            \wedge
            \basedualxi{\indexJp} 
        \end{matrix*}
        \\
        +
        \\
        \displaystyle\sum_{\indexjpc{l}\in\indexJpc}
        \displaystyle\sum_{r\in\indexI}
        \displaystyle\sum_{u^{'}\in\indexJp}
        \begin{matrix*}[l]
            (-1)^{p}
            \sgn
                {u^{'},r,\indexI\backslash\{r\},\order{(\indexJp\backslash\{u^{'}\})}}
                {\indexI,\indexJp}
            \sgn
                {\order{(\indexJp\backslash\{u^{'}\})\cup\{\indexjpc{l}\}}}
                {\indexjpc{l},(\indexJp\backslash\{u^{'}\})}
            \\
            \coefB{\indexjpc{l}}{u^{'} r}
            \basexi{\indexI\backslash\{r\}}
            \wedge
            \basedualxi{\order{(\indexJp\backslash\{u^{'}\})\cup\{\indexjpc{l}\}}} 
        \end{matrix*}
    \end{pmatrix*}.
\end{align*}
We can now calculate $
    \sqrt{-1}
    \basexi{\alpha}
    \wedge
    \big(
    -(-1)^{n^2}
    \star\,\LieagbdPartialbar\,\star
    \big)
    (\basexi{\indexI}\wedge\basedualxi{\indexJp})
$ as
\begin{align*}
    -2\sqrt{-1}
    \begin{pmatrix*}[l]
        \displaystyle\sum_{\indexic{k}\in\indexIc}
        \displaystyle\sum_{r\in\indexI}
        \begin{matrix*}[l]
            \sgn
                {r,\indexic{k}}
                {\order{\indexic{k},r}}
            \sgn
                {r,\indexI\backslash\{r\}}                    
                {\indexI}
            \coefD{\indexic{k}}{\order{\indexic{k},r}}
            \basexi{\alpha}
            \wedge
            \basexi{\indexI\backslash\{r\}}
            \wedge
            \basedualxi{\indexJp} 
        \end{matrix*}
        \\
        +
        \\
        \displaystyle\sum_{\indexic{k}\in\indexIc\backslash\{\alpha\}}
        \displaystyle\sum_{r_1\in\indexI}
        \displaystyle\sum_{r_2\in\indexI}
        \begin{matrix*}[l]
            (-1)
            \sgn
                {\order{r_1,r_2},\order{(\indexI\backslash\{r_1,r_2\})}}
                {\indexI}
            \sgn
                {\order{(\indexI\backslash\{r_1,r_2\})\cup\{\indexic{k}\}}}
                {\indexic{k},\indexI\backslash\{r_1,r_2\})}
            \\
            \frac{1}{2}
            \coefD{\indexic{k}}{\order{r_1,r_2}}
            \basexi{\alpha}
            \wedge
            \basexi{\order{(\indexI\backslash\{r_1,r_2\})\cup\{\indexic{k}\}}}
            \wedge
            \basedualxi{\indexJp} 
        \end{matrix*}
        \\
        +
        \\
        \displaystyle\sum_{\indexjpc{l}\in\indexJpc}
        \displaystyle\sum_{r\in\indexI}
        \begin{matrix*}[l]
            (-1)
            \sgn
                {r,\indexI\backslash\{r\}}
                {\indexI}
            \coefB{\indexjpc{l}}{\indexjpc{l} r}
            \basexi{\alpha}
            \wedge
            \basexi{\indexI\backslash\{r\}}
            \wedge
            \basedualxi{\indexJp} 
        \end{matrix*}
        \\
        +
        \\
        \displaystyle\sum_{\indexjpc{l}\in\indexJpc}
        \displaystyle\sum_{r\in\indexI}
        \displaystyle\sum_{u^{'}\in\indexJp}
        \begin{matrix*}[l]
            (-1)^{p}
            \sgn
                {u^{'},r,\indexI\backslash\{r\},\order{(\indexJp\backslash\{u^{'}\})}}
                {\indexI,\indexJp}
            \sgn
                {\order{(\indexJp\backslash\{u^{'}\})\cup\{\indexjpc{l}\}}}
                {\indexjpc{l},(\indexJp\backslash\{u^{'}\})}
            \\
            \coefB{\indexjpc{l}}{u^{'} r}
            \basexi{\alpha}
            \wedge
            \basexi{\indexI\backslash\{r\}}
            \wedge
            \basedualxi{\order{(\indexJp\backslash\{u^{'}\})\cup\{\indexjpc{l}\}}} 
        \end{matrix*}
    \end{pmatrix*}
    ,
\end{align*}
which completes the proof of Lemma \ref{lemma-Kahler identity on Lie agbd-For alpha,(I,J), calculation for RHS}.
\end{proof}

\subsection{
    Calculation of \texorpdfstring{$
    \sqrt{-1}\LieagbdPartial^{*}(\basexi{\alpha}\wedge\basexi{\indexI}\wedge\basedualxi{\indexJp})
    $}{}
    }
\label{subsection-For (alpha, I, J), calculation for RHS}
\begin{lemma}
\label{EQ-Kahler identity on Lie agbd-For (alpha,I,J), calculation for RHS-final target part 2}
$
\sqrt{-1}\LieagbdPartial^{*}
(\basexi{\alpha}\wedge\basexi{\indexI}\wedge\basedualxi{\indexJp})
$ is equal to (a) + (b) + (c) where
\\
    (a):
\begin{align}
    -2\sqrt{-1}
    \begin{pmatrix*}[l]
        \displaystyle\sum_{\indexic{k}\in\indexIc\backslash\{\alpha\}}
        \begin{matrix*}[l]
            \sgn
                {\order{\indexic{k},\alpha}}
                {\alpha,\indexic{k}}
            \coefD{\indexic{k}}{\order{\indexic{k}, \alpha}}
            \basexi{\indexI}
            \wedge
            \basedualxi{\indexJp}
        \end{matrix*}
        \\
        +
        \\
        \displaystyle\sum_{\indexjpc{l}\in\indexJpc}
        \begin{matrix*}[l]
            (-1)
            \coefB{\indexjpc{l}}{\indexjpc{l} \alpha}
            \basexi{\indexI}
            \wedge
            \basedualxi{\indexJp}
        \end{matrix*}
    \end{pmatrix*}
\label{EQ-Kahler identity on Lie agbd-For (alpha,I,J), calculation for RHS-final target part 1}
\end{align}
\\
    (b):
\begin{align}
    -2\sqrt{-1}
    \begin{pmatrix*}[l]
        \displaystyle\sum_{\indexic{k}\in\indexIc\backslash\{\alpha\}}
        \displaystyle\sum_{r\in\indexI}
        \begin{matrix*}[l]
            (-1)
            \sgn
                {r,\indexI\backslash\{r\}}
                {\indexI}
            \sgn
                {\order{\indexic{k},r}}
                {r,\indexic{k}}
            \coefD{\indexic{k}}{\order{\indexic{k}, r}}
            \basexi{\alpha}
            \wedge
            \basexi{\indexI\backslash\{r\}}
            \wedge
            \basedualxi{\indexJp}
        \end{matrix*}
        \\
        +
        \\
        \displaystyle\sum_{\indexic{k}\in\indexIc\backslash\{\alpha\}}
        \displaystyle\sum_{r_1\in\indexI}
        \displaystyle\sum_{r_2\in\indexI}
        \begin{matrix*}[l]
            \sgn
                {\order{r_1,r_2},\indexI\backslash\{r_1,r_2\}}
                {\indexI}
            \sgn
                {\order{(\indexI\backslash\{r_1,r_2\})\cup\{\indexic{k}\}}}
                {\indexic{k},(\indexI\backslash\{r_1,r_2\})}
        \\
            \frac{1}{2}
            \coefD{\indexic{k}}{\order{r_1,r_2}}
            \basexi{\alpha}
            \wedge
            \basexi{\order{(\indexI\backslash\{r_1,r_2\})\cup\{\indexic{k}\}}}
            \wedge
            \basedualxi{\indexJp}
        \end{matrix*}
        \\
        +
        \\
        \displaystyle\sum_{\indexjpc{l}\in\indexJpc}
        \displaystyle\sum_{r\in\indexI}
        \begin{matrix*}[l]
            \sgn
                {r,\indexI\backslash\{r\}}
                {\indexI}
            \coefB{\indexjpc{l}}{\indexjpc{l} r}
            \basexi{\alpha}
            \wedge
            \basexi{\indexI\backslash\{r\}}
            \wedge
            \basedualxi{\indexJp}
    \end{matrix*}
        \\
        +
        \\
        \displaystyle\sum_{\indexjpc{l}\in\indexJpc}
        \displaystyle\sum_{r\in\indexI}
        \displaystyle\sum_{u^{'}\in\indexJp}
        \begin{matrix*}[l]    
            (-1)^{p+1}
            \sgn
                {u^{'},r,\indexI\backslash\{r\},\indexJp\backslash\{u^{'}\}}
                {\indexI,\indexJp}
            \sgn
                {\order{(\indexJp\backslash\{u^{'}\})\cup\{\indexjpc{l}\}}}
                {\indexjpc{l},\indexJp\backslash\{u^{'}\}}
        \\
            \coefB{\indexjpc{l}}{u^{'} r}
            \basexi{\alpha}
            \wedge
            \basexi{\indexI\backslash\{r\}}
            \wedge
            \basedualxi{\order{(\indexJp\backslash\{u^{'}\})\cup\{\indexjpc{l}\}}}
        \end{matrix*}
    \end{pmatrix*}
\end{align}
\\
    (c): 
\begin{align} 
    -2\sqrt{-1}
    \begin{pmatrix*}[l]
        \displaystyle\sum_{\indexic{k}\in\indexIc\backslash\{\alpha\}}
        \displaystyle\sum_{r\in\indexI}
        \begin{matrix*}[l]
            (-1)
            \sgn
                {r,\indexI\backslash\{r\}}
                {\indexI}
            \sgn
                {\order{(\indexI\backslash\{r\})\cup\{\indexic{k}\}}}
                {\indexic{k},\indexI\backslash\{r\}}
            \sgn
                {\order{\alpha,r}}
                {\alpha,r}
        \\
            \coefD{\indexic{k}}{\order{\alpha, r}}
            \basexi{\order{(\indexI\backslash\{r\})\cup\{\indexic{k}\}}}
            \wedge
            \basedualxi{\indexJp}
        \end{matrix*}
        \\
        +
        \\   
        \displaystyle\sum_{\indexjpc{l}\in\indexJpc}
        \displaystyle\sum_{u^{'}\in\indexJp}
        \begin{matrix*}[l]
            (-1)
            \sgn
                {u^{'},\indexI,\order{(\indexJp\backslash\{u^{'}\})\cup\{\indexjpc{l}\}}}
                {\indexjpc{l},\indexI,\indexJp}
            \coefB{\indexjpc{l}}{u^{'} \alpha}
            \basexi{\indexI}
            \wedge
            \basedualxi{\order{(\indexJp\backslash\{u^{'}\})\cup\{\indexjpc{l}\}}}
        \end{matrix*}
    \end{pmatrix*}
    .
\label{EQ-Kahler identity on Lie agbd-For (alpha,I,J), calculation for RHS-final target part 3}
\end{align}
\end{lemma}
\begin{proof}
Recall that  $\LieagbdPartial^{*}=-(-1)^{n^{2}}\star\,\LieagbdPartialbar\,\star$. Our calculation comprises four steps:
\begin{enumerate}
    \item Calculate 
$
    \star(\basexi{\alpha}\wedge\basexi{\indexI}\wedge\basedualxi{\indexJp})
$.
    \item Calculate
$
    \LieagbdPartialbar\,\star(\basexi{\alpha}\wedge\basexi{\indexI}\wedge\basedualxi{\indexJp})
$.
    \item Calculate
$
    -(-1)^{n^{2}}\sqrt{-1}\star\,\LieagbdPartialbar\,\star
    (\basexi{\alpha}\wedge\basexi{\indexI}\wedge\basedualxi{\indexJp})
$.
    \item Decompose the result from Step 3 into three parts, where the first part is related to the expression of $
    (\sqrt{-1}\LieagbdPartial^{*}\basexi{\alpha})
    \wedge\basexi{\indexI}\wedge\basedualxi{\indexJp})
    $ and the second part is related to the expression of $
    \basexi{\alpha}\wedge
    \sqrt{-1}\LieagbdPartial^{*}(\basexi{\indexI}\wedge\basedualxi{\indexJp})
    $.
\end{enumerate}
Completing Step 4 gives the formulas in the statement of Lemma \ref{EQ-Kahler identity on Lie agbd-For (alpha,I,J), calculation for RHS-final target part 2}.
\\
\keyword{Step 1:} Using the definition of the Hodge star, we obtain
\begin{align}
    \star
    (\basexi{\alpha}\wedge\basexi{\indexI}\wedge\basedualxi{\indexJp})
    =
    2^{p+1+q-n}
    \sgn
        {1,\cdots,n,1^{'},\cdots,n^{'}}
        {\alpha,\indexI,\indexJp,\indexIc\backslash\{\alpha\},\indexJc}
    \basedualxi{\indexIc\backslash\{\alpha\}}
    \wedge
    \basexi{\indexJpc},
\label{EQ-For (alpha, I, J), calculation for RHS-Step 1, No1}
\end{align}
where $\{\indexic{1},\cdots,\indexic{n-p}\}=\indexIc$ and $\{\indexjpc{1},\cdots,\indexjpc{n-q}\}=\indexJpc$ are the complement sets of $\indexI$ and $\indexJp$.
\keyword{Step 2:}
We apply $\LieagbdPartialbar$ to (\ref{EQ-For (alpha, I, J), calculation for RHS-Step 1, No1}):
\begin{align}
    \LieagbdPartialbar\,\star
    (\basexi{\alpha}\wedge\basexi{\indexI}\wedge\basedualxi{\indexJp})
    =
        2^{p+1+q-n}
        \begin{pmatrix*}[l]
        \displaystyle\sum_{\indexic{k}\in\indexIc\backslash\{\alpha\}}
        \begin{matrix*}[l]
            \sgn
                {1,\cdots,n,1^{'},\cdots,n^{'}}
                {\alpha,\indexI,\indexJp,\indexIc\backslash\{\alpha\},\indexJc}
            \sgn
                {\indexIc\backslash\{\alpha\}}
                {\indexic{k},\indexIc\backslash\{\alpha,\indexic{k}\}}
            \\
            \basedualxi{\indexic{1}}
            \wedge\cdots\wedge
            \LieagbdPartialbar\basedualxi{\indexic{k}}
            \wedge\cdots\wedge
            \widehat{\basedualxi{\alpha}}
            \wedge\cdots\wedge
            \basedualxi{\indexic{n-p}}
            \wedge
            \basexi{\indexJpc}
        \end{matrix*}
        \\
        +
        \\
        \displaystyle\sum_{\indexjpc{l}\in\indexJpc}
        \begin{matrix*}[l]
            \sgn
                {1,\cdots,n,1^{'},\cdots,n^{'}}
                {\alpha,\indexI,\indexJp,\indexIc\backslash\{\alpha\},\indexJc}
            \sgn
                {\indexIc\backslash\{\alpha\},\indexJpc}
                {\indexjpc{l},\indexIc\backslash\{\alpha\},\indexJpc\backslash\{\indexjpc{l}\}}
            \\
            \basedualxi{\indexIc\backslash\{\alpha\}}
            \wedge
            \basexi{\indexjpc{1}}
            \wedge\cdots\wedge
            \LieagbdPartialbar\basexi{\indexjpc{l}}
            \wedge\cdots\wedge
            \basexi{\indexjpc{n-q}}
        \end{matrix*}
    \end{pmatrix*}.
\label{EQ-For (alpha, I, J), calculation for RHS, calculation for RHS-Step 2, No1}
\end{align}
Recalling that 
\begin{align*}
    \LieagbdPartialbar \xi_i
    =
    \sum_{j,k}^n
    \coefB{i}{jk}
    \basexi{j} \wedge \basedualxi{k}
    \qquad\text{and}\qquad
    \LieagbdPartialbar \basedualxi{i}
    =
    \sum_{j<k}^n
    \coefD{i}{jk}
    \basedualxi{j} \wedge \basedualxi{k}
    ,
\end{align*}
we plug these into (\ref{EQ-For (alpha, I, J), calculation for RHS, calculation for RHS-Step 2, No1}) and obtain
\begin{align}
    2^{p+1+q-n}
    \begin{pmatrix*}[l]
        \displaystyle\sum_{\indexic{k}\in\indexIc\backslash\{\alpha\}}
        \begin{pmatrix*}[l]
            \sgn
                {1,\cdots,n,1^{'},\cdots,n^{'}}
                {\alpha,\indexI,\indexJp,\indexIc\backslash\{\alpha\},\indexJc}
            \sgn
                {\indexIc\backslash\{\alpha\}}
                {\indexic{k},\indexIc\backslash\{\alpha,\indexic{k}\}}
            \\
            \basedualxi{\indexic{1}}
            \wedge\cdots\wedge
            \basedualxi{\indexic{k-1}}\wedge
            \begin{pmatrix*}[l]
                \frac{1}{2}
                (
                \coefD{\indexic{k}}{\indexic{k} \alpha}
                \basedualxi{\indexic{k}}\wedge\basedualxi{\alpha}
                +
                \coefD{\indexic{k}}{\alpha \indexic{k}}
                \basedualxi{\alpha}\wedge\basedualxi{\indexic{k}}
                )
                \\
                +
                \\
                \frac{1}{2}
                \displaystyle\sum_{r \in \indexI}
                \coefD{\indexic{k}}{\indexic{k} r}
                \basedualxi{\indexic{k}}\wedge\basedualxi{r}
                +
                \coefD{\indexic{k}}{r \indexic{k}}
                \basedualxi{r}\wedge\basedualxi{\indexic{k}}
                \\
                +
                \\
                \frac{1}{2}
                \displaystyle\sum_{r \in \indexI}
                \coefD{\indexic{k}}{r \alpha}
                \basedualxi{r}\wedge\basedualxi{\alpha}
                +
                \coefD{\indexic{k}}{\alpha r}
                \basedualxi{\alpha}\wedge\basedualxi{r}
                \\
                +
                \\
                \frac{1}{2}
                \displaystyle\sum_{r_1 \in \indexI}
                \displaystyle\sum_{r_2 \in \indexI}
                \coefD{\indexic{k}}{r_1 r_2}
                \basedualxi{r_1}\wedge\basedualxi{r_2}
            \end{pmatrix*}
            \wedge\basedualxi{\indexic{k+1}}
            \\
            \wedge\cdots\wedge
            \widehat{\basedualxi{\alpha}}
            \wedge\cdots\wedge
            \basedualxi{\indexic{n-p}}
            \wedge
            \basexi{\indexJpc}
        \end{pmatrix*}
        \\
        +
        \\
        \displaystyle\sum_{\indexjpc{l}\in\indexJpc}
        \begin{pmatrix*}[l]
            \sgn
                {1,\cdots,n,1^{'},\cdots,n^{'}}
                {\alpha,\indexI,\indexJp,\indexIc\backslash\{\alpha\},\indexJc}
            \sgn
                {\indexIc\backslash\{\alpha\},\indexJpc}
                {\indexjpc{l},\indexIc\backslash\{\alpha\},\indexJpc\backslash\{\indexjpc{l}\}}
            \\
            \basedualxi{\indexIc\backslash\{\alpha\}}
            \wedge
            \basexi{\indexjpc{1}}
            \wedge\cdots\wedge
            \basexi{\indexjpc{l-1}}\wedge
            \begin{pmatrix*}[l]
                \coefB{\indexjpc{l}}{\indexjpc{l} \alpha}
                \basexi{\indexjpc{l}}\wedge\basedualxi{\alpha}
                \\
                +
                \\
                \displaystyle\sum_{r\in\indexI}
                \coefB{\indexjpc{l}}{\indexjpc{l} r}
                \basexi{\indexjpc{l}}\wedge\basedualxi{r}
                \\
                +
                \\
                \displaystyle\sum_{u^{'}\in\indexJp}
                \coefB{\indexjpc{l}}{u^{'} \alpha}
                \basexi{u^{'}}\wedge\basedualxi{\alpha}
                \\
                +
                \\
                \displaystyle\sum_{r\in\indexI}
                \displaystyle\sum_{u\in\indexJp}
                \coefB{\indexjpc{l}}{u r}
                \basexi{u}\wedge\basedualxi{r}
            \end{pmatrix*}
            \wedge\basexi{\indexjpc{l+1}}
            \\
            \wedge\cdots\wedge
            \basexi{\indexjpc{n-q}}
        \end{pmatrix*}
    \end{pmatrix*}.
\label{EQ-For (alpha, I, J), calculation for RHS, calculation for RHS-Step 2, No2}
\end{align}

Recall that $\coefD{i}{j k}=-\coefD{i}{k j}$. Let us use the notation
$
    \basedualxi{\order{j k}}=
    \begin{cases}
        \basedualxi{j}\wedge\basedualxi{k}
        \qquad\text{if }j<k
        \\
        \basedualxi{k}\wedge\basedualxi{j}
        \qquad\text{if }k<j
    \end{cases}
$ and obtain the following expression:
\begin{align*}
        \coefD{i}{j k}
        \basedualxi{j}\wedge\basedualxi{k}
    =
        \coefD{i}{k j}
        \basedualxi{k}\wedge\basedualxi{j}
    =  
        \coefD{i}{\order{j k}}
        \basedualxi{\order{j k}}
    .
\end{align*}
Using this trick, we can simplify the terms in (\ref{EQ-For (alpha, I, J), calculation for RHS, calculation for RHS-Step 2, No2}) as follows:
\begin{align*}
    \begin{pmatrix*}[l]
        \frac{1}{2}
        (
        \coefD{\indexic{k}}{\indexic{k} \alpha}
        \basedualxi{\indexic{k}}\wedge\basedualxi{\alpha}
        +
        \coefD{\indexic{k}}{\alpha \indexic{k}}
        \basedualxi{\alpha}\wedge\basedualxi{\indexic{k}}
        )
        \\
        +
        \\
        \frac{1}{2}
        \displaystyle\sum_{r \in \indexI}
        \coefD{\indexic{k}}{\indexic{k} r}
        \basedualxi{\indexic{k}}\wedge\basedualxi{r}
        +
        \coefD{\indexic{k}}{r \indexic{k}}
        \basedualxi{r}\wedge\basedualxi{\indexic{k}}
        \\
        +
        \\
        \frac{1}{2}
        \displaystyle\sum_{r \in \indexI}
        \coefD{\indexic{k}}{r \alpha}
        \basedualxi{r}\wedge\basedualxi{\alpha}
        +
        \coefD{\indexic{k}}{\alpha r}
        \basedualxi{\alpha}\wedge\basedualxi{r}
        \\
        +
        \\
        \frac{1}{2}
        \displaystyle\sum_{r_1 \in \indexI}
        \displaystyle\sum_{r_2 \in \indexI}
        \coefD{\indexic{k}}{r_1 r_2}
        \basedualxi{r_1}\wedge\basedualxi{r_2}
    \end{pmatrix*}
    =
    \begin{pmatrix*}[l]
        \coefD{\indexic{k}}{\order{\indexic{k} \alpha}}
        \basedualxi{\order{\indexic{k} \alpha}}
        \\
        +
        \\
        \displaystyle\sum_{r \in \indexI}
        \coefD{\indexic{k}}{\order{\indexic{k} r}}
        \basedualxi{\order{\indexic{k} r}}
        \\
        +
        \\
        \displaystyle\sum_{r \in \indexI}
        \coefD{\indexic{k}}{\order{r \alpha}}
        \basedualxi{\order{r \alpha}}
        \\
        +
        \\
        \frac{1}{2}
        \displaystyle\sum_{r_1 \in \indexI}
        \displaystyle\sum_{r_2 \in \indexI}
        \coefD{\indexic{k}}{\order{r_1 r_2}}
        \basedualxi{\order{r_1 r_2}}
    \end{pmatrix*}.
\end{align*}

Next, we rearrange the terms in (\ref{EQ-For (alpha, I, J), calculation for RHS, calculation for RHS-Step 2, No2}). The calculation for ordering the wedge product of frames is as follows. In the first parenthesis, we have:
\begin{itemize}
    \item 
$
    \quad
        \basedualxi{\indexic{1}}
        \wedge\cdots\wedge
        \basedualxi{\indexic{k-1}}\wedge
        \coefD{\indexic{k}}{\order{\indexic{k} \alpha}}
        \basedualxi{\order{\indexic{k} \alpha}}
        \wedge\basedualxi{\indexic{k+1}}
        \wedge\cdots\wedge
        \widehat{\basedualxi{\alpha}}
        \wedge\cdots\wedge
        \basedualxi{\indexic{n-p}}
        \wedge
        \basexi{\indexJpc}
    \\
    =
        \sgn
            {\indexic{1},\cdots,\indexic{k-1},\order{\indexic{k},\alpha},\indexic{k+1},\cdots,\indexic{n-p}}
            {\indexIc}
        \coefD{\indexic{k}}{\order{\indexic{k}, \alpha}}
        \basedualxi{\indexIc}
        \wedge
        \basexi{\indexJpc}
$,

    \item 
$
    \quad
        \basedualxi{\indexic{1}}
        \wedge\cdots\wedge
        \basedualxi{\indexic{k-1}}\wedge
            \coefD{\indexic{k}}{\order{\indexic{k} r}}
            \basedualxi{\order{\indexic{k} r}}
        \wedge\basedualxi{\indexic{k+1}}
        \wedge\cdots\wedge
        \widehat{\basedualxi{\alpha}}
        \wedge\cdots\wedge
        \basedualxi{\indexic{n-p}}
        \wedge
        \basexi{\indexJpc}
    \\
    =
        \sgn
            {\indexic{1},\cdots,\indexic{k-1},\order{\indexic{k},r},\indexic{k+1},\cdots,\widehat{\alpha},\cdots,\indexic{n-p}}
            {\order{(\indexIc\backslash\{\alpha\})\cup\{r\}}}
        \coefD{\indexic{k}}{\order{\indexic{k}, r}}
        \basedualxi{\order{(\indexIc\backslash\{\alpha\})\cup\{r\}}}
        \wedge
        \basexi{\indexJpc}
$,
    \item 
$
    \quad
        \basedualxi{\indexic{1}}
        \wedge\cdots\wedge
        \basedualxi{\indexic{k-1}}\wedge
            \coefD{\indexic{k}}{\order{r \alpha}}
            \basedualxi{\order{r \alpha}}
        \wedge\basedualxi{\indexic{k+1}}
        \wedge\cdots\wedge
        \widehat{\basedualxi{\alpha}}
        \wedge\cdots\wedge
        \basedualxi{\indexic{n-p}}
        \wedge
        \basexi{\indexJpc}
    \\
    =
        \sgn
            {\indexic{1},\cdots,\indexic{k-1},\order{\alpha,r},\indexic{k+1},\cdots,\indexic{n-p}}
            {\order{(\indexIc\backslash\{\indexic{k}\})\cup\{r\}}}
        \coefD{\indexic{k}}{\order{\alpha, r}}
        \basedualxi{\order{(\indexIc\backslash\{\indexic{k}\})\cup\{r\}}}
        \wedge
        \basexi{\indexJpc}
$,
    \item 
$
    \quad
        \basedualxi{\indexic{1}}
        \wedge\cdots\wedge
        \basedualxi{\indexic{k-1}}\wedge
        \frac{1}{2}
            \coefD{\indexic{k}}{\order{r_1 r_2}}
            \basedualxi{\order{r_1 r_2}}
        \wedge\basedualxi{\indexic{k+1}}
        \wedge\cdots\wedge
        \widehat{\basedualxi{\alpha}}
        \wedge\cdots\wedge
        \basedualxi{\indexic{n-p}}
        \wedge
        \basexi{\indexJpc}
    \\
    =
        \sgn
            {\indexic{1},\cdots,\indexic{k-1},\order{r_1,r_2},\indexic{k+1},\cdots,\widehat{\alpha},\cdots,\indexic{n-p}}
            {\order{(\indexIc\backslash\{\indexic{k},\alpha\})\cup\{r_1,r_2\}}}
        \frac{1}{2}
        \coefD{\indexic{k}}{\order{r_1,r_2}}
        \basedualxi{\order{(\indexIc\backslash\{\indexic{k},\alpha\})\cup\{r_1,r_2\}}}
        \wedge
        \basexi{\indexJpc}
$.
\end{itemize}
In the second parenthesis, we have:
\begin{itemize}
    \item 
$
    \quad
        \basedualxi{\indexIc\backslash\{\alpha\}}
        \wedge
        \basexi{\indexjpc{1}}
        \wedge\cdots\wedge
        \basexi{\indexjpc{l-1}}\wedge
            \coefB{\indexjpc{l}}{\indexjpc{l} \alpha}
            \basexi{\indexjpc{l}}\wedge\basedualxi{\alpha}
        \wedge\basexi{\indexjpc{l+1}}
        \wedge\cdots\wedge
        \basexi{\indexjpc{n-q}}
    \\
    =
        \sgn
            {\indexIc\backslash\{\alpha\},\indexjpc{1},\cdots,\indexjpc{l-1},\indexjpc{l},\alpha,\indexjpc{l+1},\cdots,\indexjpc{n-q}}
            {\indexIc,\indexJpc}
        \coefB{\indexjpc{l}}{\indexjpc{l} \alpha}
        \basedualxi{\indexIc}
        \wedge
        \basexi{\indexJpc}
$,
    \item 
$
    \quad
        \basedualxi{\indexIc\backslash\{\alpha\}}
        \wedge
        \basexi{\indexjpc{1}}
        \wedge\cdots\wedge
        \basexi{\indexjpc{l-1}}\wedge
            \coefB{\indexjpc{l}}{\indexjpc{l} r}
            \basexi{\indexjpc{l}}\wedge\basedualxi{r}
        \wedge\basexi{\indexjpc{l+1}}
        \wedge\cdots\wedge
        \basexi{\indexjpc{n-q}}
    \\
    =
        \sgn
            {\indexIc\backslash\{\alpha\},\indexjpc{1},\cdots,\indexjpc{l-1},\indexjpc{l},r,\indexjpc{l+1},\cdots,\indexjpc{n-q}}
            {\order{(\indexIc\backslash\{\alpha\})\cup\{r\}},\indexJpc}
        \coefB{\indexjpc{l}}{\indexjpc{l} r}
        \basedualxi{\order{(\indexIc\backslash\{\alpha\})\cup\{r\}}}
        \wedge
        \basexi{\indexJpc}
$,
    \item 
$
    \quad
        \basedualxi{\indexIc\backslash\{\alpha\}}
        \wedge
        \basexi{\indexjpc{1}}
        \wedge\cdots\wedge
        \basexi{\indexjpc{l-1}}\wedge
            \coefB{\indexjpc{l}}{u^{'} \alpha}
            \basexi{u^{'}}\wedge\basedualxi{\alpha}
        \wedge\basexi{\indexjpc{l+1}}
        \wedge\cdots\wedge
        \basexi{\indexjpc{n-q}}
    \\
    =   
        \sgn
            {\indexIc\backslash\{\alpha\},\indexjpc{1},\cdots,\indexjpc{l-1},u^{'},\alpha,\indexjpc{l+1},\cdots,\indexjpc{n-q}}
            {\indexIc,\order{(\indexJpc\backslash\{\indexjpc{l}\})\cup\{u^{'}\}}}
        \coefB{\indexjpc{l}}{u^{'} \alpha}
        \basedualxi{\indexIc}
        \wedge
        \basexi{\order{(\indexJpc\backslash\{\indexjpc{l}\})\cup\{u^{'}\}}}
$,
    \item 
$
    \quad
        \basedualxi{\indexIc\backslash\{\alpha\}}
        \wedge
        \basexi{\indexjpc{1}}
        \wedge\cdots\wedge
        \basexi{\indexjpc{l-1}}\wedge
            \coefB{\indexjpc{l}}{u r}
            \basexi{u}\wedge\basedualxi{r}
        \wedge\basexi{\indexjpc{l+1}}
        \wedge\cdots\wedge
        \basexi{\indexjpc{n-q}}
    \\
    =
        \sgn
            {\indexIc\backslash\{\alpha\},\indexjpc{1},\cdots,\indexjpc{l-1},u^{'},r,\indexjpc{l+1},\cdots,\indexjpc{n-q}}
            {\order{(\indexIc\backslash\{\alpha\})\cup\{r\}},\order{(\indexJpc\backslash\{\indexjpc{l}\})\cup\{u^{'}\}}}
        \coefB{\indexjpc{l}}{u^{'} r}
        \basedualxi{\order{(\indexIc\backslash\{\alpha\})\cup\{r\}}}
        \wedge
        \basexi{\order{(\indexJpc\backslash\{\indexjpc{l}\})\cup\{u^{'}\}}}
$.
\end{itemize}

We conclude that $
\LieagbdPartialbar\,\star
(\basexi{\alpha}\wedge\basexi{\indexI}\wedge\basedualxi{\indexJp})
$ is equal to
\begin{align}
    &
    2^{p+1+q-n}
    \displaystyle\sum_{\indexic{k}\in\indexIc\backslash\{\alpha\}}
    \begin{pmatrix*}[l]
        \sgn
            {1,\cdots,n,1^{'},\cdots,n^{'}}
            {\alpha,\indexI,\indexJp,\indexIc\backslash\{\alpha\},\indexJc}
        \sgn
            {\indexIc\backslash\{\alpha\}}
            {\indexic{k},\indexIc\backslash\{\alpha,\indexic{k}\}}
        \sgn
            {\indexic{1},\cdots,\indexic{k-1},\order{\indexic{k},\alpha},\indexic{k+1},\cdots,\indexic{n-p}}
            {\indexIc}
        \\
        \coefD{\indexic{k}}{\order{\indexic{k}, \alpha}}
        \basedualxi{\indexIc}
        \wedge
        \basexi{\indexJpc}
    \end{pmatrix*}
\label{EQ-For (alpha, I, J), calculation for RHS, calculation for RHS-Step 2, result-term 1}
    \\
    &+\notag
    \\
    &
    2^{p+1+q-n}
    \displaystyle\sum_{\indexic{k}\in\indexIc\backslash\{\alpha\}}
    \displaystyle\sum_{r\in\indexI}
    \begin{pmatrix*}[l]
        \sgn
            {1,\cdots,n,1^{'},\cdots,n^{'}}
            {\alpha,\indexI,\indexJp,\indexIc\backslash\{\alpha\},\indexJc}
        \sgn
            {\indexIc\backslash\{\alpha\}}
            {\indexic{k},\indexIc\backslash\{\alpha,\indexic{k}\}}
        \\
        \sgn
            {\indexic{1},\cdots,\indexic{k-1},\order{\indexic{k},r},\indexic{k+1},\cdots,\widehat{\alpha},\cdots,\indexic{n-p}}
            {\order{(\indexIc\backslash\{\alpha\})\cup\{r\}}}
        \\
        \coefD{\indexic{k}}{\order{\indexic{k}, r}}
        \basedualxi{\order{(\indexIc\backslash\{\alpha\})\cup\{r\}}}
        \wedge
        \basexi{\indexJpc}
    \end{pmatrix*}
\label{EQ-For (alpha, I, J), calculation for RHS, calculation for RHS-Step 2, result-term 2}
    \\
    &+\notag
    \\
    &
    2^{p+1+q-n}
    \displaystyle\sum_{\indexic{k}\in\indexIc\backslash\{\alpha\}}
    \displaystyle\sum_{r\in\indexI}
    \begin{pmatrix*}[l]
        \sgn
            {1,\cdots,n,1^{'},\cdots,n^{'}}
            {\alpha,\indexI,\indexJp,\indexIc\backslash\{\alpha\},\indexJc}
        \sgn
            {\indexIc\backslash\{\alpha\}}
            {\indexic{k},\indexIc\backslash\{\alpha,\indexic{k}\}}
        \\
        \sgn
            {\indexic{1},\cdots,\indexic{k-1},\order{\alpha,r},\indexic{k+1},\cdots,\indexic{n-p}}
            {\order{(\indexIc\backslash\{\indexic{k}\})\cup\{r\}}}
        \\
        \coefD{\indexic{k}}{\order{\alpha, r}}
        \basedualxi{\order{(\indexIc\backslash\{\indexic{k}\})\cup\{r\}}}
        \wedge
        \basexi{\indexJpc}
    \end{pmatrix*}
\label{EQ-For (alpha, I, J), calculation for RHS, calculation for RHS-Step 2, result-term 3}
    \\
    &+\notag
    \\
    &
    2^{p+1+q-n}
    \displaystyle\sum_{\indexic{k}\in\indexIc\backslash\{\alpha\}}
    \displaystyle\sum_{r_1\in\indexI}
    \displaystyle\sum_{r_2\in\indexI}
    \begin{pmatrix*}[l]
        \sgn
            {1,\cdots,n,1^{'},\cdots,n^{'}}
            {\alpha,\indexI,\indexJp,\indexIc\backslash\{\alpha\},\indexJc}
        \sgn
            {\indexIc\backslash\{\alpha\}}
            {\indexic{k},\indexIc\backslash\{\alpha,\indexic{k}\}}
        \\
        \sgn
            {\indexic{1},\cdots,\indexic{k-1},\order{r_1,r_2},\indexic{k+1},\cdots,\widehat{\alpha},\cdots,\indexic{n-p}}
            {\order{(\indexIc\backslash\{\indexic{k},\alpha\})\cup\{r_1,r_2\}}}
        \\
        \frac{1}{2}
        \coefD{\indexic{k}}{\order{r_1,r_2}}
        \basedualxi{\order{(\indexIc\backslash\{\indexic{k},\alpha\})\cup\{r_1,r_2\}}}
        \wedge
        \basexi{\indexJpc}
    \end{pmatrix*}
\label{EQ-For (alpha, I, J), calculation for RHS, calculation for RHS-Step 2, result-term 4}
    \\
    &+\notag
    \\
    &
    2^{p+1+q-n}
    \displaystyle\sum_{\indexjpc{l}\in\indexJpc}
    \begin{pmatrix*}[l]
        \sgn
            {1,\cdots,n,1^{'},\cdots,n^{'}}
            {\alpha,\indexI,\indexJp,\indexIc\backslash\{\alpha\},\indexJc}
        \sgn
            {\indexIc\backslash\{\alpha\},\indexJpc}
            {\indexjpc{l},\indexIc\backslash\{\alpha\},\indexJpc\backslash\{\indexjpc{l}\}}
        \\
        \sgn
            {\indexIc\backslash\{\alpha\},\indexjpc{1},\cdots,\indexjpc{l-1},\indexjpc{l},\alpha,\indexjpc{l+1},\cdots,\indexjpc{n-q}}
            {\indexIc,\indexJpc}
        \\
        \coefB{\indexjpc{l}}{\indexjpc{l} \alpha}
        \basedualxi{\indexIc}
        \wedge
        \basexi{\indexJpc}
    \end{pmatrix*}
\label{EQ-For (alpha, I, J), calculation for RHS, calculation for RHS-Step 2, result-term 5}
    \\
    &+\notag
    \\
    &
    2^{p+1+q-n}
    \displaystyle\sum_{\indexjpc{l}\in\indexJpc}
    \displaystyle\sum_{r\in\indexI}
    \begin{pmatrix*}[l]
        \sgn
            {1,\cdots,n,1^{'},\cdots,n^{'}}
            {\alpha,\indexI,\indexJp,\indexIc\backslash\{\alpha\},\indexJc}
        \sgn
            {\indexIc\backslash\{\alpha\},\indexJpc}
            {\indexjpc{l},\indexIc\backslash\{\alpha\},\indexJpc\backslash\{\indexjpc{l}\}}
        \\
        \sgn
            {\indexIc\backslash\{\alpha\},\indexjpc{1},\cdots,\indexjpc{l-1},\indexjpc{l},r,\indexjpc{l+1},\cdots,\indexjpc{n-q}}
            {\order{(\indexIc\backslash\{\alpha\})\cup\{r\}},\indexJpc}
        \\
        \coefB{\indexjpc{l}}{\indexjpc{l} r}
        \basedualxi{\order{(\indexIc\backslash\{\alpha\})\cup\{r\}}}
        \wedge
        \basexi{\indexJpc}
    \end{pmatrix*}
\label{EQ-For (alpha, I, J), calculation for RHS, calculation for RHS-Step 2, result-term 6}
    \\
    &+\notag
    \\
    &
    2^{p+1+q-n}
    \displaystyle\sum_{\indexjpc{l}\in\indexJpc}
    \displaystyle\sum_{u^{'}\in\indexJp}
    \begin{pmatrix*}[l]
        \sgn
            {1,\cdots,n,1^{'},\cdots,n^{'}}
            {\alpha,\indexI,\indexJp,\indexIc\backslash\{\alpha\},\indexJc}
        \sgn
            {\indexIc\backslash\{\alpha\},\indexJpc}
            {\indexjpc{l},\indexIc\backslash\{\alpha\},\indexJpc\backslash\{\indexjpc{l}\}}
        \\
        \sgn
            {\indexIc\backslash\{\alpha\},\indexjpc{1},\cdots,\indexjpc{l-1},u^{'},\alpha,\indexjpc{l+1},\cdots,\indexjpc{n-q}}
            {\indexIc,\order{(\indexJpc\backslash\{\indexjpc{l}\})\cup\{u^{'}\}}}
        \\
        \coefB{\indexjpc{l}}{u^{'} \alpha}
        \basedualxi{\indexIc}
        \wedge
        \basexi{\order{(\indexJpc\backslash\{\indexjpc{l}\})\cup\{u^{'}\}}}
    \end{pmatrix*}
\label{EQ-For (alpha, I, J), calculation for RHS, calculation for RHS-Step 2, result-term 7}
    \\
    &+\notag
    \\
    &
    2^{p+1+q-n}
    \displaystyle\sum_{\indexjpc{l}\in\indexJpc}
    \displaystyle\sum_{r\in\indexI}
    \displaystyle\sum_{u^{'}\in\indexJp}
    \begin{pmatrix*}[l]
        \sgn
            {1,\cdots,n,1^{'},\cdots,n^{'}}
            {\alpha,\indexI,\indexJp,\indexIc\backslash\{\alpha\},\indexJc}
        \sgn
            {\indexIc\backslash\{\alpha\},\indexJpc}
            {\indexjpc{l},\indexIc\backslash\{\alpha\},\indexJpc\backslash\{\indexjpc{l}\}}
        \\
        \sgn
            {\indexIc\backslash\{\alpha\},\indexjpc{1},\cdots,\indexjpc{l-1},u^{'},r,\indexjpc{l+1},\cdots,\indexjpc{n-q}}
            {\order{(\indexIc\backslash\{\alpha\})\cup\{r\}},\order{(\indexJpc\backslash\{\indexjpc{l}\})\cup\{u^{'}\}}}
        \\
        \coefB{\indexjpc{l}}{u^{'} r}
        \basedualxi{\order{(\indexIc\backslash\{\alpha\})\cup\{r\}}}
        \wedge
        \basexi{\order{(\indexJpc\backslash\{\indexjpc{l}\})\cup\{u^{'}\}}}
    \end{pmatrix*}
    .
\label{EQ-For (alpha, I, J), calculation for RHS, calculation for RHS-Step 2, result-term 8}
\end{align}
\keyword{Step 3:}
Recall that
\begin{align*}
    \star
    (\basedualxi{\indexIc}\wedge\basexi{\indexJpc})
    =
    2^{(|\indexIc|+|\indexJpc|-n)}
    \sgn
        {1^{'},\cdots,n^{'},1,\cdots,n}{\indexIc,\indexJpc,\indexI,\indexJp}\basexi{\indexI}\wedge\basedualxi{\indexJp}.
\end{align*}
Let us take $\star$ on $
\LieagbdPartialbar\,\star
(\basexi{\alpha}\wedge\basexi{\indexI}\wedge\basedualxi{\indexJp})
$. Each term in $
\LieagbdPartialbar\,\star
(\basexi{\alpha}\wedge\basexi{\indexI}\wedge\basedualxi{\indexJp})
$ is as follows: 
\begin{itemize}
    \item applying $\star$ to (\ref{EQ-For (alpha, I, J), calculation for RHS, calculation for RHS-Step 2, result-term 1}), we calculate
\begin{align*}
        \star
        (\basedualxi{\indexIc}\wedge\basexi{\indexJpc})
    =
        2^{n-p-q}
        \sgn
            {1^{'},\cdots,n^{'},1,\cdots,n}
            {\indexIc,\indexJpc,\indexI,\indexJp}
        \basexi{\indexI}\wedge\basedualxi{\indexJp}
    ;
\end{align*}  

    \item applying $\star$ to (\ref{EQ-For (alpha, I, J), calculation for RHS, calculation for RHS-Step 2, result-term 2}), we calculate
\begin{align*}
        \star
        (
        \basedualxi{\order{(\indexIc\backslash\{\alpha\})\cup\{r\}}}
        \wedge
        \basexi{\indexJpc}
        )
    =
        2^{n-p-q}
        \sgn
            {1^{'},\cdots,n^{'},1,\cdots,n}
            {\order{(\indexIc\backslash\{\alpha\})\cup\{r\}},\indexJpc,\alpha,\indexI\backslash\{r\},\indexJp}
        \basexi{\alpha}
        \wedge
        \basexi{\indexI\backslash\{r\}}
        \wedge
        \basedualxi{\indexJp}
    ,
\end{align*}
    where $\basexi{\alpha}$ is at the head of the wedge product of frames since we use the assumption that $\alpha<\indexi{1}$;

    \item applying $\star$ to (\ref{EQ-For (alpha, I, J), calculation for RHS, calculation for RHS-Step 2, result-term 3}), we calculate
\begin{align*}
    &
        \star
        (
        \basedualxi{\order{(\indexIc\backslash\{\indexic{k}\})\cup\{r\}}}
        \wedge
        \basexi{\indexJpc}
        )
    \\
    =&
        2^{n-p-q}
        \sgn
            {1^{'},\cdots,n^{'},1,\cdots,n}
            {\order{(\indexIc\backslash\{\indexic{k}\})\cup\{r\}},\indexJpc,\order{(\indexI\backslash\{r\})\cup\{\indexic{k}\}},\indexJp}
        \basexi{\order{(\indexI\backslash\{r\})\cup\{\indexic{k}\}}}
        \wedge
        \basedualxi{\indexJp}
    ;
\end{align*}  

    \item applying $\star$ to (\ref{EQ-For (alpha, I, J), calculation for RHS, calculation for RHS-Step 2, result-term 4}), we calculate
\begin{align*}
    &
        \star
        (
        \basedualxi{\order{(\indexIc\backslash\{\indexic{k},\alpha\})\cup\{r_1,r_2\}}}
        \wedge
        \basexi{\indexJpc}
        )
    \\
    =&
        2^{n-p-q}
        \begin{pmatrix*}[l]
            \sgn
                {1^{'},\cdots,n^{'},1,\cdots,n}
                {\order{(\indexIc\backslash\{\indexic{k},\alpha\})\cup\{r_1,r_2\}},\indexJpc,\alpha,\order{(\indexI\backslash\{r_1,r_2\})\cup\{\indexic{k}\}},\indexJp} 
            \\
            \basexi{\alpha}
            \wedge
            \basexi{\order{(\indexI\backslash\{r_1,r_2\})\cup\{\indexic{k}\}}}
            \wedge
            \basedualxi{\indexJp}
        \end{pmatrix*}
    ;
\end{align*}  
    
    \item applying $\star$ to (\ref{EQ-For (alpha, I, J), calculation for RHS, calculation for RHS-Step 2, result-term 5}), we calculate
\begin{align*}
        \star
        (
        \basedualxi{\indexIc}
        \wedge
        \basexi{\indexJpc}
        )
    =
        2^{n-p-q}
        \sgn
            {1^{'},\cdots,n^{'},1,\cdots,n}
            {\indexIc,\indexJpc,\indexI,\indexJp}
        \basexi{\indexI}
        \wedge
        \basedualxi{\indexJp}
    ;
\end{align*}  
    
    \item applying $\star$ to (\ref{EQ-For (alpha, I, J), calculation for RHS, calculation for RHS-Step 2, result-term 6}), we calculate
\begin{align*}
        \star
        (
        \basedualxi{\order{(\indexIc\backslash\{\alpha\})\cup\{r\}}}
        \wedge
        \basexi{\indexJpc}
        )
    =
        2^{n-p-q}
        \sgn
            {1^{'},\cdots,n^{'},1,\cdots,n}
            {\order{(\indexIc\backslash\{\alpha\})\cup\{r\}},\indexJpc,\alpha,\indexI\backslash\{r\},\indexJp}
        \basexi{\alpha}
        \wedge
        \basexi{\indexI\backslash\{r\}}
        \wedge
        \basedualxi{\indexJp}
    ,
\end{align*}  
    where $\basexi{\alpha}$ is at the head of the wedge product of frames since we use the assumption that $\alpha<\indexi{1}$;
    
    \item applying $\star$ to (\ref{EQ-For (alpha, I, J), calculation for RHS, calculation for RHS-Step 2, result-term 7}), we calculate
\begin{align*}
    &
        \star
        (
        \basedualxi{\indexIc}
        \wedge
        \basexi{\order{(\indexJpc\backslash\{\indexjpc{l}\})\cup\{u^{'}\}}}
        )
    \\
    =&
        2^{n-p-q}
        \sgn
            {1^{'},\cdots,n^{'},1,\cdots,n}
            {\indexIc,\order{(\indexJpc\backslash\{\indexjpc{l}\})\cup\{u^{'}\}},\indexI,\order{(\indexJp\backslash\{u^{'}\})\cup\{\indexjpc{l}\}}}
        \basexi{\indexI}
        \wedge
        \basedualxi{\order{(\indexJp\backslash\{u^{'}\})\cup\{\indexjpc{l}\}}}
    ;
\end{align*}  
    
    \item applying $\star$ to (\ref{EQ-For (alpha, I, J), calculation for RHS, calculation for RHS-Step 2, result-term 8}), we calculate
\begin{align*}
    &
        \star
        (
        \basedualxi{\order{(\indexIc\backslash\{\alpha\})\cup\{r\}}}
        \wedge
        \basexi{\order{(\indexJpc\backslash\{\indexjpc{l}\})\cup\{u^{'}\}}}
        )
    \\
    =&
        2^{n-p-q}
        \begin{pmatrix*}
            \sgn
                {1^{'},\cdots,n^{'},1,\cdots,n}
                {\order{(\indexIc\backslash\{\alpha\})\cup\{r\}},\order{(\indexJpc\backslash\{\indexjpc{l}\})\cup\{u^{'}\}},\alpha,\indexI\backslash\{r\},\order{(\indexJp\backslash\{u^{'}\})\cup\{\indexjpc{l}\}}}
            \\
            \basexi{\alpha}
            \wedge
            \basexi{\indexI\backslash\{r\}}
            \wedge
            \basedualxi{\order{(\indexJp\backslash\{u^{'}\})\cup\{\indexjpc{l}\}}}
        \end{pmatrix*}
    ,
\end{align*}  
    where $\basexi{\alpha}$ is at the head of the wedge product of frames since we use the assumption that $\alpha<\indexi{1}$.
\end{itemize}

Recall that 
\begin{align*}
    \sqrt{-1}\LieagbdPartial^{*}(\basexi{\alpha}\wedge\basexi{\indexI}\wedge\basedualxi{\indexJp})
    =
    -\sqrt{-1}(-1)^{n^2}
    \star\,\LieagbdPartialbar\,\star
    (\basexi{\alpha}\wedge\basexi{\indexI}\wedge\basedualxi{\indexJp})
    .
\end{align*}
As we have already applied $\star$ to each term in $
\LieagbdPartialbar\,\star
(\basexi{\alpha}\wedge\basexi{\indexI}\wedge\basedualxi{\indexJp})
$, we just need to multiply by $-\sqrt{-1}(-1)^{n^2}$ to obtain the expression of $
\sqrt{-1}\LieagbdPartial^{*}
(\basexi{\alpha}\wedge\basexi{\indexI}\wedge\basedualxi{\indexJp})
$:
\begin{align}
    &
    \begin{pmatrix*}[l]
        -2\sqrt{-1}(-1)^{n^2}
        \displaystyle\sum_{\indexic{k}\in\indexIc\backslash\{\alpha\}}
        \sgn
            {1,\cdots,n,1^{'},\cdots,n^{'}}
            {\alpha,\indexI,\indexJp,\indexIc\backslash\{\alpha\},\indexJc}
        \sgn
            {\indexIc\backslash\{\alpha\}}
            {\indexic{k},\indexIc\backslash\{\alpha,\indexic{k}\}}
        \\
        \sgn
            {\indexic{1},\cdots,\indexic{k-1},\order{\indexic{k},\alpha},\indexic{k+1},\cdots,\indexic{n-p}}
            {\indexIc}
        \sgn
            {1^{'},\cdots,n^{'},1,\cdots,n}
            {\indexIc,\indexJpc,\indexI,\indexJp}
        \coefD{\indexic{k}}{\order{\indexic{k}, \alpha}}
        \basexi{\indexI}
        \wedge
        \basedualxi{\indexJp}
    \end{pmatrix*}
\label{EQ-For (alpha, I, J), calculation for RHS-Step 3, unsimplfied term 1}
    \\
    &+
    \notag\\
    &
    \begin{pmatrix*}[l]
        -2\sqrt{-1}(-1)^{n^2}
        \displaystyle\sum_{\indexic{k}\in\indexIc\backslash\{\alpha\}}
        \displaystyle\sum_{r\in\indexI}
        \sgn
            {1,\cdots,n,1^{'},\cdots,n^{'}}
            {\alpha,\indexI,\indexJp,\indexIc\backslash\{\alpha\},\indexJc}
        \sgn
            {\indexIc\backslash\{\alpha\}}
            {\indexic{k},\indexIc\backslash\{\alpha,\indexic{k}\}}
        \\
        \sgn
            {\indexic{1},\cdots,\indexic{k-1},\order{\indexic{k},r},\indexic{k+1},\cdots,\widehat{\alpha},\cdots,\indexic{n-p}}
            {\order{(\indexIc\backslash\{\alpha\})\cup\{r\}}}
        \sgn
            {1^{'},\cdots,n^{'},1,\cdots,n}
            {\order{(\indexIc\backslash\{\alpha\})\cup\{r\}},\indexJpc,\alpha,\indexI\backslash\{r\},\indexJp}
        \coefD{\indexic{k}}{\order{\indexic{k}, r}}
        \\
        \basexi{\alpha}
        \wedge
        \basexi{\indexI\backslash\{r\}}
        \wedge
        \basedualxi{\indexJp}
    \end{pmatrix*}
\label{EQ-For (alpha, I, J), calculation for RHS-Step 3, unsimplfied term 2}
    \\
    &+
    \notag\\
    &
    \begin{pmatrix*}[l]
        -2\sqrt{-1}(-1)^{n^2}
        \displaystyle\sum_{\indexic{k}\in\indexIc\backslash\{\alpha\}}
        \displaystyle\sum_{r\in\indexI}
        \sgn
            {1,\cdots,n,1^{'},\cdots,n^{'}}
            {\alpha,\indexI,\indexJp,\indexIc\backslash\{\alpha\},\indexJc}
        \sgn
            {\indexIc\backslash\{\alpha\}}
            {\indexic{k},\indexIc\backslash\{\alpha,\indexic{k}\}}
        \\
        \sgn
            {\indexic{1},\cdots,\indexic{k-1},\order{\alpha,r},\indexic{k+1},\cdots,\indexic{n-p}}
            {\order{(\indexIc\backslash\{\indexic{k}\})\cup\{r\}}}
        \sgn
            {1^{'},\cdots,n^{'},1,\cdots,n}
            {\order{(\indexIc\backslash\{\indexic{k}\})\cup\{r\}},\indexJpc,\order{(\indexI\backslash\{r\})\cup\{\indexic{k}\}},\indexJp}
        \\
        \coefD{\indexic{k}}{\order{\alpha, r}}
        \basexi{\order{(\indexI\backslash\{r\})\cup\{\indexic{k}\}}}
        \wedge
        \basedualxi{\indexJp}
    \end{pmatrix*}
\label{EQ-For (alpha, I, J), calculation for RHS-Step 3, unsimplfied term 3}
    \\
    &+
    \notag\\
    &
    \begin{pmatrix*}[l]
        -2\sqrt{-1}(-1)^{n^2}
        \displaystyle\sum_{\indexic{k}\in\indexIc\backslash\{\alpha\}}
        \displaystyle\sum_{r_1\in\indexI}
        \displaystyle\sum_{r_2\in\indexI}
        \sgn
            {1,\cdots,n,1^{'},\cdots,n^{'}}
            {\alpha,\indexI,\indexJp,\indexIc\backslash\{\alpha\},\indexJc}
        \sgn
            {\indexIc\backslash\{\alpha\}}
            {\indexic{k},\indexIc\backslash\{\alpha,\indexic{k}\}}
        \\
        \sgn
            {\indexic{1},\cdots,\indexic{k-1},\order{r_1,r_2},\indexic{k+1},\cdots,\widehat{\alpha},\cdots,\indexic{n-p}}
            {\order{(\indexIc\backslash\{\indexic{k},\alpha\})\cup\{r_1,r_2\}}}
        \\
        \sgn
            {1^{'},\cdots,n^{'},1,\cdots,n}
            {\order{(\indexIc\backslash\{\indexic{k},\alpha\})\cup\{r_1,r_2\}},\indexJpc,\alpha,\order{(\indexI\backslash\{r_1,r_2\})\cup\{\indexic{k}\}},\indexJp} 
        \\
        \frac{1}{2}
        \coefD{\indexic{k}}{\order{r_1,r_2}}
        \basexi{\alpha}
        \wedge
        \basexi{\order{(\indexI\backslash\{r_1,r_2\})\cup\{\indexic{k}\}}}
        \wedge
        \basedualxi{\indexJp}
    \end{pmatrix*}
\label{EQ-For (alpha, I, J), calculation for RHS-Step 3, unsimplfied term 4}
    \\
    &+
    \notag\\
    &
    \begin{pmatrix*}[l]
        -2\sqrt{-1}(-1)^{n^2}
        \displaystyle\sum_{\indexjpc{l}\in\indexJpc}
        \sgn
            {1,\cdots,n,1^{'},\cdots,n^{'}}
            {\alpha,\indexI,\indexJp,\indexIc\backslash\{\alpha\},\indexJc}
        \sgn
            {\indexIc\backslash\{\alpha\},\indexJpc}
            {\indexjpc{l},\indexIc\backslash\{\alpha\},\indexJpc\backslash\{\indexjpc{l}\}}
        \\
        \sgn
            {\indexIc\backslash\{\alpha\},\indexjpc{1},\cdots,\indexjpc{l-1},\indexjpc{l},\alpha,\indexjpc{l+1},\cdots,\indexjpc{n-q}}
            {\indexIc,\indexJpc}
        \sgn
            {1^{'},\cdots,n^{'},1,\cdots,n}
            {\indexIc,\indexJpc,\indexI,\indexJp}
        \coefB{\indexjpc{l}}{\indexjpc{l} \alpha}
        \basexi{\indexI}
        \wedge
        \basedualxi{\indexJp}
    \end{pmatrix*}
\label{EQ-For (alpha, I, J), calculation for RHS-Step 3, unsimplfied term 5}
    \\
    &+
    \notag\\
    &
    \begin{pmatrix*}[l]
        -2\sqrt{-1}(-1)^{n^2}
        \displaystyle\sum_{\indexjpc{l}\in\indexJpc}
        \displaystyle\sum_{r\in\indexI}
        \sgn
            {1,\cdots,n,1^{'},\cdots,n^{'}}
            {\alpha,\indexI,\indexJp,\indexIc\backslash\{\alpha\},\indexJc}
        \sgn
            {\indexIc\backslash\{\alpha\},\indexJpc}
            {\indexjpc{l},\indexIc\backslash\{\alpha\},\indexJpc\backslash\{\indexjpc{l}\}}
        \\
        \sgn
            {\indexIc\backslash\{\alpha\},\indexjpc{1},\cdots,\indexjpc{l-1},\indexjpc{l},r,\indexjpc{l+1},\cdots,\indexjpc{n-q}}
            {\order{(\indexIc\backslash\{\alpha\})\cup\{r\}},\indexJpc}
        \sgn
            {1^{'},\cdots,n^{'},1,\cdots,n}
            {\order{(\indexIc\backslash\{\alpha\})\cup\{r\}},\indexJpc,\alpha,\indexI\backslash\{r\},\indexJp}
        \\
        \coefB{\indexjpc{l}}{\indexjpc{l} r}
        \basexi{\alpha}
        \wedge
        \basexi{\indexI\backslash\{r\}}
        \wedge
        \basedualxi{\indexJp}
    \end{pmatrix*}
\label{EQ-For (alpha, I, J), calculation for RHS-Step 3, unsimplfied term 6}
    \\
    &+
    \notag\\
    &
    \begin{pmatrix*}[l]
        -2\sqrt{-1}(-1)^{n^2}
        \displaystyle\sum_{\indexjpc{l}\in\indexJpc}
        \displaystyle\sum_{u^{'}\in\indexJp}
        \sgn
            {1,\cdots,n,1^{'},\cdots,n^{'}}
            {\alpha,\indexI,\indexJp,\indexIc\backslash\{\alpha\},\indexJc}
        \sgn
            {\indexIc\backslash\{\alpha\},\indexJpc}
            {\indexjpc{l},\indexIc\backslash\{\alpha\},\indexJpc\backslash\{\indexjpc{l}\}}
        \\
        \sgn
            {\indexIc\backslash\{\alpha\},\indexjpc{1},\cdots,\indexjpc{l-1},u^{'},\alpha,\indexjpc{l+1},\cdots,\indexjpc{n-q}}
            {\indexIc,\order{(\indexJpc\backslash\{\indexjpc{l}\})\cup\{u^{'}\}}}
        \sgn
            {1^{'},\cdots,n^{'},1,\cdots,n}
            {\indexIc,\order{(\indexJpc\backslash\{\indexjpc{l}\})\cup\{u^{'}\}},\indexI,\order{(\indexJp\backslash\{u^{'}\})\cup\{\indexjpc{l}\}}}
        \\
        \coefB{\indexjpc{l}}{u^{'} \alpha}
        \basexi{\indexI}
        \wedge
        \basedualxi{\order{(\indexJp\backslash\{u^{'}\})\cup\{\indexjpc{l}\}}}
    \end{pmatrix*}
\label{EQ-For (alpha, I, J), calculation for RHS-Step 3, unsimplfied term 7}
    \\
    &+
    \notag\\
    &
    \begin{pmatrix*}[l]
        -2\sqrt{-1}(-1)^{n^2}
        \displaystyle\sum_{\indexjpc{l}\in\indexJpc}
        \displaystyle\sum_{r\in\indexI}
        \displaystyle\sum_{u^{'}\in\indexJp}
        \sgn
            {1,\cdots,n,1^{'},\cdots,n^{'}}
            {\alpha,\indexI,\indexJp,\indexIc\backslash\{\alpha\},\indexJc}
        \sgn
            {\indexIc\backslash\{\alpha\},\indexJpc}
            {\indexjpc{l},\indexIc\backslash\{\alpha\},\indexJpc\backslash\{\indexjpc{l}\}}
        \\
        \sgn
            {\indexIc\backslash\{\alpha\},\indexjpc{1},\cdots,\indexjpc{l-1},u^{'},r,\indexjpc{l+1},\cdots,\indexjpc{n-q}}
            {\order{(\indexIc\backslash\{\alpha\})\cup\{r\}},\order{(\indexJpc\backslash\{\indexjpc{l}\})\cup\{u^{'}\}}}
        \\
        \sgn
            {1^{'},\cdots,n^{'},1,\cdots,n}
            {\order{(\indexIc\backslash\{\alpha\})\cup\{r\}},\order{(\indexJpc\backslash\{\indexjpc{l}\})\cup\{u^{'}\}},\alpha,\indexI\backslash\{r\},\order{(\indexJp\backslash\{u^{'}\})\cup\{\indexjpc{l}\}}}
        \\
        \coefB{\indexjpc{l}}{u^{'} r}
        \basexi{\alpha}
        \wedge
        \basexi{\indexI\backslash\{r\}}
        \wedge
        \basedualxi{\order{(\indexJp\backslash\{u^{'}\})\cup\{\indexjpc{l}\}}}
    \end{pmatrix*}.
\label{EQ-For (alpha, I, J), calculation for RHS-Step 3, unsimplfied term 8}
\end{align}
We simplify the sign of permutations in Section \ref{apex-simplification for RHs of (alpha,I,J)}. The reader can check Lemmas \ref{Lemma-simplification for RHS of (alpha,I,J)-lemma 1}–\ref{Lemma-simplification for RHS of (alpha,I,J)-lemma 8} for details. Thus, we conclude that the simplified expression of $
\sqrt{-1}\LieagbdPartial^{*}
(\basedualxi{\indexIc}\wedge\basexi{\indexJpc})
$ is
\begin{align}
    &
    -2\sqrt{-1}(-1)^{n^2}
    \displaystyle\sum_{\indexic{k}\in\indexIc\backslash\{\alpha\}}
    \begin{matrix*}[l]
        (-1)^{n^2}
        \sgn
            {\order{\indexic{k},\alpha}}
            {\alpha,\indexic{k}}
        \coefD{\indexic{k}}{\order{\indexic{k}, \alpha}}
        \basexi{\indexI}
        \wedge
        \basedualxi{\indexJp}
    \end{matrix*}
\label{EQ-Kahler identity on Lie agbd-For (alpha,I,J), calculation for RHS-simplified half target no1}
    \\
    &+
    \notag\\
    &
    \begin{pmatrix*}[l]
    -2\sqrt{-1}(-1)^{n^2}
    \displaystyle\sum_{\indexic{k}\in\indexIc\backslash\{\alpha\}}
    \displaystyle\sum_{r\in\indexI}
        -(-1)^{n^2}
        \begin{pmatrix*}[l]
            \sgn
                {r,\indexI\backslash\{r\}}
                {\indexI}
            \sgn
                {\order{\indexic{k},r}}
                {r,\indexic{k}}
        \end{pmatrix*}
        \coefD{\indexic{k}}{\order{\indexic{k}, r}}
        \\
        \basexi{\alpha}
        \wedge
        \basexi{\indexI\backslash\{r\}}
        \wedge
        \basedualxi{\indexJp}
    \end{pmatrix*}
\label{EQ-Kahler identity on Lie agbd-For (alpha,I,J), calculation for RHS-simplified half target no2}
    \\
    &+
    \notag\\
    &
    \begin{pmatrix*}[l]
    -2\sqrt{-1}(-1)^{n^2}
    \displaystyle\sum_{\indexic{k}\in\indexIc\backslash\{\alpha\}}
    \displaystyle\sum_{r\in\indexI}
        (-1)^{n^2}
        \sgn
            {r,\order{(\indexI\backslash\{r\})\cup\{\indexic{k}\}}}
            {\indexic{k},\indexI}
        \sgn
            {\order{\alpha,r}}
            {\alpha,r}
        \\
        \coefD{\indexic{k}}{\order{\alpha, r}}
        \basexi{\order{(\indexI\backslash\{r\})\cup\{\indexic{k}\}}}
        \wedge
        \basedualxi{\indexJp}
    \end{pmatrix*}
\label{EQ-Kahler identity on Lie agbd-For (alpha,I,J), calculation for RHS-simplified half target no3}
    \\
    &+
    \notag\\
    &
    \begin{pmatrix*}[l]
    -2\sqrt{-1}(-1)^{n^2}
    \displaystyle\sum_{\indexic{k}\in\indexIc\backslash\{\alpha\}}
    \displaystyle\sum_{r_1\in\indexI}
    \displaystyle\sum_{r_2\in\indexI}
        (-1)^{n^2}
        \sgn
            {\order{r_1,r_2},\indexI\backslash\{r_1,r_2\}}
            {\indexI}
        \sgn
            {\order{(\indexI\backslash\{r_1,r_2\})\cup\{\indexic{k}\}}}
            {\indexic{k},(\indexI\backslash\{r_1,r_2\})}
        \\
        \frac{1}{2}
        \coefD{\indexic{k}}{\order{r_1,r_2}}
        \basexi{\alpha}
        \wedge
        \basexi{\order{(\indexI\backslash\{r_1,r_2\})\cup\{\indexic{k}\}}}
        \wedge
        \basedualxi{\indexJp}
    \end{pmatrix*}
\label{EQ-Kahler identity on Lie agbd-For (alpha,I,J), calculation for RHS-simplified half target no4}
    \\
    &+
    \notag\\
    &
    -2\sqrt{-1}(-1)^{n^2}
    \displaystyle\sum_{\indexjpc{l}\in\indexJpc}
    \displaystyle\sum_{r\in\indexI}
    \begin{matrix*}[l]
        -(-1)^{n^2}
        \coefB{\indexjpc{l}}{\indexjpc{l} \alpha}
        \basexi{\indexI}
        \wedge
        \basedualxi{\indexJp}
    \end{matrix*}
\label{EQ-Kahler identity on Lie agbd-For (alpha,I,J), calculation for RHS-simplified half target no5}
    \\
    &+
    \notag\\
    &
    -2\sqrt{-1}(-1)^{n^2}
    \displaystyle\sum_{\indexjpc{l}\in\indexJpc}
    \displaystyle\sum_{r\in\indexI}
    \begin{matrix*}[l]
        (-1)^{n^2}
        \sgn
            {r,\indexI\backslash\{r\}}
            {\indexI}
        \coefB{\indexjpc{l}}{\indexjpc{l} r}
        \basexi{\alpha}
        \wedge
        \basexi{\indexI\backslash\{r\}}
        \wedge
        \basedualxi{\indexJp}
    \end{matrix*}
\label{EQ-Kahler identity on Lie agbd-For (alpha,I,J), calculation for RHS-simplified half target no6}
    \\
    &+
    \notag\\
    &
    -2\sqrt{-1}(-1)^{n^2}
    \displaystyle\sum_{\indexjpc{l}\in\indexJpc}
    \displaystyle\sum_{u^{'}\in\indexJp}
    \begin{matrix*}[l]
        -(-1)^{n^2}
        \sgn
            {u^{'},\indexI,\order{(\indexJp\backslash\{u^{'}\})\cup\{\indexjpc{l}\}}}
            {\indexjpc{l},\indexI,\indexJp}
        \coefB{\indexjpc{l}}{u^{'} \alpha}
        \basexi{\indexI}
        \wedge
        \basedualxi{\order{(\indexJp\backslash\{u^{'}\})\cup\{\indexjpc{l}\}}}
    \end{matrix*}
\label{EQ-Kahler identity on Lie agbd-For (alpha,I,J), calculation for RHS-simplified half target no7}
    \\
    &+
    \notag\\
    &
    \begin{pmatrix*}[l]
        -2\sqrt{-1}(-1)^{n^2}
        \displaystyle\sum_{\indexjpc{l}\in\indexJpc}
        \displaystyle\sum_{r\in\indexI}
        \displaystyle\sum_{u^{'}\in\indexJp}
        (-1)^{p+1}
        (-1)^{n^2}
        \sgn
            {u^{'},r,\indexI\backslash\{r\},\indexJp\backslash\{u^{'}\}}
            {\indexI,\indexJp}
        \sgn
            {\order{(\indexJp\backslash\{u^{'}\})\cup\{\indexjpc{l}\}}}
            {\indexjpc{l},\indexJp\backslash\{u^{'}\}}
        \\
        \coefB{\indexjpc{l}}{u^{'} r}
        \basexi{\alpha}
        \wedge
        \basexi{\indexI\backslash\{r\}}
        \wedge
        \basedualxi{\order{(\indexJp\backslash\{u^{'}\})\cup\{\indexjpc{l}\}}}
    \end{pmatrix*}
    .
\label{EQ-Kahler identity on Lie agbd-For (alpha,I,J), calculation for RHS-simplified half target no8}
\end{align}
\keyword{Step 4:}
Cancel out $(-1)^{n^2}$ and rearrange the expression according to the expressions of $
\sqrt{-1}(\LieagbdPartial^{*}\basexi{\alpha})
\wedge\basexi{\indexI}\wedge\basedualxi{\indexJp}
$ and $
\basexi{\alpha}\wedge
\sqrt{-1}\LieagbdPartial^{*}(\basexi{\indexI}\wedge\basedualxi{\indexJp})
$. The expression of $
\sqrt{-1}
\LieagbdPartial^{*}(
\basexi{\alpha}\wedge
\basexi{\indexI}\wedge\basedualxi{\indexJp})
$ is a sum of three parts.
\begin{itemize}
    \item The first part is (\ref{EQ-Kahler identity on Lie agbd-For (alpha,I,J), calculation for RHS-simplified half target no1}) + (\ref{EQ-Kahler identity on Lie agbd-For (alpha,I,J), calculation for RHS-simplified half target no5}), whose expression is
    \begin{align}
        -2\sqrt{-1}
        \begin{pmatrix*}[l]
            \displaystyle\sum_{\indexic{k}\in\indexIc\backslash\{\alpha\}}
            \begin{matrix*}[l]
                \sgn
                    {\order{\indexic{k},\alpha}}
                    {\alpha,\indexic{k}}
                \coefD{\indexic{k}}{\order{\indexic{k}, \alpha}}
                \basexi{\indexI}
                \wedge
                \basedualxi{\indexJp}
            \end{matrix*}
            \\
            +
            \\
            \displaystyle\sum_{\indexjpc{l}\in\indexJpc}
            \begin{matrix*}[l]
                (-1)
                \coefB{\indexjpc{l}}{\indexjpc{l} \alpha}
                \basexi{\indexI}
                \wedge
                \basedualxi{\indexJp}
            \end{matrix*}
        \end{pmatrix*}
        .
    \label{EQ-Kahler identity on Lie agbd-For (alpha,I,J), the first part}
    \end{align}
    This part is related to the expression of $
    \sqrt{-1}(\LieagbdPartial^{*}\basexi{\alpha})
    \wedge\basexi{\indexI}\wedge\basedualxi{\indexJp}
    $.

    \item The second part is
    (\ref{EQ-Kahler identity on Lie agbd-For (alpha,I,J), calculation for RHS-simplified half target no2})
    +
    (\ref{EQ-Kahler identity on Lie agbd-For (alpha,I,J), calculation for RHS-simplified half target no4})
    +
    (\ref{EQ-Kahler identity on Lie agbd-For (alpha,I,J), calculation for RHS-simplified half target no6})
    +
    (\ref{EQ-Kahler identity on Lie agbd-For (alpha,I,J), calculation for RHS-simplified half target no8}), which is equal to the sum of the following terms:
    \begin{align}
        -2\sqrt{-1}
        \begin{pmatrix*}[l]
            \displaystyle\sum_{\indexic{k}\in\indexIc\backslash\{\alpha\}}
            \displaystyle\sum_{r\in\indexI}
            \begin{matrix*}[l]
                (-1)
                    \sgn
                        {r,\indexI\backslash\{r\}}
                        {\indexI}
                    \sgn
                        {\order{\indexic{k},r}}
                        {r,\indexic{k}}
            \\
                \coefD{\indexic{k}}{\order{\indexic{k}, r}}
                \basexi{\alpha}
                \wedge
                \basexi{\indexI\backslash\{r\}}
                \wedge
                \basedualxi{\indexJp}
            \end{matrix*}
            \\
            +
            \\
            \displaystyle\sum_{\indexic{k}\in\indexIc\backslash\{\alpha\}}
            \displaystyle\sum_{r_1\in\indexI}
            \displaystyle\sum_{r_2\in\indexI}
            \begin{matrix*}[l]
                \sgn
                    {\order{r_1,r_2},\indexI\backslash\{r_1,r_2\}}
                    {\indexI}
                \sgn
                    {\order{(\indexI\backslash\{r_1,r_2\})\cup\{\indexic{k}\}}}
                    {\indexic{k},(\indexI\backslash\{r_1,r_2\})}
            \\
                \frac{1}{2}
                \coefD{\indexic{k}}{\order{r_1,r_2}}
                \basexi{\alpha}
                \wedge
                \basexi{\order{(\indexI\backslash\{r_1,r_2\})\cup\{\indexic{k}\}}}
                \wedge
                \basedualxi{\indexJp}
            \end{matrix*}
            \\
            +
            \\
            \displaystyle\sum_{\indexjpc{l}\in\indexJpc}
            \displaystyle\sum_{r\in\indexI}
            \begin{matrix*}[l]
                \sgn
                    {r,\indexI\backslash\{r\}}
                    {\indexI}
                \coefB{\indexjpc{l}}{\indexjpc{l} r}
                \basexi{\alpha}
                \wedge
                \basexi{\indexI\backslash\{r\}}
                \wedge
                \basedualxi{\indexJp}
        \end{matrix*}
            \\
            +
            \\
            \displaystyle\sum_{\indexjpc{l}\in\indexJpc}
            \displaystyle\sum_{r\in\indexI}
            \displaystyle\sum_{u^{'}\in\indexJp}
            \begin{matrix*}[l]    
                (-1)^{p+1}
                \sgn
                    {u^{'},r,\indexI\backslash\{r\},\indexJp\backslash\{u^{'}\}}
                    {\indexI,\indexJp}
                \sgn
                    {\order{(\indexJp\backslash\{u^{'}\})\cup\{\indexjpc{l}\}}}
                    {\indexjpc{l},\indexJp\backslash\{u^{'}\}}
            \\
                \coefB{\indexjpc{l}}{u^{'} r}
                \basexi{\alpha}
                \wedge
                \basexi{\indexI\backslash\{r\}}
                \wedge
                \basedualxi{\order{(\indexJp\backslash\{u^{'}\})\cup\{\indexjpc{l}\}}}
            \end{matrix*}
        \end{pmatrix*}
        .
    \label{EQ-Kahler identity on Lie agbd-For (alpha,I,J), the second part}
    \end{align}
    This part is related to $
    \basexi{\alpha}\wedge
    \sqrt{-1}\LieagbdPartial^{*}(\basexi{\indexI}\wedge\basedualxi{\indexJp})
    $.

    \item The third part is 
    (\ref{EQ-Kahler identity on Lie agbd-For (alpha,I,J), calculation for RHS-simplified half target no3})
    +
    (\ref{EQ-Kahler identity on Lie agbd-For (alpha,I,J), calculation for RHS-simplified half target no7}), which is equal to
    \begin{align}
        (-2\sqrt{-1})
        \begin{pmatrix*}[l]
            \displaystyle\sum_{\indexic{k}\in\indexIc\backslash\{\alpha\}}
            \displaystyle\sum_{r\in\indexI}
            \begin{matrix*}[l]
                \sgn
                    {r,\order{(\indexI\backslash\{r\})\cup\{\indexic{k}\}}}
                    {\indexic{k},\indexI}
                \sgn
                    {\order{\alpha,r}}
                    {\alpha,r}
                \\
                \coefD{\indexic{k}}{\order{\alpha, r}}
                \basexi{\order{(\indexI\backslash\{r\})\cup\{\indexic{k}\}}}
                \wedge
                \basedualxi{\indexJp}
            \end{matrix*}
            \\
            +
            \\   
            \displaystyle\sum_{\indexjpc{l}\in\indexJpc}
            \displaystyle\sum_{u^{'}\in\indexJp}
            \begin{matrix*}[l]
                (-1)
                \sgn
                    {u^{'},\indexI,\order{(\indexJp\backslash\{u^{'}\})\cup\{\indexjpc{l}\}}}
                    {\indexjpc{l},\indexI,\indexJp}
                \\
                \coefB{\indexjpc{l}}{u^{'} \alpha}
                \basexi{\indexI}
                \wedge
                \basedualxi{\order{(\indexJp\backslash\{u^{'}\})\cup\{\indexjpc{l}\}}}
            \end{matrix*}
        \end{pmatrix*}
        .
    \label{EQ-Kahler identity on Lie agbd-For (alpha,I,J), the third part}
    \end{align}
\end{itemize}

\end{proof}

\newpage
\subsection{Proof of \texorpdfstring{
    $
    [\LieagbdPartialbar,\iotaW]
    (\basexi{\alpha}\wedge\basexi{\indexI}\wedge\basedualxi{\indexJp})
    =
    \sqrt{-1}\LieagbdPartial^{*}
    (\basexi{\alpha}\wedge\basexi{\indexI}\wedge\basedualxi{\indexJp})
    $}{}
}
\label{subsection-prove LHS is equal to RHS}
\begin{proof}
Let us look at the right-hand side first. Recall that $
\sqrt{-1}\LieagbdPartial^{*}
(\basexi{\alpha}\wedge\basexi{\indexI}\wedge\basedualxi{\indexJp})
$ is a sum of three parts: \keyword{part I} is (\ref{EQ-Kahler identity on Lie agbd-For (alpha,I,J), the first part}), \keyword{part II} is (\ref{EQ-Kahler identity on Lie agbd-For (alpha,I,J), the second part}), and \keyword{part III} is  (\ref{EQ-Kahler identity on Lie agbd-For (alpha,I,J), the third part}).

The expression of \keyword{part I} is
\begin{align*}
    -2\sqrt{-1}
    \begin{pmatrix*}[l]
        \displaystyle\sum_{\indexic{k}\in\indexIc\backslash\{\alpha\}}
        \begin{matrix*}[l]
            \sgn
                {\order{\indexic{k},\alpha}}
                {\alpha,\indexic{k}}
            \coefD{\indexic{k}}{\order{\indexic{k}, \alpha}}
            \basexi{\indexI}
            \wedge
            \basedualxi{\indexJp}
        \end{matrix*}
        \\
        +
        \\
        \displaystyle\sum_{\indexjpc{l}\in\indexJpc}
        \begin{matrix*}[l]
            (-1)
            \coefB{\indexjpc{l}}{\indexjpc{l} \alpha}
            \basexi{\indexI}
            \wedge
            \basedualxi{\indexJp}
        \end{matrix*}
    \end{pmatrix*}
    .
\end{align*}
As a result of (\ref{EQ-Kahler identity on Lie agbd-For (alpha),I,J, calculation for RHS-target}) in the statement of Lemma \ref{lemma-Kahler identity on Lie agbd-For (alpha),I,J, calculation for RHS}, the expression of $
\sqrt{-1}(\LieagbdPartial^{*}\basexi{\alpha})
\wedge\basexi{\indexI}\wedge\basedualxi{\indexJp}
$ is
\begin{align}                                                       
    -2\sqrt{-1}
    \begin{pmatrix*}[l]
        \displaystyle\sum_{k\neq\alpha}
        \begin{matrix*}[l]
            \sgn
                {\order{k,\alpha}}
                {\alpha,k}
            \coefD{k}{\order{k,\alpha}}
            \basexi{\indexI}\wedge\basedualxi{\indexJp}
        \end{matrix*}
        \\
        +
        \\
        \displaystyle\sum_{l^{'}=1}^{n}
        \begin{matrix*}[l]
            (-1)\coefB{l^{'}}{l^{'} \alpha}
            \basexi{\indexI}\wedge\basedualxi{\indexJp}
        \end{matrix*}
    \end{pmatrix*}
    .
\label{EQ-Kahler identity on Lie agbd-For (alpha,I,J), aim of the first part}
\end{align}
Thus, we can split (\ref{EQ-Kahler identity on Lie agbd-For (alpha,I,J), aim of the first part}) into a sum of two parts:
\begin{align}
    &
        -2\sqrt{-1}
        \begin{pmatrix*}[l]
            \displaystyle\sum_{\indexic{k}\in\indexIc\backslash\{\alpha\}}
            \begin{matrix*}[l]
                \sgn
                    {\order{\indexic{k},\alpha}}
                    {\alpha,\indexic{k}}
                \coefD{\indexic{k}}{\order{\indexic{k}, \alpha}}
                \basexi{\indexI}
                \wedge
                \basedualxi{\indexJp}
            \end{matrix*}
            \\
            +
            \\
            \displaystyle\sum_{\indexjpc{l}\in\indexJpc}
            \begin{matrix*}[l]
                (-1)
                \coefB{\indexjpc{l}}{\indexjpc{l} \alpha}
                \basexi{\indexI}
                \wedge
                \basedualxi{\indexJp}
            \end{matrix*}
        \end{pmatrix*}
\label{EQ-prove LHS is equal to RHS-split the first part, term 1}
    \\
    +&\notag
    \\
    &
        -2\sqrt{-1}
        \begin{pmatrix*}[l]
            \displaystyle\sum_{k\in I}
            \begin{matrix*}[l]
                \sgn
                    {\order{k,\alpha}}
                    {\alpha,k}
                \coefD{k}{\order{k,\alpha}}
                \basexi{\indexI}\wedge\basedualxi{\indexJp}
            \end{matrix*}
            \\
            +
            \\
            \displaystyle\sum_{l\in\indexJp}
            \begin{matrix*}[l]
                (-1)\coefB{l^{'}}{l^{'} \alpha}
                \basexi{\indexI}\wedge\basedualxi{\indexJp}
            \end{matrix*}
        \end{pmatrix*}.
\notag
\end{align}
The term (\ref{EQ-prove LHS is equal to RHS-split the first part, term 1}) is just the expression of $
(\sqrt{-1}\partial^{*}\basexi{\alpha})
\wedge
(\basexi{\indexI}\wedge\basedualxi{\indexJp})
$. Therefore, we can conclude that \keyword{part I} is equal to
\begin{align}
    &
    \begin{pmatrix*}[l]
        \sqrt{-1}
        (\partial^{*}\basexi{\alpha})
        \wedge
        (\basexi{\indexI}\wedge\basedualxi{\indexJp})
        \\
        +
        \\
        2\sqrt{-1}
        \begin{pmatrix*}[l]
            \displaystyle\sum_{k\in I}
            \begin{matrix*}[l]
                \sgn
                    {\order{k,\alpha}}
                    {\alpha,k}
                \coefD{k}{\order{k,\alpha}}
                \basexi{\indexI}\wedge\basedualxi{\indexJp}
            \end{matrix*}
            \\
            +
            \\
            \displaystyle\sum_{l\in\indexJp}
            \begin{matrix*}[l]
                (-1)\coefB{l^{'}}{l^{'} \alpha}
                \basexi{\indexI}\wedge\basedualxi{\indexJp}
            \end{matrix*}
        \end{pmatrix*}
    \end{pmatrix*}
\notag
    \\
    =&
    \begin{pmatrix*}[l]
        \sqrt{-1}
        (\partial^{*}\basexi{\alpha})
        \wedge
        (\basexi{\indexI}\wedge\basedualxi{\indexJp})
        \\
        +
        \\
        (-2\sqrt{-1})
        \begin{pmatrix*}[l]
            \displaystyle\sum_{k\in I}
            \begin{matrix*}[l]
                \sgn
                    {\order{k,\alpha}}
                    {k,\alpha}
                \coefD{k}{\order{k,\alpha}}
                \basexi{\indexI}\wedge\basedualxi{\indexJp}
            \end{matrix*}
            \\
            +
            \\
            \displaystyle\sum_{l\in\indexJp}
            \begin{matrix*}[l]
                \coefB{l^{'}}{l^{'} \alpha}
                \basexi{\indexI}\wedge\basedualxi{\indexJp}
            \end{matrix*}
        \end{pmatrix*}
    \end{pmatrix*},
\label{EQ-prove LHS is equal to RHS, result of the first part}
\end{align}
where we use $\sgn{\order{k,\alpha}}{\alpha,k}=-\sgn{\order{k,\alpha}}{k,\alpha}$.

The expression of \keyword{part II} is
\begin{align}
    -2\sqrt{-1}
    \begin{pmatrix*}[l]
        \displaystyle\sum_{\indexic{k}\in\indexIc\backslash\{\alpha\}}
        \displaystyle\sum_{r\in\indexI}
        \begin{matrix*}[l]
            (-1)
                \sgn
                    {r,\indexI\backslash\{r\}}
                    {\indexI}
                \sgn
                    {\order{\indexic{k},r}}
                    {r,\indexic{k}}
            \coefD{\indexic{k}}{\order{\indexic{k}, r}}
            \basexi{\alpha}
            \wedge
            \basexi{\indexI\backslash\{r\}}
            \wedge
            \basedualxi{\indexJp}
        \end{matrix*}
        \\
        +
        \\
        \displaystyle\sum_{\indexic{k}\in\indexIc\backslash\{\alpha\}}
        \displaystyle\sum_{r_1\in\indexI}
        \displaystyle\sum_{r_2\in\indexI}
        \begin{matrix*}[l]
            \begin{pmatrix*}[l]
                \sgn
                    {\order{r_1,r_2},\indexI\backslash\{r_1,r_2\}}
                    {\indexI}
                \\
                \sgn
                    {\order{(\indexI\backslash\{r_1,r_2\})\cup\{\indexic{k}\}}}
                    {\indexic{k},(\indexI\backslash\{r_1,r_2\})}
            \\
                \frac{1}{2}
                \coefD{\indexic{k}}{\order{r_1,r_2}}
                \basexi{\alpha}
                \wedge
                \basexi{\order{(\indexI\backslash\{r_1,r_2\})\cup\{\indexic{k}\}}}
                \wedge
                \basedualxi{\indexJp}
            \end{pmatrix*}
        \end{matrix*}
        \\
        +
        \\
        \displaystyle\sum_{\indexjpc{l}\in\indexJpc}
        \displaystyle\sum_{r\in\indexI}
        \begin{matrix*}[l]
            \sgn
                {r,\indexI\backslash\{r\}}
                {\indexI}
            \coefB{\indexjpc{l}}{\indexjpc{l} r}
            \basexi{\alpha}
            \wedge
            \basexi{\indexI\backslash\{r\}}
            \wedge
            \basedualxi{\indexJp}
    \end{matrix*}
        \\
        +
        \\
        \displaystyle\sum_{\indexjpc{l}\in\indexJpc}
        \displaystyle\sum_{r\in\indexI}
        \displaystyle\sum_{u^{'}\in\indexJp}
        \begin{matrix*}[l]    
            (-1)^{p+1}
            \begin{pmatrix*}[l]
                \sgn
                    {u^{'},r,\indexI\backslash\{r\},\indexJp\backslash\{u^{'}\}}
                    {\indexI,\indexJp}
                \\
                \sgn
                    {\order{(\indexJp\backslash\{u^{'}\})\cup\{\indexjpc{l}\}}}
                    {\indexjpc{l},\indexJp\backslash\{u^{'}\}}
                \\
                \coefB{\indexjpc{l}}{u^{'} r}
                \basexi{\alpha}
                \wedge
                \basexi{\indexI\backslash\{r\}}
                \wedge
                \basedualxi{\order{(\indexJp\backslash\{u^{'}\})\cup\{\indexjpc{l}\}}}
            \end{pmatrix*}
        \end{matrix*}
    \end{pmatrix*}
    .
\label{EQ-prove LHS is equal to RHS, the second part}
\end{align}
Recall (\ref{EQ-Kahler identity on Lie agbd-For alpha,(I,J), calculation for RHS-target}) in Lemma \ref{lemma-Kahler identity on Lie agbd-For alpha,(I,J), calculation for RHS}, which shows that $
\basexi{\alpha}\wedge
\sqrt{-1}\LieagbdPartial^{*}(\basexi{\indexI}\wedge\basedualxi{\indexJp})
$ is
\begin{align}
    -2\sqrt{-1}
    \begin{pmatrix*}[l]
        \displaystyle\sum_{\indexic{k}\in\indexIc}
        \displaystyle\sum_{r\in\indexI}
        \begin{matrix*}[l]
            \sgn
                {r,\indexic{k}}
                {\order{\indexic{k},r}}
            \sgn
                {r,\indexI\backslash\{r\}}                    
                {\indexI}
            \coefD{\indexic{k}}{\order{\indexic{k},r}}
            \basexi{\alpha}
            \wedge
            \basexi{\indexI\backslash\{r\}}
            \wedge
            \basedualxi{\indexJp} 
        \end{matrix*}
        \\
        +
        \\
        \displaystyle\sum_{\indexic{k}\in\indexIc\backslash\{\alpha\}}
        \displaystyle\sum_{r_1\in\indexI}
        \displaystyle\sum_{r_2\in\indexI}
        \begin{pmatrix*}[l]
            (-1)
            \sgn
                {\order{r_1,r_2},\order{(\indexI\backslash\{r_1,r_2\})}}
                {\indexI}
            \\
            \sgn
                {\order{(\indexI\backslash\{r_1,r_2\})\cup\{\indexic{k}\}}}
                {\indexic{k},\indexI\backslash\{r_1,r_2\})}
            \\
            \frac{1}{2}
            \coefD{\indexic{k}}{\order{r_1,r_2}}
            \basexi{\alpha}
            \wedge
            \basexi{\order{(\indexI\backslash\{r_1,r_2\})\cup\{\indexic{k}\}}}
            \wedge
            \basedualxi{\indexJp} 
        \end{pmatrix*}
        \\
        +
        \\
        \displaystyle\sum_{\indexjpc{l}\in\indexJpc}
        \displaystyle\sum_{r\in\indexI}
        \begin{matrix*}[l]
            (-1)
            \sgn
                {r,\indexI\backslash\{r\}}
                {\indexI}
            \coefB{\indexjpc{l}}{\indexjpc{l} r}
            \basexi{\alpha}
            \wedge
            \basexi{\indexI\backslash\{r\}}
            \wedge
            \basedualxi{\indexJp} 
        \end{matrix*}
        \\
        +
        \\
        \displaystyle\sum_{\indexjpc{l}\in\indexJpc}
        \displaystyle\sum_{r\in\indexI}
        \displaystyle\sum_{u^{'}\in\indexJp}
        \begin{pmatrix*}[l]
            (-1)^{p}
            \sgn
                {u^{'},r,\indexI\backslash\{r\},\order{(\indexJp\backslash\{u^{'}\})}}
                {\indexI,\indexJp}
            \\
            \sgn
                {\order{(\indexJp\backslash\{u^{'}\})\cup\{\indexjpc{l}\}}}
                {\indexjpc{l},(\indexJp\backslash\{u^{'}\})}
            \\
            \coefB{\indexjpc{l}}{u^{'} r}
            \basexi{\alpha}
            \wedge
            \basexi{\indexI\backslash\{r\}}
            \wedge
            \basedualxi{\order{(\indexJp\backslash\{u^{'}\})\cup\{\indexjpc{l}\}}} 
        \end{pmatrix*}
    \end{pmatrix*}
    .
\label{EQ-prove LHS is equal to RHS, aim of the second part}
\end{align}
Comparing (\ref{EQ-prove LHS is equal to RHS, aim of the second part}) with (\ref{EQ-prove LHS is equal to RHS, the second part}), the two formulas are almost the same except for a minus sign difference and an extra term in the first summation of (\ref{EQ-prove LHS is equal to RHS, aim of the second part}) such that $\indexic{k}=\alpha$. We can conclude that \keyword{part II} is equal to
\begin{align}
    \begin{pmatrix*}[l]
        -
        \basexi{\alpha}
        \wedge
        \sqrt{-1}\partial^{*}(\basexi{\indexI}\wedge\basedualxi{\indexJp})
        \\
        +
        \\
        (-2\sqrt{-1})
        \displaystyle\sum_{r\in\indexI}
        \begin{matrix*}[l]
            \sgn
                {r,\indexI\backslash\{r\}}                    
                {\indexI}
            \sgn
                {r,\alpha}
                {\order{\alpha,r}}            
            \coefD{\alpha}{\order{\alpha,r}}
            \basexi{\alpha}
            \wedge
            \basexi{\indexI\backslash\{r\}}
            \wedge
            \basedualxi{\indexJp} 
        \end{matrix*}
    \end{pmatrix*}.
\label{EQ-prove LHS is equal to RHS, result of the second part}
\end{align}

Now, let us sum up \keyword{part I}, \keyword{part II}, and \keyword{part III}:
\begin{align*}
    &
    \begin{pmatrix*}[l]
        \sqrt{-1}
        (\partial^{*}\basexi{\alpha})
        \wedge
        (\basexi{\indexI}\wedge\basedualxi{\indexJp})
        \\
        +
        \\
        (-2\sqrt{-1})
        \begin{pmatrix*}[l]
            \displaystyle\sum_{k\in I}
            \begin{matrix*}[l]
                \sgn
                    {\order{k,\alpha}}
                    {k,\alpha}
                \coefD{k}{\order{k,\alpha}}
                \basexi{\indexI}\wedge\basedualxi{\indexJp}
            \end{matrix*}
            \\
            +
            \\
            \displaystyle\sum_{l\in\indexJp}
            \begin{matrix*}[l]
                \coefB{l^{'}}{l^{'} \alpha}
                \basexi{\indexI}\wedge\basedualxi{\indexJp}
            \end{matrix*}
        \end{pmatrix*}
    \end{pmatrix*}
    \\
    +&
    \\
    &
    \begin{pmatrix*}[l]
        -
        \basexi{\alpha}
        \wedge
        \sqrt{-1}\partial^{*}(\basexi{\indexI}\wedge\basedualxi{\indexJp})
        \\
        +
        \\
        (-2\sqrt{-1})
        \displaystyle\sum_{r\in\indexI}
        \begin{matrix*}[l]
            \sgn
                {r,\indexI\backslash\{r\}}                    
                {\indexI}
            \sgn
                {r,\alpha}
                {\order{\alpha,r}}            
            \coefD{\alpha}{\order{\alpha,r}}
            \basexi{\alpha}
            \wedge
            \basexi{\indexI\backslash\{r\}}
            \wedge
            \basedualxi{\indexJp} 
        \end{matrix*}
    \end{pmatrix*}
    \\
    +&
    \\
    &
    (-2\sqrt{-1})
        \begin{pmatrix*}[l]
            \displaystyle\sum_{\indexic{k}\in\indexIc\backslash\{\alpha\}}
            \displaystyle\sum_{r\in\indexI}
            \begin{pmatrix*}[l]
                \sgn
                    {r,\order{(\indexI\backslash\{r\})\cup\{\indexic{k}\}}}
                    {\indexic{k},\indexI}
                \sgn
                    {\order{\alpha,r}}
                    {\alpha,r}
                \\
                \coefD{\indexic{k}}{\order{r,\alpha}}
                \basexi{\order{(\indexI\backslash\{r\})\cup\{\indexic{k}\}}}
                \wedge
                \basedualxi{\indexJp}
            \end{pmatrix*}
            \\
            +
            \\   
            \displaystyle\sum_{\indexjpc{l}\in\indexJpc}
            \displaystyle\sum_{u^{'}\in\indexJp}
            \begin{pmatrix*}[l]
                (-1)
                \sgn
                    {u^{'},\indexI,\order{(\indexJp\backslash\{u^{'}\})\cup\{\indexjpc{l}\}}}
                    {\indexjpc{l},\indexI,\indexJp}
                \\
                \coefB{\indexjpc{l}}{u^{'} \alpha}
                \basexi{\indexI}
                \wedge
                \basedualxi{\order{(\indexJp\backslash\{u^{'}\})\cup\{\indexjpc{l}\}}}
            \end{pmatrix*}
        \end{pmatrix*}
    .
\end{align*}
Rearranging the previous sum, we obtain the expression of the right-hand side:
\begin{align*}
    \begin{pmatrix*}[l]
        (\sqrt{-1}\partial^{*}\basexi{\alpha})
        \wedge
        \basexi{\indexI}\wedge\basedualxi{\indexJp}
        -
        \basexi{\alpha}
        \wedge
        \sqrt{-1}\partial^{*}
        (\basexi{\indexI}\wedge\basedualxi{\indexJp})
        \\
        +
        \\
        (-2\sqrt{-1})
        \begin{pmatrix*}[l]
            \displaystyle\sum_{k\in\indexI}
            \begin{matrix*}[l]
                \sgn
                    {\order{k,\alpha}}
                    {k,\alpha}
                \coefD{k}{\order{k,\alpha}}
                \basexi{\indexI}
                \wedge
                \basedualxi{\indexJp}
            \end{matrix*}
            \\
            +
            \\
            \displaystyle\sum_{l^{'}\in\indexJp}
            \begin{matrix*}[l]
                \coefB{l^{'}}{l^{'} \alpha}
                \basexi{\indexI}
                \wedge
                \basedualxi{\indexJp}
            \end{matrix*}
            \\
            +
            \\
            \displaystyle\sum_{r\in\indexI}
            \begin{matrix*}[l]
                \sgn
                    {r,\alpha}
                    {\order{\alpha,r}}
                \sgn
                    {r,\indexI\backslash\{r\}}
                    {\indexI}
                \coefD{\alpha}{\order{\alpha,r}}
                \basexi{\alpha}
                \wedge
                \basexi{\indexI\backslash\{r\}}
                \wedge
                \basedualxi{\indexJp}
            \end{matrix*}
            \\
            +
            \\
            \displaystyle\sum_{\indexic{k}\in\indexIc\backslash\{\alpha\}}
            \displaystyle\sum_{r\in\indexI}
            \begin{matrix*}[l]
                \sgn
                    {r,\order{(\indexI\backslash\{r\})\cup\{\indexic{k}\}}}
                    {\indexic{k},\indexI}
                \sgn
                    {\order{\alpha,r}}
                    {\alpha,r}
                \\
                \coefD{\indexic{k}}{\order{\alpha, r}}
                \basexi{\order{(\indexI\backslash\{r\})\cup\{\indexic{k}\}}}
                \wedge 
                \basedualxi{\indexJp}
            \end{matrix*}
            \\
            +
            \\
            \displaystyle\sum_{\indexjpc{l}\in\indexJpc}
            \displaystyle\sum_{u^{'}\in\indexJp}
            \begin{matrix*}[l]
                (-1)
                \sgn
                    {u^{'},\indexI,\order{(\indexJp\backslash\{u^{'}\})\cup\{\indexjpc{l}\}}}
                    {\indexjpc{l},\indexI,\indexJp}
                \coefB{\indexjpc{l}}{u^{'} \alpha}
                \basexi{\indexI}
                \wedge
                \basedualxi{\order{(\indexJp\backslash\{u^{'}\})\cup\{\indexjpc{l}\}}}
            \end{matrix*}
        \end{pmatrix*}
    \end{pmatrix*}.
\end{align*}

We complete the proof by observing that the expression of the right-hand side is completely equal to the left-hand side (\ref{EQ-For (alpha, I, J), calculation for LHS-total summary}):
\begin{align*}
    \begin{pmatrix*}[l]
            [\partialbar,\iotaW]
            \wedge
            \basexi{\indexI}\wedge\basedualxi{\indexJp}
            -
            \basexi{\alpha}
            \wedge
            [\partialbar,\iotaW]
            (\basexi{\indexI}\wedge\basedualxi{\indexJp})
            \\
            +
            \\
            (-2\sqrt{-1})
            \begin{pmatrix*}[l]
                \displaystyle\sum_{\indexi{k}\in\indexI}
                \begin{matrix*}[l]
                    \sgn
                        {\order{\indexi{k},\alpha}}
                        {\indexi{k},\alpha}
                    \coefD{\indexi{k}}{\order{\indexi{k},\alpha}}
                    \basexi{\indexI}
                    \wedge
                    \basedualxi{\indexJp}
                \end{matrix*}
                \\
                +
                \\
                \displaystyle\sum_{\indexj{l}\in\indexJp}
                \begin{matrix*}[l]
                    \coefB{\indexjp{l}}{\indexjp{l} \alpha}
                    \basexi{\indexI}
                    \wedge
                    \basedualxi{\indexJp}
                \end{matrix*}
                \\
                +
                \\
                \displaystyle\sum_{\indexi{k}\in\indexI}
                \begin{matrix*}[l]
                    \sgn
                        {\indexI}
                        {\indexi{k},\indexI\backslash\{\indexi{k}\}}
                    \sgn
                        {\order{\indexi{k},\alpha}}
                        {\indexi{k},\alpha}
                    \\
                    \coefD{\alpha}{\order{\indexi{k},\alpha}}
                    \basexi{\alpha}
                    \wedge
                    \basexi{(\indexI\backslash\{\indexi{k}\})}
                    \wedge
                    \basedualxi{\indexJp}
                \end{matrix*}
                \\
                +
                \\
                \displaystyle\sum_{\indexi{k}\in\indexI}
                \displaystyle\sum_{s\in\indexIc\backslash\{\alpha\}}
                \begin{matrix*}[l]
                    \sgn
                        {s,\indexI}
                        {\indexi{k},\order{(\indexI\backslash\{\indexi{k}\})\cup\{s\}}}
                    \sgn
                        {\order{\alpha,\indexi{k}}}
                        {\alpha,\indexi{k}}  
                    \\
                    \coefD{s}{\order{\indexi{k},\alpha}}
                    \basexi{\order{(\indexI\backslash\{\indexi{k}\})\cup\{s\}}}
                    \wedge
                    \basedualxi{\indexJp}
                \end{matrix*}
                \\
                +
                \\
                \displaystyle\sum_{\indexj{l}\in\indexJp}
                \displaystyle\sum_{t^{'}\in\indexJpc}
                \begin{matrix*}[l]
                    \sgn
                        {t^{'},\indexI,\indexJp}
                        {\indexjp{l},\indexI,\order{(\indexJp\backslash\{\indexjp{l}\})\cup\{t^{'}\}}}
                    \\
                    (-1)
                    \coefB{t^{'}}{\indexjp{l} \alpha}
                    \basexi{\indexI}
                    \wedge
                    \basedualxi{\order{(\indexJp\backslash\{\indexjp{l}\})\cup\{t^{'}\}}}
                \end{matrix*}
            \end{pmatrix*}
        \end{pmatrix*}
    .
\end{align*}

\end{proof}

\section{Appendix II}
\label{sec:extra long proof for lemmas of Lie algebroid Kahler identity}

\subsection{Calculation of the coefficients in \texorpdfstring{
        $
        \sqrt{-1}\LieagbdPartial^{*}(\basexi{i}\wedge\basedualxi{j})
        $}{}
    }
\label{apex-simplification for RHS of (i,j)}
In this section, we simplify the following three formulas:
\begin{align}
    &
    \begin{pmatrix*}[l]
        \sgn
            {1,\cdots,n,1^{'},\cdots,n^{'}}
            {i,j^{'},1,\cdots,\widehat{i},\cdots,n,1^{'},\cdots,\widehat{j^{'}},\cdots,n^{'}}
        \sgn
            {1,\cdots,\widehat{i},\cdots,n}
            {k,1,\cdots,\widehat{i},\cdots,\widehat{k},\cdots,n}
        \sgn
            {1,\cdots,\widehat{i},\cdots,k-1,k,i,k+1,\cdots,n}
            {1,\cdots,n}
        \\
        \sgn
            {1^{'},\cdots,n^{'},1,\cdots,n}
            {j^{'},1,\cdots,n,1^{'},\cdots,\widehat{j^{'}},\cdots,n^{'}}
    \end{pmatrix*}
    ,
\label{EQ-simplification for RHS of (i,j)-target 1}
    \\
    &
    \begin{pmatrix*}[l]
        \sgn
            {1,\cdots,n,1^{'},\cdots,n^{'}}
            {i,j^{'},1,\cdots,\widehat{i},\cdots,n,1^{'},\cdots,\widehat{j^{'}},\cdots,n^{'}}
        \sgn
            {1,\cdots,\widehat{i},\cdots,n,1^{'},\cdots,\widehat{j^{'}},\cdots,n^{'}}
            {l^{'},1,\cdots,\widehat{i},\cdots,n,1^{'},\cdots,\widehat{j^{'}},\cdots,\widehat{l^{'}},\cdots,n^{'}}
        \\
        \sgn
            {1,\cdots,\widehat{i}\cdots,n,1^{'},\cdots,\widehat{j^{'}},\cdots,(l-1)^{'},l^{'},i,(l+1)^{'},\cdots,n^{'}}
            {1,\cdots,n,1^{'},\cdots,\widehat{j^{'}},\cdots,n^{'}}
        \sgn
            {1^{'},\cdots,n^{'},1,\cdots,n}
            {j^{'},1,\cdots,n,1^{'},\cdots,\widehat{j^{'}},\cdots,n^{'}}
    \end{pmatrix*}
    ,
\label{EQ-simplification for RHS of (i,j)-target 2}
    \\
    &
    \begin{pmatrix*}[l]
        \sgn
            {1,\cdots,n,1^{'},\cdots,n^{'}}
            {i,j^{'},1,\cdots,\widehat{i},\cdots,n,1^{'},\cdots,\widehat{j^{'}},\cdots,n^{'}}
        \sgn
            {1,\cdots,\widehat{i},\cdots,n,1^{'},\cdots,\widehat{j^{'}},\cdots,n^{'}}
            {l^{'},1,\cdots,\widehat{i},\cdots,n,1^{'},\cdots,\widehat{j^{'}},\cdots,\widehat{l^{'}},\cdots,n^{'}}
        \\
        \sgn
            {1,\cdots,\widehat{i}\cdots,n,1^{'},\cdots,\widehat{j^{'}},\cdots,(l-1)^{'},j^{'},i,(l+1)^{'},\cdots,n^{'}}
            {1,\cdots,n,1^{'},\cdots,\widehat{l^{'}},\cdots,n^{'}}
        \sgn
            {1^{'},\cdots,n^{'},1,\cdots,n}
            {l^{'},1,\cdots,n,1^{'},\cdots,\widehat{l^{'}},\cdots,n^{'}}
    \end{pmatrix*}
    ,
\label{EQ-simplification for RHS of (i,j)-target 3}
\end{align}
which are from (\ref{EQ-Kahler identity on Lie agbd-calculation for RHS of (i,j)-no5}) in Lemma \ref{Lemma-Kahler identity on Lie algebroid-the first lemma}.


\begin{lemma}
$
    \begin{pmatrix*}
    \sgn
        {1,\cdots,n,1^{'},\cdots,n^{'}}
        {i,j^{'},1,\cdots,\widehat{i},\cdots,n,1^{'},\cdots,\widehat{j^{'}},\cdots,n^{'}}
    \sgn
        {1,\cdots,\widehat{i},\cdots,n}
        {k,1,\cdots,\widehat{i},\cdots,\widehat{k},\cdots,n}
    \sgn
        {1,\cdots,\widehat{i},\cdots,k-1,k,i,k+1,\cdots,n}
        {1,\cdots,n}
    \\
    \sgn
        {1^{'},\cdots,n^{'},1,\cdots,n}
        {j^{'},1,\cdots,n,1^{'},\cdots,\widehat{j^{'}},\cdots,n^{'}}
    \end{pmatrix*}
    =
    (-1)^{n^2}.
$
\end{lemma}

\begin{proof}
We use the following properties of the sign of permutation:
\begin{itemize}
    \item
$
    \sgn
        {1,\cdots,n,1^{'},\cdots,n^{'}}
        {i,j^{'},1,\cdots,\widehat{i},\cdots,n,1^{'},\cdots,\widehat{j^{'}},\cdots,n^{'}}
    =
    -
    \sgn
        {1,\cdots,n,1^{'},\cdots,n^{'}}
        {j^{'},i,1,\cdots,\widehat{i},\cdots,n,1^{'},\cdots,\widehat{j^{'}},\cdots,n^{'}},
$
    \item 
$
    \sgn
        {1,\cdots,\widehat{i},\cdots,k-1,k,i,k+1,\cdots,n}
        {1,\cdots,n}
    =
    \sgn
        {k,i,1,\cdots,\widehat{i},\cdots,k-1,k+1,\cdots,n}
        {1,\cdots,n}
    =
    -\sgn
        {i,k,1,\cdots,\widehat{i},\cdots,k-1,k+1,\cdots,n}
        {1,\cdots,n},
$
    \item
$
    \sgn
        {1,\cdots,\widehat{i},\cdots,n}
        {k,1,\cdots,\widehat{i},\cdots,\widehat{k},\cdots,n}
    \sgn
        {i,k,1,\cdots,\widehat{i},\cdots,k-1,k+1,\cdots,n}
        {1,\cdots,n}
    =
    \sgn
        {i,1,\cdots,\widehat{i},\cdots,n}
        {1,\cdots,n}.
$
\end{itemize}
Expression (\ref{EQ-simplification for RHS of (i,j)-target 1}) is simplified to
\begin{align}
    (-1)
    \sgn
        {1,\cdots,n,1^{'},\cdots,n^{'}}
        {j^{'},i,1,\cdots,\widehat{i},\cdots,n,1^{'},\cdots,\widehat{j^{'}},\cdots,n^{'}}
    (-1)
    \sgn
        {i,1,\cdots,\widehat{i},\cdots,n}
        {1,\cdots,n}
    \sgn
        {1^{'},\cdots,n^{'},1,\cdots,n}
        {j^{'},1,\cdots,n,1^{'},\cdots,\widehat{j^{'}},\cdots,n^{'}}.
    \label{EQ-simplification for RHS of (i,j)-simplify of target 1-no1}
\end{align}
We continue the computation with the following properties of the sign of permutation:
\begin{itemize}
    \item 
$
    \sgn
        {1,\cdots,n,1^{'},\cdots,n^{'}}
        {j^{'},i,1,\cdots,\widehat{i},\cdots,n,1^{'},\cdots,\widehat{j^{'}},\cdots,n^{'}}
    \sgn
        {i,1,\cdots,\widehat{i},\cdots,n}
        {1,\cdots,n}
    =
     \sgn
        {1,\cdots,n,1^{'},\cdots,n^{'}}
        {j^{'},1,\cdots,n,1^{'},\cdots,\widehat{j^{'}},\cdots,n^{'}},
$
    \item
$
    \sgn
        {1,\cdots,n,1^{'},\cdots,n^{'}}
        {j^{'},1,\cdots,n,1^{'},\cdots,\widehat{j^{'}},\cdots,n^{'}}
    \sgn
        {1^{'},\cdots,n^{'},1,\cdots,n}
        {j^{'},1,\cdots,n,1^{'},\cdots,\widehat{j^{'}},\cdots,n^{'}}
    =
    \sgn
        {1,\cdots,n,1^{'},\cdots,n^{'}}
        {1^{'},\cdots,n^{'},1,\cdots,n}.
$
\end{itemize}
Expression (\ref{EQ-simplification for RHS of (i,j)-simplify of target 1-no1}) is then simplified to
\begin{align*}
    \sgn
        {1,\cdots,n,1^{'},\cdots,n^{'}}
        {1^{'},\cdots,n^{'},1,\cdots,n}
    =
    (-1)^{n^2}.
\end{align*}
\end{proof}


\begin{lemma}
    $
    \begin{pmatrix*}[l]
        \sgn
            {1,\cdots,n,1^{'},\cdots,n^{'}}
            {i,j^{'},1,\cdots,\widehat{i},\cdots,n,1^{'},\cdots,\widehat{j^{'}},\cdots,n^{'}}
        \sgn
            {1,\cdots,\widehat{i},\cdots,n,1^{'},\cdots,\widehat{j^{'}},\cdots,n^{'}}
            {l^{'},1,\cdots,\widehat{i},\cdots,n,1^{'},\cdots,\widehat{j^{'}},\cdots,\widehat{l^{'}},\cdots,n^{'}}
        \\
        \sgn
            {1,\cdots,\widehat{i}\cdots,n,1^{'},\cdots,\widehat{j^{'}},\cdots,(l-1)^{'},l^{'},i,(l+1)^{'},\cdots,n^{'}}
            {1,\cdots,n,1^{'},\cdots,\widehat{j^{'}},\cdots,n^{'}}
        \sgn
            {1^{'},\cdots,n^{'},1,\cdots,n}
            {j^{'},1,\cdots,n,1^{'},\cdots,\widehat{j^{'}},\cdots,n^{'}}
    \end{pmatrix*}
    =
    (-1)^{n^2}.
    $
\end{lemma}

\begin{proof}
We use the following properties of the sign of permutation:
\begin{itemize}
    \item
$
    \sgn
        {1,\cdots,n,1^{'},\cdots,n^{'}}
        {i,j^{'},1,\cdots,\widehat{i},\cdots,n,1^{'},\cdots,\widehat{j^{'}},\cdots,n^{'}}
    =
    (-1)
    \sgn
        {1,\cdots,n,1^{'},\cdots,n^{'}}
        {j^{'},i,1,\cdots,\widehat{i},\cdots,n,1^{'},\cdots,\widehat{j^{'}},\cdots,n^{'}}
$,
    \item 
$
    \sgn
        {1,\cdots,\widehat{i}\cdots,n,1^{'},\cdots,\widehat{j^{'}},\cdots,(l-1)^{'},l^{'},i,(l+1)^{'},\cdots,n^{'}}
        {1,\cdots,n,1^{'},\cdots,\widehat{j^{'}},\cdots,n^{'}}
    =
    \sgn
        {l^{'},i,1,\cdots,\widehat{i}\cdots,n,1^{'},\cdots,\widehat{j^{'}},\cdots,(l-1)^{'},(l+1)^{'},\cdots,n^{'}}
        {1,\cdots,n,1^{'},\cdots,\widehat{j^{'}},\cdots,n^{'}}
$
    \\
$
    =
    (-1)
    \sgn
        {i,l^{'},1,\cdots,\widehat{i}\cdots,n,1^{'},\cdots,\widehat{j^{'}},\cdots,(l-1)^{'},(l+1)^{'},\cdots,n^{'}}
        {1,\cdots,n,1^{'},\cdots,\widehat{j^{'}},\cdots,n^{'}}
$,
    \item
$
    \sgn
            {1,\cdots,\widehat{i},\cdots,n,1^{'},\cdots,\widehat{j^{'}},\cdots,n^{'}}
            {l^{'},1,\cdots,\widehat{i},\cdots,n,1^{'},\cdots,\widehat{j^{'}},\cdots,\widehat{l^{'}},\cdots,n^{'}}
    \sgn
        {i,l^{'},1,\cdots,\widehat{i}\cdots,n,1^{'},\cdots,\widehat{j^{'}},\cdots,\widehat{l^{'}},\cdots,n^{'}}
        {1,\cdots,n,1^{'},\cdots,\widehat{j^{'}},\cdots,n^{'}}
    =
    \sgn
        {i,1,\cdots,\widehat{i},\cdots,n,1^{'},\cdots,\widehat{j^{'}},\cdots,n^{'}}
        {1,\cdots,n,1^{'},\cdots,\widehat{j^{'}},\cdots,n^{'}}
$.
\end{itemize}
We simplify expression (\ref{EQ-simplification for RHS of (i,j)-target 2}) to
\begin{align}
    \sgn
        {1,\cdots,n,1^{'},\cdots,n^{'}}
        {j^{'},i,1,\cdots,\widehat{i},\cdots,n,1^{'},\cdots,\widehat{j^{'}},\cdots,n^{'}}
    \sgn
        {i,1,\cdots,\widehat{i},\cdots,n,1^{'},\cdots,\widehat{j^{'}},\cdots,n^{'}}
        {1,\cdots,n,1^{'},\cdots,\widehat{j^{'}},\cdots,n^{'}}
    \sgn
        {1^{'},\cdots,n^{'},1,\cdots,n}
        {j^{'},1,\cdots,n,1^{'},\cdots,\widehat{j^{'}},\cdots,n^{'}}.
    \label{EQ-simplification for RHS of (i,j)-simplify of target 2-no1}
\end{align}
We continue the computation with the following properties of the sign of permutation:
\begin{itemize}
    \item 
$
    \sgn
        {1,\cdots,n,1^{'},\cdots,n^{'}}
        {j^{'},i,1,\cdots,\widehat{i},\cdots,n,1^{'},\cdots,\widehat{j^{'}},\cdots,n^{'}}
    \sgn
        {i,1,\cdots,\widehat{i},\cdots,n,1^{'},\cdots,\widehat{j^{'}},\cdots,n^{'}}
        {1,\cdots,n,1^{'},\cdots,\widehat{j^{'}},\cdots,n^{'}}
    =
    \sgn
        {1,\cdots,n,1^{'},\cdots,n^{'}}
        {j^{'},1,\cdots,n,1^{'},\cdots,\widehat{j^{'}},\cdots,n^{'}},
$
    \item
$
    \sgn
        {1,\cdots,n,1^{'},\cdots,n^{'}}
        {j^{'},1,\cdots,n,1^{'},\cdots,\widehat{j^{'}},\cdots,n^{'}}
    \sgn
        {1^{'},\cdots,n^{'},1,\cdots,n}
        {j^{'},1,\cdots,n,1^{'},\cdots,\widehat{j^{'}},\cdots,n^{'}}
    =
    \sgn
        {1^{'},\cdots,n^{'},1,\cdots,n}
        {1,\cdots,n,1^{'},\cdots,n^{'}}
$.
\end{itemize}
We obtain the simplification of (\ref{EQ-simplification for RHS of (i,j)-simplify of target 2-no1}):
\begin{align*}
    \sgn
        {1^{'},\cdots,n^{'},1,\cdots,n}
        {1,\cdots,n,1^{'},\cdots,n^{'}}
    =
    (-1)^{n^2}.
\end{align*}
\end{proof}


\begin{lemma}
$
    \begin{pmatrix*}
        \sgn
            {1,\cdots,n,1^{'},\cdots,n^{'}}
            {i,j^{'},1,\cdots,\widehat{i},\cdots,n,1^{'},\cdots,\widehat{j^{'}},\cdots,n^{'}}
        \sgn
            {1,\cdots,\widehat{i},\cdots,n,1^{'},\cdots,\widehat{j^{'}},\cdots,n^{'}}
            {l^{'},1,\cdots,\widehat{i},\cdots,n,1^{'},\cdots,\widehat{j^{'}},\cdots,\widehat{l^{'}},\cdots,n^{'}}
        \\
        \sgn
            {1,\cdots,\widehat{i}\cdots,n,1^{'},\cdots,\widehat{j^{'}},\cdots,(l-1)^{'},j^{'},i,(l+1)^{'},\cdots,n^{'}}
            {1,\cdots,n,1^{'},\cdots,\widehat{l^{'}},\cdots,n^{'}}
        \sgn
            {1^{'},\cdots,n^{'},1,\cdots,n}
            {l^{'},1,\cdots,n,1^{'},\cdots,\widehat{l^{'}},\cdots,n^{'}}
    \end{pmatrix*}
    =
    -(-1)^{n^2}.
$
\end{lemma}

\begin{proof}
We use the following properties of the sign of permutation:
\begin{itemize}
    \item 
$
    \sgn
        {1,\cdots,n,1^{'},\cdots,n^{'}}
        {i,j^{'},1,\cdots,\widehat{i},\cdots,n,1^{'},\cdots,\widehat{j^{'}},\cdots,n^{'}}
    =
    (-1)
    \sgn
        {1,\cdots,n,1^{'},\cdots,n^{'}}
        {j^{'},i,1,\cdots,\widehat{i},\cdots,n,1^{'},\cdots,\widehat{j^{'}},\cdots,n^{'}}
$,
    \item
$
    \begin{matrix*}[l]
    \sgn
        {1,\cdots,\widehat{i}\cdots,n,1^{'},\cdots,\widehat{j^{'}},\cdots,(l-1)^{'},j^{'},i,(l+1)^{'},\cdots,n^{'}}
        {1,\cdots,n,1^{'},\cdots,\widehat{l^{'}},\cdots,n^{'}}
    &=
    \sgn
        {j^{'},i,1,\cdots,\widehat{i}\cdots,n,1^{'},\cdots,\widehat{j^{'}},\cdots,\widehat{l^{'}},\cdots,n^{'}}
        {1,\cdots,n,1^{'},\cdots,\widehat{l^{'}},\cdots,n^{'}}
    \\
    &=
    \sgn
        {j^{'},i,l^{'},1,\cdots,\widehat{i}\cdots,n,1^{'},\cdots,\widehat{j^{'}},\cdots,\widehat{l^{'}},\cdots,n^{'}}
        {l^{'},1,\cdots,n,1^{'},\cdots,\widehat{l^{'}},\cdots,n^{'}}
    \end{matrix*}
$,
    \item 
$
    \sgn
        {1,\cdots,\widehat{i},\cdots,n,1^{'},\cdots,\widehat{j^{'}},\cdots,n^{'}}
        {l^{'},1,\cdots,\widehat{i},\cdots,n,1^{'},\cdots,\widehat{j^{'}},\cdots,\widehat{l^{'}},\cdots,n^{'}}
    \sgn
        {j^{'},i,l^{'},1,\cdots,\widehat{i}\cdots,n,1^{'},\cdots,\widehat{j^{'}},\cdots,\widehat{l^{'}},\cdots,n^{'}}
        {l^{'},1,\cdots,n,1^{'},\cdots,\widehat{l^{'}},\cdots,n^{'}}
    =
    \sgn
        {j^{'},i,1,\cdots,\widehat{i}\cdots,n,1^{'},\cdots,\widehat{j^{'}},\cdots,n^{'}}
        {l^{'},1,\cdots,n,1^{'},\cdots,\widehat{l^{'}},\cdots,n^{'}}
$.
\end{itemize}
Expression (\ref{EQ-simplification for RHS of (i,j)-target 3}) is then simplified to
\begin{align}
    (-1)
    \sgn
        {1,\cdots,n,1^{'},\cdots,n^{'}}
        {j^{'},i,1,\cdots,\widehat{i},\cdots,n,1^{'},\cdots,\widehat{j^{'}},\cdots,n^{'}}
    \sgn
        {j^{'},i,1,\cdots,\widehat{i},\cdots,n,1^{'},\cdots,\widehat{j^{'}},\cdots,n^{'}}
        {l^{'},1,\cdots,n,1^{'},\cdots,\widehat{l^{'}},\cdots,n^{'}}
    \sgn
        {1^{'},\cdots,n^{'},1,\cdots,n}
        {l^{'},1,\cdots,n,1^{'},\cdots,\widehat{l^{'}},\cdots,n^{'}}.
\label{EQ-simplification for RHS of (i,j)-simplify of target 3-no1}
\end{align}
We continue the computation with the following properties:
\begin{itemize}
    \item 
$
    \sgn
        {j^{'},i,1,\cdots,\widehat{i},\cdots,n,1^{'},\cdots,\widehat{j^{'}},\cdots,n^{'}}
        {l^{'},1,\cdots,n,1^{'},\cdots,\widehat{l^{'}},\cdots,n^{'}}
    \sgn
        {1^{'},\cdots,n^{'},1,\cdots,n}
        {l^{'},1,\cdots,n,1^{'},\cdots,\widehat{l^{'}},\cdots,n^{'}}
    =
    \sgn
        {j^{'},i,1,\cdots,\widehat{i},\cdots,n,1^{'},\cdots,\widehat{j^{'}},\cdots,n^{'}}
        {1^{'},\cdots,n^{'},1,\cdots,n}
$,
    \item 
$
    \sgn
        {1,\cdots,n,1^{'},\cdots,n^{'}}
        {j^{'},i,1,\cdots,\widehat{i},\cdots,n,1^{'},\cdots,\widehat{j^{'}},\cdots,n^{'}}
    \sgn
        {j^{'},i,1,\cdots,\widehat{i},\cdots,n,1^{'},\cdots,\widehat{j^{'}},\cdots,n^{'}}
        {1^{'},\cdots,n^{'},1,\cdots,n}
    =
     \sgn
        {1,\cdots,n,1^{'},\cdots,n^{'}}
        {1^{'},\cdots,n^{'},1,\cdots,n}
$.
\end{itemize}
We then simplify (\ref{EQ-simplification for RHS of (i,j)-simplify of target 3-no1}) as follows:
\begin{align*}
    (-1)
    \sgn
        {1,\cdots,n,1^{'},\cdots,n^{'}}
        {1^{'},\cdots,n^{'},1,\cdots,n}
    =
    -(-1)^{n^2}
    .
\end{align*}
\end{proof}

\subsection{Calculation of the coefficients in \texorpdfstring{$
        \sqrt{-1}\LieagbdPartial^{*}
        (\basexi{\alpha}\wedge\basexi{\indexI}\wedge\basedualxi{\indexJp})
        $}{}
    }
\label{apex-simplification for RHs of (alpha,I,J)}
In this section we calculate the following formulas:
\begin{itemize}
    \item 
$
    \begin{pmatrix*}[l]
        \sgn
            {1,\cdots,n,1^{'},\cdots,n^{'}}
            {\alpha,\indexI,\indexJp,\indexIc\backslash\{\alpha\},\indexJpc}
        \sgn
            {\indexIc\backslash\{\alpha\}}
            {\indexic{k},\indexIc\backslash\{\alpha,\indexic{k}\}}
        \\
        \sgn
            {\indexic{1},\cdots,\indexic{k-1},\order{\indexic{k},\alpha},\indexic{k+1},\cdots,\indexic{n-p}}
            {\indexIc}
        \sgn
            {1^{'},\cdots,n^{'},1,\cdots,n}
            {\indexIc,\indexJpc,\indexI,\indexJp}
    \end{pmatrix*}
$,
    \item 
$
    \begin{pmatrix*}[l]
    \sgn
        {1,\cdots,n,1^{'},\cdots,n^{'}}
        {\alpha,\indexI,\indexJp,\indexIc\backslash\{\alpha\},\indexJpc}
    \sgn
        {\indexIc\backslash\{\alpha\}}
        {\indexic{k},\indexIc\backslash\{\alpha,\indexic{k}\}}
    \\
    \sgn
        {\indexic{1},\cdots,\indexic{k-1},\order{\indexic{k},r},\indexic{k+1},\cdots,\widehat{\alpha},\cdots,\indexic{n-p}}
        {\order{(\indexIc\backslash\{\alpha\})\cup\{r\}}}
    \\
    \sgn
        {1^{'},\cdots,n^{'},1,\cdots,n}
        {\order{(\indexIc\backslash\{\alpha\})\cup\{r\}},\indexJpc,\alpha,\indexI\backslash\{r\},\indexJp}
    \end{pmatrix*}
$,
    \item 
$
    \begin{pmatrix*}[l]
        \sgn
            {1,\cdots,n,1^{'},\cdots,n^{'}}
            {\alpha,\indexI,\indexJp,\indexIc\backslash\{\alpha\},\indexJpc}
        \sgn
            {\indexIc\backslash\{\alpha\}}
            {\indexic{k},\indexIc\backslash\{\alpha,\indexic{k}\}}
        \\
        \sgn
            {\indexic{1},\cdots,\indexic{k-1},\order{\alpha,r},\indexic{k+1},\cdots,\widehat{\alpha},\cdots,\indexic{n-p}}
            {\order{(\indexIc\backslash\{\indexic{k}\})\cup\{r\}}}
        \\
        \sgn
            {1^{'},\cdots,n^{'},1,\cdots,n}
            {\order{(\indexIc\backslash\{\indexic{k}\})\cup\{r\}},\indexJpc,\order{(\indexI\backslash\{r\})\cup\{\indexic{k}\}},\indexJp}
    \end{pmatrix*}
$,
    \item 
$
    \begin{pmatrix*}[l]
        \sgn
            {1,\cdots,n,1^{'},\cdots,n^{'}}
            {\alpha,\indexI,\indexJp,\indexIc\backslash\{\alpha\},\indexJpc}
        \sgn
            {\indexIc\backslash\{\alpha\}}
            {\indexic{k},\indexIc\backslash\{\alpha,\indexic{k}\}}
        \\
        \sgn
            {\indexic{1},\cdots,\indexic{k-1},\order{r_1,r_2},\indexic{k+1},\cdots,\widehat{\alpha},\cdots,\indexic{n-p}}
            {\order{(\indexIc\backslash\{\indexic{k},\alpha\})\cup\{r_1,r_2\}}}
        \\
        \sgn
            {1^{'},\cdots,n^{'},1,\cdots,n}
            {\order{(\indexIc\backslash\{\indexic{k},\alpha\})\cup\{r_1,r_2\}},\indexJpc,\alpha,\order{(\indexI\backslash\{r_1,r_2\})\cup\{\indexic{k}\}},\indexJp}
    \end{pmatrix*}    
$,
    \item 
$
    \begin{pmatrix*}[l]
        \sgn
            {1,\cdots,n,1^{'},\cdots,n^{'}}
            {\alpha,\indexI,\indexJp,\indexIc\backslash\{\alpha\},\indexJpc}
        \sgn
            {\indexIc\backslash\{\alpha\},\indexJpc}
            {\indexjpc{l},\indexIc\backslash\{\alpha\},\indexJpc\backslash\{\indexjpc{l}\}}
        \\
        \sgn
            {\indexIc\backslash\{\alpha\},\indexjpc{1},\cdots,\indexjpc{l-1},\indexjpc{l},\alpha,\indexjpc{l+1},\cdots,\indexjpc{n-q}}
            {\indexIc,\indexJpc}
        \sgn
            {1^{'},\cdots,n^{'},1,\cdots,n}
            {\indexIc,\indexJpc,\indexI,\indexJp}
    \end{pmatrix*}
$,
    \item 
$
    \begin{pmatrix*}[l]
        \sgn
            {1,\cdots,n,1^{'},\cdots,n^{'}}
            {\alpha,\indexI,\indexJp,\indexIc\backslash\{\alpha\},\indexJpc}
        \sgn
            {\indexIc\backslash\{\alpha\},\indexJpc}
            {\indexjpc{l},\indexIc\backslash\{\alpha\},\indexJpc\backslash\{\indexjpc{l}\}}
        \\
        \sgn
            {\indexIc\backslash\{\alpha\},\indexjpc{1},\cdots,\indexjpc{l-1},\indexjpc{l},r,\indexjpc{l+1},\cdots,\indexjpc{n-q}}
            {\order{(\indexIc\backslash\{\alpha\})\cup\{r\}},\indexJpc}
        \\
        \sgn
            {1^{'},\cdots,n^{'},1,\cdots,n}
            {\order{(\indexIc\backslash\{\alpha\})\cup\{r\}},\indexJpc,\alpha,\indexI\backslash\{r\},\indexJp}
    \end{pmatrix*}
$,
    \item 
$
    \begin{pmatrix*}[l]
        \sgn
            {1,\cdots,n,1^{'},\cdots,n^{'}}
            {\alpha,\indexI,\indexJp,\indexIc\backslash\{\alpha\},\indexJpc}
        \sgn
            {\indexIc\backslash\{\alpha\},\indexJpc}
            {\indexjpc{l},\indexIc\backslash\{\alpha\},\indexJpc\backslash\{\indexjpc{l}\}}
        \\
        \sgn
            {\indexIc\backslash\{\alpha\},\indexjpc{1},\cdots,\indexjpc{l-1},u^{'},\alpha,\indexjpc{l+1},\cdots,\indexjpc{n-q}}
            {\indexIc,\order{(\indexJpc\backslash\{\indexjpc{l}\})\cup\{u^{'}\}}}
        \\
        \sgn
            {1^{'},\cdots,n^{'},1,\cdots,n}
            {\indexIc,\order{(\indexJpc\backslash\{\indexjpc{l}\})\cup\{u^{'}\}},\indexI,\order{(\indexJp\backslash\{u^{'}\})\cup\{\indexjpc{l}\}}}
    \end{pmatrix*}
$,
    \item 
$
    \begin{pmatrix*}[l]
        \sgn
            {1,\cdots,n,1^{'},\cdots,n^{'}}
            {\alpha,\indexI,\indexJp,\indexIc\backslash\{\alpha\},\indexJpc}
        \sgn
            {\indexIc\backslash\{\alpha\},\indexJpc}
            {\indexjpc{l},\indexIc\backslash\{\alpha\},\indexJpc\backslash\{\indexjpc{l}\}}
        \\
        \sgn
            {\indexIc\backslash\{\alpha\},\indexjpc{1},\cdots,\indexjpc{l-1},u^{'},r,\indexjpc{l+1},\cdots,\indexjpc{n-q}}
            {\order{(\indexIc\backslash\{\alpha\})\cup\{r\}},\order{(\indexJpc\backslash\{\indexjpc{l}\})\cup\{u^{'}\}}}
        \\
        \sgn
            {1^{'},\cdots,n^{'},1,\cdots,n}
            {\order{(\indexIc\backslash\{\alpha\})\cup\{r\}},\order{(\indexJpc\backslash\{\indexjpc{l}\})\cup\{u^{'}\}},\alpha,\indexI\backslash\{r\},\order{(\indexJp\backslash\{u^{'}\})\cup\{\indexjpc{l}\}}}
    \end{pmatrix*}
$,
\end{itemize}
which are from (\ref{EQ-For (alpha, I, J), calculation for RHS-Step 3, unsimplfied term 1}) to (\ref{EQ-For (alpha, I, J), calculation for RHS-Step 3, unsimplfied term 8}) in Section \ref{subsection-For (alpha, I, J), calculation for RHS}. The purpose is to simplify the expression of $
\sqrt{-1}\LieagbdPartial^{*}
(\basexi{\alpha}\wedge\basexi{\indexI}\wedge\basedualxi{\indexJp})
$.


\begin{lemma}
\label{Lemma-simplification for RHS of (alpha,I,J)-lemma 1}
\begin{align*}
    \begin{pmatrix*}[l]
        \sgn
            {1,\cdots,n,1^{'},\cdots,n^{'}}
            {\alpha,\indexI,\indexJp,\indexIc\backslash\{\alpha\},\indexJpc}
        \sgn
            {\indexIc\backslash\{\alpha\}}
            {\indexic{k},\indexIc\backslash\{\alpha,\indexic{k}\}}
        \\
        \sgn
            {\indexic{1},\cdots,\indexic{k-1},\order{\indexic{k},\alpha},\indexic{k+1},\cdots,\indexic{n-p}}
            {\indexIc}
        \sgn
            {1^{'},\cdots,n^{'},1,\cdots,n}
            {\indexIc,\indexJpc,\indexI,\indexJp}
    \end{pmatrix*}
    =
    (-1)^{n^2}
    \sgn
        {\order{\indexic{k},\alpha}}
        {\alpha,\indexic{k}}
    .
\end{align*}
\end{lemma}
\begin{proof}
We use the following properties of the sign of permutation:
\begin{itemize}
    \item 
$
    \sgn
        {\indexic{1},\cdots,\indexic{k-1},\order{\indexic{k},\alpha},\indexic{k+1},\cdots,\indexic{n-p}}
        {\indexIc}
    =
    \sgn
        {\order{\indexic{k},\alpha},\indexIc\backslash\{\indexic{k},\alpha\}}
        {\indexIc}
$,
    \item 
$
    \sgn
        {1,\cdots,n,1^{'},\cdots,n^{'}}
        {\alpha,\indexI,\indexJp,\indexIc\backslash\{\alpha\},\indexJpc}
    \sgn
        {\indexIc\backslash\{\alpha\}}
        {\indexic{k},\indexIc\backslash\{\alpha,\indexic{k}\}}
    =
    \sgn
        {1,\cdots,n,1^{'},\cdots,n^{'}}
        {\alpha,\indexI,\indexJp,\indexic{k},\indexIc\backslash\{\alpha,\indexic{k}\},\indexJpc}
$.
\end{itemize}
We simplify the left-hand side of the equation in Lemma \ref{Lemma-simplification for RHS of (alpha,I,J)-lemma 1} to
\begin{align}
    \sgn
        {1,\cdots,n,1^{'},\cdots,n^{'}}
        {\alpha,\indexI,\indexJp,\indexic{k},\indexIc\backslash\{\alpha,\indexic{k}\},\indexJpc}
    \sgn
        {\order{\indexic{k},\alpha},\indexIc\backslash\{\indexic{k}\}}
        {\indexIc}
    \sgn
        {1^{'},\cdots,n^{'},1,\cdots,n}
        {\indexIc,\indexJpc,\indexI,\indexJp}.
\label{EQ-simplification for RHS of (alpha,I,J)-lemma 1-no1}
\end{align}
We use the following properties of the sign of permutation:
\begin{itemize}
    \item 
$
    \sgn
        {1,\cdots,n,1^{'},\cdots,n^{'}}
        {\alpha,\indexI,\indexJp,\indexic{k},\indexIc\backslash\{\alpha,\indexic{k}\},\indexJpc}
    =
    (-1)^{p+q}
    \sgn
        {1,\cdots,n,1^{'},\cdots,n^{'}}
        {\indexI,\indexJp,\alpha,\indexic{k},\indexIc\backslash\{\alpha,\indexic{k}\},\indexJpc}
$,
    \item 
$
    \sgn
        {1,\cdots,n,1^{'},\cdots,n^{'}}
        {\indexI,\indexJp,\alpha,\indexic{k},\indexIc\backslash\{\alpha,\indexic{k}\},\indexJpc}
    \sgn
        {\alpha,\indexic{k},\indexIc\backslash\{\indexic{k}\}}
        {\indexIc}
    =
    \sgn
        {1,\cdots,n,1^{'},\cdots,n^{'}}
        {\indexI,\indexJp,\indexIc,\indexJpc}
$,
    \item
$
    \sgn
        {\order{\indexic{k},\alpha},\indexIc\backslash\{\indexic{k}\}}
        {\indexIc}
    =
    \sgn
        {\alpha,\indexic{k},\indexIc\backslash\{\indexic{k}\}}
        {\indexIc}
    \sgn
        {\order{\indexic{k},\alpha}}
        {\alpha,\indexic{k}}
$.
\end{itemize}
We simplify (\ref{EQ-simplification for RHS of (alpha,I,J)-lemma 1-no1}) to
\begin{align}
    (-1)^{p+q}
    \sgn
        {1,\cdots,n,1^{'},\cdots,n^{'}}
        {\indexI,\indexJp,\indexIc,\indexJpc}
    \sgn
        {1^{'},\cdots,n^{'},1,\cdots,n}
        {\indexIc,\indexJpc,\indexI,\indexJp}
    \sgn
        {\order{\indexic{k},\alpha}}
        {\alpha,\indexic{k}}
    .
\label{EQ-simplification for RHS of (alpha,I,J)-lemma 1-no2}
\end{align}
We use the following properties of the sign of permutation:
\begin{itemize}
    \item 
$
    \sgn
        {1^{'},\cdots,n^{'},1,\cdots,n}
        {\indexIc,\indexJpc,\indexI,\indexJp}
    =
    (-1)^{(p+q)(n-p+n-q)}
    \sgn
        {1^{'},\cdots,n^{'},1,\cdots,n}
        {\indexI,\indexJp,\indexIc,\indexJpc}
    \\
    =
    (-1)^{(p+q)^2}
    \sgn
        {1^{'},\cdots,n^{'},1,\cdots,n}
        {\indexI,\indexJp,\indexIc,\indexJpc}
$,
    \item 
$
    \sgn
        {1,\cdots,n,1^{'},\cdots,n^{'}}
        {\indexI,\indexJp,\indexIc,\indexJpc}
    \sgn
        {1^{'},\cdots,n^{'},1,\cdots,n}
        {\indexI,\indexJp,\indexIc,\indexJpc}
    =
    \sgn
        {1,\cdots,n,1^{'},\cdots,n^{'}}
        {1^{'},\cdots,n^{'},1,\cdots,n}
    =
    (-1)^{n^2}
$.
\end{itemize}
We simplify (\ref{EQ-simplification for RHS of (alpha,I,J)-lemma 1-no2}) to
\begin{align*}
    (-1)^{(p+q)(p+q+1)}
    (-1)^{n^2}
    \sgn
        {\order{\indexic{k},\alpha}}
        {\alpha,\indexic{k}}
=
    (-1)^{n^2}
    \sgn
        {\order{\indexic{k},\alpha}}
        {\alpha,\indexic{k}}
    ,
\end{align*}
which completes the proof of Lemma \ref{Lemma-simplification for RHS of (alpha,I,J)-lemma 1}.
\end{proof}

\begin{lemma}
\label{Lemma-simplification for RHS of (alpha,I,J)-lemma 2}
\begin{align*} 
    \begin{pmatrix*}[l]
    \sgn
        {1,\cdots,n,1^{'},\cdots,n^{'}}
        {\alpha,\indexI,\indexJp,\indexIc\backslash\{\alpha\},\indexJpc}
    \sgn
        {\indexIc\backslash\{\alpha\}}
        {\indexic{k},\indexIc\backslash\{\alpha,\indexic{k}\}}
    \\
    \sgn
        {\indexic{1},\cdots,\indexic{k-1},\order{\indexic{k},r},\indexic{k+1},\cdots,\widehat{\alpha},\cdots,\indexic{n-p}}
        {\order{(\indexIc\backslash\{\alpha\})\cup\{r\}}}
    \\
    \sgn
        {1^{'},\cdots,n^{'},1,\cdots,n}
        {\order{(\indexIc\backslash\{\alpha\})\cup\{r\}},\indexJpc,\alpha,\indexI\backslash\{r\},\indexJp}
    \end{pmatrix*}
    =
    -(-1)^{n^2}
    \sgn
        {r,\indexI\backslash\{r\}}
        {\indexI}
    \sgn
        {\order{\indexic{k},r}}
        {r,\indexic{k}}
    .
\end{align*}
\end{lemma}

\begin{proof}
We use the following properties of the sign of permutation:
\begin{itemize}
    \item 
$
    \sgn
        {\indexic{1},\cdots,\indexic{k-1},\order{\indexic{k},r},\indexic{k+1},\cdots,\widehat{\alpha},\cdots,\indexic{n-p}}
        {\order{(\indexIc\backslash\{\alpha\})\cup\{r\}}}
    =
    \sgn
        {\order{\indexic{k},r},\indexIc\backslash\{\indexic{k},\alpha\}}
        {\order{(\indexIc\backslash\{\alpha\})\cup\{r\}}}
$,
    \item 
$
    \sgn
        {\order{\indexic{k},r},\indexIc\backslash\{\indexic{k},\alpha\}}
        {\order{(\indexIc\backslash\{\alpha\})\cup\{r\}}}
    =
    \sgn
        {r,\indexic{k},\indexIc\backslash\{\indexic{k},\alpha\}}
        {\order{(\indexIc\backslash\{\alpha\})\cup\{r\}}}
    \sgn
        {\order{\indexic{k},r}}
        {r,\indexic{k}}
$,
    \item
$
    \sgn
        {\indexIc\backslash\{\alpha\}}
        {\indexic{k},\indexIc\backslash\{\alpha,\indexic{k}\}}
    \sgn
        {r,\indexic{k},\indexIc\backslash\{\indexic{k},\alpha\}}
        {\order{(\indexIc\backslash\{\alpha\})\cup\{r\}}}
    =
    \sgn
        {r,\indexIc\backslash\{\alpha\}}
        {\order{(\indexIc\backslash\{\alpha\})\cup\{r\}}}
$.
\end{itemize}
We simplify the left-hand side of the equation in Lemma \ref{Lemma-simplification for RHS of (alpha,I,J)-lemma 2} to
\begin{align}
    \begin{pmatrix*}[l]
        \sgn
        {1,\cdots,n,1^{'},\cdots,n^{'}}
        {\alpha,\indexI,\indexJp,\indexIc\backslash\{\alpha\},\indexJpc}
    \sgn
        {r,\indexIc\backslash\{\alpha\}}
        {\order{(\indexIc\backslash\{\alpha\})\cup\{r\}}}
    \\
    \sgn
        {\order{\indexic{k},r}}
        {r,\indexic{k}}
    \sgn
        {1^{'},\cdots,n^{'},1,\cdots,n}
        {\order{(\indexIc\backslash\{\alpha\})\cup\{r\}},\indexJpc,\alpha,\indexI\backslash\{r\},\indexJp}
    \end{pmatrix*}
    .
\label{EQ-simplification for RHS of (alpha,I,J)-lemma 2-no1}
\end{align}
We use the following properties of the sign of permutation:
\begin{itemize}
    \item 
$
    \sgn
        {1,\cdots,n,1^{'},\cdots,n^{'}}
        {\alpha,\indexI,\indexJp,\indexIc\backslash\{\alpha\},\indexJpc}
    =
    \sgn
        {r,\indexI\backslash\{r\}}
        {\indexI}
    \sgn
        {1,\cdots,n,1^{'},\cdots,n^{'}}
        {\alpha,r,\indexI\backslash\{r\},\indexJp,\indexIc\backslash\{\alpha\},\indexJpc}
$,
    \item 
$
    \sgn
        {1,\cdots,n,1^{'},\cdots,n^{'}}
        {\alpha,r,\indexI\backslash\{r\},\indexJp,\indexIc\backslash\{\alpha\},\indexJpc}
    =
    (-1)^{p-1+q}
    \sgn
        {1,\cdots,n,1^{'},\cdots,n^{'}}
        {\alpha,\indexI\backslash\{r\},\indexJp,r,\indexIc\backslash\{\alpha\},\indexJpc}
    \\
    =
    -(-1)^{p+q}
    \sgn
        {1,\cdots,n,1^{'},\cdots,n^{'}}
        {\alpha,\indexI\backslash\{r\},\indexJp,r,\indexIc\backslash\{\alpha\},\indexJpc}
$,
    \item 
$
    \sgn
        {1,\cdots,n,1^{'},\cdots,n^{'}}
        {\alpha,\indexI\backslash\{r\},\indexJp,r,\indexIc\backslash\{\alpha\},\indexJpc}
    \sgn
        {r,\indexIc\backslash\{\alpha\}}
        {\order{(\indexIc\backslash\{\alpha\})\cup\{r\}}}
    =
    \sgn
        {1,\cdots,n,1^{'},\cdots,n^{'}}
        {\alpha,\indexI\backslash\{r\},\indexJp,\order{(\indexIc\backslash\{\alpha\})\cup\{r\}},\indexJpc}
$.
\end{itemize}
We simplify (\ref{EQ-simplification for RHS of (alpha,I,J)-lemma 2-no1}) to
\begin{align}
    -(-1)^{p+q}
    \begin{pmatrix*}[l]
        \sgn
            {r,\indexI\backslash\{r\}}
            {\indexI}
        \sgn
            {1,\cdots,n,1^{'},\cdots,n^{'}}
            {\alpha,\indexI\backslash\{r\},\indexJp,\order{(\indexIc\backslash\{\alpha\})\cup\{r\}},\indexJpc}
        \\
        \sgn
            {\order{\indexic{k},r}}
            {r,\indexic{k}}
        \sgn
            {1^{'},\cdots,n^{'},1,\cdots,n}
            {\order{(\indexIc\backslash\{\alpha\})\cup\{r\}},\indexJpc,\alpha,\indexI\backslash\{r\},\indexJp}
    \end{pmatrix*}
    .
\label{EQ-simplification for RHS of (alpha,I,J)-lemma 2-no2}
\end{align}
We use the following properties of the sign of permutation:
\begin{itemize}
    \item 
$  
    \sgn
        {1,\cdots,n,1^{'},\cdots,n^{'}}
        {\alpha,\indexI\backslash\{r\},\indexJp,\order{(\indexIc\backslash\{\alpha\})\cup\{r\}},\indexJpc}
    \\
    =
    (-1)^{(p+q)(n-p+n-q)}
    \sgn
        {1,\cdots,n,1^{'},\cdots,n^{'}}
        {\order{(\indexIc\backslash\{\alpha\})\cup\{r\}},\indexJpc,\alpha,\indexI\backslash\{r\},\indexJp}
    \\
    =
    (-1)^{(p+q)^2}
    \sgn
        {1,\cdots,n,1^{'},\cdots,n^{'}}
        {\order{(\indexIc\backslash\{\alpha\})\cup\{r\}},\indexJpc,\alpha,\indexI\backslash\{r\},\indexJp}
$,
    \item 
$
    \sgn
        {1,\cdots,n,1^{'},\cdots,n^{'}}
        {\order{(\indexIc\backslash\{\alpha\})\cup\{r\}},\indexJpc,\alpha,\indexI\backslash\{r\},\indexJp}
    \sgn
        {1^{'},\cdots,n^{'},1,\cdots,n}
        {\order{(\indexIc\backslash\{\alpha\})\cup\{r\}},\indexJpc,\alpha,\indexI\backslash\{r\},\indexJp}
    \\
    =
    \sgn
        {1,\cdots,n,1^{'},\cdots,n^{'}}
        {1^{'},\cdots,n^{'},1,\cdots,n}
    =(-1)^{n^2}
$.
\end{itemize}
We simplify (\ref{EQ-simplification for RHS of (alpha,I,J)-lemma 2-no2}) to
\begin{align*}
    &
        -(-1)^{p+q}
        (-1)^{(p+q)^2}
        (-1)^{n^2}
        \sgn
            {r,\indexI\backslash\{r\}}
            {\indexI}
        \sgn
            {\order{\indexic{k},r}}
            {r,\indexic{k}}
    \\
    =&
        -(-1)^{n^2}
        \sgn
            {r,\indexI\backslash\{r\}}
            {\indexI}
        \sgn
            {\order{\indexic{k},r}}
            {r,\indexic{k}}
    ,
\end{align*}
which completes the proof of Lemma \ref{Lemma-simplification for RHS of (alpha,I,J)-lemma 2}.
\end{proof}

\begin{lemma}
\label{Lemma-simplification for RHS of (alpha,I,J)-lemma 3}
\begin{align*}
    \begin{pmatrix*}[l]
        \sgn
            {1,\cdots,n,1^{'},\cdots,n^{'}}
            {\alpha,\indexI,\indexJp,\indexIc\backslash\{\alpha\},\indexJpc}
        \sgn
            {\indexIc\backslash\{\alpha\}}
            {\indexic{k},\indexIc\backslash\{\alpha,\indexic{k}\}}
        \\
        \sgn
            {\indexic{1},\cdots,\indexic{k-1},\order{\alpha,r},\indexic{k+1},\cdots,\widehat{\alpha},\cdots,\indexic{n-p}}
            {\order{(\indexIc\backslash\{\indexic{k}\})\cup\{r\}}}
        \\
        \sgn
            {1^{'},\cdots,n^{'},1,\cdots,n}
            {\order{(\indexIc\backslash\{\indexic{k}\})\cup\{r\}},\indexJpc,\order{(\indexI\backslash\{r\})\cup\{\indexic{k}\}},\indexJp}
    \end{pmatrix*}
    =
    (-1)^{n^2}
    \begin{pmatrix*}[l]
        \sgn
            {r,\order{(\indexI\backslash\{r\})\cup\{\indexic{k}\}}}
            {\indexic{k},\indexI}
        \\
        \sgn
            {\order{\alpha,r}}
            {\alpha,r}
    \end{pmatrix*}
    .
\end{align*}
\end{lemma}
\begin{proof}
We use the following properties of the sign of permutation:
\begin{itemize}
    \item 
$
    \sgn
        {1,\cdots,n,1^{'},\cdots,n^{'}}
        {\alpha,\indexI,\indexJp,\indexIc\backslash\{\alpha\},\indexJpc}
    \sgn
        {\indexIc\backslash\{\alpha\}}
        {\indexic{k},\indexIc\backslash\{\alpha,\indexic{k}\}}
    =
    \sgn
        {1,\cdots,n,1^{'},\cdots,n^{'}}
        {\alpha,\indexI,\indexJp,\indexic{k},\indexIc\backslash\{\alpha,\indexic{k}\},\indexJpc}
$,
    \item 
$
    \sgn
        {\indexic{1},\cdots,\indexic{k-1},\order{\alpha,r},\indexic{k+1},\cdots,\widehat{\alpha},\cdots,\indexic{n-p}}
        {\order{(\indexIc\backslash\{\indexic{k}\})\cup\{r\}}}
    =
    \sgn
        {\order{\alpha,r},\indexIc\backslash\{\indexic{k},\alpha\}}
        {\order{(\indexIc\backslash\{\indexic{k}\})\cup\{r\}}}
$,
    \item 
$
    \begin{pmatrix*}[l]
        \sgn
        {\order{\alpha,r},\indexIc\backslash\{\indexic{k},\alpha\}}
        {\order{(\indexIc\backslash\{\indexic{k}\})\cup\{r\}}}
    \\
    \sgn
        {1^{'},\cdots,n^{'},1,\cdots,n}
        {\order{(\indexIc\backslash\{\indexic{k}\})\cup\{r\}},\indexJpc,\order{(\indexI\backslash\{r\})\cup\{\indexic{k}\}},\indexJp}
    \end{pmatrix*}
    \\
    =
    \sgn
        {1^{'},\cdots,n^{'},1,\cdots,n}
        {\order{\alpha,r},\indexIc\backslash\{\indexic{k},\alpha\},\indexJpc,\order{(\indexI\backslash\{r\})\cup\{\indexic{k}\}},\indexJp}
$.
\end{itemize}
We simplify the left-hand side of the equation in Lemma \ref{Lemma-simplification for RHS of (alpha,I,J)-lemma 3} to
\begin{align}
    \sgn
        {1,\cdots,n,1^{'},\cdots,n^{'}}
        {\alpha,\indexI,\indexJp,\indexic{k},\indexIc\backslash\{\alpha,\indexic{k}\},\indexJpc}
    \sgn
        {1^{'},\cdots,n^{'},1,\cdots,n}
        {\order{\alpha,r},\indexIc\backslash\{\indexic{k},\alpha\},\indexJpc,\order{(\indexI\backslash\{r\})\cup\{\indexic{k}\}},\indexJp}
    .
\label{EQ-simplification for RHS of (alpha,I,J)-lemma 3-no1}
\end{align}

We use the following properties of the sign of permutation:
\begin{itemize}
    \item 
$
    \sgn
        {1,\cdots,n,1^{'},\cdots,n^{'}}
        {\alpha,\indexI,\indexJp,\indexic{k},\indexIc\backslash\{\alpha,\indexic{k}\},\indexJpc}
    =
    \sgn
        {r,\indexI\backslash\{r\}}
        {\indexI}
    \sgn
        {1,\cdots,n,1^{'},\cdots,n^{'}}
        {\alpha,r,\indexI\backslash\{r\},\indexJp,\indexic{k},\indexIc\backslash\{\alpha,\indexic{k}\},\indexJpc}
$,
    \item 
$
    \sgn
        {1^{'},\cdots,n^{'},1,\cdots,n}
        {\order{\alpha,r},\indexIc\backslash\{\indexic{k},\alpha\},\indexJpc,\order{(\indexI\backslash\{r\})\cup\{\indexic{k}\}},\indexJp}
    \\
    =
    \sgn
        {1^{'},\cdots,n^{'},1,\cdots,n}
        {\alpha,r,\indexIc\backslash\{\indexic{k},\alpha\},\indexJpc,\order{(\indexI\backslash\{r\})\cup\{\indexic{k}\}},\indexJp}
    \sgn
        {\order{\alpha,r}}
        {\alpha,r}
$.
\end{itemize}
We simplify (\ref{EQ-simplification for RHS of (alpha,I,J)-lemma 3-no1}) to
\begin{align}
    \begin{pmatrix*}[l]
        \sgn
            {r,\indexI\backslash\{r\}}
            {\indexI}
        \sgn
            {1,\cdots,n,1^{'},\cdots,n^{'}}
            {\alpha,r,\indexI\backslash\{r\},\indexJp,\indexic{k},\indexIc\backslash\{\alpha,\indexic{k}\},\indexJpc}
        \\
        \sgn
            {1^{'},\cdots,n^{'},1,\cdots,n}
            {\alpha,r,\indexIc\backslash\{\indexic{k},\alpha\},\indexJpc,\order{(\indexI\backslash\{r\})\cup\{\indexic{k}\}},\indexJp}
        \sgn
            {\order{\alpha,r}}
            {\alpha,r}
    \end{pmatrix*}
    .
\label{EQ-simplification for RHS of (alpha,I,J)-lemma 3-no2}
\end{align}

We use the following properties of the sign of permutation:
\begin{itemize}
    \item 
$
    \sgn
        {1,\cdots,n,1^{'},\cdots,n^{'}}
        {\alpha,r,\indexI\backslash\{r\},\indexJp,\indexic{k},\indexIc\backslash\{\alpha,\indexic{k}\},\indexJpc}
    =
    (-1)^{p-1+q}
    \sgn
        {1,\cdots,n,1^{'},\cdots,n^{'}}
        {\alpha,r,\indexic{k},\indexI\backslash\{r\},\indexJp,\indexIc\backslash\{\alpha,\indexic{k}\},\indexJpc}
$,
    \item 
$
    \sgn
        {1^{'},\cdots,n^{'},1,\cdots,n}
        {\alpha,r,\indexIc\backslash\{\indexic{k},\alpha\},\indexJpc,\order{(\indexI\backslash\{r\})\cup\{\indexic{k}\}},\indexJp}
    \\
    =
    \sgn
        {1^{'},\cdots,n^{'},1,\cdots,n}
        {\alpha,r,\indexIc\backslash\{\indexic{k},\alpha\},\indexJpc,\indexic{k},\indexI\backslash\{r\},\indexJp}
    \sgn
        {\order{(\indexI\backslash\{r\})\cup\{\indexic{k}\}}}
        {\indexic{k},\indexI\backslash\{r\}}
$,
    \item 
$
    \sgn
        {1^{'},\cdots,n^{'},1,\cdots,n}
        {\alpha,r,\indexIc\backslash\{\indexic{k},\alpha\},\indexJpc,\indexic{k},\indexI\backslash\{r\},\indexJp}
    =
    (-1)^{n-p-2+n-q}
    \sgn
        {1^{'},\cdots,n^{'},1,\cdots,n}
        {\alpha,r,\indexic{k},\indexIc\backslash\{\indexic{k},\alpha\},\indexJpc,\indexI\backslash\{r\},\indexJp}
$.
\end{itemize}
We simplify (\ref{EQ-simplification for RHS of (alpha,I,J)-lemma 3-no2}) to
\begin{align}
    \begin{pmatrix*}[l]
        (-1)
        \sgn
            {r,\indexI\backslash\{r\}}
            {\indexI}
        \sgn
            {1,\cdots,n,1^{'},\cdots,n^{'}}
            {\alpha,r,\indexic{k},\indexI\backslash\{r\},\indexJp,\indexIc\backslash\{\alpha,\indexic{k}\},\indexJpc}
        \sgn
            {1^{'},\cdots,n^{'},1,\cdots,n}
            {\alpha,r,\indexic{k},\indexIc\backslash\{\indexic{k},\alpha\},\indexJpc,\indexI\backslash\{r\},\indexJp}
        \\
        \sgn
            {\order{(\indexI\backslash\{r\})\cup\{\indexic{k}\}}}
            {\indexic{k},\indexI\backslash\{r\}}
        \sgn
            {\order{\alpha,r}}
            {\alpha,r}
    \end{pmatrix*}
    .
\label{EQ-simplification for RHS of (alpha,I,J)-lemma 3-no3}
\end{align}
We use the following properties of the sign of permutation:
\begin{itemize}
    \item 
$
    \sgn
        {1^{'},\cdots,n^{'},1,\cdots,n}
        {\alpha,r,\indexic{k},\indexIc\backslash\{\indexic{k},\alpha\},\indexJpc,\indexI\backslash\{r\},\indexJp}
    =
        (-1)^{(n-p-2+n-q)(p-1+q)}
        \sgn
            {1^{'},\cdots,n^{'},1,\cdots,n}
            {\alpha,r,\indexic{k},\indexI\backslash\{r\},\indexJp,\indexIc\backslash\{\indexic{k},\alpha\},\indexJpc}
    \\
    =
        (-1)^{(p+q)(p+q+1)}
        \sgn
            {1^{'},\cdots,n^{'},1,\cdots,n}
            {\alpha,r,\indexic{k},\indexI\backslash\{r\},\indexJp,\indexIc\backslash\{\indexic{k},\alpha\},\indexJpc}
    =
        \sgn
            {1^{'},\cdots,n^{'},1,\cdots,n}
            {\alpha,r,\indexic{k},\indexI\backslash\{r\},\indexJp,\indexIc\backslash\{\indexic{k},\alpha\},\indexJpc}
$,
    \item 
$
    \sgn
        {1,\cdots,n,1^{'},\cdots,n^{'}}
        {\alpha,r,\indexic{k},\indexI\backslash\{r\},\indexJp,\indexIc\backslash\{\alpha,\indexic{k}\},\indexJpc}
    \sgn
            {1^{'},\cdots,n^{'},1,\cdots,n}
            {\alpha,r,\indexic{k},\indexI\backslash\{r\},\indexJp,\indexIc\backslash\{\indexic{k},\alpha\},\indexJpc}
    \\
    =
    \sgn
        {1,\cdots,n,1^{'},\cdots,n^{'}}
        {1^{'},\cdots,n^{'},1,\cdots,n}
    =
    (-1)^{n^2}
$.
\end{itemize}
We simplify (\ref{EQ-simplification for RHS of (alpha,I,J)-lemma 3-no3}) to
\begin{align}
    -(-1)^{n^2}
    \sgn
        {r,\indexI\backslash\{r\}}
        {\indexI}
    \sgn
        {\order{(\indexI\backslash\{r\})\cup\{\indexic{k}\}}}
        {\indexic{k},\indexI\backslash\{r\}}
    \sgn
        {\order{\alpha,r}}
        {\alpha,r}
    .
\label{EQ-simplification for RHS of (alpha,I,J)-lemma 3-no4}
\end{align}

We use the following properties of the sign of permutation:
\begin{itemize}
    \item 
$
    \sgn
        {\order{(\indexI\backslash\{r\})\cup\{\indexic{k}\}}}
        {\indexic{k},\indexI\backslash\{r\}}
    =
    \sgn
        {r,\order{(\indexI\backslash\{r\})\cup\{\indexic{k}\}}}
        {r,\indexic{k},\indexI\backslash\{r\}}
    =
    (-1)
    \sgn
        {r,\order{(\indexI\backslash\{r\})\cup\{\indexic{k}\}}}
        {\indexic{k},r,\indexI\backslash\{r\}}
$,
    \item 
$
    \sgn
        {r,\indexI\backslash\{r\}}
        {\indexI}
    \sgn
        {r,\order{(\indexI\backslash\{r\})\cup\{\indexic{k}\}}}
        {\indexic{k},r,\indexI\backslash\{r\}}
    =
    \sgn
        {r,\order{(\indexI\backslash\{r\})\cup\{\indexic{k}\}}}
        {\indexic{k},\indexI}
$.
\end{itemize}
We simplify (\ref{EQ-simplification for RHS of (alpha,I,J)-lemma 3-no4}) to
\begin{align*}
    &
        -(-1)^{n^2}
        (-1)
        \sgn
            {r,\order{(\indexI\backslash\{r\})\cup\{\indexic{k}\}}}
            {\indexic{k},\indexI}
        \sgn
            {\order{\alpha,r}}
            {\alpha,r}
    \\
    =&
        (-1)^{n^2}
        \sgn
            {r,\order{(\indexI\backslash\{r\})\cup\{\indexic{k}\}}}
            {\indexic{k},\indexI}
        \sgn
            {\order{\alpha,r}}
            {\alpha,r}
    ,
\end{align*}
which completes the proof of Lemma \ref{Lemma-simplification for RHS of (alpha,I,J)-lemma 3}.
\end{proof}

\begin{lemma}
\label{Lemma-simplification for RHS of (alpha,I,J)-lemma 4}
\begin{align*}
    &
    \begin{pmatrix*}[l]
        \sgn
            {1,\cdots,n,1^{'},\cdots,n^{'}}
            {\alpha,\indexI,\indexJp,\indexIc\backslash\{\alpha\},\indexJpc}
        \sgn
            {\indexIc\backslash\{\alpha\}}
            {\indexic{k},\indexIc\backslash\{\alpha,\indexic{k}\}}
        \\
        \sgn
            {\indexic{1},\cdots,\indexic{k-1},\order{r_1,r_2},\indexic{k+1},\cdots,\widehat{\alpha},\cdots,\indexic{n-p}}
            {\order{(\indexIc\backslash\{\indexic{k},\alpha\})\cup\{r_1,r_2\}}}
        \\
        \sgn
            {1^{'},\cdots,n^{'},1,\cdots,n}
            {\order{(\indexIc\backslash\{\indexic{k},\alpha\})\cup\{r_1,r_2\}},\indexJpc,\alpha,\order{(\indexI\backslash\{r_1,r_2\})\cup\{\indexic{k}\}},\indexJp}
    \end{pmatrix*}   
    \\
    =&
    (-1)^{n^2}
    \begin{pmatrix*}[l]
        \sgn
            {\order{r_1,r_2},\indexI\backslash\{r_1,r_2\}}
            {\indexI}
        \\
        \sgn
            {\order{(\indexI\backslash\{r_1,r_2\})\cup\{\indexic{k}\}}}
            {\indexic{k},(\indexI\backslash\{r_1,r_2\})}
    \end{pmatrix*}
    .
\end{align*}
\end{lemma}
\begin{proof}
We use the following properties of the sign of permutation:
\begin{itemize}
    \item 
$
    \sgn
        {1,\cdots,n,1^{'},\cdots,n^{'}}
        {\alpha,\indexI,\indexJp,\indexIc\backslash\{\alpha\},\indexJpc}
    \sgn
        {\indexIc\backslash\{\alpha\}}
        {\indexic{k},\indexIc\backslash\{\alpha,\indexic{k}\}}
    =
    \sgn
        {1,\cdots,n,1^{'},\cdots,n^{'}}
        {\alpha,\indexI,\indexJp,\indexic{k},\indexIc\backslash\{\alpha,\indexic{k}\},\indexJpc}
$,
    \item 
$
    \sgn
        {\indexic{1},\cdots,\indexic{k-1},\order{r_1,r_2},\indexic{k+1},\cdots,\widehat{\alpha},\cdots,\indexic{n-p}}
        {\order{(\indexIc\backslash\{\indexic{k},\alpha\})\cup\{r_1,r_2\}}}
    =
    \sgn
        {\order{r_1,r_2},\indexIc\backslash\{\indexic{k},\alpha\}}
        {\order{(\indexIc\backslash\{\indexic{k},\alpha\})\cup\{r_1,r_2\}}}
$,
    \item 
$
    \sgn
        {\order{r_1,r_2},\indexIc\backslash\{\indexic{k},\alpha\}}
        {\order{(\indexIc\backslash\{\indexic{k},\alpha\})\cup\{r_1,r_2\}}}
    \sgn
        {1^{'},\cdots,n^{'},1,\cdots,n}
        {\order{(\indexIc\backslash\{\indexic{k},\alpha\})\cup\{r_1,r_2\}},\indexJpc,\alpha,\order{(\indexI\backslash\{r_1,r_2\})\cup\{\indexic{k}\}},\indexJp}
    \\
    =
    \sgn
        {1^{'},\cdots,n^{'},1,\cdots,n}
        {\order{r_1,r_2},\indexIc\backslash\{\indexic{k},\alpha\},\indexJpc,\alpha,\order{(\indexI\backslash\{r_1,r_2\})\cup\{\indexic{k}\}},\indexJp}
$.
\end{itemize}
We simplify the left-hand side of the equation in Lemma \ref{Lemma-simplification for RHS of (alpha,I,J)-lemma 4} to
\begin{align}
    \sgn
        {1,\cdots,n,1^{'},\cdots,n^{'}}
        {\alpha,\indexI,\indexJp,\indexic{k},\indexIc\backslash\{\alpha,\indexic{k}\},\indexJpc}
    \sgn
        {1^{'},\cdots,n^{'},1,\cdots,n}
        {\order{r_1,r_2},\indexIc\backslash\{\indexic{k},\alpha\},\indexJpc,\alpha,\order{(\indexI\backslash\{r_1,r_2\})\cup\{\indexic{k}\}},\indexJp}
    .
\label{EQ-simplification for RHS of (alpha,I,J)-lemma 4-no1}
\end{align}

We use the following properties of the sign of permutation:
\begin{itemize}
    \item 
$
    \sgn
    {1^{'},\cdots,n^{'},1,\cdots,n}
    {\order{r_1,r_2},\indexIc\backslash\{\indexic{k},\alpha\},\indexJpc,\alpha,\order{(\indexI\backslash\{r_1,r_2\})\cup\{\indexic{k}\}},\indexJp}
    \\
    =
    \begin{pmatrix*}[l]
        \sgn
            {1^{'},\cdots,n^{'},1,\cdots,n}
            {\order{r_1,r_2},\indexIc\backslash\{\indexic{k},\alpha\},\indexJpc,\alpha,\indexic{k},\indexI\backslash\{r_1,r_2\},\indexJp}
        \\
        \sgn
            {\order{(\indexI\backslash\{r_1,r_2\})\cup\{\indexic{k}\}}}
            {\indexic{k},\indexI\backslash\{r_1,r_2\}}
    \end{pmatrix*}
$,
    \item 
$
        \sgn
            {1^{'},\cdots,n^{'},1,\cdots,n}
            {\order{r_1,r_2},\indexIc\backslash\{\indexic{k},\alpha\},\indexJpc,\alpha,\indexic{k},\indexI\backslash\{r_1,r_2\},\indexJp}
        =
        \sgn
            {1^{'},\cdots,n^{'},1,\cdots,n}
            {\indexIc\backslash\{\indexic{k},\alpha\},\indexJpc,\alpha,\indexic{k},\order{r_1,r_2},\indexI\backslash\{r_1,r_2\},\indexJp}
        \\
        =
        \sgn
            {1^{'},\cdots,n^{'},1,\cdots,n}
            {\indexIc\backslash\{\indexic{k},\alpha\},\indexJpc,\alpha,\indexic{k},\indexI,\indexJp}
        \sgn
            {\order{r_1,r_2},\indexI\backslash\{r_1,r_2\}}
            {\indexI}  
$.
\end{itemize}
We simplify (\ref{EQ-simplification for RHS of (alpha,I,J)-lemma 4-no1}) to
\begin{align}
    \begin{pmatrix*}[l]
        \sgn
            {1,\cdots,n,1^{'},\cdots,n^{'}}
            {\alpha,\indexI,\indexJp,\indexic{k},\indexIc\backslash\{\alpha,\indexic{k}\},\indexJpc}
        \sgn
            {1^{'},\cdots,n^{'},1,\cdots,n}
            {\indexIc\backslash\{\indexic{k},\alpha\},\indexJpc,\alpha,\indexic{k},\indexI,\indexJp}
        \\
        \sgn
            {\order{r_1,r_2},\indexI\backslash\{r_1,r_2\}}
            {\indexI}
        \sgn
            {\order{(\indexI\backslash\{r_1,r_2\})\cup\{\indexic{k}\}}}
            {\indexic{k},(\indexI\backslash\{r_1,r_2\})}
    \end{pmatrix*}
    .
\label{EQ-simplification for RHS of (alpha,I,J)-lemma 4-no2}
\end{align}

We use the following properties of the sign of permutation:
\begin{itemize}
    \item 
$
    \sgn
        {1,\cdots,n,1^{'},\cdots,n^{'}}
        {\alpha,\indexI,\indexJp,\indexic{k},\indexIc\backslash\{\alpha,\indexic{k}\},\indexJpc}
    =
    (-1)^{p+q}
    \sgn
        {1,\cdots,n,1^{'},\cdots,n^{'}}
        {\alpha,\indexic{k},\indexI,\indexJp,\indexIc\backslash\{\alpha,\indexic{k}\},\indexJpc}
$,
    \item 
$
        \sgn
        {1^{'},\cdots,n^{'},1,\cdots,n}
        {\indexIc\backslash\{\indexic{k},\alpha\},\indexJpc,\alpha,\indexic{k},\indexI,\indexJp}
        =
        (-1)^{(p+q+2)(n-p-2+n-q)}
        \sgn
            {1^{'},\cdots,n^{'},1,\cdots,n}
            {\alpha,\indexic{k},\indexI,\indexJp,\indexIc\backslash\{\indexic{k},\alpha\},\indexJpc}
        \\
        =
        (-1)^{(p+q)}
        \sgn
            {1^{'},\cdots,n^{'},1,\cdots,n}
            {\alpha,\indexic{k},\indexI,\indexJp,\indexIc\backslash\{\indexic{k},\alpha\},\indexJpc}
$,
    \item
$
    \sgn
        {1,\cdots,n,1^{'},\cdots,n^{'}}
        {\alpha,\indexic{k},\indexI,\indexJp,\indexIc\backslash\{\alpha,\indexic{k}\},\indexJpc}
    \sgn
        {1^{'},\cdots,n^{'},1,\cdots,n}
        {\alpha,\indexic{k},\indexI,\indexJp,\indexIc\backslash\{\indexic{k},\alpha\},\indexJpc}
    =
    \sgn
        {1,\cdots,n,1^{'},\cdots,n^{'}}
        {1^{'},\cdots,n^{'},1,\cdots,n}
    =
    (-1)^{n^2}
$.
\end{itemize}
We simplify (\ref{EQ-simplification for RHS of (alpha,I,J)-lemma 4-no2}) to
\begin{align*}
    (-1)^{n^2}
    \sgn
        {\order{r_1,r_2},\indexI\backslash\{r_1,r_2\}}
        {\indexI}
    \sgn
        {\order{(\indexI\backslash\{r_1,r_2\})\cup\{\indexic{k}\}}}
        {\indexic{k},(\indexI\backslash\{r_1,r_2\})}
    ,
\end{align*}
which completes the proof of Lemma \ref{Lemma-simplification for RHS of (alpha,I,J)-lemma 4}.
\end{proof}

\begin{lemma}
\label{Lemma-simplification for RHS of (alpha,I,J)-lemma 5}
\begin{align*}
    \begin{pmatrix*}[l]
        \sgn
            {1,\cdots,n,1^{'},\cdots,n^{'}}
            {\alpha,\indexI,\indexJp,\indexIc\backslash\{\alpha\},\indexJpc}
        \sgn
            {\indexIc\backslash\{\alpha\},\indexJpc}
            {\indexjpc{l},\indexIc\backslash\{\alpha\},\indexJpc\backslash\{\indexjpc{l}\}}
        \\
        \sgn
            {\indexIc\backslash\{\alpha\},\indexjpc{1},\cdots,\indexjpc{l-1},\indexjpc{l},\alpha,\indexjpc{l+1},\cdots,\indexjpc{n-q}}
            {\indexIc,\indexJpc}
        \sgn
            {1^{'},\cdots,n^{'},1,\cdots,n}
            {\indexIc,\indexJpc,\indexI,\indexJp}
    \end{pmatrix*}
    =
    -(-1)^{n^2}
    .
\end{align*}
\end{lemma}

\begin{proof}
We use the following properties of the sign of permutation:
\begin{itemize}
    \item 
$
        \sgn
            {\indexIc\backslash\{\alpha\},\indexJpc}
            {\indexjpc{l},\indexIc\backslash\{\alpha\},\indexJpc\backslash\{\indexjpc{l}\}}
        =
        \sgn
            {\alpha,\indexIc\backslash\{\alpha\},\indexJpc}
            {\alpha,\indexjpc{l},\indexIc\backslash\{\alpha\},\indexJpc\backslash\{\indexjpc{l}\}}
        \\
        =
        (-1)
        \sgn
            {\alpha,\indexIc\backslash\{\alpha\},\indexJpc}
            {\indexjpc{l},\alpha,\indexIc\backslash\{\alpha\},\indexJpc\backslash\{\indexjpc{l}\}}
        =
        (-1)
        \sgn
            {\alpha,\indexIc\backslash\{\alpha\},\indexJpc}
            {\indexIc\backslash\{\alpha\},\indexjpc{1},\cdots,\indexjpc{l-1},\indexjpc{l},\alpha,\indexjpc{l+1},\cdots,\indexjpc{n-q}}
$,
    \item 
$
        \sgn
            {\alpha,\indexIc\backslash\{\alpha\},\indexJpc}
            {\indexIc\backslash\{\alpha\},\indexjpc{1},\cdots,\indexjpc{l-1},\indexjpc{l},\alpha,\indexjpc{l+1},\cdots,\indexjpc{n-q}}
        \sgn
            {\indexIc\backslash\{\alpha\},\indexjpc{1},\cdots,\indexjpc{l-1},\indexjpc{l},\alpha,\indexjpc{l+1},\cdots,\indexjpc{n-q}}
            {\indexIc,\indexJpc}
        \\
        =
        \sgn
            {\alpha,\indexIc\backslash\{\alpha\},\indexJpc}
            {\indexIc,\indexJpc}
$.
\end{itemize}
We simplify the left-hand side of the equation in Lemma \ref{Lemma-simplification for RHS of (alpha,I,J)-lemma 5} to
\begin{align}
    \sgn
        {1,\cdots,n,1^{'},\cdots,n^{'}}
        {\alpha,\indexI,\indexJp,\indexIc\backslash\{\alpha\},\indexJpc}
    (-1)
    \sgn
        {\alpha,\indexIc\backslash\{\alpha\},\indexJpc}
        {\indexIc,\indexJpc}
    \sgn
        {1^{'},\cdots,n^{'},1,\cdots,n}
        {\indexIc,\indexJpc,\indexI,\indexJp}
    .
\label{EQ-simplification for RHS of (alpha,I,J)-lemma 5-no1}
\end{align}

We use the following properties of the sign of permutation:
\begin{itemize}
    \item 
$
    \sgn
        {1,\cdots,n,1^{'},\cdots,n^{'}}
        {\alpha,\indexI,\indexJp,\indexIc\backslash\{\alpha\},\indexJpc}
    =
    (-1)^{p+q}
    \sgn
        {1,\cdots,n,1^{'},\cdots,n^{'}}
        {\indexI,\indexJp,\alpha,\indexIc\backslash\{\alpha\},\indexJpc}
$,
    \item 
$
    \sgn
        {1,\cdots,n,1^{'},\cdots,n^{'}}
        {\indexI,\indexJp,\alpha,\indexIc\backslash\{\alpha\},\indexJpc}
    \sgn
        {\alpha,\indexIc\backslash\{\alpha\},\indexJpc}
        {\indexIc,\indexJpc}
    =
    \sgn
        {1,\cdots,n,1^{'},\cdots,n^{'}}
        {\indexI,\indexJp,\indexIc,\indexJpc}
    \\
    =
    (-1)^{(p+q)(n-p+n-q)}
    \sgn
        {1,\cdots,n,1^{'},\cdots,n^{'}}
        {\indexIc,\indexJpc,\indexI,\indexJp}
$,
    \item 
$
    \sgn
        {1,\cdots,n,1^{'},\cdots,n^{'}}
        {\indexIc,\indexJpc,\indexI,\indexJp}
    \sgn
        {1^{'},\cdots,n^{'},1,\cdots,n}
        {\indexIc,\indexJpc,\indexI,\indexJp}
    =
    \sgn
        {1,\cdots,n,1^{'},\cdots,n^{'}}
        {1^{'},\cdots,n^{'},1,\cdots,n}
    =
    (-1)^{n^2}
$.
\end{itemize}
We simplify (\ref{EQ-simplification for RHS of (alpha,I,J)-lemma 5-no1}) to
\begin{align*}
    (-1)^{p+q}
    (-1)^{(p+q)(n-p+n-q)}
    (-1)^{n^2}
    =
    (-1)^{(p+q)(p+q+1)}
    (-1)^{n^2}
    =
    -(-1)^{n^2}
    ,
\end{align*}
which completes the proof of Lemma \ref{Lemma-simplification for RHS of (alpha,I,J)-lemma 5}.
\end{proof}

\begin{lemma}
\label{Lemma-simplification for RHS of (alpha,I,J)-lemma 6}
\begin{align*}
    \begin{pmatrix*}[l]
        \sgn
            {1,\cdots,n,1^{'},\cdots,n^{'}}
            {\alpha,\indexI,\indexJp,\indexIc\backslash\{\alpha\},\indexJpc}
        \sgn
            {\indexIc\backslash\{\alpha\},\indexJpc}
            {\indexjpc{l},\indexIc\backslash\{\alpha\},\indexJpc\backslash\{\indexjpc{l}\}}
        \\
        \sgn
            {\indexIc\backslash\{\alpha\},\indexjpc{1},\cdots,\indexjpc{l-1},\indexjpc{l},r,\indexjpc{l+1},\cdots,\indexjpc{n-q}}
            {\order{(\indexIc\backslash\{\alpha\})\cup\{r\}},\indexJpc}
        \\
        \sgn
            {1^{'},\cdots,n^{'},1,\cdots,n}
            {\order{(\indexIc\backslash\{\alpha\})\cup\{r\}},\indexJpc,\alpha,\indexI\backslash\{r\},\indexJp}
    \end{pmatrix*}
    =
    (-1)^{n^2}
    \sgn
        {r,\indexI\backslash\{r\}}
        {\indexI}
    .
\end{align*}
\end{lemma}

\begin{proof}
We use the following properties of the sign of permutation:
\begin{itemize}
    \item 
$
    \sgn
        {\indexIc\backslash\{\alpha\},\indexjpc{1},\cdots,\indexjpc{l-1},\indexjpc{l},r,\indexjpc{l+1},\cdots,\indexjpc{n-q}}
        {\order{(\indexIc\backslash\{\alpha\})\cup\{r\}},\indexJpc}
    =
    \sgn
        {\indexjpc{l},r,\indexIc\backslash\{\alpha\},\indexJpc\backslash\{\indexjpc{l}\}}
        {\order{(\indexIc\backslash\{\alpha\})\cup\{r\}},\indexJpc}
    \\
    =
    (-1)
    \sgn
        {r,\indexjpc{l},\indexIc\backslash\{\alpha\},\indexJpc\backslash\{\indexjpc{l}\}}
        {\order{(\indexIc\backslash\{\alpha\})\cup\{r\}},\indexJpc}
$,
    \item 
$
    \sgn
        {\indexIc\backslash\{\alpha\},\indexJpc}
        {\indexjpc{l},\indexIc\backslash\{\alpha\},\indexJpc\backslash\{\indexjpc{l}\}}
    \sgn
        {r,\indexjpc{l},\indexIc\backslash\{\alpha\},\indexJpc\backslash\{\indexjpc{l}\}}
        {\order{(\indexIc\backslash\{\alpha\})\cup\{r\}},\indexJpc}
    =
    \sgn
        {r,\indexIc\backslash\{\alpha\},\indexJpc}
        {\order{(\indexIc\backslash\{\alpha\})\cup\{r\}},\indexJpc}
$.
\end{itemize}
We simplify the left-hand side of the equation in Lemma \ref{Lemma-simplification for RHS of (alpha,I,J)-lemma 6} to
\begin{align}
    \sgn
        {1,\cdots,n,1^{'},\cdots,n^{'}}
        {\alpha,\indexI,\indexJp,\indexIc\backslash\{\alpha\},\indexJpc}
    (-1)
    \sgn
        {r,\indexIc\backslash\{\alpha\},\indexJpc}
        {\order{(\indexIc\backslash\{\alpha\})\cup\{r\}},\indexJpc}
    \sgn
        {1^{'},\cdots,n^{'},1,\cdots,n}
        {\order{(\indexIc\backslash\{\alpha\})\cup\{r\}},\indexJpc,\alpha,\indexI\backslash\{r\},\indexJp}
    .
\label{EQ-simplification for RHS of (alpha,I,J)-lemma 6-no1}
\end{align}

We use the following properties of the sign of permutation:
\begin{itemize}
    \item 
$
    \sgn
        {r,\indexIc\backslash\{\alpha\},\indexJpc}
        {\order{(\indexIc\backslash\{\alpha\})\cup\{r\}},\indexJpc}
    \sgn
        {1^{'},\cdots,n^{'},1,\cdots,n}
        {\order{(\indexIc\backslash\{\alpha\})\cup\{r\}},\indexJpc,\alpha,\indexI\backslash\{r\},\indexJp}
    \\
    =
    \sgn
        {1^{'},\cdots,n^{'},1,\cdots,n}
        {r,\indexIc\backslash\{\alpha\},\indexJpc,\alpha,\indexI\backslash\{r\},\indexJp}
$,
    \item 
$
    \sgn
        {1^{'},\cdots,n^{'},1,\cdots,n}
        {r,\indexIc\backslash\{\alpha\},\indexJpc,\alpha,\indexI\backslash\{r\},\indexJp}
    =
    (-1)
    \sgn
        {1^{'},\cdots,n^{'},1,\cdots,n}
        {\alpha,\indexIc\backslash\{\alpha\},\indexJpc,r,\indexI\backslash\{r\},\indexJp}
$,
    \item 
$
    \sgn
        {1^{'},\cdots,n^{'},1,\cdots,n}
        {\alpha,\indexIc\backslash\{\alpha\},\indexJpc,r,\indexI\backslash\{r\},\indexJp}
    =
    \sgn
        {1^{'},\cdots,n^{'},1,\cdots,n}
        {\alpha,\indexIc\backslash\{\alpha\},\indexJpc,\indexI,\indexJp}
    \sgn
        {r,\indexI\backslash\{r\}}
        {\indexI}
$.
\end{itemize}
We simplify (\ref{EQ-simplification for RHS of (alpha,I,J)-lemma 6-no1}) to
\begin{align}
    &
    \sgn
        {1,\cdots,n,1^{'},\cdots,n^{'}}
        {\alpha,\indexI,\indexJp,\indexIc\backslash\{\alpha\},\indexJpc}
    (-1)(-1)
    \sgn
        {1^{'},\cdots,n^{'},1,\cdots,n}
        {\alpha,\indexIc\backslash\{\alpha\},\indexJpc,\indexI,\indexJp}
    \sgn
        {r,\indexI\backslash\{r\}}
        {\indexI}
    \notag
    \\
    =&
    \sgn
        {1,\cdots,n,1^{'},\cdots,n^{'}}
        {\alpha,\indexI,\indexJp,\indexIc\backslash\{\alpha\},\indexJpc}
    \sgn
        {1^{'},\cdots,n^{'},1,\cdots,n}
        {\alpha,\indexIc\backslash\{\alpha\},\indexJpc,\indexI,\indexJp}
    \sgn
        {r,\indexI\backslash\{r\}}
        {\indexI}
    .
\label{EQ-simplification for RHS of (alpha,I,J)-lemma 6-no2}
\end{align}

We use the following properties of the sign of permutation:
\begin{itemize}
    \item 
$
    \sgn
        {1,\cdots,n,1^{'},\cdots,n^{'}}
        {\alpha,\indexI,\indexJp,\indexIc\backslash\{\alpha\},\indexJpc}
    =
    (-1)^{(p+q)(n-p-1+n-q)}
    \sgn
        {1,\cdots,n,1^{'},\cdots,n^{'}}
        {\alpha,\indexIc\backslash\{\alpha\},\indexJpc,\indexI,\indexJp}
$,
    \item 
$
    \sgn
        {1,\cdots,n,1^{'},\cdots,n^{'}}
        {\alpha,\indexIc\backslash\{\alpha\},\indexJpc,\indexI,\indexJp}
    \sgn
        {1^{'},\cdots,n^{'},1,\cdots,n}
        {\alpha,\indexIc\backslash\{\alpha\},\indexJpc,\indexI,\indexJp}
    =
    \sgn
        {1,\cdots,n,1^{'},\cdots,n^{'}}
        {1^{'},\cdots,n^{'},1,\cdots,n}
    =
    (-1)^{n^2}
$.
\end{itemize}
We simplify (\ref{EQ-simplification for RHS of (alpha,I,J)-lemma 6-no2}) to
\begin{align*}
    &
    (-1)^{(p+q)(n-p-1+n-q)}
    (-1)^{n^2}
    \sgn
        {1,\cdots,n,1^{'},\cdots,n^{'}}
        {\alpha,\indexIc\backslash\{\alpha\},\indexJpc,\indexI,\indexJp}
    \\
    =&
    (-1)^{n^2}
    \sgn
        {1,\cdots,n,1^{'},\cdots,n^{'}}
        {\alpha,\indexIc\backslash\{\alpha\},\indexJpc,\indexI,\indexJp}
    ,
\end{align*}
which completes the proof of Lemma \ref{Lemma-simplification for RHS of (alpha,I,J)-lemma 6}.
\end{proof}

\begin{lemma}
\label{Lemma-simplification for RHS of (alpha,I,J)-lemma 7}
\begin{align*}
    \begin{pmatrix*}[l]
        \sgn
            {1,\cdots,n,1^{'},\cdots,n^{'}}
            {\alpha,\indexI,\indexJp,\indexIc\backslash\{\alpha\},\indexJpc}
        \sgn
            {\indexIc\backslash\{\alpha\},\indexJpc}
            {\indexjpc{l},\indexIc\backslash\{\alpha\},\indexJpc\backslash\{\indexjpc{l}\}}
        \\
        \sgn
            {\indexIc\backslash\{\alpha\},\indexjpc{1},\cdots,\indexjpc{l-1},u^{'},\alpha,\indexjpc{l+1},\cdots,\indexjpc{n-q}}
            {\indexIc,\order{(\indexJpc\backslash\{\indexjpc{l}\})\cup\{u^{'}\}}}
        \\
        \sgn
            {1^{'},\cdots,n^{'},1,\cdots,n}
            {\indexIc,\order{(\indexJpc\backslash\{\indexjpc{l}\})\cup\{u^{'}\}},\indexI,\order{(\indexJp\backslash\{u^{'}\})\cup\{\indexjpc{l}\}}}
    \end{pmatrix*}
    =
    -(-1)^{n^2}
    \sgn
        {u^{'},\indexI,\order{(\indexJp\backslash\{u^{'}\})\cup\{\indexjpc{l}\}}}
        {\indexjpc{l},\indexI,\indexJp}
    .
\end{align*}
\end{lemma}

\begin{proof}
We use the following properties of the sign of permutation:
\begin{itemize}
    \item 
$
    \sgn
        {1,\cdots,n,1^{'},\cdots,n^{'}}
        {\alpha,\indexI,\indexJp,\indexIc\backslash\{\alpha\},\indexJpc}
    \sgn
        {\indexIc\backslash\{\alpha\},\indexJpc}
        {\indexjpc{l},\indexIc\backslash\{\alpha\},\indexJpc\backslash\{\indexjpc{l}\}}
    =
    \sgn
        {1,\cdots,n,1^{'},\cdots,n^{'}}
        {\alpha,\indexI,\indexJp,\indexjpc{l},\indexIc\backslash\{\alpha\},\indexJpc\backslash\{\indexjpc{l}\}}
$,
    \item 
$
    \sgn
        {\indexIc\backslash\{\alpha\},\indexjpc{1},\cdots,\indexjpc{l-1},u^{'},\alpha,\indexjpc{l+1},\cdots,\indexjpc{n-q}}
        {\indexIc,\order{(\indexJpc\backslash\{\indexjpc{l}\})\cup\{u^{'}\}}}
    =
    \sgn
        {u^{'},\alpha,\indexIc\backslash\{\alpha\},\indexJpc\backslash\{\indexjpc{l}\}}
        {\indexIc,\order{(\indexJpc\backslash\{\indexjpc{l}\})\cup\{u^{'}\}}}
$,
    \item 
$
    \sgn
        {u^{'},\alpha,\indexIc\backslash\{\alpha\},\indexJpc\backslash\{\indexjpc{l}\}}
        {\indexIc,\order{(\indexJpc\backslash\{\indexjpc{l}\})\cup\{u^{'}\}}}
    \sgn
        {1^{'},\cdots,n^{'},1,\cdots,n}
        {\indexIc,\order{(\indexJpc\backslash\{\indexjpc{l}\})\cup\{u^{'}\}},\indexI,\order{(\indexJp\backslash\{u^{'}\})\cup\{\indexjpc{l}\}}}
    \\
    =
    \sgn
        {1^{'},\cdots,n^{'},1,\cdots,n}
        {u^{'},\alpha,\indexIc\backslash\{\alpha\},\indexJpc\backslash\{\indexjpc{l}\},\indexI,\order{(\indexJp\backslash\{u^{'}\})\cup\{\indexjpc{l}\}}}
$.
\end{itemize}
We simplify the left-hand side of the equation in Lemma \ref{Lemma-simplification for RHS of (alpha,I,J)-lemma 7} to
\begin{align}
    \sgn
        {1,\cdots,n,1^{'},\cdots,n^{'}}
        {\alpha,\indexI,\indexJp,\indexjpc{l},\indexIc\backslash\{\alpha\},\indexJpc\backslash\{\indexjpc{l}\}}
    \sgn
        {1^{'},\cdots,n^{'},1,\cdots,n}
        {u^{'},\alpha,\indexIc\backslash\{\alpha\},\indexJpc\backslash\{\indexjpc{l}\},\indexI,\order{(\indexJp\backslash\{u^{'}\})\cup\{\indexjpc{l}\}}}
    .
\label{EQ-simplification for RHS of (alpha,I,J)-lemma 7-no1}
\end{align}

We use the following properties of the sign of permutation:
\begin{itemize}
    \item 
$
    \sgn
        {1,\cdots,n,1^{'},\cdots,n^{'}}
        {\alpha,\indexI,\indexJp,\indexjpc{l},\indexIc\backslash\{\alpha\},\indexJpc\backslash\{\indexjpc{l}\}}
    =
    (-1)^{p+q}
    \sgn
        {1,\cdots,n,1^{'},\cdots,n^{'}}
        {\alpha,\indexjpc{l},\indexI,\indexJp,\indexIc\backslash\{\alpha\},\indexJpc\backslash\{\indexjpc{l}\}}
$,
    \item 
$
    \sgn
        {1^{'},\cdots,n^{'},1,\cdots,n}
        {u^{'},\alpha,\indexIc\backslash\{\alpha\},\indexJpc\backslash\{\indexjpc{l}\},\indexI,\order{(\indexJp\backslash\{u^{'}\})\cup\{\indexjpc{l}\}}}
    \\
    =
    (-1)^{n-p+n-q-1}
    \sgn
        {1^{'},\cdots,n^{'},1,\cdots,n}
        {\alpha,\indexIc\backslash\{\alpha\},\indexJpc\backslash\{\indexjpc{l}\},u^{'},\indexI,\order{(\indexJp\backslash\{u^{'}\})\cup\{\indexjpc{l}\}}}
$,
    \item 
$
    \sgn
        {1^{'},\cdots,n^{'},1,\cdots,n}
        {\alpha,\indexIc\backslash\{\alpha\},\indexJpc\backslash\{\indexjpc{l}\},u^{'},\indexI,\order{(\indexJp\backslash\{u^{'}\})\cup\{\indexjpc{l}\}}}
    \\
    =
    \sgn
        {1^{'},\cdots,n^{'},1,\cdots,n}
        {\alpha,\indexIc\backslash\{\alpha\},\indexJpc\backslash\{\indexjpc{l}\},\indexjpc{l},\indexI,\indexJp}
    \sgn
        {u^{'},\indexI,\order{(\indexJp\backslash\{u^{'}\})\cup\{\indexjpc{l}\}}}
        {\indexjpc{l},\indexI,\indexJp}
$.
\end{itemize}
We simplify (\ref{EQ-simplification for RHS of (alpha,I,J)-lemma 7-no1}) to
\begin{align}
    &
    \begin{pmatrix*}
        (-1)^{p+q}
    \sgn
        {1,\cdots,n,1^{'},\cdots,n^{'}}
        {\alpha,\indexjpc{l},\indexI,\indexJp,\indexIc\backslash\{\alpha\},\indexJpc\backslash\{\indexjpc{l}\}}
    \\
    (-1)^{n-p+n-q-1}
    \sgn
        {1^{'},\cdots,n^{'},1,\cdots,n}
        {\alpha,\indexIc\backslash\{\alpha\},\indexJpc\backslash\{\indexjpc{l}\},\indexjpc{l},\indexI,\indexJp}
    \sgn
        {u^{'},\indexI,\order{(\indexJp\backslash\{u^{'}\})\cup\{\indexjpc{l}\}}}
        {\indexjpc{l},\indexI,\indexJp}
    \end{pmatrix*}
    \notag
    \\
    =&
    (-1)
    \sgn
        {1,\cdots,n,1^{'},\cdots,n^{'}}
        {\alpha,\indexjpc{l},\indexI,\indexJp,\indexIc\backslash\{\alpha\},\indexJpc\backslash\{\indexjpc{l}\}}
    \sgn
        {1^{'},\cdots,n^{'},1,\cdots,n}
        {\alpha,\indexIc\backslash\{\alpha\},\indexJpc\backslash\{\indexjpc{l}\},\indexjpc{l},\indexI,\indexJp}
    \sgn
        {u^{'},\indexI,\order{(\indexJp\backslash\{u^{'}\})\cup\{\indexjpc{l}\}}}
        {\indexjpc{l},\indexI,\indexJp}    
    .
\label{EQ-simplification for RHS of (alpha,I,J)-lemma 7-no2}
\end{align}

We use the following properties of the sign of permutation:
\begin{itemize}
    \item 
$
    \sgn
        {1^{'},\cdots,n^{'},1,\cdots,n}
        {\alpha,\indexIc\backslash\{\alpha\},\indexJpc\backslash\{\indexjpc{l}\},\indexjpc{l},\indexI,\indexJp}
    =
    (-1)^{(n-p-1+n-q-1)(p+q+1)}
    \sgn
        {1^{'},\cdots,n^{'},1,\cdots,n}
        {\alpha,\indexjpc{l},\indexI,\indexJp,\indexIc\backslash\{\alpha\},\indexJpc\backslash\{\indexjpc{l}\}}
$,
    \item
$
    \sgn
        {1,\cdots,n,1^{'},\cdots,n^{'}}
        {\alpha,\indexjpc{l},\indexI,\indexJp,\indexIc\backslash\{\alpha\},\indexJpc\backslash\{\indexjpc{l}\}}
    \sgn
        {1^{'},\cdots,n^{'},1,\cdots,n}
        {\alpha,\indexjpc{l},\indexI,\indexJp,\indexIc\backslash\{\alpha\},\indexJpc\backslash\{\indexjpc{l}\}}
    \\
    =
    \sgn
        {1,\cdots,n,1^{'},\cdots,n^{'}}
        {1^{'},\cdots,n^{'},1,\cdots,n}
    =
    (-1)^{n^2}
$.
\end{itemize}
We simplify (\ref{EQ-simplification for RHS of (alpha,I,J)-lemma 7-no2}) to
\begin{align*}
    &
    (-1)
    (-1)^{(p+q)(p+q+1)}
    (-1)^{n^2}
    \sgn
        {u^{'},\indexI,\order{(\indexJp\backslash\{u^{'}\})\cup\{\indexjpc{l}\}}}
        {\indexjpc{l},\indexI,\indexJp}
    \\
    =&
    -(-1)^{n^2}
    \sgn
        {u^{'},\indexI,\order{(\indexJp\backslash\{u^{'}\})\cup\{\indexjpc{l}\}}}
        {\indexjpc{l},\indexI,\indexJp}
    ,
\end{align*}
which completes the proof of Lemma \ref{Lemma-simplification for RHS of (alpha,I,J)-lemma 7}.
\end{proof}

\begin{lemma}
\label{Lemma-simplification for RHS of (alpha,I,J)-lemma 8}
\begin{align*}
    &
    \begin{pmatrix*}[l]
        \sgn
            {1,\cdots,n,1^{'},\cdots,n^{'}}
            {\alpha,\indexI,\indexJp,\indexIc\backslash\{\alpha\},\indexJpc}
        \sgn
            {\indexIc\backslash\{\alpha\},\indexJpc}
            {\indexjpc{l},\indexIc\backslash\{\alpha\},\indexJpc\backslash\{\indexjpc{l}\}}
        \\
        \sgn
            {\indexIc\backslash\{\alpha\},\indexjpc{1},\cdots,\indexjpc{l-1},u^{'},r,\indexjpc{l+1},\cdots,\indexjpc{n-q}}
            {\order{(\indexIc\backslash\{\alpha\})\cup\{r\}},\order{(\indexJpc\backslash\{\indexjpc{l}\})\cup\{u^{'}\}}}
        \\
        \sgn
            {1^{'},\cdots,n^{'},1,\cdots,n}
            {\order{(\indexIc\backslash\{\alpha\})\cup\{r\}},\order{(\indexJpc\backslash\{\indexjpc{l}\})\cup\{u^{'}\}},\alpha,\indexI\backslash\{r\},\order{(\indexJp\backslash\{u^{'}\})\cup\{\indexjpc{l}\}}}
    \end{pmatrix*}
    \\
    =&
    \begin{pmatrix*}
        (-1)^{p+1}
        (-1)^{n^2}
        \\
        \sgn
            {u^{'},r,\indexI\backslash\{r\},\indexJp\backslash\{u^{'}\}}
            {\indexI,\indexJp}
        \\
        \sgn
            {\order{(\indexJp\backslash\{u^{'}\})\cup\{\indexjpc{l}\}}}
            {\indexjpc{l},\indexJp\backslash\{u^{'}\}}
    \end{pmatrix*}
    .
\end{align*}
\end{lemma}
\begin{proof}
We use the following properties of the sign of permutation:
\begin{itemize}
    \item 
$
    \sgn
        {\indexIc\backslash\{\alpha\},\indexjpc{1},\cdots,\indexjpc{l-1},u^{'},r,\indexjpc{l+1},\cdots,\indexjpc{n-q}}
        {\order{(\indexIc\backslash\{\alpha\})\cup\{r\}},\order{(\indexJpc\backslash\{\indexjpc{l}\})\cup\{u^{'}\}}}
    =
    \sgn
        {u^{'},r,\indexIc\backslash\{\alpha\},\indexJpc\backslash\{\indexjpc{l}\}}
        {\order{(\indexIc\backslash\{\alpha\})\cup\{r\}},\order{(\indexJpc\backslash\{\indexjpc{l}\})\cup\{u^{'}\}}}
$,
    \item 
$
    \begin{pmatrix*}
        \sgn
            {u^{'},r,\indexIc\backslash\{\alpha\},\indexJpc\backslash\{\indexjpc{l}\}}
            {\order{(\indexIc\backslash\{\alpha\})\cup\{r\}},\order{(\indexJpc\backslash\{\indexjpc{l}\})\cup\{u^{'}\}}}
        \\
        \sgn
            {1^{'},\cdots,n^{'},1,\cdots,n}
            {\order{(\indexIc\backslash\{\alpha\})\cup\{r\}},\order{(\indexJpc\backslash\{\indexjpc{l}\})\cup\{u^{'}\}},\alpha,\indexI\backslash\{r\},\order{(\indexJp\backslash\{u^{'}\})\cup\{\indexjpc{l}\}}}
    \end{pmatrix*}
    \\
    =
    \sgn
        {1^{'},\cdots,n^{'},1,\cdots,n}
        {u^{'},r,\indexIc\backslash\{\alpha\},\indexJpc\backslash\{\indexjpc{l}\},\alpha,\indexI\backslash\{r\},\order{(\indexJp\backslash\{u^{'}\})\cup\{\indexjpc{l}\}}}
$.
\end{itemize}
We simplify the left-hand side of the equation in Lemma \ref{Lemma-simplification for RHS of (alpha,I,J)-lemma 8} to
\begin{align}
    \sgn
        {1,\cdots,n,1^{'},\cdots,n^{'}}
        {\alpha,\indexI,\indexJp,\indexIc\backslash\{\alpha\},\indexJpc}
    \sgn
        {\indexIc\backslash\{\alpha\},\indexJpc}
        {\indexjpc{l},\indexIc\backslash\{\alpha\},\indexJpc\backslash\{\indexjpc{l}\}}
    \sgn
        {1^{'},\cdots,n^{'},1,\cdots,n}
        {u^{'},r,\indexIc\backslash\{\alpha\},\indexJpc\backslash\{\indexjpc{l}\},\alpha,\indexI\backslash\{r\},\order{(\indexJp\backslash\{u^{'}\})\cup\{\indexjpc{l}\}}}
    .
\label{EQ-simplification for RHS of (alpha,I,J)-lemma 8-no1}
\end{align}

We use the following properties of the sign of permutation:
\begin{itemize}
    \item 
$
    \sgn
        {1^{'},\cdots,n^{'},1,\cdots,n}
        {u^{'},r,\indexIc\backslash\{\alpha\},\indexJpc\backslash\{\indexjpc{l}\},\alpha,\indexI\backslash\{r\},\order{(\indexJp\backslash\{u^{'}\})\cup\{\indexjpc{l}\}}}
    \\
    =
    \begin{pmatrix*}[l]
        \sgn
            {1^{'},\cdots,n^{'},1,\cdots,n}
            {u^{'},r,\indexIc\backslash\{\alpha\},\indexJpc\backslash\{\indexjpc{l}\},\alpha,\indexI\backslash\{r\},\indexjpc{l},\indexJp\backslash\{u^{'}\}}
        \\
        \sgn
            {\order{(\indexJp\backslash\{u^{'}\})\cup\{\indexjpc{l}\}}}
            {\indexjpc{l},\indexJp\backslash\{u^{'}\}}
    \end{pmatrix*}
$,
    \item 
$
    \sgn
        {1^{'},\cdots,n^{'},1,\cdots,n}
        {u^{'},r,\indexIc\backslash\{\alpha\},\indexJpc\backslash\{\indexjpc{l}\},\alpha,\indexI\backslash\{r\},\indexjpc{l},\indexJp\backslash\{u^{'}\}}
    \\
    =
    (-1)^{p+n-q-1+n-p-1}
    \sgn
        {1^{'},\cdots,n^{'},1,\cdots,n}
        {u^{'},r,\indexIc\backslash\{\alpha\},\indexjpc{l},\indexJpc\backslash\{\indexjpc{l}\},\alpha,\indexI\backslash\{r\},\indexJp\backslash\{u^{'}\}}
$,
    \item 
$
    \sgn
        {\indexIc\backslash\{\alpha\},\indexJpc}
        {\indexjpc{l},\indexIc\backslash\{\alpha\},\indexJpc\backslash\{\indexjpc{l}\}}
    \sgn
        {1^{'},\cdots,n^{'},1,\cdots,n}
        {u^{'},r,\indexIc\backslash\{\alpha\},\indexjpc{l},\indexJpc\backslash\{\indexjpc{l}\},\alpha,\indexI\backslash\{r\},\indexJp\backslash\{u^{'}\}}
    \\
    =
    \sgn
        {1^{'},\cdots,n^{'},1,\cdots,n}
        {u^{'},r,\indexIc\backslash\{\alpha\},\indexJpc,\alpha,\indexI\backslash\{r\},\indexJp\backslash\{u^{'}\}}
$.
\end{itemize}
We simplify (\ref{EQ-simplification for RHS of (alpha,I,J)-lemma 8-no1}) to
\begin{align}
    &
    \begin{pmatrix*}[l]
        \sgn
            {1,\cdots,n,1^{'},\cdots,n^{'}}
            {\alpha,\indexI,\indexJp,\indexIc\backslash\{\alpha\},\indexJpc}
        (-1)^{p+n-q-1+n-p-1}
        \\
        \sgn
            {1^{'},\cdots,n^{'},1,\cdots,n}
            {u^{'},r,\indexIc\backslash\{\alpha\},\indexJpc,\alpha,\indexI\backslash\{r\},\indexJp\backslash\{u^{'}\}}
        \sgn
            {\order{(\indexJp\backslash\{u^{'}\})\cup\{\indexjpc{l}\}}}
            {\indexjpc{l},\indexJp\backslash\{u^{'}\}}
    \end{pmatrix*}
    \notag
    \\
    =&
    (-1)^{q}
    \sgn
        {1,\cdots,n,1^{'},\cdots,n^{'}}
        {\alpha,\indexI,\indexJp,\indexIc\backslash\{\alpha\},\indexJpc}  
    \sgn
        {1^{'},\cdots,n^{'},1,\cdots,n}
        {u^{'},r,\indexIc\backslash\{\alpha\},\indexJpc,\alpha,\indexI\backslash\{r\},\indexJp\backslash\{u^{'}\}}
    \sgn
        {\order{(\indexJp\backslash\{u^{'}\})\cup\{\indexjpc{l}\}}}
        {\indexjpc{l},\indexJp\backslash\{u^{'}\}}
    .
\label{EQ-simplification for RHS of (alpha,I,J)-lemma 8-no2}
\end{align}

We use the following properties of the sign of permutation:
\begin{itemize}
    \item 
$
    \sgn
        {1^{'},\cdots,n^{'},1,\cdots,n}
        {u^{'},r,\indexIc\backslash\{\alpha\},\indexJpc,\alpha,\indexI\backslash\{r\},\indexJp\backslash\{u^{'}\}}
    =
    \sgn
        {1^{'},\cdots,n^{'},1,\cdots,n}
        {\indexIc\backslash\{\alpha\},\indexJpc,\alpha,u^{'},r,\indexI\backslash\{r\},\indexJp\backslash\{u^{'}\}}
$,
    \item 
$
    \sgn
        {1^{'},\cdots,n^{'},1,\cdots,n}
        {\indexIc\backslash\{\alpha\},\indexJpc,\alpha,u^{'},r,\indexI\backslash\{r\},\indexJp\backslash\{u^{'}\}}
    =
    (-1)^{n-p-1+n-q}
    \sgn
        {1^{'},\cdots,n^{'},1,\cdots,n}
        {\alpha,\indexIc\backslash\{\alpha\},\indexJpc,u^{'},r,\indexI\backslash\{r\},\indexJp\backslash\{u^{'}\}} 
$,
    \item 
$
    \sgn
        {1^{'},\cdots,n^{'},1,\cdots,n}
        {\alpha,\indexIc\backslash\{\alpha\},\indexJpc,u^{'},r,\indexI\backslash\{r\},\indexJp\backslash\{u^{'}\}} 
    =
    \sgn
        {1^{'},\cdots,n^{'},1,\cdots,n}
        {\alpha,\indexIc\backslash\{\alpha\},\indexJpc,\indexI,\indexJp}
    \sgn
        {u^{'},r,\indexI\backslash\{r\},\indexJp\backslash\{u^{'}\}}
        {\indexI,\indexJp}
$.
\end{itemize}
We simplify (\ref{EQ-simplification for RHS of (alpha,I,J)-lemma 8-no2}) to
\begin{align}
    &
        \begin{pmatrix*}[l]
            \sgn
            {1,\cdots,n,1^{'},\cdots,n^{'}}
            {\alpha,\indexI,\indexJp,\indexIc\backslash\{\alpha\},\indexJpc}
        (-1)^{q}
        (-1)^{n-p-1+n-q}
        \\
        \sgn
            {1^{'},\cdots,n^{'},1,\cdots,n}
            {\alpha,\indexIc\backslash\{\alpha\},\indexJpc,\indexI,\indexJp}
        \sgn
            {u^{'},r,\indexI\backslash\{r\},\indexJp\backslash\{u^{'}\}}
            {\indexI,\indexJp}
        \sgn
            {\order{(\indexJp\backslash\{u^{'}\})\cup\{\indexjpc{l}\}}}
            {\indexjpc{l},\indexJp\backslash\{u^{'}\}}
        \end{pmatrix*}
    \notag
    \\
    =&
        \begin{pmatrix*}[l]
            \sgn
                {1,\cdots,n,1^{'},\cdots,n^{'}}
                {\alpha,\indexI,\indexJp,\indexIc\backslash\{\alpha\},\indexJpc}
            (-1)^{p+1}
            \sgn
                {1^{'},\cdots,n^{'},1,\cdots,n}
                {\alpha,\indexIc\backslash\{\alpha\},\indexJpc,\indexI,\indexJp}
            \\
            \sgn
                {u^{'},r,\indexI\backslash\{r\},\indexJp\backslash\{u^{'}\}}
                {\indexI,\indexJp}
            \sgn
                {\order{(\indexJp\backslash\{u^{'}\})\cup\{\indexjpc{l}\}}}
                {\indexjpc{l},\indexJp\backslash\{u^{'}\}}
        \end{pmatrix*}
    .
\label{EQ-simplification for RHS of (alpha,I,J)-lemma 8-no3}
\end{align}

We use the following properties of the sign of permutation:
\begin{itemize}
    \item 
$
    \sgn
        {1^{'},\cdots,n^{'},1,\cdots,n}
        {\alpha,\indexIc\backslash\{\alpha\},\indexJpc,\indexI,\indexJp}
    =
    (-1)^{(n-p-1+n-q)(p+q)}
    \sgn
        {1^{'},\cdots,n^{'},1,\cdots,n}
        {\alpha,\indexI,\indexJp,\indexIc\backslash\{\alpha\},\indexJpc}
    \\
    =
    \sgn
        {1^{'},\cdots,n^{'},1,\cdots,n}
        {\alpha,\indexI,\indexJp,\indexIc\backslash\{\alpha\},\indexJpc}
$,
    \item 
$
    \sgn
        {1,\cdots,n,1^{'},\cdots,n^{'}}
        {\alpha,\indexI,\indexJp,\indexIc\backslash\{\alpha\},\indexJpc}
    \sgn
        {1^{'},\cdots,n^{'},1,\cdots,n}
        {\alpha,\indexI,\indexJp,\indexIc\backslash\{\alpha\},\indexJpc}
    =
    \sgn
        {1,\cdots,n,1^{'},\cdots,n^{'}}
        {1^{'},\cdots,n^{'},1,\cdots,n}
    =
    (-1)^{n^2}
$.
\end{itemize}
We simplify (\ref{EQ-simplification for RHS of (alpha,I,J)-lemma 8-no3}) to
\begin{align*}
    (-1)^{p+1}
    (-1)^{n^2}
    \sgn
        {u^{'},r,\indexI\backslash\{r\},\indexJp\backslash\{u^{'}\}}
        {\indexI,\indexJp}
    \sgn
        {\order{(\indexJp\backslash\{u^{'}\})\cup\{\indexjpc{l}\}}}
        {\indexjpc{l},\indexJp\backslash\{u^{'}\}}
    ,
\end{align*}
which completes the proof of Lemma \ref{Lemma-simplification for RHS of (alpha,I,J)-lemma 8}.
\end{proof}

\subsection{Calculation of the coefficients in \texorpdfstring{$
    \basexi{\alpha}\wedge
    \sqrt{-1}\LieagbdPartial^{*}
    (\basexi{\indexI}\wedge\basedualxi{\indexJp})
    $}{}
}
\label{apex-simplification for RHS of alpha,(I,J)}
In this section, we calculate the following signs:
\begin{itemize}
    \item 
$
    \begin{pmatrix*} 
        \sgn
            {1,\cdots,n,1^{'},\cdots,n^{'}}
            {\indexI,\indexJp,\indexIc,\indexJpc}
        \sgn
            {\indexIc}
            {\indexic{k},\indexIc\backslash\{\indexic{k}\}}
        \sgn
            {\indexic{1},\cdots,\indexic{k-1},\order{\indexic{k},r},\indexic{k+1},\cdots,\indexic{n-p}}
            {\order{\indexIc\cup\{r\}}}
        \\
        \sgn
            {1^{'},\cdots,n^{'},1,\cdots,n}
            {\order{\indexIc\cup\{r\}},\indexJpc,\indexI\backslash\{r\},\indexJp}
    \end{pmatrix*} 
$,
    \item 
$
    \begin{pmatrix*}
        \sgn
            {1,\cdots,n,1^{'},\cdots,n^{'}}
            {\indexI,\indexJp,\indexIc,\indexJpc}
        \sgn
            {\indexIc}
            {\indexic{k},\indexIc\backslash\{\indexic{k}\}}
        \sgn
            {\indexic{1},\cdots,\indexic{k-1},\order{r_1,r_2},\indexic{k+1},\cdots,\indexic{n-p}}
            {\order{(\indexIc\backslash\{\indexic{k}\})\cup\{r_1,r_2\}}}
        \\
        \sgn
            {1^{'},\cdots,n^{'},1,\cdots,n}
            {\order{(\indexIc\backslash\{\indexic{k}\})\cup\{r_1,r_2\}},\indexJpc,\order{(\indexI\backslash\{r_1,r_2\})\cup\{\indexic{k}\}},\indexJp}
    \end{pmatrix*}
$,
    \item 
$
    \begin{pmatrix*}
            \sgn
                {1,\cdots,n,1^{'},\cdots,n^{'}}
                {\indexI,\indexJp,\indexIc,\indexJpc}
            \sgn
                {\indexIc,\indexJpc}
                {\indexjpc{l},\indexIc,\indexJpc\backslash\{\indexjpc{l}\}}
            \\
            \sgn
                {\indexIc,\indexjpc{1},\cdots,\indexjpc{l-1},\indexjpc{l},r,\indexjpc{l+1},\cdots,\indexjpc{n-1}}
                {\order{\indexIc\cup\{r\}},\indexJpc}
            \sgn
                {1^{'},\cdots,n^{'},1,\cdots,n}
                {\order{\indexIc\cup\{r\}},\indexJpc,\indexI\backslash\{r\},\indexJp}
        \end{pmatrix*}
$,
    \item 
$
    \begin{pmatrix*}
            \sgn
                {1,\cdots,n,1^{'},\cdots,n^{'}}
                {\indexI,\indexJp,\indexIc,\indexJpc}
            \sgn
                {\indexIc,\indexJpc}
                {\indexjpc{l},\indexIc,\indexJpc\backslash\{\indexjpc{l}\}}
            \sgn
                {\indexIc,\indexjpc{1},\cdots,\indexjpc{l-1},u^{'},r,\indexjpc{l+1},\cdots,\indexjpc{n-1}}
                {\order{\indexIc\cup\{r\}},\order{(\indexJpc\backslash\{\indexjpc{l}\})\cup\{u^{'}\}}}
            \\
            \sgn
                {1^{'},\cdots,n^{'},1,\cdots,n}
                {\order{\indexIc\cup\{r\}},\order{(\indexJpc\backslash\{\indexjpc{l}\})\cup\{u^{'}\}},\indexI\backslash\{r\},\order{(\indexJp\backslash\{u^{'}\})\cup\{\indexjpc{l}\}}}
        \end{pmatrix*}
$,
\end{itemize}
which are from (\ref{EQ-For alpha, (I ,J), calculation for RHS-step 3,unsimplified close to the target}) in Section \ref{subsection-For (alpha, I, J), calculation for RHS}. We need these signs to simplify the expression of $
\star\,\LieagbdPartialbar\,\star
(\basexi{\indexI}\wedge\basedualxi{\indexJp})
$ and compute $
\basexi{\alpha}\wedge
\sqrt{-1}\LieagbdPartial^{*}
(\basexi{\indexI}\wedge\basedualxi{\indexJp})
$.

\begin{lemma}
\label{Lemma-simplification for RHS of alpha,(I,J)-lemma 1}
    \begin{align*}
        \begin{pmatrix*} 
            \sgn
                {1,\cdots,n,1^{'},\cdots,n^{'}}
                {\indexI,\indexJp,\indexIc,\indexJpc}
            \sgn
                {\indexIc}
                {\indexic{k},\indexIc\backslash\{\indexic{k}\}}
            \\
            \sgn
                {\indexic{1},\cdots,\indexic{k-1},\order{\indexic{k},r},\indexic{k+1},\cdots,\indexic{n-p}}
                {\order{\indexIc\cup\{r\}}}
            \\
            \sgn
                {1^{'},\cdots,n^{'},1,\cdots,n}
                {\order{\indexIc\cup\{r\}},\indexJpc,\indexI\backslash\{r\},\indexJp}
        \end{pmatrix*} 
        =
        (-1)^{n^2}
        \sgn
            {\order{\indexic{k},r}}
            {r,\indexic{k}}
        \sgn
            {r,\indexI\backslash\{r\}}
            {\indexI}.
    \end{align*}
\end{lemma}
\begin{proof}
We use the following properties of the sign of permutation:
\begin{itemize}
    \item 
$
        \sgn
            {\indexic{1},\cdots,\indexic{k-1},\order{\indexic{k},r},\indexic{k+1},\cdots,\indexic{n-p}}
            {\order{\indexIc\cup\{r\}}}
        =
        \sgn
            {\order{\indexic{k},r},\indexIc\backslash\{\indexic{k}\}}
            {\order{\indexIc\cup\{r\}}}    
$,
    \item 
$
        \sgn
            {\order{\indexic{k},r},\indexIc\backslash\{\indexic{k}\}}
            {\order{\indexIc\cup\{r\}}}
        =
        \sgn
            {\order{\indexic{k},r}}
            {r,\indexic{k}}
        \sgn
            {r,\indexic{k},\indexIc\backslash\{\indexic{k}\}}
            {\order{\indexIc\cup\{r\}}}
$,
    \item 
$
        \sgn
            {r,\indexic{k},\indexIc\backslash\{\indexic{k}\}}
            {\order{\indexIc\cup\{r\}}}
        \sgn
            {1^{'},\cdots,n^{'},1,\cdots,n}
            {\order{\indexIc\cup\{r\}},\indexJpc,\indexI\backslash\{r\},\indexJp}
        =
        \sgn
            {1^{'},\cdots,n^{'},1,\cdots,n}
            {r,\indexic{k},\indexIc\backslash\{\indexic{k}\},\indexJpc,\indexI\backslash\{r\},\indexJp}
$.
\end{itemize}
We simplify the left-hand side of the equation in Lemma \ref{Lemma-simplification for RHS of alpha,(I,J)-lemma 1} to
\begin{align}
    \sgn
        {1,\cdots,n,1^{'},\cdots,n^{'}}
        {\indexI,\indexJp,\indexIc,\indexJpc}
    \sgn
        {\indexIc}
        {\indexic{k},\indexIc\backslash\{\indexic{k}\}}
    \sgn
        {\order{\indexic{k},r}}
        {r,\indexic{k}}
    \sgn
        {1^{'},\cdots,n^{'},1,\cdots,n}
        {r,\indexic{k},\indexIc\backslash\{\indexic{k}\},\indexJpc,\indexI\backslash\{r\},\indexJp}
    .
    \label{EQ-simplification for RHS of alpha,(I,J)-lemma 1-no1}
\end{align}

We use the following properties of the sign of permutation:
\begin{itemize}
    \item 
$
        \sgn
            {\indexIc}
            {\indexic{k},\indexIc\backslash\{\indexic{k}\}}
        \sgn
            {1^{'},\cdots,n^{'},1,\cdots,n}
            {r,\indexic{k},\indexIc\backslash\{\indexic{k}\},\indexJpc,\indexI\backslash\{r\},\indexJp}
        =
        \sgn
            {1^{'},\cdots,n^{'},1,\cdots,n}
            {r,\indexIc,\indexJpc,\indexI\backslash\{r\},\indexJp}
$,
    \item
$
        \sgn
            {1^{'},\cdots,n^{'},1,\cdots,n}
            {r,\indexIc,\indexJpc,\indexI\backslash\{r\},\indexJp}
        =
        (-1)^{n-p+n-q}
        \sgn
            {1^{'},\cdots,n^{'},1,\cdots,n}
            {\indexIc,\indexJpc,r,\indexI\backslash\{r\},\indexJp}
        =
        (-1)^{p+q}
        \sgn
            {1^{'},\cdots,n^{'},1,\cdots,n}
            {\indexIc,\indexJpc,r,\indexI\backslash\{r\},\indexJp},
$
    \item 
$
        \sgn
            {1^{'},\cdots,n^{'},1,\cdots,n}
            {\indexIc,\indexJpc,r,\indexI\backslash\{r\},\indexJp}
        =
        \sgn
            {r,\indexI\backslash\{r\}}
            {\indexI}
        \sgn
            {1^{'},\cdots,n^{'},1,\cdots,n}
            {\indexIc,\indexJpc,\indexI,\indexJp}
$.
\end{itemize}
We simplify (\ref{EQ-simplification for RHS of alpha,(I,J)-lemma 1-no1}) to
\begin{align}
    \sgn
        {1,\cdots,n,1^{'},\cdots,n^{'}}
        {\indexI,\indexJp,\indexIc,\indexJpc}
    \sgn
        {\order{\indexic{k},r}}
        {r,\indexic{k}}
    (-1)^{p+q}
    \sgn
        {r,\indexI\backslash\{r\}}
        {\indexI}
    \sgn
        {1^{'},\cdots,n^{'},1,\cdots,n}
        {\indexIc,\indexJpc,\indexI,\indexJp}
    .
\label{EQ-simplification for RHS of alpha,(I,J)-lemma 1-no2}
\end{align}

We use the following properties of the sign of permutation:
\begin{itemize}
    \item 
$
        \sgn
            {1^{'},\cdots,n^{'},1,\cdots,n}
            {\indexIc,\indexJpc,\indexI,\indexJp}
        =
        (-1)^{(n-p+n-q)(p+q)}
        \sgn
            {1^{'},\cdots,n^{'},1,\cdots,n}
            {\indexI,\indexJp,\indexIc,\indexJpc}
        =
        (-1)^{(p+q)^2}
        \sgn
            {1^{'},\cdots,n^{'},1,\cdots,n}
            {\indexI,\indexJp,\indexIc,\indexJpc}
$,
    \item 
$
        \sgn
            {1,\cdots,n,1^{'},\cdots,n^{'}}
            {\indexI,\indexJp,\indexIc,\indexJpc}
        \sgn
            {1^{'},\cdots,n^{'},1,\cdots,n}
            {\indexI,\indexJp,\indexIc,\indexJpc}
        =
        \sgn
            {1,\cdots,n,1^{'},\cdots,n^{'}}
            {1^{'},\cdots,n^{'},1,\cdots,n}
        =
        (-1)^{n^2}
$.
\end{itemize}
We simplify (\ref{EQ-simplification for RHS of alpha,(I,J)-lemma 1-no2}) to
\begin{align*}
    (-1)^{(p+q)}
    (-1)^{(p+q)^2}
    (-1)^{n^2}
    \sgn
        {\order{\indexic{k},r}}
        {r,\indexic{k}}
    \sgn
        {r,\indexI\backslash\{r\}}
        {\indexI}
    =
    (-1)^{n^2}
    \sgn
        {\order{\indexic{k},r}}
        {r,\indexic{k}}
    \sgn
        {r,\indexI\backslash\{r\}}
        {\indexI}
    ,
\end{align*}
which completes the proof of Lemma \ref{Lemma-simplification for RHS of alpha,(I,J)-lemma 1}.

\end{proof}


\begin{lemma}
\label{Lemma-simplification for RHS of alpha,(I,J)-lemma 2}
    \begin{align*}
        \begin{pmatrix*}
            \sgn
                {1,\cdots,n,1^{'},\cdots,n^{'}}
                {\indexI,\indexJp,\indexIc,\indexJpc}
            \sgn
                {\indexIc}
                {\indexic{k},\indexIc\backslash\{\indexic{k}\}}
            \\
            \sgn
                {\indexic{1},\cdots,\indexic{k-1},\order{r_1,r_2},\indexic{k+1},\cdots,\indexic{n-p}}
                {\order{(\indexIc\backslash\{\indexic{k}\})\cup\{r_1,r_2\}}}
            \\
            \sgn
                {1^{'},\cdots,n^{'},1,\cdots,n}
                {\order{(\indexIc\backslash\{\indexic{k}\})\cup\{r_1,r_2\}},\indexJpc,\order{(\indexI\backslash\{r_1,r_2\})\cup\{\indexic{k}\}},\indexJp}
        \end{pmatrix*}
        =
        \begin{pmatrix*}[l]
            -(-1)^{n^2}
            \\
            \sgn
                {\order{r_1,r_2},\indexI\backslash\{r_1,r_2\}}
                {\indexI}
            \\
            \sgn
                {\order{(\indexI\backslash\{r_1,r_2\})\cup\{\indexic{k}\}}}
                {\indexic{k},\indexI\backslash\{r_1,r_2\}}
        \end{pmatrix*}
        .
    \end{align*}
\end{lemma}

\begin{proof}
We use the following properties of the sign of permutation:
\begin{itemize}
    \item 
$
        \sgn
            {\indexic{1},\cdots,\indexic{k-1},\order{r_1,r_2},\indexic{k+1},\cdots,\indexic{n-p}}
            {\order{(\indexIc\backslash\{\indexic{k}\})\cup\{r_1,r_2\}}}
        =
        \sgn
            {\order{r_1,r_2},\indexIc\backslash\{\indexic{k}\}}
            {\order{(\indexIc\backslash\{\indexic{k}\})\cup\{r_1,r_2\}}}
$,
    \item 
$
        \sgn
            {\order{r_1,r_2},\indexIc\backslash\{\indexic{k}\}}
            {\order{(\indexIc\backslash\{\indexic{k}\})\cup\{r_1,r_2\}}}
        =
        \sgn
            {\indexic{k},\order{r_1,r_2},\indexIc\backslash\{\indexic{k}\}}
            {\indexic{k},\order{(\indexIc\backslash\{\indexic{k}\})\cup\{r_1,r_2\}}}
        \\
        =
        \sgn
            {\order{r_1,r_2},\indexic{k},\indexIc\backslash\{\indexic{k}\}}
            {\indexic{k},\order{(\indexIc\backslash\{\indexic{k}\})\cup\{r_1,r_2\}}}
$,
    \item 
$
        \sgn
            {\indexIc}
            {\indexic{k},\indexIc\backslash\{\indexic{k}\}}
        \sgn
            {\order{r_1,r_2},\indexic{k},\indexIc\backslash\{\indexic{k}\}}
            {\indexic{k},\order{(\indexIc\backslash\{\indexic{k}\})\cup\{r_1,r_2\}}}
        =
        \sgn
            {\order{r_1,r_2},\indexIc}
            {\indexic{k},\order{(\indexIc\backslash\{\indexic{k}\})\cup\{r_1,r_2\}}}
$.
\end{itemize}
We simplify the left-hand side of the equation in Lemma \ref{Lemma-simplification for RHS of alpha,(I,J)-lemma 2} to
\begin{align}
    \begin{pmatrix*}[l]
        \sgn
            {1,\cdots,n,1^{'},\cdots,n^{'}}
            {\indexI,\indexJp,\indexIc,\indexJpc}
        \sgn
            {\order{r_1,r_2},\indexIc}
            {\indexic{k},\order{(\indexIc\backslash\{\indexic{k}\})\cup\{r_1,r_2\}}}
        \\
        \sgn
            {1^{'},\cdots,n^{'},1,\cdots,n}
            {\order{(\indexIc\backslash\{\indexic{k}\})\cup\{r_1,r_2\}},\indexJpc,\order{(\indexI\backslash\{r_1,r_2\})\cup\{\indexic{k}\}},\indexJp}
    \end{pmatrix*}.
\label{EQ-simplification for RHS of alpha,(I,J)-lemma 2-no1}
\end{align}

We use the following properties of the sign of permutation:
\begin{itemize}
    \item 
$
        \sgn
            {1^{'},\cdots,n^{'},1,\cdots,n}
            {\order{(\indexIc\backslash\{\indexic{k}\})\cup\{r_1,r_2\}},\indexJpc,\order{(\indexI\backslash\{r_1,r_2\})\cup\{\indexic{k}\}},\indexJp}
        \\
        =
        \sgn
            {1^{'},\cdots,n^{'},1,\cdots,n}
            {\order{(\indexIc\backslash\{\indexic{k}\})\cup\{r_1,r_2\}},\indexJpc,\indexic{k},\indexI\backslash\{r_1,r_2\},\indexJp}
        \sgn
            {\order{(\indexI\backslash\{r_1,r_2\})\cup\{\indexic{k}\}}}
            {\indexic{k},\indexI\backslash\{r_1,r_2\}}
$,
    \item 
$
        \sgn
            {1^{'},\cdots,n^{'},1,\cdots,n}
            {\order{(\indexIc\backslash\{\indexic{k}\})\cup\{r_1,r_2\}},\indexJpc,\indexic{k},\indexI\backslash\{r_1,r_2\},\indexJp}
        \\
        =
        (-1)^{n-p+1+n-q}
        \sgn
            {1^{'},\cdots,n^{'},1,\cdots,n}
            {\indexic{k},\order{(\indexIc\backslash\{\indexic{k}\})\cup\{r_1,r_2\}},\indexJpc,\indexI\backslash\{r_1,r_2\},\indexJp}
$,
    \item 
$
        \sgn
            {\order{r_1,r_2},\indexIc}
            {\indexic{k},\order{(\indexIc\backslash\{\indexic{k}\})\cup\{r_1,r_2\}}}
        \sgn
            {1^{'},\cdots,n^{'},1,\cdots,n}
            {\indexic{k},\order{(\indexIc\backslash\{\indexic{k}\})\cup\{r_1,r_2\}},\indexJpc,\indexI\backslash\{r_1,r_2\},\indexJp}
        \\
        =
        \sgn
            {1^{'},\cdots,n^{'},1,\cdots,n}
            {\order{r_1,r_2},\indexIc,\indexJpc,\indexI\backslash\{r_1,r_2\},\indexJp}
$.
\end{itemize}
We simplify (\ref{EQ-simplification for RHS of alpha,(I,J)-lemma 2-no1}) to
\begin{align}
    &
        \sgn
            {1,\cdots,n,1^{'},\cdots,n^{'}}
            {\indexI,\indexJp,\indexIc,\indexJpc}
        (-1)^{n-p+1+n-q}
        \sgn
            {1^{'},\cdots,n^{'},1,\cdots,n}
            {\order{r_1,r_2},\indexIc,\indexJpc,\indexI\backslash\{r_1,r_2\},\indexJp}
        \sgn
            {\order{(\indexI\backslash\{r_1,r_2\})\cup\{\indexic{k}\}}}
            {\indexic{k},\indexI\backslash\{r_1,r_2\}}
    \notag
    \\
    =&
        (-1)^{p+q+1}
        \sgn
            {1,\cdots,n,1^{'},\cdots,n^{'}}
            {\indexI,\indexJp,\indexIc,\indexJpc}
        \sgn
            {1^{'},\cdots,n^{'},1,\cdots,n}
            {\order{r_1,r_2},\indexIc,\indexJpc,\indexI\backslash\{r_1,r_2\},\indexJp}
        \sgn
            {\order{(\indexI\backslash\{r_1,r_2\})\cup\{\indexic{k}\}}}
            {\indexic{k},\indexI\backslash\{r_1,r_2\}}
    .
\label{EQ-simplification for RHS of alpha,(I,J)-lemma 2-no2}
\end{align}

We use the following properties of the sign of permutation:
\begin{itemize}
    \item 
$
        \sgn
            {1^{'},\cdots,n^{'},1,\cdots,n}
            {\order{r_1,r_2},\indexIc,\indexJpc,\indexI\backslash\{r_1,r_2\},\indexJp}
        =
        (-1)^{(n-p+n-q)(p-2+q)}
        \sgn
            {1^{'},\cdots,n^{'},1,\cdots,n}
            {\order{r_1,r_2},\indexI\backslash\{r_1,r_2\},\indexJp,\indexIc,\indexJpc}
$,
    \item 
$
        \sgn
            {1^{'},\cdots,n^{'},1,\cdots,n}
            {\order{r_1,r_2},\indexI\backslash\{r_1,r_2\},\indexJp,\indexIc,\indexJpc}
        =
        \sgn
            {1^{'},\cdots,n^{'},1,\cdots,n}
            {\indexI,\indexJp,\indexIc,\indexJpc}
        \sgn
            {\order{r_1,r_2},\indexI\backslash\{r_1,r_2\}}
            {\indexI}
$,
    \item 
$
        \sgn
            {1,\cdots,n,1^{'},\cdots,n^{'}}
            {\indexI,\indexJp,\indexIc,\indexJpc}  
        \sgn
            {1^{'},\cdots,n^{'},1,\cdots,n}
            {\indexI,\indexJp,\indexIc,\indexJpc}
        =
        \sgn
            {1,\cdots,n,1^{'},\cdots,n^{'}}
            {1^{'},\cdots,n^{'},1,\cdots,n}
        =
        (-1)^{n^2}
$.
\end{itemize}
We simplify (\ref{EQ-simplification for RHS of alpha,(I,J)-lemma 2-no2}) to
\begin{align*}
    &
        (-1)^{p+q+1}
        (-1)^{(n-p+n-q)(p-2+q)}
        (-1)^{n^2}
        \sgn
                {\order{r_1,r_2},\indexI\backslash\{r_1,r_2\}}
                {\indexI}
        \sgn
            {\order{(\indexI\backslash\{r_1,r_2\})\cup\{\indexic{k}\}}}
            {\indexic{k},\indexI\backslash\{r_1,r_2\}}
    \\
    =&
        -(-1)^{p+q}
        (-1)^{(p+q)^2}
        (-1)^{n^2}
        \sgn
                {\order{r_1,r_2},\indexI\backslash\{r_1,r_2\}}
                {\indexI}
        \sgn
            {\order{(\indexI\backslash\{r_1,r_2\})\cup\{\indexic{k}\}}}
            {\indexic{k},\indexI\backslash\{r_1,r_2\}}
    \\
    =&
        -(-1)^{(p+q)(p+q+1)}
        (-1)^{n^2}
        \sgn
                {\order{r_1,r_2},\indexI\backslash\{r_1,r_2\}}
                {\indexI}
        \sgn
            {\order{(\indexI\backslash\{r_1,r_2\})\cup\{\indexic{k}\}}}
            {\indexic{k},\indexI\backslash\{r_1,r_2\}}
    \\
    =&
        -(-1)^{n^2}
        \sgn
                {\order{r_1,r_2},\indexI\backslash\{r_1,r_2\}}
                {\indexI}
        \sgn
            {\order{(\indexI\backslash\{r_1,r_2\})\cup\{\indexic{k}\}}}
            {\indexic{k},\indexI\backslash\{r_1,r_2\}}
    ,
\end{align*}
which completes the proof of Lemma \ref{Lemma-simplification for RHS of alpha,(I,J)-lemma 2}.
\end{proof}


\begin{lemma}
\label{Lemma-simplification for RHS of alpha,(I,J)-lemma 3}
    \begin{align*}
        \begin{pmatrix*}
            \sgn
                {1,\cdots,n,1^{'},\cdots,n^{'}}
                {\indexI,\indexJp,\indexIc,\indexJpc}
            \sgn
                {\indexIc,\indexJpc}
                {\indexjpc{l},\indexIc,\indexJpc\backslash\{\indexjpc{l}\}}
            \\
            \sgn
                {\indexIc,\indexjpc{1},\cdots,\indexjpc{l-1},\indexjpc{l},r,\indexjpc{l+1},\cdots,\indexjpc{n-1}}
                {\order{\indexIc\cup\{r\}},\indexJpc}
            \sgn
                {1^{'},\cdots,n^{'},1,\cdots,n}
                {\order{\indexIc\cup\{r\}},\indexJpc,\indexI\backslash\{r\},\indexJp}
        \end{pmatrix*}
        =
        -(-1)^{n^2}
        \sgn
            {r,\indexI\backslash\{r\}}
            {\indexI}.
    \end{align*}
\end{lemma}

\begin{proof}
We use the following properties of the sign of permutation:
\begin{itemize}
    \item 
$
        \sgn
            {\indexIc,\indexJpc}
            {\indexjpc{l},\indexIc,\indexJpc\backslash\{\indexjpc{l}\}}
        =
        (-1)^{n-p+l-1}
$,
    \item 
$
        \sgn
            {\indexIc,\indexjpc{1},\cdots,\indexjpc{l-1},\indexjpc{l},r,\indexjpc{l+1},\cdots,\indexjpc{n-1}}
            {\order{\indexIc\cup\{r\}},\indexJpc}
        =
        (-1)^{n-p+l}
        \sgn
            {r,\indexIc,\indexJpc}
            {\order{\indexIc\cup\{r\}},\indexJpc}
$,
    \item 
$
        \sgn
            {r,\indexIc,\indexJpc}
            {\order{\indexIc\cup\{r\}},\indexJpc}
        \sgn
            {1^{'},\cdots,n^{'},1,\cdots,n}
            {\order{\indexIc\cup\{r\}},\indexJpc,\indexI\backslash\{r\},\indexJp}
        =
        \sgn
            {1^{'},\cdots,n^{'},1,\cdots,n}
            {r,\indexIc,\indexJpc,\indexI\backslash\{r\},\indexJp}
$.
\end{itemize}
We simplify the left-hand side of the equation in Lemma \ref{Lemma-simplification for RHS of alpha,(I,J)-lemma 3} to
\begin{align}
    &
        (-1)^{n-p+l-1}
        (-1)^{n-p+l}
        \sgn
            {1,\cdots,n,1^{'},\cdots,n^{'}}
            {\indexI,\indexJp,\indexIc,\indexJpc}
        \sgn
            {1^{'},\cdots,n^{'},1,\cdots,n}
            {r,\indexIc,\indexJpc,\indexI\backslash\{r\},\indexJp}
    \notag
    \\
    =&
        (-1)
        \sgn
            {1,\cdots,n,1^{'},\cdots,n^{'}}
            {\indexI,\indexJp,\indexIc,\indexJpc}
        \sgn
            {1^{'},\cdots,n^{'},1,\cdots,n}
            {r,\indexIc,\indexJpc,\indexI\backslash\{r\},\indexJp}
    .
\label{EQ-simplification for RHS of alpha,(I,J)-lemma 3-no1}
\end{align}

We use the following properties of the sign of permutation:
\begin{itemize}
    \item 
$
        \sgn
            {1^{'},\cdots,n^{'},1,\cdots,n}
            {r,\indexIc,\indexJpc,\indexI\backslash\{r\},\indexJp}
        =
        (-1)^{(n-p+n-q)(p-q+1)}
        \sgn
            {1^{'},\cdots,n^{'},1,\cdots,n}
            {r,\indexI\backslash\{r\},\indexJp,\indexIc,\indexJpc},
$
    \item 
$
        \sgn
            {1^{'},\cdots,n^{'},1,\cdots,n}
            {r,\indexI\backslash\{r\},\indexJp,\indexIc,\indexJpc}
        =
        (-1)^{(n-p+n-q)(p-q+1)}
        \sgn
            {1^{'},\cdots,n^{'},1,\cdots,n}
            {r,\indexI\backslash\{r\},\indexJp,\indexIc,\indexJpc}
$,
    \item 
$
        \sgn
            {1,\cdots,n,1^{'},\cdots,n^{'}}
            {\indexI,\indexJp,\indexIc,\indexJpc}
        \sgn
            {1^{'},\cdots,n^{'},1,\cdots,n}
            {\indexI,\indexJp,\indexIc,\indexJpc}
        =
        \sgn
            {1,\cdots,n,1^{'},\cdots,n^{'}}
            {1^{'},\cdots,n^{'},1,\cdots,n}
        =
        (-1)^{n^2}
$.
\end{itemize}
We simplify (\ref{EQ-simplification for RHS of alpha,(I,J)-lemma 3-no1}) to
\begin{align*}
    &
        (-1)
        (-1)^{n^2}
        (-1)^{(n-p+n-q)(p-q+1)}
        \sgn
            {r,\indexI\backslash\{r\}}
            {\indexI}
    \\
    =&
        (-1)
        (-1)^{n^2}
        (-1)^{(p+q)(p+q+1)}
        \sgn
            {r,\indexI\backslash\{r\}}
            {\indexI}
    \\
        =&
        -(-1)^{n^2}
        \sgn
            {r,\indexI\backslash\{r\}}
            {\indexI}
    ,
\end{align*}
which completes the proof of Lemma \ref{Lemma-simplification for RHS of alpha,(I,J)-lemma 3}.
\end{proof}

\begin{lemma}
\label{Lemma-simplification for RHS of alpha,(I,J)-lemma 4}
    \begin{align*}
        &
        \begin{pmatrix*}
            \sgn
                {1,\cdots,n,1^{'},\cdots,n^{'}}
                {\indexI,\indexJp,\indexIc,\indexJpc}
            \sgn
                {\indexIc,\indexJpc}
                {\indexjpc{l},\indexIc,\indexJpc\backslash\{\indexjpc{l}\}}
            \sgn
                {\indexIc,\indexjpc{1},\cdots,\indexjpc{l-1},u^{'},r,\indexjpc{l+1},\cdots,\indexjpc{n-1}}
                {\order{\indexIc\cup\{r\}},\order{(\indexJpc\backslash\{\indexjpc{l}\})\cup\{u^{'}\}}}
            \\
            \sgn
                {1^{'},\cdots,n^{'},1,\cdots,n}
                {\order{\indexIc\cup\{r\}},\order{(\indexJpc\backslash\{\indexjpc{l}\})\cup\{u^{'}\}},\indexI\backslash\{r\},\order{(\indexJp\backslash\{u^{'}\})\cup\{\indexjpc{l}\}}}
        \end{pmatrix*}
        \\
        =&
            (-1)^{p}
            (-1)^{n^2}
            \sgn
                {u^{'},r,\indexI\backslash\{r\},\indexJpc\backslash\{u^{'}\}}
                {\indexI,\indexJp}
            \sgn
                {\order{(\indexJp\backslash\{u^{'}\})\cup\{\indexjpc{l}\}}}
                {\indexjpc{l},\indexJpc\backslash\{u^{'}\}}.
    \end{align*}
\end{lemma}

\begin{proof}
We use the following properties of the sign of permutation:
\begin{itemize}
    \item 
$
        \sgn
            {\indexIc,\indexjpc{1},\cdots,\indexjpc{l-1},u^{'},r,\indexjpc{l+1},\cdots,\indexjpc{n-1}}
            {\order{\indexIc\cup\{r\}},\order{(\indexJpc\backslash\{\indexjpc{l}\})\cup\{u^{'}\}}}
        =
        \sgn
            {u^{'},r,\indexIc,\indexJpc\backslash\{\
            \indexjpc{l}\}}
            {\order{\indexIc\cup\{r\}},\order{(\indexJpc\backslash\{\indexjpc{l}\})\cup\{u^{'}\}}}
$,
    \item
$
        \sgn
            {u^{'},r,\indexIc,\indexJpc\backslash\{\
            \indexjpc{l}\}}
            {\order{\indexIc\cup\{r\}},\order{(\indexJpc\backslash\{\indexjpc{l}\})\cup\{u^{'}\}}}
        =
        \sgn
            {u^{'},r,\indexjpc{l},\indexIc,\indexJpc\backslash\{\
            \indexjpc{l}\}}
            {\indexjpc{l},\order{\indexIc\cup\{r\}},\order{(\indexJpc\backslash\{\indexjpc{l}\})\cup\{u^{'}\}}}
$,
    \item 
$
        \sgn
            {\indexIc,\indexJpc}
            {\indexjpc{l},\indexIc,\indexJpc\backslash\{\indexjpc{l}\}}
        \sgn
            {u^{'},r,\indexjpc{l},\indexIc,\indexJpc\backslash\{\
            \indexjpc{l}\}}
            {\indexjpc{l},\order{\indexIc\cup\{r\}},\order{(\indexJpc\backslash\{\indexjpc{l}\})\cup\{u^{'}\}}}
        \\
        =
        \sgn
            {u^{'},r,\indexIc,\indexJpc}
            {\indexjpc{l},\order{\indexIc\cup\{r\}},\order{(\indexJpc\backslash\{\indexjpc{l}\})\cup\{u^{'}\}}}
$.
\end{itemize}
We simplify the left-hand side of the equation in Lemma \ref{Lemma-simplification for RHS of alpha,(I,J)-lemma 4} to
\begin{align}
    \begin{pmatrix*}
        \sgn
            {1,\cdots,n,1^{'},\cdots,n^{'}}
            {\indexI,\indexJp,\indexIc,\indexJpc}
        \sgn
            {u^{'},r,\indexIc,\indexJpc}
            {\indexjpc{l},\order{\indexIc\cup\{r\}},\order{(\indexJpc\backslash\{\indexjpc{l}\})\cup\{u^{'}\}}}
        \\
        \sgn
            {1^{'},\cdots,n^{'},1,\cdots,n}
            {\order{\indexIc\cup\{r\}},\order{(\indexJpc\backslash\{\indexjpc{l}\})\cup\{u^{'}\}},\indexI\backslash\{r\},\order{(\indexJp\backslash\{u^{'}\})\cup\{\indexjpc{l}\}}}
    \end{pmatrix*}
    .
\label{EQ-simplification for RHS of alpha,(I,J)-lemma 4-no1}
\end{align}

We use the following properties of the sign of permutation:
\begin{itemize}
    \item 
$
        \sgn
            {1^{'},\cdots,n^{'},1,\cdots,n}
            {\order{\indexIc\cup\{r\}},\order{(\indexJpc\backslash\{\indexjpc{l}\})\cup\{u^{'}\}},\indexI\backslash\{r\},\order{(\indexJp\backslash\{u^{'}\})\cup\{\indexjpc{l}\}}}
        \\
        =
        \sgn
            {1^{'},\cdots,n^{'},1,\cdots,n}
            {\order{\indexIc\cup\{r\}},\order{(\indexJpc\backslash\{\indexjpc{l}\})\cup\{u^{'}\}},\indexI\backslash\{r\},\indexjpc{l},\indexJpc\backslash\{u^{'}\}}
        \sgn
            {\order{(\indexJp\backslash\{u^{'}\})\cup\{\indexjpc{l}\}}}
            {\indexjpc{l},\indexJpc\backslash\{u^{'}\}}
$,
    \item 
$
        \sgn
            {1^{'},\cdots,n^{'},1,\cdots,n}
            {\order{\indexIc\cup\{r\}},\order{(\indexJpc\backslash\{\indexjpc{l}\})\cup\{u^{'}\}},\indexI\backslash\{r\},\indexjpc{l},\indexJpc\backslash\{u^{'}\}}
        \\
        =
        (-1)^{p-1+n-q+n-p+1}
        \sgn
            {1^{'},\cdots,n^{'},1,\cdots,n}
            {\indexjpc{l},\order{\indexIc\cup\{r\}},\order{(\indexJpc\backslash\{\indexjpc{l}\})\cup\{u^{'}\}},\indexI\backslash\{r\},\indexJpc\backslash\{u^{'}\}}
$,
    \item 
$
        \begin{pmatrix*}[l]
            \sgn
            {u^{'},r,\indexIc,\indexJpc}
            {\indexjpc{l},\order{\indexIc\cup\{r\}},\order{(\indexJpc\backslash\{\indexjpc{l}\})\cup\{u^{'}\}}}
            \\
            \sgn
                {1^{'},\cdots,n^{'},1,\cdots,n}
                {\indexjpc{l},\order{\indexIc\cup\{r\}},\order{(\indexJpc\backslash\{\indexjpc{l}\})\cup\{u^{'}\}},\indexI\backslash\{r\},\indexJpc\backslash\{u^{'}\}}
        \end{pmatrix*}
        \\
        =
        \sgn
            {1^{'},\cdots,n^{'},1,\cdots,n}
            {u^{'},r,\indexIc,\indexJpc,\indexI\backslash\{r\},\indexJpc\backslash\{u^{'}\}}
$.
\end{itemize}
We simplify (\ref{EQ-simplification for RHS of alpha,(I,J)-lemma 4-no1}) to
\begin{align}
    &
        \sgn
            {1,\cdots,n,1^{'},\cdots,n^{'}}
            {\indexI,\indexJp,\indexIc,\indexJpc}
        (-1)^{p-1+n-q+n-p+1}
        \sgn
            {1^{'},\cdots,n^{'},1,\cdots,n}
            {u^{'},r,\indexIc,\indexJpc,\indexI\backslash\{r\},\indexJpc\backslash\{u^{'}\}}
        \sgn
            {\order{(\indexJp\backslash\{u^{'}\})\cup\{\indexjpc{l}\}}}
            {\indexjpc{l},\indexJpc\backslash\{u^{'}\}} 
    \notag
    \\
    =&
        \sgn
            {1,\cdots,n,1^{'},\cdots,n^{'}}
            {\indexI,\indexJp,\indexIc,\indexJpc}
        (-1)^{q}
        \sgn
            {1^{'},\cdots,n^{'},1,\cdots,n}
            {u^{'},r,\indexIc,\indexJpc,\indexI\backslash\{r\},\indexJpc\backslash\{u^{'}\}}
        \sgn
            {\order{(\indexJp\backslash\{u^{'}\})\cup\{\indexjpc{l}\}}}
            {\indexjpc{l},\indexJpc\backslash\{u^{'}\}} 
    .
\label{EQ-simplification for RHS of alpha,(I,J)-lemma 4-no2}
\end{align}

We use the following properties of the sign of permutation:
\begin{itemize}
    \item 
$
        \sgn
            {1^{'},\cdots,n^{'},1,\cdots,n}
            {u^{'},r,\indexIc,\indexJpc,\indexI\backslash\{r\},\indexJpc\backslash\{u^{'}\}}
        =
        (-1)^{(n-p+n-q)(p-1+q-1)}
        \sgn
            {1^{'},\cdots,n^{'},1,\cdots,n}
            {u^{'},r,\indexI\backslash\{r\},\indexJpc\backslash\{u^{'}\},\indexIc,\indexJpc}
$,
    \item 
$
        \sgn
            {1^{'},\cdots,n^{'},1,\cdots,n}
            {u^{'},r,\indexI\backslash\{r\},\indexJpc\backslash\{u^{'}\},\indexIc,\indexJpc}
        =
        \sgn
            {1^{'},\cdots,n^{'},1,\cdots,n}
            {\indexI,\indexJp,\indexIc,\indexJpc}
        \sgn
            {u^{'},r,\indexI\backslash\{r\},\indexJpc\backslash\{u^{'}\}}
            {\indexI,\indexJp}
$,
    \item 
$
        \sgn
            {1,\cdots,n,1^{'},\cdots,n^{'}}
            {\indexI,\indexJp,\indexIc,\indexJpc}
        \sgn
            {1^{'},\cdots,n^{'},1,\cdots,n}
            {\indexI,\indexJp,\indexIc,\indexJpc}
        =
        \sgn
            {1,\cdots,n,1^{'},\cdots,n^{'}}
            {1^{'},\cdots,n^{'},1,\cdots,n}
        =
        (-1)^{n^2}
$.
\end{itemize}
We simplify (\ref{EQ-simplification for RHS of alpha,(I,J)-lemma 4-no2}) to
\begin{align*}
    &
        (-1)^{q}
        (-1)^{(n-p+n-q)(p-1+q-1)}
        (-1)^{n^2}
        \sgn
            {u^{'},r,\indexI\backslash\{r\},\indexJpc\backslash\{u^{'}\}}
            {\indexI,\indexJp}
        \sgn
            {\order{(\indexJp\backslash\{u^{'}\})\cup\{\indexjpc{l}\}}}
            {\indexjpc{l},\indexJpc\backslash\{u^{'}\}}
    \\
    =&
        (-1)^{q}
        (-1)^{(p+q)^2}
        (-1)^{n^2}
        \sgn
            {u^{'},r,\indexI\backslash\{r\},\indexJpc\backslash\{u^{'}\}}
            {\indexI,\indexJp}
        \sgn
            {\order{(\indexJp\backslash\{u^{'}\})\cup\{\indexjpc{l}\}}}
            {\indexjpc{l},\indexJpc\backslash\{u^{'}\}}
    .
\end{align*}
Since $(-1)^{(p+q)^2}=(-1)^{p+q}$, we can simplify $(-1)^{q}(-1)^{p+q}=(-1)^{p+2q}=(-1)^{p}$, which completes the proof of Lemma \ref{Lemma-simplification for RHS of alpha,(I,J)-lemma 4}.
\end{proof}

\subsection{Calculation of the coefficients for the expression of \texorpdfstring{$
\iota_{\basedualxi{\alpha}}\circ\LieagbdPartialbar
(\basexi{\indexI} \wedge \basedualxi{\indexJp})$}{}
}
\label{apex-simplification for LHS, extra 1, part 1}
In this section, we first calculate 
\begin{itemize}
    \item 
$
        \sgn
            {\indexI}
            {\indexi{k},\indexI\backslash\{\indexi{k}\}}
        \sgn
            {\indexi{1},\cdots,\indexi{k},t^{'},\indexi{k+1},\cdots,\indexi{p},\indexJp}
            {\indexI,\order{\indexJp\cup\{t^{'}\}}}
        \sgn
            {\indexI,\order{\indexJp\cup\{t^{'}\}}}
            {\alpha,\indexI,\order{(\indexJp\backslash\{\alpha^{'}\})\cup\{t^{'}\}}}
$,
    \item 
$
        \sgn
            {\indexI}
            {\indexi{k},\indexI\backslash\{\indexi{k}\}}
        \sgn
            {\indexi{1},\cdots,\indexi{k-1},s,t^{'},\indexi{k+1},\cdots,\indexi{p},\indexJp}
            {\order{(\indexI\backslash\{\indexi{k}\})\cup\{s\}},\order{\indexJp\cup\{t^{'}\}}}
$
\\
$
        \sgn
            {\order{(\indexI\backslash\{\indexi{k}\})\cup\{s\}},\order{\indexJp\cup\{t^{'}\}}}
            {\alpha^{'},\order{(\indexI\backslash\{\indexi{k}\})\cup\{s\}},\order{(\indexJp\backslash\{\alpha^{'}\})\cup\{t^{'}\}}}
$,
    \item 
$
        \sgn
            {\indexI,\indexJp}
            {\indexjp{l},\indexI,\indexJp\backslash\{\indexjp{l}\}}
        \sgn
            {\indexjp{1},\cdots,\indexjp{l-1},\order{\indexjp{l},t^{'}},\indexjp{l+1},\cdots,\indexjp{q}}
            {\order{\indexJp\cup\{t^{'}\}}}
        \sgn
            {\indexI,\order{\indexJp\cup\{t^{'}\}}}
            {\alpha^{'},\indexI,\order{(\indexJp\backslash\{\alpha^{'}\})\cup\{t^{'}\}}}    
$,
    \item 
$
        \sgn
            {\indexI,\indexJp}
            {\indexjp{l},\indexI,\indexJp\backslash\{\indexjp{l}\}}
        \sgn
            {\indexjp{1},\cdots,\indexjp{l-1},\order{t^{'},u^{'}},\indexjp{l+1},\cdots,\indexjp{q}}
            {\order{(\indexJp\backslash\{\indexjp{l}\})\cup\{t^{'},u^{'}\}}}
        \sgn
            {\indexI,\order{(\indexJp\backslash\{\indexjp{l}\})\cup\{t^{'},u^{'}\}}}
            {\alpha^{'},\indexI,\order{(\indexJp\backslash\{\indexjp{l},\alpha^{'}\})\cup\{t^{'},u^{'}\}}}
$,
\end{itemize}
in Lemmas \ref{lemma-simplification for LHS, extra 1, part 1-lemma 1}–\ref{lemma-simplification for LHS, extra 1, part 1-lemma 3}, which are needed for (\ref{EQ-For (alpha, I, J), calculation for LHS, extra 1-unsimplified the first half, alpha in J}), the expression of $
\iota_{\basedualxi{\alpha}}\circ\LieagbdPartialbar
(\basexi{\indexI} \wedge \basedualxi{\indexJp})$ such that $\alpha\in\indexJ$.

We then simplify
\begin{itemize}
    \item 
$   
    \sgn
        {\indexI}
        {\indexi{k},\indexI\backslash\{\indexi{k}\}}
    \sgn
        {\indexi{1},\cdots,\indexi{k},\alpha^{'},\indexi{k+1},\cdots,\indexi{p},\indexJp}
        {\indexI,\order{\indexJp\cup\{\alpha^{'}\}}}
    \sgn
        {\indexI,\order{\indexJp\cup\{\alpha^{'}\}}}
        {\alpha^{'},\indexI,\indexJp}
$,
    \item 
$
    \sgn
        {\indexI}
        {\indexi{k},\indexI\backslash\{\indexi{k}\}}
    \sgn
        {\indexi{1},\cdots,\indexi{k-1},s,\alpha^{'},\indexi{k+1},\cdots,\indexi{p},\indexJp}
        {\order{(\indexI\backslash\{\indexi{k}\})\cup\{s\}},\order{\indexJp\cup\{\alpha^{'}\}}}
    \sgn
        {\order{(\indexI\backslash\{\indexi{k}\})\cup\{s\}},\order{\indexJp\cup\{\alpha^{'}\}}}
        {\alpha^{'},\order{(\indexI\backslash\{\indexi{k}\})\cup\{s\}},\indexJp}
$,
    \item 
$
    \sgn
        {\indexI,\indexJp}
        {\indexjp{l},\indexI,\indexJp\backslash\{\indexjp{l}\}}
    \sgn
        {\indexjp{1},\cdots,\indexjp{l-1},\order{\indexjp{l},\alpha^{'}},\indexjp{l+1},\cdots,\indexjp{q}}
        {\order{\indexJp\cup\{\alpha^{'}\}}}
    \sgn
        {\indexI,\order{\indexJp\cup\{\alpha^{'}\}}}
        {\alpha^{'},\indexI,\indexJp}
$,
\item 
$
    \sgn
        {\indexI,\indexJp}
        {\indexjp{l},\indexI,\indexJp\backslash\{\indexjp{l}\}}
    \sgn
        {\indexjp{1},\cdots,\indexjp{l-1},\order{t^{'},\alpha^{'}},\indexjp{l+1},\cdots,\indexjp{q}}
        {\order{(\indexJp\backslash\{\indexjp{l}\})\cup\{t^{'},\alpha^{'}\}}}
    \sgn
        {\indexI,\order{(\indexJp\backslash\{\indexjp{l}\})\cup\{t^{'},\alpha^{'}\}}}
        {\alpha^{'},\indexI,\order{(\indexJp\backslash\{\indexjp{l}\})\cup\{t^{'}\}}}
$,
\end{itemize}
in Lemmas \ref{lemma-simplification for LHS, extra 1, part 1-lemma 5}–\ref{lemma-simplification for LHS, extra 1, part 1-lemma 8}, which are needed to simplify (\ref{EQ-For (alpha, I, J), calculation for LHS, extra 1-unsimplified the first half, alpha not in J}), the expression of $
\iota_{\basedualxi{\alpha}}\circ\LieagbdPartialbar
(\basexi{\indexI} \wedge \basedualxi{\indexJp})$ such that $\alpha\notin\indexJ$.


\begin{lemma}
\begin{align*}
    \begin{pmatrix*}[l]
        \sgn
            {\indexI}
            {\indexi{k},\indexI\backslash\{\indexi{k}\}}
        \\
        \sgn
            {\indexi{1},\cdots,\indexi{k},t^{'},\indexi{k+1},\cdots,\indexi{p},\indexJp}
            {\indexI,\order{\indexJp\cup\{t^{'}\}}}
        \\
        \sgn
            {\indexI,\order{\indexJp\cup\{t^{'}\}}}
            {\alpha,\indexI,\order{(\indexJp\backslash\{\alpha^{'}\})\cup\{t^{'}\}}}
    \end{pmatrix*}
    =
    \begin{pmatrix*}
        (-1)^{k}
        \\
        \sgn
            {\indexI}
            {\indexi{k},\indexI\backslash\{\indexi{k}\}}
        \\
        \sgn
                {t^{'},\indexI,\indexJp}
                {\alpha,\indexI,\order{(\indexJp\backslash\{\alpha^{'}\})\cup\{t^{'}\}}}
    \end{pmatrix*}.
\end{align*}
\label{lemma-simplification for LHS, extra 1, part 1-lemma 1}
\end{lemma}
\begin{proof}
We use the following properties of the sign of permutation:
\begin{itemize}
    \item 
$
        \sgn
            {\indexi{1},\cdots,\indexi{k},t^{'},\indexi{k+1},\cdots,\indexi{p},\indexJp}
            {\indexI,\order{\indexJp\cup\{t^{'}\}}}
        =
        (-1)^{k}
        \sgn
            {t^{'},\indexI,\indexJp}
            {\indexI,\order{\indexJp\cup\{t^{'}\}}}
$,
    \item 
$
        \sgn
            {t^{'},\indexI,\indexJp}
            {\indexI,\order{\indexJp\cup\{t^{'}\}}}
        \sgn
            {\indexI,\order{\indexJp\cup\{t^{'}\}}}
            {\alpha,\indexI,\order{(\indexJp\backslash\{\alpha^{'}\})\cup\{t^{'}\}}}
        =
        \sgn
            {t^{'},\indexI,\indexJp}
            {\alpha,\indexI,\order{(\indexJp\backslash\{\alpha^{'}\})\cup\{t^{'}\}}}
$.
\end{itemize}
We simplify the left-hand side of the equation in Lemma \ref{lemma-simplification for LHS, extra 1, part 1-lemma 1} to
\begin{align*}
    \sgn
        {\indexI}
        {\indexi{k},\indexI\backslash\{\indexi{k}\}}
    (-1)^{k}
    \sgn
            {t^{'},\indexI,\indexJp}
            {\alpha,\indexI,\order{(\indexJp\backslash\{\alpha^{'}\})\cup\{t^{'}\}}}
    ,
\end{align*}
which completes the proof of Lemma \ref{lemma-simplification for LHS, extra 1, part 1-lemma 1}.
\end{proof}


\begin{lemma}
\begin{align*}
    &
    \begin{pmatrix*}[l]
        \sgn
            {\indexI}
            {\indexi{k},\indexI\backslash\{\indexi{k}\}}
        \sgn
            {\indexi{1},\cdots,\indexi{k-1},s,t^{'},\indexi{k+1},\cdots,\indexi{p},\indexJp}
            {\order{(\indexI\backslash\{\indexi{k}\})\cup\{s\}},\order{\indexJp\cup\{t^{'}\}}}
        \\
        \sgn
            {\order{(\indexI\backslash\{\indexi{k}\})\cup\{s\}},\order{\indexJp\cup\{t^{'}\}}}
            {\alpha^{'},\order{(\indexI\backslash\{\indexi{k}\})\cup\{s\}},\order{(\indexJp\backslash\{\alpha^{'}\})\cup\{t^{'}\}}}
    \end{pmatrix*}  
    \\
    =&
        \sgn
        {s,t^{'},\indexI,\indexJp}
        {\indexi{k},\alpha^{'},\order{(\indexI\backslash\{\indexi{k}\})\cup\{s\}},\order{(\indexJp\backslash\{\alpha^{'}\})\cup\{t^{'}\}}}.
\end{align*}
\label{lemma-simplification for LHS, extra 1, part 1-lemma 2}
\end{lemma}
\begin{proof}
We use the following properties of the sign of permutation:
\begin{itemize}
    \item 
$
        \sgn
            {\indexi{1},\cdots,\indexi{k-1},s,t^{'},\indexi{k+1},\cdots,\indexi{p},\indexJp}
            {\order{(\indexI\backslash\{\indexi{k}\})\cup\{s\}},\order{\indexJp\cup\{t^{'}\}}}
        =
        \sgn
            {s,t^{'},\indexI\backslash\{\indexi{k}\},\indexJp}
            {\order{(\indexI\backslash\{\indexi{k}\})\cup\{s\}},\order{\indexJp\cup\{t^{'}\}}}
$,
    \item 
$   
        \qquad
        \sgn
            {s,t^{'},\indexI\backslash\{\indexi{k}\},\indexJp}
            {\order{(\indexI\backslash\{\indexi{k}\})\cup\{s\}},\order{\indexJp\cup\{t^{'}\}}}
        \sgn
            {\order{(\indexI\backslash\{\indexi{k}\})\cup\{s\}},\order{\indexJp\cup\{t^{'}\}}}
            {\alpha^{'},\order{(\indexI\backslash\{\indexi{k}\})\cup\{s\}},\order{(\indexJp\backslash\{\alpha^{'}\})\cup\{t^{'}\}}}
$
\\
$
        =
        \sgn
            {s,t^{'},\indexI\backslash\{\indexi{k}\},\indexJp}
            {\alpha^{'},\order{(\indexI\backslash\{\indexi{k}\})\cup\{s\}},\order{(\indexJp\backslash\{\alpha^{'}\})\cup\{t^{'}\}}}
$.
\end{itemize}
We simplify the left-hand side of the equation in Lemma \ref{lemma-simplification for LHS, extra 1, part 1-lemma 2} to
\begin{align}
    \sgn
        {\indexI}
        {\indexi{k},\indexI\backslash\{\indexi{k}\}}
    \sgn
        {s,t^{'},\indexI\backslash\{\indexi{k}\},\indexJp}
        {\alpha^{'},\order{(\indexI\backslash\{\indexi{k}\})\cup\{s\}},\order{(\indexJp\backslash\{\alpha^{'}\})\cup\{t^{'}\}}}
    .
\label{EQ-simplification for LHS, extra 1, part 1-lemma 2-no1}
\end{align}

We use the following properties of the sign of permutation:
\begin{itemize}
    \item 
$
    \qquad
    \sgn
        {s,t^{'},\indexI\backslash\{\indexi{k}\},\indexJp}
        {\alpha^{'},\order{(\indexI\backslash\{\indexi{k}\})\cup\{s\}},\order{(\indexJp\backslash\{\alpha^{'}\})\cup\{t^{'}\}}}
$
\\
$
    =
    \sgn
        {\indexi{k},s,t^{'},\indexI\backslash\{\indexi{k}\},\indexJp}
        {\indexi{k},\alpha^{'},\order{(\indexI\backslash\{\indexi{k}\})\cup\{s\}},\order{(\indexJp\backslash\{\alpha^{'}\})\cup\{t^{'}\}}}
$
\\
$
    =
    \sgn
        {s,t^{'},\indexi{k},\indexI\backslash\{\indexi{k}\},\indexJp}
        {\indexi{k},\alpha^{'},\order{(\indexI\backslash\{\indexi{k}\})\cup\{s\}},\order{(\indexJp\backslash\{\alpha^{'}\})\cup\{t^{'}\}}}
$,
\item 
$
    \qquad
    \sgn
        {\indexI}
        {\indexi{k},\indexI\backslash\{\indexi{k}\}}
    \sgn
        {s,t^{'},\indexi{k},\indexI\backslash\{\indexi{k}\},\indexJp}
        {\indexi{k},\alpha^{'},\order{(\indexI\backslash\{\indexi{k}\})\cup\{s\}},\order{(\indexJp\backslash\{\alpha^{'}\})\cup\{t^{'}\}}}
$
\\
$
    =
    \sgn
        {s,t^{'},\indexI,\indexJp}
        {\indexi{k},\alpha^{'},\order{(\indexI\backslash\{\indexi{k}\})\cup\{s\}},\order{(\indexJp\backslash\{\alpha^{'}\})\cup\{t^{'}\}}}
$,
\end{itemize}
We simplify (\ref{EQ-simplification for LHS, extra 1, part 1-lemma 2-no1}) to
\begin{align*}
    \sgn
        {s,t^{'},\indexI,\indexJp}
        {\indexi{k},\alpha^{'},\order{(\indexI\backslash\{\indexi{k}\})\cup\{s\}},\order{(\indexJp\backslash\{\alpha^{'}\})\cup\{t^{'}\}}}
    ,
\end{align*}
which completes the proof of Lemma \ref{lemma-simplification for LHS, extra 1, part 1-lemma 2}.
\end{proof}

\begin{lemma}
\begin{align*}
    \begin{pmatrix*}[l]
        \sgn
            {\indexI,\indexJp}
            {\indexjp{l},\indexI,\indexJp\backslash\{\indexjp{l}\}}
        \\
        \sgn
            {\indexjp{1},\cdots,\indexjp{l-1},\order{\indexjp{l},t^{'}},\indexjp{l+1},\cdots,\indexjp{q}}
            {\order{\indexJp\cup\{t^{'}\}}}
        \\
        \sgn
            {\indexI,\order{\indexJp\cup\{t^{'}\}}}
            {\alpha^{'},\indexI,\order{(\indexJp\backslash\{\alpha^{'}\})\cup\{t^{'}\}}}
    \end{pmatrix*}
    =
        \sgn
            {\order{\indexjp{l},t^{'}}}
            {t^{'},\indexjp{l}}
        \sgn
            {t^{'},\indexI,\indexJp}
            {\alpha^{'},\indexI,\order{(\indexJp\backslash\{\alpha^{'}\})\cup\    {t^{'}\}}}}.
\end{align*}
\label{lemma-simplification for LHS, extra 1, part 1-lemma 3}
\end{lemma}

\begin{proof}
    We use the following properties of the sign of permutation:
\begin{itemize}
    \item 
$
        \sgn
            {\indexjp{1},\cdots,\indexjp{l-1},\order{\indexjp{l},t^{'}},\indexjp{l+1},\cdots,\indexjp{q}}
            {\order{\indexJp\cup\{t^{'}\}}}
        =
        \sgn
            {\order{\indexjp{l},t^{'}},\indexJp\backslash\{\indexjp{l}\}}
            {\order{\indexJp\cup\{t^{'}\}}}
$,
    \item 
$
        \sgn
            {\order{\indexjp{l},t^{'}},\indexJp\backslash\{\indexjp{l}\}}
            {\order{\indexJp\cup\{t^{'}\}}}
        \sgn
            {\indexI,\order{\indexJp\cup\{t^{'}\}}}
            {\alpha^{'},\indexI,\order{(\indexJp\backslash\{\alpha^{'}\})\cup\{t^{'}\}}}
        =
        \sgn
            {\indexI,\order{\indexjp{l},t^{'}},\indexJp\backslash\{\indexjp{l}\}}
            {\alpha^{'},\indexI,\order{(\indexJp\backslash\{\alpha^{'}\})\cup\{t^{'}\}}}
$.
\end{itemize}
We simplify the left-hand side of the equation in Lemma \ref{lemma-simplification for LHS, extra 1, part 1-lemma 3} to
\begin{align}
    \sgn
        {\indexI,\indexJp}
        {\indexjp{l},\indexI,\indexJp\backslash\{\indexjp{l}\}}
    \sgn
        {\indexI,\order{\indexjp{l},t^{'}},\indexJp\backslash\{\indexjp{l}\}}
        {\alpha^{'},\indexI,\order{(\indexJp\backslash\{\alpha^{'}\})\cup\{t^{'}\}}}
    .
\label{EQ-simplification for LHS, extra 1, part 1-lemma 3-no1}
\end{align}

We use the following properties of the sign of permutation:
\begin{itemize}
    \item 
$
        \sgn
            {\indexI,\order{\indexjp{l},t^{'}},\indexJp\backslash\{\indexjp{l}\}}
            {\alpha^{'},\indexI,\order{(\indexJp\backslash\{\alpha^{'}\})\cup\{t^{'}\}}}
        =
        \sgn
            {\order{\indexjp{l},t^{'}},\indexI,\indexJp\backslash\{\indexjp{l}\}}
            {\alpha^{'},\indexI,\order{(\indexJp\backslash\{\alpha^{'}\})\cup\{t^{'}\}}}
$,
\item 
$
        \sgn
            {\order{\indexjp{l},t^{'}},\indexI,\indexJp\backslash\{\indexjp{l}\}}
            {\alpha^{'},\indexI,\order{(\indexJp\backslash\{\alpha^{'}\})\cup\{t^{'}\}}}
        =
        \sgn
            {\order{\indexjp{l},t^{'}}}
            {t^{'},\indexjp{l}}
        \sgn
            {t^{'},\indexjp{l},\indexI,\indexJp\backslash\{\indexjp{l}\}}
            {\alpha^{'},\indexI,\order{(\indexJp\backslash\{\alpha^{'}\})\cup\{t^{'}\}}}
$,
\item 
$
        \sgn
            {\indexI,\indexJp}
            {\indexjp{l},\indexI,\indexJp\backslash\{\indexjp{l}\}}
        \sgn
            {t^{'},\indexjp{l},\indexI,\indexJp\backslash\{\indexjp{l}\}}
            {\alpha^{'},\indexI,\order{(\indexJp\backslash\{\alpha^{'}\})\cup\{t^{'}\}}}
        =
        \sgn
            {t^{'},\indexI,\indexJp}
            {\alpha^{'},\indexI,\order{(\indexJp\backslash\{\alpha^{'}\})\cup\{t^{'}\}}}
$.
\end{itemize}
We simplify (\ref{EQ-simplification for LHS, extra 1, part 1-lemma 3-no1}) to
\begin{align*}
    \sgn
        {\order{\indexjp{l},t^{'}}}
        {t^{'},\indexjp{l}}
    \sgn
        {t^{'},\indexI,\indexJp}
        {\alpha^{'},\indexI,\order{(\indexJp\backslash\{\alpha^{'}\})\cup\{t^{'}\}}}
    ,
\end{align*}
which completes the proof of Lemma \ref{lemma-simplification for LHS, extra 1, part 1-lemma 3}.
\end{proof}

\begin{lemma}
\begin{align*}
    \begin{pmatrix*}[l]
        \sgn
            {\indexI,\indexJp}
            {\indexjp{l},\indexI,\indexJp\backslash\{\indexjp{l}\}}
        \\
        \sgn
            {\indexjp{1},\cdots,\indexjp{l-1},\order{t^{'},u^{'}},\indexjp{l+1},\cdots,\indexjp{q}}
            {\order{(\indexJp\backslash\{\indexjp{l}\})\cup\{t^{'},u^{'}\}}}
        \\
        \sgn
            {\indexI,\order{(\indexJp\backslash\{\indexjp{l}\})\cup\{t^{'},u^{'}\}}}
            {\alpha^{'},\indexI,\order{(\indexJp\backslash\{\indexjp{l},\alpha^{'}\})\cup\{t^{'},u^{'}\}}}
    \end{pmatrix*}
    =
    \sgn
        {\order{t^{'},u^{'}},\indexI,\indexJp}
        {\indexjp{l},\alpha^{'},\indexI,\order{(\indexJp\backslash\{\indexjp{l},\alpha^{'}\})\cup\{t^{'},u^{'}\}}}.
\end{align*}
\label{lemma-simplification for LHS, extra 1, part 1-lemma 4}
\end{lemma}
\begin{proof}
    We use the following properties of the sign of permutation:
\begin{itemize}
    \item 
$
        \sgn
            {\indexjp{1},\cdots,\indexjp{l-1},\order{t^{'},u^{'}},\indexjp{l+1},\cdots,\indexjp{q}}
            {\order{(\indexJp\backslash\{\indexjp{l}\})\cup\{t^{'},u^{'}\}}}
        =
        \sgn
            {\order{t^{'},u^{'}},\indexJp\backslash\{\indexjp{l}\}}
            {\order{(\indexJp\backslash\{\indexjp{l}\})\cup\{t^{'},u^{'}\}}}
$,
    \item 
$
        \begin{pmatrix*}
            \sgn
                {\order{t^{'},u^{'}},\indexJp\backslash\{\indexjp{l}\}}
                {\order{(\indexJp\backslash\{\indexjp{l}\})\cup\{t^{'},u^{'}\}}}
            \\
            \sgn
                {\indexI,\order{(\indexJp\backslash\{\indexjp{l}\})\cup\{t^{'},u^{'}\}}}
                {\alpha^{'},\indexI,\order{(\indexJp\backslash\{\indexjp{l},\alpha^{'}\})\cup\{t^{'},u^{'}\}}}
        \end{pmatrix*}
        =
        \sgn
            {\indexI,\order{t^{'},u^{'}},\indexJp\backslash\{\indexjp{l}\}}
            {\alpha^{'},\indexI,\order{(\indexJp\backslash\{\indexjp{l},\alpha^{'}\})\cup\{t^{'},u^{'}\}}}
$.
\end{itemize}
We simplify the left-hand side of the equation in Lemma \ref{lemma-simplification for LHS, extra 1, part 1-lemma 4} to
\begin{align}
    \sgn
        {\indexI,\indexJp}
        {\indexjp{l},\indexI,\indexJp\backslash\{\indexjp{l}\}}
    \sgn
        {\indexI,\order{t^{'},u^{'}},\indexJp\backslash\{\indexjp{l}\}}
        {\alpha^{'},\indexI,\order{(\indexJp\backslash\{\indexjp{l},\alpha^{'}\})\cup\{t^{'},u^{'}\}}}
    .
\label{EQ-simplification for LHS, extra 1, part 1-lemma 4-no1}
\end{align}

We use the following properties of the sign of permutation:
\begin{itemize}
    \item 
$
        \sgn
            {\indexI,\order{t^{'},u^{'}},\indexJp\backslash\{\indexjp{l}\}}
            {\alpha^{'},\indexI,\order{(\indexJp\backslash\{\indexjp{l},\alpha^{'}\})\cup\{t^{'},u^{'}\}}}
        =
        \sgn
            {\indexjp{l},\order{t^{'},u^{'}},\indexI,\indexJp\backslash\{\indexjp{l}\}}
            {\indexjp{l},\alpha^{'},\indexI,\order{(\indexJp\backslash\{\indexjp{l},\alpha^{'}\})\cup\{t^{'},u^{'}\}}}
$,
    \item 
$
        \sgn
            {\indexjp{l},\order{t^{'},u^{'}},\indexI,\indexJp\backslash\{\indexjp{l}\}}
            {\indexjp{l},\alpha^{'},\indexI,\order{(\indexJp\backslash\{\indexjp{l},\alpha^{'}\})\cup\{t^{'},u^{'}\}}}
        =
        \sgn
            {\order{t^{'},u^{'}},\indexjp{l},\indexI,\indexJp\backslash\{\indexjp{l}\}}
            {\indexjp{l},\alpha^{'},\indexI,\order{(\indexJp\backslash\{\indexjp{l},\alpha^{'}\})\cup\{t^{'},u^{'}\}}}
$,
    \item 
$
        \sgn
            {\indexI,\indexJp}
            {\indexjp{l},\indexI,\indexJp\backslash\{\indexjp{l}\}}
        \sgn
            {\order{t^{'},u^{'}},\indexjp{l},\indexI,\indexJp\backslash\{\indexjp{l}\}}
            {\indexjp{l},\alpha^{'},\indexI,\order{(\indexJp\backslash\{\indexjp{l},\alpha^{'}\})\cup\{t^{'},u^{'}\}}}
        =
        \sgn
            {\order{t^{'},u^{'}},\indexI,\indexJp}
            {\indexjp{l},\alpha^{'},\indexI,\order{(\indexJp\backslash\{\indexjp{l},\alpha^{'}\})\cup\{t^{'},u^{'}\}}}
$
    .
\end{itemize}
We simplify (\ref{EQ-simplification for LHS, extra 1, part 1-lemma 4-no1}) to
\begin{align*}
    \sgn
        {\order{t^{'},u^{'}},\indexI,\indexJp}
        {\indexjp{l},\alpha^{'},\indexI,\order{(\indexJp\backslash\{\indexjp{l},\alpha^{'}\})\cup\{t^{'},u^{'}\}}}
    ,
\end{align*}
which completes the proof of Lemma \ref{lemma-simplification for LHS, extra 1, part 1-lemma 4}.
\end{proof}




We then simplify the coefficient in $
\iota_{\basedualxi{\alpha}}\circ\LieagbdPartialbar
(\basexi{\indexI} \wedge \basedualxi{\indexJp})$ such that $\alpha\notin\indexJ$.

\begin{lemma}
\begin{align*}
    \begin{pmatrix*}
        \sgn
            {\indexI}
            {\indexi{k},\indexI\backslash\{\indexi{k}\}}
        \\
        \sgn
            {\indexi{1},\cdots,\indexi{k},\alpha^{'},\indexi{k+1},\cdots,\indexi{p},\indexJp}
            {\indexI,\order{\indexJp\cup\{\alpha^{'}\}}}
        \\
        \sgn
            {\indexI,\order{\indexJp\cup\{\alpha^{'}\}}}
            {\alpha^{'},\indexI,\indexJp}
    \end{pmatrix*}
    =
    (-1)
    .
\end{align*}
\label{lemma-simplification for LHS, extra 1, part 1-lemma 5}
\end{lemma}

\begin{proof}
We use the following properties of the sign of permutation:
\begin{itemize}
    \item 
$
    \sgn
            {\indexI}
            {\indexi{k},\indexI\backslash\{\indexi{k}\}}
    =
    (-1)^{k-1}
$,
    \item 
$
    \sgn
        {\indexi{1},\cdots,\indexi{k},\alpha^{'},\indexi{k+1},\cdots,\indexi{p},\indexJp}
        {\indexI,\order{\indexJp\cup\{\alpha^{'}\}}}
    =
    (-1)^{k}
    \sgn
        {\indexI,\indexJp}
        {\indexI,\order{\indexJp\cup\{\alpha^{'}\}}}
$,
    \item 
$
    \sgn
        {\indexI,\indexJp}
        {\indexI,\order{\indexJp\cup\{\alpha^{'}\}}}
    \sgn
            {\indexI,\order{\indexJp\cup\{\alpha^{'}\}}}
            {\alpha^{'},\indexI,\indexJp}
    =1
$.
\end{itemize}
We simplify the left-hand side of the equation in Lemma \ref{lemma-simplification for LHS, extra 1, part 1-lemma 5} to
\begin{align*}
    (-1)^{k-1}(-1)^{k}
    =
    (-1)
    ,
\end{align*}
which completes the proof of Lemma \ref{lemma-simplification for LHS, extra 1, part 1-lemma 5}.
\end{proof}


\begin{lemma}
\begin{align*}
    \begin{pmatrix*}[l]
        \sgn
            {\indexI}
            {\indexi{k},\indexI\backslash\{\indexi{k}\}}
        \\
        \sgn
            {\indexi{1},\cdots,\indexi{k-1},s,\alpha^{'},\indexi{k+1},\cdots,\indexi{p},\indexJp}
            {\order{(\indexI\backslash\{\indexi{k}\})\cup\{s\}},\order{\indexJp\cup\{\alpha^{'}\}}}
        \\
        \sgn
            {\order{(\indexI\backslash\{\indexi{k}\})\cup\{s\}},\order{\indexJp\cup\{\alpha^{'}\}}}
            {\alpha^{'},\order{(\indexI\backslash\{\indexi{k}\})\cup\{s\}},\indexJp}
    \end{pmatrix*}  
    =
    \sgn
        {s,\indexI,\indexJp}
        {\indexi{k},\order{(\indexI\backslash\{\indexi{k}\})\cup\{s\}},\indexJp}.
\end{align*}
\label{lemma-simplification for LHS, extra 1, part 1-lemma 6}
\end{lemma}

\begin{proof}
We use the following properties of the sign of permutation:
\begin{itemize}
    \item 
$
    \sgn
        {\indexi{1},\cdots,\indexi{k-1},s,\alpha^{'},\indexi{k+1},\cdots,\indexi{p},\indexJp}
        {\order{(\indexI\backslash\{\indexi{k}\})\cup\{s\}},\order{\indexJp\cup\{\alpha^{'}\}}}
    =
    \sgn
        {s,\alpha^{'},\indexI\backslash\{\indexi{k}\},\indexJp}
        {\order{(\indexI\backslash\{\indexi{k}\})\cup\{s\}},\order{\indexJp\cup\{\alpha^{'}\}}}
$,
    \item 
$
    \quad
    \sgn
        {s,\alpha^{'},\indexI\backslash\{\indexi{k}\},\indexJp}
        {\order{(\indexI\backslash\{\indexi{k}\})\cup\{s\}},\order{\indexJp\cup\{\alpha^{'}\}}}
    \sgn
            {\order{(\indexI\backslash\{\indexi{k}\})\cup\{s\}},\order{\indexJp\cup\{\alpha^{'}\}}}
            {\alpha^{'},\order{(\indexI\backslash\{\indexi{k}\})\cup\{s\}},\indexJp}
    \\
    =
    \sgn
        {s,\alpha^{'},\indexI\backslash\{\indexi{k}\},\indexJp}
        {\alpha^{'},\order{(\indexI\backslash\{\indexi{k}\})\cup\{s\}},\indexJp}
    =
    (-1)
    \sgn
        {\alpha^{'},s,\indexI\backslash\{\indexi{k}\},\indexJp}
        {\alpha^{'},\order{(\indexI\backslash\{\indexi{k}\})\cup\{s\}},\indexJp}
    \\
    =
    (-1)
    \sgn
        {s,\indexI\backslash\{\indexi{k}\},\indexJp}
        {\order{(\indexI\backslash\{\indexi{k}\})\cup\{s\}},\indexJp}
$.
\end{itemize}
We simplify the left-hand side of the equation in Lemma \ref{lemma-simplification for LHS, extra 1, part 1-lemma 6} to
\begin{align}
    \sgn
        {\indexI}
        {\indexi{k},\indexI\backslash\{\indexi{k}\}}
    (-1)
    \sgn
        {s,\indexI\backslash\{\indexi{k}\},\indexJp}
        {\order{(\indexI\backslash\{\indexi{k}\})\cup\{s\}},\indexJp}
    .
\label{EQ-simplification for LHS, extra 1, part 1-lemma 6-no1}
\end{align}

We use the following properties of the sign of permutation:
\begin{itemize}
    \item 
$
    \quad
    \sgn
        {s,\indexI\backslash\{\indexi{k}\},\indexJp}
        {\order{(\indexI\backslash\{\indexi{k}\})\cup\{s\}},\indexJp}
    =
    \sgn
        {\indexi{k},s,\indexI\backslash\{\indexi{k}\},\indexJp}
        {\indexi{k},\order{(\indexI\backslash\{\indexi{k}\})\cup\{s\}},\indexJp}
    \\
    =
    (-1)
    \sgn
        {s,\indexi{k},\indexI\backslash\{\indexi{k}\},\indexJp}
        {\indexi{k},\order{(\indexI\backslash\{\indexi{k}\})\cup\{s\}},\indexJp}
$,
    \item 
$
    \sgn
        {\indexI}
        {\indexi{k},\indexI\backslash\{\indexi{k}\}}
    \sgn
        {s,\indexi{k},\indexI\backslash\{\indexi{k}\},\indexJp}
        {\indexi{k},\order{(\indexI\backslash\{\indexi{k}\})\cup\{s\}},\indexJp}
    =
    \sgn
        {s,\indexI,\indexJp}
        {\indexi{k},\order{(\indexI\backslash\{\indexi{k}\})\cup\{s\}},\indexJp}
$.
\end{itemize}
We simplify (\ref{EQ-simplification for LHS, extra 1, part 1-lemma 6-no1}) to
\begin{align*}
    &
    (-1)(-1)
    \sgn
        {s,\indexI,\indexJp}
        {\indexi{k},\order{(\indexI\backslash\{\indexi{k}\})\cup\{s\}},\indexJp}
    \\
    =&
    \sgn
        {s,\indexI,\indexJp}
        {\indexi{k},\order{(\indexI\backslash\{\indexi{k}\})\cup\{s\}},\indexJp}
    ,
\end{align*}
which completes the proof of Lemma \ref{lemma-simplification for LHS, extra 1, part 1-lemma 6}.
\end{proof}



\begin{lemma}
\begin{align*}
    \begin{pmatrix*}[l]
        \sgn
            {\indexI,\indexJp}
            {\indexjp{l},\indexI,\indexJp\backslash\{\indexjp{l}\}}
        \\
        \sgn
            {\indexjp{1},\cdots,\indexjp{l-1},\order{\indexjp{l},\alpha^{'}},\indexjp{l+1},\cdots,\indexjp{q}}
            {\order{\indexJp\cup\{\alpha^{'}\}}}
        \\
        \sgn
            {\indexI,\order{\indexJp\cup\{\alpha^{'}\}}}
            {\alpha^{'},\indexI,\indexJp}
    \end{pmatrix*}
    =
    \sgn
        {\order{\indexjp{l},\alpha^{'}}}
        {\alpha^{'},\indexjp{l}}
    .
\end{align*}
\label{lemma-simplification for LHS, extra 1, part 1-lemma 7}
\end{lemma}

\begin{proof}
We use the following properties of the sign of permutation:
\begin{itemize}
    \item 
$
    \sgn
        {\indexjp{1},\cdots,\indexjp{l-1},\order{\indexjp{l},\alpha^{'}},\indexjp{l+1},\cdots,\indexjp{q}}
        {\order{\indexJp\cup\{\alpha^{'}\}}}
    =
    \sgn
        {\order{\indexjp{l},\alpha^{'}},\indexJp\backslash\{\indexjp{l}\}}
        {\order{\indexJp\cup\{\alpha^{'}\}}}
$,
    \item 
$
    \quad
    \sgn
        {\order{\indexjp{l},\alpha^{'}},\indexJp\backslash\{\indexjp{l}\}}
        {\order{\indexJp\cup\{\alpha^{'}\}}}
    \sgn
        {\indexI,\order{\indexJp\cup\{\alpha^{'}\}}}
        {\alpha^{'},\indexI,\indexJp}
    =
    \sgn
        {\indexI,\order{\indexjp{l},\alpha^{'}},\indexJp\backslash\{\indexjp{l}\}}
        {\alpha^{'},\indexI,\indexJp}
    \\
    =\sgn
        {\order{\indexjp{l},\alpha^{'}},\indexI,\indexJp\backslash\{\indexjp{l}\}}
        {\alpha^{'},\indexI,\indexJp},
$
    \item
$
    \sgn
        {\order{\indexjp{l},\alpha^{'}},\indexI,\indexJp\backslash\{\indexjp{l}\}}
        {\alpha^{'},\indexI,\indexJp}
    =
    \sgn
        {\order{\indexjp{l},\alpha^{'}}}
        {\alpha^{'},\indexjp{l}}
    \sgn
        {\alpha^{'},\indexjp{l},\indexI,\indexJp\backslash\{\indexjp{l}\}}
        {\alpha^{'},\indexI,\indexJp}
$,
    \item 
$
    \sgn
        {\indexI,\indexJp}
        {\indexjp{l},\indexI,\indexJp\backslash\{\indexjp{l}\}}
    \sgn
        {\alpha^{'},\indexjp{l},\indexI,\indexJp\backslash\{\indexjp{l}\}}
        {\alpha^{'},\indexI,\indexJp}
    =
    \sgn
        {\alpha^{'},\indexI,\indexJp}
        {\alpha^{'},\indexI,\indexJp}
    =
    1
$.
\end{itemize}
We simplify the left-hand side of the equation in Lemma \ref{lemma-simplification for LHS, extra 1, part 1-lemma 7} to
\begin{align*}
    \sgn
        {\order{\indexjp{l},\alpha^{'}}}
        {\alpha^{'},\indexjp{l}}
    ,
\end{align*}
which completes the proof of Lemma \ref{lemma-simplification for LHS, extra 1, part 1-lemma 7}.
\end{proof}



\begin{lemma}
\begin{align*}
    \begin{pmatrix*}
        \sgn
            {\indexI,\indexJp}
            {\indexjp{l},\indexI,\indexJp\backslash\{\indexjp{l}\}}
        \\
        \sgn
            {\indexjp{1},\cdots,\indexjp{l-1},\order{t^{'},\alpha^{'}},\indexjp{l+1},\cdots,\indexjp{q}}
            {\order{(\indexJp\backslash\{\indexjp{l}\})\cup\{t^{'},\alpha^{'}\}}}
        \\
        \sgn
            {\indexI,\order{(\indexJp\backslash\{\indexjp{l}\})\cup\{t^{'},\alpha^{'}\}}}
            {\alpha^{'},\indexI,\order{(\indexJp\backslash\{\indexjp{l}\})\cup\{t^{'}\}}}
    \end{pmatrix*}
    =
    \sgn
        {\order{\alpha^{'},t^{'}}}
        {t^{'},\alpha^{'}}
    \sgn
        {t^{'},\indexI,\indexJp}
        {\indexjp{l},\indexI,\order{(\indexJp\backslash\{\indexjp{l}\})\cup\{t^{'}\}}}
    .
\end{align*}
\label{lemma-simplification for LHS, extra 1, part 1-lemma 8}
\end{lemma}

\begin{proof}
We use the following properties of the sign of permutation:
\begin{itemize}
    \item 
$
    \sgn
        {\indexjp{1},\cdots,\indexjp{l-1},\order{t^{'},\alpha^{'}},\indexjp{l+1},\cdots,\indexjp{q}}
        {\order{(\indexJp\backslash\{\indexjp{l}\})\cup\{t^{'},\alpha^{'}\}}}
    =
    \sgn
        {\order{t^{'},\alpha^{'}},\indexJp\backslash\{\indexjp{}\}}
        {\order{(\indexJp\backslash\{\indexjp{l}\})\cup\{t^{'},\alpha^{'}\}}}
$,
    \item 
$
    \sgn
        {\order{t^{'},\alpha^{'}},\indexJp\backslash\{\indexjp{}\}}
        {\order{(\indexJp\backslash\{\indexjp{l}\})\cup\{t^{'},\alpha^{'}\}}}
    \sgn
        {\indexI,\order{(\indexJp\backslash\{\indexjp{l}\})\cup\{t^{'},\alpha^{'}\}}}
        {\alpha^{'},\indexI,\order{(\indexJp\backslash\{\indexjp{l}\})\cup\{t^{'}\}}}
    =
    \sgn
        {\indexI,\order{t^{'},\alpha^{'}},\indexJp\backslash\{\indexjp{}\}}
        {\alpha^{'},\indexI,\order{(\indexJp\backslash\{\indexjp{l}\})\cup\{t^{'}\}}}
$.
\end{itemize}
We simplify the left-hand side of the equation in Lemma \ref{lemma-simplification for LHS, extra 1, part 1-lemma 8} to
\begin{align}
    \sgn
        {\indexI,\indexJp}
        {\indexjp{l},\indexI,\indexJp\backslash\{\indexjp{l}\}}
    \sgn
        {\indexI,\order{t^{'},\alpha^{'}},\indexJp\backslash\{\indexjp{}\}}
        {\alpha^{'},\indexI,\order{(\indexJp\backslash\{\indexjp{l}\})\cup\{t^{'}\}}}
    .
\label{EQ-simplification for LHS, extra 1, part 1-lemma 8-no1}
\end{align}

We use the following properties of the sign of permutation:
\begin{itemize}
    \item 
$
    \quad
    \sgn
        {\indexI,\order{t^{'},\alpha^{'}},\indexJp\backslash\{\indexjp{}\}}
        {\alpha^{'},\indexI,\order{(\indexJp\backslash\{\indexjp{l}\})\cup\{t^{'}\}}}
    =
    \sgn
        {\order{t^{'},\alpha^{'}},\indexI,\indexJp\backslash\{\indexjp{}\}}
        {\alpha^{'},\indexI,\order{(\indexJp\backslash\{\indexjp{l}\})\cup\{t^{'}\}}}
    \\
    =
    \sgn
        {\indexjp{l},\order{t^{'},\alpha^{'}},\indexI,\indexJp\backslash\{\indexjp{}\}}
        {\indexjp{l},\alpha^{'},\indexI,\order{(\indexJp\backslash\{\indexjp{l}\})\cup\{t^{'}\}}}
    =
    \sgn
        {\order{t^{'},\alpha^{'}},\indexjp{l},\indexI,\indexJp\backslash\{\indexjp{}\}}
        {\indexjp{l},\alpha^{'},\indexI,\order{(\indexJp\backslash\{\indexjp{l}\})\cup\{t^{'}\}}}
$,
    \item 
$
    \sgn
        {\indexI,\indexJp}
        {\indexjp{l},\indexI,\indexJp\backslash\{\indexjp{l}\}}
    \sgn
        {\order{t^{'},\alpha^{'}},\indexjp{l},\indexI,\indexJp\backslash\{\indexjp{}\}}
        {\indexjp{l},\alpha^{'},\indexI,\order{(\indexJp\backslash\{\indexjp{l}\})\cup\{t^{'}\}}}
    =
    \sgn
        {\order{t^{'},\alpha^{'}},\indexI,\indexJp}
        {\indexjp{l},\alpha^{'},\indexI,\order{(\indexJp\backslash\{\indexjp{l}\})\cup\{t^{'}\}}}
$,
    \item 
$
    \sgn
        {\order{t^{'},\alpha^{'}},\indexI,\indexJp}
        {\indexjp{l},\alpha^{'},\indexI,\order{(\indexJp\backslash\{\indexjp{l}\})\cup\{t^{'}\}}}
    =
    \sgn
        {\order{\alpha^{'},t^{'}}}
        {t^{'},\alpha^{'}}
    \sgn
        {t^{'},\indexI,\indexJp}
        {\indexjp{l},\indexI,\order{(\indexJp\backslash\{\indexjp{l}\})\cup\{t^{'}\}}}
$.
\end{itemize}
We simplify (\ref{EQ-simplification for LHS, extra 1, part 1-lemma 8-no1}) to 
\begin{align*}
    \sgn
        {\order{\alpha^{'},t^{'}}}
        {t^{'},\alpha^{'}}
    \sgn
        {t^{'},\indexI,\indexJp}
        {\indexjp{l},\indexI,\order{(\indexJp\backslash\{\indexjp{l}\})\cup\{t^{'}\}}}
    ,
\end{align*}
which completes the proof of Lemma \ref{lemma-simplification for LHS, extra 1, part 1-lemma 8}.
\end{proof}


\subsection{Calculation of the coefficients for the expression of \texorpdfstring{$
\LieagbdPartialbar\circ\iota_{\basedualxi{\alpha}}
(\basexi{\indexI} \wedge \basedualxi{\indexJp})
$ such that $\alpha\in\indexJ$}{}
}
\label{apex-simplification for LHS, extra 1, part 2}
In this section, we calculate
\begin{itemize}
    \item 
$
    \sgn
        {\indexI,\indexJp}
        {\alpha^{'},\indexI,\indexJp\backslash\{\alpha^{'}\}}
    \sgn
        {\indexI}
        {\indexi{k},\indexI\backslash\{\indexi{k}\}}
    \sgn
        {\indexi{1},\cdots,\indexi{k},\alpha^{'},\indexi{k+1},\cdots,\indexi{p},\indexJp\backslash\{\alpha^{'}\}}
        {\indexI,\indexJp}
$,
    \item 
$
    \sgn
        {\indexI,\indexJp}
        {\alpha^{'},\indexI,\indexJp\backslash\{\alpha^{'}\}}
    \sgn
        {\indexI}
        {\indexi{k},\indexI\backslash\{\indexi{k}\}}
    \sgn
        {\indexi{1},\cdots,\indexi{k},t^{'},\indexi{k+1},\cdots,\indexi{p},\indexJp\backslash\{\alpha^{'}\}}
        {\indexI,\order{(\indexJp\backslash\{\alpha^{'}\})\cup\{t^{'}\}}}
$,
    \item 
$
    \sgn
        {\indexI,\indexJp}
        {\alpha^{'},\indexI,\indexJp\backslash\{\alpha^{'}\}}
    \sgn
        {\indexI}
        {\indexi{k},\indexI\backslash\{\indexi{k}\}}
    \sgn
        {\indexi{1},\cdots,\indexi{k-1},s,\alpha^{'},\indexi{k+1},\cdots,\indexi{p},\indexJp\backslash\{\alpha^{'}\}}
        {\order{(\indexI\backslash\{\indexi{k}\})\cup\{s\}},\indexJp}
$,
    \item 
$
    \sgn
        {\indexI,\indexJp}
        {\alpha^{'},\indexI,\indexJp\backslash\{\alpha^{'}\}}
    \sgn
        {\indexI}
        {\indexi{k},\indexI\backslash\{\indexi{k}\}}
    \sgn
        {\indexi{1},\cdots,\indexi{k-1},s,t^{'},\indexi{k+1},\cdots,\indexi{p},\indexJp\backslash\{\alpha^{'}\}}
        {\order{(\indexI\backslash\{\indexi{k}\})\cup\{s\}},\order{(\indexJp\backslash\alpha^{'})\cup{t^{'}}}}
$,
    \item 
$
    \sgn
        {\indexI,\indexJp}
        {\alpha^{'},\indexI,\indexJp\backslash\{\alpha^{'}\}}
    \sgn
        {\indexI,\indexJp\backslash\{\alpha^{'}\}}
        {\indexjp{l},\indexI,\indexJp\backslash\{\alpha^{'},\indexjp{l}\}}
    \sgn
        {\indexjp{1},\cdots,\indexjp{l-1},\order{\indexjp{l},\alpha^{'}},\indexjp{l+1},\cdots,\widehat{\alpha^{'}},\cdots,\indexjp{q}}
        {\indexJp}
$,
    \item 
$
    \sgn
        {\indexI,\indexJp}
        {\alpha^{'},\indexI,\indexJp\backslash\{\alpha^{'}\}}
    \sgn
        {\indexI,\indexJp\backslash\{\alpha^{'}\}}
        {\indexjp{l},\indexI,\indexJp\backslash\{\alpha^{'},\indexjp{l}\}}
    \sgn
        {\indexjp{1},\cdots,\indexjp{l-1},\order{\indexjp{l},t^{'}},\indexjp{l+1},\cdots,\widehat{\alpha^{'}},\cdots,\indexjp{q}}
        {\order{(\indexJp\backslash\{\alpha^{'}\})\cup\{t^{'}\}}}
$,
    \item 
$
    \sgn
        {\indexI,\indexJp}
        {\alpha^{'},\indexI,\indexJp\backslash\{\alpha^{'}\}}
    \sgn
        {\indexI,\indexJp\backslash\{\alpha^{'}\}}
        {\indexjp{l},\indexI,\indexJp\backslash\{\alpha^{'},\indexjp{l}\}}
    \sgn
        {\indexjp{1},\cdots,\indexjp{l-1},\order{\alpha^{'},t^{'}},\indexjp{l+1},\cdots,\widehat{\alpha^{'}},\cdots,\indexjp{q}}
        {\order{(\indexJp\backslash\{\indexjp{l}\})\cup\{t^{'}\}}}
$,
    \item 
$
    \sgn
        {\indexI,\indexJp}
        {\alpha^{'},\indexI,\indexJp\backslash\{\alpha^{'}\}}
    \sgn
        {\indexI,\indexJp\backslash\{\alpha^{'}\}}
        {\indexjp{l},\indexI,\indexJp\backslash\{\alpha^{'},\indexjp{l}\}}
    \sgn
        {\indexjp{1},\cdots,\indexjp{l-1},\order{t^{'},u^{'}},\indexjp{l+1},\cdots,\widehat{\alpha^{'}},\cdots,\indexjp{q}}
        {\order{(\indexJp\backslash\{\indexjp{l},\alpha^{'}\})\cup\{t^{'},u^{'}\}}}
$,
\end{itemize}
in Lemmas \ref{lemma-simplification for LHS, extra 1, part 2-lemma 1}–\ref{lemma-simplification for LHS, extra 1, part 2-lemma 8}, which are needed in (\ref{EQ-For (alpha, I, J), calculation for LHS, extra 1-unsimplified the first half, alpha not in J}) to simplify the expression of $
\LieagbdPartialbar\circ\iota_{\basedualxi{\alpha}}
(\basexi{\indexI} \wedge \basedualxi{\indexJp})
$ such that $\alpha\in\indexJ$.


\begin{lemma}
\begin{align*}
    \begin{pmatrix*}[l]
        \sgn
            {\indexI,\indexJp}
            {\alpha^{'},\indexI,\indexJp\backslash\{\alpha^{'}\}}
        \\
        \sgn
            {\indexI}
            {\indexi{k},\indexI\backslash\{\indexi{k}\}}
        \\
        \sgn
            {\indexi{1},\cdots,\indexi{k},\alpha^{'},\indexi{k+1},\cdots,\indexi{p},\indexJp\backslash\{\alpha^{'}\}}
            {\indexI,\indexJp}
    \end{pmatrix*}
    =
    (-1).
\end{align*}
\label{lemma-simplification for LHS, extra 1, part 2-lemma 1}
\end{lemma}
\begin{proof}
We use the following properties of the sign of permutation:
\begin{itemize}
    \item 
$
    \sgn
        {\indexI}
        {\indexi{k},\indexI\backslash\{\indexi{k}\}}
    =
    (-1)^{k-1}
$,
    \item 
$
    \sgn
            {\indexi{1},\cdots,\indexi{k},\alpha^{'},\indexi{k+1},\cdots,\indexi{p},\indexJp\backslash\{\alpha^{'}\}}
            {\indexI,\indexJp}
    =
    (-1)^{k}
    \sgn
            {\alpha^{'},\indexI,\indexJp\backslash\{\alpha^{'}\}}
            {\indexI,\indexJp}
$,
\item 
$
    \sgn
        {\indexI,\indexJp}
        {\alpha^{'},\indexI,\indexJp\backslash\{\alpha^{'}\}}
    \sgn
        {\alpha^{'},\indexI,\indexJp\backslash\{\alpha^{'}\}}
        {\indexI,\indexJp}
    =
    1
$.
\end{itemize}
We simplify the left-hand side of the equation in Lemma \ref{lemma-simplification for LHS, extra 1, part 2-lemma 1} to
\begin{align*}
    (-1)^{k-1}(-1)^{k}
    =
    -1
    ,
\end{align*}
which completes the proof of Lemma \ref{lemma-simplification for LHS, extra 1, part 2-lemma 1}.
\end{proof}


\begin{lemma}
\begin{align*}
    \begin{pmatrix*}[l]
        \sgn
            {\indexI,\indexJp}
            {\alpha^{'},\indexI,\indexJp\backslash\{\alpha^{'}\}}
        \\
        \sgn
            {\indexI}
            {\indexi{k},\indexI\backslash\{\indexi{k}\}}
        \\
        \sgn
            {\indexi{1},\cdots,\indexi{k},t^{'},\indexi{k+1},\cdots,\indexi{p},\indexJp\backslash\{\alpha^{'}\}}
            {\indexI,\order{(\indexJp\backslash\{\alpha^{'}\})\cup\{t^{'}\}}}
    \end{pmatrix*}
    =
    \begin{pmatrix*}[l]
        -(-1)^{k}
        \\
        \sgn
            {\indexI}
            {\indexi{k},\indexI\backslash\{\indexi{k}\}}
        \\
        \sgn
            {t^{'},\indexI,\indexJp}
            {\alpha,\indexI,\order{(\indexJp\backslash\{\alpha^{'}\})\cup\{t^{'}\}}}
    \end{pmatrix*}.
\end{align*}
\label{lemma-simplification for LHS, extra 1, part 2-lemma 2}
\end{lemma}
\begin{proof}
    We use the following properties of the sign of permutation:
\begin{itemize}
    \item 
$
    \sgn
        {\indexi{1},\cdots,\indexi{k},t^{'},\indexi{k+1},\cdots,\indexi{p},\indexJp\backslash\{\alpha^{'}\}}
        {\indexI,\order{(\indexJp\backslash\{\alpha^{'}\})\cup\{t^{'}\}}}
    =
    (-1)^{k}
    \sgn
        {t^{'},\indexI,\indexJp\backslash\{\alpha^{'}\}}
        {\indexI,\order{(\indexJp\backslash\{\alpha^{'}\})\cup\{t^{'}\}}}
$,
    \item 
$
    \sgn
        {t^{'},\indexI,\indexJp\backslash\{\alpha^{'}\}}
        {\indexI,\order{(\indexJp\backslash\{\alpha^{'}\})\cup\{t^{'}\}}}
    =
    (-1)
    \sgn
        {t^{'},\alpha,\indexI,\indexJp\backslash\{\alpha^{'}\}}
        {\alpha,\indexI,\order{(\indexJp\backslash\{\alpha^{'}\})\cup\{t^{'}\}}}
$,
    \item 
$
    \sgn
            {\indexI,\indexJp}
            {\alpha^{'},\indexI,\indexJp\backslash\{\alpha^{'}\}}
    \sgn
        {t^{'},\alpha,\indexI,\indexJp\backslash\{\alpha^{'}\}}
        {\alpha,\indexI,\order{(\indexJp\backslash\{\alpha^{'}\})\cup\{t^{'}\}}}
    =
    \sgn
        {t^{'},\indexI,\indexJp}
        {\alpha,\indexI,\order{(\indexJp\backslash\{\alpha^{'}\})\cup\{t^{'}\}}}
$.
\end{itemize}
We simplify the left-hand side of the equation in Lemma \ref{lemma-simplification for LHS, extra 1, part 2-lemma 2} to
\begin{align}
    -(-1)^{k}
    \sgn
        {\indexI}
        {\indexi{k},\indexI\backslash\{\indexi{k}\}}
    \sgn
        {t^{'},\indexI,\indexJp}
        {\alpha,\indexI,\order{(\indexJp\backslash\{\alpha^{'}\})\cup\{t^{'}\}}}
    ,
\end{align}
which completes the proof of Lemma \ref{lemma-simplification for LHS, extra 1, part 2-lemma 2}.
\end{proof}

\begin{lemma}
\begin{align*}
    \begin{pmatrix*}[l]
        \sgn
            {\indexI,\indexJp}
            {\alpha^{'},\indexI,\indexJp\backslash\{\alpha^{'}\}}
        \\
        \sgn
            {\indexI}
            {\indexi{k},\indexI\backslash\{\indexi{k}\}}
        \\
        \sgn
            {\indexi{1},\cdots,\indexi{k-1},s,\alpha^{'},\indexi{k+1},\cdots,\indexi{p},\indexJp\backslash\{\alpha^{'}\}}
            {\order{(\indexI\backslash\{\indexi{k}\})\cup\{s\}},\indexJp}
    \end{pmatrix*}
    =
    \sgn
        {s,\indexI,\indexJp}
        {\indexi{k},\order{(\indexI\backslash\{\indexi{k}\})\cup\{s\}},\indexJp}
    .
\end{align*}
\label{lemma-simplification for LHS, extra 1, part 2-lemma 3}
\end{lemma}
\begin{proof}
We use the following properties of the sign of permutation:
\begin{itemize}
    \item 
$
    \quad
    \sgn
        {\indexi{1},\cdots,\indexi{k-1},s,\alpha^{'},\indexi{k+1},\cdots,\indexi{p},\indexJp\backslash\{\alpha^{'}\}}
        {\order{(\indexI\backslash\{\indexi{k}\})\cup\{s\}},\indexJp}
    =
    \sgn
        {s,\alpha^{'},\indexI\backslash\{\indexi{k}\},\indexJp\backslash\{\alpha^{'}\}}
        {\order{(\indexI\backslash\{\indexi{k}\})\cup\{s\}},\indexJp}
    \\
    =
    \sgn
        {\indexi{k},s,\alpha^{'},\indexI\backslash\{\indexi{k}\},\indexJp\backslash\{\alpha^{'}\}}
        {\indexi{k},\order{(\indexI\backslash\{\indexi{k}\})\cup\{s\}},\indexJp}
    =
    \sgn
        {s,\alpha^{'},\indexi{k},\indexI\backslash\{\indexi{k}\},\indexJp\backslash\{\alpha^{'}\}}
        {\indexi{k},\order{(\indexI\backslash\{\indexi{k}\})\cup\{s\}},\indexJp}
$,
    \item 
$
    \sgn
        {\indexI}
        {\indexi{k},\indexI\backslash\{\indexi{k}\}}
    \sgn
        {s,\alpha^{'},\indexi{k},\indexI\backslash\{\indexi{k}\},\indexJp\backslash\{\alpha^{'}\}}
        {\indexi{k},\order{(\indexI\backslash\{\indexi{k}\})\cup\{s\}},\indexJp}
    =
    \sgn
        {s,\alpha^{'},\indexI,\indexJp\backslash\{\alpha^{'}\}}
        {\indexi{k},\order{(\indexI\backslash\{\indexi{k}\})\cup\{s\}},\indexJp}
$,
    \item 
$
    \sgn
        {\indexI,\indexJp}
        {\alpha^{'},\indexI,\indexJp\backslash\{\alpha^{'}\}}
    \sgn
        {s,\alpha^{'},\indexI,\indexJp\backslash\{\alpha^{'}\}}
        {\indexi{k},\order{(\indexI\backslash\{\indexi{k}\})\cup\{s\}},\indexJp}
    =
    \sgn
        {s,\indexI,\indexJp}
        {\indexi{k},\order{(\indexI\backslash\{\indexi{k}\})\cup\{s\}},\indexJp}
$.
\end{itemize}
We simplify the left-hand side of the equation in Lemma \ref{lemma-simplification for LHS, extra 1, part 2-lemma 3} to
\begin{align*}
    \sgn
        {s,\indexI,\indexJp}
        {\indexi{k},\order{(\indexI\backslash\{\indexi{k}\})\cup\{s\}},\indexJp}
    ,
\end{align*}
which completes the proof of Lemma \ref{lemma-simplification for LHS, extra 1, part 2-lemma 3}.
\end{proof}


\begin{lemma}
\begin{align*}
    \begin{pmatrix*}[l]
        \sgn
            {\indexI,\indexJp}
            {\alpha^{'},\indexI,\indexJp\backslash\{\alpha^{'}\}}
        \\
        \sgn
            {\indexI}
            {\indexi{k},\indexI\backslash\{\indexi{k}\}}
        \\
        \sgn
            {\indexi{1},\cdots,\indexi{k-1},s,t^{'},\indexi{k+1},\cdots,\indexi{p},\indexJp\backslash\{\alpha^{'}\}}
            {\order{(\indexI\backslash\{\indexi{k}\})\cup\{s\}},\order{(\indexJp\backslash\alpha^{'})\cup{t^{'}}}}
    \end{pmatrix*}
    =
        -
        \sgn
        {s,t^{'},\indexI,\indexJp}
        {\indexi{k},\alpha^{'},\order{(\indexI\backslash\{\indexi{k}\})\cup\{s\}},\order{(\indexJp\backslash\alpha^{'})\cup{t^{'}}}}.
\end{align*}
\label{lemma-simplification for LHS, extra 1, part 2-lemma 4}
\end{lemma}
\begin{proof}
We use the following properties of the sign of permutation:
\begin{itemize}
    \item 
$
    \sgn
            {\indexi{1},\cdots,\indexi{k-1},s,t^{'},\indexi{k+1},\cdots,\indexi{p},\indexJp\backslash\{\alpha^{'}\}}
            {\order{(\indexI\backslash\{\indexi{k}\})\cup\{s\}},\order{(\indexJp\backslash\alpha^{'})\cup{t^{'}}}}
    =
    \sgn
            {s,t^{'},\indexI\backslash\{\indexi{k}\},\indexJp\backslash\{\alpha^{'}\}}
            {\order{(\indexI\backslash\{\indexi{k}\})\cup\{s\}},\order{(\indexJp\backslash\alpha^{'})\cup{t^{'}}}},
$
    \item 
$
    \sgn
            {s,t^{'},\indexI\backslash\{\indexi{k}\},\indexJp\backslash\{\alpha^{'}\}}
            {\order{(\indexI\backslash\{\indexi{k}\})\cup\{s\}},\order{(\indexJp\backslash\alpha^{'})\cup{t^{'}}}}
    =
    \sgn
            {s,t^{'},\indexi{k},\indexI\backslash\{\indexi{k}\},\indexJp\backslash\{\alpha^{'}\}}
            {\indexi{k},\order{(\indexI\backslash\{\indexi{k}\})\cup\{s\}},\order{(\indexJp\backslash\alpha^{'})\cup{t^{'}}}}
$,
    \item 
$
    \quad
    \sgn
            {\indexI}
            {\indexi{k},\indexI\backslash\{\indexi{k}\}}
    \sgn
            {s,t^{'},\indexi{k},\indexI\backslash\{\indexi{k}\},\indexJp\backslash\{\alpha^{'}\}}
            {\indexi{k},\order{(\indexI\backslash\{\indexi{k}\})\cup\{s\}},\order{(\indexJp\backslash\alpha^{'})\cup{t^{'}}}}
    \\
    =
    \sgn
            {s,t^{'},\indexI,\indexJp\backslash\{\alpha^{'}\}}
            {\indexi{k},\order{(\indexI\backslash\{\indexi{k}\})\cup\{s\}},\order{(\indexJp\backslash\alpha^{'})\cup{t^{'}}}}
$.
\end{itemize}
We simplify the left-hand side of the equation in Lemma \ref{lemma-simplification for LHS, extra 1, part 2-lemma 4} to
\begin{align}
    \sgn
        {\indexI,\indexJp}
        {\alpha^{'},\indexI,\indexJp\backslash\{\alpha^{'}\}}
    \sgn
        {s,t^{'},\indexI,\indexJp\backslash\{\alpha^{'}\}}
        {\indexi{k},\order{(\indexI\backslash\{\indexi{k}\})\cup\{s\}},\order{(\indexJp\backslash\alpha^{'})\cup{t^{'}}}}
    .
\label{EQ-simplification for LHS, extra 1, part 2-lemma 3-equation 1}
\end{align}

We use the following properties of the sign of permutation:
\begin{itemize}
    \item 
$
    \sgn
        {s,t^{'},\indexI,\indexJp\backslash\{\alpha^{'}\}}
        {\indexi{k},\order{(\indexI\backslash\{\indexi{k}\})\cup\{s\}},\order{(\indexJp\backslash\alpha^{'})\cup{t^{'}}}}
    =
    \sgn
        {s,t^{'},\alpha^{'},\indexI,\indexJp\backslash\{\alpha^{'}\}}
        {\alpha^{'},\indexi{k},\order{(\indexI\backslash\{\indexi{k}\})\cup\{s\}},\order{(\indexJp\backslash\alpha^{'})\cup{t^{'}}}}
$,
    \item 
$
    \quad
    \sgn
        {\indexI,\indexJp}
        {\alpha^{'},\indexI,\indexJp\backslash\{\alpha^{'}\}}
    \sgn
        {s,t^{'},\alpha^{'},\indexI,\indexJp\backslash\{\alpha^{'}\}}
        {\alpha^{'},\indexi{k},\order{(\indexI\backslash\{\indexi{k}\})\cup\{s\}},\order{(\indexJp\backslash\alpha^{'})\cup{t^{'}}}}
    \\
    =
    \sgn
        {s,t^{'},\indexI,\indexJp}
        {\alpha^{'},\indexi{k},\order{(\indexI\backslash\{\indexi{k}\})\cup\{s\}},\order{(\indexJp\backslash\alpha^{'})\cup{t^{'}}}}
$.
\end{itemize}
We simplify (\ref{EQ-simplification for LHS, extra 1, part 2-lemma 3-equation 1}) to
\begin{align*}
    &
    \sgn
        {s,t^{'},\indexI,\indexJp}
        {\alpha^{'},\indexi{k},\order{(\indexI\backslash\{\indexi{k}\})\cup\{s\}},\order{(\indexJp\backslash\alpha^{'})\cup{t^{'}}}}
    \\
    =&
    (-1)
    \sgn
        {s,t^{'},\indexI,\indexJp}
        {\indexi{k},\alpha^{'},\order{(\indexI\backslash\{\indexi{k}\})\cup\{s\}},\order{(\indexJp\backslash\alpha^{'})\cup{t^{'}}}}
    ,
\end{align*}
which completes the proof of Lemma \ref{lemma-simplification for LHS, extra 1, part 2-lemma 4}.

\end{proof}

\begin{lemma}
\begin{align*}
    \begin{pmatrix*}[l]
        \sgn
            {\indexI,\indexJp}
            {\alpha^{'},\indexI,\indexJp\backslash\{\alpha^{'}\}}
        \\
        \sgn
            {\indexI,\indexJp\backslash\{\alpha^{'}\}}
            {\indexjp{l},\indexI,\indexJp\backslash\{\alpha^{'},\indexjp{l}\}}
        \\
        \sgn
            {\indexjp{1},\cdots,\indexjp{l-1},\order{\indexjp{l},\alpha^{'}},\indexjp{l+1},\cdots,\widehat{\alpha^{'}},\cdots,\indexjp{q}}
            {\indexJp}
    \end{pmatrix*}
    =
    \sgn
        {\order{\indexjp{l},\alpha^{'}}}
        {\alpha^{'},\indexjp{l}}.
\end{align*}
\label{lemma-simplification for LHS, extra 1, part 2-lemma 5}
\end{lemma}
\begin{proof}
We use the following properties of the sign of permutation:
\begin{itemize}
    \item 
$
    \sgn
        {\indexI,\indexJp\backslash\{\alpha^{'}\}}
        {\indexjp{l},\indexI,\indexJp\backslash\{\alpha^{'},\indexjp{l}\}}
    =
    \sgn
        {\alpha^{'},\indexI,\indexJp\backslash\{\alpha^{'}\}}
        {\alpha^{'},\indexjp{l},\indexI,\indexJp\backslash\{\alpha^{'},\indexjp{l}\}}
$,
    \item 
$
    \sgn
        {\indexI,\indexJp}
        {\alpha^{'},\indexI,\indexJp\backslash\{\alpha^{'}\}}
    \sgn
        {\alpha^{'},\indexI,\indexJp\backslash\{\alpha^{'}\}}
        {\alpha^{'},\indexjp{l},\indexI,\indexJp\backslash\{\alpha^{'},\indexjp{l}\}}
    =
    \sgn
        {\indexI,\indexJp}
        {\alpha^{'},\indexjp{l},\indexI,\indexJp\backslash\{\alpha^{'},\indexjp{l}\}}
$,
    \item 
$
    \sgn
        {\indexjp{1},\cdots,\indexjp{l-1},\order{\indexjp{l},\alpha^{'}},\indexjp{l+1},\cdots,\widehat{\alpha^{'}},\cdots,\indexjp{q}}
        {\indexJp}
    =
    \sgn
        {\order{\indexjp{l},\alpha^{'}},\indexJp\backslash\{\indexjp{l},\alpha^{'}\}}
        {\indexJp}
$.
\end{itemize}
We simplify the left-hand side of the equation in Lemma \ref{lemma-simplification for LHS, extra 1, part 2-lemma 5} to
\begin{align}
    \sgn
        {\indexI,\indexJp}
        {\alpha^{'},\indexjp{l},\indexI,\indexJp\backslash\{\alpha^{'},\indexjp{l}\}}
    \sgn
        {\order{\indexjp{l},\alpha^{'}},\indexJp\backslash\{\indexjp{l},\alpha^{'}\}}
        {\indexJp}
    .
\label{EQ-simplification for LHS, extra 1, part 2-lemma 5-equation 1}
\end{align}

We use the following properties of the sign of permutation:
\begin{itemize}
    \item 
$
    \quad
    \sgn
        {\order{\indexjp{l},\alpha^{'}},\indexJp\backslash\{\indexjp{l},\alpha^{'}\}}
        {\indexJp}
    =
    \sgn
        {\indexI,\order{\indexjp{l},\alpha^{'}},\indexJp\backslash\{\indexjp{l},\alpha^{'}\}}
        {\indexI,\indexJp}
    =
    \sgn
        {\order{\indexjp{l},\alpha^{'}},\indexI,\indexJp\backslash\{\indexjp{l},\alpha^{'}\}}
        {\indexI,\indexJp}
    \\
    =
    \sgn
        {\order{\indexjp{l},\alpha^{'}}}
        {\alpha^{'},\indexjp{l}}
    \sgn
        {\alpha^{'},\indexjp{l},\indexI,\indexJp\backslash\{\indexjp{l},\alpha^{'}\}}
        {\indexI,\indexJp},
$
    \item 
$
    \sgn
        {\indexI,\indexJp}
        {\alpha^{'},\indexjp{l},\indexI,\indexJp\backslash\{\alpha^{'},\indexjp{l}\}}
    \sgn
        {\alpha^{'},\indexjp{l},\indexI,\indexJp\backslash\{\indexjp{l},\alpha^{'}\}}
        {\indexI,\indexJp}
    =
    1
$.
\end{itemize}
We simplify (\ref{EQ-simplification for LHS, extra 1, part 2-lemma 5-equation 1}) to
\begin{align*}
    \sgn
        {\order{\indexjp{l},\alpha^{'}}}
        {\alpha^{'},\indexjp{l}}
    ,
\end{align*}
which completes the proof of Lemma \ref{lemma-simplification for LHS, extra 1, part 2-lemma 5}.
\end{proof}


\begin{lemma}
\begin{align*}
    \begin{pmatrix*}[l]
        \sgn
            {\indexI,\indexJp}
            {\alpha^{'},\indexI,\indexJp\backslash\{\alpha^{'}\}}
        \\
        \sgn
            {\indexI,\indexJp\backslash\{\alpha^{'}\}}
            {\indexjp{l},\indexI,\indexJp\backslash\{\alpha^{'},\indexjp{l}\}}
        \\
        \sgn
            {\indexjp{1},\cdots,\indexjp{l-1},\order{\indexjp{l},t^{'}},\indexjp{l+1},\cdots,\widehat{\alpha^{'}},\cdots,\indexjp{q}}
            {\order{(\indexJp\backslash\{\alpha^{'}\})\cup\{t^{'}\}}}
    \end{pmatrix*}
    =
    (-1)
    \sgn
        {\order{\indexjp{l},t^{'}}}
        {t^{'},\indexjp{l}}
    \sgn
        {t^{'},\indexI,\indexJp}
        {\alpha^{'},\indexI,\order{(\indexJp\backslash\{\alpha^{'}\})\cup\{t^{'}\}}}
    .
\end{align*}
\label{lemma-simplification for LHS, extra 1, part 2-lemma 6}
\end{lemma}
\begin{proof}
We use the following properties of the sign of permutation:
\begin{itemize}
    \item 
$
    \sgn
        {\indexI,\indexJp\backslash\{\alpha^{'}\}}
        {\indexjp{l},\indexI,\indexJp\backslash\{\alpha^{'},\indexjp{l}\}}
    =
    \sgn
        {\alpha^{'},\indexI,\indexJp\backslash\{\alpha^{'}\}}
        {\alpha^{'},\indexjp{l},\indexI,\indexJp\backslash\{\alpha^{'},\indexjp{l}\}}
$,
    \item 
$
    \sgn
        {\indexI,\indexJp}
        {\alpha^{'},\indexI,\indexJp\backslash\{\alpha^{'}\}}
    \sgn
        {\alpha^{'},\indexI,\indexJp\backslash\{\alpha^{'}\}}
        {\alpha^{'},\indexjp{l},\indexI,\indexJp\backslash\{\alpha^{'},\indexjp{l}\}}
    =
    \sgn
        {\indexI,\indexJp}
        {\alpha^{'},\indexjp{l},\indexI,\indexJp\backslash\{\alpha^{'},\indexjp{l}\}}
$,
    \item 
$
    \sgn
        {\indexjp{1},\cdots,\indexjp{l-1},\order{\indexjp{l},t^{'}},\indexjp{l+1},\cdots,\widehat{\alpha^{'}},\cdots,\indexjp{q}}
        {\order{(\indexJp\backslash\{\alpha^{'}\})\cup\{t^{'}\}}}
    =
    \sgn
        {\order{\indexjp{l},t^{'}},\indexJp\backslash\{\alpha^{'},\indexjp{l}\}}
        {\order{(\indexJp\backslash\{\alpha^{'}\})\cup\{t^{'}\}}}
$,
    \item 
$
    \sgn
        {\order{\indexjp{l},t^{'}},\indexJp\backslash\{\alpha^{'},\indexjp{l}\}}
        {\order{(\indexJp\backslash\{\alpha^{'}\})\cup\{t^{'}\}}}
    =
    \sgn
        {\indexI,\order{\indexjp{l},t^{'}},\indexJp\backslash\{\alpha^{'},\indexjp{l}\}}
        {\indexI,\order{(\indexJp\backslash\{\alpha^{'}\})\cup\{t^{'}\}}}
    =
    \sgn
        {\order{\indexjp{l},t^{'}},\indexI,\indexJp\backslash\{\alpha^{'},\indexjp{l}\}}
        {\indexI,\order{(\indexJp\backslash\{\alpha^{'}\})\cup\{t^{'}\}}}
$.
\end{itemize}
We simplify the left-hand side of the equation in Lemma \ref{lemma-simplification for LHS, extra 1, part 2-lemma 6} to
\begin{align}
    \sgn
        {\indexI,\indexJp}
        {\alpha^{'},\indexjp{l},\indexI,\indexJp\backslash\{\alpha^{'},\indexjp{l}\}}
    \sgn
        {\order{\indexjp{l},t^{'}},\indexI,\indexJp\backslash\{\alpha^{'},\indexjp{l}\}}
        {\indexI,\order{(\indexJp\backslash\{\alpha^{'}\})\cup\{t^{'}\}}}
    .
\label{EQ-simplification for LHS, extra 1, part 2-lemma 6-equation 1}
\end{align}

We use the following properties of the sign of permutation:
\begin{itemize}
    \item 
$
    \sgn
        {\order{\indexjp{l},t^{'}},\indexI,\indexJp\backslash\{\alpha^{'},\indexjp{l}\}}
        {\indexI,\order{(\indexJp\backslash\{\alpha^{'}\})\cup\{t^{'}\}}}
    =
    \sgn
        {\indexjp{l},t^{'},\indexI,\indexJp\backslash\{\alpha^{'},\indexjp{l}\}}
        {\indexI,\order{(\indexJp\backslash\{\alpha^{'}\})\cup\{t^{'}\}}}
    \sgn
        {\order{\indexjp{l},t^{'}}}
        {\indexjp{l},t^{'}}
$,
    \item 
$
    \sgn
        {\indexjp{l},t^{'},\indexI,\indexJp\backslash\{\alpha^{'},\indexjp{l}\}}
        {\indexI,\order{(\indexJp\backslash\{\alpha^{'}\})\cup\{t^{'}\}}}
    =
    \sgn
        {\alpha^{'},\indexjp{l},t^{'},\indexI,\indexJp\backslash\{\alpha^{'},\indexjp{l}\}}
        {\alpha^{'},\indexI,\order{(\indexJp\backslash\{\alpha^{'}\})\cup\{t^{'}\}}}
    \\
    =
    \sgn
        {t^{'},\alpha^{'},\indexjp{l},\indexI,\indexJp\backslash\{\alpha^{'},\indexjp{l}\}}
        {\alpha^{'},\indexI,\order{(\indexJp\backslash\{\alpha^{'}\})\cup\{t^{'}\}}}
$,
    \item 
$
    \sgn
        {\indexI,\indexJp}
        {\alpha^{'},\indexjp{l},\indexI,\indexJp\backslash\{\alpha^{'},\indexjp{l}\}}
    \sgn
        {t^{'},\alpha^{'},\indexjp{l},\indexI,\indexJp\backslash\{\alpha^{'},\indexjp{l}\}}
        {\alpha^{'},\indexI,\order{(\indexJp\backslash\{\alpha^{'}\})\cup\{t^{'}\}}}
    =
    \sgn
        {t^{'},\indexI,\indexJp}
        {\alpha^{'},\indexI,\order{(\indexJp\backslash\{\alpha^{'}\})\cup\{t^{'}\}}}
$.
\end{itemize}
We simplify (\ref{EQ-simplification for LHS, extra 1, part 2-lemma 6-equation 1}) to 
\begin{align}
    &
    \sgn
        {\order{\indexjp{l},t^{'}}}
        {\indexjp{l},t^{'}}
    \sgn
        {t^{'},\indexI,\indexJp}
        {\alpha^{'},\indexI,\order{(\indexJp\backslash\{\alpha^{'}\})\cup\{t^{'}\}}}
    \\
    =&
    (-1)
    \sgn
        {\order{\indexjp{l},t^{'}}}
        {t^{'},\indexjp{l}}
    \sgn
        {t^{'},\indexI,\indexJp}
        {\alpha^{'},\indexI,\order{(\indexJp\backslash\{\alpha^{'}\})\cup\{t^{'}\}}}
    ,
\end{align}
which completes the proof of Lemma \ref{lemma-simplification for LHS, extra 1, part 2-lemma 6}.
\end{proof}


\begin{lemma}
\begin{align*}
    \begin{pmatrix*}[l]
        \sgn
            {\indexI,\indexJp}
            {\alpha^{'},\indexI,\indexJp\backslash\{\alpha^{'}\}}
        \\
        \sgn
            {\indexI,\indexJp\backslash\{\alpha^{'}\}}
            {\indexjp{l},\indexI,\indexJp\backslash\{\alpha^{'},\indexjp{l}\}}
        \\
        \sgn
            {\indexjp{1},\cdots,\indexjp{l-1},\order{\alpha^{'},t^{'}},\indexjp{l+1},\cdots,\widehat{\alpha^{'}},\cdots,\indexjp{q}}
            {\order{(\indexJp\backslash\{\indexjp{l}\})\cup\{t^{'}\}}}
    \end{pmatrix*}
    =
    \sgn
        {\order{\alpha^{'},t^{'}}}
        {t^{'},\alpha^{'}}
    \sgn
        {t^{'},\indexI,\indexJp}
        {\indexjp{l},\order{(\indexJp\backslash\{\indexjp{l}\})\cup\{t^{'}\}}}
    .
\end{align*}
\label{lemma-simplification for LHS, extra 1, part 2-lemma 7}
\end{lemma}
\begin{proof}
We use the following properties of the sign of permutation:
\begin{itemize}
    \item 
$
    \sgn
        {\indexI,\indexJp\backslash\{\alpha^{'}\}}
        {\indexjp{l},\indexI,\indexJp\backslash\{\alpha^{'},\indexjp{l}\}}
    =
    \sgn
        {\alpha^{'},\indexI,\indexJp\backslash\{\alpha^{'}\}}
        {\alpha^{'},\indexjp{l},\indexI,\indexJp\backslash\{\alpha^{'},\indexjp{l}\}}
$,
    \item 
$
    \sgn
        {\indexI,\indexJp}
        {\alpha^{'},\indexI,\indexJp\backslash\{\alpha^{'}\}}
    \sgn
        {\alpha^{'},\indexI,\indexJp\backslash\{\alpha^{'}\}}
        {\alpha^{'},\indexjp{l},\indexI,\indexJp\backslash\{\alpha^{'},\indexjp{l}\}}
    =
    \sgn
        {\indexI,\indexJp}
        {\alpha^{'},\indexjp{l},\indexI,\indexJp\backslash\{\alpha^{'},\indexjp{l}\}}
$,
    \item 
$
    \sgn
        {\indexjp{1},\cdots,\indexjp{l-1},\order{\alpha^{'},t^{'}},\indexjp{l+1},\cdots,\widehat{\alpha^{'}},\cdots,\indexjp{q}}
        {\order{(\indexJp\backslash\{\indexjp{l}\})\cup\{t^{'}\}}}
    =
    \sgn
        {\order{\alpha^{'},t^{'}},\indexJp\backslash \{\indexjp{l},\alpha^{'} \}}
        {\order{(\indexJp\backslash\{\indexjp{l}\})\cup\{t^{'}\}}}
$.
\end{itemize}
We simplify the left-hand side of the equation in Lemma \ref{lemma-simplification for LHS, extra 1, part 2-lemma 7} to
\begin{align}
    \sgn
        {\indexI,\indexJp}
        {\alpha^{'},\indexjp{l},\indexI,\indexJp\backslash\{\alpha^{'},\indexjp{l}\}}
    \sgn
        {\order{\alpha^{'},t^{'}},\indexJp\backslash \{\indexjp{l},\alpha^{'} \}}
        {\order{(\indexJp\backslash\{\indexjp{l}\})\cup\{t^{'}\}}}
    .
\label{EQ-simplification for LHS, extra 1, part 2-lemma 7-equation 1}
\end{align}

We use the following properties of the sign of permutation:
\begin{itemize}
    \item 
$
    \sgn
        {\order{\alpha^{'},t^{'}},\indexJp\backslash \{\indexjp{l},\alpha^{'} \}}
        {\order{(\indexJp\backslash\{\indexjp{l}\})\cup\{t^{'}\}}}
    =
    \sgn
        {\order{\alpha^{'},t^{'}}}
        {t^{'},\alpha^{'}}
    \sgn
        {t^{'},\alpha^{'},\indexJp\backslash \{\indexjp{l},\alpha^{'} \}}
        {\order{(\indexJp\backslash\{\indexjp{l}\})\cup\{t^{'}\}}}
$,
    \item 
$
    \sgn
        {t^{'},\alpha^{'},\indexJp\backslash \{\indexjp{l},\alpha^{'} \}}
        {\order{(\indexJp\backslash\{\indexjp{l}\})\cup\{t^{'}\}}}
    =
    \sgn
        {\indexjp{l},t^{'},\alpha^{'},\indexJp\backslash \{\indexjp{l},\alpha^{'} \}}
        {\indexjp{l},\order{(\indexJp\backslash\{\indexjp{l}\})\cup\{t^{'}\}}}
    =
    \sgn
        {t^{'},\alpha^{'},\indexjp{l},\indexJp\backslash \{\indexjp{l},\alpha^{'} \}}
        {\indexjp{l},\order{(\indexJp\backslash\{\indexjp{l}\})\cup\{t^{'}\}}}
$,
    \item 
$
    \sgn
        {\indexI,\indexJp}
        {\alpha^{'},\indexjp{l},\indexI,\indexJp\backslash\{\alpha^{'},\indexjp{l}\}}
    \sgn
        {t^{'},\alpha^{'},\indexjp{l},\indexJp\backslash \{\indexjp{l},\alpha^{'} \}}
        {\indexjp{l},\order{(\indexJp\backslash\{\indexjp{l}\})\cup\{t^{'}\}}}
    =
    \sgn
        {t^{'},\indexI,\indexJp}
        {\indexjp{l},\order{(\indexJp\backslash\{\indexjp{l}\})\cup\{t^{'}\}}}
$.
\end{itemize}
We simplify (\ref{EQ-simplification for LHS, extra 1, part 2-lemma 7-equation 1}) to
\begin{align*}
    \sgn
        {\order{\alpha^{'},t^{'}}}
        {t^{'},\alpha^{'}}
    \sgn
        {t^{'},\indexI,\indexJp}
        {\indexjp{l},\order{(\indexJp\backslash\{\indexjp{l}\})\cup\{t^{'}\}}}
    ,
\end{align*}
which completes the proof of Lemma \ref{lemma-simplification for LHS, extra 1, part 2-lemma 7}.
\end{proof}


\begin{lemma}
\begin{align*}
    \begin{pmatrix*}[l]
        \sgn
            {\indexI,\indexJp}
            {\alpha^{'},\indexI,\indexJp\backslash\{\alpha^{'}\}}
        \\
        \sgn
            {\indexI,\indexJp\backslash\{\alpha^{'}\}}
            {\indexjp{l},\indexI,\indexJp\backslash\{\alpha^{'},\indexjp{l}\}}
        \\
        \sgn
            {\indexjp{1},\cdots,\indexjp{l-1},\order{t^{'},u^{'}},\indexjp{l+1},\cdots,\widehat{\alpha^{'}},\cdots,\indexjp{q}}
            {\order{(\indexJp\backslash\{\indexjp{l},\alpha^{'}\})\cup\{t^{'},u^{'}\}}}
    \end{pmatrix*}
    =
    -
    \sgn
        {\order{t^{'},u^{'}},\indexI,\indexJp}
        {\indexjp{l},\alpha^{'},\indexI,\order{(\indexJp\backslash\{\indexjp{l},\alpha^{'}\})\cup\{t^{'},u^{'}\}}}.
\end{align*}
\label{lemma-simplification for LHS, extra 1, part 2-lemma 8}
\end{lemma}
\begin{proof}
We use the following properties of the sign of permutation:
\begin{itemize}
    \item 
$
    \sgn
        {\indexI,\indexJp\backslash\{\alpha^{'}\}}
        {\indexjp{l},\indexI,\indexJp\backslash\{\alpha^{'},\indexjp{l}\}}
    =
    \sgn
        {\alpha^{'},\indexI,\indexJp\backslash\{\alpha^{'}\}}
        {\alpha^{'},\indexjp{l},\indexI,\indexJp\backslash\{\alpha^{'},\indexjp{l}\}}
$,
    \item 
$
    \sgn
        {\indexI,\indexJp}
        {\alpha^{'},\indexI,\indexJp\backslash\{\alpha^{'}\}}
    \sgn
        {\alpha^{'},\indexI,\indexJp\backslash\{\alpha^{'}\}}
        {\alpha^{'},\indexjp{l},\indexI,\indexJp\backslash\{\alpha^{'},\indexjp{l}\}}
    =
    \sgn
        {\indexI,\indexJp}
        {\alpha^{'},\indexjp{l},\indexI,\indexJp\backslash\{\alpha^{'},\indexjp{l}\}}
$,
    \item 
$
    \sgn
        {\indexjp{1},\cdots,\indexjp{l-1},\order{t^{'},u^{'}},\indexjp{l+1},\cdots,\widehat{\alpha^{'}},\cdots,\indexjp{q}}
        {\order{(\indexJp\backslash\{\indexjp{l},\alpha^{'}\})\cup\{t^{'},u^{'}\}}}
    =
    \sgn
        {\order{t^{'},u^{'}},\indexJp\backslash\{\indexjp{l},\alpha^{'}\}}
        {\order{(\indexJp\backslash\{\indexjp{l},\alpha^{'}\})\cup\{t^{'},u^{'}\}}}
    \\
    =
    \sgn
        {\indexI,\order{t^{'},u^{'}},\indexJp\backslash\{\indexjp{l},\alpha^{'}\}}
        {\indexI,\order{(\indexJp\backslash\{\indexjp{l},\alpha^{'}\})\cup\{t^{'},u^{'}\}}}
    =
    \sgn
        {\order{t^{'},u^{'}},\indexI,\indexJp\backslash\{\indexjp{l},\alpha^{'}\}}
        {\indexI,\order{(\indexJp\backslash\{\indexjp{l},\alpha^{'}\})\cup\{t^{'},u^{'}\}}}
$.
\end{itemize}
We simplify the left-hand side of the equation in Lemma \ref{lemma-simplification for LHS, extra 1, part 2-lemma 8} to
\begin{align}
    \sgn
        {\indexI,\indexJp}
        {\alpha^{'},\indexjp{l},\indexI,\indexJp\backslash\{\alpha^{'},\indexjp{l}\}}
    \sgn
        {\order{t^{'},u^{'}},\indexI,\indexJp\backslash\{\indexjp{l},\alpha^{'}\}}
        {\indexI,\order{(\indexJp\backslash\{\indexjp{l},\alpha^{'}\})\cup\{t^{'},u^{'}\}}}
    .
\label{EQ-simplification for LHS, extra 1, part 2-lemma 8-equation 1}
\end{align}

We use the following properties of the sign of permutation:
\begin{itemize}
    \item
$
    \sgn
        {\order{t^{'},u^{'}},\indexI,\indexJp\backslash\{\indexjp{l},\alpha^{'}\}}
        {\indexI,\order{(\indexJp\backslash\{\indexjp{l},\alpha^{'}\})\cup\{t^{'},u^{'}\}}}
    =
    \sgn
        {\alpha^{'},\indexjp{l},\order{t^{'},u^{'}},\indexI,\indexJp\backslash\{\indexjp{l},\alpha^{'}\}}
        {\alpha^{'},\indexjp{l},\indexI,\order{(\indexJp\backslash\{\indexjp{l},\alpha^{'}\})\cup\{t^{'},u^{'}\}}}
    \\
    =
    \sgn
        {\order{t^{'},u^{'}},\alpha^{'},\indexjp{l},\indexI,\indexJp\backslash\{\indexjp{l},\alpha^{'}\}}
        {\alpha^{'},\indexjp{l},\indexI,\order{(\indexJp\backslash\{\indexjp{l},\alpha^{'}\})\cup\{t^{'},u^{'}\}}}
$,
    \item 
$
    \sgn
        {\indexI,\indexJp}
        {\alpha^{'},\indexjp{l},\indexI,\indexJp\backslash\{\alpha^{'},\indexjp{l}\}}
    \sgn
        {\order{t^{'},u^{'}},\alpha^{'},\indexjp{l},\indexI,\indexJp\backslash\{\indexjp{l},\alpha^{'}\}}
        {\alpha^{'},\indexjp{l},\indexI,\order{(\indexJp\backslash\{\indexjp{l},\alpha^{'}\})\cup\{t^{'},u^{'}\}}}
    \\
    =
    \sgn
        {\order{t^{'},u^{'}},\indexI,\indexJp}
        {\alpha^{'},\indexjp{l},\indexI,\order{(\indexJp\backslash\{\indexjp{l},\alpha^{'}\})\cup\{t^{'},u^{'}\}}}
$.
\end{itemize}
We simplify (\ref{EQ-simplification for LHS, extra 1, part 2-lemma 8-equation 1}) to
\begin{align*}
    &
    \sgn
        {\order{t^{'},u^{'}},\indexI,\indexJp}
        {\alpha^{'},\indexjp{l},\indexI,\order{(\indexJp\backslash\{\indexjp{l},\alpha^{'}\})\cup\{t^{'},u^{'}\}}}
    \\
    =&
    (-1)
    \sgn
        {\order{t^{'},u^{'}},\indexI,\indexJp}
        {\indexjp{l},\alpha^{'},\indexI,\order{(\indexJp\backslash\{\indexjp{l},\alpha^{'}\})\cup\{t^{'},u^{'}\}}}
    ,
\end{align*}
which completes the proof of Lemma \ref{lemma-simplification for LHS, extra 1, part 2-lemma 8}.
\end{proof}


\section{Appendix III}
\label{sec:proofs of some short properties}
The proof of proposition \ref{prop: a list of prop for Hodge star and Lie algebroid diff} is as follows:
\subsection{The proof of 1.}
\begin{proof}
    The expression on the left-hand side is
\begin{align*}
    \star(f\,\basexi{\indexI}\wedge\basedualxi{\indexJp})
    =&
        2^{p+q-n} f
        \sgn
            {1,\cdots,n,1',\cdots,n'}
            {\indexI,\indexJp,\indexIc,\indexJpc}
        \basedualxi{\indexIc}
        \wedge
        \basexi{\indexJpc}
    \\
    =&
        2^{p+q-n} \bar{f}
        \sgn
            {1,\cdots,n,1',\cdots,n'}
            {\indexI,\indexJp,\indexIc,\indexJpc}
        \basexi{\indexIc}
        \wedge
        \basedualxi{\indexJpc}
    . 
\end{align*}
The expression on the right-hand side is
\begin{align*}
    \star
        (\overline{f\,\basexi{\indexI}\wedge\basedualxi{\indexJp}})
    =&
    \star
        (\bar{f}\,\basedualxi{\indexI}\wedge\basexi{\indexJp})
    \\
    =&
    \star
        (
        \bar{f}
        \sgn
            {\indexI,\indexJp}
            {\indexJp,\indexI}
        \basexi{\indexJp}\wedge\basedualxi{\indexI}
        )
    \\
    =&
        2^{p+q-n} \bar{f}
        \sgn    
            {\indexI,\indexJp}
            {\indexJp,\indexI}
        \sgn
            {1,\cdots,n,1',\cdots,n'}
            {\indexJp,\indexI,\indexJpc,\indexIc} 
        \basedualxi{\indexJpc}\wedge\basexi{\indexIc}
    \\
    =&
        2^{p+q-n} \bar{f}
        \sgn    
            {\indexI,\indexJp}
            {\indexJp,\indexI}
        \sgn
            {1,\cdots,n,1',\cdots,n'}
            {\indexJp,\indexI,\indexJpc,\indexIc} 
        \sgn
            {\indexJpc,\indexIc}
            {\indexIc,\indexJpc}
        \basexi{\indexIc}\wedge\basedualxi{\indexJpc}
    \\
    =&
        2^{p+q-n} \bar{f}
        \sgn
            {1,\cdots,n,1',\cdots,n'}
            {\indexI,\indexJp,\indexIc,\indexJpc}
        \basexi{\indexIc}
        \wedge
        \basedualxi{\indexJpc}
    .
\end{align*}
\end{proof}

\subsection{The proof of 2.}
\begin{proof}
    Because of 
\begin{align*}
    \LieagbdDiff 
    (f\,\zeta_{\indexi{1}}\wedge\cdots\wedge\zeta_{\indexi{k}})
    =
    (\LieagbdDiff f)
    \wedge\zeta_{\indexi{1}}\wedge\cdots\wedge\zeta_{\indexi{k}}
    +
    \displaystyle\sum_{j=1}^{k}
    (-1)^{j}
    f\, \zeta_{\indexi{1}}
    \wedge\cdots\wedge
    (\LieagbdDiff \zeta_{\indexi{j}})
    \wedge\cdots\wedge
    \zeta_{\indexi{k}}
    ,
\end{align*}
it is sufficient to complete the proof on the neighborhood chart through two parts. The first part is to show $\overline{\LieagbdDiff\,f}=\LieagbdDiff\,\overline{f}$; the second part is to show $\overline{\LieagbdDiff \basexi{}}=\LieagbdDiff \overline{\basedualxi{}}=\LieagbdDiff \basedualxi{}$. We achieve this using the local frame $\{\basev{i},\basedualv{i}\}_{i=1}^{n}$ spanning $\gC|_{U}$ and the corresponding dual frame $\{\basexi{i},\basedualxi{i}\}_{i=1}^{n}$ spanning $\gstarC|_{U}$.

\textbf{Step 1:} Let us look at the first part. Suppose $f=g+\sqrt{-1} h$ where $u,v$ are real-valued functions. We take the Lie algebroid differential
\begin{align*}
    \LieagbdDiff f
    =&
    \displaystyle\sum_{i=1}^{n}
    \big( \rho(\basev{i}) f \big) \basexi{i}
    +
    \big( \rho(\basedualv{i}) f \big) \basedualxi{i}
\notag
    \\
    =&
    \displaystyle\sum_{i=1}^{n}
    \big(
        \rho(\basev{i}) g + \sqrt{-1} \rho(\basev{i}) h
    \big) \basexi{i}
    +
    \big(
        \rho(\basedualv{i}) g + \sqrt{-1} \rho(\basedualv{i}) h
    \big) \basedualxi{i}
    .
\end{align*}
Setting $\alpha_{i}:=\basev{i}+\basedualv{i}\in\secsp{U}{\g|_{U}}$ and $\beta_{i}:=\sqrt{-1}(\basev{i}-\basedualv{i})\in\secsp{U}{\g|_{U}}$, we can rewrite $\basev{i}=\frac{1}{2}(\alpha_{i}-\sqrt{-1}\beta_{i})$ and $\basedualv{i}=\frac{1}{2}(\alpha_{i}+\sqrt{-1}\beta_{i})$ and plug these expressions into the expression of $\LieagbdDiff f$:
\begin{align*}
    &
    \frac{1}{2}
    \displaystyle\sum_{i=1}^{n}
    \begin{pmatrix*}
        \big(
            \rho(\alpha_{i}-\sqrt{-1}\beta_{i}) g 
            +
            \sqrt{-1} \rho(\alpha_{i}-\sqrt{-1}\beta_{i}) h
        \big) \basexi{i}
    \\
        +
    \\
        \big(
            \rho(\alpha_{i}+\sqrt{-1}\beta_{i}) g 
            + 
            \sqrt{-1} \rho(\alpha_{i}+\sqrt{-1}\beta_{i}) h
        \big) \basedualxi{i}
    \end{pmatrix*}
\\
    =&
    \frac{1}{2}
    \displaystyle\sum_{i=1}^{n}
    \begin{pmatrix*}
        \bigg(
            \big(
                \rho(\alpha_i) g + \rho(\beta_i) h
            \big) 
            -
            \sqrt{-1}
            \big(
                \rho(\beta_i)g - \rho(\alpha_i) h
            \big) 
        \bigg)
        \basexi{i}
        \\
        +
        \\
        \bigg(
            \big(
                \rho(\alpha_i) g -\rho(\beta_i) h
            \big) \basedualxi{i}
            +
            \sqrt{-1}
            \big(
                \rho(\beta_i) g + \rho(\alpha_i) h
            \big) 
        \bigg)
        \basedualxi{i}
    \end{pmatrix*}
    .
\end{align*}
Taking the conjugation, we obtain the expression of $\overline{\LieagbdDiff f}$:
\begin{align}
    \overline{\LieagbdDiff f}
    =
    \frac{1}{2}
    \displaystyle\sum_{i=1}^{n}
    \begin{pmatrix*}
        \bigg(
            \big(
                \rho(\alpha_i) g + \rho(\beta_i) h
            \big) 
            +
            \sqrt{-1}
            \big(
                \rho(\beta_i)g - \rho(\alpha_i) h
            \big) 
        \bigg)
        \basedualxi{i}
        \\
        +
        \\
        \bigg(
            \big(
                \rho(\alpha_i) g -\rho(\beta_i) h
            \big) \basedualxi{i}
            -
            \sqrt{-1}
            \big(
                \rho(\beta_i) g + \rho(\alpha_i) h
            \big) 
        \bigg)
        \basexi{i}
    \end{pmatrix*}
    ,
\label{eq-kahler Lie agbd- -overline of df}
\end{align}
where the action of $\rho(\alpha_i)$ and $\rho(\beta_i)$ on $g$ and $h$ are real-valued functions.

On the other side, the Lie algebroid differential on the conjugation of $f$ is
\begin{align}
    \LieagbdDiff \bar{f}
    =&
    \displaystyle\sum_{i=1}^{n}
    \big( \rho(\basev{i}) \bar{f} \big) \basexi{i}
    +
    \big( \rho(\basedualv{i}) \bar{f} \big) \basedualxi{i}
\notag
\\
    =&
    \displaystyle\sum_{i=1}^{n}
    \big(
        \rho(\basev{i}) g - \sqrt{-1} \rho(\basev{i}) h
    \big) \basexi{i}
    +
    \big(
        \rho(\basedualv{i}) g - \sqrt{-1} \rho(\basedualv{i}) h
    \big) \basedualxi{i}
\notag
\\
    =&
    \frac{1}{2}
    \displaystyle\sum_{i=1}^{n}
    \begin{pmatrix*}
        \big(
            \rho(\alpha_{i}-\sqrt{-1}\beta_{i}) g 
            -
            \sqrt{-1} \rho(\alpha_{i}-\sqrt{-1}\beta_{i}) h
        \big) \basexi{i}
    \\
        +
    \\
        \big(
            \rho(\alpha_{i}+\sqrt{-1}\beta_{i}) g 
            - 
            \sqrt{-1} \rho(\alpha_{i}+\sqrt{-1}\beta_{i}) h
        \big) \basedualxi{i}
    \end{pmatrix*}
\notag
\\
    =&
    \frac{1}{2}
    \displaystyle\sum_{i=1}^{n}
    \begin{pmatrix*}
        \bigg(
            \big(
                \rho(\alpha_i)g -\rho(\beta_i)h
            \big) 
            -
            \sqrt{-1}
            \big(
                \rho(\beta_i)g + \rho(\alpha_i) h
            \big) 
        \bigg)
        \basexi{i}
    \\
        +
    \\
        \bigg(
            \big(
                \rho(\alpha_i) g + \rho(\beta_i) h
            \big) \basedualxi{i}
            +
            \sqrt{-1}
            \big(
                \rho(\beta_i)g - \rho(\alpha_i) h
            \big) 
        \bigg)
        \basedualxi{i}
    \end{pmatrix*}
    .
\label{eq-kahler Lie agbd- -d on overline of f}
\end{align}
Comparing (\ref{eq-kahler Lie agbd- -overline of df}) to (\ref{eq-kahler Lie agbd- -d on overline of f}), we prove the equality $\overline{\LieagbdDiff f}=\LieagbdDiff \bar{f}$.

\textbf{Step 2:} Recalling (\ref{eq-kahler Lie agbd-d xi i}) and (\ref{eq-kahler Lie agbd-d xi bar i}), which are the Lie algebroid differentials on the frames $\basexi{i}$ and $\basedualxi{i}$, we calculate
\begin{align*}
    \overline{\LieagbdDiff \basexi{i}}
    =&
        \sum_{j<k}^n
        \overline{
            \coefA{i}{jk}
            \basexi{j} \wedge \basexi{k}
        }
        +
        \sum_{j,k}^n
        \overline{
            \coefB{i}{jk}
            \basexi{j} \wedge \basedualxi{k}
        }
\\
    =&
        \sum_{j<k}^n
        \overline{\coefA{i}{jk}}
        \basedualxi{j} \wedge \basedualxi{k}
        +
        \sum_{j,k}^n
        \overline{\coefB{i}{jk}}
        \basedualxi{j} \wedge \basexi{k}
\\
    =&
        \sum_{j<k}^n
        \overline{\coefA{i}{jk}}
        \basedualxi{j} \wedge \basedualxi{k}
        +
        \sum_{j,k}^n
        (-\overline{\coefB{i}{jk}})
        \basexi{k} \wedge \basedualxi{j} 
\\
    =&
        \sum_{j<k}^n
        \overline{\coefA{i}{jk}}
        \basedualxi{j} \wedge \basedualxi{k}
        +
        \sum_{j,k}^n
        (-\overline{\coefB{i}{kj}})
        \basexi{j} \wedge \basedualxi{k} 
\end{align*}
and
\begin{align*}
    \LieagbdDiff \basedualxi{}
    =&
        \sum_{j,k}^n
        \coefC{\indexi{}}{j k}
        \xi_j \wedge \basedualxi{k}
        +
        \sum_{j<k}^n
        \coefD{i}{jk}
        \basedualxi{j} \wedge \basedualxi{k}
    .
\end{align*}
According to (\ref{eq-kahler lie agbd-bar Akij= Dkij}) and (\ref{eq-kahler lie agbd-bar Bkij=-Ckji}), the coefficients $A,B,C,$ and $D$ are constrained:
\begin{align*}
    \overline{\coefA{k}{i j}}=\coefD{k}{i j}
    \qquad\text{and}\qquad
    \overline{\coefB{k}{i j}}=-\coefC{k}{j i}
    .
\end{align*}
Therefore, we can conclude that $\overline{\LieagbdDiff \basexi{i}}=\LieagbdDiff \basedualxi{}$.

\end{proof}

\subsection{The proof of 3.}
\begin{proof}
    It is enough to verify the identity on a local chart $U$. The local expression of the Hodge star on $f\,\basexi{\indexI}\wedge\basedualxi{\indexJp}\in\Apq{p}{q}(U)$ is
    \begin{align*}
        \star
        (f\,\basexi{\indexI}\wedge\basedualxi{\indexJp})
        =
        2^{p+q-n}
        \sgn
            {1,\cdots,n,1^{'},\cdots,n^{'}}
            {\indexI,\indexJp,\indexIc,\indexJpc}
        f\,
        \basedualxi{\indexIc}\wedge\basexi{\indexJpc}.
    \end{align*}
    Applying the Hodge star to $\basedualxi{\indexIc}\wedge\basexi{\indexJpc}$, we obtain
    \begin{align}
        \star
        (\basedualxi{\indexIc}\wedge\basexi{\indexJpc})
        &=
            \star
            \bigg(
            \sgn
                {\indexIc,\indexJpc}
                {\indexJpc,\indexIc}
            \basexi{\indexJpc}\wedge\basedualxi{\indexIc}
            \bigg)
        \nonumber
        \\
        &=
            \sgn
                {\indexIc,\indexJpc}
                {\indexJpc,\indexIc}
            2^{n-q+n-p-n}
            \sgn
                {1^{'},\cdots,n^{'},1,\cdots,n,}
                {\indexJpc,\indexIc,\indexJp,\indexI}
            \basedualxi{\indexJp}\wedge\basexi{\indexI}
        \nonumber
        \\
        &=
            \sgn
                {\indexIc,\indexJpc}
                {\indexJpc,\indexIc}
            2^{n-q+n-p-n}
            \sgn
                {1^{'},\cdots,n^{'},1,\cdots,n,}
                {\indexJpc,\indexIc,\indexJp,\indexI}
            \sgn 
                {\indexJp,\indexI}
                {\indexI,\indexJp}
            \basexi{\indexI}\wedge\basedualxi{\indexJp}
        \nonumber
        \\
        &=
            \sgn
                {\indexIc,\indexJpc}
                {\indexJpc,\indexIc}
            2^{n-q+n-p-n}
            \sgn
                {1^{'},\cdots,n^{'},1,\cdots,n,}
                {\indexJpc,\indexIc,\indexJp,\indexI}
            \sgn 
                {\indexJp,\indexI}
                {\indexI,\indexJp}
            \basexi{\indexI}\wedge\basedualxi{\indexJp}
        \nonumber
        \\
        &=
            2^{n-q+n-p-n}
            \sgn
                {1^{'},\cdots,n^{'},1,\cdots,n,}
                {\indexIc,\indexJpc,\indexI,\indexJp}
            \basexi{\indexI}\wedge\basedualxi{\indexJp}.
    \label{EQ-Part 1-(1)-star on Ic Jpc}
    \end{align}
    Observe that $\star^{2}(f\,\basexi{\indexI}\wedge\basedualxi{\indexJp})$ is equal to
    \begin{align}
        2^{p+q-n}
        \sgn
            {1,\cdots,n,1^{'},\cdots,n^{'}}
            {\indexI,\indexJp,\indexIc,\indexJpc}
        f\,
        \star
        (\basedualxi{\indexIc}\wedge\basexi{\indexJpc}).
    \label{EQ-Part 1-(1)-the second time apply star}
    \end{align}
    We plug (\ref{EQ-Part 1-(1)-star on Ic Jpc}) into (\ref{EQ-Part 1-(1)-the second time apply star}) and obtain
    \begin{align*}
        &
            2^{p+q-n}
            \sgn
                {1,\cdots,n,1^{'},\cdots,n^{'}}
                {\indexI,\indexJp,\indexIc,\indexJpc}
            2^{n-q+n-p-n}
            \sgn
                {1^{'},\cdots,n^{'},1,\cdots,n,}
                {\indexIc,\indexJpc,\indexI,\indexJp}
            f\,
            \basexi{\indexI}\wedge\basedualxi{\indexJp}
        \\
        =&
            \sgn
                {1,\cdots,n,1^{'},\cdots,n^{'}}
                {\indexI,\indexJp,\indexIc,\indexJpc}
            \sgn
                {1^{'},\cdots,n^{'},1,\cdots,n,}
                {\indexI,\indexJp,\indexIc,\indexJpc}
            \sgn
                {\indexIc,\indexJpc,\indexI,\indexJp}
                {\indexI,\indexJp,\indexIc,\indexJpc}
            f\,
            \basexi{\indexI}\wedge\basedualxi{\indexJp}
        \\
        =&
            \sgn
                {1,\cdots,n,1^{'},\cdots,n^{'}}
                {1^{'},\cdots,n^{'},1,\cdots,n,}
            (-1)^{(p+q)(n-p+n-q)}
            f\,
            \basexi{\indexI}\wedge\basedualxi{\indexJp}
        \\
        =&
            (-1)^{n^2}
            (-1)^{(p+q)^2}
            f\,
            \basexi{\indexI}\wedge\basedualxi{\indexJp},
    \end{align*}
    where $(-1)^{(p+q)^2}=(-1)^{p+q}$.
\end{proof}

\subsection{The proof of 4.}
See reference \cite{ELW1999}.

\subsection{The proof of 5.}
\begin{proof}
    Suppose  $\psi\in\Apq{p}{q}\subset\Ak{p+q}(M)$ and $\phi\in\Ak{p+q-1}(M)$, then we can use Stoke's theorem on $\phi\wedge\overline{\star\psi}$, which is a $(2n-1)$-Lie algebroid form, and obtain
\begin{align}
    \int_{M}
    \innerprod{\LieagbdDiff (\phi\wedge\overline{\star\psi})}{\Omega}{}
    =0.
    \label{EQ-Part 1-(1)-use stock on phi wedge bar star psi}
\end{align}
Plugging the equality $\LieagbdDiff (\phi\wedge\overline{\star\psi})=
    (\LieagbdDiff \phi)
    \wedge
    \overline{\star\psi}
    +
    (-1)^{p+q-1}
    \phi
    \wedge
    \LieagbdDiff(\overline{\star\psi})
$ into (\ref{EQ-Part 1-(1)-use stock on phi wedge bar star psi}), we obtain
\begin{align*}
    0=&
    \int_{M}
    \innerprod
        {(\LieagbdDiff \phi)\wedge\overline{\star\psi}}
        {\Omega}
        {}
    +
    (-1)^{p+q-1}
    \int_{M}
    \innerprod
        {\phi\wedge\LieagbdDiff(\overline{\star\psi})}
        {\Omega}
        {}
    \\
    =&
    \pinnerprod
        {\LieagbdDiff \phi}
        {\psi}
        {\Ak{p+q}(M)}
    +
    (-1)^{p+q-1}
    \int_{M}
    \innerprod
        {\phi\wedge\LieagbdDiff(\overline{\star\psi})}
        {\Omega}
        {}
    \\
    =&
    \pinnerprod
        {\phi}
        {\LieagbdDiff^{*}\psi}
        {\Ak{p+q}(M)}
    +
    (-1)^{p+q-1}
    \int_{M}
    \innerprod
        {\phi\wedge\LieagbdDiff(\overline{\star\psi})}
        {\Omega}
        {}.
\end{align*}
We plug $(-1)^{n^2}(-1)^{p+q}\star\star=\id$ in front of the second $\LieagbdDiff$ and obtain
\begin{align*}
    0=&
        \pinnerprod
            {\phi}
            {\LieagbdDiff^{*}\psi}
            {\Ak{p+q}(M)}
        +
        (-1)^{p+q-1}
        \int_{M}
        \innerprod
            {
                \phi\wedge
                (-1)^{n^2}(-1)^{p+q}\star\star
                \LieagbdDiff(\overline{\star\psi})
            }
            {\Omega}
            {}
    \\
    =&
        \pinnerprod
            {\phi}
            {\LieagbdDiff^{*}\psi}
            {\Ak{p+q}(M)}
        +
        (-1)^{p+q-1}
        \int_{M}
        \innerprod
            {
                \phi\wedge
                (-1)^{n^2}(-1)^{p+q}\star
                \overline{\star\LieagbdDiff\star\psi}
            }
            {\Omega}
            {}
    \text{potential problem}
    \\
    =&
        \pinnerprod
            {\phi}
            {\LieagbdDiff^{*}\psi}
            {\Ak{p+q}(M)}
        +
        (-1)
        (-1)^{n^2}
        \pinnerprod
            {\phi}
            {\star\LieagbdDiff\star\psi}
            {},
\end{align*}
which implies $\LieagbdDiff^{*}=(-1)^{n^2}\star\LieagbdDiff\star$.

\end{proof}

\bibliographystyle{unsrtnat}
\bibliography{references}  






\end{document}